\documentclass[10pt,a4paper]{article}
\usepackage[nottoc]{tocbibind}
\usepackage{ragged2e}
\usepackage{graphicx}
\usepackage{slashed}
\usepackage{enumerate}
\usepackage{amsmath}
\usepackage[bottom]{footmisc}
\usepackage{bbm}
\usepackage[hidelinks]{hyperref}
{

}
\usepackage{amsfonts}
\usepackage{cancel}
\usepackage{psfrag}
\usepackage{color}
\usepackage{mathtools}
\usepackage{pgf,tikz}
\usepackage{mathrsfs}
\usetikzlibrary{arrows}
\usepackage{fancyhdr}
\usepackage{amssymb}
\usepackage{slashed}
\usepackage{enumitem}

\usepackage[T1]{fontenc}
\usepackage[utf8]{inputenc}
\usepackage{authblk}

\newtheorem{definition}{Definition}
\newtheorem{remark}{Remark}

\newtheorem{proposition}{Proposition}[section]

\newtheorem{theorem}{Theorem}[section]

\newtheorem{lemma}{Lemma}[section]

\numberwithin{equation}{section}
\allowdisplaybreaks[4]

\newenvironment{proof}{\smallskip\noindent\emph{Proof.}\hspace{1pt}}%
{\hspace{-5pt}{\nobreak\quad\nobreak\hfill\nobreak$\square$\vspace{8pt}%
		\par}\smallskip\goodbreak}
  
\newcommand{\pbar}{\overline{p}}
\newcommand{\hnabla}{\widehat{\nabla}}
\newcommand{\Lbar}{\underline{L}}

\newcommand{\Hbar}{\underline{H}}
\newcommand{\aln}{(\al \snabla)^i}
\newcommand{\upr}{\lvert u^\prime \rvert}
\newcommand{\tbeta}{\tilde{\beta}}
\newcommand{\tbetabar}{\tilde{\betabar}}
\newcommand{\omegabar}{\underline{\omega}}
\newcommand{\sumitm}{\sum_{i_1+i_2+i_3=i-1}}

\newcommand{\ubarprime}{\ubar^{\prime}}

\newcommand{\Vl}{\mathcal{V}\ell}

\newcommand{\scaleinfinitySu}[1]{\lVert{#1} \rVert_{\mathcal{L}^\infty_{(sc)}(S_{u,\ubar})}}
\newcommand{\sumifi}{\sum_{i_1+i_2+i_3+i_4+i_5=i}}
\newcommand{\sumifim}{\sum_{i_1+i_2+i_3+i_4+i_5=i-1}}
\newcommand{\modu}{\lvert u \rvert}

\newcommand{\Te}{\slashed{T}}

\newcommand{\chihat}{\hat{\chi}}

\newcommand{\pslash}{\slashed{p}}
\newcommand{\al}{a^{\frac{1}{2}}}
\newcommand{\chibar}{\underline{\chi}}
\newcommand{\chibarhat}{\underline{\hat{\chi}}}
\newcommand{\ubar}{\underline{u}}

\newcommand{\db}{\dione \psi_g^{i_2}\dit }
\newcommand{\tr}{\operatorname{tr}} 
\newcommand{\be}{\begin{equation}}
\newcommand{\ee}{\end{equation}}

\newcommand{\bm}{\begin{multline}}
\newcommand{\enm}{\end{multline}}

\newcommand{\supp}{\operatorname{supp}}

\newcommand{\scaleinfinitySuprime}[1]{\lVert{#1} \rVert_{\mathcal{L}^\infty_{(sc)}(S_{u^{\prime},\ubar})}}
\newcommand{\scalefourSuprime}[1]{\lVert{#1} \rVert_{\mathcal{L}^4_{(sc)}(S_{u^{\prime},\ubar})}}

\newcommand{\scalefourSuzero}[1]{\lVert{#1} \rVert_{\mathcal{L}^4_{(sc)}(S_{u_{\infty},\ubar})}}

\newcommand{\scaleinftySu}[1]{\lVert{#1} \rVert_{\mathcal{L}^{\infty}_{(sc)}(S_{u,\ubar})}}

\newcommand{\scalefourSuubarprime}[1]{\lVert{#1} \rVert_{\mathcal{L}^4_{(sc)}(S_{u,\ubar^\prime})}}
\newcommand{\scaleinftySuubarprime}[1]{\lVert{#1} \rVert_{\mathcal{L}^{\infty}_{(sc)}(S_{u,\ubar^\prime})}}
\newcommand{\scaleinftySuprime}[1]{\lVert{#1} \rVert_{\mathcal{L}^{\infty}_{(sc)}(S_{u^{\prime},\ubar})}}

\newcommand{\tildetr}{\widetilde{\tr \chibar}}

\newcommand{\Y}{\mathrm{\Upsilon}}

\newcommand{\alell}{(\al)^{\ell}}
\newcommand{\twoSu}[1]{\lVert{#1} \rVert_{L^2(S_{u,\ubar})}}
\newcommand{\oneSu}[1]{\lVert{#1} \rVert_{L^1(S_{u,\ubar})}}
\newcommand{\inftySu}[1]{\lVert{#1} \rVert_{L^{\infty}(S_{u,\ubar})}}
\newcommand{\fourSu}[1]{\lVert{#1} \rVert_{L^4(S_{u,\ubar})}}
\newcommand{\psisone}{\psi^{(s_1)}}
\newcommand{\psistwo}{\psi^{(s_2)}}
\newcommand{\fourSuinf}[1]{\lVert{#1} \rVert_{L^4(S_{u_{\infty},\ubar})}}
\newcommand{\scaletwoSu}[1]{\lVert{#1} \rVert_{\mathcal{L}^2_{(sc)}(S_{u,\ubar})}}
\newcommand{\scalefourSu}[1]{\lVert{#1} \rVert_{\mathcal{L}^4_{(sc)}(S_{u,\ubar})}}
\newcommand{\scaletwoSuprime}[1]{\lVert{#1} \rVert_{\mathcal{L}^2_{(sc)}(S_{u^\prime,\ubar})}}

\newcommand{\scaleoneSuprimeubarprime}[1]{\lVert{#1} \rVert_{\mathcal{L}^1_{(sc)}(S_{u^\prime,\ubar^\prime})}}
\newcommand{\scaletwoSuprimeubarprime}[1]{\lVert{#1} \rVert_{\mathcal{L}^2_{(sc)}(S_{u^\prime,\ubar^\prime})}}
\newcommand{\scaleinfinitySuubarprime}[1]{\lVert{#1} \rVert_{\mathcal{L}^{\infty}_{(sc)}(S_{u,\ubar^\prime})}}

\newcommand{\scalefourSuinf}[1]{\lVert{#1} \rVert_{\mathcal{L}^4_{(sc)}(S_{u_{\infty},\ubar})}}
\newcommand{\scaletwoSuubarprime}[1]{\lVert{#1} \rVert_{\mathcal{L}^2_{(sc)}(S_{u,\ubar^\prime})}}
\newcommand{\ScaletwoSuubarprime}[1]{\Big \lVert{#1} \Big \rVert_{\mathcal{L}^2_{(sc)}(S_{u,\ubar^\prime})}}

\newcommand{\alphabar}{\underline{\alpha}}
\newcommand{\betabar}{\underline{\beta}}
\newcommand{\etabar}{\underline{\eta}}

\newcommand{\scaletwoHzero}[1]{\lVert{#1} \rVert_{\mathcal{L}^2_{(sc)}(H_{u_{\infty}}^{(0,\underline{u})})}}

\newcommand{\scaletwoHu}[1]{\lVert{#1} \rVert_{\mathcal{L}^2_{(sc)}(H_u^{(0,\underline{u})})}}
\newcommand{\scaletwoHbarzero}[1]{\lVert{#1} \rVert_{\mathcal{L}^2_{(sc)}(\underline{H}_{0}^{(u_{\infty},u)})}}
\newcommand{\scaletwoHbaru}[1]{\lVert{#1} \rVert_{\mathcal{L}^2_{(sc)}(\underline{H}_{\underline{u}}^{(u_{\infty},u)})}}
\newcommand{\sumthreetwo}{\sum_{i_1+i_2+i_3=2}}

\newcommand{\sumfourone}{\sum_{i_1+i_2+i_3+i_4=1}}
\newcommand{\sumfourtwo}{\sum_{i_1+i_2+i_3+i_4=2}}
\newcommand{\sumfiveone}{\sum_{i_1+i_2+i_3+i_4+i_5=1}}

\newcommand{\sumsixone}{\sum_{i_1+i_2+i_3+i_4+i_5+i_6=1}}

\newcommand{\Hor}{\text{Hor}}

\newcommand{\psis}{\psi^{(s)}}
\newcommand{\Psis}{\Psi^{(s)}}
\newcommand{\Vls}{\Vl^{(s)}}
\newcommand{\Hodge}[1]{\prescript{*}{}{#1}}

\newcommand{\sumit}{\sum_{i_1+i_2+i_3=i}}
\newcommand{\dione}{\mathcal{D}^{i_1}}
\newcommand{\sumif}{\sum_{i_1+i_2+i_3+i_4=i}}
\newcommand{\sumifm}{\sum_{i_1+i_2+i_3+i_4=i-1}}
\newcommand{\dit}{\mathcal{D}^{i_3}}
\newcommand{\dif}{\mathcal{D}^{i_4}}
\newcommand{\difi}{\mathcal{D}^{i_5}}

\newcommand{\intu}{\int_{u_{\infty}}^u}
\newcommand{\intubar}{\int_0^{\ubar}}
\newcommand{\duprime}{\hspace{.5mm} \text{d}u^{\prime}}
\newcommand{\dubarprime}{\hspace{.5mm} \text{d} \ubar^{\prime}}

\newcommand{\bal}{\begin{align}}
\newcommand{\eal}{\end{align}}
\newcommand{\ub}{\underline{u}}

\newcommand{\Lb}{\underline{L}}
\newcommand{\Tslash}{\slashed{T}}

\newcommand{\di}{\mathcal{D}^{i_1}\psi_g^{i_2}\mathcal{D}^{i_3}}

\def\ub {\underline{u}}

\def\Lb {\underline{L}}
\def\Hb {\underline{H}}

\def\f {\frac}

\def\Hu{H_u^{(0,\underline{u})}}
\def\Hbu{\underline{H}_{\underline{u}}^{(u_{\infty},u)}}

\renewcommand{\div}{\mbox{div }}
\newcommand{\curl}{\mbox{curl }}

\newcommand{\De}{\mathcal{D}}
\newcommand{\Ve}{\mathcal{V}}

\newcommand{\hsp}{\hspace{.5mm}}
\newcommand{\gslash}{\slashed{g}}

\newcommand{\sdiv}{\mbox{div}\mkern-19mu /\,\,\,\,}
\newcommand{\scurl}{\mbox{curl}\mkern-19mu /\,\,\,\,}

\newcommand{\Omegaterm}{\big(\frac{1}{\Omega^2}-1\big)}

\newcommand{\uinf}{\lvert u_{\infty}\rvert}

\newcommand{\snabla}{\slashed{\nabla}}
\newcommand{\Gammaslash}{\slashed{\Gamma}}

\newcommand{\nablasl}{\slashed{\nabla}}
\newcommand\restri[2]{{
		\left.\kern-\nulldelimiterspace 
		#1 
		\right|_{#2} 
}}

\definecolor{ffqqqq}{rgb}{1.,0.,0.}
\definecolor{uuuuuu}{rgb}{0.26666666666666666,0.26666666666666666,0.26666666666666666}

\title{The Einstein-Vlasov system in a Large Data Regime}
\author{Nikolaos Athanasiou\thanks{\text{nathanai@math.auth.gr}} -- Puskar Mondal\thanks{\text{pushkarmondal@gmail.com}} -- Shing-Tung Yau\thanks{\text{yau@math.harvard.edu}}}
\date{}

\begin{document}

\maketitle

\begin{abstract}
\noindent 

\noindent In this article, we study the Einstein-Vlasov system for massless particles. Our main contribution is the provision of a novel method for controlling the Vlasov matter in a double null gauge, relying purely on vector field commutation and bypassing the need for introducing Jacobi fields, which has so far been the only existing technique outside of symmetry in this gauge. To obtain it, we have to overcome stringent regularity issues that exist along this path. Because it relies solely on vector field commutation, our technique is fundamentally simple and flexible enough to be applicable in any data regime. We thus anticipate this to be a helpful tool for many subsequent problems regarding the Einstein-Vlasov system, including, for instance, simplified approaches to the proof of nonlinear stability of Minkowski spacetime with massless Vlasov matter, scattering problems for data close to black holes among others. Within the present double-null commutator framework, we show that two derivatives of curvature suffice to close the coupled Einstein--Vlasov hierarchy, without derivative loss, given smooth initial data. Using this method, we obtain a large data semi-global existence theorem and a dynamical trapped surface formation statement for our system. 
\end{abstract}

\textit{MSC2020 Mathematics Subject Classification: 35Q83, 83C55, 83C57, 83C75. }
\tableofcontents
\section{Introduction}
In the present work, we investigate the problems of large data semi-global existence and dynamical trapped surface formation for the Einstein-Vlasov system for a $(3+1)-$dimensional connected and time-oriented Lorentzian manifold $\left(\mathcal{M},g\right)$. Local coordinates $(x^{\mu})$ on $\mathcal{M}$ induce coordinates $(x^{\mu},p^{\mu})$ on the tangent bundle $T\mathcal{M}$, where the fibre variables $p^{\mu}$ represent the components of a tangent vector $p^{\mu}\partial_{x^{\mu}}|_{x}\in T_{x}\mathcal{M}$. The \emph{mass shell} $\mathcal{P}\subset T\mathcal{M}$ is defined by  
\[
   \mathcal{P} := \bigl\{ (x,p) \;\big|\; g_x(p,p)=0,\; p \text{ future-directed} \bigr\}.
\]

\noindent A nonnegative function $f:\mathcal{P}\to \mathbb{R}_{\geq 0}$ is said to solve the (massless) Vlasov equation if it is preserved by the geodesic flow on $T\mathcal{M}$. Equivalently,
\begin{equation}\label{eq:Vlasov}
   X(f)=0,
\end{equation}
where $X$ denotes the restriction to $\mathcal{P}$ of the geodesic spray:
\[
   \tilde{X}= p^{\mu}\partial_{\mu} - \Gamma^{\lambda}_{\alpha\beta}p^{\alpha}p^{\beta}\partial_{p^{\lambda}} \;\in \Gamma(TT\mathcal{M}).
\]

\noindent The \emph{Einstein--Vlasov system} consists of \eqref{eq:Vlasov} coupled to the Einstein field equations
\begin{equation}\label{eq:Einstein}
   \operatorname{Ric}(g) - \tfrac{1}{2} R(g)\, g \;=\;  T,
\end{equation}
(in the units $8\pi G=1=c$, $c$ is the light speed) where the energy--momentum tensor is given by
\begin{equation}\label{energymomentumtensor}
   T_{\mu\nu}(x)
   = \int_{0}^{\infty}\!\!  \int_{\mathbb{R}^2}
     f(x,p)\, p_{\mu}p_{\nu}\;
     \frac{1}{p^{3}}\,\sqrt{\det\slashed{g}}\,
     dp^{1}dp^{2}dp^{3}.
\end{equation}
Here, $p^{4}$ is determined from $(p^{1},p^{2},p^{3})$ and the induced metric $\slashed{g}_{AB}$ on the topological $2-$spheres (the leaves of the double null foliation) by the mass--shell constraint
\begin{equation}\label{massshell}
   \slashed{g}_{AB} p^{A}p^{B} - 4p^{3}p^{4} = 0.
\end{equation}
In the massless setting under consideration, the scalar curvature vanishes and \eqref{eq:Einstein} reduces to
\[
   R_{\mu\nu}=T_{\mu\nu}.
\]

\vspace{3mm}
    \begin{figure}
    \begin{center}
    \begin{tikzpicture}
		\draw [white](3,-1)-- node[midway, sloped, below,black]{$H_{u_{\infty}}(u=u_{\infty})$}(4,0);
		
		\draw [white](2,2)--node [midway,sloped,above,black] {$\Hb_1(\ub=1)$}(4,0);
		\draw [white](1,1)--node [midway,sloped, below,black] {$\Hb_{0}(\ub=0)$}(3,-1);
		\draw [dashed] (0, 4)--(0, -4);
		\draw [dashed] (0, -4)--(4,0)--(0,4);
		\draw [dashed] (0,0)--(2,2);
		\draw [dashed] (0,-4)--(2,-2);
		\draw [dashed] (0,2)--(3,-1);
		\draw [very thick] (1,1)--(3,-1)--(4,0)--(2,2)--(1,1);
		\fill [black!35!white]  (1,1)--(3,-1)--(4,0)--(2,2)--(1,1);
		\draw [white](1,1)-- node[midway,sloped,above,black]{$H_{u}$}(2,2);
		\draw [->] (3.3,-0.6)-- node[midway, sloped, above,black]{$L'$}(3.6,-0.3);
		\draw [->] (1.4,1.3)-- node[midway, sloped, below,black]{$L'$}(1.7,1.6);
		\draw [->] (3.3,0.6)-- node[midway, sloped, below,black]{$\Lb'$}(2.7,1.2);
		\draw [->] (2.4,-0.3)-- node[midway, sloped, above,black]{$\Lb'$}(1.7,0.4);
		\end{tikzpicture}
			\caption{Schematic depiction of the spacetime region of existence}
	\end{center}
	\end{figure}

\subsection{Historical background} \noindent Trapped surfaces have been an object of significant interest within the classical theory of General Relativity for almost sixty years. After Schwarzschild's discovery of his eponymous metric in 1915, it took almost twenty years before researchers came to realize the existence, within it, of a region $\mathcal{B}$ with the following surprising yet salient features: First of all, observers situated inside $\mathcal{B}$ cannot send signals to observers situated at an ideal \textit{conformal boundary at infinity}, called $\mathcal{I}^+$.
Furthermore, any observer located inside $\mathcal{B}$ lives only for finite proper time\footnote{Sbierski \cite{Sb} moreover showed that the termination of the observer's proper time manifests in a particularly ferocious way, as they, in fact, get torn apart by infinitely strong tidal forces.} (geodesic incompleteness). 
The characteristics of this region (which later came to be known as a black hole) took most of the researchers of the time aback. The consensus seemed to be that these observed phenomena have to be accidents; pathologies, only present because of the strong (spherical) symmetry inherent in the Schwarzschild solution and that, in general solutions to the Einstein equations\footnote{The meaning of this phrase was not rigorous at the time, as the setup for the initial value problem in General Relativity had not yet been discovered.}, such phenomena would not arise. However, in the 60's, this belief was spectacularly falsified by Roger Penrose through his celebrated incompleteness theorem\footnote{This theorem, in fact, was the main reason why he was awarded the Nobel Prize in Physics back in 2020, "for the discovery that black hole formation is a robust prediction of the general theory of relativity".}. It was Penrose \cite{P73} who introduced the notion of a \textit{trapped suface}, without which one cannot state his eponymous, celebrated, incompleteness theorem :

\begin{definition}
Given a $(3+1)$- dimensional Lorentzian manifold $(\mathcal{M}, g)$, a closed spacelike $2-$surface $S$ is caled \textbf{trapped} if the following two fundamental forms $\chi$ and $\chibar$ have everywhere pointwise negative expansions on $S$:

\[  \chi(X,Y) := g(D_X L, Y), \hspace{2mm} \chibar(X, Y) := g(D_X \Lbar, Y).    \]Here $D$ denotes the Levi-Civita connection of $g$, $L$ and $\Lbar$ denote a null basis of the 2-dimensional orthogonal complement of $T_p S$ in $T_p \mathcal{M}$, extended as smooth vector fields and $X, Y$ are arbitrary $S-$tangent vector fields. 
\end{definition}In other words, a surface is called trapped if both $\tr\chi$ and $\tr\chibar$ are pointwise negative everywhere on $S$. These traces signify the infinitesimal changes in area along the null generators normal to $S$, whence one can interpret trapped surfaces as closed, spacelike $2-$surfaces that infinitesimally decrease in area "along any possible future direction".

\vspace{3mm}

\par\noindent The incompleteness theorem is now presented.

\begin{theorem}[Penrose Incompleteness] Let $(\mathcal{M} ,g)$ be a spacetime containing a non-compact Cauchy hypersurface. If $(\mathcal{M}, g)$ moreover satisfies the null energy codition and contains a closed trapped surface, it is geodesically incomplete. 

\end{theorem}The existence of a trapped surface is a stable feature in the context of dynamics. Indeed, sufficiently small perturbations of Schwarzschild initial data must also contain such surfaces, by Cauchy stability. As such, incompleteness is not an accident, but rather a recurring theme in the dynamics of the Einstein equations. 

\vspace{3mm}

\par\noindent At the time, the existence of a trapped surface in a spacetime was too strong an assumption to begin with. In fact, the only way back then to guarantee its existence was to assume it at the level of initial data, but this itself can be a highly non-trivial question. The first trapped surface formation result at the level of initial data was given 
 in \cite{SY83}. Later, \cite{yau2001geometry} provided a stronger result by proving the existence of a black hole at the level of the initial data set because of the boundary effect. In particular, this result did not require the inclusion of matter density (in fact, matter density could be allowed to be negative, even though such a case may be outside of physical interest; the latter two authors are pursuing this approach in the evolutionary framework). This, however, begs the question of whether trapped surfaces are dynamical objects, meaning whether they can be formed in evolution starting with data devoid of trapped surfaces. This problem bears high physical significance and serves as a \textit{test of reality} of black holes, in the following sense. The mathematical definition of a black hole region would be without physical meaning if it did not accurately capture what physicists perceive as black holes (this is more meaningful now than ever, as scientists recently succeeded in capturing the first-ever image of a black hole). Hence if "mathematical" black holes describe "physical" black holes, they should mathematically verify certain physical properties, one of which is dynamical formation. 

\vspace{3mm} \noindent The first results along this direction were obtained by Christodoulou for the Einstein equations coupled to a massless scalar field in spherical symmetry. Through a series of works \cite{C91}, \cite{C93}, \cite{C94}, and \cite{C99}, Christodoulou managed to not only prove trapped surface formation, but to understand the picture of gravitational collapse in its entirety for the given model and under the given symmetry. The breakthrough in the absence of symmetry came in \cite{C09} by the same author. In this work, Christodoulou introduced a hierarchy of small and large components in the initial data which (almost) persists under the evolution of the Einstein equations. He termed his method the \textit{short pulse} method. After Christodoulou, the work \cite{Kl-Rod} by Klainerman-Rodnianski reduces the size of Christodoulou's work from about 600 to approximately 120 pages, by using a slightly different hierarchy. Moreover, it reduces the number of derivatives of curvature required to prove semi-global existence from two to one. A few years later, An  \cite{AnThesis} introduces the signature for decay rates $s_2$ on his way to proving an extension of \cite{Kl-Rod} from a finite region to a region close to past null infinity. In 2014, An and Luk \cite{AL17} prove the first \textit{scale-critical} trapped surface formation criterion for the vacuum equations in the absence of symmetry. While Christodoulou's data in \cite{C09} were  large in $\dot{H}^1(\mathbb{R}^3)$, An and Luk give data which only have to be large in $\dot{H}^{\frac{3}{2}}(\mathbb{R}^3)$, which is a scale-critical norm for the initial data. Taking advantage of the scale criticality in \cite{AL17}, An \cite{A17} constructs initial data that give rise not merely to trapped surfaces, but an \textit{apparent horizon}, a smooth 3-dimensional hypersurface consisting of marginally outer trapped surfaces. In 2019, An \cite{A19} produces a 55-page proof of trapped surface formation for the vacuum equations, making use of the signature for decay rates and obtaining an existence result from a region close to past null infinity. In \cite{AnAth}, An and the first author extend \cite{A19} to the case of the Einstein-Maxwell system. Later, the present authors proved the large data scale-invariant semi-global existence and trapped surface formation result in \cite{AnAth}.

\subsection{The Einstein-Vlasov System}
First of all, we emphasize the physical importance of the Vlasov matter in general relativity. Most large-scale astrophysical structures- galaxies, galaxy clusters, dark matter halos are made of matter whose mean free path is extremely long. Particles (e.g. stars in a galaxy, or dark matter particles) interact predominantly via gravity, not collisions. Thus, the collisionless Vlasov equation is the natural model. The formation and stability of galaxies and clusters can be described via the Einstein–Vlasov (in the relativistic regime) or Vlasov–Poisson (in  the Newtonian regime) equations. Lastly and most importantly, Vlasov matter provides a \textit{clean} matter model (no artificial pressures, no exotic assumptions) that allows the study of whether self-gravitating distributions can undergo collapse, form singularities, or settle into equilibrium. It is central in rigorous mathematical relativity because:
\vspace{3mm}

a. It satisfies all standard energy conditions (dominant, strong, weak),

\vspace{3mm}

b. It avoids unphysical pathologies that occur in some fluid models,

\vspace{3mm}

c. It leads to nontrivial dynamics, all the while remaining mathematically tractable.

\vspace{3mm}

\noindent The Einstein–Vlasov system also models physically realistic systems (galaxies, dark matter). It thus constitutes a natural testing ground for the major conjectures of general relativity. It has been the subject of extensive investigation in recent decades. In the case of spherical symmetry, Rein and Rendall \cite{rendall}-\cite{rendallerratum} established the global existence of small-data solutions. Subsequently, the first breakthrough in the absence of symmetry was made by Taylor \cite{taylor} who proved the global nonlinear stability of Minkowski spacetime when coupled to massless (null) Vlasov matter. His proof employs the Sasaki metric on the mass shell together with estimates on Jacobi fields, and relies on small-data characteristic initial value theory. The argument further incorporates the vacuum stability results to conclude the global stability. Taylor’s work is of particular relevance here, since his analysis is carried out in double null gauge and yields gauge-invariant estimates for the Weyl curvature; our approach is in part motivated by this framework. Nevertheless, as we shall see, the large-data regime introduces substantial new difficulties. In subsequent work, Lindblad and Taylor \cite{TL} established the global stability of Minkowski spacetime for the massive Einstein--Vlasov system. Their analysis is conducted in spacetime harmonic coordinates, where the field equations reduce to a system of quasilinear wave equations for the metric satisfying the weak null condition, coupled to a transport equation for the particle distribution function. Simultaneously, Fajman, Joudioux, and Smulevici \cite{fajman} extended these results significantly by proving stability without assuming compact support in the velocity variable, thereby broadening the admissible class of initial data.

In the context of gravitational collapse, Dafermos and Rendall \cite{dafermos} established the first rigorous result on black hole formation from the evolution of collisionless matter, proving that the collapse of Vlasov matter from suitably regular initial data in spherical symmetry leads to the formation of black holes. This result was further improved  by Andreasson, Kunze and Rein \cite{andr1}, who showed global existence (in Schwarzschild coordinates) of initial data leading to trapped surface formation. Andreasson also exhibited that the spacetimes created in \cite{andr2} possess a complete regular past. Recently, Andreasson and Rein generalize the classical result by Oppenheimer and Snyder for collapsing dust to Vlasov matter \cite{andr3}. Further investigations of the Einstein--Vlasov system, both in asymptotically flat and cosmological settings (sometimes coupled with additional fields), can be found in \cite{fajman1, fajman2, An}. The purpose of the present work is to go beyond symmetry-restricted settings by considering an open class of large characteristic initial data prescribed on past null infinity $\mathscr{I}^{-}$. This provides the first symmetry-free large-data analysis of the Einstein--Vlasov system and, in particular, establishes a dynamical trapped surface formation result. Our approach introduces several new techniques, including the replacement of Jacobi field methods with a direct analysis of commuted transport equations. The argument crucially exploits the null structure inherent in the commuted Vlasov equations. The main theorem, together with the underlying difficulties and innovations of the proof, will be described in the subsequent sections.


\subsection{Main Theorem}

In this section, we state our main results. Our first result is a semi-global existence result for the Einstein-Vlasov system for large data. The second, complementing the first, is a formation of trapped surfaces statement.

\begin{theorem}[Semi--Global Existence and Trapped Surface Formation]
\label{mainone}
Fix a smooth double null foliation of a four--dimensional Lorentzian manifold $(\mathcal{M},g)$ with optical functions $(u,\ubar)$ such that the level sets $H_u$ and $H_{\ubar}$ are outgoing and incoming null hypersurfaces, respectively, and
\[
S_{u,\ubar} := H_u \cap  H_{\ubar}
\]
denotes the corresponding two--sphere sections equipped with the induced metric $\slashed{g}$. Let $\snabla$ and $\snabla_4$ denote the angular and null derivatives associated with the null frame
\[
\{ e_1, e_2, e_3,e_4\},
\]
and let $f$ be a smooth nonnegative distribution function on the future mass--shell $\mathcal{P}\subset T\mathcal{M}$ satisfying the Vlasov equation
\[
X[f] = 0,
\]
where $X$ is the geodesic spray.

\vspace{3mm}

\noindent\textbf{(a) Semi--Global Existence.}
For every fixed constant $\mathcal{I}>0$, there exists $a_0=a_0(\mathcal{I})\gg 1$ such that the following holds.  
Let $a>a_0$ and let $N\in\mathbb{N}$ be sufficiently large. Suppose that a smooth initial data set
\[
(\hat{\chi}_0, f_0)
\]
is prescribed along the characteristic hypersurfaces $H_{u_\infty}$ and $\Hbar_{0}$, where $u_\infty<0$ is large in magnitude, satisfying:

\item[\textup{(i)}] (\textit{Shear bound along $u=u_\infty$})
\[
\sum_{0\le i+m\le N} a^{-\frac{1}{2}}
  \Big\| \snabla_4^m \big( |u_\infty| \snabla \big)^i \hat{\chi}_0 \Big\|_{L^\infty(S_{u_\infty,\ubar})}
  \le \mathcal{I},
\]
for all $\ubar\in[0,1]$.

\item[\textup{(ii)}] (\textit{Minkowskian incoming data})
All geometric data along $\ubar=0$ coincide with the corresponding Minkowskian values and
and the Vlasov distribution vanishes: $f|_{\ubar=0}=0$.

\item[\textup{(iii)}] (\textit{Momentum support and regularity of $f_0$})
The function $f_0$ is smooth on the restriction $\mathcal{P}|_{u=u_\infty}$ and obeys the uniform bound
\[
\frac{1}{a} \sup_{u=u_\infty} \sum_{k=0}^{3} 
\sum_{i_1,\dots,i_k=1}^{7} \big| F_{i_1}\cdots F_{i_k} f_0 \big|
\le \mathcal{I},
\]
where $\{F_1,\dots,F_7\}$ denotes the orthonormal frame on $\mathcal{P}$ defined in~\eqref{eq:frame}.  
Moreover, its momentum support satisfies for fixed constants $C_{p^A},C_{p^3},C_{p^4}>0$ independent of $u_\infty$,
\begin{equation}\label{eq:momentum_bounds}
0 < p^3 \le C_{p^3}, \qquad
0 \le |u|^2 p^4 \le C_{p^4} p^3, \qquad
|u|^2 |p^A| \le C_{p^A} p^3, \quad A=1,2,
\end{equation}
for all $(x,p)\in\supp(f_0)\cap \mathcal{P}_{u=u_\infty}$.

\noindent Then there exists a unique smooth solution $(g,f)$ of the Einstein--Vlasov system in the region
\[
\mathcal{D} := \big\{ (u,\ubar) \mid u_\infty \le u \le -\tfrac{a}{4},\;\; 0\le \ubar \le 1 \big\},
\]
satisfying the constraint and transport equations, with all geometric quantities $\psi\in\{\tr\chi,\hat{\chi},\chibarhat,\tr\chibar,\eta,\etabar,\alpha,\betabar,\rho,\sigma\}$ remaining uniformly controlled by constants depending only on $(\mathcal{I},C_{p^A},C_{p^3},C_{p^4})$.

\vspace{1em}
\noindent\textbf{(b) Formation of a Trapped Surface.}
Suppose, in addition to \textup{(i)--(iii)}, that the initial energy flux along $C_{u_\infty}$ satisfies the uniform lower bound
\begin{equation}\label{eq:trapped_condition}
\int_0^1 \Big( |u_\infty|^2 |\hat{\chi}_0|^2 + |u_\infty|^{2}\, T_{44}(u_\infty,\ubar') \Big)\,\mathrm{d}\ubar'
\ge a,
\end{equation}
uniformly for every direction along $u=u_{\infty}$, where $T_{44}$ denotes the energy–momentum density component associated with the Vlasov field. Then the outgoing null hypersurface $H_{-a/4}$ possesses a closed two–sphere
\[
S_{-a/4,\,1} = H_{-a/4}\cap \Hbar_1
\]
such that
\[
\tr\chi(S_{-a/4,\,1})<0 \quad\text{and}\quad \tr\chibar(S_{-a/4,\,1})<0.
\]
Hence $S_{-a/4,\,1}$ is a trapped surface in $(\mathcal{M},g)$.
\end{theorem}

\begin{remark}
Notice that one can impose a weaker condition on the Vlasov data in terms of the integral over the null hypersurface $H_{u_{\infty}}$ instead of a point-wise control, since that is exactly what we need in the ultimate estimates for the Vlasov equations. Nevertheless, the stronger assumption on the Vlasov initial data automatically implies the weaker integral bound. In addition the initial condition  \[0< p^3\leq C_{p^3}, \hspace{2mm} 0\leq \modu^2 p^4 \leq C_{p^4} p^3, \hspace{2mm} \lvert u^2 p^A\rvert \leq C_{p^A}p^3, \hspace{2mm} \text{for}\hsp A=1,2,\] in $\restri{\supp f}{\restri{\mathcal{P}}{u=u_\infty}}$, for some fixed constants $C_{p_1},\dots, C_{p^4}$ independent of $u_{\infty}$ is precisely the \textbf{focusing} condition for the Vlasov matter.    
\end{remark}

\noindent Importantly, we address the issue of the existence of initial data that satisfy the conditions of Theorem \ref{mainone}. Based on the conformal methods developed in \cite{C09} and also inspired by \cite{LukNull}, we prescribe a seed datum $\Psi$ together with an initial Vlasov seed function $F$ which, together, uniquely prescribe the entirety of initial data, in a manner consistent with the requirements of Theorem 1.2. It therefore follows that the initial data set we impose is not the null set, but rather an open set in the moduli space of initial data. The construction itself is found in Section \ref{initialdatasection}.

\noindent A further interesting property we obtain along the way is that we are able to restrain the spread, in a sense, of the Vlasov matter upon evolution. This is the content of the following Lemma:

\begin{lemma} \label{spreadlemma}
Assume a characteristic initial value problem with initial data as given in Theorem 
\ref{mainone}. Assume that $\pi(\restri{\supp f}{\restri{\mathcal{P}}{u=u_\infty}})
\subset 
\{ (u,\ubar) \hsp \mid u= u_{\infty},\hspace{1mm} 0\leq \ubar \leq \ubar_0 <1, \}$, for some constant $\ubar_0$ in $(0,1)$. Then, for $a_0$ large enough and for the spacetime obtained in Theorem \ref{mainone},  there exists a uniform constant $
\tilde{\ubar}_0$ in $(0,1)$, depending only on $
\ubar_0$, such that \[  \pi(\restri{\supp f}{\mathcal{P}}) \subset \big\{ (u,\ubar) \hsp \mid u_{\infty}\leq u \leq -\frac{a}{4}, 0\leq \ubar \leq \tilde{\ubar}_0\big. \}    \]Here by $\pi$ we denote the natural projection $T\mathcal{M}\to \mathcal{M}$.
\end{lemma}
\begin{figure}
\begin{center}
\begin{tikzpicture}
\draw [white](3,-1)-- node[midway, sloped, below,black]{$H_{u_{\infty}}(u=u_{\infty})$}(4,0);
\draw [white](2,2)--node [midway,sloped,above,black] {$\mathcal{H}_1(\bar{u}=1)$}(4,0);
\draw [white](1,1)--node [midway,sloped, below,black] {$\mathcal{H}_{0}(\bar{u}=0)$}(3,-1);
\draw [dashed] (0, 4)--(0, -4);
\draw [dashed] (0, -4)--(4,0)--(0,4);
\draw [dashed] (0,0)--(2,2);
\draw [dashed] (0,-4)--(2,-2);
\draw [dashed] (0,2)--(3,-1);

\fill [black!15!white] (1,1) -- (2,2) -- (1.78,1.7) 
    .. controls (2.5, 1.1) and (3.3, 0.5) .. (3.6, -0.4) 
    -- (3,-1) -- (1,1);

\draw [very thick] (1,1)--(3,-1)--(4,0)--(2,2)--(1,1);
\draw [white](1,1)-- node[midway,sloped,above,black]{$H_{u}$}(2,2);

\draw [thick, ->] (3.6, -0.4)
.. controls (3.3, 0.5) and (2.5, 1.1) ..
(1.78, 1.7);
\end{tikzpicture} 
\caption{Schematic depiction of the ultimate behaviour of Vlasov matter. The shaded region corresponds to the projection of the support of $f$ to the spacetime. Along the incoming geodesic, $\ubar$ always increases, but negligibly so, ultimately. }
\end{center}
\end{figure}\vspace{3mm}

\noindent The above lemma is shown in Section \ref{subsectiondecaymomentum} and essentially dictates that initial data follow an almost-null incoming path-- and that there is no leakage of Vlasov radiation towards future null infinity. This is shown in the picture above.

\subsection{Summary of the Novelty and Impact of the Results}
We briefly summarize the novel aspects and impact of our work.
\vspace{3mm}

\noindent\textbf{A new commutator calculus on the mass shell.}
A central innovation of the paper is the introduction of an adapted family of commutation vector fields on the mass shell, tailored to the null geometry and to the geodesic spray operator. We prove that, when the transport operator is commuted with these vector fields, derivatives of the connection coefficients arise only through antisymmetric combinations that can be rewritten in terms of curvature and lower-order quadratic expressions. As a consequence, top-order commutators exhibit a curvature-type structure and avoid uncontrolled derivatives of Ricci coefficients. Whether such an approach is feasible was raised as a question in \cite{taylor}. This restores both integrability and regularity at the highest derivative level and replaces the classical Jacobi-field–based approach by a direct commuted-transport framework compatible with the null structure of the Einstein equations. By replacing Jacobi–field–based phase–space analysis with a direct commuted–transport approach adapted to null geometry, the paper provides a technically simpler and more flexible high–order framework. This approach is likely to be useful in stability and asymptotic problems for massless kinetic matter where characteristic foliations are natural.

\medskip
\noindent\textbf{Reduced elliptic scheme adapted to kinetic coupling.}
The analysis identifies a structural obstruction to the use of full elliptic estimates for all null connection coefficients in the Einstein--Vlasov setting: such an approach would force top-order control of matter terms that cannot be propagated within the available transport regularity. We overcome this difficulty by developing a reduced elliptic scheme in which elliptic estimates are applied only to a minimal subset of geometric quantities on selected null hypersurfaces, together with renormalized curvature–connection combinations whose evolution equations eliminate the dangerous top-order source terms. This selective elliptic strategy is specific to kinetic coupling and constitutes a structural reorganization of the null geometry suitable for phase-space matter models.

\medskip
\noindent\textbf{Optimal regularity for the coupled hierarchy.}
We need to be very careful about what we mean by regularity here. We do not mean regularity of initial data, which are in general considered to be smooth (naturally, it is a well-known property of the Einstein equations that an imposition of $C^k$ data for $k$ large enough would hint the existence of a spacetime with slightly less $C^j-$regularity. This phenomenon is called persistence of regularity). Instead, here by regularity we mean the question of how many derivatives are needed to close the energy estimates and obtain the semi-global existence theorem. In most matter models, the optimal number of derivatives of curvature required is $1$ (see for example \cite{Kl-Rod} or \cite{M-Y}). Instead, here, we show that two derivatives of curvature are required-- and that any attempt to lower this would require a small breakthrough. The proof isolates the precise regularity balance between curvature, Ricci coefficients, and velocity averages of derivatives of the distribution function that is necessary to close the nonlinear hierarchy. In particular, the argument shows how top-order geometric control can be achieved while requiring only a minimal, directionally distributed set of phase-space derivatives of the Vlasov field, propagated in norms that are compatible with transport along null hypersurfaces. This provides a sharp regularity framework for large-data problems in Einstein–kinetic systems.

\medskip
\noindent\textbf{A sharp trapped-surface criterion with kinetic flux.}
The trapped-surface formation mechanism is driven by a lower bound on an incoming null flux that combines shear energy and Vlasov stress-energy density in a scale-consistent way. This yields a clean quantitative criterion for dynamical trapping expressed directly in terms of characteristic initial data and extends the short-pulse paradigm to kinetic matter. What would be expected of the matter model is obtained: Our Main Theorem implies in addition that the focusing alone of massless Vlasov matter in a spherically symmetric spacetime would also lead to a trapped surface (compare this to the Andr\'easson--Kunze-Rein results \cite{andr1}-\cite{andr3} in the presence of spherical symmetry) . 

\medskip
\noindent
\justifying Together, these contributions introduce new structural tools—commuted mass-shell transport, curvature-type commutator identities, and reduced elliptic control—that are expected to be applicable beyond the present theorem, including in stability and long-time problems for Einstein equations coupled to kinetic matter models.

\subsection{Global estimates for the derivatives of \texorpdfstring{$f$}{TEXT} without the use of Jacobi fields}

The first main contribution of this work concerns the delicate interplay between the regularity of the matter variables and the curvature components, as dictated by the contracted Bianchi identities. To set notation, let $\Psi$ denote the collection of null components of the Weyl curvature tensor relative to a double null frame, and let $\mathcal{T}$ denote the corresponding null components of the energy--momentum tensor induced by the Vlasov field. As is well known from the structure of the Bianchi equations, the source terms on the right-hand side involve derivatives of $\mathcal{T}$. Consequently, a priori estimates for $\Psi$ necessitate control of one additional derivative of $\mathcal{T}$. 

\vspace{3mm}

\noindent In the present framework, our analysis requires two derivatives of the curvature $\Psi$ in order to close the nonlinear energy hierarchy. It therefore follows that one must establish uniform control of three derivatives of the Vlasov distribution function $f$. This forms the starting point of our regularity analysis.

\vspace{3mm}

\noindent We now recall the governing transport equation for the Vlasov distribution function. Let $(\mathcal{M},g)$ be a globally hyperbolic Lorentzian manifold, and denote by $\mathcal{P}\subset T\mathcal{M}$ the mass shell corresponding to massive or massless particles. The Vlasov distribution function is a nonnegative scalar function
\[
f : T\mathcal{M} \to \mathbb{R}_{\geq 0},
\]
assumed to be sufficiently regular and compactly supported in the momentum variables. It is transported along the geodesic flow, that is, it is invariant under the Liouville vector field $X$. Explicitly,
\begin{equation}\label{eq:vlasov_transport}
X[f] = 0.
\end{equation}

\noindent Let $\{e_4,e_3,e_1,e_2\}$ be a double null frame on $M$, with $e_4,e_3$ null and normalized so that $g(e_3,e_4) = -2$. Extend this frame by adjoining the natural coordinate derivatives in the momentum variables, $\{\partial_{p^4}, \partial_{p^3}, \partial_{p^1}, \partial_{p^2}\}$. Relative to this frame, the Liouville operator takes the form
\begin{equation}\label{eq:liouville}
X = p^\mu e_\mu - \Gamma^{\mu}_{\nu\lambda} p^\nu p^\lambda \partial_{p^\mu},
\end{equation}
where $p^\mu$ are the fiber coordinates on the tangent bundle, constrained by the mass shell relation $g_{\mu\nu} p^\mu p^\nu = -m^2$ (with $m=0$ in the massless case-which is the case here), and $\Gamma^{\mu}_{\nu\lambda}$ are the Christoffel symbols of $g$. Accordingly, the transport equation \eqref{eq:vlasov_transport} reads in components:
\begin{equation}\label{eq:transport_components}
p^\mu e_\mu(f) - \Gamma^{\mu}_{\nu\lambda} p^\nu p^\lambda \, \partial_{p^\mu} f = 0.
\end{equation}

\noindent Equation \eqref{eq:transport_components} is the precise analytic formulation of the free-streaming of Vlasov particles in curved spacetime. Its hyperbolic--transport structure is the fundamental mechanism by which derivatives of $f$ interact with the connection coefficients, and thereby enter into the higher-order commutator estimates necessary for controlling the derivatives of the curvature tensor.
 In order to obtain control of higher derivatives of the distribution function $f$, a natural idea is to commute the transport equation \eqref{eq:vlasov_transport} with carefully chosen vector fields adapted to the null geometry and then integrate along the geodesic flow. Concretely, suppose $V$ is such a commutation vector field. Then the commuted equation reads
\[
X[Vf] =  [X,V]f.
\]
The error term $[X,V]f$ contains contributions of the form $V[\Gamma^{\mu}_{\nu\lambda}]$, which, at the heuristic level, lie at the same regularity as the Weyl curvature components. However, as is already emphasized in \cite{taylor}, certain components of these commutators present significant technical difficulties. 

\noindent For instance, terms of the form $V[\Gamma^{A}_{3B}]$ involve derivatives of the shift vector $b$, since
\begin{equation}\label{eq:gamma_shift}
\Gamma^{A}_{3B} = \chibar^{A}_{\;\;B} - e_A(b^B).
\end{equation}
Thus, $V[\Gamma^{A}_{3B}]$ generates terms with two derivatives of $b$. To estimate $b$, one relies on its transport equation,
\begin{equation}\label{eq:shift_transport}
\snabla_4 b \sim \eta - \etabar + \chi \cdot b,
\end{equation}
which couples $b$ to the torsion coefficients $\eta, \etabar$ and the null second fundamental form $\chi$.

\noindent As discussed, closing the energy estimates requires control of two derivatives of the Weyl curvature, which in turn necessitates control of three derivatives of $f$. Thus, one is led to commute the Vlasov transport equation \eqref{eq:vlasov_transport} three times with $V$, schematically:
\[
X[V^3 f] =  [X,V^3] f+.......
\]
This produces error terms of the type $V^3(\Gamma^{A}_{3B})$, which, by \eqref{eq:gamma_shift}, involve \emph{four derivatives} of the shift $b$. However, the transport equation \eqref{eq:shift_transport}, together with available elliptic estimates on the connection coefficients, yields control of at most \emph{three} derivatives. Hence a derivative gap of one order appears, obstructing the closure of the argument. 

\noindent An analogous difficulty arises in controlling terms such as $V^3(\Gammaslash^{A}_{BC})$, where $\Gammaslash$ denotes the Levi--Civita connection coefficients of the induced metric on the $S_{u,\ubar}$ spheres. These coefficients obey an equation of the schematic form
\begin{equation}\label{eq:spherical_connection}
\snabla_4 \Gammaslash \sim \nablasl \chi + \Gammaslash \cdot \Gamma,
\end{equation}
so that $V^3(\Gammaslash)$ also requires control of four derivatives of $\chi$, which again exceeds the regularity threshold permitted in our framework.

\noindent Taylor \cite{taylor} circumvented this derivative loss by avoiding direct commutation altogether. Instead, he introduced Jacobi fields associated with the geodesic flow, which permit control of higher derivatives of $f$ without producing uncontrolled higher derivatives of the connection (note that the Jacobi fields are more fundamental than the shift vector $b$). One fundamental structural distinction between the stability problem and the problem of trapped surface formation concerns the analytical role played by Jacobi fields along the mass shell. In the global small-data regime considered by Taylor, the use of Jacobi fields is particularly well-suited to capturing the long-time dispersive decay of the Vlasov field. By contrast, in the regime of large, concentrated initial data---precisely the regime relevant to the formation of trapped surfaces-the same framework breaks down, essentially because the curvature and shear amplify rather than decay.

\noindent To recall, let $P \subset T\mathcal{M}$ denote the future mass shell of the underlying spacetime $(\mathcal{M},g)$, equipped with the Sasaki metric. The geodesic spray on $P$ generates the characteristic flow of the Vlasov equation. If $X$ denotes the horizontal lift of the geodesic vector field, then a Jacobi field $J$ along the flow satisfies the second-order variational equation
\begin{align}
\label{eq:jacobi}
 \hnabla_{X}\hnabla_{X}J \;=\; \widehat{R}(X,J)X,    
\end{align}
where $\hnabla$ is the Levi--Civita connection induced by the Sasaki metric and $\widehat{R}$ its associated curvature tensor. The Jacobi fields thus encode the infinitesimal linearization of the geodesic flow on phase space and constitute the principal tool for propagating regularity of the distribution function $f$.

\noindent In the small-data regime of perturbations of Minkowski spacetime, the ambient curvature remains uniformly weak. Consequently, the right-hand side of the Jacobi equation acts only as a perturbative source term, and Jacobi fields exhibit at most mild growth along the flow. Taylor exploits this boundedness to propagate uniform estimates for the derivatives of $f$. This, in turn, yields the dispersive decay of the matter terms in the energy--momentum tensor, which is the cornerstone of the global stability analysis.

\noindent Our method introduces a completely different analytic framework for the Einstein--Vlasov system. It not only streamlines the derivation of higher-order estimates but, more importantly, provides a systematic avenue for addressing the large-data regime. This encompasses not only the problem of trapped surface formation but also a broader class of dynamical phenomena that one could hope to study in the large data regime. The central innovation lies in the identification of an adapted system of commutation vector fields and the concomitant development of a commutator calculus intrinsically suited to the null geometric structure of the problem. A key observation is that the commutators of the form  
\[
[X,V]
\]  
generate derivatives of the connection coefficients exclusively through antisymmetric combinations. Concretely, the fundamental identity reads  
\begin{equation}\label{eq:curvature_identity}
e_\mu(\Gamma^{\alpha}_{\beta\nu}) - e_\beta(\Gamma^{\alpha}_{\mu\nu}) 
= R^{\alpha}_{\;\;\nu\mu\beta} 
- \Gamma^{\lambda}_{\beta\nu}\Gamma^{\alpha}_{\mu\lambda} 
+ \Gamma^{\lambda}_{\mu\nu}\Gamma^{\alpha}_{\beta\lambda} 
+ \bigl(\Gamma^{\lambda}_{\mu\beta} - \Gamma^{\lambda}_{\beta\mu}\bigr)\Gamma^{\alpha}_{\lambda\nu},
\end{equation}
so that derivatives of the connection never appear in isolation, but only in the curvature combination \eqref{eq:curvature_identity}, together with quadratic lower-order terms. Let us provide a concrete example of how the issue of both regularity and integrability is miraculously restored. Consider, for example, $V=V_{(A)}= \Hor(e_A)-\frac{p^3}{\modu}\partial_{\pbar^A}$ (see section \ref{vlasov} for the definition of the commuting $V$ vector fields) and observe the following in the standard basis (see section \ref{vlasov} for the exact definition):
\begin{align} [X, Hor(e_A)-\frac{p^3}{\modu}\partial_{\pbar^A}]^{4+B}&= -X(\Gamma^B_{A\nu}p^{\nu}) +Hor(e_A)(\Gamma^B_{\mu\nu}p^{\mu}p^{\nu}) -X \left(\frac{p^3}{\lvert u \rvert}{\delta_A}^B\right) -\frac{p^3}{\modu}\partial_{\pbar^A}(\Gamma^B_{\mu\nu}p^{\mu}p^{\nu}) \notag \\ &= \underbrace{(e_A(\Gamma^B_{\mu\nu})-e_{\mu}(\Gamma^B_{A\nu}))\hsp p^{\mu}p^{\nu}}_\text{I}+\Gamma^C_{\mu\nu} \hsp p^{\mu}\hsp p^{\nu} \partial_{\pbar^C}(p^\lambda)\Gamma^B_{A\lambda} + \Gamma^3_{\mu\nu}p^{\mu}p^{\nu} \partial_{\pbar^3}(p^\lambda) \Gamma^B_{A\lambda} \notag \\ &\underbrace{-\Gamma^C_{A\lambda}p^{\lambda}\Gamma^B_{\mu\nu} \partial_{\pbar^C}(p^{\mu}p^{\nu})}_\text{II}-\Gamma^3_{A\lambda}p^{\lambda}\Gamma^B_{\mu\nu}\partial_{\pbar^3}(p^{\mu}p^{\nu}) +\left(\underbrace{-\frac{(p^3)^2}{\modu^2}{\delta_A}^B}_\text{I} +\frac{{\delta_A}^B}{\modu}\Gamma^3_{\mu\nu}p^{\mu}p^{\nu}\right)\notag \\&\underbrace{-\frac{p^3}{\modu} \partial_{\pbar^A}(p^{\mu}p^{\nu})\Gamma^B_{\mu\nu}}_\text{II}.\end{align}Each pair of terms in brackets labelled (I) and (II) leads to a cancellation that improves the order of decay. Indeed, the first term in bracket (II), is $o(\modu^{-3})-{\chibar_A}^C\hsp p^3 \hsp \Gamma^B_{\mu\nu}\hsp \partial_{\pbar^C}(p^{\mu}p^{\nu}) = -\frac{1}{2}\tr\chibar\hsp p^3\hsp \Gamma^B_{\mu\nu}\hsp \partial_{\pbar^A}({p^{\mu}p^{\nu}})+o(\modu^{-3})$. By itself, this term is only $o(\modu^{-2})$. When combined, however, with the second term in bracket $(II)$, the whole bracket becomes
 \[   (II) =   -\frac{1}{2}p^3 \hsp \big(\tr\chibar+\frac{2}{\modu}\big)\Gamma^B_{\mu\nu}\partial_{\pbar^A}(p^{\mu}p^{\nu})+ o(\modu^{-3})=o(\modu^{-3}).            \]As for the terms in bracket (I), it is here where not only the issue of integrability, but also the issue of regularity in the terms is overcome. Indeed, notice that

\begin{equation}\label{fundid}
     e_A(\Gamma^B_{\mu\nu})-e_{\mu}(\Gamma^B_{A\nu}) = {R^B}_{\nu A \mu}- \Gamma^{\epsilon}_{\mu\nu}\Gamma^{B}_{A\epsilon} +\Gamma^{\epsilon}_{A\nu}\Gamma^B_{\mu \epsilon} +(\Gamma^{\epsilon}_{A\mu}-\Gamma^{\epsilon}_{\mu A})\Gamma^B_{\epsilon \nu}.
\end{equation}In other words, throughout the entire argument, \textit{derivatives of Christoffel symbols will always appear in differences that give rise to Riemann curvature components (up to terms of the form $\Gamma\Gamma$, which are not only harmless, but indispensable as we shall briefly explain! )}.

\vspace{3mm}
\par\noindent As a consequence, the first term in $(I)$ is $o(\modu^{-3}) + \Gamma^{\epsilon}_{A\nu}\Gamma^B_{\mu\epsilon}p^{\mu}p^{\nu}$, where the latter fails to be $o(\modu^{-3})$ precisely when $\mu=\nu=3$. In that case, the worst term is \[   \Gamma^C_{A3}\Gamma^{B}_{3C} p^3 p^3 \hspace{1mm} \text{(summation over C=1,2)} =\frac{1}{4}(\tr\chibar)^2 (p^3)^2 =o(\modu^{-2}).   \]When combined, however, with the second term in bracket (I), the whole bracket becomes \[ (I) = \frac{1}{4} (p^3)^2 \big( \tr\chibar + \frac{2}{\modu} \big) \big(\tr\chibar - \frac{2}{\modu }\big){\delta_A}^B +o(\modu^{-3}) =o(\modu^{-3}).       \]
\noindent It thus becomes apparent that \textit{the error terms arising in \eqref{fundid} are neither negligible nor useless, but rather provide the missing link for ensuring the overall required integrability of the commutator components.} This property is consistent, for all commutator components, throughout all seven basis vector fields $V$ introduced later on in Section 9.

\vspace{3mm}

\noindent Another important observation is that, \textit{in this framework, the quantities $e_A(b^B)$ never arise independently, but always in the coupled form with $\chibar^{A}_{B}$, thereby combining into the connection coefficient $\Gamma^{A}_{3B}$ associated with the endomorphism structure of the tangent bundle of the leaves $S_{u,\ubar}$}. This structural cancellation is decisive: once the commutation vector fields $V$ are aligned with the null frame, the apparent derivative loss is eliminated and a significant reduction in the associated elliptic estimates is achieved. This is novel, and we discuss this in the subsequent sections.  

\noindent As a consequence, curvature terms can be controlled precisely at the target order, while the quadratic connection contributions, being effectively of lower differential order, enjoy stronger decay. These improved decay properties ensure that no obstruction to integrability emerges in the higher-order energy estimates.  

\noindent The development of this commutator calculus in the adapted $V$-basis (introduced systematically in Chapter~\ref{vlasov}) constitutes one of the main novelties of our method. Beyond its immediate application to the Einstein-Vlasov system in regimes of strong gravitational interaction and absence of symmetry, we anticipate that this framework will prove useful in a broader context, including, for instance, simplified approaches to the proof of nonlinear stability of Minkowski spacetime with massless Vlasov matter.

\noindent To summarize, our new approach yields a novel  framework in the large data and symmetry-free regime (which is also conceptually simple) that avoids derivative losses and provides additional technical advantages: a reduction in the reliance on elliptic estimates, sharp control at the optimal regularity level (in the sense we explained), and greater transparency in the energy hierarchy. We will exploit these features systematically in the sections that follow. 

\subsection{Fundamental analytic obstruction in the characteristic treatment of the Vlasov system in the absence of symmetries}

Beyond the high–order regularity issues discussed in the preceding subsection, a principal analytic obstruction in the present characteristic formulation of the Einstein–Vlasov system is the lack of \emph{top–order} control of the Vlasov field along the incoming null hypersurface. In the framework under consideration, the characteristic initial value problem is posed with nontrivial data prescribed on a past outgoing null hypersurface $\mathcal{H}_{u_{\infty}}$ asymptotic to past null infinity, and \emph{trivial} geometric data induced from Minkowski space on the intersecting past incoming null hypersurface $\underline{\mathcal{H}}_{\ubar_{\infty}}$.

\noindent Unlike the pure vacuum setting or with other matter models such as Einstein-Maxwell, Einstein-Yang-Mills, or Einstein-Scalar system, where one can avoid the need to estimate all derivatives (derivatives tangential to the spheres as well as null derivatives) of the Ricci coefficients, the Weyl curvature components, and the stress-energy tensors if the regularity level is assumed to be sufficiently high. This is different in the case of the Vlasov matter. Heuristically speaking, the Vlasov stress-energy tensor 
\begin{align}
 T_{\mu\nu}:=\int_{0}^{\infty} \!\! \int_{\mathbb{R}^2}
     f(x,p)\, p_{\mu}p_{\nu}\;
     \frac{1}{p^{3}}\,\sqrt{\det\slashed{g}}\,
     dp^{1}dp^{2}dp^{3}.  
\end{align}
is a non-local entity in the tangent bundle $T\mathcal{M}$ of the physical spacetime $\mathcal{M}$. In particular, the distribution function $f$ is a function on the mass-shell $P\subset T\mathcal{M}$, a seven-dimensional variety of the tangent bundle. Therefore, in order to obtain higher order estimates for $f$, one needs to devise the vector fields $V$ defined on the mass shell or a section of the double tangent bundle $TT\mathcal{M}$. Any appropriate vector field that verifies the regularity closure property, as illustrated in the previous section, would contain every derivative in its projection onto $T\mathcal{M}$. Now, in the higher order estimates ($\geq 2$) for the Vlasov distribution function $f$ through commuting the transport equation 
\begin{align}
X[f]=0,
\end{align}
as we shall observe in section \ref{vlasov}, would produce derivatives of the connection coefficients and Weyl curvature components along every direction $e_{A},e_{4},e_{3}$. This necessitates the control of every derivative of the Ricci coefficients and the Weyl curvature components, irrespective of the regularity level under consideration. This complicates the analysis and requires careful renormalization of Weyl curvature and Ricci coefficients at appropriate stages. Moreover, this necessitates treating $|u|$ as a weighted function with signature $s_{2}(|u|)=-1$ and $s_{2}(|u|^{-1})=+1$. As we shall explain in later sections, this causes an apparent obstruction at the top order due to the fact that the vlasov source can only be controlled on $H$ at the top order in this framework.      

\noindent Since the Vlasov equation is a transport equation along null geodesics in the mass shell, it admits integration along incoming null geodesics provided one has a suitable control on the relation between the affine parameter $s$ and the optical function $u$. Concretely, if $\gamma(s)$ denotes the projection to the spacetime manifold of the null geodesic in the mass–shell generated by the geodesic spray $X$, then
\begin{equation}\label{eq:support_annals}
  \frac{p^{3}(s)}{\frac{du}{ds}} =1,
\end{equation}
where $u(s) := u(\gamma(s))$ and $(p^{4},p^{3},p^{1},p^{2})$ denote the components of the momentum in the chosen double–null frame $\{e_{1},e_{2},e_{3},e_{4}\}$, with $e_{3}$ tangent to $\underline{\mathcal{H}}_{\ubar}$ and $e_{4}$ tangent to $\mathcal{H}_{u}$.

\noindent In this gauge, the zero–th order bootstrap bounds for the momentum components take the form
\begin{align}
\label{eq:bootstrap_p3}
 & \frac{1}{4} \, p^3(0) \ \leq \ p^3(s) \ \leq \ 2 \, p^3(0), \\
\label{eq:bootstrap_p4}
 & 0 \ \leq \ |u(s)|^{2} \, p^4(s) \ \leq \ 2 \, C_{p^4} \, p^3(0), \\
\label{eq:bootstrap_pA}
 & |u(s)|^{2} \, |p^A(s)| \ \leq \ 2 \, C_{p^A} \, p^3(0),
\end{align}
for constants $C_{p^4},C_{p^A}>0$ fixed by the initial data, where indices $A \in \{1,2\}$ refer to the angular momentum components if the initial momentum component $p^{3}(0)$ satisfies
\begin{equation}\label{eq:momentum_support}
  C_{1} \ < \ p^{3}(0) \ \leq \ C_{2},~0\leq |u_{\infty}|^{2}p^{4}(u_{\infty})\leq Cp^{3}(0),~|u_{\infty}|^{2} \, |p^A(u_{\infty})| \ \leq \ 2 \, C_{p^A} \, p^3(0)
\end{equation}
for constants $0 < C_{1} < C_{2} < \infty$. 
\vspace{3mm}

\noindent The Vlasov equation in geometric form reads
\begin{equation}
 X[f] = 0,
\end{equation}
where $X$ is the geodesic spray vector field on the mass–shell. Bounds \eqref{eq:support_annals}–\eqref{eq:bootstrap_pA} imply that \eqref{eq:vlasov_transport} admits direct integration only along the \emph{incoming} ($u$–direction) null geodesics. However, in the coupled Einstein–Vlasov system, the energy–momentum tensor
\[
T_{\alpha\beta}(x) = \int_{\mathcal{P}_{x}} f(x,p) \, p_{\alpha} \, p_{\beta} \, \mathrm{d}\mu_{\mathcal{P}_{x}}(p)
\]
enters as a source term in the transport equations for the Ricci coefficients. Estimating these Ricci coefficients therefore requires integration in \emph{both} the $u$– and $\ubar$–directions.
\vspace{3mm}

\noindent For instance, the Raychaudhuri equation for $\tr\chibar$ takes the form
\begin{equation}\label{eq:raychaudhuri_trchibar}
 \snabla_{3} \tr\chibar + \frac12 \, (\tr\chibar)^{2}
 = - |\chihat|^{2} - 2\, \omegabar \, \tr\chibar - T_{33}.
\end{equation}
To control $\|\snabla^{2} \tr\chibar\|_{L^{2}(S_{u,\ubar})}$, one must commute \eqref{eq:raychaudhuri_trchibar} with angular derivatives and thus estimate $\snabla^{2} T_{33}$ in $L^{2}(\Hbar_{\ubar})$. By the definition of $T_{33}$ and \eqref{eq:vlasov_transport}, this necessitates $L^{2}(\Hbar_{\ubar}) L^{2}(\mathcal{P}_{x})$–control of two derivatives of $f$ on $\underline{\mathcal{H}}_{\ubar}$, whereas \eqref{eq:support_annals} only permits direct control of such norms along $\mathcal{H}_{u}$. 

\noindent Consequently, a naive approach fails: top–order angular derivatives of $T_{33}$ cannot be propagated along $\underline{\mathcal{H}}_{\ubar}$ using only \eqref{eq:vlasov_transport}. The resolution lies in identifying the \emph{correct functional norms} of $f$ to be propagated, and determining \emph{which submanifolds of the mass–shell} those norms should be estimated on. In the present setting—where control of two derivatives of the Weyl curvature is required—one must control \emph{three} derivatives of $f$. Crucially, the derivatives of $f$ appear in different norms over different hypersurfaces:
Controlling $\snabla^{2} \Te_{33}$ in $L^{2}(S_{u,\ubar})$ requires $f \in L^{2}(S_{u,\ubar})L^{2}(\mathcal{P}_{x})$.
To bound $V^{2} f$ (with $V$ one of the commutation vector fields) in $L^{2}(S_{u,\ubar})L^{2}(\mathcal{P}_{x})$, one needs uniform $L^{2}(\mathcal{P}_{x})$ control of $V f$ over $S_{u,\ubar}$.
This in turn implies the need to control $V^{3}f$ in $L^{2}(\mathcal{H}_{u})L^{2}(\mathcal{P}_{x})$, a norm that \emph{is} directly accessible from \eqref{eq:vlasov_transport}.

\noindent In this way, different derivative orders of $f$ are controlled on distinct families of null hypersurfaces in the mass–shell, and the principal technical challenge is to propagate these norms in a manner compatible with the geometry of the characteristic initial value problem.

\noindent An additional obstruction appears at the elliptic estimate stage: In the Einstein equations, elliptic estimates on $S_{u,\ubar}$ are employed to control top–order angular derivatives of the Ricci coefficients. For ``good'' Ricci coefficients, this is compatible with the Vlasov propagation, but for ``bad'' coefficients such as $\tr\chibar$, elliptic estimates would require $L^{2}(\underline{\mathcal{H}}_{\ubar})L^{2}(\mathcal{P}_{x})$–control of $V^{3}f$, which is not directly available from \eqref{eq:vlasov_transport}. 

\noindent Our approach circumvents this difficulty by a refined commutator analysis: by carefully commuting \eqref{eq:vlasov_transport} with appropriately chosen $V$, and tracking the precise structure of the resulting error terms, we eliminate the need for elliptic estimates for $\tr\chibar$ (and other \textit{bad} ones) at top order. The analysis reveals an intrinsic asymmetry between the $e_{4}$ and $e_{3}$ directions in the Vlasov propagation (breaking the naive $e_{4} \leftrightarrow e_{3}$ symmetry present in the vacuum Einstein equations modulo a $\mathbb{Z}_{2}$ symmetry of certain Ricci coefficients), which in turn underlines the fundamental obstruction to simultaneous top–order control of Ricci coefficients on both $\mathcal{H}_{u}$ and $\underline{\mathcal{H}}_{\ubar}$.

\subsection{Reduced Elliptic Estimates}

This section is devoted to the implementation of the second principal structural innovation underpinning our analysis. We uncover a novel geometric feature intrinsic to the Einstein--Vlasov system that enables the selective application of elliptic estimates to a proper subset of the null connection coefficients, specifically \((\chihat, \tr\chi, \eta)\), restricted to the outgoing null hypersurface \(H\). This refinement permits us to dispense entirely with elliptic control of the remaining Ricci coefficients, thereby circumventing the full elliptic theory traditionally invoked for all connection variables in energy-type arguments. We stress on the fact that this is absolutely essential since elliptic estimates for Ricci coefficients such as $\tr\chibar$ causes significant obstruction that we describe towards the end of this section. 

\noindent Let us introduce precise notation to describe the relevant geometric quantities. Denote by
\[
\psi := \{\chihat, \tr\chi, \chibarhat, \tr\chibar, \eta, \etabar, \omega = 0, \omegabar, \Gammaslash\}
\]
the full set of Ricci coefficients (i.e., connection coefficients associated with the double null foliation), where the notation is standard and follows the conventions of the null frame formalism. The Weyl curvature components are encoded in the set
\[
\Psi := \{\alpha, \alphabar, \beta, \betabar, \rho, \sigma\},
\]
while we denote by \(\mathfrak{M}\) the matter degrees of freedom associated with the Vlasov field, which include the distribution function \(f\) and its geometric derivatives along the fibers of the mass shell bundle.

\noindent It is a fundamental feature of the Einstein-Vlasov system that the coupling between the matter and curvature variables is nontrivial at the level of derivatives: the Bianchi equations, when written in the null frame formalism, contain terms involving first-order derivatives of the matter field \(f\) along the null and sphere directions. Consequently, any attempt to close an energy hierarchy for the matter variables necessarily requires control of connection coefficients at a level commensurate with that of \(f\), which—owing to its transport nature in both base and momentum variables—typically requires higher regularity than the curvature components \(\Psi\) themselves.

\noindent The classical approach to this problem entails establishing elliptic estimates for all elements of \(\psi\). This strategy ensures that one can gain sufficient regularity to estimate \(f\) via the Vlasov equation and simultaneously control the geometric side via the Einstein equations. At the level of energy estimates for the geometric variables, one typically seeks to estimate
\[
\|\mathcal{D}^{N} \Psi\|_{L^{2}(H)} + \|\mathcal{D}^{N} \Psi\|_{L^{2}(\underline{H})},~ \quad N \geq 1,
\]
where \(\mathcal{D}\) denotes a combination of angular derivatives and projected null derivatives compatible with the geometry of the foliation. Note that the curvature components \(\alpha\) and \(\alphabar\) admit control only on the outgoing and incoming null hypersurfaces \(H\) and \(\underline{H}\), respectively, due to their signature and the hyperbolic nature of the Bianchi system. We remark that even in this optimal setting, the current analytic techniques are insufficient to reach full \(N=1\) regularity for the Weyl curvature field, a point which we elaborate upon in subsequent sections.

\noindent The structure of the null Bianchi equations governing the curvature tensor requires the estimation of
\[
\|\mathcal{D}^{N+1} T_{\mu\nu}\|_{L^{2}(H)}\]
where \(T_{\mu\nu}\) denotes the energy-momentum tensor derived from the Vlasov distribution function. Since \(T_{\mu\nu}\) contains velocity moments of \(f\), and these moments involve integration over the mass shell coupled to the spacetime geometry, the control of \(T_{\mu\nu}\) at order \(N+1\) through the transport equation imposes stringent requirements on the differentiability of the metric coefficients. More precisely, it necessitates estimates of
\[
\|\mathcal{D}^{N+1} \psi\|_{L^{2}(H)} + \|\mathcal{D}^{N+1} \psi\|_{L^{2}(\underline{H})}
\]
which cannot, in general, be achieved via transport equations alone. The propagation equations for \(\psi\) along the null directions involve losing one derivative. As such, one must appeal to the elliptic structure inherent in the sphere-level constraint equations—specifically, the \(\div\)--\(\curl\) systems for \(\chihat\), \(\tr\chi\), and \(\eta\)—to gain the necessary derivative and close the hierarchy. This highlights the geometric necessity of invoking elliptic estimates and motivates the structural refinement we exploit in this section, which allows us to localize these elliptic estimates exclusively to the subset \((\chihat, \tr\chi, \eta)\) along \(H\), while avoiding their application to the remaining connection variables.

\noindent However, we uncover a key structural simplification in the Einstein–Vlasov system that permits us to drastically reduce the number of connection coefficients requiring elliptic control. Specifically, we demonstrate that it suffices to derive elliptic estimates only for \((\chihat, \tr\chi, \eta)\) on the null hypersurface \(H\). This is in sharp contrast to prior works, where elliptic control over the entire set \(\psi\) was deemed indispensable.

\noindent To elucidate this point, recall that in order to estimate derivatives of the Vlasov distribution function \(f\), one must commute the transport equation
\begin{equation}
X[f] = 0,
\end{equation}
with a suitable system of vector fields \(V\). Upon commutation, one obtains
\begin{equation}
X[Vf] = [X,V]f.
\end{equation}
Schematically, the commutator term takes the form
\begin{equation}
[X, V]f \sim (\Psi, \psi) \cdot Vf,
\end{equation}
indicating that the derivatives of both the Weyl curvature components and the connection coefficients enter at the same differential level in the evolution of \(Vf\). This structural feature, specific to the Vlasov system, distinguishes it from other matter models—such as Maxwell or Yang–Mills fields—whose hyperbolic character leads to fundamentally different regularity hierarchies.
\vspace{3mm}

\noindent In this study, we consider $N=2$, meaning the Weyl curvature is controlled up to the second derivative. Commuting to the top order, we encounter terms of the schematic form
\begin{equation}
X[V^3 f] \sim (\Psi, \psi) V^3 f + \mathcal{D}^2 (\Psi, \psi) V f + \mathcal{D} (\Psi, \psi) V^2 f,
\end{equation}
revealing the appearance of up to second derivatives of the connection and curvature components. One might thus surmise that, as long as the curvature is controlled at second order, elliptic estimates for the connection coefficients could be dispensed with. However, the difficulty lies in certain components of \(\psi\) that are not expressible directly in terms of curvature and whose higher derivatives cannot be obtained without elliptic analysis.

\noindent A concrete manifestation of this obstruction is the appearance of the term \(e_A( b^B)\), where \(b^A\) denotes the shift vector field tangential to the sphere $S_{u,\ubar}$. Since this expression involves a derivative of the shift, it introduces a need to control third derivatives of \(b\) while estimating the desired \(
\|\mathcal{D}^{N+1} T_{\mu\nu}\|_{L^{2}(H)},~N=2\). From the evolution equation
\begin{equation}
\snabla_4 b \sim (\eta - \etabar) + \chi \cdot b,
\end{equation}
we see that controlling \(\mathcal{D}^3 b\) necessitates control over \(\mathcal{D}^3 \chi\), \(\mathcal{D}^3 \eta\), and \(\mathcal{D}^3 \etabar\). In parallel, the connection \(\Gammaslash\) satisfies the equation
\begin{equation}
\snabla_4 \Gammaslash \sim \nablasl \chi + \Gammaslash \cdot \chi,
\end{equation}
which again involves third derivatives of \(\chi\). These considerations suggest that elliptic control over \((\chihat, \tr\chi, \eta, \etabar)\) is required. But the elliptic estimates for $\etabar$ require elliptic estimates for $\tr\chibar$ and $\chibar$, which in turn require elliptic estimates for $\omegabar$. Therefore, it appears one needs elliptic estimates for all the connection coefficients. 

\noindent However, a key observation allows us to avoid this conclusion. \textit{The potentially dangerous term \(e_A (b^B) \) never appears in isolation.} Instead, it always appears in the combination
\begin{equation}
\Gamma^B_{3A} = {\chibar_A}^B - e_A( b^B),
\end{equation}
that is, as a component of the connection in the null frame.  
Let us give a concrete examples of how this term appears. For the geodesic vector field $X:=p^{\mu}e_{\mu}-\Gamma^{\overline{\lambda}}_{\alpha\beta}p^{\alpha}p^{\beta}\frac{\partial}{\partial \pbar^{\lambda}}, \lambda=1,2,3$ and a vector field $V_{A}:=\text{Hor}(e_{A})-\frac{p^{3}}{|u|}\frac{\partial}{\partial \pbar^{A}}$, if one computes $[X,V_{A}]$, then semi-schematically 
\begin{align}
[X,V_{A}]= \Bigg(-p^{3}e_{A}(b^{B})+\Gamma^{B}_{A\nu}p^{\nu}+\frac{p^{3}}{|u|}\delta^{B}_{A}\Bigg)V_{B}+....   
\end{align}
Notice the appearance of $e_{A}(b^{B})$ in the commutation as the error term (and it appears in the exact same fashion everywhere else). But we can write the full error term $-p^{3}e_{A}(b^{B})+\Gamma^{B}_{A\nu}p^{\nu}+\frac{p^{3}}{|u|}\delta^{B}_{A}$ as follows 
\begin{align}
-p^{3}e_{A}(b^{B})+\Gamma^{B}_{A\nu}p^{\nu}+\frac{p^{3}}{|u|}\delta^{B}_{A}=\underbrace{\Gamma^{B}_{A3}p^{3}+\frac{p^{3}}{|u|}\delta^{B}_{A}-p^{3}e_{A}(b^{B})}_{I}+\Gamma^{B}_{A4}p^{4}+\Gammaslash^{B}_{AC}p^{C}.    
\end{align}
Now consider the term $I$ (the remaining terms have good enough integrability), which involves a non-integrable term $\frac{p^{3}}{|u|}$ (note that $p^{3}$ does not decay) and the term involving the shift vector field. First, recall the following 
\begin{align}
 \Gamma^{B}_{A3}=\Gamma^{B}_{3A}+e_{A}(b^{B}),   
\end{align}
which yields cancellation of the term $-p^{3}e_{A}(b^{B})$. One can proceed to obtain an equation for $\Gamma^{B}_{3A}$. Denoting $\Gamma^{B}_{3A}$ by $\mathcal{G}^{B}_{A}$ (technically speaking we will control $\Gamma[g]^{B}_{3A}-\Gamma[\eta]^{B}_{3A}$ since $\Gamma^{B}_{3A}$ is not a tensor rather can be interpreted as an endomorphism of the tangent bundle of the sphere $S_{u,\ubar}$), schematically, such an equation reads 
\begin{align}\label{Gintroduction}
\snabla_4 \mathcal{G} \sim \nablasl \eta + \mathcal{G} \cdot \chi + \omegabar \cdot \chi + \Gammaslash \cdot (\eta, \etabar) + \eta \cdot \etabar + (\rho,\sigma)+\Tslash,    
\end{align}
where one notices that the only derivative term that appears is $\nablasl\eta$. This hints that one may estimate $\mathcal{G}$ by the elliptic estimates for $\eta$ (and $(\chihat,\tr\chi)$ since they appear for the elliptic estimates for $\eta$). However, one still needs to take care of the non-integrable term $\frac{p^{3}}{|u|}$. To this end, we make the following set of manipulations 
\begin{align}
 \mathcal{G}^{A}_{B}=\widehat{\mathcal{G}}^{A}_{B}+\frac{1}{2}\tr[\mathcal{G}]\delta^{A}_{B},
\end{align}
where explicitly 
\begin{align}
\widehat{\mathcal{G}}^{A}_{B}=\widehat{\Gamma}^{A}_{B3}-[e_{A}(b^{B})-\frac{1}{2}e_{C}(b^{C})\delta^{A}_{B}],~\tr[\mathcal{G}]=\tr\chibar-e_{C}(b^{C})    
\end{align}
and note that 
\begin{align} \notag
 p^{3}\Gamma^{A}_{B3}-e_{A}(b^{B})p^{3}+\frac{p^{3}}{|u|}=&p^{3}\Bigg(\widehat{\Gamma}^{A}_{B3}-[e_{A}(b^{B})-\frac{1}{2}e_{C}(b^{C})\delta^{A}_{B}]+\frac{1}{2}(\tr\chibar+\frac{2}{|u|}-e_{C}(b^{C}))\delta^{A}_{B}\Bigg) \\
 =&p^{3}\widehat{\mathcal{G}}^{A}_{B}+\frac{1}{2}p^{3}\Bigg(\tr[\mathcal{G}]+\frac{2}{|u|}\Bigg)\delta^{A}_{B}.
\end{align}
Schematically, the equation for $\tr[\mathcal{G}]+\frac{2}{|u|}$ and $\widehat{\mathcal{G}}$ reads 
\begin{align}
\snabla_{4}\Bigg(tr[\mathcal{G}]+\frac{2}{|u|}\Bigg)\sim& \chihat \hsp \widehat{\mathcal{G}}+\div\eta+|\eta|^{2}+\eta\hsp \etabar+\tr\chi\tr\chibar+\omegabar\hsp \tr\chi+\chihat\hsp \chibarhat+\Gammaslash(\eta,\etabar)+(\rho,\sigma)+\Tslash,\\
\snabla_{4}\widehat{\mathcal{G}}\sim& \nablasl \eta + \widehat{\mathcal{G}} \cdot \chi+\chi \chi + \omegabar \hsp \chi + \Gammaslash \cdot (\eta, \etabar) + \eta (\eta,\etabar) + (\rho,\sigma)+\Tslash. 
\end{align}

\noindent Consequently, controlling \(\mathcal{D}^2 \widehat{\mathcal{G}},\hsp \mathcal{D}^{2}(tr[\mathcal{G}]+\frac{2}{|u|})\) in \(L^2(S)\) requires only control over \(\mathcal{D}^3 \eta\) along \(H\), together with lower-order control over the remaining terms. This is the central insight: The entire analytic burden reduces to deriving elliptic estimates for \(\eta\) along \(H\), eliminating the need for such estimates on \((\chibarhat, \tr\chibar, \omegabar, \etabar)\). We now illustrate how we can decouple the elliptic estimates for $\eta$ from those of $\etabar$ and subsequently $(\chibarhat,\tr\chibar,\omegabar)$.

\noindent To estimate \(\eta\), one invokes elliptic estimates associated to \((\chihat, \tr\chi)\). Crucially, the elliptic estimates for \((\chihat, \tr\chi)\) do not propagate derivatives in the incoming null direction, and are therefore entirely compatible with working on \(H\) alone. Thus, we are able to establish full regularity for the matter distribution \(f\), up to the requisite top order, while performing elliptic estimates exclusively for \((\chihat, \tr\chi, \eta)\) along \(H\). However, there is a subtlety in the elliptic estimates for $\eta$. Since the equation for $\eta$ contains Weyl curvature $\beta$, one can not directly use this transport equation. Instead, the standard approach is to define the mass aspect function 
\begin{align}
 \mu:=-\div\eta-\rho+\frac{1}{2}\chihat\cdot\chibarhat   
\end{align}
which, upon being acted on by $\snabla_{4}$ and through the use of the equation for $\rho$, eliminates the dangerous $\snabla\beta$ term. In the current framework, however, this is not enough to decouple $\eta$ and $\etabar$ since $\div(\nabla_{4}\eta)$ would contain terms $\chi\cdot \nablasl\etabar$ potentially obstructing this route. To this end, we define a new $\mu$ as follows 
\begin{align}
 \mu:=-\div\eta-\rho+\frac{1}{2}\chihat\cdot\chibarhat+\frac{1}{4}\tr\chi\widetilde{\tr\chibar}   
\end{align}
which precisely cancels out the extra $\nablasl\etabar$ term:
\begin{align}
 \snabla_{4}\mu\sim (\eta,\etabar)\nablasl(\chihat,\tr\chi)+\chihat\nablasl\eta+\snabla_{4}T_{43}+\snabla_{3}T_{44}+l.o.t.   
\end{align}
This, together with the equation for $\curl\eta$, immediately yields the top elliptic estimate for $\eta$ without the need for elliptic estimates for $\etabar$. Since the elliptic estimates for $\etabar$ are unnecessary in this context, one avoids having to establish the elliptic estimates for $\chibarhat,\tr\chibar,$ and $\omegabar$, which would have to be obtained first. In addition, there would be serious complications in working with elliptic estimates for $\chibarhat$ and $\tr\chibar$ in the present framework, where the Vlasov matter source is only estimated on $H$ at the top order. This is because of the structure of the equation $\tr\chibar$
\begin{align}
    \snabla_{3} \tr\chibar + \frac12 \, (\tr\chibar)^{2}
 = - |\chihat|^{2} - 2\, \omegabar \, \tr\chibar - T_{33}
\end{align}
Note that controlling $\nablasl^{3}\tr\chibar$ over $\Hbar_{\ubar}$ would require controlling the topmost order norm $\nablasl^{3}T_{33}$ over $\Hbar_{\ubar}$ which is impossible. Therefore, without this aforementioned reduction in elliptic estimates, this program is doomed to fail. 

\noindent This reduction in analytic complexity—made possible by careful exploitation of the geometric structure of the Einstein–Vlasov system—constitutes a key technical innovation of our work and, to the best of our knowledge, has not been previously observed in the literature.

\subsection{Regularity closure for the coupled Einstein-Vlasov hierarchy}

\noindent In this section, we address what we call the regularity issue, namely what is the least number of  derivatives of curvature and of Vlasov matter that one requires control on to close energy estimates (in the smooth initial data setting). This is an important question. To place the issue in context, recall that in the short-pulse analysis of the vacuum Einstein equations in double-null gauge, Klainerman and Rodnianski \cite{Kl-Rod} obtained control of one derivative of the Weyl curvature tensor in the scale-invariant norm \( \mathcal{L}^{2}_{sc}(H \slash \Hbar) \) in a breakthrough work. Their analysis relies crucially on refined trace estimates on codimension-two submanifolds, which compensate for the borderline failure of the Sobolev embedding
\[
        H^{1}(S)\not\hookrightarrow L^{\infty}(S)
\]
in two dimensions. In the present Einstein--Vlasov setting, the same regularity level encounters a new obstruction coming from the kinetic nature of the matter field. The point is not that such a lower-regularity theory is impossible in principle, but rather that it is not compatible with the specific hierarchy of estimates used here.

\noindent Let us explain this mechanism. Given a tensor field
\(\varphi\in \Gamma({}^{K}\otimes TS_{u,\ubar})\), recall the trace norms
\begin{align}
\|\varphi\|_{\operatorname{tr}(H)}
    := \sup_{S_{u,\ubar}}
       \left( \int_{0}^{\ubar} |\varphi|^{2}\,d\ubar' \right)^{1/2},
    \qquad
\|\varphi\|_{\operatorname{tr}(\Hbar)}
    := \sup_{S_{u,\ubar}}
       \left( \int_{0}^{u} |\varphi|^{2}\,du' \right)^{1/2}.
\end{align}
These norms are designed to capture the borderline behavior of curvature components along null hypersurfaces. Suppose one attempts to close the coupled hierarchy with only one angular derivative of the Weyl curvature in the scale-invariant energy space \( \mathcal{L}^{2}_{sc}(H\slash \Hbar) \). Since the null Bianchi equations contain derivatives of the stress-energy tensor, this would require corresponding control of angular derivatives of the Vlasov stress-energy tensor. Because the stress-energy tensor is a velocity average of the distribution function \(f\), this in turn requires estimates for derivatives of \(f\) on the mass shell.

\noindent The difficulty is that the Vlasov field is propagated by the geodesic flow in phase space. At top order, the useful propagation is along the incoming null direction from the initial hypersurface \(H_{u_{\infty}}\). Indeed, for a null geodesic \(s\mapsto \gamma(s)\), with \(u(s)=u(\gamma(s))\), the momentum component satisfies schematically
\begin{align}
        \frac{p^{3}(s)}{\frac{du}{ds}} = 1.
\end{align}
Thus the Vlasov estimates may be integrated along the incoming direction in a way compatible with the characteristic geometry. However, this also means that one cannot freely recover missing transverse regularity without producing derivative losses.

\noindent  We now isolate the obstruction in the present argument. Let \(X\) denote the geodesic spray on the mass shell and let \(V\) be one of the commutation vector fields adapted to the null frame. Schematically, the commuted Vlasov equation contains terms of the form
\begin{align}
X[Vf] &\sim \alpha \cdot \pslash \cdot Vf+\text{l.o.t.},\\
X[V^{2}f] &\sim (\mathcal{D}\alpha)\cdot \pslash \cdot Vf
              + \alpha\cdot V^{2}f+\text{l.o.t.},
\end{align}
where \(\alpha\) denotes the outgoing curvature component and \(\mathcal{D}\alpha\) denotes an angular derivative. If one tries to estimate the second-order phase-space energy of \(f\), one is led to terms of the schematic form
\begin{align}
II :=
\int_{u_{\infty}}^{u}
\int_{0}^{1}
\int_{S_{u,\ubar}}
\int_{\mathcal{P}_{x}}
        V^{2}f \cdot (\mathcal{D}\alpha)\cdot \pslash \cdot Vf
        \,d\mu_{\mathcal{P}_{x}}\,d\mu_{S_{u,\ubar}}\,d\ubar\,du .
\end{align}
At this regularity level, \(\mathcal{D}\alpha\) is controlled in an \(L^{2}\)-type curvature energy, and hence the remaining Vlasov factor must be estimated in a uniform-in-\(S_{u,\ubar}\) norm, for instance through
\begin{align}
\sup_{\ubar\in[0,1]}\sup_{S_{u,\ubar}}
\left(
    \int_{\mathcal{P}_{x}} |Vf|^{2}\,d\mu_{\mathcal{P}_{x}}
\right)^{1/2}.
\end{align}
To propagate this quantity, one integrates the first commuted transport equation. This produces the borderline term
\begin{align}
III :=
\int_{u_{\infty}}^{u}
\int_{\mathcal{P}_{x}}
        \alpha \cdot \pslash \cdot |Vf|^{2}
        \,d\mu_{\mathcal{P}_{x}}\,du .
\end{align}
A natural attempt would be to use the codimension-two trace estimate and write
\begin{align}
|III|
    \lesssim
    \sup_{S_{u,\ubar}}
    \left(
        \int_{u_{\infty}}^{u} |\alpha|^{2}\,du'
    \right)^{1/2}
    \left(
        \int_{u_{\infty}}^{u} |u'|^{-2}\,du'
    \right)^{1/2}
    \sup_{u}
    \int_{\mathcal{P}_{x}} |Vf|^{2}\,d\mu_{\mathcal{P}_{x}},
\end{align}
or equivalently
\begin{align}
|III|
    \lesssim
    \|\alpha\|_{\operatorname{tr}(\Hbar)}
    |u|^{-1/2}
    \sup_{u}
    \int_{\mathcal{P}_{x}} |Vf|^{2}\,d\mu_{\mathcal{P}_{x}} .
\end{align}
The obstruction is that, in the curvature hierarchy used here, \(\alpha\) is controlled along outgoing null hypersurfaces but not in the incoming trace norm
\(\operatorname{tr}(\Hbar)\). Thus the above estimate does not close within the present scheme.

\noindent For this reason, our argument works one derivative higher in curvature. With two derivatives of the Weyl curvature available in the scale-invariant energy norms, one obtains the required \(L^{\infty}(S_{u,\ubar})\) control of the curvature components by combining the \(L^{4}(S_{u,\ubar})\) control of first angular derivatives with the codimension-one trace estimates, modulo lower-order terms. Correspondingly, the Vlasov hierarchy is propagated up to three derivatives in
\[
        L^{2}(H)L^{2}(\mathcal{P}_{x}).
\]
This is the regularity level at which the top-order commuted transport estimates, the curvature energy estimates, and the reduced elliptic estimates are mutually compatible.

\noindent To illustrate the closure, consider for simplicity the vector field \(V_{3}\). The relevant norms are
\begin{align}
 A &:= \int_{\mathcal{P}_{x}} |V_{3}f|^{2}\,d\mu_{\mathcal{P}_{x}},\\
 B &:= \int_{S_{u,\ubar}}\int_{\mathcal{P}_{x}}
        |V_{3}^{2}f|^{2}\,d\mu_{\mathcal{P}_{x}}\,d\mu_{S_{u,\ubar}},\\
 C &:= \int_{0}^{1}\int_{S_{u,\ubar}}\int_{\mathcal{P}_{x}}
        |V_{3}^{3}f|^{2}\,d\mu_{\mathcal{P}_{x}}\,d\mu_{S_{u,\ubar}}\,d\ubar .
\end{align}
The most delicate estimate is the one for \(C\). From the transport equation,
\begin{align}
 C(u)
 \leq C(u_{\infty})
 +2\int_{u_{\infty}}^{u}
    \int_{0}^{1}\int_{S_{u,\ubar}}\int_{\mathcal{P}_{x}}
        V_{3}^{3}f\, X[V^{3}f]
        \,d\mu_{\mathcal{P}_{x}}\,d\mu_{S_{u,\ubar}}\,d\ubar\,du
 + \text{Gronwall terms}.
\end{align}
Since schematically
\begin{align}
        X[V_{3}f]\sim \alpha_{AB}p^{A}p^{B}V_{3}f+\cdots ,
\end{align}
one obtains
\begin{align}
X[V^{3}f]
    \sim V_{3}^{2}(\alpha Vf)+\cdots
    \sim
    \underbrace{\alpha_{AB}p^{A}p^{B}V^{3}f}_{I}
    +
    \underbrace{\mathcal{D}\alpha_{AB}p^{A}p^{B}V_{3}^{2}f}_{II}
    +
    \underbrace{\mathcal{D}^{2}\alpha_{AB}p^{A}p^{B}V_{3}f}_{III}
    +\cdots .
\end{align}
The term \(I\) is controlled by placing \(\alpha\) in \(L^{\infty}(S_{u,\ubar})\). The term \(III\) is controlled by placing \(\mathcal{D}^{2}\alpha\) in the curvature energy norm and using the uniform control of \(A\). The middle term \(II\) is the most delicate one, but it is precisely here that the hierarchy of \(A\), \(B\), and \(C\), together with the improved curvature control available at two derivatives, gives the necessary interpolation and trace bounds. Thus the Vlasov and curvature estimates close simultaneously.

\noindent The conclusion is therefore not that two derivatives of curvature are universally optimal for the Einstein--Vlasov system. Rather, the conclusion is that two derivatives of curvature, together with three derivatives of the Vlasov distribution function in the adapted commuted-transport norms, form the regularity threshold at which the present large-data double-null argument closes.

\vspace{3mm}

\noindent Thus, the Einstein--Vlasov system presents a genuinely new analytic challenge. While the failure of closure at this critical level (note that this `critical' is unrelated to the criticality of the norm required for the Einstein's equations in $3+1$ dimensions) seems to preclude a direct extension of the characteristic Cauchy framework, it remains conceivable that a novel method---possibly circumventing the need for trace control of \( \alpha \) or exploiting the structure of the Vlasov equation in a more refined way---could provide an alternate resolution to the characteristic initial value problem in this setting.
\\

\noindent This difficulty is sharply contrasted with other matter models such as the Yang--Mills equations coupled to Einstein, wherein the energy-momentum tensor is hyperbolic and possesses a better compatibility with the curvature structure. In those settings, the trace norms of the null curvature components are controllable, and one-derivative estimates on the Weyl tensor suffice to propagate two derivatives of the matter field; see \cite{M-Y} for the relevant analysis.

\subsection{Acknowledgements}
\noindent P.M. is supported by the Center of Mathematical Sciences and Applications of Mathematics Department at Harvard University and Beijing Institute of Mathematical Sciences and Applications of Yau Mathematical Science Center at Tsinghua University. Moreover, N.A. wishes to thank Mihalis Dafermos for useful discussions and comments on the project.

\section{Basic setup}
	
	\subsection{Construction of the double null gauge}
We begin by assuming that the functions $u$ and $v$ satisfy 

\[ g^{\mu\nu}\partial_{\mu}u \partial_{\nu}u=0, \hspace{2mm} g^{\mu\nu}\partial_{\mu}\ubar \partial_{\nu}\ubar =0. \]Define null vector fields
\[ L^{\mu} = -2 g^{\mu\nu}\partial_{\nu}u, \hspace{2mm} \Lbar^{\mu} = -2g^{\mu\nu}\partial_{\nu}\ubar .\]  Also define $\Omega >0$ by \[ \Omega^{-2} = -\frac{1}{2}g(L, \Lbar). \]To define $(\theta^1, \theta^2)$ on the spacetime, we proceed as follows. Let $(\theta^1, \theta^2)$ be a coordinate system in some open subset $U_1$ of the initial sphere $S_{u_{\infty},0}.$ Define $(\theta^1, \theta^2)$ on $\begin{Bmatrix} u= u_{\infty} \end{Bmatrix}$ by solving $L(\theta^A)=0, \hspace{1mm} A=1,2$ and then extend to $u > u_{\infty}$ by solving $\Lbar(\theta^A)=0, \hspace{1mm} A=1,2.$ \vspace{3mm}

\noindent As such, we have defined coordinates $(u,v,\theta^1, \theta^2)$ on a region $D(U_1)$, defined as the image of $U_1$ under the diffeomorphisms of $L$ on $\begin{Bmatrix} u = u_{\infty} \end{Bmatrix}$ and then the diffeomorphisms generated by $\Lbar$. By repeating the procedure for a complementary open set $U_2$ of $S_{u_{\infty},0}$ such that $U_1 \cup U_2 = S_{u_{\infty},0}$, we get coordinate charts that cover the entire spacetime region that will be of interest. The spheres of constant $u$ and $v$ will be denoted by $S_{u,v}$ and the restriction of $g$ on these spheres will be denoted by $\gslash$. Throughout this work, we shall be working with the following double null frame:

\be \label{ourframe} e_A= \partial_{\theta^A}, \hspace{2mm} \text{for} \hspace{1mm} A=1,2, \hspace{3mm} e_{3} := \partial_u+ b^A \hsp \partial_{\theta^A}, \hspace{3mm} e_4 = \frac{1}{\Omega^2}\partial_{\ubar} . \ee In our case, the reason why this choice is convenient is that $\omega$, to be defined later, is zero in this frame and the torsion $\zeta$ equals $-\etabar$.


\vspace{3mm}

\noindent In these coordinates, the metric then takes the following form:

\be g=-2\Omega^2(\text{d}u \otimes \text{d}v + \text{d}v \otimes \text{d}u) +\gslash_{AB}(\text{d}\theta^A - b^A \text{d}u)\otimes (\text{d}\theta^B-b^B du), \ee where we impose that the shift $b$, tangent to the spheres $S_{u,v}$ be vanishing at $\begin{Bmatrix} u= u_{\infty}\end{Bmatrix} \hsp \cup \hsp \begin{Bmatrix} \ubar =0 \end{Bmatrix}$ and that the lapse $\Omega \equiv 1$ on the same set. Integration of a function $\varphi$ is $S_{u,\ubar}$ is defined as 

\[ \int_{S_{u,\ubar}}\varphi \hsp \text{d}\mu_{S_{u,\ubar}}=\sum_{i=1,2}\int_{\theta^1}\int_{\theta^2}\varphi \hsp \tau_i \sqrt{\det\gslash}
\hsp \text{d}\theta^1 \text{d}\theta^2,     \]where $\tau_1,\tau_2$ is a partition of unity subordinate to $D_{U_1},D_{U_2}$ at $u,\ubar$.

\subsection{Coordinates on the mass shell and the mass shell relation} Let $(p^{\mu}: \hsp \mu=1,2,3,4)$ be the conjugate coordinate, on each tangent space, to the double null frame just given, so that $(x^{\mu}, \hsp p^{\mu})$ denotes the point

\[ \restri{p^{\mu}e_{\mu}}{x} \in T_x\mathcal{M}.\]This then gives a frame $\begin{Bmatrix} e_{\mu}, \hsp \partial_{p^{\mu}} \hsp \mid \mu=1,2,3,4 \end{Bmatrix}$ on $T\mathcal{M}$. The Vlasov equation, written in this frame, then reads

\[  p^{\mu}e_{\mu}(f) - \Gamma^\lambda_{\mu\nu} p^{\mu}p^{\nu}\partial_{p^{\lambda}}(f) =0.    \]When $f$ is viewed as a function on the mass shell $\mathcal{P}$, this rewrites as

\[ p^{\mu}e_{\mu}(f) -\Gamma^{\hat{\lambda}}_{\mu\nu}p^{\mu}p^{\nu}\partial_{\pbar^{\hat{\lambda}}}(f) =0, \]where $\hat{\lambda} = 1,2,3$ and $p^{\hat{\lambda}}$ denotes the restriction of the coordinate $p^{\lambda}$ to $\mathcal{P}$. Finally, the $\partial_{\pbar^i}$, for $i=1,2,3$, denote the partial derivatives with respect to this restricted coordinate system. It is easy to verify, using \eqref{massshell}, that the following hold true:

\be \partial_{\pbar^A} = \partial_{p^A} + \frac{\gslash_{AB}p^B}{2p^3}\partial_{p^4}, \hspace{2mm} \partial_{\pbar^3}= \partial_{p^3}-\frac{p^4}{p^3}\partial_{p^4}. \ee
\vspace{3mm}
\noindent We recall that the mass shell $P \subset T\mathcal{M}$ is defined as the set of all future-pointing null vectors. Using the form of the metric given in the previous subsection together with the definition of the coordinates $p^{\mu}$ and the fact that the particles are massless, we obtain the relation

\be \label{ms}
\gslash_{AB}p^Ap^B -4p^3p^4 =0,  \hspace{3mm}\text{throughout the support of $f$.} \ee The above is known as the mass shell relation. In particular, $P\subset T\mathcal{M}$ is a $7-$dimensional hypersurface in $T\mathcal{M}$ and can be parametrised by coordinates $(u,v,\theta^1, \theta^2, p^1, p^2, p^3)$, with $p^4$ given by \eqref{ms}.

\begin{proposition}
\label{integrate}
 Along a geodesic, the following holds in our choice of frame
 \begin{align}
  \frac{p^{3}(s)}{\frac{du}{ds}}= 1,   
 \end{align}
 where $u(s):=u(\gamma(s))$.  
\end{proposition}
\begin{proof}
Making use of the fact that $\dot{u}(s):=\frac{du}{ds} =  p^{3}(s)$ and that, given $s\mapsto \gamma(s)$ is an incoming geodesic emanating from $u=u_{\infty}$, we may use $\restri{u}{\gamma}= \restri{u}{\gamma}(s)$ as a parameter for the geodesic.The conclusion follows.
\end{proof}

\begin{remark}
 This relation allows us to convert the inegration along the geodesic parameter $s$ to the double null retarded coordinate parameter $u$ i.e., $|\int f(u(s)) ds|=|\int f(s)\frac{ds}{du}du|\lesssim \int |f(u)|du$ in light of the estimates verified by $p^{3}(s)$ (see proposition \ref{momentum}) since $|\Omega^{-1}-1|\lesssim \Gamma a|u|^{-2}$.  
\end{remark}
\subsection{Volume form on the mass shell}Recall that the \textbf{mass shell} $\mathcal{P}\subset T\mathcal{M}$ is defined to be the set of all future-pointing null vectors.  If a given $(x,p)\in \mathcal{P}$, then $p=\restri{p^{\mu}e_{\mu}}{x}$ as a vector on $T_x\mathcal{M}$ satisfies $g_x(p,p)=0$, which translates to the mass shell relation \be \gslash_{AB}p^A p^B -4p^3 p^4=0. \ee To make sense of the integral in the definition of the energy-momentum tensor, we need to define a suitable volume form (Leray form) on $\mathcal{P}_x$ over each  $x$. We seek a one-form that is normal to $\mathcal{P}_{x}$. A canonical such choice is the differential of the function $\Lambda_x:T_{x}\mathcal{M}\to \mathbb{R}$:

\[ \Lambda_x(X)=g(X,X).      \]This is the function that measures the length of each vector in $T_x\mathcal{M}$. Taking the normal $-\f12 \text{d}\Lambda_x$ to $\mathcal{P}_x$, the volume form in the $(u,\ubar,\theta^1,\theta^2,p^1,p^2,p^3)$ system can be defined as 

\[  \frac{\sqrt{\det\gslash}}{p^3}\text{d}p^3\wedge \text{d}p^1 \wedge \text{d}p^2.      \]This is the unique volume form on $\mathcal{P}_x$ \textit{compatible} with the normal $-\f12 \text{d}\Lambda_x$ (see also \cite{taylor}). The energy momentum tensor at $x\in\mathcal{M}$ thus takes the form

\[ T_{\mu\nu}(x)= \int_0^{\infty}\int_{-\infty}^{\infty}\int_{-\infty}^{\infty}f \hsp p_{\mu}\hsp p_{\nu} \frac{\sqrt{\det\gslash}}{p^3}\text{d}p^1 \text{d}p^2\text{d}p^3.      \]

\subsection{Ricci coefficients and curvature components} Let $A, B$ take values in $\begin{Bmatrix} 1,2 \end{Bmatrix}$. We define the following (Weyl) null curvature components:

    \[ \alpha_{AB} := W(e_A, e_4, e_B, e_4), \hspace{2mm} \alphabar_{AB}:= W(e_A, e_3, e_B, e_3),       \] \[ \beta_A := \frac{1}{2} W(e_A, e_4 , e_3 ,e_4), \hspace{2mm} \betabar_A := \frac{1}{2} W(e_A, e_3, e_3, e_4),     \] \[  \rho := \frac{1}{4} W(e_3, e_4, e_3, e_4), \hspace{2mm} \sigma  = \frac{1}{4} \Hodge{W}(e_3, e_4, e_3 ,e_4).      \]Moreover, for reasons to be explained later on, we define the \textit{renormalized curvature components }

    \be \label{tbetatbetabar} \tbeta_A = \beta_A - \frac{1}{2}R_{A4}, \hspace{2mm} \tbetabar_A = \betabar_A + \frac{1}{2} R_{A3}.       \ee For the Ricci coefficients in this frame, we decompose as follows:

    \[  \chi_{AB} := g(D_A e_4, \hsp e_B), \hspace{2mm} \chibar_{AB} := g(D_A e_3, e_B),            \] \[   \eta _A := -\frac{1}{2} g(D_3 e_A , \hsp e_4), \hspace{2mm} \etabar_A := -\frac{1}{2} g(D_4 e_A, e_3), \]\[  \omega := \frac{1}{2} g(D_4 e_3 ,\hsp e_4)=0,  \hspace{2mm} \omegabar :=\frac{1}{2}g(D_3 e_4, e_3),                \]\[ \zeta_A := \frac{1}{2} g(D_A e_4, \hsp e_3) = -\etabar_A.  \]Moreover, if $\gamma$ denotes the induced metric on $S_{u,\ubar}$, we make the following further decomposition:

    \[ \chi = \chihat + \frac{1}{2} \tr\chi  \gamma, \hspace{2mm} \chibar = \chibarhat + \frac{1}{2} \tr\chibar \hsp \gamma.    \]The Christoffel symbols of $(S_{u,v},\gslash)$ with respect to $\begin{Bmatrix} e_1, e_2 \end{Bmatrix}$ are denoted by $\Gammaslash^C_{AB}:$
    \be \nablasl_{e_A}e_B = \Gammaslash^C_{AB}\hsp e_C. \ee We also let $\big(\Gammaslash^{\circ}\big)^C_{AB}$ denote the spherical Christoffel symbols for the Minkowski metric (on the initial incoming cone). 
    
    \vspace{3mm}
    
    \noindent For the Vlasov part, define $\slashed{T}$ to be the restriction of the energy-momentum tensor defined in equation \eqref{energymomentumtensor} to vector fields  tangent to spheres $S_{u,\ubar}$:
\begin{equation}
\slashed{T}(Y,Z)= Y^A Z^B \int_{\mathcal{P}_{x}}f\hsp p_A \hsp p_B, \hspace{4mm}\text{for}\hspace{2mm} Y=Y^A e_A,\hsp Z=Z^B e_B.
\end{equation}We denote by $\slashed{T}_3, \hsp \slashed{T}_4$ the $S_{u,\ubar}$ 1-forms defined by restricting the $1-$forms $T(e_3,\cdot ), \hsp T(e_4, \cdot)$ to vector fields tangent to the spheres $S_{u,\ubar}$:

\be \Tslash_3 (Y) = Y^A \int_{\mathcal{P}_{x}} f \hsp p_3 \hsp p_A, \hspace{4mm} \Tslash_4(Y) =Y^A \int_{\mathcal{P}_{x}} f\hsp p_4 \hsp p_A.\ee Finally, we define the functions $\Tslash_{33}, \hsp \Tslash_{34}$ and $\Tslash_{44}$ 
by
\be \Tslash_{33}=\int_{\mathcal{P}_{x}} f\hsp p_3 \hsp p_3, \hspace{3mm} \Tslash_{34}=\int_{\mathcal{P}_{x}} f\hsp p_3 \hsp p_4, \hspace{3mm} \Tslash_{44}=\int_{\mathcal{P}_{x}} f\hsp p_4 \hsp p_4. \ee

\subsection{The renormalized structure and Bianchi equations}

We have the following table of Ricci coefficients (here and throughout this paper, $\nabla=D$ will be used to denote the spacetime connection):
\begin{gather} \label{tableone}
\nabla_{e_A}e_B = 
\Gammaslash^C_{AB}\hsp e_C + \frac{1}{2}\chi_{AB}\hsp e_3 +\frac{1}{2}\chibar_{AB}\hsp  e_4, \\ \nabla_{e_A}e_3 = {
\chibar_A}^B e_B - \etabar_A \hsp e_3, \hspace{3mm} \nabla_{e_A}e_4 = {\chi_A}^B \hsp e_B + \etabar_A \hsp e_4, \\ \nabla_3 e_A = 
\big({\chibar_A}^B-e_A(b^B)\big) \hsp e_B + \eta_A \hsp e_3, \hspace{3mm}  \nabla_{e_4}e_A =  {\chi_A}^B \hsp e_B +\etabar_A \hsp e_4, \\ \nabla_{e_3}e_4= 2 \eta^A \hsp e_A, \hspace{3mm} \nabla_4 e_3 = 2\etabar^B \hsp e_B, \\ \nabla_3 e_3 =\omegabar \hsp e_3, \hspace{3mm} \nabla_4 e_4 = 0. \label{tablefinal}
\end{gather}For the massless (and the massless only) Vlasov energy-momentum tensor, there holds $g^{\alpha\beta}T_{\alpha\beta}=0$, whence it follows that the scalar curvature vanishes and the Einstein equations reduce to $R_{\mu\nu}=T_{\mu\nu}$. Using this, the renormalized Bianchi equations take the form:

\begin{equation}
	    \begin{split}
	        \snabla_3 \alpha + \frac{1}{2}\tr\chibar \alpha = &\snabla \hat{\otimes}\beta + 4 \omegabar \alpha - 3\left(\chihat \rho + \Hodge{\chihat} \sigma \right)+ (\zeta+4\eta)\hat{\otimes}\beta \\ & - D_4 R_{AB} + \frac{1}{2}\left(D_B R_{4A} + D_A R_{4B}\right) + \frac{1}{2}\left(D_4 R_{43} - D_3 R_{44}\right)\gslash_{AB},
	    \end{split}
	\end{equation}
	\begin{equation}
	        \snabla_4 \beta + 2 \tr\chi \beta = \slashed{\text{div}} \alpha -2 \omega \beta + \left(\eta - 2 \zeta\right)\cdot \alpha + \frac{1}{2} \left(D_4 R_{4A} - D_A R_{44}\right),
	\end{equation}
	\begin{equation}
	    \snabla_3 \beta + \tr\chibar \beta = \snabla \rho+ \Hodge{\snabla}\sigma + 2\omegabar \beta +2 \chihat \cdot \betabar + 3 \left(\eta \rho+ \Hodge{\eta}\sigma \right) + \frac{1}{2} \left( D_A R_{43} - D_4 R_{3A}\right),
	\end{equation}
	
	\begin{equation}
	    \snabla_4 \sigma + \frac{3}{2} \tr\chi \sigma = - \slashed{\text{div}} \Hodge{\beta} + \frac{1}{2}\hsp \chibarhat \cdot \Hodge{\alpha} - (\zeta+2 \etabar) \cdot \Hodge{\beta} - \frac{1}{4}\left( D_{\mu} R_{4\nu} - D_{\nu}R_{4\mu}\right){\epsilon^{\mu\nu}}_{34},
	\end{equation}
	\begin{equation}
	    \snabla_3 \sigma + \frac{3}{2}\tr\chibar \sigma = -\slashed{\text{div}}\Hodge{\betabar} + \frac{1}{2}\hsp\chihat \cdot \Hodge{\alphabar} -(\zeta + 2\eta)\cdot \Hodge{\betabar} +\frac{1}{4}\left( D_{\mu} R_{3\nu} - D_{\nu}R_{3\mu}\right){\epsilon^{\mu\nu}}_{34},
	\end{equation}
	\begin{equation}
	    \snabla_4 \rho + \frac{3}{2}\tr\chi \rho = \slashed{\text{div}}
        \beta - \frac{1}{2}\chibarhat \cdot \alpha + (\zeta + 2\etabar) \cdot \beta - \frac{1}{4}\left(D_3 R_{44} - D_4 R_{43} \right), 
	\end{equation}
	\begin{equation}
	    \snabla_3 \rho + \frac{3}{2}\tr\chibar \rho = -\slashed{\text{div}}\betabar - \frac{1}{2}\chihat \cdot \alphabar +(\zeta-2\eta)\cdot \betabar + \frac{1}{4} \left(D_3 R_{34} - D_4 R_{33}\right), 
	\end{equation}
	
	\begin{equation}
	    \snabla_4 \betabar + \tr\chi \betabar = -\snabla \rho + \Hodge{\snabla} \sigma + 2 \omega \betabar + 2\chibarhat \cdot \beta -3\left(\etabar \rho - \Hodge{\etabar}\sigma \right) - \frac{1}{2}\left(D_A R_{43} - D_3R_{4A} \right), 
	\end{equation}
	\begin{equation}
	    \snabla_3 \betabar + 2 \tr\chibar\hsp  \betabar = -\slashed{\text{div}}\alphabar -2 \omegabar \hsp  \betabar + \etabar\cdot\alphabar +\frac{1}{2}\left(D_A R_{33} - D_3 R_{3A}\right),
	\end{equation}
	\begin{equation}
	    \begin{split}
	        \snabla_4 \alphabar + \frac{1}{2}\tr\chi \hsp \alphabar = &-\snabla \hat{\otimes}\betabar +4 \omega\hsp \alphabar -3\left(\chibarhat \rho - \Hodge{\chibarhat}\sigma \right) +\left(\zeta - 4\etabar\right)\hat{\otimes}\betabar \\ &- D_3 R_{AB} + \frac{1}{2}\left(D_A R_{3B} + D_B R_{3A}\right) + \frac{1}{2}\left(D_3 R_{34} - D_4 R_{43}\right) \slashed{g}_{AB}.
	    \end{split}
	\end{equation}We introduce several notations. For an $S_{u,\ubar}$ 1-form or a $(0,2) \hsp S_{u,\ubar}-$tensor the left Hodge dual${}^*$ is defined as 
    \[ \prescript{*}{}{\xi}_A = \slashed{\epsilon}_{AB}\gslash^{BC}\xi_C \hspace{3mm} \text{and}\hspace{3mm} \prescript{*}{}{\xi}_{AB}=\slashed{\epsilon}_{AC}\gslash^{CD}\xi_{DB},  \]respectively. Moreover, $\slashed{\epsilon}$ denotes the volume form associated with the metric $\gslash$ and for a $(0,2)$ $S_{u,\ubar}$ tensor $\xi$, \[ \slashed{\epsilon}\cdot \xi = \slashed{\epsilon}^{AB}\xi_{AB}.\]The symmetric traceless product of two $S_{u,\ubar}$ 1 forms is defined by 
    \[ (\xi \hat \otimes \xi^{\prime})_{AB} = \xi_A \hsp \xi^{\prime}_B +\xi_B \hsp \xi^{\prime}_A-\gslash_{AB}(\gslash^{CD}\hsp\xi_C \hsp \xi^{\prime}_D),     \]while the antisymmetric products are defined by \[ \xi \wedge \xi^{\prime} =\slashed{\epsilon}^{AB}\xi_A \xi^{\prime}_B \hspace{3mm} \text{and} \hspace{3mm} \xi \wedge \xi^{\prime} =\slashed{\epsilon}^{AB}\gslash^{CD}\xi_{AC}\xi^{\prime}_{BD}, \]for two $S_{u,\ubar}$ one-forms and $S_{u,\ubar}$ $(0,2)-$tensors respectively. Also, \[(\xi \times \xi^{\prime})_{AB}=\gslash^{CD}\xi_{AC}\xi^{\prime}_{BD},\] for $S_{u,\ubar}$ $(0,2)-$tensors $\xi,\xi^{\prime}.$ The symmetric trace-free derivative of an $S_{u,\ubar}$ $1-$form is defined as 
\[ (\snabla\hat{\otimes}\xi)_{AB} =\snabla_A \xi_B +\snabla_B \xi_A -(\sdiv \xi)\gslash_{AB}.      \]Finally, define the $\gslash$-inner product of two $(0,n)$ $S_{u,\ubar}$-tensors by \[ (\xi,\xi^{\prime}) = \gslash^{A_1B_1}\cdot \gslash^{A_nB_n}\xi_{A_1\dots A_n}\xi^{\prime}_{B_1\dots B_n}. \]The norm of a $(0,n)$ $S_{u,\ubar}-$tensors is then given by \[ \lvert \xi\rvert ^2 = \gslash^{A_1B_1}\dots \gslash^{A_nB_n}\xi_{A_1\dots A_n}\xi_{B_1\dots B_n}.    \]The notation $\lvert \cdot \rvert$ will also later be used, when applied to components of $S_{u,\ubar}$ tensors, to denote the standard absolute value on $\mathbb{R}$. It will always be clear from the context which one is meant. For example, if $\xi$ is an $S_{u,\ubar}$ $1-$form, then $\lvert \xi \rvert$ denotes the $\gslash$ norm as above, whilst $\lvert \xi_A \rvert$ denotes the absolute value of $\xi(e_{A})$. We now give the null structure and the metric quantities below.

\be
\snabla_3 \etabar +\frac{1}{2}\tr\chibar\hsp \etabar = \frac{1}{2}\tr\chibar \eta -\chibarhat\cdot(\etabar-\eta)+\betabar-\frac{1}{2}\Te_3, 
\ee
\be
    \snabla_3 \tr\chibar +\frac{1}{2}(\tr\chibar)^2 = -\lvert \chibarhat \rvert^2 -\Te_{33}, 
\ee
\be
\snabla_3 \chibarhat +\tr\chibar \hsp \chibarhat =\omegabar \chibarhat-\alphabar, 
\ee
\be \snabla_3\big(1-\frac{1}{\Omega^2}\big)= \omegabar -\omegabar\hsp \big(1-\frac{1}{\Omega^2}\big), \ee

\be \snabla_4\omegabar = 2\hsp \eta\cdot \etabar-\lvert \eta \rvert^2-\rho-\f12 \Te_{34}, \ee
\be
\snabla_4 \eta+\f12 \tr\chi \hsp \eta = \frac{1}{2}\tr\chi \hsp \etabar -\chihat\cdot(\eta-\etabar)-\beta-\f12 \Te_4,
\ee
\be
\snabla_4 \tr\chi+\frac{1}{2}(\tr\chi)^2 = -\lvert \chihat \rvert^2 -\Te_{44},
\ee
\be
    \snabla_4 \chihat +\tr\chi \hsp \chihat = -\alpha,
\ee 
\be \snabla_4 b = 2(\etabar-\eta)+\chi\cdot b, \ee 
The additional propagation equations for $\chibarhat$, $\tr\chibar$,

\be \snabla_4 \chibarhat+\frac{1}{2}\tr\chi\hsp \chibarhat = \snabla \hat{\otimes}\etabar -\frac{1}{2}\tr\chibar \chihat +\etabar \hat{\otimes}\etabar +\hat{\slashed{T}}, \ee
\be \snabla_4 \tr\chi = 2 \hsp \sdiv\etabar + 2 \lvert \etabar \rvert^2 +\chihat \cdot \chibarhat +\frac{1}{2}\tr\chi\tr\chibar +\rho-\frac{1}{2}\Te_{34}, \ee will also be used later in Section \ref{ellipticsection} to derive a propagation equation for the (renormalized) aspect function $\mu$. Here $\hat{\Te}=\Te-\Te_{34}\gslash$ is the trace-free part of $\Te$. The following first variational formula for the induced metric on the spheres will also be used,

\be \mathcal{L}_{e_3}\gslash=2\chibar, \ee \be \mathcal{L}_{e_4}\gslash = 2\chi, \ee where $\mathcal{L}$ denotes the Lie derivative.
\vspace{3mm}

\noindent \textbf{An equation for $\Gamma^A_{3B}$} 

\vspace{3mm}
\noindent Adjoint to the above equations, the equation for $\Gamma^A_{3B}= {\chibar_A}^B-e_A(b^B)$ will be pivotal to our analysis. We derive it below. First of all, there holds

\begin{align}
\Gamma^A_{3B}= \gslash^{AC} \langle \nabla_3 e_B, e_C\rangle.
\end{align}Therefore,

\be \nabla_4 \Gamma^A_{3B}= \langle \nabla_3 e_B, e_C\rangle \hsp\nabla_4 (\gslash^{AC}) +\gslash^{AC}(\langle \nabla_4 \nabla_3 e_B,e_C)+\langle \nabla_3 e_B, \nabla_4 e_C\rangle). \ee
 We have \[\gslash^{AC}\langle \nabla_4 \nabla_3 e_B, e_C\rangle = \gslash^{AC}\big(\langle \nabla_3\nabla_4 e_B,e_C\rangle +R_{43BC} +\langle \nabla_{[e_4,e_3]}e_B, e_C\rangle\big).\] However,
\begin{itemize}
\item The third term \[ [e_4,e_3] = 2(\etabar^D-\eta^D)e_D+\omegabar \hsp e_4,\]hence \be \gslash^{AC} \langle\nabla_{[e_4,e_3]}e_B,e_C\rangle = \gslash^{AC} \big(2(\etabar^D-\eta^D)\hsp \Gammaslash^E_{DB}\hsp \gslash_{EC} +\omegabar \hsp \chi_{BC}\big) =2(\etabar^D-\eta^D)\hsp \Gammaslash^A_{DB}+ \omegabar \hsp {\chi_B}^A , \ee
\item The second term  $\gslash^{AC} R_{43BC}={R^A}_{B34}$,

\item The first term

\begin{align}\gslash^{AC}
\langle \nabla_3 \nabla_4 e_B, e_C \rangle =& \gslash^{AC} \langle \nabla_3({\chi_B}^D e_D +\etabar_B e_4), e_C\rangle = \gslash^{AC}( \gslash_{DC}\nabla_3 {\chi_B}^D+{\chi_B}^D \Gammaslash^{E}_{3D}\gslash_{EC}+\etabar_B \eta^D\gslash_{DC})\notag \\ =& \nabla_3 {\chi_B}^A +{\chi_B}^D \Gammaslash^A_{3D}+\etabar_B \eta^A. 
\end{align}
\end{itemize}For the term $\gslash^{AC}\langle \nabla_3 e_B, \nabla_4 e_C\rangle$ a straightforward computation gives \be  \gslash^{AC} \hsp \langle \nabla_3 e_B, \nabla_4 e_C\rangle = {\chi_D}^A \hsp \Gammaslash^D_{3B}- 2\hsp \eta_B\hsp  \etabar^A \ee 
Putting everything together, we get the important equation

\begin{align}
\nabla_4 \Gamma^A_{3B}= -{\chi_D}^{A}\Gamma^D_{3B} +\nabla_3 {\chi_B}^A+{\chi_B}^D \Gamma^A_{3D} +\etabar_B \eta^A +{R^A}_{B34}+2(\etabar^D-\eta^D)\Gammaslash^A_{DB}+\omegabar \hsp{\chi_B}^A -2\eta_B \etabar^A.
\end{align}Through it, one can avoid having to obtain elliptic estimates for most of the Ricci coefficients as well as for $b$ itself. Indeed, the crucial observation is that $\nabla_3 \chi$ contains only $\snabla \eta$ in it.

\subsection{Integration and $L^{p}$ Norms}

Let $(\mathcal{M},g)$ be a smooth Lorentzian manifold foliated by a double null gauge, with null coordinates $(u,\ubar,\theta^1,\theta^2)$ and associated null hypersurfaces 
\[
H_{u} := \{ (u,\ubar,\theta^1,\theta^2) : \ubar \geq 0 \}, 
\qquad \Hbar_{\ubar} := \{ (u,\ubar,\theta^1,\theta^2) : u \geq u_\infty \}.
\] 
For each pair $(u,\ubar)$, we denote by $S_{u,\ubar}$ the spacelike $2$–sphere of intersection 
\[
S_{u,\ubar} := H_{u} \cap \Hbar_{\ubar}.
\] 
The induced Riemannian metric on $S_{u,\ubar}$ is denoted by $\gamma = \gslash_{AB}\,d\theta^A \otimes d\theta^B$, with determinant $\det\gslash$.  
 
 \vspace{3mm}
 
\noindent Fix a smooth coordinate chart $U \subset S_{u,\ubar}$ and a smooth partition of unity $\{p_U\}$ subordinate to such charts. For a measurable scalar function $\phi : S_{u,\ubar} \to \mathbb{R}$ we define the surface integral by
\begin{equation}\label{eq:surface_integral}
    \int_{S_{u,\ubar}} \phi 
    := \sum_{U} \int_{U} \phi(\theta^1,\theta^2)\, p_U(\theta^1,\theta^2)\, \sqrt{\det\gslash(\theta)} \, d\theta^1 d\theta^2,
\end{equation}
which is independent of the chosen partition.
The induced metric on the null hypersurface is degenerate. We define the volume form on $H_u$ induced by $(g,\gamma)$ as follows
\[
d\mu_{H_u} = 2\Omega \, \sqrt{\det\gslash} \, d\ubar\, d\theta^1 d\theta^2,
\]
where $\Omega$ denotes the null lapse. For $\phi : H_u \to \mathbb{R}$, we set
\begin{equation}\label{eq:null_integral}
    \int_{H_u} \phi 
    := \sum_{U} \int_{0}^{\ubar} \int_{U} \phi(u,\ubar',\theta)\, p_U(\theta)\, 2\Omega(u,\ubar',\theta)\, \sqrt{\det\gslash(u,\ubar',\theta)} \, d\theta^1 d\theta^2\, d\ubar'.
\end{equation}
Similarly, on the incoming null hypersurface $\Hbar_{\ubar}^{(u_\infty,u)} := \{ \ubar = \mathrm{const}, \, u_\infty \leq u' \leq u \}$, we define the volume form as
\[
d\mu_{\Hbar_{\ubar}} = 2\Omega \, \sqrt{\det\gslash}\, du\, d\theta^1 d\theta^2,
\]
and we define
\begin{equation}
    \int_{\Hbar_{\ubar}^{(u_\infty,u)}} \phi 
    := \sum_{U} \int_{u_\infty}^{u} \int_{U} \phi(u',\ubar,\theta)\, p_U(\theta)\, 2\Omega(u',\ubar,\theta)\,\sqrt{\det\gslash(u',\ubar,\theta)} \, d\theta^1 d\theta^2\, du'.
\end{equation}
  For the spacetime slab
\[
\mathcal{D}_{u,\ubar} := \{ (u',\ubar',\theta^1,\theta^2) \mid u_\infty \leq u' \leq u,\ 0 \leq \ubar' \leq \ubar \},
\]
the natural $4$–volume form induced by $g$ in double null coordinates is
\[
d\mu_{g} = \Omega^2 \sqrt{\det \gslash}\, du\, d\ubar\, d\theta^1 d\theta^2.
\]
Accordingly, for $\phi : \mathcal{D}_{u,\ubar} \to \mathbb{R}$, we set
\begin{equation}
    \int_{\mathcal{D}_{u,\ubar}} \phi 
    := \sum_{U} \int_{u_\infty}^{u} \int_{0}^{\ubar} \int_{U} \phi(u',\ubar',\theta)\, p_U(\theta)\, \Omega^2(u',\ubar',\theta)\, \sqrt{\det \gslash(u',\ubar',\theta)} \, d\theta^1 d\theta^2\, d\ubar'\, du'.
\end{equation}
 Let $\phi$ be a tensor field tangent to $S_{u,\ubar}$, and let $\langle \cdot,\cdot\rangle_\gslash$ denote the pointwise inner product induced by $\gslash$. For $1 \leq p < \infty$ we define
\begin{align}
    \|\phi\|_{L^p(S_{u,\ubar})}^p &:= \int_{S_{u,\ubar}} \langle \phi,\phi \rangle_\gslash^{\frac{p}{2}}, \\
    \|\phi\|_{L^p(H_u)}^p &:= \int_{H_u} \langle \phi,\phi \rangle_\gslash^{\frac{p}{2}}, \\
    \|\phi\|_{L^p(\Hbar_{\ubar})}^p &:= \int_{\Hbar_{\ubar}} \langle \phi,\phi \rangle_\gslash^{\frac{p}{2}}.
\end{align}
For the case $p=\infty$, we set
\begin{equation}
    \|\phi\|_{L^\infty(S_{u,\ubar})} := \sup_{(\theta^1,\theta^2)\in S_{u,\ubar}} \langle \phi,\phi \rangle_\gslash^{1/2}(\theta^1,\theta^2).
\end{equation}
All norms are understood with respect to the appropriate induced measures defined above.

\subsection{Signature for decay rates and scale-invariant norms}

\noindent Perhaps the most challenging aspect of trapped surface formation results, historically, has been the attempt to find initial data that are, in an appropriate sense, large (this is by necessity, as is implied by the monumental work of \cite{ChrKl}) but also small enough to allow for an existence result of a spacetime region that gives trapped surfaces the time they would require to form. The first such initial data set, in the absence of symmetries, was given by \cite{C09}. Later contributions include \cite{Kl-Rod}, \cite{AL17} and \cite{A17}. Moreover, one would have to construct norms which preserve, at least approximately, the hierarchy present in the initial data upon evolution of the Einstein equations. The signature for decay rates, which was first introduced in \cite{AnThesis}, is the tool we will use in the present paper to build \textit{scale-invariant norms}. These will be norms that, upon evolution of the initial data, remain bounded above by a uniform constant (with the exception of a few anomalous terms). For another application of this framework, see \cite{AnAth}.

\vspace{3mm}

\noindent To each $\phi \in \begin{Bmatrix}
\alpha, \alphabar, \tbeta, \tbetabar, \rho, \sigma, \eta, \etabar, \chi, \chibar, \omega, \omegabar, \zeta,\gamma
\end{Bmatrix}$ we associate its \textit{signature for decay rates} $s_2(\phi)$:

\[ s_2(\phi) = 0\cdot N_4(\phi) + \frac{1}{2}N_A(\phi) + 1\cdot N_3(\phi)-1.\]Here $N_\alpha(\phi)$ $(\alpha = 1,2,3,4)$ denotes the number of times $e_\alpha$ appears in the definition of $\phi$. We also extend the definition of $s_2$ in the same way to the set of energy-momentum tensor components. For the Ricci coefficients and curvature components, we get the following table of signatures.

{\renewcommand{\arraystretch}{1.25}
\begin{center}
\begin{tabular}{||c || c c c c c c c c c c c c c c c||} 
 \hline$\phi$ &
 $\alpha$ & $\alphabar$ & $\beta$ & $\betabar$ &$\rho$ &$\sigma$ & $\eta$ & $\etabar$ & $\chi$ &$\chibar$ & $\omegabar$ & $\zeta$ & $\gslash$ & $\modu$ \\ [1ex]  \hline\hline
 $s_2(\phi)$ & 0 & 2 & 0.5 & 1.5 & 1  & 1 & 0.5 & 0.5  &0 & 1 & 1 &0.5 &0 & -1 \\ 
 \hline
\end{tabular}
\end{center}Accordingly, for the matter components, we get the following table of signatures: 

\begin{center}
\begin{tabular}{||c || c c c c c c||}
\hline $T_{\mu\nu}$ & $\Te_{44}$ & $\Te_{34}$ & $\Te_{33}$ & $\Te_3$ & $\Te_4$ & $\Te$ \\ \hline\hline $s_2(T_{\mu\nu})$ & 0& 1 & 2 & $\f32$ & $\f12$ & 1 \\ \hline
\end{tabular}
\end{center}

\par\noindent Several properties of $s_2$ follow:

\[ s_2(\snabla_4 \phi) = s_2(\phi), \]\[s_2(\snabla \phi) = s_2(\phi) +\frac{1}{2} \hspace{2mm} s_2(\snabla_3 \phi) = s_2(\phi) + 1.\]Finally, perhaps the most important property of $s_2$ is \textit{signature conservation}: \be \label{sc} s_2(\phi_1 \cdot\phi_2) = s_2(\phi_1)+ s_2(\phi_2), \hspace{2mm} \ee This allows for the (almost)-preservation of the scale-invariant norms upon evolution, as we shall see. 

\vspace{3mm}
\noindent  {\begin{center}\textbf{An aside: Assigning signature to $\modu, \frac{1}{\modu}$}\end{center}}

\vspace{3mm}

\noindent In the paper \cite{A19} that introduced the current scale-invariant framework, the only type of derivative that required control at the level of Bianchi or structure equations was that of an angular derivative $\snabla$. Indeed, in vacuum, there exists no extra inhomogeneity in the Bianchi equations stemming from matter error terms. Therefore, to obtain control on the Ricci coefficients up to top order (and hence to obtain energy control of the curvature components), it was enough to only use angular derivatives\footnote{Note that, strictly speaking, the spacetime existence theorem only requires control on angular derivatives of the curvature components. However, in non-vacuum models and in particular the Einstein-Yang-Mills system, for example, obtaining such control and closing the energy arguments requires control on $\snabla_3\alphabar_F, \snabla_4 \alpha_F$, where $\alphabar_F$ and $\alpha_F$ correspond to the Yang-Mills source term (see \cite{A-M-Y}).}. In our current setting, obtaining control on two derivatives of curvature requires control on the Ricci coefficients up to second order (up to third order for $\eta,\chihat,\tr\chi$), which in turn (through the use of the structure equations) necessitates control on the matter terms and ultimately on $\De^i T_{\mu\nu}$, where $\De= \{\modu \snabla_3, \snabla_4 , \al \snabla \}$ and $i=0,\dots,3$. Therefore, $\modu$ appears as part of error terms for various equations and it is useful to assign a signature to it. Because \[ \snabla_3 \modu = -1,\]the function $\modu$ inherits a signature $s_2(\modu)=-1$. Similar considerations\footnote{Interestingly, the term $\frac{1}{\modu}$ does not appear as part of any error terms, but that is only because of the frame choice. In the frame $e_A=\partial_{\theta^A}, e_{3} := \partial_u+ b^A \hsp \partial_{\theta^A}, \hspace{3mm} e_4 = \frac{1}{\Omega^2}\partial_v $, say, terms like $\frac{\Omega^{-1}-1}{\modu^2}$ would come up (this particular one would appear in the structure equation for $\tildetr$, for example). This is another simplification arising from our choice of frame.} imply $s_2\big(\frac{1}{\modu}\big)=1$.

\vspace{3mm}

\noindent For any horizontal tensor-field or stress-energy-momentum tensor component $\phi$, we define the following norms:
\begin{equation}
    \scaleinfinitySu{\phi} := a^{-s_2(\phi)} \lvert u \rvert^{2s_2(\phi)+1}\inftySu{\phi},
\end{equation}
\begin{equation}
       \scaletwoSu{\phi } := a^{-s_2(\phi )} \lvert u \rvert^{2s_2(\phi)}\twoSu{\phi},
\end{equation}
and more generally 
\begin{align}
    ||\varphi||_{L^{p}(S_{u,\ubar})}=\frac{a^{s_{2}(\varphi})}{|u|^{2s_{2}(\varphi)+\frac{p-2}{p}}}||\varphi||_{L^{p}_{sc}(S_{u,\ubar})},~~1\leq p<\infty  
\end{align}
Notice the difference in the $u$-weights amongst the definitions. 

\vspace{3mm}

\noindent A crucial property of the above norms is the \textit{scale-invariant H\"older's inequalities} that they satisfy. For $\Y$ denoting an arbitrary $S-$tensor field, the following hold:
\begin{equation}
||\Y_1 \cdot \Y_2||_{L^{1}_{sc}(S_{u,\ubar})} \leq \frac{1}{\lvert u \rvert} \scaletwoSu{\Y_1}\scaletwoSu{\Y_2},
\end{equation}
\begin{equation}
   ||\Y_1 \cdot \Y_2||_{L^{1}_{sc}(S_{u,\ubar})} \leq \frac{1}{\lvert u \rvert} \scaleinfinitySu{\Y_1}||\Y_2||_{L^{1}_{sc}(S_{u,\ubar})},
\end{equation}
\begin{equation}
    \label{257}\scaletwoSu{\Y_1 \cdot \Y_2} \leq \frac{1}{\lvert u \rvert} \scaleinfinitySu{\Y_1}\scaletwoSu{\Y_2},
\end{equation}
and 
\begin{equation}
||\Y_1 \cdot \Y_2||_{L^{2}_{sc}(S_{u,\ubar})}\lesssim \frac{1}{\lvert u \rvert}||\Y_1||_{L^{4}_{sc}(S_{u,\ubar})}    ||\Y_2||_{L^{4}_{sc}(S_{u,\ubar})}.    
\end{equation}
Notice that this is possible partly thanks to the signature conservation property \eqref{sc}. In the region of study, the factor $\frac{1}{\lvert u\rvert}$ plays the role that $\delta^{\frac{1}{2}}$ plays in the definition of the corresponding norms\footnote{See Section 2.17 in \cite{Kl-Rod}.} in [Kl-Rod], namely that of measuring the \textit{smallness} of the nonlinear terms. The above inequalities are the primary tools that will be used to close the bootstrap argument required for the existence part.

\subsection{Norms} 
\label{bootstrapassumption}
Following An's \cite{AnThesis}, we define $\prescript{(S)}{}{\mathcal{O}_{s,p}}$ norms on a given $S_{u,\ubar}$. 

\begin{align}
\prescript{(S)}{}{\mathcal{O}_{0,\infty}}(u,\ubar) := &\scaleinfinitySu{\eta,\etabar, \omega }+\frac{1}{\al}\scaleinfinitySu{\chihat}+\frac{\al}{\modu}\scaleinfinitySu{\chibarhat}\notag \\ +&\frac{a}{\modu}\scaleinfinitySu{\tildetr} +\frac{a}{\modu^2}\scaleinfinitySu{\tr\chibar}, 
\end{align}

\begin{align}
\prescript{(S)}{}{\mathcal{O}_{0,4}}(u,\ubar) := &\scalefourSu{\eta,\etabar, \omega }+\frac{1}{\al}\scaleinfinitySu{\chihat}+\frac{\al}{\modu}\scalefourSu{\chibarhat}\notag \\ +&\frac{a}{\modu}\scalefourSu{\tildetr} +\frac{a}{\modu^2}\scalefourSu{\tr\chibar}, 
\end{align}

\begin{align}
\prescript{(S)}{}{\mathcal{O}_{1,4}}(u,\ubar) := \sum_{\mathcal{D}\in\{\modu\snabla_3,\snabla_4, \al\snabla\}} &\scalefourSu{\mathcal{D}(\eta,\etabar, \omega) }+\frac{1}{\al}\scalefourSu{\mathcal{D}\chihat}+\frac{\al}{\modu}\scalefourSu{\mathcal{D}\chibarhat}\notag \\ +&\frac{a}{\modu}\scalefourSu{\mathcal{D}\tildetr} +\frac{a}{\modu^2}\scalefourSu{\mathcal{D}\tr\chibar},
\end{align}

\begin{align}
\prescript{(S)}{}{\mathcal{O}_{2,2}}(u,\ubar) := \sum_{\mathcal{D}\in\{\modu\snabla_3,\snabla_4, \al\snabla\}} &\scaletwoSu{\mathcal{D}^2(\eta,\etabar, \omegabar) }+\frac{1}{\al}\scaletwoSu{\mathcal{D}^2\chihat}+\frac{\al}{\modu}\scaletwoSu{\mathcal{D}^2\chibarhat}\notag \\ +&\frac{a}{\modu}\scaletwoSu{\mathcal{D}^2\tildetr} +\frac{a}{\modu^2}\scaletwoSu{\mathcal{D}^2\tr\chibar}. 
\end{align}We also define the elliptic norm
\begin{align}
\prescript{(S)}{}{\mathcal{O}_{3,2}}(u,\ubar) := \sum_{\mathcal{D}\in\{\modu\snabla_3,\snabla_4, \al\snabla\}} &\scaletwoHu{\mathcal{D}^2 \snabla(\eta,\tr\chi,\frac{1}{\al}\chihat)}.
\end{align}\noindent We denote by $\mathcal{O}_{0,\infty}, \mathcal{O}_{0,4}, \mathcal{O}_{1,4}, \mathcal{O}_{2,4}, \mathcal{O}_{3,2}$ the suprema of the corresponding norms, respectively, over all values $(u,\ubar)$, where $u_{\infty}\leq u\leq -a/4$ and $0\leq \ubar\leq 1$.  Finally, we define the total Ricci norm

\[  \mathcal{O}:= \mathcal{O}_{0,\infty}+ \mathcal{O}_{0,4}+
\mathcal{O}_{1,4}+ \mathcal{O}_{2,2}+ \mathcal{O}_{3,2},  \]while by $\mathcal{O}^{(0)}$ we denote the total norm on the initial hypersurface $H_{u_{\infty}}$. Moreover, we introduce the following norms on curvature that will be estimated.

\begin{align}
\mathcal{R}_0 = \frac{1}{\al}\scaletwoHu{\alpha} + \scaletwoHu{(\tbeta,\rho,\sigma,\tbetabar)},
\end{align}
\begin{align}
\underline{\mathcal{R}}_0=\frac{1}{\al}\scaletwoHbaru{\tbeta}+\scaletwoHbaru{(\rho,\sigma,\tbetabar,\alphabar)},
\end{align}
\begin{align}
\mathcal{R}_1 := \sum_{\mathcal{D}\in \{ \modu \snabla_3, \snabla_4,\al\snabla \}  } \frac{1}{\al}\scaletwoHu{\mathcal{D}\alpha} + \scaletwoHu{\mathcal{D}(\tbeta,\rho,\sigma,\tbetabar)},\end{align}
\begin{align}
\underline{\mathcal{R}}_1=\sum_{\mathcal{D}\in \{ \modu \snabla_3, \snabla_4,\al\snabla \}}  \frac{1}{\al}\scaletwoHbaru{\mathcal{D}\tbeta}+\scaletwoHbaru{\mathcal{D}(\rho,\sigma,\tbetabar,\alphabar)},
\end{align}
\begin{align}
\mathcal{R}_2 := \sum_{\mathcal{D}\in \{ \modu \snabla_3, \snabla_4,\al\snabla \}  } \frac{1}{\al}\scaletwoHu{\mathcal{D}^2\alpha} + \scaletwoHu{\mathcal{D}^2(\tbeta,\rho,\sigma,\tbetabar)},
\end{align}
\begin{align}
\underline{\mathcal{R}}_2=\sum_{\mathcal{D}\in \{ \modu \snabla_3, \snabla_4,\al\snabla \}}  \frac{1}{\al}\scaletwoHbaru{\mathcal{D}^2\tbeta}+\scaletwoHbaru{\mathcal{D}^2(\rho,\sigma,\tbetabar,\alphabar)},
\end{align}
\noindent We set $\mathcal{R}_0, \mathcal{R}_1$ and $\mathcal{R}_2$ to be the supremum over all $(u,\ubar)$ in the spacetime slab of $\mathcal{R}_0(u,\ubar), \mathcal{R}_1(u,\ubar)$ and $\mathcal{R}_2(u,\ubar)$ respectively. Similarly, we define $\underline{\mathcal{R}}_0, \underline{\mathcal{R}}_1$ and $\underline{\mathcal{R}}_2$. We define the total curvature norms \be \mathcal{R}= \mathcal{R}_0+\mathcal{R}_1+\mathcal{R}_2, \hspace{3mm} \underline{\mathcal{R}}= \underline{\mathcal{R}}_0+ \underline{\mathcal{R}}_1+\underline{\mathcal{R}}_2.\ee Finally, we define $\mathcal{R}^{(0)}$ to be the initial value of the norm:

\be \mathcal{R}^{(0)} :=\sup_{0\leq u \leq 1}\big(\mathcal{R}_0(u_{\infty},u)+ \mathcal{R}_1(u_{\infty},u)+ \mathcal{R}_2(u_{\infty},u)\big).   \ee

\noindent For the Vlasov norm, we define, similarly, the following quantities:
\be \mathcal{V}_0 = \frac{\modu^2}{a} \lvert \Te_{44}\rvert + \frac{\modu^4}{a}\lvert \Te_{43} \rvert+\frac{\modu^6}{a}\lvert \Te_{33}\rvert + \frac{\modu^5}{a}\lvert \Te_5 \rvert + \frac{\modu^3}{a}\lvert \Te_3\rvert + \frac{\modu^4}{a}\lvert \Te \rvert .  \ee For $j=1,2,3,$ let us first introduce the following vector fields:

\begin{gather}
V_{(A)}:= Hor(e_A)- \frac{p^3}{\lvert u \rvert} \partial_{\pbar^A}, \hsp V_{(3)} := Hor(e_4), \hspace{3mm} V_{(4)} := p^3 \partial_{\pbar^3} - \lvert u \rvert Hor(e_3), \hspace{3mm} V_{(4+A)}:= \frac{p^3}{\lvert u \rvert^2}\partial_{\pbar^A},
\end{gather}together with the vector field \be V_{(0)}:= \lvert u \rvert Hor(e_3). \ee
We define the set 

\be \label{tildeVfirstappearance} \widetilde{V} := \{ \modu V_{0}, \modu V_{(4+A)}, V_{(A)}, V_{(3)},V_{(4)} \}. \ee We introduce the Vlasov norms

\be \mathfrak{V}_1(u,\ubar) := \frac{\modu^2}{a^2} \sum_{\widetilde{V}_{(j)}\in \widetilde{V}} \sup_{S_{u,\ubar}} \int_{\mathcal{P}_x}\lvert \widetilde{V}_{(j)}f \rvert^2 \hsp \sqrt{\det\gslash}\hsp \frac{\text{d}p^1\text{d}p^2\text{d}p^3}{p^3},  \ee
\begin{eqnarray}
\mathfrak{V}_{2}(u,\ubar):=\frac{1}{a^{2}}\sum_{V_{(1)},V_{(2)}\in \tilde{V}}\int_{S_{u,\ubar}}\nonumber \int_{\mathcal{P}_{x}}|\widetilde{V}_{(1)}\widetilde{V}_{(2)} f|^{2}\sqrt{\det\gslash}\frac{dp^{1}dp^{2}dp^{3}}{p^{3}}\sqrt{\det\gslash}\hsp d^{2}\theta 
\end{eqnarray} and 
\begin{eqnarray}
\mathfrak{V}_{3}(u):=\frac{1}{a^{2}}\sum_{\tilde{V}_{(1)},\tilde{V}_{(2)}, \tilde{V}_{(3)}\in \tilde{V}}\int_0^1 \int_{S_{u,\ubar^{\prime}}}\nonumber \int_{\mathcal{P}_{x}}|\tilde{V}_{(1)}\tilde{V}_{(2)} \tilde{V}_{(3)} f|^{2}\sqrt{\det\gslash}\hsp \frac{dp^{1}dp^{2}dp^{3}}{p^{3}}\sqrt{\det\gslash}\hsp d^{2}\theta \dubarprime.
\end{eqnarray}
We let $\mathcal{V}_{(i)}$, for $i=0$ to $3$, denote the suprema of the corresponding-order Vlasov norms respectively over $(u,\ubar)$ in the spacetime slab and define the total norm


\[\mathcal{V}= \mathcal{V}_0 +\mathcal{V}_1+\mathcal{V}_2+\mathcal{V}_3.\]We proceed to make the fundamental bootstrap assumptions:\be \label{boundsbootstrap}\mathcal{O}\leq O, \hspace{3mm} \mathcal{R}+\underline{\mathcal{R}}\leq R, \hspace{3mm} \mathcal{V}\leq V, \ee where\footnote{Let us note that throughout the text, we might also use $\Gamma$ as a symbol for $O$} $O,R,V$ are large constants such that, however, $(O+R+V)^{20}\ll a^{\frac{1}{16}}.$ 

\vspace{3mm}

\section{Preliminary estimates} \label{s3}
\subsection{Preliminary bootstrap assumptions}
\label{section:s31}
\noindent We will be employing a bootstrap argument to obtain a priori bounds on $\Gamma$, $\mathcal{R}$ and $\mathcal{V}$. Along the initial hypersurfaces $H_{u_{\infty}}$ and $\Hbar_0$, an analysis of the initial data using transport equations (see for example, \cite{luklocal}) yields
 
 \[ \Gamma_0 + \mathcal{R}_0 + \mathcal{V}_0 \lesssim \mathcal{I}. \]Our goal is to show that in the entire region \[ \mathcal{D}:= \begin{Bmatrix} (u,\ubar,\theta^1, \theta^2) \hspace{1mm} \mid \hspace{1mm} u_{\infty}\leq u \leq -\frac{a}{4}, \hsp 0\leq \ubar \leq 1 \end{Bmatrix}\] there exists a constant $c(\mathcal{I}) = \mathcal{I}^4+\mathcal{I}^2 +\mathcal{I}+1$ such that \[ \Gamma + \mathcal{R} + \mathcal{V} \lesssim c(\mathcal{I}). \] We assume, as a bootstrap assumption the following:
 
 \begin{equation}\label{bootstrap}
 \Gamma \leq \Gamma, \hspace{2mm}\mathcal{R}\leq R, \hspace{2mm} \mathcal{V} \leq M,
 \end{equation}where $\Gamma, \hsp R$ and $M$ are large so that 
 
 \[ \mathcal{I}^4 +\mathcal{I}^2 +\mathcal{I}+1 \ll \text{min}\begin{Bmatrix}
 \Gamma, \hsp R,\hsp M\end{Bmatrix}\]but also such that \[ (\Gamma+R+M)^{20}\leq a^{\frac{1}{16}}.\]

\noindent We first show that $\Omega$ is always close to 1 in the supremum norm.

\begin{proposition}\label{Omegabounds}
There holds \[ \big\lVert 1- \frac{1}{\Omega^2} \big\rVert_{L^{\infty}(S_{u,\ubar})} \lesssim \frac{a \hsp O^2}{\modu^2}       \] everywhere in the slab of existence. 
\end{proposition}
\begin{proof}
Using the auxiliary bootstrap assumption $\big\lVert 1- \frac{1}{\Omega^2} \big\rVert_{L^{\infty}(S_{u,\ubar})} \leq O,$ we use the equation 
\be \snabla_3\big(1-\frac{1}{\Omega^2}\big) = \omegabar -\omegabar \big(1-\frac{1}{\Omega^2}\big), \ee we obtain

\be \big\lVert 1- \frac{1}{\Omega^2} \big\rVert_{L^{\infty}(S_{u,\ubar})} \lesssim \intu \lVert \omegabar \rVert_{L^{\infty}(S_{u^\prime,\ubar})}(1+O) \duprime \leq  \intu \frac{a \hsp O}{\upr^3}(1+O)\duprime \leq \frac{a\hsp O^2}{\modu^2}. \ee Here we have made use of the bootstrap assumption $\scaleinftySu{\omegabar}= a^{-1}\modu^3 \inftySu{\omegabar}\leq O$.
\end{proof}
\noindent We proceed to obtain control on the induced metric $\gslash$ on $S_{u,\ubar}$.
\begin{proposition} \label{propdetgslash}
Under the assumptions of Theorem \ref{mainone} and the bootstrap assumptions \eqref{boundsbootstrap}, for the induced metric $\gslash$ on $S_{u,\ubar}$ we have 
\[   c^{\prime}\leq \frac{\det \gslash(u,\ubar,\theta^1,\theta^2)}{\modu^4} \leq C^{\prime}.    \]Here $c^{\prime}$ and $C^{\prime}$
are constants depending only on the initial data. Moreover, throughout the slab of existence there holds 

\[ \lvert \gslash_{AB}(u,\ubar,\theta^1,\theta^2)\rvert \leq C^{\prime} \modu^2, \hspace{3mm}\lvert \gslash^{AB}(u,\ubar,\theta^1,\theta^2) \rvert \leq \frac{C^{\prime}}{\modu^2} .\]     \end{proposition}

\begin{proof}
We argue through the first variational formula $\slashed{\mathcal{L}}_{e_4}\gslash_{AB} = 2\chihat_{AB}$. This implies 

\begin{align}
\frac{\partial}{\partial \ubar} \log(\det \gslash) = 2 \tr\chi. 
\end{align}Let $\gslash_0(u,\ubar,\theta^1,\theta^2)=\gslash(u,0,\theta^1,\theta^2).$ Then with $\lvert 2\tr\chi \rvert \leq O/\modu$, it follows

\[  \frac{\det\gslash}{\det\gslash_0} = \mathrm{e}^{\intubar 2\tr\chi \dubarprime}\leq \mathrm{e}^{\frac{O}{a}}.   \]By Taylor expansion, this implies

\[  \lvert \det \gslash- \det\gslash_0\rvert \leq \det\gslash_0\lvert 1 -\mathrm{e}^{\frac{O}{a}}\rvert\lesssim \frac{O}{a}\det\gslash_0,   \] which gives the upper and lower bounds for $\det\gslash$. For $\gslash$, assume a bootstrap assumption of the form 

\[ \lvert \gslash_{AB} \rvert \leq \modu^2 G,     \]where $G$ is a large constant with $G\ll a$.
let $\Lambda$ denote the greater eigenvalue. We have\[ \Lambda \leq \sup_{A,B=1,2}\gslash_{AB} \leq \modu^2 G,\]\[ \sum_{A,B=1,2}\lvert \chi_{AB} \rvert \leq \Lambda \inftySu{\chi},\]\[ \lvert \gslash_{AB}-(\gslash_0)_{AB}\rvert \leq \intubar \lvert \chi_{AB} \rvert \dubarprime \leq \Lambda \frac{\al}{\modu}O \leq \modu^2 \frac{O\hsp G}{\al} \leq \modu^2.   \]Given the value of $(\gslash_0)_{AB}$, this improves the bootstrap assumption and proves the result. We further bound $\gslash^{AB}$ from above by using the upper bound for $\lvert \gamma_{AB}\rvert$ and the lower bound for $\det\gslash$.
\end{proof}
\noindent For the metric $\gslash$, we will also need the following:

\begin{proposition}
\label{31}
We continue to work under the assumptions of Theorem \ref{mainone} and the bootstrap assumptions \eqref{boundsbootstrap}. Fix a point $(u,\theta^1,\theta^2)$ on the initial hypersurface $\Hbar_0$. Along the outgoing null geodesics emitting from $(u,\theta^1,\theta^2),$ denote by $\Lambda(\ubar)$ and $\lambda(\ubar)$ the larger and smaller eigenvalues of $\gslash^{-1}(u,\ubar=0,\theta^1,\theta^2)\gslash(u,\ubar,\theta^1,\theta^2)$ respectively. Then we have

\[ \lvert \Lambda(\ubar)-1\rvert + \lvert \lambda(\ubar)-1\rvert \leq \frac{1}{\al}.   \]\end{proposition}
\begin{proof}
Define $\nu(\ubar):=\sqrt{\frac{\Lambda(\ubar)}{\lambda(\ubar)}}$. Following \cite{A19} or the derivation of $(5.93)$ in \cite{C09}, by the first variational formula we can derive

\[ \nu(\ubar)\leq 1 + \intubar \lvert \chihat(\ubar^{\prime})\rvert_{\gslash}\hsp \nu(\ubar^{\prime})\dubarprime.    \]By Gr\"onwall's inequality, this implies \[  \lvert \nu(\ubar)\rvert \lesssim 1 \hspace{3mm} \text{and}\hspace{3mm}\lvert \nu(\ubar)-1\rvert \leq \frac{\al\hsp O}{\modu^2}\leq \frac{1}{a}.  \]The desired estimate then follows.
\end{proof}The two Propositions above also imply control on the area of the spheres $S_{u,\ubar}$:

\begin{proposition}\label{34}
Under the assumptions of Theorem \ref{mainone} and the bootstrap assumptions \eqref{boundsbootstrap}, there holds
\[\sup_{\ubar}\lvert \text{Area}(S_{u,\ubar})-\text{Area}(S_{u,0})\rvert \leq \frac{O^{\f12}}{\al}\modu^2.\]
\end{proposition}
\begin{proof}
The result follows from the fact that \[ \text{Area}(S_{u,\ubar})=\int_{S_{u,\ubar}}\sqrt{\det\gslash}\hsp\text{d}\theta^1 \text{d}\theta^2,\]together with the bounds on $\det\gslash$ obtained in Proposition \ref{propdetgslash}.
\end{proof}
Regarding the Vlasov matter, we finish this section by showing the bounds that our initial data assumptions in Theorem \ref{mainone} impose on the support of $f$.

\subsection{Decay for the momentum vector components and the support of $f$} \label{subsectiondecaymomentum}
In this section, we discuss the behaviour of the various energy-momentum tensor components for the Vlasov matter (in the zeroth order). To this end, we establish a control on the size of the support of $f$. 

\vspace{3mm}
\noindent Throughout this section,  $\mathfrak{\gamma}$ will denote a null geodesic emanating from $\begin{Bmatrix} u= u_{\infty}\end{Bmatrix}$, meaning $u(\gamma(0)) =u_{\infty}$. Assume that $(\gamma(0), \hsp \dot{\gamma}(0)) \in \hsp \text{supp}(f).$ The tangent vector to $\gamma$ at time $s$ is then written 

\begin{equation}
    \dot{\gamma}(s) = \hsp p^A(s) \hsp e_A + p^3(s)\hsp e_3 +p^4(s) \hsp e_4.
\end{equation}By assumption, we note that $(\gamma(0),\dot{\gamma}(0) )\in \hsp \text{supp}(f)$ implies that there exist constants $C_{p^A}, \hsp C_{p^3}, \hsp C_{p^4}>0$, independent of $u_{\infty}$ (here A=1,\hsp 2), such that $\text{max}\begin{Bmatrix}C_{p^A},\hsp  C_{p^3},\hsp C_{p^4}\end{Bmatrix} \ll a$ and 

\be \label{initial}   0<  p^3(0) \leq C_{p^3}, \hsp 0 \leq \lvert u_{\infty} \rvert^{2} \hsp p^4(0) \leq C_{p^4} \hsp  \hsp p^3(0),     \ee \be \label{tprove} \lvert u_{\infty}\rvert^{2} \lvert p^A(0)\rvert \leq C_{p^A} \hsp p^3(0) .   \ee
The main result of this section is a propagation-of-decay statement for the momentum components. This is the content of the following Proposition.  
\begin{proposition}
\label{momentum}
Along such a geodesic $\gamma$, given the initial conditions \eqref{initial}-\eqref{tprove}, the following must hold:

\[ \frac{1}{2}\hsp p^3(0) \leq \hsp p^3(s) \leq \frac{3}{2} p^3(0), \hspace{3mm} 0\leq  \lvert u(s) \rvert^{2} \hsp p^4(s) \leq \frac{3}{2} \hsp C_{p^4} \hsp p^3(0),   \]\[ \lvert u(s)\rvert^{2} \hsp \lvert p^A(s) \rvert \leq \frac{3}{2} \hsp C_{p^A} \hsp p^3(0),  \]for $A=1, \hsp 2$ and for all $s$ such that $\gamma(s)$ lies in the slab of existence $\mathcal{D}$.
\end{proposition}

\begin{proof}
The evolution of the various momentum components along the geodesic is given by the \textit{geodesic equations}:

\[ \dot{p}^{\mu}(s) + \Gamma^{\mu}_{\alpha \beta}(s) \hsp p^{\alpha}(s)\hsp p^{\beta}(s) =0.    \]In the frame constructed, they read as follows:

\begin{equation}
\begin{split}
\dot{p}^4(s) = -\frac{1}{4}\tr\chibar \hsp \gslash_{AB} \hsp p^A(s) \hsp p^B(s) - \frac{1}{2}\hsp \chibarhat_{AB} \hsp p^A(s) \hsp p^B(s) -2 \hsp \etabar_A \hsp p^A(s) p^4(s),
\end{split}\label{pfourd}
\end{equation}
\begin{equation}
\begin{split}\label{pthreed}
    \dot{p}^3(s) = -\frac{1}{4}\tr\chi \hsp \gslash_{AB} \hsp p^A(s) \hsp p^B(s) - \frac{1}{2}\hsp \chihat_{AB} \hsp p^A(s) \hsp p^B(s) -(\eta_A - \etabar_A) \hsp p^A(s) p^3(s)-\omegabar \hsp p^3(s)p^3(s),
\end{split}
\end{equation}

\vspace{1mm}

\be \label{pAd} \begin{split} \dot{p}^A(s) =& -\slashed{\Gamma}^A_{BC} p^B \hsp p^C - \tr\chibar \hsp p^A \hsp p^3 - \tr\chi \hsp p^A \hsp p^4 -2 \hsp \chibarhat^A_B \hsp p^B \hsp p^3 - 2 \hsp \chihat^A_B \hsp p^B \hsp p^4 +(\slashed{\nabla}_{e_B}b)^A \hsp p^B \hsp p^4 \\&- b^C \hsp \hsp \slashed{\Gamma}^A_{BC} \hsp p^B \hsp p^4 -2 \hsp p^3 \hsp p^4 \hsp(\eta_A + \etabar_A).
\end{split} \ee At this point, we note the control on the various Ricci coefficients and the metric coefficient $\gamma$ stemming from the hierarchy:
 \be  \begin{split} \lvert \tr\chibar \rvert \lesssim &\frac{1}{\lvert u \rvert}, \hspace{3mm} \lvert \chibarhat \rvert \lesssim \frac{\al}{\lvert u\rvert^2} , \hspace{3mm} \lvert \etabar \rvert \lesssim \frac{\al}{\lvert u \rvert^2}, \hspace{3mm}   \lvert \omegabar \rvert \lesssim \frac{a}{\lvert u \rvert^3},      \\  &\lvert \tr\chi \rvert \lesssim \frac{1}{\lvert u \rvert}, \hspace{3mm} \lvert \chihat \rvert \lesssim \frac{\al}{\lvert u\rvert}, \hspace{3mm}  \lvert \eta \rvert \lesssim \frac{\al}{\lvert u \rvert^2},\\& \inftySu{\Gammaslash} \lesssim \frac{\al}{\lvert u \rvert^2}, \hsp \inftySu{b} \lesssim \frac{1}{\al}.  \end{split}  \label{bounds}   \ee
We use the following bootstrap assumptions:
\be \label{zeroorderbootstrap1}  \frac{1}{4}\hsp p^3(0) \leq \hsp p^3(s) \leq 2 p^3(0), \hspace{3mm} 0\leq  \lvert u(s) \rvert^{2} \hsp p^4(s) \leq 2 \hsp C_{p^4} \hsp p^3(0),   \ee \be \lvert u(s)\rvert^{2} \hsp \lvert p^A(s) \rvert \leq 2 \hsp C_{p^A} \hsp p^3(0), \label{zeroorderbootstrap2} \ee for $A=1, \hsp 2$ and for all $s$ such that $\gamma(s)$ lies in the slab of existence $\mathcal{D}$.

\vspace{3mm}

\noindent 
We begin with $p^3$. Given the control above, we arrive at

\be \lvert \dot{p}^3(s) \rvert \leq \frac{4\hsp C_{p^4} \left(p^3(0)\right)^2}{\lvert u\rvert^3 }\ee Since the geodesic is incoming, one could use the restricted function $\restri{u}{\gamma}$ as a different parameter for the geodesic. The relation between $u$ and $s$ is given by \be \label{du}  \frac{\text{d}u}{\text{d}s}= p^3(s)   \ee This implies that \be \frac{\text{d}p^3}{\text{d}u} =  \hsp \frac{1}{p^3} \hsp \dot{p}^3.\ee We thus arrive at 
\begin{equation}
   \big\lvert  \frac{\text{d}p^3 }{\text{d}u} \big\rvert=  \big \lvert \frac{\dot{p}^3}{p^3} \big \rvert \leq \frac{4\hsp C_{p^4} \left(p^3(0)\right)^2}{\frac{1}{4} \hsp p^3(0) \hsp \lvert u\rvert^3 } \leq \frac{16 \hsp C_{p^4} \hsp p^3(0)}{\lvert u \rvert^3}.\label{313}
\end{equation}Integrating along the $e_3$ direction, we thus obtain

\be \lvert p^3(u)-p^3(u_{\infty}) \rvert \leq \frac{16 \hsp C_{p^4} \hsp p^3(u_{\infty})}{\lvert u \rvert^2}, \ee whence we arrive at 

\begin{equation}
\left( 1- \frac{16 \hsp C_{p^4}}{\lvert u \rvert^2} \right) \hsp p^3(u_{\infty})\leq p^3(u) \hsp \leq \left( 1+ \frac{16 \hsp C_{p^4}}{\lvert u \rvert^2} \right) \hsp p^3(u_{\infty}),
\end{equation}which implies the desired control on $p^3$. For $p^4,$ have 

\begin{equation}
\begin{split}
\frac{d (\lvert u \rvert^2 \hsp p^4)}{du} &= \lvert u \rvert^2 \hsp \frac{\dot{p}^4}{p^3} -2 \lvert u \rvert \hsp p^4 = \lvert u \rvert^2  \hsp \left( -(\tr\chibar+ \frac{2}{\lvert u \rvert})\hsp p^4 - \frac{1}{2}\hsp \frac{\chibarhat_{AB}\hsp p^A \hsp p^B}{p^3}- \hsp \frac{2 \hsp \etabar_A \hsp p^A \hsp p^4}{p^3}- \hsp \omegabar \hsp  p^3 \right). 
\end{split}
\end{equation}Using the bound $\big \lvert \tr\chibar + \frac{2}{\lvert u \rvert}\big\rvert \leq \frac{O}{\lvert u \rvert^2}$ , where $O$ is the bootstrap constant for the total Ricci norm (TBD), we have the following:

\begin{equation}
\big \lvert \left(\tr\chibar + \frac{2}{\lvert u \rvert}\right) \hsp p^4 \big\rvert \leq \frac{2 \hsp C_{p^4} \hsp p^3(u_{\infty})}{\lvert u\rvert^4},
\end{equation}

\begin{equation}
\big\lvert \hsp \frac{\chibarhat_{AB}\hsp p^A \hsp p^B}{p^3} \big\rvert \leq \frac{\inftySu{\chibarhat} \hsp \slashed{g}_{AB} \hsp p^A \hsp p^B}{p^3}\leq \frac{2 \hsp \al \hsp C_{p^4} \hsp p^3(u_{\infty})}{\lvert u \rvert^4},
\end{equation}

\be 
\big \lvert \frac{\etabar_A \hsp p^A \hsp p^4}{ p^3} \big\rvert \leq \frac{\inftySu{\etabar}\hsp \sqrt{\gslash_{AB} \hsp p^A \hsp p^B} \hsp p^4}{p^3} \leq \frac{4 \al \hsp C_{p^4}\hsp \sqrt{3\hsp C_{p^4}} \hsp p^3(u_{\infty})}{\lvert u \rvert^5},\ee

\be \big \lvert \omegabar \hsp p^3 \big\rvert \leq 2p^3(u_{\infty})\inftySu{\omegabar}\leq 2p^3(u_{\infty}) \frac{a \hsp O}{\modu^3}. \ee

Putting everything together, we arrive at
\begin{align}
\frac{d|u|^{2}p^{4}}{du}\lesssim&  \frac{2 \hsp C_{p^4} \hsp p^3(u_{\infty})}{\lvert u\rvert^2}+ \frac{2 \hsp \al \hsp C_{p^4} \hsp p^3(u_{\infty})}{\lvert u \rvert^2}+ \frac{4 \al \hsp C_{p^4}\hsp \sqrt{3\hsp C_{p^4}} \hsp p^3(u_{\infty})}{\lvert u \rvert^3}+\frac{8 \hsp C_{p^4} \hsp p^3(u_{\infty})}{\lvert u \rvert^3} \nonumber \\ 
+&\frac{17 \hsp C_{p^4} p^3(u_{\infty}) O}{\lvert u \rvert^3}
\end{align}
and therefore 
\begin{eqnarray}
 |u|^{2}p^{4}(u)\lesssim |u_{\infty}|^{2}p^{4}(u_{\infty})+ \frac{\al}{\modu}.   
\end{eqnarray}
Moving on to $p^A$, in a way similar to what we did for $p^4$, we arrive at the following:

\begin{equation}
\begin{split}
\frac{\text{d}\left( \lvert u \rvert^2 p^A \right)}{\text{d}u}&= \hsp \lvert u \rvert^2  \bigg( \frac{-\slashed{\Gamma}^A_{BC} \hsp p^B \hsp p^C}{p^3} - \left( \tr\chibar +\frac{2}{\lvert u \rvert} \right) \hsp p^A - \hsp \frac{\tr\chi \hsp p^A \hsp p^4}{p^3} - \hsp \chibarhat^A_{B}\hsp p^B - \hsp \frac{2 \chihat^A_B \hsp p^B \hsp p^4}{p^3} \\ &+ \frac{(\slashed{\nabla}_{e_B}b)^A \hsp p^B \hsp p^4}{p^3}- \hsp \frac{b^C \hsp \slashed{\Gamma}^A_{BC} \hsp p^B \hsp p^4}{p^3} - \hsp 2 \hsp p^4 \hsp (\eta_A +\etabar_A)\bigg) 
\end{split}
\end{equation}Using the bootstrap assumptions \eqref{zeroorderbootstrap1}-\eqref{zeroorderbootstrap2} along with the bound $\big\lvert \tr\chibar + \frac{2}{\lvert u \rvert}\big\rvert \leq \frac{1}{\lvert u \rvert^2},$ we arrive at the bound

\begin{equation} \label{325}
 \frac{\text{d}\left( \lvert u \rvert^2 \lvert p^A \rvert \right)}{\text{d}u}  \leq \big\lvert \frac{\text{d}\left( \lvert u \rvert^2 p^A \right)}{\text{d}u} \big\rvert \leq \frac{20 \hsp C_{p^A}\hsp p^3(0)}{\lvert u \rvert^2}.
\end{equation}Integrating \eqref{325}, we arrive at

\begin{equation}
|u|^{2}|p^{A}(u)|\leq |u_{\infty}|^{2}|p^{A}(\infty)|+\frac{20 \hsp C_{p^A}\hsp p^3(0)}{\lvert u \rvert}.
\end{equation}	This concludes the bootstrap argument.
\end{proof}
\noindent We are now in a position to prove Lemma \ref{spreadlemma}.

\vspace{3mm}

\begin{proof}
Consider, as above, an incoming geodesic $(\theta^A(s), u(s), \ubar(s), p^A(s), p^3(s)).$ We want to calculate $\frac{d\ubar}{\text{d}u}$. We have \[ \frac{\text{d}\ubar(s)}{\text{d}s} = \frac{p^4(s)}{\Omega^2}.     \]Therefore, essentially,

\[  \frac{d\ubar}{\text{d}u} = \frac{\frac{\text{d}\ubar(s)}{\text{d}s}}{\frac{\text{d}u}{\text{d}s}}= \frac{p^4(s(u))}{\Omega^2 p^3}.  \]Given the bound \[ p^4(u)\leq \frac{C_{p^4}}{\lvert u\rvert^2}\]and the almost-constant bounds on $\Omega^2, p^3$, we have, for this given geodesic,

\begin{align}
\ubar(-\frac{a}{4}) = \ubar(u_{\infty}) +\int_{u_{\infty}}^{-\frac{a}{4}} \frac{p^4}{\Omega^2 p^3} \text{d}u^{
\prime}
 \lesssim \ubar(u_{\infty})+ \frac{C}{\modu},
\end{align}where $C$ is a constant which is very small compared to $a$ for $a$ large enough (since it only relies on $C_{p^4}, C_{p^3})$. Since $\frac{C}{\modu}$ is negligible compared to $1- \ubar(u_{\infty})$,  it follows that $0\leq \tilde{\ubar}_0 \leq 1$
. Finally, let us note that since $p^3>0$ always, we have $\frac{\text{d}\ubar}{\text{d}u}$ is always positive, whence the incoming geodesic always tilts towards $\ubar=1$, but never reaches it for our given set of data.\end{proof}


\subsection{Transport inequalities}
\begin{proposition}
Under the assumptions of Theorem \ref{mainone} and the bootstrap assumptions \eqref{boundsbootstrap}, there holds

\be \twoSu{\phi} \lesssim \lVert \phi \rVert_{L^2(S_{u,\ubar^{\prime}})}+ \int_{\ubar^{\prime}}^{\ubar} \lVert \snabla_4 \phi \rVert_{L^2(S_{u,\ubar^{\prime\prime}})}\text{d}\ubar^{\prime\prime}, \ee

\be \lVert{\phi}\rVert_{L^4(S_{u,\ubar})} \lesssim \lVert \phi \rVert_{L^4(S_{u,\ubar^{\prime}})}+ \int_{\ubar^{\prime}}^{\ubar} \lVert \snabla_4 \phi \rVert_{L^4(S_{u,\ubar^{\prime\prime}})}\text{d}\ubar^{\prime\prime},\ee for an $S_{u,\ubar}-$tangent tensor $\phi$ of arbitrary rank. 
\end{proposition}
\begin{proof}
For a scalar $f$, we use the $1^{\text{st}}$ variational formula to get

\[ \frac{\text{d}}{\text{d}\ubar}\int_{S_{u,\ubar}}f \hsp \text{d}\mu_{S_{u,\ubar}} =  \int_{S_{u,\ubar}} \bigg(\frac{\text{d}f}{\text{d}\ubar}+ \Omega^2 \tr\chi \hsp f\bigg) \hsp  \text{d}\mu_{S_{u,\ubar}} = \int_{S_{u,\ubar}}\Omega^2(e_4(f)+\tr\chi \hsp f) \text{d}\mu_{S_{u,\ubar}}.\]
\noindent Taking $f=\lvert \phi\rvert_{\gslash}^2$, applying the Cauchy-Schwartz inequality on the sphere and utilising the bounds on $\Omega$ and $\tr\chi$, we have
\[ 2\twoSu{\phi} \hsp \frac{\text{d}}{\text{d}\ubar}\twoSu{\phi} \lesssim \twoSu{\phi}\twoSu{\snabla_4 \phi}+\frac{O}{\modu}\twoSu{\phi}^2, \]which implies 
\[\frac{\text{d}}{\text{d}\ubar}\twoSu{\phi} \lesssim \twoSu{\snabla_4 \phi}+\frac{O}{\modu}\twoSu{\phi}.\] An application of Gr\"onwall's inequality now yields the desired result. Taking $f=\lvert \phi\rvert_{\gslash}^4$ yields the corresponding result for the $L^4-$norm.\end{proof}

\begin{proposition}
Under the assumptions of Theorem \ref{mainone} and the bootstrap assumptions \eqref{boundsbootstrap}, there holds

\be \twoSu{\phi} \lesssim \lVert \phi \rVert_{L^2(S_{u^{\prime},\ubar})}+ \int_{u^{\prime}}^{u} \lVert \snabla_3 \phi \rVert_{L^2(S_{u^{\prime\prime}\ubar)}}\text{d}u^{\prime\prime}, \ee

\be \fourSu{\phi} \lesssim \lVert \phi \rVert_{L^4(S_{u^{\prime},\ubar})}+ \int_{u^{\prime}}^{u} \lVert \snabla_3 \phi \rVert_{L^4(S_{u^{\prime\prime}\ubar)}}\text{d}u^{\prime\prime}, \ee for an $S_{u,\ubar}-$tangent tensor $\phi$ of arbitrary rank.
\end{proposition}
\begin{proof}
We again take the $1^{\text{st}}$ variational formula in the $u-$direction to obtain

\be \underline{L}\int_{S_{u,\ubar}} \hsp f \hsp \text{d}\mu_{S_{u,\ubar}} = \int_{S_{u,\ubar}} \big(\frac{\text{d}f}{\text{d}u} + \tr\chibar \hsp f\big)   \hsp \text{d}\mu_{S_{u,\ubar}} = \int_{S_{u,\ubar}} (e_3(f)+\tr\chibar \hsp f) \hsp \text{d}\mu_{S_{u,\ubar}}. \ee Working as before and using the fact that $\tr\chibar<0$, we obtain 

\[ \twoSu{\phi}\frac{\text{d}}{\text{d}u }\twoSu{\phi} \lesssim \twoSu{\phi}\twoSu{\snabla_3\phi}.    \]This implies \[\frac{\text{d}}{\text{d}u }\twoSu{\phi} \lesssim \twoSu{\snabla_3\phi},\]which in turn implies the result. 
\end{proof}For the purposes of our analysis, it will be necessary to obtain and exploit certain finer bounds, along the $e_3$-direction, than those given by the above Proposition. These bounds will come from incorporating weights into the norms, weights which themselves depend on the coefficients in front of the linear term with a $\tr\chibar$-factor. The essential observation is that, under the bootstrap assumptions \eqref{boundsbootstrap}, $\tr\chibar$ may 
 be viewed as $-\frac{2}{\modu}$, so that $\tr\chibar + \frac{2}{\modu}$ is, roughly speaking, $O(\modu^{-2})$ and hence, itself, integrable along the $u-$direction. 

 \begin{proposition}\label{evolemma}(Evolution Lemma) We continue to work under the assumptions of Theorem \ref{mainone} and the bootstrap assumptions \eqref{boundsbootstrap}. Assume that $\phi$ and $F$ are $S_{u,\ubar}-$tangent tensor fields of rank $k$ satisfying the following transport equation

 \[ \snabla_3 \phi_{A_1\dots A_k}+ \lambda_0 \tr\chibar \phi_{A_1\dots A_k}=F_{A_1\dots A_k}.    \]Let $p \in \{2,4\}$. Denoting $\lambda_1:= 2(\lambda_0-\frac{1}{p})$, the following inequality holds true:

 \[ \lvert u \rvert^{\lambda_1}\lVert \phi \rVert_{L^p(S_{u,\ubar})} \lesssim \lvert u_{\infty}\rvert^{\lambda_1}\lVert \phi \rVert_{L^p(S_{u_{\infty},\ubar})} +\intu \lvert u^{\prime}\rvert^{\lambda_1}\lVert F \rVert_{L^p(S_{u^{\prime},\ubar})}\duprime .   \]

 \end{proposition}
\begin{proof}
This is the same as in \cite{AnThesis}, the only exception being the different bootstrap assumption on $\tr\chibar+\frac{2}{\modu}$. 
\end{proof}
In addition, as we work in a low regularity regime, we also have to establish so-called \textit{codimension-1 trace inequalities}, in order to control the $L^{4}(S_{u,\ubar})-$norm of a tensorfield in terms of $L^{2}(H/\Hbar)-$norms. 
\begin{proposition}[Codimension-1 trace inequality along incoming null hypersurfaces]
\label{L4rough}
Let the assumptions of Theorem~\ref{mainone} hold, and suppose the bootstrap assumptions~\eqref{boundsbootstrap} are in force. Let $\varphi$ be a scalar (or tensorial) function defined on the domain of dependence associated with the double null foliation, and suppose that $\varphi$ is sufficiently regular so that the quantities below are well-defined. Then the following codimension-one trace inequality holds:
\begin{align}
\label{eq:L4trace}
\|\varphi\|_{L^4(S_{u,\ubar})}
&\lesssim \|\varphi\|_{L^4(S_{u_\infty,\ubar})}
+ \|\snabla_3 \varphi\|^{1/2}_{L^2(\Hbar_{\ubar}^{(u_{\infty},u)})}
\left( \left\| |u'|^{-1} \varphi \right\|^{1/2}_{L^2(\Hbar_{\ubar}^{(u_{\infty},u)})}
+ \|\slashed{\nabla} \varphi\|^{1/2}_{L^2(\Hbar_{\ubar}^{(u_{\infty},u)})} \right).
\end{align}
\end{proposition}
\begin{proof}
We assume the setting of Theorem~\ref{mainone}, including the bootstrap assumptions \eqref{boundsbootstrap}, and work in the double null foliation framework. Let $\varphi$ be a scalar (or more generally tensorial) field defined on the spacetime domain under consideration. All norms and volume forms are defined with respect to the induced geometric structures on the spheres $S_{u,\ubar}$ and the incoming null hypersurface $\Hbar_{\ubar}$.

\vspace{1ex}

\noindent We begin by applying the geometric isoperimetric inequality on the sphere $S_{u,\ubar}$:
\begin{equation} \label{eq:isoperimetric}
\int_{S_{u,\ubar}} |\varphi|^6 \, d\mu_{\slashed{g}}
\lesssim
\left( \int_{S_{u,\ubar}} |\varphi|^4 \, d\mu_{\slashed{g}} \right)
\left( \int_{S_{u,\ubar}} \left( |u|^{-2} |\varphi|^2 + |\slashed{\nabla} \varphi|^2 \right) \, d\mu_{\slashed{g}} \right),
\end{equation}
where $d\mu_{\slashed{g}}$ denotes the volume form induced on $S_{u,\ubar}$ by the metric $\slashed{g}$. Integrating \eqref{eq:isoperimetric} in $u' \in [u_\infty, u]$ along the incoming null hypersurface $\Hbar_{\ubar}$, we obtain:
\begin{equation} \label{eq:isop-integrated}
\int_{u_\infty}^{u} \int_{S_{u',\ubar}} |\varphi|^6 \, d\mu_{\slashed{g}} \, du'
\lesssim
\sup_{u' \in [u_\infty, u]} \left( \int_{S_{u',\ubar}} |\varphi|^4 \, d\mu_{\slashed{g}} \right)
\int_{u_\infty}^{u} \int_{S_{u',\ubar}} \left( |u'|^{-2} |\varphi|^2 + |\slashed{\nabla} \varphi|^2 \right) \, d\mu_{\slashed{g}} \, du'.
\end{equation}

\noindent Next, we estimate the supremum in the right-hand side of \eqref{eq:isop-integrated} via a transport inequality derived from integrating along the incoming null generator $\snabla_3$ on $\Hbar_{\ubar}$. Using the fundamental theorem of calculus and the structure of null hypersurfaces, we compute:
\begin{align}
\label{eq:transport-L4}
\|\varphi\|_{L^4(S_{u,\ubar})}^4
&\lesssim
\|\varphi\|^4_{L^4(S_{u_\infty,\ubar})}
+ \int_{u_\infty}^{u} \frac{d}{du'} \left( \int_{S_{u',\ubar}} |\varphi|^4 \, d\mu_{\slashed{g}} \right) du' \nonumber \\
&= \|\varphi\|^4_{L^4(S_{u_\infty,\ubar})}
+ \int_{u_\infty}^{u} \int_{S_{u',\ubar}} \left( 4 |\varphi|^2 \langle \varphi, \snabla_3 \varphi \rangle + \tr\chibar |\varphi|^4 \right) \, d\mu_{\slashed{g}} \, du'.
\end{align}

\noindent To control the term involving $\tr\chibar$, we decompose
\[
\tr\chibar = \left( \tr\chibar + \frac{2}{|u|} \right) - \frac{2}{|u|},
\]
and observe that the anomalous term $\tr\chibar + \frac{2}{|u|}$ is integrable in $u'$ along $\Hbar_{\ubar}$ under the bootstrap assumptions. Applying Grönwall's inequality (if needed) and using Cauchy–Schwartz and Hölder, we estimate:
\begin{align}
\int_{u_\infty}^{u} \int_{S_{u',\ubar}} |\varphi|^2 |\snabla_3 \varphi| \, d\mu_{\slashed{g}} \, du'
&\leq \|\varphi\|_{L^6(\Hbar_{\ubar}^{(u_{\infty},u)})}^3 \|\snabla_3 \varphi\|_{L^2(\Hbar_{\ubar}^{(u_{\infty},u)})}, \nonumber
\end{align}
and conclude from \eqref{eq:transport-L4}:
\begin{equation} \label{eq:L4growth}
\|\varphi\|^4_{L^4(S_{u,\ubar})}
\lesssim
\|\varphi\|^4_{L^4(S_{u_\infty,\ubar})}
+ \|\varphi\|^3_{L^6(\Hbar_{\ubar}^{(u_{\infty},u)})} \|\snabla_3 \varphi\|_{L^2(\Hbar_{\ubar}^{(u_{\infty},u)})}.
\end{equation}

\noindent We now substitute the bound \eqref{eq:L4growth} into the right-hand side of \eqref{eq:isop-integrated} to deduce:
\begin{align}
&\int_{u_\infty}^{u} \int_{S_{u',\ubar}} |\varphi|^6 \, d\mu_{\slashed{g}} \, du' \notag\\
\lesssim&  
\left(
\|\varphi\|^4_{L^4(S_{u_\infty,\ubar})}
+ \|\varphi\|^3_{L^6(\Hbar_{\ubar}^{(u_{\infty},u)})} \|\snabla_3 \varphi\|_{L^2(\Hbar_{\ubar}^{(u_{\infty},u)})}
\right)
\int_{\Hbar_{\ubar}} \left( |u'|^{-2} |\varphi|^2 + |\slashed{\nabla} \varphi|^2 \right) \, d\mu_{\slashed{g}} \, du'.
\end{align}
\noindent We now apply Young's inequality to control the mixed term and deduce:
\begin{align}
\int_{\Hbar_{\ubar}} |\varphi|^6 \, d\mu_{\slashed{g}} \, du'
&\lesssim
\|\snabla_3 \varphi\|_{L^2(\Hbar_{\ubar}^{(u_{\infty},u)})}^2
\left(
\int_{\Hbar_{\ubar}} \left( |u|^{-2} |\varphi|^2 + |\slashed{\nabla} \varphi|^2 \right) \, d\mu_{\slashed{g}} \, du'
\right)^2 \nonumber \\
&\quad+
\|\varphi\|^4_{L^4(S_{u_\infty,\ubar})}
\int_{\Hbar_{\ubar}} \left( |u|^{-2} |\varphi|^2 + |\slashed{\nabla} \varphi|^2 \right) \, d\mu_{\slashed{g}} \, du'.
\end{align}
\noindent Substituting this estimate back into \eqref{eq:L4growth} and taking appropriate roots yields the desired inequality:
\[
\|\varphi\|_{L^4(S_{u,\ubar})}
\lesssim
\|\varphi\|_{L^4(S_{u_\infty,\ubar})}
+ \|\snabla_3 \varphi\|^{1/2}_{L^2(\Hbar_{\ubar}^{(u_{\infty},u)})}
\left(
\||u|^{-1} \varphi\|^{1/2}_{L^2(\Hbar_{\ubar}^{(u_{\infty},u)})}
+ \|\slashed{\nabla} \varphi\|^{1/2}_{L^2(\Hbar_{\ubar}^{(u_{\infty},u)})}
\right),
\]
which completes the proof.
\end{proof}

\begin{proposition}
\label{codimension11}
Let the assumptions of Theorem~\ref{mainone} hold, and suppose the bootstrap assumptions~\eqref{boundsbootstrap} are in force. Let $\varphi$ be a scalar (or tensorial) function defined on the domain of dependence associated with the double null foliation, and suppose that $\varphi$ is sufficiently regular so that the quantities below are well-defined. The following scale-invariant codimension 1 trace inequality holds: 
 \begin{eqnarray}
 ||\varphi||_{L^{4}_{sc}(S_{u,\ubar})}\lesssim ||\varphi||_{L^{4}_{sc}(S_{u_{\infty},\ubar})}+||a^{\frac{1}{2}}\nablasl_{3}\varphi||^{\frac{1}{2}}_{L^{2}_{sc}(\Hbar)}\left(a^{-\frac{1}{4}}||\varphi||^{\frac{1}{2}}_{L^{2}_{sc}(\Hbar)}+||\nablasl\varphi||^{\frac{1}{2}}_{L^{2}_{sc}(\Hbar)}\right).      
 \end{eqnarray}
\end{proposition}
\begin{proof}
We begin by recalling the relevant scale-invariant norms and signature conventions. For a tensorial quantity $\varphi$ defined on the 2-sphere $S_{u,\ubar}$ in a double null foliation, let $s_2(\varphi)$ denote the scale weight associated to $\varphi$ under the rescaling symmetry. Then, by definition of the scale-invariant $L^p$ norm, we have:
\begin{equation} \label{eq:sc-Lp-def}
\|\varphi\|_{L^p(S_{u,\ubar})} = \frac{a^{s_2(\varphi)}}{|u|^{2s_2(\varphi) + \frac{p-2}{p}}} \|\varphi\|_{L^p_{sc}(S_{u,\ubar})},
\end{equation}
where $\|\cdot\|_{L^p_{sc}}$ denotes the scale-invariant norm. The scale weights for derivatives obey:
\begin{equation} \label{eq:sc-weights}
s_2(\slashed{\nabla} \varphi) = s_2(\varphi) + \frac{1}{2}, \quad s_3(\snabla_3 \varphi) = s_2(\varphi) + 1.
\end{equation}

\noindent Let us now recall from Proposition~\ref{L4rough} the scale-invariant $L^4$ inequality for $\varphi$ derived previously, restated with the above normalization:
\begin{align}
\|\varphi\|_{L^4(S_{u,\ubar})}
&\lesssim
\|\varphi\|_{L^4(S_{u_\infty,\ubar})}
+
\|\snabla_3 \varphi\|^{1/2}_{L^2(\Hbar_{\ubar}^{(u_{\infty},u)})}
\left( \||u|^{-1} \varphi\|^{1/2}_{L^2(\Hbar_{\ubar}^{(u_{\infty},u)})}
+ \|\slashed{\nabla} \varphi\|^{1/2}_{L^2((\Hbar_{\ubar}^{(u_{\infty},u)})} \right).
\end{align}

\noindent We now express this inequality entirely in terms of the scale-invariant norms. Using \eqref{eq:sc-Lp-def}, we write:
\begin{align}
\frac{a^{s_2}}{|u|^{2s_2 + \frac{1}{2}}} \|\varphi\|_{L^4_{sc}(S_{u,\ubar})}
&\lesssim
\frac{a^{s_2}}{|u_\infty|^{2s_2 + \frac{1}{2}}} \|\varphi\|_{L^4_{sc}(S_{u_\infty,\ubar})}
+ \left( \int_{u_\infty}^{u} \frac{a^{2(s_2+1)}}{|u'|^{4(s_2+1)}} \|\snabla_3 \varphi\|^2_{\mathcal{L}^2_{(sc)}(S_{u^{\prime},\ubar})} \, \duprime \right)^{1/4} \nonumber \\
\quad &\times \left[
\left( \int_{u_\infty}^{u} \frac{a^{2s_2}}{|u'|^{4s_2+2}} \|\varphi\|^2_{\mathcal{L}^2_{(sc)}(S_{u',\ubar})} \, du' \right)^{1/4}
+ \left( \int_{u_\infty}^{u} \frac{a^{2(s_2+\frac{1}{2})}}{|u'|^{4(s_2+\frac{1}{2})}} \|\slashed{\nabla} \varphi\|^2_{\mathcal{L}^2_{(sc)}(S_{u',\ubar})} \, du' \right)^{1/4}
\right].
\end{align}

\noindent Multiplying both sides by $|u|^{2s_2 + \frac{1}{2}} a^{-s_2}$, we obtain:
\begin{align}
\|\varphi\|_{L^4_{sc}(S_{u,\ubar})}
&\lesssim
\frac{|u|^{2s_2 + \frac{1}{2}}}{|u_\infty|^{2s_2 + \frac{1}{2}}} \|\varphi\|_{L^4_{sc}(S_{u_\infty,\ubar})}
+ a^{\frac{1}{4}} \left( \int_{u_\infty}^{u} \frac{a}{|u'|^2} \|\snabla_3 \varphi\|^2_{\mathcal{L}^2_{(sc)}(S_{u',\ubar})} \frac{|u|^{4s_2 + 2}}{|u'|^{4s_2 + 2}} \, du' \right)^{1/4} \nonumber \\
\quad &\times \left[
a^{-\frac{1}{4}} \left( \int_{u_\infty}^{u} \frac{a}{|u'|^2} \|\varphi\|^2_{\mathcal{L}^2_{(sc)}(S_{u',\ubar})} \frac{|u|^{4s_2}}{|u'|^{4s_2}} \, du' \right)^{1/4}
+
\left( \int_{u_\infty}^{u} \frac{a}{|u'|^2} \|\slashed{\nabla} \varphi\|^2_{\mathcal{L}^2_{(sc)}(S_{u',\ubar})} \frac{|u|^{4s_2}}{|u'|^{4s_2}} \, du' \right)^{1/4}
\right].
\end{align}

\noindent Now, observe that for all $u' \in [u_\infty, u]$, we have the inequality
\[
\frac{|u|}{|u'|} \leq 1,
\]
which implies
\[
\frac{|u|^{4s_2 + 2}}{|u'|^{4s_2 + 2}} \leq 1, \quad
\frac{|u|^{4s_2}}{|u'|^{4s_2}} \leq 1.
\]
Thus, all the multiplicative factors involving $\frac{|u|}{|u'|}$ are bounded above by $1$, and we conclude:
\begin{eqnarray}
 ||\varphi||_{L^{4}_{sc}(S_{u,\ubar})}\lesssim  ||\varphi||_{L^{4}_{sc}(S_{u_{\infty},\ubar})}+a^{\frac{1}{4}}\left(\int_{u_{\infty}}^{u}\frac{a}{|u^{\prime}|^{2}}||\nablasl_{3}\varphi||^{2}_{L^{2}_{sc}(S_{u^{\prime},\ubar}}\frac{|u|^{4s_{2}+2}}{|u^{\prime}|^{4s_{2}+2}}\text{d}u^{\prime}\right)^{\frac{1}{4}}\\\nonumber 
 \left[a^{-\frac{1}{4}}\left(\int_{u_{\infty}}^{u}\frac{a}{|u^{\prime}|^{2}}||\varphi||^{2}_{L^{2}_{sc}(S_{u^{\prime},\ubar})}\frac{|u|^{4s_{2}}}{|u^{\prime}|^{4s_{2}}}\text{d}u^{\prime} \right)^{\frac{1}{4}}+\left(\int_{u_{\infty}}^{u}\frac{a}{|u^{\prime}|^{2}} ||\nablasl\varphi||^{2}_{L^{2}_{sc}(S_{u^{\prime},\ubar})}\frac{|u|^{4s_{2}}}{|u^{\prime}|^{4s_{2}}}\text{d}u^{\prime} \right)^{\frac{1}{4}}\right].
\end{eqnarray}
as claimed.
\end{proof}

\begin{proposition}\label{codimensiononetraceineq}
Let the assumptions of Theorem~\ref{mainone} hold, and suppose the bootstrap assumptions~\eqref{boundsbootstrap} are in force. Let $\varphi$ be a scalar (or tensorial) function defined on the domain of dependence associated with the double null foliation, and suppose that $\varphi$ is sufficiently regular so that the quantities below are well-defined. The following codimension 1 trace inequality holds 
\begin{eqnarray}
  ||\varphi||_{L^{4}(S_{u,\ubar})}\lesssim ||\varphi||_{L^{4}(S_{u,0})}+||\snabla_{4}\varphi||^{\frac{1}{2}}_{L^{2}(H)}\left(|u|^{-\frac{1}{2}}||\varphi||^{\frac{1}{2}}_{L^{2}(H)}+||\nablasl\varphi||^{\frac{1}{2}}_{L^{2}(H)}\right)  
\end{eqnarray}
\end{proposition}
\begin{proof}
 Exactly similar only we integrate in the $e_{4}$ direction. 
\end{proof}
The following scale-invariant version holds 
\begin{proposition}
\label{codimension12}
Let the assumptions of Theorem~\ref{mainone} hold, and suppose the bootstrap assumptions~\eqref{boundsbootstrap} are in force. Let $\varphi$ be a scalar (or tensorial) function defined on the domain of dependence associated with the double null foliation, and suppose that $\varphi$ is sufficiently regular so that the quantities below are well-defined. The following scale-invariant trace inequality holds 
 \begin{eqnarray}
  ||\varphi||_{L^{4}_{sc}(S_{u,\ubar})}\lesssim ||\varphi||_{L^{4}_{sc}(S_{u,0})}+||\snabla_{4}\varphi||^{\frac{1}{2}}_{L^{2}_{sc}(H)}\left(|||\varphi||^{\frac{1}{2}}_{L^{2}_{sc}(H)}+||a^{\frac{1}{2}}\nablasl\varphi||^{\frac{1}{2}}_{L^{2}_{sc}(H)}\right) 
\end{eqnarray}
\end{proposition}

\subsection{Sobolev embedding}
\noindent With the derived estimates for the metric $\gamma$, we can obtain a bound on the isoperimetric constant for a topological $2-$sphere $S$: 
\begin{eqnarray}
I(S):=\sup_{U\subset S,~\partial U\in C^{1}}\frac{\min\left\{Area(U),Area(U^{c})\right\}}{[Perimeter(\partial U)]^{2}}.
\end{eqnarray}
The following proposition yields an upper bound for $I(S)$.\\
\begin{proposition}\textit{Under the assumption on the initial data and the bootstrap assumption (2.10), the isoperimetric constant obeys the following bound 
\begin{eqnarray} \label{37}
I(S_{u,\underline{u}})\leq \frac{1}{\pi}
\end{eqnarray}
for $u\in [u_{\infty},-\frac{a}{4}]$ and $\underline{u}\in [0,1]$.
}\end{proposition}
\begin{proof}  Fix a $u$. For $U_{\ubar}$ a subset of $S_{u,\ubar}$, denote by $U_0 \subset S_{u,0}$ the backward image of $U_{\ubar}$ under the diffeomorphism generated by the equivariant vector field $L.$ Using Propositions \ref{31} and \ref{34} and their proof, we can obtain the estimates

\[      \frac{\text{Perimeter}(\partial U_{\ubar}) }{\text{Perimeter}(\partial U_0)}\geq \sqrt{\inf_{S_{u,0}}\lambda(\ubar) }  \]and

\[ \frac{\text{Area}(U_{\ubar})}{\text{Area}(U_0)} \leq \sup_{S_{u,0}}\frac{\det(\gamma_{\ubar})}{\det(\gamma_0)}, \hspace{2mm}\frac{\text{Area}(U_{\ubar}^c)}{\text{Area}(U_0^c)} \leq \sup_{S_{u,0}}\frac{\det(\gamma_{\ubar})}{\det(\gamma_0)}. \]The conclusion then follows from the fact that $I(S_{u,0}) = \frac{1}{2\pi}$ and the bounds from Propositions \ref{31}, \ref{34}, and \ref{propdetgslash}.       \end{proof}
Throughout this work, we will be using an $L^2-L^{\infty}$ Sobolev estimate. To obtain it, utilizing the basic estimates above, we may proceed to write down the following gauge-invariant Sobolev inequalities for the topological $2-$ sphere $S$.\\
\begin{proposition}
Let $(S,\gamma)$ be a compact Riemannian $2$--manifold with isoperimetric constant $I(S)$ and area $Area(S)$. For any tensor field $\mathcal{G}\in \Gamma(\,^{N}\!\otimes T^{*}S)$ and any $p\in (2,\infty)$, one has
\begin{equation}\label{eq:Lp}
Area(S)^{-\frac{1}{p}}\|\mathcal{G}\|_{L^{p}(S)}
\;\leq\; C_{p}\,\big(\max\{1,I(S)\}\big)^{\tfrac12}
\left(\|\snabla \mathcal{G}\|_{L^{2}(S)} + Area(S)^{-\tfrac12}\|\mathcal{G}\|_{L^{2}(S)}\right),
\end{equation}
where $C_{p}$ depends only on $p$.
\end{proposition}

\begin{proposition} \textit{Let $(S,\gamma)$ be a Riemannian $2-$manifold with the isoperimetric constant $I(S)$. Then the following Sobolev embedding holds for any $\mathcal{G}\in \Gamma(^{N}\otimes T^{*}S)$}
\begin{eqnarray}
\label{eq:Linfty}
||\mathcal{G}||_{L^{\infty}(S)}\leq C_{p}\left(\max(1,I(S))\right)^{\frac{1}{2}}[Area(S)]^{\frac{1}{2}-\frac{1}{p}}\left(||\snabla\mathcal{G}||_{L^{p}(S)}+Area(S)^{-\frac{1}{2}}||\mathcal{G}||_{L^{p}(S)}\right)
\end{eqnarray}
for any $p\in (2,\infty)$.\end{proposition}
\begin{proof} A calculation similar to the previous one and the standard $L^{\infty}-L^{p}$ Sobolev inequality on the Riemannian manifold $(S,\gamma)$ yield the result. \end{proof}

\noindent The two inequalities above, together with Propositions \ref{31}-\ref{34}, allow us to control the $L^{2}$-norm of $\mathcal{G}$ in terms of its $H^{2}$-norm. Following the area estimates, we have $Area(S_{u,\underline{u}})\approx u^{2}$. Therefore, we obtain the following important inequality.\\ 
\begin{proposition}\label{Sobolev}
Let the hypotheses of Theorem \ref{mainone} hold and assume the bootstrap bounds \eqref{bootstrap}. 
Assume moreover that, for each $(u,\ub)$ in the region under consideration, the $2$--surface
$S_{u,\ub}$ is diffeomorphic to $S^2$ and its induced metric $\gslash$ satisfies the uniform comparison from propositions \ref{31},\ref{34}, and \ref{propdetgslash}.
Under these hypotheses, for every tensor field
$\mathcal G\in\Gamma(^{N}\otimes T^*S_{u,\ub})$ the following holds:
\begin{equation}\label{eq:Sobolev-weighted}
\|\mathcal G\|_{L^\infty(S_{u,\ub})}
\lesssim  \sum_{I=0}^2 \big\||u|^{\,I-1}\,\snabla^{\,I}\mathcal G\big\|_{L^2(S_{u,\ub})}.
\end{equation}
Equivalently, in the scale--invariant norms used in the paper,
\begin{equation}\label{eq:Sobolev-scale-inv}
\|\mathcal G\|_{L^\infty_{sc}(S_{u,\ub})}
\lesssim  \sum_{I=0}^2 \big\|(a^{1/2}\snabla)^{I}\mathcal G\big\|_{\mathcal{L}^2_{(sc)}(S_{u,\ub})}.
\end{equation}
\end{proposition}
\begin{proof} Substitute $p=4$ in the previous proposition (\ref{eq:Linfty}) and estimate the right hand side by means of (\ref{eq:Lp}), the estimate on the area given by Proposition \ref{34} and on the isoperimetric constant of $S_{u,\underline{u}}$ given by Proposition \ref{37}. \end{proof}

\subsection{The schematic form of the Equations and Commutation}\label{commute} Let us restate that, in order to obtain the desired Theorem \ref{mainone} it is vital that higher derivatives of both Ricci coefficients as well as curvature components be estimated in $L^2$. In order to deal with the nonlinearities manifesting in the equations, certain error terms are estimated in $L^{\infty}$ on the spheres. These $L^{\infty}$ bounds themselves are obtained not through a direct bootstrap argument, but rather through $L^2$ or $L^4$ estimates for higher derivatives, using Sobolev inequalities. These 
higher order estimates are obtained through commuting the null structure and Bianchi equations respectively with suitable differential operators. Let us mention here the important fact that the Vlasov matter system necessitates that these differential operators are more than one. For example, in \cite{AnAth} or the more recent \cite{A-M-Y}, only the angular derivative $\snabla$ was required. However, in our case, control of the Vlasov matter necessitates control on $Vf$ and its derivatives $VVf$ and so on, where the $V'$s are appropriately chosen vector fields (chosen so as to remain essentially constant under the evolution of the geodesic flow). The very definition of these vector fields implies that estimates on $VVf, VVVf$  necessitate control on $\De(\psi,\Psi)$ and $\De^2(\psi,\Psi)$, where $\De$ now denotes not only an angular derivative, but rather any of a triple of differential operators in the set \[S_\mathcal{D}:= \{ \modu \snabla_3, \snabla_4,\al\snabla\}.\] Notice, of course, that each of the operators $\snabla_3,\snabla_4, \snabla$ could be separately multiplied by a different function, giving rise to another set $S^{\prime}_{\mathcal{D}}$. However, the set $S_{\mathcal{D}}$ is so defined such that an arbitrary $\mathcal{D}\in S_{\mathcal{D}}$ should preserve the rate of decay of any tensorfield that it is applied to. This is summarized in the informal:  

\begin{center}
\textit{Commutation Principle: Applying an arbitrary operator $\mathcal{D}$ to any of the $\Gamma,\Psi,\mathcal{T}$ should not alter their respective rate of decay.}
\end{center}See also \cite{Dafermosscattering} and \cite{taylor}. For an $S_{u,\ubar}$-tensor field $\varphi$, let $\mathcal{D}^k\varphi$ denote any fixed $k-$tuple $\mathcal{D}_k\De_{k-1}\dots\De_1$ applied to $\varphi$, where each $\De_i\in S_{\mathcal{D}}$. The following commutation lemmata will be used to derive expressions for the commuted structure/Bianchi equations in the schematic notation. 

\begin{lemma}
If $\xi$ is a $(0,k)$ $S_{u,\ubar}-$tensor, then

\begin{gather*}
(\snabla_3 \xi)_{A_1\dots A_k} = e_3(\xi_{A_1\dots A_k})- \sum_{i=1}^k\big({\chibar_{A_i}}^B -\snabla_A b^B +b^C \Gammaslash^B_{AC}\big)\xi_{A_1 \dots A_{i-1}B A_{i+1}\dots A_k}, \\ (\snabla_4 \xi)_{A_1\dots A_k}= e_4(\xi_{A_1\dots A_k})-\sum_{i=1}^k {\chi_{A_i}}^B \hsp \xi_{A_1 \dots A_{i-1}B A_{i+1}\dots A_k}, \\ (\snabla_B \xi)_{A_1\dots A_k} = e_B(\xi_{A_1\dots A_k})-\sum_{i=1}^k \Gammaslash^C_{BA_i}\hsp \xi_{A_1 \dots A_{i-1}B A_{i+1}\dots A_k}.
\end{gather*}

\end{lemma}
\begin{proof}
This follows from the definition of how a covariant derivative acts on a tensor field, after taking into account the table \eqref{tableone}-\eqref{tablefinal} of connection coefficients.
\end{proof}
The above lemma leads us to commutation formulae for the projected covariant derivatives $\snabla_4,\snabla_3$ and $\snabla$. The proof can be found, for example, in \cite{taylor}, \cite{ChrKl}, \cite{Dafermosscattering}, among others.
\begin{lemma}\label{preliminarycommutationlemma} If $\xi$ is a $(0,k)$ $S_{u,\ubar}-$tensor, then

\begin{align}
[\snabla_4,\snabla_B]\xi_{A_1\dots A_k} =& -{\chi_B}^C \snabla_C \xi_{A_1\dots A_k} \notag \\ +&\sum_{i=1}^k \big(\chi_{A_iB}\etabar^C-\etabar_{A_i}{\chi_B}^C +\Hodge{\beta}_B{\slashed{\epsilon}_{A_i}}^C +\f12 {\gslash_B}^C\Te_{4A_i}-\f12 \gslash_{A_iB}{\Te_4}^C\big)\xi_{A_1\dots A_{i-1}CA_{i+1}\dots A_k}, \label{4Bcommutationf}
\end{align}
\begin{align}
[\snabla_3,\snabla_B]\xi_{A_1\dots A_k} =&(\eta_B+\etabar_B)\snabla_3 \xi_{A_1\dots A_k}-{\chibar_B}^C\snabla_C \xi_{A_1\dots A_k} \notag \\ +& \sum_{i=1}^k\big(\chibar_{A_iB}\eta^C-\eta_{A_i}{\chibar_B}^C -\Hodge{\betabar}_B{\slashed{\epsilon}_{A_i}}^C +\frac{1}{2}{\gslash_B}^C \Te_{3A_i}-\frac{1}{2}\gslash_{A_iB}{\Te_3}^C\big)\xi_{A_1\dots A_{i-1}CA_{i+1}\dots A_k}, \label{3Bcommutationf}
\end{align}
\begin{align}
[\snabla_3,\snabla_4]\xi_{A_1\dots A_k} =& 2\big(\eta^C-\etabar^C\big) \snabla_C \xi_{A_1\dots A_k} -\omegabar \snabla_4\xi_{A_1\dots A_k}\notag \\ +&2 \sum_{i=1}^k \big(\etabar_{A_i}\eta^C - \eta_{A_i}\etabar^C - \sigma {\slashed{\epsilon}_{A_i}}^C\big)\xi_{A_1\dots A_{i-1}CA_{i+1}\dots A_k}. \label{34 commutationf}
\end{align}Finally, for $B,C \in \{1,2\}$, the angular commutation relation 
\be [\snabla_B,\snabla_C]\xi_{A_1\dots A_k} = K \sum_{i=1}^k(\gslash_{BA_i}\xi_{A_1\dots A_{i-1}CA_{i+1}\dots A_k}-\gslash_{CA_i}\xi_{A_1\dots A_{i-1}BA_{i+1}\dots A_k})    \ee holds, where we recall that  $K$ is the Gauss curvature of $(S_{u,\ubar},\gslash)$. 
\end{lemma}
\subsubsection{Commutation formulae}
Suppose $\mathcal{D} \in \begin{Bmatrix} \modu \snabla_3,\snabla_4,\al \snabla\end{Bmatrix}$. In this subsection we obtain commutation formulae.
\begin{lemma}\label{commutationlemma1}
Suppose $\Phi$ satisfies an equation of the form
\be \snabla_3 \Phi +k \tr\chibar \hsp \Phi =\mathcal{Y}. \ee Then, there holds
\begin{align*}\nonumber \snabla_3 \big(\modu \snabla_3 \Phi\big) + k \tr\chibar \modu \snabla_3 \Phi = \modu \snabla_3 \mathcal{Y} +k \tr\chibar \hsp \Phi \hsp \bigg[ \frac{\modu}{2}\tildetr \bigg] + k\modu (\lvert \chibarhat \rvert^2 -\omegabar \tr\chibar  +\slashed{T}_{33}) \Phi -\mathcal{Y}. \end{align*}

\end{lemma}

\begin{proof}
Simple computation using the formula $[\snabla_3,\modu\snabla_3 ]=-\snabla_3$ (recall that $e_3=\partial_u+b^A \partial_{\theta^A}$ in our frame).
\end{proof}

\begin{lemma}\label{commutationlemma2}Suppose $\Phi$ satisfies an equation of the form $\snabla_3 \Phi +k \tr\chibar \Phi = \mathcal{Y}.$ Then \begin{align*}\nonumber \snabla_3 \snabla_4 \Phi + k \tr\chibar \snabla_4 \Phi =& \snabla_4 \mathcal{Y}  -k(\snabla_4 \tr\chibar)\Phi  + (\eta,\etabar)\snabla \Phi +(\eta,\etabar)(\eta,\etabar)\Phi \notag  +\sigma \hsp \Phi + \omegabar \snabla_4 \Phi \\ =& \snabla_4\mathcal{Y} -k\big(2\sdiv \etabar+2\lvert \etabar \rvert^2 +\chihat\cdot \chibarhat +\frac{1}{2}\tr\chi\tr\chibar +\rho-\frac{1}{2}\Te_{34}\big)\Phi +(\eta,\etabar)\snabla\Phi +(\eta,\etabar)^2\Phi \\ +& \sigma \Phi+\omegabar \snabla_4 \Phi.\end{align*}
\end{lemma}
\begin{proof}
Simple computation, using the commutation formula \eqref{34 commutationf} for $[\snabla_3,\snabla_4]$. 

\end{proof}
\begin{lemma}\label{commutationlemma3}
Suppose $\Phi$ satisfies an equation of the form $\snabla_3 \Phi +k \tr\chibar \Phi = \mathcal{Y}.$ Then \begin{align*} \nonumber 
\snabla_3 \snabla\Phi +\big(k+\frac{1}{2}\big)\tr\chibar \snabla \Phi =& \snabla \mathcal{Y} -k \big(\snabla \tildetr\big) \Phi- k(\eta,\etabar)\tr\chibar \hsp \Phi +(\eta,\etabar)\mathcal{Y} \\ -& \chibarhat \snabla \Phi +\eta(\chibarhat,\tr\chibar)\Phi +\betabar \Phi +\slashed{T}_3 \Phi.
\end{align*}
\end{lemma}
\begin{proof}
    Simple computation, using the commutation formula \eqref{3Bcommutationf} for $[\snabla_3,\snabla_B]$.
\end{proof}We now give the remaining commutation lemmata that describe the pair $(\snabla_4,\mathcal{D})$.
\begin{lemma}\label{commutationlemma4}
Suppose $\Phi$ satisfies an equation of the form $\snabla_4 \Phi =\mathcal{X}$. Then
\begin{align*} \snabla_4(\modu\snabla_3 \Phi)=& \modu \snabla_4\snabla_3 \Phi = \modu \snabla_3 \mathcal{X} +\modu\bigg[ (\eta,\etabar)\snabla \Phi +(\eta,\etabar)(\eta,\etabar)\Phi+\sigma\Phi  \bigg] +\bcancel{\omega \big(\modu \snabla_3 \Phi\big)}-\modu\hspace{.5mm}\omegabar \hsp\mathcal{X}. \end{align*}
\end{lemma}
\begin{proof}
Simple computation, using the commutation formula \eqref{34 commutationf} for $\snabla_3,\snabla_4$ and the fact that $\snabla_4 \modu =0$.
\end{proof}
\begin{lemma}\label{commutationlemma5} Suppose $\Phi$ satisfies an equation of the form $\snabla_4 \Phi =\mathcal{X}$. Then

\[ \snabla_4 \snabla \Phi = \snabla \mathcal{X}+(\chihat,\tr\chi)\snabla\Phi +(\chihat,\tr\chi)\hsp\etabar \hsp\Phi + (\beta-\frac{1}{2}\slashed{T}_4)\Phi. \]
\end{lemma}
\begin{proof}
Simple computation, using the commutation formula for \eqref{4Bcommutationf} $\snabla_4$ and $\snabla$.
\end{proof}
Combining the results of Lemmata \ref{commutationlemma1}-\ref{commutationlemma5}, we arrive at two important commutation formulae, along each of the $e_4$ and $e_3$ directions, which we give in the form of two Propositions. 

\begin{proposition}
\label{commute4}
Assume a tensor-field $\phi$ satisfies $\snabla_4 \phi = F_0$. Then, if $\mathcal{D} \in \{\modu \snabla_3, \snabla_4, \al \snabla \}$,  there holds $\snabla_4 \mathcal{D}^i \phi :=F_i,$ where

\begin{align}
\nonumber F_i =& \sum_{i_1+i_2+i_3=i}\mathcal{D}^{i_1}(\modu\omegabar)^{i_2}\mathcal{D}^{i_3}F_0  \\  \nonumber +& \frac{1}{\al} \sum_{i_1+i_2+i_3+i_4+i_5=i-1}\mathcal{D}^{i_1}\psi_g^{i_2}\mathcal{D}^{i_3}\modu \mathcal{D}^{i_4}(\eta,\etabar) \mathcal{D}^{i_5+1}\phi \\ \nonumber +&\sum_{i_1+i_2+i_3+i_4+1=i}\mathcal{D}^{i_1}\psi_g^{i_2}\mathcal{D}^{i_3}(\tr\chi,\chihat)\mathcal{D}^{i_4+1}\phi \\ \nonumber &+ \al \sumifim \dione \psi_g^{i_2} \dit(\chihat,\tr\chi)\dif \etabar \difi \phi \\ \nonumber +& \al \sumifm \dione \psi_g^{i_2}\dit (\beta,\slashed{T}_{4}) \dif \phi  \\  \nonumber+&\sum_{i_1+i_2+i_3+i_4+i_5=i-1}\mathcal{D}^{i_1}\psi_g^{i_2}\mathcal{D}^{i_3}\modu \mathcal{D}^{i_4}\sigma\mathcal{D}^{i_5}\phi \\ +& \sum_{i_1+\dots+i_6=i-1}\mathcal{D}^{i_1}\psi_g^{i_2}\mathcal{D}^{i_3}\modu \mathcal{D}^{i_4}(\eta,\etabar)\mathcal{D}^{i_5}(\eta,\etabar)\mathcal{D}^{i_6}\phi. \label{commutationformulanabla4}
\end{align}We recall that here $\psi_g= \modu\hsp \omegabar$.
\end{proposition}

\begin{proof}
We use Lemmata \ref{commutationlemma4} and \ref{commutationlemma5} For the base case:
If $\mathcal{D}=\al\snabla,$ then
\be \snabla_4 (\al\snabla) \phi = \al \snabla F_0 + (\chihat,\tr\chi)(\al\snabla) \phi +\al \etabar(\chihat,\tr\chi)\phi +\al (\beta,\slashed{T}_4)\phi, \ee which, as we see, agrees with our claim. If $\mathcal{D}=\modu \snabla_3$, we have \begin{equation}
\snabla_4(\modu\snabla_3\phi) = \modu\snabla_3 F_0 +\modu (\eta,\etabar)\snabla\phi + \modu (\eta,\etabar)(\eta,\etabar)\phi+\modu \sigma \phi + \bcancel{\omega (\modu\snabla_3)\phi}-\modu \omegabar \snabla_4 \phi,
\end{equation}which also agrees with \eqref{commutationformulanabla4}, since $\snabla_4\phi=F_0$. Assume it holds for $i-1$. Then

\begin{align}\snabla_4 \mathcal{D}^i\phi = [\snabla_4, \mathcal{D}]\mathcal{D}^{i-1}\phi + \mathcal{D}F_{i-1}. \end{align}We again distinguish between two cases. If the top-order $\mathcal{D}$ operator is $\al\snabla$, then

\begin{align}\snabla_4(\al\snabla)\mathcal{D}^{i-1}\phi =& (\chihat,\tr\chi)(\al\snabla)\mathcal{D}^{i-1}\phi + \al \etabar(\chihat,\tr\chi)\mathcal{D}^{i-1}\phi + \al(\beta,\Te_4) \mathcal{D}^{i-1}\phi +(\al\snabla) F_{i-1} \nonumber \\ =& (\chihat,\tr\chi)(\al\snabla)\mathcal{D}^{i-1}\phi + \al \etabar(\chihat,\tr\chi)\mathcal{D}^{i-1}\phi + \al(\beta,\Te_4) \mathcal{D}^{i-1}\phi \nonumber \\+& \sum_{i_1+i_2+i_3=i}\mathcal{D}^{i_1}(\modu \omegabar)^{i_2}\mathcal{D}^{i_3}F_0 \notag \\  \nonumber +& \frac{1}{\al} \sum_{i_1+i_2+i_3+i_4+i_5=i-1}\mathcal{D}^{i_1}\psi_g^{i_2}\mathcal{D}^{i_3}\modu \mathcal{D}^{i_4}(\eta,\etabar) \mathcal{D}^{i_5+1}\phi \\ \nonumber +&\sum_{i_1+i_2+i_3+i_4+1=i}\mathcal{D}^{i_1}\psi_g^{i_2}\mathcal{D}^{i_3}(\tr\chi,\chihat)\mathcal{D}^{i_4+1} \phi \\ \nonumber &+ \al \sumifim \dione \psi_g^{i_2} \dit(\chihat,\tr\chi)\dif \etabar \difi \phi \\ \nonumber +& \al \sumifm \dione \psi_g^{i_2}\dit (\beta,\slashed{T}_{4}) \dif \phi  \\  \nonumber+&\sum_{i_1+i_2+i_3+i_4+i_5=i-1}\mathcal{D}^{i_1}\psi_g^{i_2}\mathcal{D}^{i_3}\modu \mathcal{D}^{i_4}\sigma\mathcal{D}^{i_5}\phi \\ +& \sum_{i_1+\dots+i_6=i-1}\mathcal{D}^{i_1}\psi_g^{i_2}\mathcal{D}^{i_3}\modu \mathcal{D}^{i_4}(\eta,\etabar)\mathcal{D}^{i_5}(\eta,\etabar)\mathcal{D}^{i_6}\phi, \end{align}which agrees exactly with \eqref{commutationformulanabla4}.
Similarly, for $\mathcal{D}=\modu \snabla_3$, we have

\begin{align}\nonumber
    \snabla_4 (\modu \snabla_3)\mathcal{D}^{i-1}\phi =& \modu\bigg[ (\eta,\etabar)\snabla \mathcal{D}^{i-1}\phi +(\eta,\etabar)(\eta,\etabar)\mathcal{D}^{i-1}\phi+\sigma\mathcal{D}^{i-1}\phi  \bigg] -\modu \omegabar \snabla_4 \mathcal{D}^{i-1}\phi +(\modu\snabla_3)F_{i-1} \\ \nonumber =& \modu\bigg[ (\eta,\etabar)\snabla \mathcal{D}^{i-1}\phi +(\eta,\etabar)(\eta,\etabar)\mathcal{D}^{i-1}\phi+\sigma\mathcal{D}^{i-1}\phi  \bigg] -\modu \omegabar \snabla_4 \mathcal{D}^{i-1}\phi\\+& \sum_{i_1+i_2+i_3=i}\mathcal{D}^{i_1}(\eta,\etabar,\modu \omegabar)^{i_2}\mathcal{D}^{i_3}F_0 \notag \\  \nonumber +& \frac{1}{\al} \sum_{i_1+i_2+i_3+i_4+i_5=i-1}\mathcal{D}^{i_1}\psi_g^{i_2}\mathcal{D}^{i_3}\modu \mathcal{D}^{i_4}(\eta,\etabar) \mathcal{D}^{i_5+1}\phi \\ \nonumber +&\sum_{i_1+i_2+i_3+i_4+1=i}\mathcal{D}^{i_1}\psi_g^{i_2}\mathcal{D}^{i_3}(\tr\chi,\chihat)\mathcal{D}^{i_4+1}\phi \\ \nonumber &+ \al \sumifim \dione \psi_g^{i_2} \dit(\chihat,\tr\chi)\dif \etabar \difi \phi \\ \nonumber +& \al \sumifm \dione \psi_g^{i_2}\dit (\beta,\slashed{T}_{4}) \dif \phi  \\  \nonumber+&\sum_{i_1+i_2+i_3+i_4+i_5=i-1}\mathcal{D}^{i_1}\psi_g^{i_2}\mathcal{D}^{i_3}\modu \mathcal{D}^{i_4}\sigma\mathcal{D}^{i_5}\phi \\ +& \sum_{i_1+\dots+i_6=i-1}\mathcal{D}^{i_1}\psi_g^{i_2}\mathcal{D}^{i_3}\modu \mathcal{D}^{i_4}(\eta,\etabar)\mathcal{D}^{i_5}(\eta,\etabar)\mathcal{D}^{i_6}\phi, \notag 
\end{align}which again agrees with \eqref{commutationformulanabla4}. Finally, for $\mathcal{D}=\snabla_4,$ we have
\begin{align}
\snabla_4 \snabla_4 \mathcal{D}^{i-1}\phi = \snabla_4 F_{i-1},
\end{align}
which, again, yields \eqref{commutationformulanabla4}. The result follows.
\end{proof}\noindent We finally provide the commutation formula for $\snabla_3$.

\begin{proposition}
\label{propositionnabla3}
Assume a tensorfield $\phi$ satisfies an equation of the form \[ \snabla_3\phi+ a \tr\chibar \phi = G_0.    \]Then if $\ell$ is the number of times $(\al\snabla)$ appears in $\mathcal{D}^{k+\ell}$ and $i:=k+\ell$, there holds:

\begin{align}
\snabla_3 \mathcal{D}^{k+\ell}\phi + \big(\frac{\ell}{2}+a\big)\tr\chibar \mathcal{D}^{k+\ell}\phi = G_i, \notag\end{align}where

\begin{align}
\nonumber G_i&= \sum_{i_1+i_2+i_3=i}\mathcal{D}^{i_1}\big(\eta,\etabar,1 \big)^{i_2} \mathcal{D}^{i_3}G_0 \\ \notag  +& \sum_{i_1+i_2+i_3+i_4+1=i}\dione\psi_g^{i_2}\dit \omegabar \mathcal{D}^{i_4+1}\phi \\ \notag &+\frac{1}{\al}\sumifm \dione \psi_g^{i_2}\dit(\eta,\etabar,\chibarhat)\mathcal{D}^{i_4+1} \phi\\  \notag +&\sumifm \dione \psi_g^{i_2}\dit \sdiv \etabar \dif \phi\\ \notag +& \sumifim \dione \psi_g^{i_2}\dit \chihat \dif \chibarhat \difi \phi \\ \notag +& \sumifim \dione \psi_g^{i_2}\dit \tr\chi \dif \tr\chibar \difi \phi \\ \notag +& \sumifm \dione \psi_g^{i_2}\dit(\rho,\Te_{34})\dif \phi \\ \notag +& \sumifm \dione \psi_g^{i_2}\dit (\betabar,\Te_3)\dif \phi   \\ \notag +&\frac{1}{\al} \sumifm \dione \psi_g^{i_2}\mathcal{D}^{i_3+1}\tildetr \dif \phi \\ \notag &+\sum_{i_1+i_2+i_3+i_4+i_5=i-1}\mathcal{D}^{i_1}\psi_g^{i_2}\mathcal{D}^{i_3}(\eta,\etabar)\mathcal{D}^{i_4}(\eta,\etabar,\chibarhat, \tr\chibar)\mathcal{D}^{i_5}\phi \nonumber \\&+ \nonumber \sum_{i_1+i_2+i_3+i_4=i-1}\mathcal{D}^{i_1}\psi_g^{i_2}\mathcal{D}^{i_3}\sigma\mathcal{D}^{i_4}\phi \notag \\ &+\nonumber\sum_{i_1+i_2+i_3+i_4+i_5+i_6=i-1}\mathcal{D}^{i_1}\psi_g^{i_2}\mathcal{D}^{i_3}\tr\chibar\mathcal{D}^{i_4}\modu \mathcal{D}^{i_5}\tildetr\mathcal{D}^{i_6}\phi \\ &+\nonumber \sum_{i_1+i_2+i_3+i_4+i_5+i_6=i-1}\mathcal{D}^{i_1}\psi_g^{i_2}\mathcal{D}^{i_3} \modu\mathcal{D}^{i_4}\chibarhat\mathcal{D}^{i_5}\chibarhat\mathcal{D}^{i_6}\phi \notag \\ &+\nonumber \sum_{i_1+i_2+i_3+i_4+i_5+i_6=i-1}\mathcal{D}^{i_1}\psi_g^{i_2}\mathcal{D}^{i_3} \modu\mathcal{D}^{i_4}\omegabar\mathcal{D}^{i_5}\tr\chibar\mathcal{D}^{i_6}\phi  \\&+ \sum_{i_1+i_2+i_3+i_4+i_5+i_6=i-1}\mathcal{D}^{i_1}\psi_g^{i_2}\mathcal{D}^{i_3} \modu\mathcal{D}^{i_4}\slashed{T}_{33}\mathcal{D}^{i_5}\phi. \label{commutationformulanabla3}
\end{align}
Here, $\psi_g \in \{\eta,\etabar,1\}.$
\end{proposition}

\begin{proof}
    This is an inductive argument similar to the above.
\end{proof}\noindent

\section{Construction of characteristic seed data}
\label{initialdatasection}

In this section, we explain how the characteristic data assumed in
Theorem~\ref{mainone} may be constructed from free seed data.  The
construction is the Einstein--Vlasov analogue of the characteristic
constraint construction of Christodoulou and Luk.  The only essential
new point is that the Raychaudhuri equation contains the Vlasov null
energy component \(T_{44}\).

\noindent We prescribe the data on the outgoing initial null hypersurface
\(H_{u_\infty}\), with \(u_\infty<0\) large in absolute value.  The data
on the incoming hypersurface \(\underline H_0\) are taken to be
Minkowskian, and the Vlasov distribution is assumed to vanish there.

\noindent Let \(S_{u_\infty,0}\) be the initial sphere.  In stereographic
coordinates \(\theta=(\theta^1,\theta^2)\), we write the induced metric
on the sections \(S_{u_\infty,\underline u}\) in the form
\[
        \gamma_{AB}(u_\infty,\underline u,\theta)
        =
        |u_\infty|^2 \Phi(\underline u,\theta)^2
        \widehat\gamma_{AB}(\underline u,\theta),
\]
where
\[
        \widehat\gamma_{AB}
        =
        \frac{m_{AB}}{\left(1+\frac14|\theta|^2\right)^2},
        \qquad
        m=\exp \Psi,
        \qquad
        \operatorname{tr}\Psi=0.
\]
Thus \(\det m=1\), so that the conformal class is carried by
\(\widehat\gamma\), while the area expansion is carried by the scalar
factor \(\Phi\).

\noindent We choose the conformal seed \(\Psi\) smooth, supported away from the
corner \(\underline u=0\), and satisfying, for \(N\) sufficiently large,
\[
        \sum_{|I|\leq N} |\partial_\theta^I \Psi|
        +
        \sum_{|I|\leq N}
        |\partial_\theta^I \partial_{\underline u}\Psi|
        \lesssim
        \frac{a^{1/2}}{|u_\infty|}.
\]
The corner conditions at the sphere $S_{u_{\infty},0}$ are
\[
        \Psi|_{\underline u=0}=0,
        \qquad
        \Phi(0,\theta)=1,
        \qquad
        \partial_{\underline u}\Phi(0,\theta)=\frac{1}{|u_\infty|}.
\]
These conditions ensure compatibility with the Minkowskian incoming
initial hypersurface.

\noindent Next, prescribe a smooth nonnegative Vlasov seed
\[
        F=F(\underline u,\theta,q)
\]
on the mass shell over \(H_{u_\infty}\), compactly supported in the
renormalized momentum variables
\[
        q^3=p^3,
        \qquad
        q^A=|u_\infty|^2 p^A.
\]
The physical distribution \(f_0\) is obtained by pulling \(F\) back to
the mass shell determined by the metric \(\gamma\).  We assume that on
\(\operatorname{supp} f_0\),
\[
        0<p^3\leq C_{p^3},
        \qquad
        0\leq |u_\infty|^2 p^4 \leq C_{p^4}p^3,
        \qquad
        |u_\infty|^2 |p^A|\leq C_{p^A}p^3,
\]
where the constants are independent of \(u_\infty\).  These are precisely
the focusing support assumptions used later in the Vlasov propagation
argument.

\noindent The outgoing null second fundamental form is
\[
        \chi_{AB}
        =
        \frac12 \partial_{\underline u}\gamma_{AB}.
\]
Since \(\det \widehat\gamma\) is independent of \(\underline u\), we have
\[
        \operatorname{tr}\chi
        =
        2\Phi^{-1}\partial_{\underline u}\Phi,
\]
and
\[
        \widehat\chi_{AB}
        =
        \frac{|u_\infty|^2\Phi^2}{2}
        \partial_{\underline u}\widehat\gamma_{AB}.
\]
Consequently,
\[
        |\widehat\chi|_\gamma^2
        =
        \frac14
        \widehat\gamma^{AC}\widehat\gamma^{BD}
        \partial_{\underline u}\widehat\gamma_{AB}
        \partial_{\underline u}\widehat\gamma_{CD}.
\]

\noindent The scalar factor \(\Phi\) is now determined by the Raychaudhuri
constraint equation.  In the Einstein--Vlasov system this equation is
\[
        \partial_{\underline u}\operatorname{tr}\chi
        +
        \frac12(\operatorname{tr}\chi)^2
        =
        -|\widehat\chi|^2
        -
        T_{44}.
\]
Since \(\operatorname{tr}\chi=2\Phi^{-1}\partial_{\underline u}\Phi\),
this becomes the linear ODE
\[
        \partial_{\underline u}^2\Phi
        +
        \frac12
        \bigl(
            |\widehat\chi|^2+T_{44}
        \bigr)\Phi
        =
        0,
\]
with initial conditions
\[
        \Phi(0,\theta)=1,
        \qquad
        \partial_{\underline u}\Phi(0,\theta)=\frac{1}{|u_\infty|}.
\]
Here
\[
        T_{44}
        =
        \int_{P_x} f_0\,p_4p_4
\]
is computed from the already prescribed Vlasov seed. There is a mild dependence on $\gamma$ because the mass-shell is metric dependent but this does not chnage the main construction.

\noindent For \(|u_\infty|\) sufficiently large relative to \(a\), the coefficient
\(|\widehat\chi|^2+T_{44}\) has small integral along each null generator on $u=u_{\infty}$.
Thus the straightforward ODE analysis of $\partial_{\underline u}^2\Phi
        +
        \frac12
        \bigl(
            |\widehat\chi|^2+T_{44}
        \bigr)\Phi
        =
        0$ with the prescribed initial condition gives
\[
        \frac12 \leq \Phi \leq 2,
\]
and, after commuting with angular derivatives,
\[
        \sum_{|I|\leq N}
        |\partial_\theta^I \Phi|
        +
        \sum_{|I|\leq N}
        |\partial_\theta^I \partial_{\underline u}\Phi|
        \lesssim 1.
\]
It follows that the induced metric, its inverse, \(\operatorname{tr}\chi\),
and \(\widehat\chi\) obey the initial norms required in
Theorem~\ref{mainone}.

\noindent It remains to recover the torsion and the remaining constrained
quantities.  The torsion one-form is obtained from the Codazzi transport
constraint on \(H_{u_\infty}\), schematically
\[
        \partial_{\underline u}\zeta
        +
        \operatorname{tr}\chi\,\zeta
        \sim
        \operatorname{div}\chi
        -
        \nabla\operatorname{tr}\chi
        +
        \mathcal T_4,
\]
where \(\mathcal T_4\) denotes the matter one-form determined by the
restriction of \(T(e_4,\cdot)\) to the spheres.  The right-hand side is
already known in terms of \(\gamma\), \(\chi\), and \(f_0\).  The momentum
support assumptions imply the required bounds for the matter term.  Hence
Gronwall's inequality gives
\[
        \sum_{j\leq N-1}
        \|\nabla^j\zeta\|_{L^\infty(S_{u_\infty,\underline u})}
        \lesssim 1.
\]

\noindent The curvature components on \(H_{u_\infty}\) are then defined through the
null structure equations.  In particular, \(\alpha\) is obtained from the
transport equation for \(\widehat\chi\), \(\beta\) from the Codazzi
equation, and \(\rho,\sigma\) from the Gauss and curl equations.  The
corresponding matter terms are all velocity moments of the same Vlasov
seed \(f_0\).  Therefore no additional free geometric data are introduced.

\noindent Finally, the lower flux condition is obtained by choosing the seeds so
that, uniformly in \(\theta\),
\[
        \int_0^1
        |u_\infty|^2
        \bigl(
            |\widehat\chi|^2
            +
            T_{44}
        \bigr)
        (u_\infty,\underline u,\theta)
        \,d\underline u
        \geq a.
\]
For the shear contribution, one chooses finitely many smooth trace-free
tensor fields \(E_j\) on \(S^2\) and profiles \(h_j\) supported away from
the corner, with
\[
        \sum_j |E_j(\theta)|^2 \geq c>0,
\]
and sets
\[
        \Psi(\underline u,\theta)
        =
        \frac{a^{1/2}}{|u_\infty|}
        \sum_j h_j(\underline u)E_j(\theta).
\]
Then
\[
        \int_0^1
        |u_\infty|^2|\widehat\chi|^2
        \,d\underline u
        \gtrsim a.
\]
Alternatively, one may place part of the focusing in the Vlasov seed by
choosing
\[
        F(\underline u,\theta,q)=a\,G(\underline u,\theta,q),
\]
with \(G\geq 0\), smooth, compactly supported in the admissible momentum
region, and normalized so that
\[
        \int_0^1 |u_\infty|^2 T_{44}\,d\underline u
        \gtrsim a.
\]

\noindent Thus, the characteristic data used in Theorem~\ref{mainone} arise from
free seed data \((\Psi,F)\) satisfying the stated support and regularity
conditions.  The full constrained characteristic initial data set is
obtained by solving the Raychaudhuri ODE for \(\Phi\), recovering
\(\chi\), \(\operatorname{tr}\chi\), and \(\zeta\), and then defining the
curvature components through the null structure equations.  The Vlasov
field enters only through positive velocity moments in the constraint
hierarchy, and these moments are controlled by the same phase-space
support assumptions used later in the evolution.

\section{$\prescript{(S)}{}{\mathcal{O}_{0,4}}$ estimates for the Ricci coefficients}\label{sectionO04}
In this section we begin the estimates on the Ricci coefficients. We commence with an estimate for $\omegabar$:

\begin{proposition}
\label{omegabar04estimate} Under the assumptions of Theorem \ref{mainone} and the bootstrap assumptions \eqref{boundsbootstrap}, there holds

\[  \scalefourSu{\omegabar}\lesssim \mathcal{R}_0^{\frac{1}{2}}[\rho]\mathcal{R}_1^{\frac{1}{2}}[\rho]+ \mathcal{R}_0[\rho]+1.  \] 
\end{proposition}

\begin{proof}
Recall that $\omegabar$ satisfies the equation \[ \snabla_4 \omegabar = 2 \eta\cdot \etabar - \lvert \eta\rvert^2-\rho-\frac{1}{2}\Te_{34}.  \]As such, we have

\begin{align}
\scalefourSu{\omegabar}\lesssim& \intubar \scalefourSuubarprime{\eta\cdot (\eta,\etabar)} \dubarprime +\scalefourSuubarprime{\rho,\Te_{34}}\dubarprime \notag \\ \lesssim& \frac{1}{\modu}\intubar \scaleinftySuubarprime{\eta,\etabar}\scalefourSuubarprime{\eta}\dubarprime + \intubar \scalefourSuubarprime{\rho,\Te_{34}}\dubarprime.
\end{align}At this point we use the following formula from Section 4.3 of \cite{AnThesis}:

\begin{align}
\lVert \psi \rVert_{L^4(S_{u,\ubar})}\lesssim \lVert \psi \rVert_{L^2(S_{u,\ubar})}^{\frac{1}{2}}\lVert \snabla \psi \rVert_{L^2(S_{u,\ubar})}^{\frac{1}{2}} + \frac{1}{\modu^{\frac{1}{2}}} \lVert \psi \rVert_{L^2(S)}
    \end{align}Translating to scale-invariant norms, we have

\begin{align*} \notag &a^{-s_2(\psi)}\modu^{2s_2(\psi)+\frac{1}{2}} \lVert \psi \rVert_{L^4(S)}\\ \lesssim &\big( a^{-s_2(\psi)}\modu^{2s_2(\psi)} \lVert \psi \rVert_{L^2(S_{u,\ubar})}\big) ^\frac{1}{2}\big( a^{-s_2(\psi)-\frac{1}{2}}\modu^{2s_2(\psi)+1} \lVert \snabla\psi \rVert_{L^2(S_{u,\ubar})}\big) ^\frac{1}{2} \cdot a^{\frac{1}{4}}+ a^{-s_2(\psi)}\modu^{2s_2(\psi)}\lVert \psi \rVert_{L^2(S)},
\end{align*}or, equivalently,

\begin{align}
\scalefourSu{\psi}\lesssim \scaletwoSu{\psi}^{\frac{1}{2}}\scaletwoSu{(\al\snabla) \psi}^{\frac{1}{2}}+\scaletwoSu{\psi}.
\end{align}As such, we obtain

\begin{align}
\scaletwoSu{\omegabar}\lesssim& \frac{O^2}{\modu}+ \intubar \scaletwoSuubarprime{\rho}^{\f12}\scaletwoSuubarprime{(\al\snabla)\rho}^{\f12}+\scaletwoSuubarprime{\rho} \dubarprime +\intu \scalefourSuubarprime{\Te_{34}}\dubarprime  \notag \\ \lesssim& \frac{O^2}{\modu}+\big(\intubar \scaletwoSuubarprime{\rho}\dubarprime\big)^\frac{1}{2} \big(\intubar \scaletwoSuubarprime{(\al\snabla)\rho}\dubarprime\big)^\frac{1}{2} \notag \\ +& \intubar \scaletwoSuubarprime{\rho}\dubarprime +\intubar \scalefourSuubarprime{\Te_{34}}\dubarprime \notag \\ \lesssim& 1+ \big(\intubar \scaletwoSuubarprime{\rho}^2\dubarprime\big)^\frac{1}{4} \big(\intubar \scaletwoSuubarprime{(\al\snabla)\rho}^2\dubarprime\big)^\frac{1}{4} \notag  + \big(\intubar \scaletwoSuubarprime{\rho}^2 \dubarprime\big)^{\f12} \notag \\+&\intubar \scalefourSuubarprime{\Te_{34}}\dubarprime.
\end{align}For $\Te_{34}$, since $s_2(\Te_{34})=1$, we have \[\scalefourSu{\Te_{34}} = a^{-1}\modu^{\frac{5}{2}}\big(\int_{S_{u,\ubar}} \lvert \Te_{34} \rvert^4 \sqrt{\det \gslash}\text{d}\theta^1\text{d}\theta^2\big)^{\f14}. \]However, \[ \lvert \Te_{34}(u,\ubar,\theta^1,\theta^2) \rvert =4 \int_{\mathcal{P}_{(u,\ubar,\theta^1,\theta^2)}} f\hsp p^4 \hsp \sqrt{\det\gslash}\hsp \text{d}p^1 \text{d}p^2\text{d}p^3 \lesssim \frac{a}{\modu^4},\]given the bounds on $f, \sqrt{\det \gslash}$ and the momentum components. As such, there holds 
\[ \intubar \scalefourSuubarprime{\Te_{34}}\dubarprime \lesssim \frac{1}{\modu}\lesssim 1.  \]The result follows.
\end{proof}

\noindent We continue with an estimate for $\chihat$.

\begin{proposition}\label{propchihat} Under the assumptions of Theorem \ref{mainone} and the bootstrap assumptions \eqref{boundsbootstrap}, there holds: 

\[ \frac{1}{\al}\scalefourSu{\chihat}\lesssim \mathcal{R}_0^{\frac{1}{2}}[\alpha]\mathcal{R}^{\frac{1}{2}}_1[\alpha]+\mathcal{R}_0[\alpha]+1.   \]
\end{proposition}
\begin{proof}
There holds \be\snabla_4 \chihat +\tr\chi \hsp \chihat = -\alpha. \ee

\noindent As such, there holds:

\begin{equation}
\frac{1}{\al} \scalefourSu{\chihat}\lesssim \frac{1}{\al}\lVert \chihat \rVert_{\mathcal{L}^4_{(sc)}(S_{u,0})}+ \frac{1}{\al} \intubar  \lVert \tr\chi \hsp\chihat  \rVert_{\mathcal{L}^4_{(sc)}(S_{u,\ubar^{\prime}})}\dubarprime +\frac{1}{\al} \intubar \lVert \alpha \rVert_{\mathcal{L}^4_{(sc)}(S_{u,\ubar^{\prime}})}\dubarprime 
\end{equation}
We therefore have, since $\chihat$ vanishes on the incoming cone,

\begin{align}\notag 
&\frac{1}{\al}\scalefourSu{\chihat}\lesssim \frac{1}{\al}\cdot \frac{1}{\modu}O^2 + \frac{1}{\al}\intubar \scaletwoSuubarprime{\alpha}^{\frac{1}{2}}\scaletwoSuubarprime{(\al\snabla){\alpha}}^{\frac{1}{2}}+\scaletwoSuubarprime{\alpha}\dubarprime \\ \lesssim &1 + \bigg( \frac{1}{\al}\intubar \scaletwoSuubarprime{\alpha }\dubarprime\bigg)^{\frac{1}{2}}\bigg(\frac{1}{\al}\intubar \scaletwoSuubarprime{(\al\snabla)\alpha}\dubarprime \bigg)^{\frac{1}{2}}+ \frac{1}{\al}\scaletwoHu{\alpha}\notag \\ \lesssim &1+\bigg(\frac{1}{\al}\scaletwoHu{\alpha}\bigg)^{\frac{1}{2}}\bigg(\frac{1}{\al} \scaletwoHu{(\al\snabla)\alpha}\bigg)^{\frac{1}{2}}+\frac{1}{\al}\scaletwoHu{\alpha} \notag \\ \lesssim & \mathcal{R}_0^{\frac{1}{2}}[\alpha]\mathcal{R}^{\frac{1}{2}}_1[\alpha]+\mathcal{R}_0[\alpha]+1.
\end{align}The claim follows.
\end{proof}
\begin{proposition}
Under the assumptions of Theorem \ref{mainone} and the bootstrap assumptions \eqref{boundsbootstrap}, there holds:  \[ \scalefourSu{\eta}\lesssim \mathcal{R}_0^{\frac{1}{2}}[\beta]\mathcal{R}_1^{\frac{1}{2}}[\beta]+\mathcal{R}_0[\beta]+1.\]
\end{proposition}
\begin{proof}
We recall the equation for $\eta$:

\[   \snabla_4\eta = \frac{1}{2}\tr\chi (\etabar-\eta) + \chihat\cdot(\etabar-\eta)-\beta-\frac{1}{2}\slashed{T}_4.    \]Taking into account that $\eta$ vanishes on the incoming cone, there holds

\begin{align}
\scalefourSu{\eta}\lesssim \frac{\al O^2}{\modu}+ \intubar \big( \lVert{\beta}\rVert_{\mathcal{L}^4_{(sc)}(S_{u,\ubar^{\prime}})}+\lVert {\slashed{T}_4}\rVert_{\mathcal{L}^4_{(sc)}(S_{u,\ubar^{\prime}})}\big) \dubarprime.
\end{align}For $\beta$ we bound the integral as follows: 
\be \intubar \lVert {\beta}\rVert_{\mathcal{L}^4_{(sc)}(S_{u,\ubar^{\prime}})}\dubarprime  \lesssim \scaletwoHu{\beta}^{\frac{1}{2}}\scaletwoHu{(\al\snabla)\beta}^{\frac{1}{2}} +\scaletwoHu{\beta}. \ee By using the bootstrap assumptions, the result holds.
We now control $\scalefourSuubarprime{\slashed{T}_4}$, keeping in mind that the structure equation for $\eta$ implies that $s_2(\Te_{4})=\frac{1}{2}$. There holds 
\[ \scalefourSuubarprime{\slashed{T}_4}= a^{-\frac{1}{2}} \modu^{\frac{3}{2}}\bigg( \int_{S_{u,\ubar^{\prime}}} \lvert \slashed{T}_4 \rvert^4 \sqrt{\det\gslash}\text{d}\theta^1\text{d}\theta^2 \bigg)^{\frac{1}{4}}.   \]For a fixed $A\in\{1,2\}$, there holds $\lvert T_{4A} \rvert = \lvert 2 \gslash_{AA^{\prime}}\int_{\mathcal{P}_x} f p^{A^{\prime}} \sqrt{\det\gslash}\hsp \text{d}p^1 \text{d}p^2\text{d}p^3 \rvert \lesssim \frac{a}{\modu^2}$. Here we have used the fact that \[ \sup_{(x,p)\in \mathcal{P}} \lvert f(x,p) \rvert  \leq a.\] Consequently, $  \slashed{g}^{AC}\slashed{T}_{4A}\slashed{T}_{4C}  \lesssim \frac{a^2}{\modu^6}$. This means that, given the fact that $\sqrt{\det{\gslash}}\lesssim \modu^2$, we have the estimate

\be \intubar \scalefourSuubarprime{\slashed{T}_4}\dubarprime \lesssim \frac{\al}{\modu} \lesssim 1. \ee 
\end{proof}
We move on with estimates for $\tr\chi$.

\begin{proposition}\label{propositiontrchi}
Under the assumptions of Theorem \ref{mainone} and the bootstrap assumptions \eqref{boundsbootstrap}, there holds: \[ \scalefourSu{\tr\chi}\lesssim \frac{a}{\modu}\big(\mathcal{R}_0^{\f12}[\alpha]\mathcal{R}_1^{\f12}[\alpha]+\mathcal{R}_0[\alpha]+1\big)\mathcal{O}_{0,\infty}[\chihat] +1 .\]
\end{proposition}
\begin{proof}
There holds 
\[ \snabla_4 \tr\chi + \frac{1}{2}(\tr\chi)^2 = -\lvert \chihat \rvert^2 -\slashed{T}_{44}.\]Consequently, 
\begin{align}
\scalefourSu{\tr\chi}\lesssim& \intubar \scalefourSuubarprime{\tr\chi \hsp\tr\chi}+\frac{1}{\modu}\scalefourSuubarprime{\chihat}\scaleinfinitySuubarprime{\chihat} +\scalefourSuubarprime{\slashed{T}_{44}} \dubarprime 
\end{align}For the first term, we have

\[ \intubar \scalefourSuubarprime{(\tr\chi)^2}\dubarprime \lesssim \frac{O^2}{\modu}\lesssim 1.   \]For the term involving $\chihat$, we have

\be \frac{1}{\modu}\intubar \scalefourSuubarprime{\chihat}\scaleinftySuubarprime{\chihat}\dubarprime \lesssim  \frac{a}{\modu}\mathcal{O}_{0,\infty}[\chihat]\mathcal{O}_{0,4}[\chihat]\lesssim \frac{a}{\modu}\big(\mathcal{R}_0^{\f12}[\alpha]\mathcal{R}_1^{\f12}[\alpha]+\mathcal{R}_0[\alpha]+1\big)\mathcal{O}_{0,\infty}[\chihat]. \ee

We proceed to control $\slashed{T}_{44}$. There holds 
\be \scalefourSuubarprime{\slashed{T}_{44}}= \modu^{\frac{1}{2}} \big(\int_{S_{u,\ubar^{\prime}}} \lvert \slashed{T}_{44}\rvert^4 \sqrt{\det{\gslash}}\hsp \text{d}\theta^1 \text{d}\theta^2\big)^{\frac{1}{4}}. \ee However, $\lvert \slashed{T}_{44} \rvert = \int_{\mathcal{P}_{x}} f p^3 \sqrt{\det\gslash}\hsp \text{d}p^1\text{d}p^2\text{d}p^3 \lesssim \frac{a}{\modu^2},$ where we have used the bounds on $f$ and $\sqrt{\det \gslash}$. As a consequence, we have
\[  \intubar \scalefourSuubarprime{\slashed{T}_{44}}\lesssim \frac{a}{\modu}\lesssim 1.  \]

\end{proof}
\begin{remark}
The estimate for $\tr\chi$ will be improved after closing the bootstrap for the supremum norm of $\chihat$, in Section \ref{section0infty}.
\end{remark}
We proceed with components which satisfy equations in the $e_3-$direction.
\begin{proposition}\label{propchibarhat}
Under the assumptions of Theorem \ref{mainone} and the boottstrap assumptions \eqref{boundsbootstrap}, there holds

\[ \frac{\al}{\modu}\scalefourSu{\chibarhat}\lesssim 1.   \]
\end{proposition}
\begin{proof}
There holds

\be \snabla_3\chibarhat +\tr\chibar \hsp \chibarhat = \omegabar \hsp \chibarhat - \alphabar. \ee Using Proposition \ref{evolemma}, passing to scale-invariant norms and multiplying by $a^{-1/2}$, we have 

\begin{align}  \frac{\al}{\modu} \scalefourSu{\chibarhat}\lesssim& \frac{\al}{\lvert u_{\infty}\rvert} \scalefourSuzero{\chibarhat } +\intu \frac{a^{\f32}}{\upr^3} \big(\scalefourSuprime{\omegabar \hsp \chibarhat} +\scalefourSuprime{\alphabar}\big) \duprime \notag \\ \lesssim& \frac{\al}{\lvert u_{\infty}\rvert} \scalefourSuzero{\chibarhat } + \intu \frac{a^{\f32}}{\upr^3} \cdot \frac{\upr}{\al}\cdot \frac{O^2}{\upr} \duprime \notag \\ +& \intu \frac{a^{\f32}}{\upr^3}\big(\scaletwoSuprime{\alphabar}^{\f12}\scaletwoSuprime{(\al\snabla)\alphabar}^{\f12}+\scaletwoSuprime{\alphabar}\big) \duprime \notag \\ \lesssim& \frac{\al}{\lvert u_{\infty}\rvert} \scalefourSuzero{\chibarhat } + \intu \frac{a^{\f32}}{\upr^3} \cdot \frac{\upr}{\al}\cdot \frac{O^2}{\upr} \duprime \notag \\ +& \frac{a}{\modu^{\f32}}\big(\scaletwoHbaru{\alphabar}^{\f12}\scaletwoHbaru{(\al\snabla)\alphabar}^{\f12}+\scaletwoHbaru{\alphabar}\big)\lesssim 1, \end{align}where in the last line we have used H\"older's inequality. The result follows.

\end{proof}
\begin{proposition} \label{propositiontildetr}
Under the assumptions of Theorem \ref{mainone} and the bootstrap assumptions \eqref{boundsbootstrap}, there holds

\[ \frac{a}{\modu^2}\scalefourSu{\tr\chibar} \lesssim 1, \hspace{3mm} \frac{a}{\modu}\scalefourSu{\tildetr}\lesssim \frac{a}{\modu}\mathcal{O}_{0,\infty}[\chibarhat]+1.     \]
\end{proposition}

\begin{proof}
There holds

\[\snabla_3 \tr\chibar+\frac{1}{2}(\tr\chibar)^2 = -\lvert \chibarhat\rvert^2-\slashed{T}_{33}\]Using the evolution lemma from Proposition \ref{evolemma} and passing to scale-invariant norms, we have

\begin{align}
\frac{a}{\modu^2}\scalefourSu{\tr\chibar} \lesssim \frac{a}{\lvert u_\infty \rvert^2}\lVert \tr\chibar \rVert_{\mathcal{L}^4_{(sc)}(S_{u_\infty,\ubar})} + \intu \frac{a^2}{\upr^4}\scalefourSuprime{\lvert \chibarhat\rvert^2, \hsp \slashed{T}_{33}}. \label{trchibarboundputback}
\end{align}For $\slashed{T}_{33},$ given that the signature is $s_2(\slashed{T}_{33})=2,$ we have \[  
\scalefourSu{\slashed{T}_{33}}=a^{-2}\modu^{\frac{9}{2}} \fourSu{\slashed{T}_{33}}=  a^{-2}\modu^{\frac{9}{2}}\bigg(\int_{S_{u,\ubar}}\lvert \slashed{T}_{33}\rvert^4 \sqrt{\det\gslash}\hsp \text{d}\theta^1 \hsp \text{d}\theta^2 \bigg)^{\frac{1}{4}}.    \] However,

\be \lvert \slashed{T}_{33}\rvert =4 \int_{\mathcal{P}_x} f p^4 p^4 \frac{\sqrt{\det \gslash}}{p^3}\text{d}p^1\text{d}p^2\text{d}p^3 \lesssim \frac{a}{\modu^6}, \ee whence \be \scalefourSuprime{\slashed{T}_{33}}\lesssim \frac{1}{a\lvert u^{\prime}\rvert}. \ee Putting this back to  \eqref{trchibarboundputback}, we get the desired result.

\vspace{3mm}

\noindent For $\tildetr$, keeping in mind that $\snabla_3=\partial_u$ in our setting, the following equation holds:

\[ \snabla_3 \tildetr + \tr\chibar \tildetr = \frac{1}{2}\tildetr \tildetr-\lvert \chibarhat\rvert^2- \slashed{T}_{33}.     \]Using the Evolution Lemma \ref{evolemma} and passing to scale-invariant norms, we obtain

\begin{align}
\frac{a}{\modu}\scalefourSu{\tildetr} \lesssim& \frac{a}{\lvert u_{\infty}\rvert} \scalefourSuzero{\tildetr} \notag \\ +& \intu \frac{a^2}{\upr^3}\big(\scalefourSuprime{\tildetr^2}\duprime + \scalefourSuprime{\lvert \chibarhat \rvert^2}+\scalefourSuprime{\slashed{T}_{33}}\big)\duprime .
\end{align}We have

\be \intu \frac{a^2}{\upr^3}\scalefourSuprime{\tildetr \hsp \tildetr}\duprime \lesssim \intubar \frac{a^2}{\upr^3}\cdot \frac{\upr^2}{a^2}\cdot \frac{O^2}{\upr} \duprime = \frac{O^2}{\modu} \lesssim 1, \ee

\be \intu \frac{a^2}{\upr^3}\scalefourSuprime{\lvert \chibarhat\rvert^2}\duprime \lesssim \intu \frac{a^2}{\upr^3}\cdot \frac{\upr^2}{a}\cdot \frac{O^2}{\upr} \duprime = \frac{a \mathcal{O}_{0,\infty}[\chibarhat]\mathcal{O}_{0,4}[\chibarhat]}{\modu} \lesssim \frac{a}{\modu}\mathcal{O}_{0,\infty}[\chibarhat], \ee
\be \intu \frac{a^2}{\upr^3}\scalefourSuprime{\Te_{33}}\duprime \lesssim \intubar \frac{a^2}{\upr^3} \cdot \frac{1}{a\upr}\duprime \lesssim \frac{a}{\modu^3}\lesssim 1. \ee

\end{proof}

\begin{proposition}
Under the assumptions of Theorem \ref{mainone} and the bootstrap assumptions \eqref{boundsbootstrap}, there holds

\be \scalefourSu{\etabar}\lesssim \prescript{(S)}{}{\mathcal{O}}_{0,\infty}[\tr\chibar] \big(\mathcal{R}_0^{\frac{1}{2}}[\beta]\mathcal{R}_1^{\frac{1}{2}}[\beta]+\underline{\mathcal{R}}_0[\beta]+1\big). \ee
\label{etabarint04}
\end{proposition}

\begin{proof}
There holds \[\snabla_3 \etabar+ \frac{1}{2}\tr\chibar \hsp \etabar = \frac{1}{2}\tr\chibar \hsp \eta- \chibarhat\cdot(\etabar-\eta)+\betabar-\frac{1}{2}\slashed{T}_3:= F_{\etabar}.  \]We recall that $s_2(\etabar)=\frac{1}{2}$. Using Proposition \ref{evolemma} with $\lambda_0=\frac{1}{2},$ we obtain

\[ \modu^{\frac{1}{2}}\fourSu{\etabar}\lesssim \lvert u_{\infty}\rvert^{\frac{1}{2}} \fourSuinf{\etabar}+ \intu \lvert u^{\prime}\rvert^{\frac{1}{2}}\lVert F_{\etabar} \rVert_{L^4(S_{u^{\prime},\ubar})}\duprime,      \]which rewrites in scale-invariant norms as \begin{align}\label{etabarl4equationone}\notag
    \frac{\al}{\modu}\scalefourSu{\etabar}\lesssim& \frac{\al}{\lvert u_{\infty}\rvert}\scalefourSuinf{\etabar} + \intu \frac{a^{\frac{3}{2}}}{\lvert u^{\prime}\rvert^3} \scalefourSuprime{(\tr\chibar,\chibarhat)(\eta,\etabar)} \duprime \\ +& \intu \frac{a^{\frac{3}{2}}}{\lvert u^{\prime}\rvert^3} \scalefourSuprime{\betabar}\duprime + \intu \frac{a^{\frac{3}{2}}}{\lvert u^{\prime}\rvert^3} \scalefourSuprime{\slashed{T}_3}\duprime .
\end{align} From the first integral, the most dangerous term is \[ \intu \frac{a^{\frac{3}{2}}}{\upr^3} \scalefourSuprime{\tr\chibar \hsp \eta}\duprime \lesssim \intu \frac{a^{\frac{3}{2}}}{\upr^3}\cdot \frac{\upr^2}{a}\cdot\frac{1}{\upr}\bigg(\frac{a}{\upr^2}\scaleinfinitySuprime{\tr\chibar}\bigg)\scalefourSuprime{\eta}\duprime.  \]Using the fact that $\frac{a}{\upr^2}\scaleinfinitySuprime{\tr\chibar}= \prescript{(S)}{}{\mathcal{O}}_{0,\infty}[\tr\chibar]$  and $\scalefourSuprime{\eta}\lesssim \mathcal{R}_0^{\frac{1}{2}}[\beta]\mathcal{R}_1^{\frac{1}{2}}[\beta]+ \mathcal{R}_0[\beta]+1$, it follows that

\be \intu \frac{a^{\frac{3}{2}}}{\lvert u^{\prime}\rvert^3} \scalefourSuprime{\tr\chibar\hsp  \eta} \duprime \lesssim \frac{\al}{\modu}\big(\mathcal{R}^{\frac{1}{2}}_0[\beta]\mathcal{R}^{\frac{1}{2}}_1[\beta]+\mathcal{R}_0[\beta]+1\big). \ee The rest of the terms in the first integral in \eqref{etabarl4equationone} are bounded by $\frac{\al}{\modu}$. For the other two terms, we see that

\be \intu\frac{a^{\frac{3}{2}}}{\upr^3} \scalefourSuprime{\betabar,\slashed{T}_3}\duprime \lesssim  \frac{\al}{\modu}\big( \scaletwoHbaru{\betabar}^{\frac{1}{2}}\scaletwoHbaru{(\al\snabla)\betabar}^{\frac{1}{2}}+\scaletwoHbaru{\betabar} +1\big). \ee We note that 

\be \fourSu{\slashed{T}_3}= \bigg(\int_{S_{u^{\prime},\ubar}} \big\lvert \gslash^{AC}\slashed{T}_{3A}\slashed{T}_{3C}\big\rvert^2 \sqrt{\det\gslash}\hsp \text{d}\theta^1 \text{d}\theta^2 \bigg)^{\frac{1}{4}}.  \ee Using the fact that  \[\lvert \slashed{T}_{3A} \rvert \lesssim \bigg\lvert\gslash_{AA^{\prime}} \int_{\mathcal{P}_x} f p^4 \hsp p^A \frac{\sqrt{\det\gslash}}{p^3}\hsp \text{d}p^1\text{d}p^2\text{d}p^3 \bigg\rvert \lesssim \frac{a}{\upr^4} ,   \]we arrive as before at \be \fourSu{\slashed{T}_3}\lesssim \frac{a}{\modu^{\frac{9}{2}}}, \ee which means that

\be \intu \frac{a^{\frac{3}{2}}}{\upr^3}\scalefourSuprime{\slashed{T}_3}\duprime \lesssim \frac{a}{\modu^3} \lesssim \frac{\al}{\modu}. \ee Dividing the above by $\frac{\al}{\modu}$, we obtain the desired result.
\end{proof}
\begin{remark}
As with other preceding terms, the above estimates for $\tr\chi, \tildetr, \etabar$ will be improved in the section with $\prescript{(S)}{}{\mathcal{O}}_{0,\infty}$ estimates. At this stage, we do not have control on $\scaleinftySu{\chihat},\scaleinftySu{\chibarhat},\scaleinftySu{\tr\chibar}$, as this requires control of their first angular derivatives in $\mathcal{L}^4_{(sc)}$ as well, which is carried out in the next section.
\end{remark}
\noindent We now obtain estimates in $\mathcal{L}^4_{(sc)}$ for the non-tensorial quantity $\mathcal{G}$ (recalling its definition from \eqref{Gintroduction}):
\begin{proposition}
 Under the assumptions of Theorem \ref{mainone} and the bootstrap assumptions \eqref{boundsbootstrap}, there holds
 \be \scalefourSu{\mathcal{G}}\lesssim \scalefourSuzero{\mathcal{G}}+ \prescript{(S)}{}{\mathcal{O}}_{1,4}[\eta]+ 1.   \ee
\end{proposition}

\begin{proof} 
First of all, by definition, $\mathcal{G}$ inherits the imposed bootstrap assumption on the Ricci coefficients
\[ \scalefourSu{\mathcal{G}}\leq \mathcal{O}.   \]

\noindent Let us recall the transport equation

\begin{align}
\snabla_4 \mathcal{G} \sim \nablasl \eta + \mathcal{G} \cdot \chi + \omegabar \cdot \chi + \Gammaslash \cdot (\eta, \etabar) + \eta \cdot \etabar + (\rho,\sigma)+\Tslash,    
\end{align}from equation \eqref{Gintroduction}. Using the transport inequality along the $e_4$ direction, together with the bootstrap assumptions, we obtain

\begin{align}
\scalefourSu{\mathcal{G}}\lesssim \scalefourSuzero{\mathcal{G}}+\intubar \scalefourSuubarprime{\snabla \eta} \dubarprime +1. 
\end{align}Here, we have used $\mathcal{L}^4\infty_{(sc)}-\mathcal{L}^{4}_{(sc)}$ estimates for the nonlinear terms $\mathcal{G}\cdot \chi, \hsp \omegabar\cdot \chi, \Gammaslash \cdot (\eta,\etabar)$ and $\eta\cdot \etabar$. The result follows.
\end{proof}
\begin{proposition}
    Under the assumptions of Theorem \ref{mainone} and the bootstrap assumptions \eqref{boundsbootstrap}, there holds

\[ \big\lVert \Gammaslash-\Gammaslash^{\circ}\big\rVert_{\mathcal{L}^4_{(sc)}(S_{u,\ubar})} \lesssim \prescript{(S)}{}{\mathcal{O}}_{1,4}[\chihat]+ 1.      \]
    
\end{proposition}

\begin{proof}
We recall the equation 

\begin{align} \notag 
\snabla_4\big(\Gammaslash-\Gammaslash^{\circ}\big)^C_{AB} =& \snabla_A {\chi_B}^C + \snabla_B {\chi_A}^C-\snabla^C {\chi_{BA}} \\ \notag  -&\big({\chi_A}^D - \snabla_A b^D +b^E \Gammaslash^D_{AE}\big)\big(\Gammaslash-\Gammaslash^{\circ}\big)^C_{DB}  \\ \notag  -&\big({\chi_B}^D - \snabla_B b^D +b^E \Gammaslash^D_{BE}\big)\big(\Gammaslash-\Gammaslash^{\circ}\big)^C_{AD} \\ +&\big({\chi_D}^C - \snabla_D b^C +b^E \Gammaslash^C_{DE}\big)\big(\Gammaslash-\Gammaslash^{\circ}\big)^D_{AB} .
\end{align}Notice that \[{\chi_A}^D - \snabla_A b^D +b^E \hsp \Gammaslash^D_{AE}={\chi_A}^D -e_A(b^D) = \Gamma^D_{3A}.\]Hence the above equation rewrites as
\begin{align} \notag 
\snabla_4\big(\Gammaslash-\Gammaslash^{\circ}\big)^C_{AB} =& \snabla_A {\chi_B}^C + \snabla_B {\chi_A}^C-\snabla^C {\chi_{BA}} \\ \notag  -&\Gamma^D_{3A} \big(\Gammaslash-\Gammaslash^{\circ}\big)^C_{DB}   -\Gamma^D_{3B}\big(\Gammaslash-\Gammaslash^{\circ}\big)^C_{AD} +\Gamma^C_{3D}\big(\Gammaslash-\Gammaslash^{\circ}\big)^D_{AB} .
\end{align} Recalling the definition of $\mathcal{G}$ from \eqref{Gintroduction}, we thus have the schematic identity

\be \snabla_4(\Gammaslash-\Gammaslash^{\circ})=\snabla\chi + \mathcal{G}(\Gammaslash-\Gammaslash^{\circ}). \ee As such, there holds 

\begin{align} \notag \fourSu{\Gammaslash-\Gammaslash^{\circ}} \lesssim& \intubar \frac{1}{\al} \scalefourSuubarprime{(\al\snabla) \chihat} + \frac{1}{\modu}\big(\sum_{i=0}^1\scalefourSuubarprime{(\al\snabla)^i \mathcal{G}}\big)\scalefourSuubarprime{\Gammaslash-\Gammaslash^{\circ}} \dubarprime \\ \lesssim& \intubar \frac{1}{\al} \scalefourSuubarprime{(\al\snabla) \chihat} \dubarprime +1.  
\end{align} In the last equation we have used the $\mathcal{L}^{\infty}_{(sc)}-\mathcal{L}^4_{(sc)}$ Sobolev embedding on $\mathcal{G}$.

\end{proof}
\noindent This concludes the $\prescript{(S)}{}{\mathcal{O}}_{0,4}$-estimates for the Ricci coefficients.




\section{$\prescript{(S)}{}{\mathcal{O}_{1,4}}$ estimates for the Ricci coefficients} \label{sectionO14}

We now move to the $L^4(S_{u,\ubar})$ estimates for the first derivatives of the Ricci coefficients. We recall at this point that $\mathcal{D}$ denotes an arbitrary differential operator in the set $\{ \modu \snabla_3, \snabla_4, \al \snabla\}$. 

\vspace{3mm}

\noindent We note the following: 

\begin{proposition}
Consider a Ricci coefficient $\psi^{(s)}$ with $s_2-$signature $s$ that satisfies an equation of the form

\[ \snabla_4 \psi^{(s)}= \sum_{s_1+s_2=s} \psi^{(s_1)}\psi^{(s_2)}+\Psi^{(s)}+\Vl^{(s)}.\]Then, if $\mathcal{D}$ is as above, there holds

\begin{align}
\snabla_4 \mathcal{D}\psi^{(s)} =& \sum_{s_1+s_2=s}\psisone \mathcal{D}\psistwo +\mathcal{D}\Psi^{(s)} +\mathcal{D}\Vl^{(s)}+ \frac{1}{\al}\modu (\eta,\etabar)\mathcal{D}\psi^{(s)}+ \modu(\eta,\etabar)(\eta,\etabar)\psis \notag \\ +& \modu \sigma \psis +(\chihat,\tr\chi)\mathcal{D}\psis + \al(\chihat,\tr\chi)\etabar \psi^{(s)} +\al(\beta,\slashed{T}_4)\psi^{(s)} \notag \\ +& \modu \hsp \omegabar\hsp \big(\sum_{s_1+s_2=s} \psisone \psistwo+\Psis +\Vls\big) .
\end{align} As such, the following estimate holds:

\begin{align}
&\scalefourSu{\mathcal{D}\psis}\notag \\ \lesssim& \scalefourSuzero{\mathcal{D}\psis} +\sum_{s_1+s_2=s}\sup_{0\leq \ubar^{\prime} \leq 1} \frac{1}{\modu}\scaleinftySuubarprime{\psisone}\scalefourSuubarprime{\mathcal{D}\psistwo} \notag \\ +& \scaletwoHu{\mathcal{D}\Psis}^{\frac{1}{2}}\scaletwoHu{(\al\snabla)\mathcal{D}\Psis}^{\frac{1}{2}}+ \scaletwoHu{\mathcal{D}\Psis}\\ +& \scaletwoHu{\mathcal{D}\Vls}^{\frac{1}{2}}\scaletwoHu{(\al\snabla)\mathcal{D}\Vls}^{\frac{1}{2}}+\scaletwoHu{\mathcal{D}\Vls} \notag \\ +& \frac{\al}{\modu^2}\sup_{0\leq \ubar^{\prime} \leq 1}\scaleinftySuubarprime{(\eta,\etabar)}\scalefourSuubarprime{\mathcal{D}\psi^{(s)}} + \frac{a}{\modu^3} \sup_{0\leq \ubar^{\prime}\leq 1}\scaleinftySuubarprime{(\eta,\etabar)}^2 \scalefourSuubarprime{\psis} \notag \\ +& \frac{a}{\modu^2}\sup_{0\leq \ubar^{\prime}\leq 1}\scaleinfinitySuubarprime{\psis}\big(\scaletwoHu{\sigma}^{\frac{1}{2}}\scaletwoHu{(\al\snabla)\sigma}^{\frac{1}{2}} + \scaletwoHu{\sigma}\big)\notag \\ +& \frac{\al}{\modu} \sup_{0\leq \ubarprime\leq 1}\scaleinftySuubarprime{(\frac{1}{\al}\chihat,\frac{1}{\al}\tr\chi)}\scalefourSuubarprime{\mathcal{D}\psis} \notag \\ +& \frac{\al}{\modu^2}\sup_{0\leq \ubarprime\leq 1}\scaleinftySuubarprime{(\frac{1}{\al}\chihat,\frac{1}{\al}\tr\chi)}\scaleinftySuubarprime{\etabar}\scalefourSuubarprime{\psis}\notag \\+& \frac{\al}{\modu}\sup_{0\leq \ubarprime\leq 1}\scaleinftySuubarprime{\psis} \big( \scaletwoHu{(\beta,\slashed{T}_4)}^{\frac{1}{2}}\scaletwoHu{(\al\snabla)(\beta,\slashed{T}_4)}^{\frac{1}{2}}+ \scaletwoHu{(\beta,\slashed{T}_4)}\big) \notag \\ +& \sum_{s_1+s_2=s}\frac{a}{\modu^3}\sup_{0\leq \ubar^\prime\leq 1}\scaleinftySuubarprime{\psisone}\scaleinftySuprime{\psistwo}\scalefourSuprime{\omegabar} \notag \\ +&\sup_{0\leq \ubarprime\leq 1}\frac{a}{\modu^2}\scaleinftySuubarprime{\omegabar} \big(\scaletwoHu{\Psis}^{\f12}\scaletwoHu{(\al\snabla)\Psis}^{\f12}+\scaletwoHu{\Psis}\big) \notag \\+& \sup_{0\leq \ubarprime\leq 1}\frac{a}{\modu^2}\scaleinftySuubarprime{\omegabar} \big(\scaletwoHu{\Vls}^{\f12}\scaletwoHu{(\al\snabla)\Vls}^{\f12}+\scaletwoHu{\Vls}\big).
\end{align}

\end{proposition}

\begin{proof}
This is a repeated use of the scale-invariant H\"older inequalities together with the co-dimension 1 trace inequality for the $\mathcal{L}^4_{(sc)}(S)-$ norms of $\sigma,\Psis,\Vls, \beta$ and $\slashed{T}_4$.
\end{proof}

\noindent Similarly, for tensor-fields satisfying equations in the $\snabla_3$-direction, the following holds:

\begin{proposition}
Assume a Ricci coefficient $\psis$ with $s_2$-signature $s$ satisfies an equation of the form
\[ \snabla_3 \psis + \lambda(\psis)\tr\chibar \hsp \psis = \sum_{s_1+s_2=s+1}\psisone\psistwo+\Psi^{(s+1)}+\Vl^{(s+1)} .   \]Then, if $\mathcal{D}$ is as above and $\ell$ is as in Proposition \ref{propositionnabla3}, the following holds: 

\begin{align*}
&\snabla_3 \mathcal{D}\psis+ \big(\lambda(\psis)+\frac{\ell}{2}\big) \tr\chibar \mathcal{D}\psis\notag \\ =& \sum_{s_1+s_2=s+1}\mathcal{D}\psisone \hsp \psistwo + \mathcal{D}\Psi^{(s+1)}+\mathcal{D}\Vl^{(s+1)}  \\ +&(\eta,\etabar,1)\big(\sum_{s_1+s_2=s+1}\psisone \psistwo +\Psi^{(s+1)}+\Vl^{(s+1)}\big)  +\modu \tildetr \tr\chibar \psis +\modu(\lvert \chibarhat \rvert^2+ \omegabar \tr\chibar +\slashed{T}_{33})\psis \\+& \frac{1}{\al}(\al \omegabar, \eta,\etabar,\chibarhat)\hsp \De \psis +\sdiv \etabar \hsp  \psis +(\chihat \cdot \chibarhat +\tr\chi\hsp \tr\chibar)\psis +(\rho,\sigma,\betabar)\psis +(\Te_{3},\Te_{34})\psis \notag \\ +& \frac{1}{\al}\De\tildetr \hsp \psis +(\eta,\etabar)(\eta,\etabar,\chibarhat,\tr\chibar)\psis := F_{\psis}. 
\end{align*}Consequently, 

\begin{align*}
    &\modu^{2\lambda(\psis)+\ell-\frac{1}{2}}\fourSu{\mathcal{D}\psis} \\\lesssim& \lvert u_{\infty}\rvert^{2\lambda(\psis)+\ell-\frac{1}{2}}\lVert \mathcal{D}\psis \rVert_{L^4(S_{u_{\infty}},\ubar)}+ \intu \upr^{2\lambda(\psis)+\ell-\frac{1}{2}}\scalefourSuprime{F_{\psis}}\duprime.
\end{align*}
\end{proposition}After translating this to scale-invariant norms, we obtain the desired inequalities by using the co-dimension 1 trace inequality on $\Psi^{(s+1)},\Vl^{(s+1)}, \betabar,\slashed{T}_3$ and standard scale-invariant H\"older inequalities.

\vspace{3mm}

\noindent We begin with the estimates for $\omegabar$.
\begin{proposition}
Under the assumptions of Theorem \ref{mainone} and the bootstrap assumptions \eqref{boundsbootstrap}, there holds 
\[ \scalefourSu{\De \omegabar}\lesssim \mathcal{R}^{\f12}_1[\rho]\mathcal{R}^{\f12}_2[\rho]+\mathcal{R}_1[\rho]+1.    \]
\end{proposition}

\begin{proof}
We recall that \[\snabla_4 \omegabar = 2 \eta\cdot \etabar - \lvert \eta \rvert^2-\rho-\frac{1}{2}\Te_{34}.\]As such, there holds:

\begin{align}
\snabla_4 \De \omegabar =&  (\eta,\etabar)\hsp \De \eta + \eta \hsp \De(\eta,\etabar)+\De \rho+ \De\Te_{34} +\modu \hsp \omegabar \hsp (\eta\hsp (\eta,\etabar)+\rho+\Te_{34})+ \frac{1}{\al}\modu (\eta,\etabar)\De \omegabar +\modu (\eta,\etabar)^2\omegabar \notag \\ +& \modu \sigma \hsp \omegabar + (\chihat,\tr\chi)\De \omegabar +\al (\chihat,\tr\chi)\etabar \hsp \omegabar +\al (\beta,\Te_{4})\omegabar.
\end{align}Consequently, since $\De\omegabar$ vanishes on $\Hbar_0$, there holds 

\begin{align}
\scaletwoSu{\De\omegabar}\lesssim \intubar \scalefourSuubarprime{\snabla_4 \De \omegabar}\dubarprime \lesssim \intubar \scalefourSuubarprime{\De \rho} + \scalefourSuubarprime{\De\Te_{34}}\dubarprime +1,
\end{align}where we have used the bootstrap assumptions \eqref{boundsbootstrap} to bound the majority of the terms above by 1.
For the integral on the right-hand side, there holds 

\be \intubar \scalefourSuubarprime{\De(\rho,\Te_{34})}\dubarprime \lesssim \scaletwoHu{\De(\rho,\Te_{34})}^{\f12}\scaletwoHu{\De^2(\rho,\Te_{34})}^{\f12}+\scaletwoHu{\De(\rho,\Te_{34})}.\ee Using the bounds assumed in \eqref{boundsbootstrap} for $\scaletwoHu{\De\Te_{34}},\scaletwoHu{\De^2\Te_{34}},$ the result follows.
\end{proof}

\begin{proposition}\label{propositionDchihat} Under the assumptions of Theorem \ref{mainone} and the bootstrap assumptions \eqref{boundsbootstrap}, there holds

\[  \frac{1}{\al} \scalefourSu{\mathcal{D}\chihat } \lesssim \mathcal{R}_1^{\frac{1}{2}}[\alpha]\mathcal{R}_2^{\frac{1}{2}}[\alpha] +\mathcal{R}_1[\alpha]+1. \]
\end{proposition}
\begin{proof}
We recall that \[ \snabla_4 \chihat = \tr\chi\hsp \chihat +\alpha .\] As such, there holds 

\begin{align}
\notag \snabla_4 \mathcal{D}\chihat=& \mathcal{D} \tr\chi \hsp \chihat +\tr\chi\mathcal{D}\chihat + \mathcal{D}\alpha + (\eta,\etabar)\hsp\tr\chi\hsp \chihat +\modu \omegabar\big(\tr\chi \hsp \chihat +\alpha\big) \\ +& \frac{1}{\al}\modu (\eta,\etabar)\mathcal{D}\chihat  + (\chihat,\tr\chi)\mathcal{D}\chihat + \al(\chihat,\tr\chi)\etabar \chihat + \al(\beta,\slashed{T}_4)\chihat \notag \\ +& \modu (\eta,\etabar)(\eta,\etabar)\chihat +\modu \sigma \chihat . \label{first}
\end{align}Let us restate that in the above schematic, for a given (fixed) $\mathcal{D}$, not all terms appear in the actual equation, but rather what is true in any case is that there is \underline{equality of total signature} between the left and right hand sides. For example, when we consider $\snabla_4(\modu \snabla_3)\chihat,$ the term $\frac{1}{\al}\modu(\eta,\etabar) \mathcal{D}\chihat$ does appear and this second time $\mathcal{D}$ is equal to $(\al\snabla)$. If, however, $\De$ on the left-hand side equals $(\al\snabla)$, then on the right-hand side $\frac{1}{\al}\modu(\eta,\etabar) \mathcal{D}\chihat$ does not appear, but instead the term
$\al(\beta,\Te_4)\chihat$ does and so on. It is a convenient fact that the schematic commutation formulae allow us to perform estimates on an arbitrary $\De$-derivative without the need to distinguish cases. This is possible precisely because, in the end, every term on the right-hand side is estimated in its own scale-invariant norm.  

\vspace{3mm}

\noindent Using the $L^{\infty}(S), L^4(S)$ bootstrap bounds on the Ricci coefficients, together with the bootstrap bounds on the curvature norms and the obtained $L^4$ bounds on  $\slashed{T}_4$, we obtain \begin{align} &\frac{1}{\al}\scalefourSu{\mathcal{D}\chihat} \notag \\ \lesssim& \intubar \frac{1}{\al} \scalefourSuubarprime{\snabla_4 \mathcal{D}\chihat}\dubarprime \lesssim \intubar \frac{1}{\al}\scalefourSuubarprime{\De\alpha}\duprime +1 \lesssim \mathcal{R}_1^{\frac{1}{2}}[\alpha]\mathcal{R}_2^{\frac{1}{2}}[\alpha] +\mathcal{R}_1[\alpha]+1. \end{align}
\end{proof}
\begin{remark}\label{improvechihat}
We can now obtain the estimate \be \frac{1}{\al}\scaleinftySu{\chihat} \lesssim \frac{1}{\al}\big(\scalefourSu{\chihat}+\scalefourSu{\De\chihat}\big) \lesssim \big(\mathcal{R}^{\f12}_0[\alpha]+\mathcal{R}^{\f12}_2[\alpha]\big)\mathcal{R}^{\frac{1}{2}}_1[\alpha] +\mathcal{R}_0[\alpha]+\mathcal{R}_1[\alpha]+1.  \ee We will need it in the next Proposition. \label{chihatinfinityr}
\end{remark}
\begin{proposition} Under the assumptions of Theorem \ref{mainone} and the bootstrap assumptions \eqref{boundsbootstrap}, there exist a constant $C$, quadratic in $\mathcal{R}_i[\alpha]$ for $i=0,1,2$, such that there holds

\[  \scalefourSu{\mathcal{D}\tr\chi } \lesssim \frac{a}{\modu} C(\mathcal{R}_0[\alpha],\mathcal{R}_1[\alpha],\mathcal{R}_2[\alpha])+\frac{a}{\modu}(\Ve_1[\Te_{44}]+\Ve_2[\Te_{44}])+1.\]
\end{proposition}
\begin{proof}
The term $\mathcal{D}\tr\chi$ satisfies a schematic equation of the form

\begin{align}
\snabla_4 \mathcal{D}\tr\chi =& \chihat \mathcal{D}\chihat +\mathcal{D}\slashed{T}_{44} \notag  + \frac{1}{\al}\modu (\eta,\etabar)\mathcal{D}\tr\chi +\modu\omegabar(\lvert \chihat\rvert^2+\Te_{44}) +\modu \sigma \tr\chi \notag \\+& \modu(\eta,\etabar)(\eta,\etabar)\tr\chi  + (\chihat,\tr\chi)\mathcal{D}\tr\chi +\al \etabar(\chihat,\tr\chi)\tr\chi+\al (\beta,\slashed{T}_4)\tr\chi.
\end{align}Integrating along $H_u$, we have

\begin{align}
\scalefourSu{\mathcal{D}\tr\chi}\lesssim& \frac{1}{\modu}  \intubar \scaleinfinitySuubarprime{\chihat} \scalefourSuubarprime{\mathcal{D}\chihat} \dubarprime +\intubar \scalefourSuubarprime{\mathcal{D}\slashed{T}_{44}}\dubarprime \notag \\+& \frac{1}{\al \modu^2} \intubar \scaleinfinitySuubarprime{\modu}\scaleinfinitySuubarprime{(\eta,\etabar)}\scalefourSuubarprime{\mathcal{D}\tr\chi} \dubarprime \notag \\ +&\frac{1}{\modu^3}\intubar \scaleinftySuubarprime{\modu}\scaleinftySuubarprime{\omegabar}\scaleinftySuubarprime{\chihat}\scalefourSuubarprime{\chihat}\dubarprime  \\ +& \frac{1}{\modu^2}\intubar \scaleinftySuubarprime{\modu}\scaleinftySuubarprime{\omegabar}\scalefourSuubarprime{\Te_{44}}\dubarprime   \\ +& \frac{1}{\modu^2}\intubar \scaleinfinitySuubarprime{\modu}\scaleinfinitySuubarprime{\tr\chi}\scalefourSuubarprime{\sigma}\dubarprime \notag
\\ +& \frac{1}{\modu^3}\intubar \scaleinfinitySuubarprime{\modu}\scaleinfinitySuubarprime{(\eta,\etabar)}^2\scalefourSuubarprime{\tr\chi}\dubarprime \notag \\ +& \frac{1}{\modu} \intubar \scaleinfinitySuubarprime{(\chihat,\tr\chi)}\scalefourSuubarprime{\mathcal{D}\tr\chi} \notag \dubarprime \\ \notag +& \frac{\al}{\modu^2}\intubar \scaleinfinitySuubarprime{\etabar}\scaleinfinitySuubarprime{(\chihat,\tr\chi)}\scalefourSuubarprime{\tr\chi} \dubarprime \\ +& \frac{\al}{\modu}\intubar \scaleinfinitySuubarprime{\tr\chi}\scalefourSuubarprime{(\beta,\slashed{T}_4)}\dubarprime.
\end{align}Most of the terms are immediately bounded above by 1 using the bootstrap assumptions. We will only give an explanation for the most difficult terms. For the first term, there holds
\be \frac{1}{\modu}\intubar \scaleinfinitySuubarprime{\chihat}\scalefourSuubarprime{\mathcal{D}\chihat}\dubarprime \lesssim \sup_{0\leq \ubar \leq 1}\frac{a}{\modu} \big(\frac{1}{\al}\scaleinfinitySu{\chihat}\big)\big(\frac{1}{\al}\scalefourSu{\mathcal{D}\chihat}\big). \ee Here a simple bootstrap will not suffice and we need the improvements on $\chihat$, found in Proposition \ref{propositionDchihat} and Remark \ref{chihatinfinityr}. Indeed, we have
\[ \frac{a}{\modu}\mathcal{O}_{0,\infty}[\chihat]\mathcal{O}_{1,4}[\chihat]\lesssim  \frac{a}{\modu}\big( (\mathcal{R}^{\f12}_0[\alpha]+\mathcal{R}^{\f12}_2[\alpha])\mathcal{R}^{\frac{1}{2}}_1[\alpha] +\mathcal{R}_0[\alpha]+\mathcal{R}_1[\alpha]+1\big)\big(\mathcal{R}^{\f12}_1[\alpha]\mathcal{R}^{\frac{1}{2}}_2[\alpha] +\mathcal{R}_0[\alpha]+\mathcal{R}_1[\alpha]+1\big).\]For the trilinear terms involving $\modu$, notice that $\scaleinfinitySu{\modu}=a,$ whence for example

\begin{align}
\frac{1}{\modu^2}\intubar \scaleinfinitySuubarprime{\modu}\scaleinfinitySuubarprime{\tr\chi}\scalefourSuubarprime{\sigma}\dubarprime \lesssim \frac{a \Gamma}{\modu^2}\big(\mathcal{R}_0[\sigma]^{\frac{1}{2}}\mathcal{R}_1[\sigma]^{\frac{1}{2}} + \mathcal{R}_0[\sigma]\big)\lesssim 1,
\end{align}where we have used the bootstrap assumptions on the curvature and so on. Finally, we have

\be \intubar \scalefourSuubarprime{\De\Te_{44}}\dubarprime \lesssim \frac{a}{\modu}\big(\Ve_1^{\f12}[T_{44}]\hsp \Ve_2^{\f12}[T_{44}]+\Ve_1[T_{44}]\big) \lesssim \frac{a}{\modu}(\Ve_1[\Te_{44}]+\Ve_2[\Te_{44}]). \ee 
\end{proof}
\begin{proposition}
There holds
\[ \scalefourSu{\mathcal{D}\eta} \lesssim \mathcal{R}_1^{\frac{1}{2}}[\beta]\mathcal{R}_2^{\frac{1}{2}}[\beta]+\mathcal{V}_1^{\frac{1}{2}}[\slashed{T}_4]\mathcal{V}_2^{\frac{1}{2}}[\slashed{T}_4]+\mathcal{R}_1[\beta]+\mathcal{V}_1[\slashed{T}_4]+1.  \]
\end{proposition}

\begin{proof}
We obtain the schematic equation for $\mathcal{D}\eta$:

\begin{align}
\notag \snabla_4 \mathcal{D}\eta =& \mathcal{D}(\tr\chi,\chihat)(\eta,\etabar)+(\tr\chi,\chihat)\mathcal{D}(\eta,\etabar) + \mathcal{D}\beta +\mathcal{D}\slashed{T}_4 + \frac{1}{\al} \modu(\eta,\etabar)\mathcal{D}\eta +\modu(\eta,\etabar)(\eta,\etabar)\eta+\modu \sigma\hsp  \eta \notag \\ +& \modu\omegabar\big((\tr\chi,\chihat)(\eta,\etabar)+\beta+\Te_4\big)+(\chihat,\tr\chi)\mathcal{D}\eta + \al(\chihat,\tr\chi)\etabar \eta +\al (\beta, \slashed{T}_4)\eta.
\end{align}By integrating along a given $H_u$, we have

\be \scalefourSu{\mathcal{D}\eta} \lesssim \lVert \mathcal{D}\eta \rVert_{\mathcal{L}^4_{(sc)}(S_{u,0})} + \intubar \scalefourSuubarprime{\snabla_4 \mathcal{D}\eta}\dubarprime . \ee Taking into account that $\mathcal{D}\eta$ vanishes on $S_{u,0}$ and using the bootstrap bounds, all of the terms emerging are bounded by 1 except those involving $\De\beta$ and $\De\Te_{4}$. For this term, we have

\be \intubar \scalefourSuubarprime{\De(\beta,\slashed{T}_4)} \dubarprime \lesssim \mathcal{R}_1^{\frac{1}{2}}[\beta]\mathcal{R}_2^{\frac{1}{2}}[\beta]+\mathcal{V}_1^{\frac{1}{2}}[\slashed{T}_4]\mathcal{V}_2^{\frac{1}{2}}[\slashed{T}_4]+\mathcal{R}_1[\beta]+\mathcal{V}_1[\slashed{T}_4]. \ee The result follows.
\end{proof} \noindent We now turn to quantities that satisfy equations in the ingoing direction.

\begin{proposition} \label{propositionDchibarhat}
Under the assumptions of Theorem \ref{mainone} and the bootstrap assumptions \eqref{boundsbootstrap}, there holds

\[ \frac{\al}{\modu}\scalefourSu{\mathcal{D}\chibarhat}\lesssim \underline{\mathcal{R}}_1^{\frac{1}{2}}[\alphabar]\hsp \underline{\mathcal{R}}_2^{\frac{1}{2}}[\alphabar]+\underline{\mathcal{R}}_1[\alphabar]+1.\]
\end{proposition}
\begin{proof}
There holds \[ \snabla_3 \chibarhat + \tr\chibar\hsp \chibarhat = \omegabar \hsp \chibarhat -\alphabar. \] As such, we have

\begin{align}
\snabla_3 \De \chibarhat + \big(1+\frac{\ell}{2}\big)\tr\chibar \De \chibarhat =& \omegabar\hsp  \De \chibarhat + \chibarhat \hsp \De \omegabar +\De \alphabar +\modu \omegabar \hsp(\omegabar \hsp \chibarhat +\alphabar) +\frac{1}{\al}\modu(\eta,\etabar)\De \chibarhat + \modu (\eta,\etabar)(\eta,\etabar)\chibarhat \notag \\ +& \modu \sigma \chibarhat +(\chibarhat,\tr\chi)\De \chibarhat +\al(\chihat,\tr\chi)\etabar \hsp \chibarhat + \al(\beta,\Te_4)\chibarhat := T_1+\dots+T_{11}.
\end{align}Using Proposition \ref{evolemma} and passing to scale-invariant norms, we have

\begin{align}
\frac{\al}{\modu}\scaletwoSu{\De\chibarhat}\lesssim \frac{\al}{\lvert u_{\infty}\rvert}\scalefourSuzero{\De \chibarhat}+ \sum_{i=1}^{11} \intu \frac{a^{\f32}}{\upr^3}\scalefourSuprime{T_i}\duprime.
\end{align}The most borderline term is \[ \intu \frac{a^{\f32}}{\upr^3}\scalefourSuprime{\De\alphabar}\duprime \lesssim \underline{\mathcal{R}}_1^{\frac{1}{2}}[\alphabar]\hsp \underline{\mathcal{R}}_2^{\frac{1}{2}}[\alphabar]+\underline{\mathcal{R}}_1[\alphabar].   \]The rest of the terms are bounded by 1 in a straightforward way.
\end{proof}

\begin{remark}
\label{improvechibarhat}The bound \[ \scaleinftySu{\chibarhat}\lesssim \sum_{i=0}^1 \scalefourSu{\De^i\chibarhat}\lesssim C(\underline{\mathcal{R}})   \] is now established and we shall need it in the next Proposition.
\end{remark}

\begin{proposition}
Under the assumptions of Theorem \ref{mainone} and the bootstrap assumptions \eqref{boundsbootstrap}, there holds

\[ \frac{a}{\modu^2}\scalefourSu{\mathcal{D}\tr\chibar}+\frac{a}{\modu}\scalefourSu{\mathcal{D}\tildetr} \lesssim 1.     \]
\end{proposition}

\begin{proof}
There holds

\[\snabla_3 \tr\chibar+\frac{1}{2}(\tr\chibar)^2 = -\lvert \chibarhat\rvert^2-\slashed{T}_{33}.\]We have, with $\ell$ as before, the schematic

\begin{align}\notag  \snabla_3 \mathcal{D}\tr\chibar + \frac{1+\ell}{2}\tr\chibar \hsp \tr\chibar =& \hsp \chibarhat \mathcal{D}\chibarhat +\mathcal{D}\slashed{T}_{33} +(\eta,\etabar,1)(\lvert \chibarhat\rvert^2+\Te_{33})+\frac{1}{\al}(\al\omegabar,\eta,\etabar,\chibarhat)\De\tr\chibar \notag + \sdiv\etabar \hsp \tr\chibar \\ +& \chihat \cdot \chibarhat \tr\chibar + \tr\chi \tr\chibar^2  +(\rho,\betabar,\Te_3,\Te_{34})\tr\chibar \notag  +(\eta,\etabar)(\eta,\etabar,\chibarhat,\tr\chibar)\tr\chibar +\sigma \tr\chibar \notag \\+& \frac{1}{\al}\De\tildetr \tr\chibar  +\modu\big(\tr\chibar \hsp (\tildetr,\hsp \omegabar) +\chibarhat\cdot \chibarhat\big)\tr\chibar+\modu \Te_{33}\tr\chibar .
\end{align}Passing to scale-invariant norms while using Proposition \ref{evolemma}, we obtain

\begin{align}
\frac{a}{\modu^2}\scalefourSu{\mathcal{D}\tr\chibar}\lesssim \frac{a}{\lvert u_{\infty}\rvert^2} \lVert \mathcal{D}\tr\chibar \rVert_{\mathcal{L}^4_{(sc)}(S_{u_{\infty},\ubar})} + \intu \frac{a^2}{\upr^4}\sum_{i=1}^{16} \scalefourSuprime{T_i}\duprime.
\end{align}Again, we discuss the most difficult terms. The terms
\begin{align} \intu \frac{a^2}{\upr^4}\big(\scaletwoSuprime{\chihat\cdot \chibarhat \hsp \tr\chibar} +\scaletwoSuprime{\tr\chibar^2 \hsp \tr\chi}\big)\duprime   \lesssim \intu \frac{a^2}{\upr^4}\big(\frac{\upr^3}{a}+\frac{\upr^4}{a^2}\big) \frac{O^3}{\upr^2}\duprime \lesssim 1.  \end{align}

\noindent The terms involving $(\sigma,  \betabar,\slashed{T}_3)\tr\chibar$ are also bounded by 1 using the bootstrap assumptions, by bounding $\sigma, \al \betabar, \al \slashed{T}_{3}$ in $\mathcal{L}^4_{(sc)}$ and eventually by their $\scaletwoHbaru{\cdot}-$norms. The first term involving $\chibarhat\mathcal{D}\chibarhat$ is handled by a simple $\mathcal{L}^{\infty}_{(sc)}-\mathcal{L}^4_{(sc)}$ H\"older inequality using the bootstrap assumptions, whereas the term
involving $\mathcal{D}\slashed{T}_{33}$ is bounded by 1 using the bootstrap assumptions on \[\scaletwoHbaru{\mathcal{D}\slashed{T}_{33}}, \scaletwoHbaru{\mathcal{D}^2\slashed{T}_{33}}\] and the codimension 1 trace inequality.
\vspace{3mm}

\noindent For $\tildetr,$ we similarly have the equation

\begin{align}
\notag &\snabla_3 \mathcal{D}\tildetr +(1+\frac{\ell}{2})\tr\chibar \mathcal{D}\tildetr \\ =&  \tildetr \hsp \De \tildetr + \chibarhat \De \chibarhat +\De \Te_{33} + (\eta,\etabar,1)(\tildetr \tildetr +\lvert \chibarhat 
\rvert^2 +\Te_{33}) +\frac{1}{\al}(\al\omegabar,\eta,\etabar,\chibarhat)\De\tildetr \notag  +\sdiv \etabar \hsp  \tildetr \notag \\ +& (\chi\cdot \chibarhat + \tr\chi \hsp \tr\chibar)\tildetr +(\rho,\betabar,\Te_3,\Te_{34})\tildetr +(\eta,\etabar)(\eta,\etabar,\chibarhat,\tr\chibar)\tildetr +\sigma \tildetr +\frac{1}{\al} \tildetr \De \tildetr \notag \\+&\modu\big(\tr\chibar(\tildetr,\omegabar)+\chibarhat\cdot \chibarhat)\big)\tildetr +\modu \Te_{33}\tildetr:=T_1+\dots T_{13}.
\end{align}
Passing to scale-invariant norms, we have 

\begin{align}
\frac{a}{\modu}\scalefourSu{\mathcal{D}\tildetr} \lesssim \frac{a}{\lvert u_{\infty}\rvert}\scalefourSuzero{\mathcal{D}\tildetr} +\intu \frac{a^2}{\upr^3}\sum_{i=1}^{13}\scalefourSuprime{T_i} \duprime.
\end{align}We again examine the top-order terms, as the zero-order ones are bounded by 1 in a straightforward way using the bootstrap assumptions. We have

\be (1,\frac{1}{\al})\intubar \frac{a^2}{\upr^3} \scalefourSuubarprime{\tildetr \hsp \De\tildetr} \dubarprime \lesssim \intu \frac{a^2}{\upr^3}\cdot \frac{\upr^2}{a^2}\cdot \frac{O^2}{\upr}\dubarprime \lesssim \frac{O^2}{\modu}\lesssim 1, \ee
\be\intubar \frac{a^{\f32}}{\upr^3} \scalefourSuubarprime{(\al\omegabar,\eta,\etabar,\chibarhat)\De\tildetr} \dubarprime \lesssim \intubar \frac{a^{\f32}}{\upr^3}\cdot \frac{\upr}{\al} \cdot \frac{\upr}{a}\cdot \frac{ \hsp O^2}{\upr} \dubarprime \lesssim \frac{O^2}{\modu}\lesssim 1, \ee
\be \intu \frac{a^2}{\upr^3}\scalefourSuubarprime{\chibarhat \De \chibarhat}\dubarprime \lesssim \frac{a}{\modu}\mathcal{O}_{0,\infty}[\chibarhat]\mathcal{O}_{1,4}[\chibarhat] \lesssim \underline{\mathcal{R}}_1^{\f12}[\alphabar]\underline{\mathcal{R}}_2^{\f12}[\alphabar]+\underline{\mathcal{R}}[\alphabar]+1, \ee where we have crucially used the improvements from Proposition \ref{propositionDchibarhat} and Remark \ref{improvechibarhat}.

\end{proof}\noindent We move on with an  $\prescript{(S)}{}{\mathcal{O}}_{1,4}$ estimate for $\etabar$.

\begin{proposition} \label{etabar14bound}
Under the assumptions of Theorem \ref{mainone} and the bootstrap assumptions \eqref{boundsbootstrap}, there holds 

\begin{align}
\scalefourSu{\mathcal{D}\etabar} \lesssim& \frac{\al}{\modu^{\frac{1}{2}}}\big(\scaletwoHbaru{\mathcal{D}\beta}^{\frac{1}{4}}\scaletwoHbaru{\mathcal{D}^2\beta}^{\frac{1}{4}} + \scaletwoHbaru{\mathcal{D}\slashed{T}_3}^{\frac{1}{4}}\scaletwoHbaru{\mathcal{D}^2\slashed{T}_3}^{\frac{1}{4}} \notag \\+& \scaletwoHbaru{\mathcal{D}\beta}+\scaletwoHbaru{\mathcal{D}\slashed{T}_3}\big)+1.  \end{align}
\end{proposition}
\begin{proof}
For $\etabar$ we recall that the following equation holds:

\[ \snabla_3\etabar+\frac{1}{2}\tr\chibar\etabar =\frac{1}{2}\tr\chibar \eta-\chibarhat\cdot(\etabar-\eta)+\betabar-\frac{1}{2}\slashed{T}_3.\]Consequently, 

\begin{align}\notag
\snabla_3 \mathcal{D}\etabar +\frac{1+\ell}{2}\tr\chibar \hsp\mathcal{D}\etabar =& \hsp \tr\chibar \hsp \mathcal{D}\eta +\eta \mathcal{D}\tr\chibar +\chibarhat (\mathcal{D}\eta, \mathcal{D}\etabar)+(\eta,\etabar)\mathcal{D}\chibarhat+\mathcal{D}\betabar+\mathcal{D}\slashed{T}_{3} \\ \notag +&\big(\frac{1}{\Omega^2},\omega,\al\eta\big)(\tr\chibar \eta +\chibarhat\cdot(\etabar,\eta)+\betabar+\slashed{T}_{3}) +\tr\chibar\big(\frac{1}{\Omega^2}-1\big) \etabar \notag \\+& \modu\tr\chibar \tildetr \etabar + \modu(\lvert \chibarhat \rvert^2+\slashed{T}_{33})\etabar +\etabar \mathcal{D}\tildetr +\frac{1}{\al}(\eta,\etabar)\mathcal{D}\etabar +(\eta,\etabar)^2\etabar \notag \\ +& \sigma \etabar +\omega \tr\chibar\hsp  \etabar   + \chibarhat \hsp \mathcal{D}\etabar + \al\etabar\hsp \chibarhat\hsp \etabar +\al(\betabar,\slashed{T}_3)\etabar := T_1+\dots+T_{21} .
\end{align}Passing to scale-invariant norms, we have 
\begin{align}
\frac{1}{\modu}\scalefourSu{\mathcal{D}\etabar}\lesssim \frac{1}{\uinf}\scalefourSuzero{\mathcal{D}\etabar}+ \intu \frac{a}{\upr^3}\sum_{j=1}^{21}\scalefourSuprime{T_j}\duprime.
\end{align}For the first term, we have

\begin{equation*} \intu \frac{a}{\upr^3}\scalefourSuprime{\tr\chibar \mathcal{D}\eta} \duprime \lesssim \frac{1}{\modu}\big(\frac{a}{\modu^2}\sup_{u_{\infty}\leq u^{\prime}\leq u}\scaleinftySuprime{\tr\chibar}\scalefourSuprime {\mathcal{D}\eta}\big). \end{equation*}However, at this point, the bound $\scaleinftySuprime{\tr\chibar}\lesssim\scalefourSuprime{\tr\chibar}+\scalefourSuprime{\mathcal{D}\tr\chibar}$ implies control on $\tr\chibar$, using the improvements. Similarly, the improved control on $\scalefourSu{\mathcal{D}\eta}$ gives a bound dependent on $\mathcal{R}, \Vl$. For the terms involving $\mathcal{D}\betabar, \mathcal{D}\slashed{T}_3$, we have

\begin{align}
&\intu \frac{a}{\upr^3}\scalefourSuprime{\mathcal{D}\betabar,\mathcal{D}\slashed{T}_{3}} \duprime \notag \\\lesssim& \frac{\al}{\modu^{\frac{3}{2}}}\big( \scaletwoHbaru{\mathcal{D}(\beta,\slashed{T}_3)}^{\frac{1}{4}}\scaletwoHbaru{\mathcal{D}(\beta,\slashed{T}_3)}^{\frac{1}{4}} +\scaletwoHbaru{\mathcal{D}(\beta,\slashed{T}_3)}\big).
\end{align}The rest of the terms are bounded above by $\frac{1}{\modu}$ using the bootstrap assumptions \eqref{boundsbootstrap}.

\end{proof}We continue with the estimate for the non-tensorial quantity $\mathcal{G}$.
\begin{proposition}
Under the assumptions of Theorem \ref{mainone} and the bootstrap assumptions \eqref{boundsbootstrap},there holds 

\be \scalefourSu{\De\mathcal{G}}\lesssim \prescript{(S)}{}{\mathcal{O}}_{3,2}[\eta]+ \mathcal{R}_2[\rho,\sigma] +\mathcal{V}_2[\Te]+1. \ee

\end{proposition}
\begin{proof}
Recall the equation \begin{align}
\snabla_4 \mathcal{G} \sim \nablasl \eta + \mathcal{G} \cdot \chi + \omegabar \cdot \chi + \Gammaslash \cdot (\eta, \etabar) + \eta \cdot \etabar + (\rho,\sigma)+\Tslash.    \notag
\end{align}Differentiating, we get 
\begin{align}
    \snabla_4 \De\mathcal{G} =& \De \snabla \eta +(\mathcal{G},\omegabar)\hsp (\chihat,\tr\chi) + (\chihat,\tr\chi)\De(\mathcal{G},\omegabar) +\De(\eta,\Gammaslash) \hsp (\eta,\etabar) +\De(\eta,\etabar) \hsp (\Gammaslash,\eta) +\De(\rho,\sigma)+\De\Te \notag \\ +& \frac{1}{\al}\modu \hsp (\eta,\etabar) \hsp \De \mathcal{G} + \modu \hsp  (\eta,\etabar)^2  \hsp \mathcal{G} + \modu \hsp  \sigma  \hsp \mathcal{G} + (\chihat,\tr\chi) \hsp \De\mathcal{G} + \al(\chihat,\tr\chi) \hsp \etabar  \hsp \mathcal{G} +\al \hsp (\beta,\Te_4)\hsp \mathcal{G} \notag \\ +& \modu\omegabar\big(\nablasl \eta + \mathcal{G} \cdot \chi + \omegabar \cdot \chi + \Gammaslash \cdot (\eta, \etabar) + \eta \cdot \etabar + (\rho,\sigma)+\Tslash\big).
\end{align}
As before, integrating along the $e_4-$direction, we see that the only top-derivative term is $\mathcal{D}\snabla \eta$, which can be bounded by the $\prescript{(S)}{}{\mathcal{O}}_{2,2}$estimate for $\eta$ as well as the elliptic estimate for $\eta$ (here we use $L^4-L^2$ Sobolev embedding):

\begin{align} 
\scalefourSu{\De\mathcal{G}}\lesssim \scalefourSuzero{\De\mathcal{G}}+ \intubar \scalefourSuubarprime{\De\snabla \eta} \dubarprime + l.o.t. \lesssim \prescript{(S)}{}{\mathcal{O}}_{3,2}[\eta] + l.o.t.
\end{align}In particular, the terms $\scalefourSu{\mathcal{D}(\rho,\sigma,\Te)}$ are bounded as before in the current section. The claim follows.
\end{proof}
\begin{proposition}
Under the assumptions of Theorem \ref{mainone} and the bootstrap assumptions \eqref{boundsbootstrap}, there holds \[  \scalefourSu{\De(\Gammaslash-\Gammaslash^{\circ})}\lesssim \prescript{(S)}{}{\mathcal{O}}_{3,2}[\chihat]+1. \]
\end{proposition}
\begin{proof}
Recall the schematic equation
\[ \snabla_4 (\Gammaslash-\Gammaslash^{\circ}) = \snabla(\chihat,\tr\chi) +\mathcal{G}\cdot(\Gammaslash-\Gammaslash^{\circ}).  \]As such, we have

\begin{align}
\snabla_4 \De (\Gammaslash-\Gammaslash^{\circ}) =& \De \snabla(\chihat,\tr\chi) + \De \mathcal{G} \hsp (\Gammaslash-\Gammaslash^{\circ}) +\mathcal{G} \hsp  \De(\Gammaslash-\Gammaslash^{\circ}) + \frac{\modu}{\al}(\eta,\etabar)\De (\Gammaslash-\Gammaslash^{\circ})+ \modu \hsp (\eta,\etabar)^2 \hsp (\Gammaslash-\Gammaslash^{\circ}) \notag \\ +& \modu \hsp \sigma \hsp (\Gammaslash-\Gammaslash^{\circ}) + (\chihat,\tr\chi)\hsp \De (\Gammaslash-\Gammaslash^{\circ}) +\al(\chihat,\tr\chi)\hsp \hsp \etabar \hsp (\Gammaslash-\Gammaslash^{\circ}) +\al\hsp (\beta,\Te_4)\hsp (\Gammaslash-\Gammaslash^{\circ})\notag \\ +& \modu \hsp \omegabar \hsp  \big(\snabla(\chihat,\tr\chi) +\mathcal{G}\cdot(\Gammaslash-\Gammaslash^{\circ})\big).
\end{align}Integrating along the $e_4$-direction and using the bootstrap assumptions, we see that the top-order term is \[\intubar \scalefourSuubarprime{\De\snabla(\chihat,\tr\chi)}\dubarprime, \]which is bounded as before by $\mathcal{O}_{3,2}[\chihat] +1$. The rest of the terms are bounded by $1$ as before. The result follows.
\end{proof}

\section{$\prescript{(S)}{}{\mathcal{O}_{0,\infty}}$ estimates for the Ricci coefficients}\label{section0infty}

In this section, we improve the bootstrap assumptions on the $\mathcal{L}^{\infty}_{(sc)}(S)-$norm of the Ricci coefficients. 

\begin{proposition}
There exists a constant $C$ depending on the total norms $\mathcal{R},\Vl$ such that \[ \prescript{(S)}{}{\mathcal{O}}_{0,\infty}\lesssim C.   \]

\end{proposition}
\begin{proof}
The scale-invariant inequality 
\be  \label{scaleinvariantfourinfinity}\scaleinftySu{\phi}\lesssim \scalefourSu{\phi}^{\frac{1}{2}}\scalefourSu{\al\snabla\phi}^{\frac{1}{2}}+\scalefourSu{\phi},    \ee holds for an arbitrary horizontal tensor-field. At first instance, Propositions 
\ref{propchihat},\ref{propositionDchihat}, 
\ref{propchibarhat} and \ref{propositionDchibarhat} imply control for $\chihat$ and $\chibarhat$ in the supremum norm, given in Remarks \ref{improvechihat}-\ref{improvechibarhat}. In turn, these estimates on the $\mathcal{L}^{\infty}_{(sc)}$-norm of $\chihat,\hsp \chibarhat$ improve Propositions \ref{propositiontrchi} and
\ref{propositiontildetr} for $\tr\chi$ and $\tildetr$ respectively.
For $\etabar$, Propositions \ref{etabarint04} and \ref{etabar14bound} imply the improvement of the bootstrap assumption on $\prescript{(S)}{}{\mathcal{O}}_{0,\infty}[\etabar]$, through equation \eqref{scaleinvariantfourinfinity}.  Consequently, we get \[\prescript{(S)}{}{\mathcal{O}}_{0,\infty}[\omegabar, \chihat,\tr\chi, \eta, \chibarhat,\tr\chibar,\tildetr, \etabar] \lesssim C(\mathcal{R},\underline{\mathcal{R}},\Ve,\underline{\Ve}),    \]for some constant $C$ depending only on the curvature and Vlasov norms. This concludes the proof of the proposition.
\end{proof}
\section{$\prescript{(S)}{}{\mathcal{O}}_{2,2}$ estimates for the Ricci coefficients}
In this section we improve the bootstrap bounds on two derivatives of Ricci coefficients in $L^2$. For efficiency of presentation, from now on and for the rest of this paper, we opt to  carry computations out in a schematic way, rather than term by term as in \ref{sectionO04}-\ref{section0infty}.  We begin this section with a series of preliminary estimates that will be important in handling the equations schematically. 
\subsection{Preliminary estimates}
\begin{proposition}\label{preliminary}
    Let $\psi$ denote an arbitrary element of the set $\big\{ \frac{1}{a}\modu, \frac{1}{\al}\chihat,\tr\chi,\omegabar,\eta,\etabar,\frac{a}{\modu}\tildetr,\frac{\al}{\modu}\chibarhat, \frac{a}{\modu^2}\tr\chi \big\}$. Let $\psi_g \in \{\modu\omegabar,\eta,\etabar,1\}$ and let $\Psi_{\ubar} \in \big\{\frac{1}{\al}\alpha,\beta,\rho,\sigma,\betabar\big\}$, $\Psi_u \in \big\{\frac{1}{\al}\beta,\rho,\sigma,\betabar,\alphabar\big\}$. Finally, denote an arbitrary element of the set $\big \{\frac{\modu}{a}\Te_{44},\frac{\modu}{a}\Te_4, \Te_{34},\Te_{33},\Te_{3},\Te \big\} $by $\Vl$. Under the assumptions of Theorem \ref{mainone} and the bootstrap assumptions \eqref{boundsbootstrap}, the following hold:

    \begin{equation}
    \sum_{i_1+i_2\leq 1} \scaletwoSu{\dione\psi_g^{i_2}}\lesssim \modu,
    \label{preliminb2}\end{equation}
    \begin{equation}
        \sum_{i_1+i_2+i_3\leq 1}\scaletwoSu{\dione \psi_g^{i_2}\dit (\psi,\Vl)}\lesssim O+V,
    \end{equation}
    \begin{equation}
    \sum_{i_1+i_2+i_3+i_4\leq 1}\scaletwoSu{\dione\psi_g^{i_2}\dit(\psi,\Vl) \dif \psi}\lesssim \frac{O(O+V)}{\modu} \end{equation}
 
    \begin{equation}
     \intubar \sum_{i_1+i_2+i_3+i_4+i_5\leq 1}\scaletwoSuubarprime{\dione \psi_g^{i_2}\dit \psi \dif \psi \difi \Psi_{\ubar}}\dubarprime \lesssim \frac{O^2 \hsp R}{\modu^2},
    \end{equation}
       \begin{equation}
     \intu \sum_{i_1+i_2+i_3+i_4+i_5\leq 1}\frac{a}{\upr^2}\scaletwoSuprime{\dione \psi_g^{i_2}\dit \psi \dif \psi \difi \Psi_{u}}\duprime \lesssim \frac{O^2 \hsp R}{\modu^2},
    \end{equation}
    \begin{equation}\label{prelimfive}
      \intubar \sum_{i_1+i_2+i_3\leq 1}\scaletwoSuubarprime{(\al)^{i_2}\dione \psi_g^{i_2}\dit \Psi_{\ubar}}\dubarprime\lesssim R, 
    \end{equation}
     \begin{equation}\label{prelimsix}
      \intu \frac{a}{\upr^2} \sum_{i_1+i_2+i_3\leq 1}\scaletwoSuprime{(\al)^{i_2}\dione \psi_g^{i_2}\dit \Psi_u} \duprime \lesssim  \frac{\al}{\modu^{\f12}} R, 
    \end{equation}
    \begin{equation}\label{prelimseven}
      \intubar  \sum_{i_1+i_2+i_3+i_4\leq 1}\scaletwoSuubarprime{(\al)^{i_2}\dione \psi_g^{i_2}\dit \psi\dif \Psi_{\ubar}} \dubarprime \lesssim \frac{R \hsp O}{\modu}, 
    \end{equation}
      \begin{equation}\label{prelimeight}
      \intu \frac{a}{\upr^2} \sum_{i_1+i_2+i_3+i_4\leq 1}\scaletwoSuprime{(\al)^{i_2}\dione \psi_g^{i_2}\dit \psi \dif \Psi_u} \duprime \lesssim \frac{R \hsp O}{\modu}, 
    \end{equation}

      \begin{equation}\label{prelimnine}
      \intubar  \sum_{i_1+i_2+i_3+i_4+i_5\leq 1}\scaletwoSuubarprime{(\al)^{i_2}\dione \psi_g^{i_2}\dit \psi\dif \psi \difi \Psi_u} \dubarprime \lesssim \frac{R \hsp O^2}{\modu^2}, 
    \end{equation}
      \begin{equation}\label{prelimten}
      \intu \frac{a}{\upr^2} \sum_{i_1+i_2+i_3+i_4+i_5\leq 1}\scaletwoSuprime{(\al)^{i_2}\dione \psi_g^{i_2}\dit \psi \dif \psi \difi \Psi_u} \duprime \lesssim \frac{R \hsp O^2}{\modu^2}, 
    \end{equation}
    \begin{equation}\label{prelimfour}\sum_{i_1+i_2+i_3+i_4+i_5\leq 1}\scaletwoSu{(\al)^{i_2}\dione \psi_g^{i_2}\dit \psi \dif\psi \difi(\psi,\Vl) }\lesssim \frac{O^2(O+V)}{\modu^2},
    \end{equation}
     \begin{equation} \label{prelimone}
       \sum_{i_1+i_2= 2}\scaletwoSu{(\al)^{i_2}\dione \psi_g^{i_2}\dit \psi \dif \psi} \lesssim \frac{O^2}{\modu}, 
    \end{equation}
    \begin{equation}
\label{preliminb} \sum_{i_1+i_2+i_3=2}\scaletwoSu{(\al)^{i_2} \dione \psi_g^{i_2}\dit \Vl}\lesssim V.
    \end{equation}\begin{equation}\label{prelimtwo}\intubar \sum_{i_1+i_2+i_3= 2}\scaletwoSuubarprime{(\al)^{i_2}\dione \psi_g^{i_2}\dit \Psi_{\ubar}} \lesssim R
    \end{equation}
    \begin{equation}\label{prelimthree}\intu \frac{a}{\upr^2} \sum_{i_1+i_2+i_3= 2}\scaletwoSuprime {(\al)^{i_2}\dione \psi_g^{i_2}\dit \Psi_u} \duprime\lesssim \frac{\al}{\modu^{\f12}}R.
    \end{equation}

\end{proposition}

\begin{proof}
Let us note that \eqref{prelimone}-\eqref{prelimthree} correspond to bounds for the highest-order terms encountered in the structure equations for $\snabla_4\De^2\psi, \snabla_3\De^2 \psi$, as indicated in the commutation formulae \eqref{commutationformulanabla4}, \eqref{commutationformulanabla3}, since for example \[ \snabla_4\psi= \psi \psi +\Psi_{\ubar} +\Vl \hspace{3mm}\text{or}\hspace{3mm} \snabla_3\psi+k \tr\chibar \psi =  \psi \psi +\Psi_u +\Vl,  \]whereas the rest of the inequalities correspond to the lower-order terms in the structure equations. For any given inequality above, if there are three terms, they are bounded either in $\mathcal{L}^{\infty}_{(sc)}-\mathcal{L}^4_{(sc)}-\mathcal{L}^4_{(sc)}$ or  $\mathcal{L}^{\infty}_{(sc)}-\mathcal{L}^{\infty}_{(sc)}-\mathcal{L}^2_{(sc)}$. If there are two terms, they are bounded either in  $\mathcal{L}^4_{(sc)}-\mathcal{L}^4_{(sc)}$ or in $\mathcal{L}^{\infty}_{(sc)}-\mathcal{L}^2_{(sc)}$. 

\begin{enumerate}
\item For \eqref{prelimone}, the case $i_2=2$ is straightforward. For $i_2=1,$ 
the operator $\mathcal{D}$ falls on exactly one of the three terms $\psi$, $\psi_g$. We bound that term in $\mathcal{L}^4_{(sc)}$ together with one of the remaining terms and the last term in $\mathcal{L}^{\infty}_{(sc)}.$ This gives \[ \al \hsp (\scaletwoSu{ \psi\psi \De \psi_g}+\scaletwoSu{\psi_g \psi \De\psi}) \lesssim \frac{O^2}{\modu}.    \]For $i_2=0$, we have\begin{align*} &\scaletwoSu{\De\psi \De\psi}+\scaletwoSu{\psi \De^2\psi} \notag \\\lesssim& \frac{1}{\modu}\big(\scalefourSu{\De\psi}\scalefourSu{\De\psi}+\scaleinftySu{\psi}\scaletwoSu{\De^2\psi}\big)\lesssim \frac{O^2}{\modu}.   \end{align*}
\item For \eqref{preliminb}, we can condition with respect to $i_2$ similarly. For $i=2$, there holds \[ a \scaletwoSu{\psi\psi \Vl}\lesssim \frac{a O^2 \hsp V}{\modu^2},   \]where we bound $\Vl$ in $\mathcal{L}^4_{(sc)}$. For $i_2=1$ we have $\al\big(\scaletwoSu{\De \psi \hsp \Vl}+\scaletwoSu{\psi_g \De \Vl}\big) \lesssim 1$, where we bound both terms in $\mathcal{L}^4_{(sc)}$. Finally, for $i_2=0$, there holds by definition \[ \scaletwoSu{\De^2\Vl}\lesssim V.     \]

\item For \eqref{prelimtwo}-\eqref{prelimthree} we follow the same strategy, keeping in mind that a given $\Psi_u$ or $\Psi_{\ubar}$ can be estimated in $\mathcal{L}^4_{(sc)}$ at zeroth and first orders, whereas for second order we control directly the energy norms.
\end{enumerate}
\end{proof}

\subsection{$\prescript{(S)}{}{\mathcal{O}}_{2,2}$ estimates.}

\label{sectionO22}
\begin{proposition}
Given the assumptions of Theorem \ref{mainone} and the bootstrap assumptions \eqref{boundsbootstrap}, there holds

\[ \scaletwoSu{\De^2 \omegabar}\lesssim \mathcal{R}_{2}[\rho]+\Ve_2[\Te_{34}]+1.  \]
\end{proposition}
\begin{proof}
Recall that \[ \snabla_4 \omegabar= 2(\eta,\etabar)\eta +\rho+\Te_{34}.\]Using Proposition \ref{commutationformulanabla4}, we have

\begin{align}
\snabla_4 \De^2 \omegabar =& \sum_{i_1+i_2+i_3+i_4=2}\dione \psi_g^{i_2}\dit \eta \dif (\eta,\etabar) +\sum_{i_1+i_2+i_3=2}\dione \psi_g^{i_2}\dit  (\rho,\Te_{34}) \notag \\ +& \frac{1}{\al}\sum_{i_1+i_2+i_3+i_4+i_5=1} \db\modu \dif(\eta,\etabar)\mathcal{D}^{i_5+1}\omegabar \notag \\ +&\sum_{i_1+i_2+i_3+i_4=1} \db(\tr\chi,\chihat)\mathcal{D}^{i_4+1}\omegabar \notag \\ +&\al \sum_{i_1+i_2+i_3+i_4+i_5=1}\db(\tr\chi,\chihat)\dif \etabar \difi \omegabar \notag \\ +& \al \sum_{i_1+i_2+i_3+i_4=1} \db (\beta,\Te_4)\dif \omegabar \notag \\ +&\sum_{i_1+i_2+i_3+i_4+i_5=1} \db \modu \dif \sigma \difi \omegabar \notag \\ +&\sum_{i_1+i_2+i_3+i_4+i_5+i_6=1} \db \modu \dif(\eta,\etabar)\difi(\eta,\etabar)\mathcal{D}^{i_6}\omegabar:= T_1+\dots +T_8.
\end{align}Thus, since $\De^2 \omegabar$ vanishes on $\Hbar_0$, we have
\be \scaletwoSu{\De^2 \omegabar}\lesssim \sum_{i=1}^8 \intubar \scaletwoSuubarprime{T_i}\dubarprime . \ee
The most dangerous term is

\begin{align} \sum_{i_1+i_2+i_3=2}\intubar \scaletwoSuubarprime{\db(\rho,\Te_{34})}\lesssim \mathcal{R}_2[\rho]+\Ve_2[\Te_{34}]+1.
 \end{align}
\noindent The rest of the terms are bounded by $1$, using the preliminary estimates from Proposition \ref{preliminary} and using the bootstrap assumptions \eqref{boundsbootstrap}. The result follows.

\end{proof}

\begin{proposition}
Given the assumptions of Theorem \ref{mainone} and the bootstrap assumptions \eqref{boundsbootstrap}, there holds

\[  \frac{1}{\al}\scaletwoSu{\De^2\chihat}\lesssim \mathcal{R}_2[\alpha]+1.   \]

\label{propositionD2chihat}

\end{proposition}
\begin{proof}
Using the commutation formula \eqref{commutationformulanabla4}, we obtain

\begin{align}
\snabla_4\De^2 \chihat =&\sum_{i_1+i_2+i_3+i_4=2}\dione (\modu\omegabar)^{i_2}\dit \tr\chi \dif \chihat \notag  +\sum_{i_1+i_2+i_3=2}\dione (\modu\omegabar)^{i_2}\dit \alpha\notag \\ +&\frac{1}{\al}\sum_{i_1+i_2+i_3+i_4+i_5= 1}\dione \psi_g^{i_2}\dit\modu \dif(\eta,\etabar)\mathcal{D}^{i_5+1}\chihat +\sum_{i_1+i_2+i_3+i_4=1}\dione \psi_g^{i_2}\dit(\tr\chi,\chihat)\mathcal{D}^{i_4+1}\chihat \notag \\ +& \al \sum_{i_1+i_2+i_3+i_4+i_5= 1} \dione \psi_g^{i_2}\dit(\chihat,\tr\chi)\dif\etabar \difi\chihat +\al \sum_{i_1+i_2+i_3+i_4+i_5= 1}\dione \psi_g^{i_2}\dit(\beta,\Te_4)\dif \chihat \notag \\ +& \sum_{i_1+i_2+i_3+i_4+i_5=1} \dione \psi_g^{i_2}\dit\modu \dif \sigma \difi\chihat  \notag \\ +&\sum_{i_1+i_2+i_3+i_4+i_5+i_6=1}\dione \psi_g^{i_2}\dit\modu \dif (\eta,\etabar)\difi (\eta,\etabar)\mathcal{D}^{i_6}\chihat := T_1+\dots+T_8.
\end{align}As such, we have

\begin{align}
\frac{1}{\al}\scaletwoSu{\De^2\chihat}\lesssim \frac{1}{\al} \sum_{i=1}^8\intubar \scaletwoSuubarprime{T_i }\dubarprime .
\end{align}Using the preliminary estimates from Proposition \ref{preliminary}, we obtain the result. We only further explain the most borderline terms, namely

\be \frac{1}{\al}\intubar \sum_{i_1+i_2+i_3=2}\scaletwoSu{\dione(\modu \omegabar)^{i_2}\dit \alpha} \lesssim \mathcal{R}_2[\alpha]+1, \ee
\begin{align} 
&\frac{1}{\al}\sum_{i_1+i_2+i_3+i_4+i_5= 1}\intubar\scaletwoSuubarprime{\dione \psi_g^{i_2}\dit\modu \dif(\eta,\etabar)\mathcal{D}^{i_5+1}\chihat}\dubarprime \notag \\ \lesssim& \frac{1}{\al}\big(\intubar \scaletwoSuubarprime{\modu(\eta,\etabar)\De^2\chihat} \dubarprime +\sum_{i_1+i_2+i_3+i_4=1}\intubar \scaletwoSuubarprime{\dione \psi_g^{i_2}\dit \modu \dif(\eta,\etabar)\difi \chihat}\dubarprime \big) \notag \\ \lesssim& \frac{a \hsp O^2}{\modu^2} +1\lesssim 1,
\end{align}
\begin{align} 
&\frac{1}{\al}\sum_{i_1+i_2+i_3+i_4= 1}\intubar\scaletwoSuubarprime{\dione \psi_g^{i_2}\dit(\tr\chi,\chihat)\mathcal{D}^{i_4+1}\chihat}\dubarprime \notag \\ \lesssim& \frac{1}{\al}\big(\intubar \scaletwoSuubarprime{(\tr\chi,\chihat)\De^2\chihat} \dubarprime +\sum_{i_1+i_2+i_3=1}\intubar \scaletwoSuubarprime{\dione \psi_g^{i_2}\dit (\tr\chi,\chihat) \De \chihat}\dubarprime \big) \notag \\ \lesssim& \frac{\al \hsp O^2}{\modu} +1\lesssim 1.
\end{align}

\end{proof}
\begin{proposition}
Under the assumptions of Theorem \ref{mainone} and the bootstrap assumptions \eqref{boundsbootstrap}, there holds 

\[ \scaletwoSu{\De^2 \eta}\lesssim \mathcal{R}_2[\beta]+1.   \]
\end{proposition}
\begin{proof}
Recall the schematic equation 

\[ \snabla_4 \eta = (\tr\chi,\chihat)(\eta,\etabar)+\beta+\Te_{4}.    \]Using Proposition \ref{commutationformulanabla4}, the equation

\begin{align}
\snabla_4 \De^2 \eta =& \sum_{i_1+i_2+i_3+i_4=2}\db(\tr\chi,\chihat)\dif(\eta,\etabar) +\sum_{i_1+i_2+i_3=2}\db (\beta,\Te_{4})\\ +&\frac{1}{\al}\sum_{i_1+i_2+i_3+i_4+i_5= 1}\dione \psi_g^{i_2}\dit\modu \dif(\eta,\etabar)\mathcal{D}^{i_5+1}\eta +\sum_{i_1+i_2+i_3+i_4=1}\dione \psi_g^{i_2}\dit(\tr\chi,\chihat)\mathcal{D}^{i_4+1}\eta \notag \\ +& \al \sum_{i_1+i_2+i_3+i_4+i_5= 1} \dione \psi_g^{i_2}\dit(\chihat,\tr\chi)\dif\etabar \difi\eta +\al \sum_{i_1+i_2+i_3+i_4+i_5= 1}\dione \psi_g^{i_2}\dit(\beta,\Te_4)\dif \eta \notag \\ +& \sum_{i_1+i_2+i_3+i_4+i_5=1} \dione \psi_g^{i_2}\dit\modu \dif \sigma \difi\eta  \notag \\ +&\sum_{i_1+i_2+i_3+i_4+i_5+i_6=1}\dione \psi_g^{i_2}\dit\modu \dif (\eta,\etabar)\difi (\eta,\etabar)\mathcal{D}^{i_6}\eta := T_1+\dots+T_8.
\end{align}Again, we have

\be \scaletwoSu{\De^2 \eta} \lesssim \sum_{i=1}^8 \intubar \scaletwoSuubarprime{T_i}\dubarprime. \ee The most borderline term is 

\begin{align} \sum_{i_1+i_2+i_3=2}\intubar \scaletwoSuubarprime{\db (\beta,\Te_4)}\dubarprime \lesssim& \mathcal{R}_2[\beta]+\intubar \scaletwoSuubarprime{\De^2 \Te_{4}}\dubarprime+1 \notag \\ \lesssim& \mathcal{R}_2[\beta] +\frac{\al}{\modu}\Ve_2[\Te_{34}]+1 \lesssim \mathcal{R}_2[\beta] +\frac{\al}{\modu}V+1 \lesssim \mathcal{R}_2[\beta] +1.  \end{align}The rest are bounded by 1 using the preliminary estimates from Proposition \ref{preliminary} and the bootstrap assumptions. 
\end{proof}

\begin{proposition}
Under the assumptions of Theorem \ref{mainone} and the bootstrap assumptions \eqref{boundsbootstrap}, there exists a constant $C$ depending on the curvature and Vlasov norms, such that there holds 

\[ \scaletwoSu{\De^2 \tr\chi}\lesssim C(\mathcal{R},\underline{\mathcal{R}},\Ve,\underline{\Ve}) .   \]

\end{proposition}

\begin{proof}
There holds 

\begin{align}
\snabla_4 \De^2 \tr\chi =& \sum_{i_1+i_2+i_3+i_4=2}\db (\tr\chi,\chihat)\dif(\tr\chi,\chihat) +\sum_{i_1+i_2+i_3=2}\db \Te_{44} \notag \\ +&\frac{1}{\al}\sum_{i_1+i_2+i_3+i_4+i_5= 1}\dione \psi_g^{i_2}\dit\modu \dif(\eta,\etabar)\mathcal{D}^{i_5+1}\eta +\sum_{i_1+i_2+i_3+i_4=1}\dione \psi_g^{i_2}\dit(\tr\chi,\chihat)\mathcal{D}^{i_4+1}\eta \notag \\ +& \al \sum_{i_1+i_2+i_3+i_4+i_5= 1} \dione \psi_g^{i_2}\dit(\chihat,\tr\chi)\dif\etabar \difi\eta +\al \sum_{i_1+i_2+i_3+i_4+i_5= 1}\dione \psi_g^{i_2}\dit(\beta,\Te_4)\dif \eta \notag \\ +& \sum_{i_1+i_2+i_3+i_4+i_5=1} \dione \psi_g^{i_2}\dit\modu \dif \sigma \difi\eta  \notag \\ +&\sum_{i_1+i_2+i_3+i_4+i_5+i_6=1}\dione \psi_g^{i_2}\dit\modu \dif (\eta,\etabar)\difi (\eta,\etabar)\mathcal{D}^{i_6}\eta := T_1+\dots+T_8.
\end{align}As such, there holds
\[ \scaletwoSu{\tr\chi}\leq \lVert \tr\chi \rVert_{\mathcal{L}^2_{(sc)}(S_{u,0})} + \sum_{i=1}^8 \intubar \scaletwoSuubarprime{T_i}\dubarprime.    \]
 The most dangerous terms are as follows:

\begin{align}
&\intubar \sum_{i_1+i_2+i_3+i_4=2}\scaletwoSuubarprime{\db(\tr\chi,\chihat)\dif(\tr\chi,\chihat)}\dubarprime \notag \\ \lesssim& \intubar (\scaletwoSuubarprime{\chihat \cdot \De^2 \chihat}+\scaletwoSuubarprime{\De \chihat \cdot \De \chihat} )\dubarprime +1 \notag \\ \lesssim& \frac{a}{\modu}(\mathcal{R}_2[\alpha]+1)\big[\big(\mathcal{R}^{\f12}_0[\alpha]+\mathcal{R}^{\f12}_2[\alpha]\big)\mathcal{R}^{\frac{1}{2}}_1[\alpha] +\mathcal{R}_0[\alpha]+\mathcal{R}_1[\alpha]+1\big] \notag \\ +&\frac{a}{\modu}\big(\mathcal{R}_1^{\f12}[\alpha]\mathcal{R}_2^{\f12}[\alpha]+\mathcal{R}_1[\alpha]+1\big)^2 +1,
\end{align}where we have used the improved estimates on $\chihat$ given in Proposition \ref{propositionDchihat}, Remark \ref{chihatinfinityr} and Proposition \ref{propositionD2chihat}.
For the matter term, we have 

\be \intubar \sum_{i_1+i_2+i_3=2}\scaletwoSuubarprime{\db\Te_{44}}\dubarprime \lesssim \frac{a}{\modu}\Ve_{2}[\Te_{44}]+1, \ee since the slowest-decaying term is the one when two derivatives fall on $\Te_{44}$. The estimates from Proposition \ref{preliminary} and bootstrap assumptions \eqref{boundsbootstrap} bound the rest of the terms, which are non-borderline. The result follows.
\end{proof}
\noindent We now move on to estimates for the components satisfying $\snabla_3$-equations. 

\begin{proposition}
Under the assumptions of Theorem \ref{mainone} and the bootstrap assumptions \eqref{boundsbootstrap}, there holds

\[ \frac{\al}{\modu}\scaletwoSu{\De^2\chibarhat}\lesssim 1.   \]
\end{proposition}

\begin{proof}
Given the structure equation for $\chibarhat$ and the commutation formula \eqref{commutationformulanabla3}, we have 

\begin{align}
\snabla_3 \De^2 \chibarhat + \big(1+\frac{\ell}{2}\big) \tr\chibar \hsp \chibarhat =& \sum_{i_1+i_2+i_3+i_4=2} \db \omegabar \dif \chibarhat +\sum_{i_1+i_2+i_3=2}\db \alphabar \notag \\ +&\sum_{i_1+i_2+i_3+i_4=1} \db \omegabar \mathcal{D}^{i_4+1}\chibarhat \notag  + \frac{1}{\al}\sum_{i_1+i_2+i_3+i_4=1}\db(\eta,\etabar,\chibarhat)\mathcal{D}^{i_4+1}\chibarhat \notag \\ +&\sum_{i_1+i_2+i_3+i_4=1}\db \sdiv \etabar \dif \chibarhat \notag + \sum_{i_1+i_2+i_3+i_4+i_5=1} \db \chihat \dif \chibarhat \difi \chibarhat \\ +& \sumfiveone \db \tr\chi \dif \tr\chibar \difi \chibarhat + \sumfourone \db (\rho,\Te_{34})\dif \chibarhat  \notag \\ +&\sumfourone \db (\betabar,\Te_3) \dif \chibarhat +\frac{1}{\al}\sumfourone \dione \psi_g^{i_2}\mathcal{D}^{i_3+1}\tildetr \dif \chibarhat \notag \\ +& \sumfiveone \db (\eta,\etabar) \dif(\eta,\etabar,\chibarhat,\tr\chibar)\difi \chibarhat \notag \\ +&\sumfiveone \db \sigma \dif \chibarhat \notag +\sumfiveone \db \modu \dif \Te_{33}\difi \chibarhat \notag \\ +& \sumsixone \db \tr\chibar \dif \modu \difi \tildetr \mathcal{D}^{i_6}\chibarhat \notag \\ +& \sumsixone \db \modu \dif \chibarhat \difi \chibarhat \mathcal{D}^{i_6}\chibarhat \notag \\ +&\sumsixone \db \modu \dif \omegabar \difi \tr\chibar \mathcal{D}^{i_6} \chibarhat := T_1+\dots +T_{16}.
\end{align}Passing to scale-invariant norms and using the Evolution Lemma \ref{evolemma}, we arrive at \begin{align}
\frac{\al}{\modu}\scaletwoSu{\De^2 \chibarhat}\lesssim \frac{\al}{\lvert u_{\infty}\rvert} \lVert \De^2 \chihat \rVert_{\mathcal{L}^2_{(sc)}(S_{u_{\infty},\ubar})} + \sum_{i=1}^{16} \intu \frac{a^{\f32}}{\upr^3}\scaletwoSuprime{T_i}\duprime.
\end{align}The top-order terms are handled as follows:

\begin{align}\sumfourtwo
\intu \frac{a^{\f32}}{\upr^3}\scaletwoSuprime{\db \omegabar \dif\chibarhat} \duprime\lesssim \intu \frac{a}{\upr^2}\cdot \frac{O^2}{\upr}\duprime \lesssim \frac{a \hsp O^2}{\modu^2}\lesssim 1, 
\end{align}For the curvature term, we have 

\be \sumfourtwo \intu \frac{a^{\f32}}{\upr^3}\scaletwoSuprime{\db \alphabar} \duprime \lesssim \intu \frac{a^{\f32}}{\upr^3}\scaletwoSuprime{\De^2\alphabar} \duprime +1 \lesssim \frac{a}{\modu^{\f32}}\underline{\mathcal{R}}_2[\alphabar]+1 \lesssim 1. \ee For the remaining top-order terms, we have

\begin{align}
\sumfourone \intu  \frac{a^{\f32}}{\upr^3}\scaletwoSuprime{\db \big(\omegabar,\eta,\etabar,\frac{1}{\al}\chibarhat\big) \mathcal{D}^{i_4+1}\chibarhat}\duprime \lesssim \intu \frac{a^{\f32}}{\upr^3}\cdot \frac{\al \hsp O^2}{\upr}\duprime \lesssim 1,
\end{align}
\begin{align}
&\sumfourone \intu \frac{a^{\f32}}{\upr^3}\scaletwoSuprime{\db \sdiv \etabar \dif \chibarhat} \duprime \notag \\ \lesssim& \sumfourone \intu \frac{a^{\f32}}{\upr^3}\cdot \frac{1}{\al}\scaletwoSuprime{\dione \psi_g^{i_2}\mathcal{D}^{i_3+1}\etabar \dif \chibarhat} \duprime \lesssim \frac{a^{\f32}O^2}{\modu^3}\lesssim 1.
\end{align} The lower-order terms are treated as in Proposition \ref{preliminary}, using the bootstrap assumptions. This concludes the proof.
\end{proof}
\begin{proposition}
Under the assumptions of Theorem \ref{mainone} and the bootstrap assumptions \eqref{boundsbootstrap}, there holds \[ \frac{a}{\modu^2}\scaletwoSu{\De^2 \tr\chibar}\lesssim 1,\frac{a}{\modu}\scaletwoSu{\De^2 \tildetr} \lesssim  1. \]
\end{proposition}
\begin{proof}
We have the following structure equation for $\tr\chibar$:

\begin{align}
&\snabla_3 \De^2 \tr\chibar + \frac{1+\ell}{2} \tr\chibar \De^2 \tr\chibar \notag \\ =& \sum_{i_1+i_2+i_3+i_4=2} 
\db \chibarhat \dif \chibarhat +\sum_{i_1+i_2+i_3=2} \db \Te_{33} \notag \\  +&\sum_{i_1+i_2+i_3+i_4=1} \db \omegabar \mathcal{D}^{i_4+1}\tr\chibar \notag  + \frac{1}{\al}\sum_{i_1+i_2+i_3+i_4=1}\db(\eta,\etabar,\chibarhat)\mathcal{D}^{i_4+1}\tr\chibar \notag \\ +&\sum_{i_1+i_2+i_3+i_4=1}\db \sdiv \etabar \dif \tr\chibar \notag + \sum_{i_1+i_2+i_3+i_4+i_5=1} \db \chihat \dif \chibarhat \difi \etabar \\ +& \sumfiveone \db \tr\chi \dif \tr\chibar \difi \tr\chibar + \sumfourone \db (\rho,\Te_{34})\dif \tr\chibar  \notag \\ +&\sumfourone \db (\betabar,\Te_3) \dif \tr\chibar +\frac{1}{\al}\sumfourone \dione \psi_g^{i_2}\mathcal{D}^{i_3+1}\tildetr \dif \tr\chibar \notag \\ +& \sumfiveone \db (\eta,\etabar) \dif(\eta,\etabar,\chibarhat,\tr\chibar)\difi \tr\chibar \notag \\ +&\sumfiveone \db \sigma \dif \tr\chibar \notag +\sumfiveone \db \modu \dif \Te_{33}\difi \tr\chibar \notag \\ +& \sumsixone \db \tr\chibar \dif \modu \difi \tildetr \mathcal{D}^{i_6}\tr\chibar \notag \\ +& \sumsixone \db \modu \dif \chibarhat \difi \chibarhat \mathcal{D}^{i_6}\tr\chibar \notag \\ +&\sumsixone \db \modu \dif \omegabar \difi \tr\chibar \mathcal{D}^{i_6} \tr\chibar := T_1+\dots +T_{16}.
\end{align}Making use of Proposition \ref{evolemma} and passing to scale-invariant norms, we have 

\be \frac{a}{\modu^2} \scaletwoSu{\De^2 \tr\chibar}\lesssim \frac{a}{\lvert u_{\infty} \rvert^2}\lVert{\De^2 \tr\chibar}\rVert_{\mathcal{L}^2_{(sc)}(S_{u_{\infty},\ubar})}+  \sum_{i=1}^{16} \intu \frac{a^2}{\upr^4} \scaletwoSuprime{T_i}\duprime . \ee For the top-order terms, we have 

\begin{align} \sumfourtwo 
\intu \frac{a^2}{\upr^4}\scaletwoSuprime{\db \chibarhat \dif \chibarhat} \duprime\lesssim \intu \frac{a^2}{\upr^4}\frac{\upr^2}{a} \frac{O^2}{\upr} \duprime \lesssim 1,
\end{align}

\begin{align}
&\sumthreetwo \intu \frac{a^2}{\upr^4}\scaletwoSuprime{\db \Te_{33}} \duprime \notag  \lesssim \intu \frac{a^2}{\upr^4}\scaletwoSuprime{\De^2{T_{33}}}\duprime +1 \notag \\ \lesssim& \frac{a^{\f32}}{\modu^{\f52}}\scaletwoHbaru{\De^2 \Te_{33}}+1 \lesssim 1,
\end{align}
\begin{align}
&\sumfourone \intu \frac{a^2}{\upr^4}\scaletwoSuprime{\db\frac{1}{\al}(\al\omegabar,\eta,\etabar,\chibarhat)\mathcal{D}^{i_4+1}\tr\chibar}\duprime \notag \\ \lesssim& \intu \frac{a}{\upr^2}\frac{O^2}{\upr}\duprime \lesssim 1,
\end{align}

\begin{align}
\sumfourone \intu\frac{a^2}{\upr^4}\scaletwoSuprime{\db \sdiv \etabar \dif\tr\chibar} \lesssim \intu \frac{1}{\al}\frac{a^2}{\upr^4}\frac{\upr^2}{a}\frac{O^2}{\upr}\duprime \lesssim 1.
\end{align}The lower order terms are bounded above by 1 in the usual way, by using the estimates from Proposition \ref{preliminary} together with the bootstrap assumptions \eqref{boundsbootstrap}. For $\tildetr$, we have

\begin{align}
&\snabla_3 \De^2 \tildetr + \big(\frac{\ell}{2}+1\big)  \tr\chibar \De^2 \tildetr \notag \\ =& \sum_{i_1+i_2+i_3+i_4=2} 
\db (\tildetr,\chibarhat) \dif (\tildetr,\chibarhat) +\sum_{i_1+i_2+i_3=2} \db \Te_{33} \notag \\  +&\sum_{i_1+i_2+i_3+i_4=1} \db \omegabar \mathcal{D}^{i_4+1}\tildetr \notag  + \frac{1}{\al}\sum_{i_1+i_2+i_3+i_4=1}\db(\eta,\etabar,\chibarhat)\mathcal{D}^{i_4+1}\tildetr \notag \\ +&\sum_{i_1+i_2+i_3+i_4=1}\db \sdiv \etabar \dif \tildetr \notag + \sum_{i_1+i_2+i_3+i_4+i_5=1} \db \chihat \dif \chibarhat \difi \etabar \\ +& \sumfiveone \db \tr\chi \dif \tr\chibar \difi \tildetr + \sumfourone \db (\rho,\Te_{34})\dif \tildetr  \notag \\ +&\sumfourone \db (\betabar,\Te_3) \dif \tildetr +\frac{1}{\al}\sumfourone \dione \psi_g^{i_2}\mathcal{D}^{i_3+1}\tildetr \dif \tildetr \notag \\ +& \sumfiveone \db (\eta,\etabar) \dif(\eta,\etabar,\chibarhat,\tr\chibar)\difi \tildetr \notag \\ +&\sumfiveone \db \sigma \dif \tildetr \notag +\sumfiveone \db \modu \dif \Te_{33}\difi \tildetr \notag \\ +& \sumsixone \db \tr\chibar \dif \modu \difi \tildetr \mathcal{D}^{i_6}\tildetr \notag \\ +& \sumsixone \db \modu \dif \chibarhat \difi \chibarhat \mathcal{D}^{i_6}\tildetr \notag \\ +&\sumsixone \db \modu \dif \omegabar \difi \tr\chibar \mathcal{D}^{i_6} \tildetr := T_1+\dots +T_{16}.
\end{align}
Passing to scale-invariant norms and using Proposition \ref{evolemma}, we have

\begin{align}
\frac{a}{\modu}\scaletwoSu{\De^2 \tildetr} \lesssim \frac{a}{\lvert u_{\infty}\rvert}\lVert \De^2 \tildetr \rVert_{\mathcal{L}^2_{(sc)}(S_{u_{\infty},\ubar})} + \sum_{i=1}^{16} \intu \frac{a^2}{\upr^3}\scaletwoSuprime{T_i}\duprime . \label{de2ttr}
\end{align}We once again only give details for the top-order terms, which are the most borderline. For the first term,

\begin{align}
&\sumfourtwo \intu\frac{a^2}{\upr^3}\scaletwoSuprime{\db(\tildetr,\chibarhat)\dif(\tildetr,\chibarhat)} \duprime\\ \lesssim& \sumfourtwo \intu \frac{a^2}{\upr^3}\scaletwoSuprime{\db \chibarhat \dif \chibarhat}\duprime +1 \notag \\ \lesssim& \frac{a}{\modu}(\mathcal{O}_{0,\infty}[\chibarhat]\mathcal{O}_{2,2}[\chibarhat]+\mathcal{O}_{1,4}[\chibarhat]^2)+1 \lesssim 1,
\end{align}where we have used the improvements on $\mathcal{O}_{0,\infty}[\chibarhat],\mathcal{O}_{1,4}[\chibarhat],\mathcal{O}_{2,2}[\chibarhat]$. For the next terms, making use of the bootstrap assumptions,  we have 
\begin{align}
\sumthreetwo \intu \frac{a^2}{\upr^3}\scaletwoSuprime{\db \Te_{33}}\duprime \lesssim \frac{a^{\f32}}{\modu^{\f32}} \scaletwoHbaru{\De^2{\Te_{33}}}+1\lesssim 1,
\end{align}

\begin{align}\sumfourone \intu\frac{a^{\f32}}{\upr^3}\scaletwoSuprime{\db(\al\omegabar,\eta,\etabar,\chibarhat)\mathcal{D}^{i_4+1}\tildetr} \duprime \lesssim \frac{O^2}{\modu}\lesssim 1, \end{align}
\begin{align}
\sumfourone\intu\frac{a^{2}}{\upr^3}\scaletwoSuprime{\db \sdiv \etabar \dif \tildetr}\duprime \lesssim \frac{\al O^2}{\modu}\lesssim 1, 
\end{align}where we have used the preliminary estimates from Proposition \ref{preliminary} and the bootstrap assumptions. These also bound all the lower-order terms in \eqref{de2ttr} by 1. The result follows.
\end{proof}

\begin{proposition}Under the assumptions of Theorem \ref{mainone} and the bootstrap assumptions \eqref{boundsbootstrap}, there exists a constant $C$, depending on the curvature and Vlasov norms, such that \[ \scaletwoSu{\De^2 \etabar}\lesssim C(\mathcal{R},\underline{\mathcal{R}},\Ve). \]
\end{proposition}
\begin{proof}
We have the following schematic identity:

\begin{align}
&\snabla_3 \De^2 \etabar + \frac{\ell+1}{2}\tr\chibar \De^2 \etabar \notag \\=& \sumfourtwo \db \tr\chibar \dif \eta +\sumfourtwo \db \chibarhat \dif(\eta,\etabar)  \notag  + \sumthreetwo \db (\betabar,\Te_{3})  \\ +&\sum_{i_1+i_2+i_3+i_4=1} \db \omegabar \mathcal{D}^{i_4+1}\etabar \notag  + \frac{1}{\al}\sum_{i_1+i_2+i_3+i_4=1}\db(\eta,\etabar,\chibarhat)\mathcal{D}^{i_4+1}\etabar \notag \\ +&\sum_{i_1+i_2+i_3+i_4=1}\db \sdiv \etabar \dif \etabar \notag + \sum_{i_1+i_2+i_3+i_4+i_5=1} \db \chihat \dif \chibarhat \difi \etabar \\ +& \sumfiveone \db \tr\chi \dif \tr\chibar \difi \etabar + \sumfourone \db (\rho,\Te_{34})\dif \etabar  \notag \\ +&\sumfourone \db (\betabar,\Te_3) \dif \etabar +\frac{1}{\al}\sumfourone \dione \psi_g^{i_2}\mathcal{D}^{i_3+1}\tildetr \dif \etabar \notag \\ +& \sumfiveone \db (\eta,\etabar) \dif(\eta,\etabar,\chibarhat,\tr\chibar)\difi \etabar \notag \\ +&\sumfiveone \db \sigma \dif \etabar \notag +\sumfiveone \db \modu \dif \Te_{33}\difi \etabar \notag \\ +& \sumsixone \db \tr\chibar \dif \modu \difi \tildetr \mathcal{D}^{i_6}\etabar \notag \\ +& \sumsixone \db \modu \dif \chibarhat \difi \chibarhat \mathcal{D}^{i_6}\etabar \notag \\ +&\sumsixone \db \modu \dif \omegabar \difi \tr\chibar \mathcal{D}^{i_6} \etabar := T_1+\dots +T_{17}.
\end{align}Using Proposition \ref{evolemma} and passing to scale-invariant norms, we have 

\begin{align}
\frac{\al}{\modu}\scalefourSu{\De^2 \etabar}\lesssim \frac{\al}{\lvert u_{\infty}\rvert}\scalefourSuzero{\De^2 \etabar}+ \sum_{i=1}^{17}\intu \frac{a^{\f32}}{\upr^3}\scaletwoSuprime{T_i}\duprime.
\end{align} For the first term, we have 

\begin{align} &\sumfourtwo \intu \frac{a^{\f32}}{\upr^3} \scaletwoSuprime{\db \tr\chibar \dif \eta} \duprime \\ \lesssim& \intu \frac{a^{\f32}}{\upr^3}\big( \scaletwoSuprime{\tr\chibar \De^2\eta}+ \scaletwoSuprime{\De\tr\chibar \De \eta }+\scaletwoSuprime{\eta \De^2 \tr\chibar}\big)\duprime +\frac{\al}{\modu} \notag \\ \lesssim& \frac{\al}{\modu}\big( \sup_{S_{u^{\prime},\ubar^{\prime}}} \scaletwoSuprimeubarprime{\De^2\eta}+ \sup_{S_{u^{\prime},\ubar^{\prime}}} \lVert \eta \rVert_{\mathcal{L}^{\infty}_{(sc)}(S_{u^{\prime},\ubar^{\prime}})} +\sup_{S_{u^{\prime},\ubar^{\prime}}} \lVert \De \eta \rVert_{\mathcal{L}^{4}_{(sc)}(S_{u^{\prime},\ubar^{\prime}})}\big)+\frac{\al}{\modu}.  \end{align}Here we have made use of the improved estimates on $\De^2 \eta, \hsp \De^2\tr\chibar $ as well as the $\mathcal{L}^4_{(sc)}, \mathcal{L}^{\infty}_{(sc)}$ estimates from Sections \ref{sectionO14} and \ref{section0infty}.  The second term is bounded by

\begin{align}
\sumfourtwo \intu \frac{a^{\f32}}{\upr^3}\scaletwoSuprime{\db \chibarhat \dif(\eta,\etabar)}\duprime \lesssim \intu \frac{a^{\f32}}{\upr^3}\cdot \frac{\al O^2}{\upr}\duprime \lesssim \frac{\al}{\modu}.
\end{align}For the third term, we have 
    \be \sumthreetwo \intu \frac{a^{\f32}}{\upr^3}\scaletwoSuprime{\db(\betabar,\Te_3)} \lesssim \frac{\al}{\modu}(\underline{\mathcal{R}}_2[\betabar]+\underline{\Ve}_2[\Te_{3}]) +\frac{\al}{\modu}. \ee The fourth, fifth and sixth terms, which also can be considered top-order, are bounded by $\frac{\al}{\modu}$, together with the rest of the terms using the results of Proposition \ref{preliminary}. Ultimately, we obtain

    \be \frac{\al}{\modu}\scaletwoSu{\De^2\etabar} \lesssim \frac{\al}{\lvert u_{\infty}\rvert}\lVert \De^2 \etabar \rVert_{\mathcal{L}^2_{(sc)}(S_{u_{\infty},\ubar})} +\frac{\al}{\modu}C(\mathcal{R},\underline{\mathcal{R}},\Ve). \ee Dividing by $\frac{\al}{\modu}$, we get the desired result.

\end{proof}
\begin{proposition}
Under the assumptions of Theorem \ref{mainone} and the bootstrap assumptions \eqref{boundsbootstrap}, there  exists a constant $C$ depending on the total norms $\mathcal{R},\underline{\mathcal{R}},\mathcal{V}$ such that there holds 

\[ \scaletwoSu{\De^2\mathcal{G}} \lesssim \prescript{(S)}{}{\mathcal{O}}_{3,2}[\eta] + C(\mathcal{R},\underline{\mathcal{R}},\mathcal{V}). \]
\end{proposition}
\begin{proof}
There holds \begin{align}
\snabla_4 \De^2 \mathcal{G} =& \sum_{i_1+i_2+i_3=2}\dione \psi_g^{i_2}\dit  \snabla \eta +\sum_{i_1+i_2+i_3+i_4=2}\dione \psi_g^{i_2}\dit  (\omegabar,\mathcal{G})\dif (\chihat,\tr\chi) \notag \\ +&\sum_{i_1+i_2+i_3+i_4=2}\dione \psi_g^{i_2}\dit  (\eta, \Gammaslash)\dif (\eta,\etabar) +\sum_{i_1+i_2+i_3+i_4=2}\dione \psi_g^{i_2}\dit  (\rho,\sigma,\Te) \notag \\ +& \frac{1}{\al}\sum_{i_1+i_2+i_3+i_4+i_5=1} \db\modu \dif(\eta,\etabar)\mathcal{D}^{i_5+1}\mathcal{G} \notag \\ +&\sum_{i_1+i_2+i_3+i_4=1} \db(\tr\chi,\chihat)\mathcal{D}^{i_4+1}\mathcal{G} \notag \\ +&\al \sum_{i_1+i_2+i_3+i_4+i_5=1}\db(\tr\chi,\chihat)\dif \etabar \difi \mathcal{G} \notag \\ +& \al \sum_{i_1+i_2+i_3+i_4=1} \db (\beta,\Te_4)\dif \mathcal{G} \notag \\ +&\sum_{i_1+i_2+i_3+i_4+i_5=1} \db \modu \dif \sigma \difi \mathcal{G} \notag \\ +&\sum_{i_1+i_2+i_3+i_4+i_5+i_6=1} \db \modu \dif(\eta,\etabar)\difi(\eta,\etabar)\mathcal{D}^{i_6}\mathcal{G} \end{align}The first term is the one of highest order and is bounded as before by 

\be \intubar \scaletwoSuubarprime{ \De^2 \snabla \eta}\dubarprime \lesssim  \mathcal{O}_{3,2}[\eta].\ee The rest of the terms are bounded in total by a constant $C(\mathcal{R},\underline{\mathcal{R}},\Ve)$ as before. The claim follows.
\end{proof}
Finally, we conclude with $\prescript{(S)}{}{\mathcal{O}}_{2,2}$ estimates for $\Gammaslash-\Gammaslash^{\circ}$.

\begin{proposition}
Under the assumptions of Theorem \ref{mainone} and the bootstrap assumptions \eqref{boundsbootstrap}, there  exists a constant $C$ depending on the total norms $\mathcal{R},\underline{\mathcal{R}},\mathcal{V}$ such that there holds 

\[ \scaletwoSu{\De^2 (\Gammaslash-\Gammaslash^{\circ})}\lesssim \prescript{(S)}{}{\mathcal{O}}_{3,2}[\chihat]+C(\mathcal{R},\underline{\mathcal{R}},\Ve).   \]
\end{proposition}
\begin{proof}
There holds
\begin{align}
\snabla_4 \De^2(\Gammaslash-\Gammaslash^{\circ}) =& \sum_{i_1+i_2+i_3=2}\dione \psi_g^{i_2}\dit  \snabla (\chihat,\tr\chi) +\sum_{i_1+i_2+i_3+i_4=2}\dione \psi_g^{i_2}\dit  \mathcal{G}\dif (\Gammaslash-\Gammaslash^{\circ}) \notag \\+& \frac{1}{\al}\sum_{i_1+i_2+i_3+i_4+i_5=1} \db\modu \dif(\eta,\etabar)\mathcal{D}^{i_5+1}(\Gammaslash-\Gammaslash^{\circ}) \notag \\ +&\sum_{i_1+i_2+i_3+i_4=1} \db(\tr\chi,\chihat)\mathcal{D}^{i_4+1}(\Gammaslash-\Gammaslash^{\circ}) \notag \\ +&\al \sum_{i_1+i_2+i_3+i_4+i_5=1}\db(\tr\chi,\chihat)\dif \etabar \difi (\Gammaslash-\Gammaslash^{\circ})\notag \\ +& \al \sum_{i_1+i_2+i_3+i_4=1} \db (\beta,\Te_4)\dif (\Gammaslash-\Gammaslash^{\circ}) \notag \\ +&\sum_{i_1+i_2+i_3+i_4+i_5=1} \db \modu \dif \sigma \difi (\Gammaslash-\Gammaslash^{\circ}) \notag \\ +&\sum_{i_1+i_2+i_3+i_4+i_5+i_6=1} \db \modu \dif(\eta,\etabar)\difi(\eta,\etabar)\mathcal{D}^{i_6}(\Gammaslash-\Gammaslash^{\circ}) \end{align}

The highest-order term is the first one. We bound this term by 

\be \intubar \scaletwoSuubarprime{\De^2\snabla(\chihat,\tr\chi)}\dubarprime \lesssim \scaletwoHu{\De^2\snabla(\chihat,\tr\chi)} \lesssim \prescript{(S)}{}{\mathcal{O}}_{3,2}[\chihat,\tr\chi].\ee
The rest of the error terms are bounded by $C(\mathcal{R},\underline{\mathcal{R}},\Ve)$ as in previous propositions. The result follows. 
\end{proof}

\section{Elliptic estimates for $\chihat,\tr\chi, \eta$}\label{ellipticsection}
We begin by proving two propositions about curvature components. 

\subsection{$\mathcal{L}^4_{(sc)}(S_{u,\ubar})$ estimates for Curvature Components}\label{section40curv}
For the curvature components $\beta,\rho,\sigma,\betabar$ we have the following equations:

\begin{gather}
    \snabla_4 \beta +2 \tr\chi \hsp \beta = \slashed{\div}\alpha +\omega \hsp \beta + (\eta^\# - 2\etabar^\#)\cdot \alpha - \frac{1}{2}(\snabla \slashed{T}_{44} -\snabla_4 \slashed{T}_4 -2\chi\cdot \slashed{T}_4 +\omega \hsp \slashed{T}_4 - \slashed{T}_{44}\eta), \\ \snabla_4 \rho + \frac{3}{2}\tr\chi \rho =  \slashed{\div}\beta -\frac{1}{2} \hsp \chibarhat^\# \cdot \alpha +(\etabar, \beta)- \frac{1}{4}(\snabla_3 \slashed{T}_{44} -\snabla_4 \slashed{T}_{34} -2 \eta \cdot \slashed{T}_4 +\omega \hsp \slashed{T}_{34}), \\ \snabla_4 \sigma + \frac{3}{2}\tr\chi \sigma  = -\slashed{\curl}\beta - \etabar \wedge \beta +\frac{1}{2}\hsp \chibarhat \wedge \alpha - \frac{1}{2}(\slashed{\epsilon}\cdot \snabla \slashed{T}_4 -\slashed{\epsilon} \cdot(\chi \times \slashed{T})- \slashed{\epsilon}\cdot(\etabar \otimes \slashed{T}_4)),
\end{gather}
\begin{align}
  \snabla_4 \betabar +\tr\chi \betabar =& -\snabla \rho + \Hodge{\snabla}\sigma -\omega \betabar +2\chibarhat^\#\cdot \beta -3 (\etabar \hsp \rho -\Hodge{\etabar}\hsp \sigma) \notag \\-& \frac{1}{2}(\snabla \slashed{T}_{34} -\snabla_3 \slashed{T}_4 -\chibar\cdot \slashed{T}_4 -\chi\cdot \slashed{T}_3 +\slashed{T}_{34}\eta +2\eta\cdot\slashed{T}). 
\end{align}
For $\beta,$ we have

\begin{align} \notag
&\scalefourSu{\beta} \\ \notag \lesssim& \intubar \frac{1}{\modu}\scaleinftySu{\omega,\tr\chi}\scalefourSuubarprime{\beta} \dubarprime + \big( \frac{1}{\al} \scaletwoHu{(\al\snabla)\alpha}\big)^{\frac{1}{2}}\big(\frac{1}{\al} \scaletwoHu{(\al\snabla)^2\alpha}\big)^{\frac{1}{2}} \\ +& \frac{1}{\al}\scaletwoHu{(\al\snabla)\alpha} + \frac{\al}{\modu}\sup_{0\leq \ubar \leq 1}\scaleinftySuubarprime{(\eta,\etabar)}\big(\mathcal{R}_0[\alpha]^{\frac{1}{2}}\mathcal{R}_1[\alpha]^{\frac{1}{2}} +\mathcal{R}_{0}[\alpha]\big) \notag \\ +& \frac{1}{\al}\big(\scaletwoHu{\mathcal{D}\slashed{T}_{44}}^{\frac{1}{2}}\scaletwoHu{\mathcal{D}^2\slashed{T}_{44}}^{\frac{1}{2}} + \scaletwoHu{\mathcal{D}\slashed{T}_{44}}\big) +\scaletwoHu{\mathcal{D}\slashed{T}_{4}}^{\frac{1}{2}}\scaletwoHu{\mathcal{D}^2\slashed{T}_{4}}^{\frac{1}{2}} \notag \\ +& \scaletwoHu{\mathcal{D}\slashed{T}_4} + \frac{\al}{\modu}\sup_{0\leq \ubar \leq 1} \scaleinftySuubarprime{(\frac{1}{\al}\chihat,\omega,\tr\chi)}\scalefourSuubarprime{\slashed{T}_4} \notag \\ +& \frac{1}{\modu}\sup_{0\leq \ubar\leq 1}\scalefourSuubarprime{\slashed{T}_{44}}\scaleinftySuubarprime{\eta}.
\end{align}

\begin{proposition} \label{propositionK-0}
Under the assumptions of Theorem \ref{mainone} and the bootstrap assumptions \eqref{boundsbootstrap}, there holds:

\begin{align*}   \scalefourSu{K-\frac{1}{\modu^2}}\lesssim& \frac{a}{\modu}\scaletwoHu{\mathcal{D}\slashed{T}_{44}}^{\frac{1}{2}}\scaletwoHu{\mathcal{D}^2\slashed{T}_{44}}^{\frac{1}{2}} +\frac{a}{\modu}\scaletwoHu{\mathcal{D}\slashed{T}_{44}}\notag \\ +& \scaletwoHu{\mathcal{D}\slashed{T}_{34}}^{\frac{1}{2}}\scaletwoHu{\mathcal{D}^2\slashed{T}_{34}}^{\frac{1}{2}} +\scaletwoHu{\mathcal{D}\slashed{T}_{34}}+1.    \end{align*}
\end{proposition}

\begin{proof} We begin with the bootstrap assumption \be \label{bootstrapK0} \frac{a}{\modu}\scalefourSu{K}\leq B_0, \ee where $B_0$ is a large constant. For $K$ there holds:


There holds \begin{align} \snabla_4 K +\tr\chi K=& - \slashed{\div}\beta -(\zeta+2\etabar)\cdot \beta +\frac{1}{2}\chihat \cdot \snabla \hat{\otimes}\etabar + \frac{1}{2}\chihat\cdot(\etabar \hat{\otimes}\etabar) -\frac{1}{2}\tr\chi \slashed{\div}\etabar- \frac{1}{2}\tr\chi \lvert \etabar \rvert^2 \notag \\ +& \chihat\cdot\slashed{T} +\frac{1}{4}\tr\chi\slashed{T}_{34}+\frac{1}{4}\tr\chibar\slashed{T}_{44}- \frac{1}{4}\big(\snabla_3\slashed{T}_{44}-2\eta\cdot\slashed{T}_{4}+\omega \slashed{T}_{34}\big)+\frac{3}{4}\snabla_4 \slashed{T}_{34}. 
\end{align}Hence, making use of the results of Sections \ref{sectionO04}-\ref{section0infty},

\begin{align}
\scalefourSu{K}\lesssim& \lVert {K}\rVert_{\mathcal{L}^4_{(sc)}(S_{u,0})}+ \intubar \frac{1}{\al}\scalefourSuubarprime{(\al\snabla) \beta}\dubarprime + \frac{a}{\modu}\intubar\scalefourSuubarprime{(\modu\snabla_3)\slashed{T}_{44}} \dubarprime \notag \\ +& \intubar \scalefourSuubarprime{\snabla_4 \slashed{T}_{34}}\dubarprime +1\notag \\ \lesssim& \lVert {K}\rVert_{\mathcal{L}^4_{(sc)}(S_{u,0})}+\frac{1}{\al}\big(\scaletwoHu{\mathcal{D} \beta}^{\frac{1}{2}}\scaletwoHu{\mathcal{D}^2\beta}^{\frac{1}{2}}+\scaletwoHu{\mathcal{D}\beta}\big) \notag \\ +& \scaletwoHu{\mathcal{D} (\slashed{T}_{34},\slashed{T}_{44})}^{\frac{1}{2}}\scaletwoHu{\mathcal{D}^2(\slashed{T}_{34},\slashed{T}_{44})}^{\frac{1}{2}}+\scaletwoHu{\mathcal{D}(\slashed{T}_{34},\slashed{T}_{44})} +1.
\end{align}Taking into account that $K=\frac{1}{\modu^2}$ on $S_{u,0}$, we have $\lVert K \rVert_{\mathcal{L}^4_{(sc)}(S_{u,0})} = \frac{\modu}{a}$, hence

\be \frac{a}{\modu}\scalefourSu{K}\lesssim 1. \ee 
This implies an improved estimate for $K-\frac{1}{\modu^2}$ via the equation

\begin{align}
\snabla_4 \big(K-\frac{1}{\modu^2}\big) =& - \slashed{\div}\beta -(\zeta+2\etabar)\cdot \beta +\frac{1}{2}\chihat \cdot \snabla \hat{\otimes}\etabar + \frac{1}{2}\chihat\cdot(\etabar \hat{\otimes}\etabar) -\frac{1}{2}\tr\chi \slashed{\div}\etabar- \frac{1}{2}\tr\chi \lvert \etabar \rvert^2 \notag \\ +& \chihat\cdot\slashed{T} +\frac{1}{4}\tr\chi\slashed{T}_{34}+\frac{1}{4}\tr\chibar\slashed{T}_{44}- \frac{1}{4}\big(\snabla_3\slashed{T}_{44}-2\eta\cdot\slashed{T}_{4}+\omega \slashed{T}_{34}\big)+\frac{3}{4}\snabla_4 \slashed{T}_{34}-\tr\chi K. 
\end{align}As such, 

\begin{align}
\scalefourSu{K-\frac{1}{\modu^2}}\lesssim& \frac{a}{\modu}\scaletwoHu{\mathcal{D}\slashed{T}_{44}}^{\frac{1}{2}}\scaletwoHu{\mathcal{D}^2\slashed{T}_{44}}^{\frac{1}{2}} +\frac{a}{\modu}\scaletwoHu{\mathcal{D}\slashed{T}_{44}}\notag \\ +& \scaletwoHu{\mathcal{D}\slashed{T}_{34}}^{\frac{1}{2}}\scaletwoHu{\mathcal{D}^2\slashed{T}_{34}}^{\frac{1}{2}} +\scaletwoHu{\mathcal{D}\slashed{T}_{34}}+1.
\end{align}
\end{proof}
\subsection{$\mathcal{L}^2_{(sc)}(S_{u,\ubar})$ estimates for First Derivatives of Curvature}Under the assumptions of Theorem \ref{mainone} and the bootstrap assumptions \ref{boundsbootstrap}, there exists a constant $C$ depending only on $\mathcal{R}_1,\mathcal{R}_2,\mathcal{V}_1$ and $\mathcal{V}_2$ such that, for first derivatives of $\beta,\rho,\sigma,\betabar$, there holds:

\begin{align} \sum_{\mathcal{D}\in \{\modu \snabla_3, \snabla_4,\al\snabla\}}
\scaletwoSu{\mathcal{D}(\beta,\rho,\sigma,\betabar) } \lesssim C(\mathcal{R}_1,\mathcal{R}_2, \mathcal{V}_1,\mathcal{V}_2).
\end{align}

\begin{proof}
We begin with the bootstrap assumptions
\be \label{bootstrapcurvoneltwo} \sum_{\De \in \{\modu \snabla_3,\snabla_4,\al\snabla\}} \scaletwoSu{\De(\beta,\rho,\sigma,\betabar)}\leq C_2, \ee where $C_2$ is a large constant. The following equations hold, 

\begin{align}
\snabla_4 \mathcal{D}\beta =& \frac{1}{\al}\mathcal{D}^2\alpha +\mathcal{D}(\omega,\tr\chi)\beta + (\omega,\tr\chi)\mathcal{D}\beta +\mathcal{D}(\eta,\etabar)\alpha + (\eta,\etabar)\mathcal{D}\alpha + \frac{1}{\al}\mathcal{D}^2\slashed{T}_{44} \notag \\ +& \mathcal{D}^2 \slashed{T}_{4} + (\chihat,\tr\chi,
\omega)\mathcal{D}\slashed{T}_4+ \mathcal{D}(\chihat,\tr\chi)\slashed{T}_4 + \eta \mathcal{D}\slashed{T}_{44}+\De \eta \slashed{T}_{44} +\frac{1}{\al}\modu (\eta,\etabar)\mathcal{D}\beta \notag \\+& \modu (\eta,\etabar)(\eta,\etabar)\beta +\modu \sigma \beta +\omega \mathcal{D}\beta + \frac{1}{\al}(\chihat,\tr\chi)\mathcal{D}\beta + (\chihat,\tr\chi)\etabar \beta+ \al(\beta,\slashed{T}_{4})\beta,
\end{align}

\begin{align}
\snabla_4 \De \rho =& \frac{1}{\al}\mathcal{D}^2 \beta + \mathcal{D}\tr\chi \rho +\tr\chi \De \rho + \chibarhat \hsp \mathcal{D}\alpha + \mathcal{D}\chibarhat\hsp  \alpha + \mathcal{D}\big(\frac{1}{\modu}\mathcal{D}\slashed{T}_{44} \big)+ \mathcal{D}^2 \slashed{T}_{34} \notag \\ +& \mathcal{D}\etabar \hsp \beta +\etabar \hsp \mathcal{D}\beta + \mathcal{D}\eta \hsp \slashed{T}_{4}+\eta \hsp \De \Te_4  + \De \omega \hsp \Te_{34}+\omega \De \Te_{34}  +\frac{1}{\al}\modu(\eta,\etabar)\mathcal{D}\rho \notag \\ +& \modu (\eta,\etabar)(\eta,\etabar)\rho + \modu \sigma \rho + \omega \De \rho + \frac{1}{\al}(\chihat,\tr\chi)\De \rho +(\chihat,\tr\chi)\etabar \rho +\al(\beta,\Te_4)\rho, 
\end{align}
\begin{align}
\snabla_4 \De \sigma =&\frac{1}{\al}\De^2 \beta + \De \tr\chi \sigma + \tr\chi \De \sigma + \etabar \De \beta + \De \etabar \hsp \beta + \chibarhat \De \alpha + \De \chibarhat \hsp \alpha + \frac{1}{\al}\De^2 \Te_4 \notag \\ +&\De (\chihat,\tr\chi)\hsp \slashed{T} +(\chihat,\tr\chi)\De \Te +\De \etabar \Te_4 + \etabar \De \Te_4 + \frac{1}{\al}\modu(\eta,\etabar)\De \sigma +\modu (\eta,\etabar)(\eta,\etabar)\sigma \notag \\ +& \modu \sigma^2 +\omega \De \sigma +\frac{1}{\al}(\chihat,\tr\chi)\De \sigma +(\chihat,\tr\chi)\etabar \sigma +\al(\beta,\slashed{T}_{4})\sigma, 
\end{align}

\begin{align}
\snabla_4 \mathcal{D}\betabar =& \frac{1}{\al}(\De^2 \rho+\De^2 \sigma) +(\tr\chi,\omega)\De \betabar +\betabar \De(\tr\chi,\omega) +\chibarhat \De \beta +\beta \De \chibarhat +(\rho,\sigma)\De \etabar + \etabar \De(\rho,\sigma) \notag \\ +& \frac{1}{\al} \mathcal{D}^2 \slashed{T}_{34} +\frac{1}{\modu}\mathcal{D}^2 \slashed{T}_{4} + \De(\chibarhat,\tr\chibar)\slashed{T}_{4}+(\chibarhat,\tr\chibar)\De\Te_4 +(\chihat,\tr\chi)\De \Te_3 +\Te_3 \De(\chihat,\tr\chi) \notag \\ +&\eta \De (\Te_{34},\Te) +(\Te_{34},\Te)\De \eta  + \frac{1}{\al}\modu (\eta,\etabar)\De \betabar + \modu (\eta,\etabar)(\eta,\etabar)\betabar+\modu \sigma \betabar+\omega \De \betabar\notag \\+&\frac{1}{\al}(\chihat,\tr\chi)\De\betabar +(\chihat,\tr\chi)\etabar \hsp \betabar +\al(\beta,\Te_4)\betabar. 
\end{align}For $\Psi \in\{\beta,\rho,\sigma,\betabar\}$, we bound 
\begin{align}
&\intubar \big(\frac{1}{\al} \scaletwoSuubarprime{\modu(\eta,\etabar)\De \Psi} + \scaletwoSuubarprime{\modu (\eta,\etabar)^2\Psi}\big)\dubarprime+\scaletwoSuubarprime{\modu\sigma\Psi}\big)\dubarprime \notag \\+&\intubar\big(\scaletwoSuubarprime{\omega\De\Psi} +\frac{1}{\al}\scaletwoSuubarprime{(\chihat,\tr\chi)\De \Psi}+\scaletwoSuubarprime{(\chihat,\tr\chi)\etabar \Psi}\big)\dubarprime \notag \\+& \al \intubar \scaletwoSuubarprime{(\beta,\Te_4)\Psi}\big)\dubarprime \lesssim 1,
\end{align}using the bootstrap assumptions on $\scaletwoSu{\mathcal{D}\Psi}$, together with the results of Sections \ref{sectionO04} and \ref{section0infty} and Subsection \ref{section40curv}. In particular, we bound \begin{align} \notag \intubar \scaletwoSuubarprime{\modu(\eta,\etabar)\De\Psi}\dubarprime \lesssim& \intubar \frac{1}{\modu^2}\scaleinftySuubarprime{\modu}\scaleinftySuubarprime{(\eta,\etabar)}\scaletwoSuubarprime{\De \Psi}\dubarprime \lesssim 1,    \end{align}where we have used the bootstrap $\scaletwoSu{\De \Psi}\leq C_2$ and the results on $(\eta,\etabar)$ from section \ref{section0infty} (even though just using the bootstrap assumption for $(\eta,\etabar)$ would also be enough). Similarly,

\[ \intubar \scaletwoSuubarprime{\modu(\eta,\etabar)^2\Psi}\dubarprime \lesssim \frac{1}{\modu^3} \scaleinftySuubarprime{\modu}\scaleinftySuubarprime{\eta,\etabar}\scalefourSuubarprime{(\eta,\etabar)}\scalefourSuubarprime{\Psi}\dubarprime \lesssim 1,    \]where we also use the results of Section \ref{section40curv} on curvature in $\mathcal{L}^4_{(sc)}$. Moreover, \[  \intubar \scaletwoSuubarprime{\modu\sigma\Psi}\dubarprime \lesssim  \frac{a}{\modu^2} \intubar \scalefourSuubarprime{\sigma}\scalefourSuubarprime{\Psi}\lesssim 1. \]For the terms involving $\scaletwoSuubarprime{\psi \De \Psi}$, we bound $\psi$ in $\mathcal{L}^{\infty}_{(sc)}$ and $\De \Psi$ in $\mathcal{L}^4_{(sc)}$, while for the term $\intubar \scaletwoSuubarprime{(\chihat,\tr\chi)\etabar \Psi}\dubarprime$, we bound $\etabar$ in $\mathcal{L}^{\infty}_{(sc)}$ and $(\chihat,\tr\chi)$ together with $\Psi$ in $\mathcal{L}^4_{(sc)}$. Finally, for the term involving $\scaletwoSuubarprime{(\beta,\slashed{T}_{4})\Psi}$, we bound both terms in $\mathcal{L}^4_{(sc)}.$ For each $\Psi \in \{\beta,\rho,\sigma,\betabar\}$, we give an explanation for the most difficult terms. For $\beta$, we have

\be \frac{1}{\al}\intubar \scaletwoSuubarprime{\De^2 \alpha}\dubarprime \lesssim \frac{1}{\al}\scaletwoHu{\De^2 \alpha}\lesssim \mathcal{R}_{2}[\alpha]. \ee Similarly,

\be \frac{1}{\al}\intubar \scaletwoSuubarprime{\De^2 \Te_{44}}\dubarprime \lesssim \frac{1}{\al}\cdot \frac{a}{\modu}\mathcal{V}_2[\Te_{44}] \lesssim \frac{\al V}{\modu}\lesssim 1. \ee For $\De^2 \Te_4$, we have 

\be
\intubar \scaletwoSuubarprime{\De^2 \Te_4}\dubarprime \lesssim \scaletwoHu{\De^2 \Te_4}=\mathcal{V}_2[\Te_4]. \ee For $\rho$, we have \be \intubar \frac{1}{\al}\scaletwoSuubarprime{\De^2 \beta}\dubarprime \lesssim \frac{1}{\al}\scaletwoHu{\De^2 \beta}\lesssim \frac{1}{\al}\mathcal{R}_2[\beta]\lesssim \frac{R}{\al}\lesssim 1, \ee
while for the borderline term
\begin{align}
&\intubar \scaletwoSuubarprime{\De\big(\frac{1}{\modu}\De \Te_{44}\big)} \dubarprime \notag \\ \lesssim& \intubar \big( \scaletwoSuubarprime{\frac{1}{\modu} \De^2 \Te_{44}}+ \scaletwoSuubarprime{\frac{1}{\modu}\Omegaterm \De\Te_{44}}\big)\dubarprime   + \intubar \scaletwoSuubarprime{\frac{1}{\modu}\De \Te_{44}}\dubarprime \notag \\ \lesssim& \frac{\modu}{a}\big(\scaletwoHu{\De \Te_{44}}+\scaletwoHu{\De^2 \Te_{44}}\big) +1\lesssim \Ve_{1}[\Te_{44}]+\Ve_2[\Te_{44}]+1.    
\end{align}Moreover, 

\be \intubar \scaletwoSuubarprime{\De^2 \Te_{34}}\dubarprime \lesssim \scaletwoHu{\De^2 \Te_{34}}\lesssim \Ve_{2}[\Te_{34}]. \ee
For $\sigma$, we have 

\be \frac{1}{\al}\intubar\big(\scaletwoSuubarprime{\De^2 \beta} +\scaletwoSuubarprime{\De^2\Te_4}\big)\dubarprime  \lesssim \frac{1}{\al} \big(\mathcal{R}_2[\beta]+\frac{a}{\modu}\Ve_{2}[\Te_{4}]\big) \lesssim 1.\ee For both $\rho$ and $\sigma$, we have the last borderline term

\be \intubar \scaletwoSuubarprime{\chibarhat\De\alpha}+\scaletwoSuubarprime{\alpha \De \chibarhat}\dubarprime \lesssim \mathcal{O}_{0,\infty}[\chibarhat]\mathcal{R}_1[\alpha]+\mathcal{O}_{1,4}[\chibarhat]\big(\mathcal{R}_0^{\frac{1}{2}}[\alpha]\mathcal{R}_1^{\frac{1}{2}}[\alpha]+\mathcal{R}_0[\alpha]\big),
\ee which is bounded by a constant depending on $\mathcal{R}_0,\mathcal{R}_1,\mathcal{R}_2$ because of the improved control on $\mathcal{O}_{0,\infty}[\chibarhat]$ and $\mathcal{O}_{1,4}[\chibarhat]$ obtained in Sections \ref{sectionO14}, \ref{section0infty}. For $\betabar$, we have

\be \intubar \frac{1}{\al}\scaletwoSuubarprime{\De^2(\rho,\sigma)}\dubarprime \lesssim \frac{1}{\al}\mathcal{R}_2[\rho,\sigma]\lesssim \frac{R}{\al}\lesssim 1. \ee The last difficult term is \be \intubar \scaletwoSuubarprime{\frac{1}{\modu}\De^2\Te_{4}}\dubarprime \lesssim \frac{\modu}{a}\scaletwoHu{\De^2  \Te_4}= \Ve_{2}[\Te_{4}]. \ee
\noindent It is easy to see that all other remaining terms are $\lesssim 1$. \end{proof}
\noindent We continue with an estimate for $\mathcal{D}K$.

\begin{proposition}
Under the assumptions of Theorem \ref{mainone} and the bootstrap assumptions \ref{boundsbootstrap}, there holds \[  \scaletwoSu{\De (K- \frac{1}{\modu^2})}\lesssim \Ve_1[\Te_{44}]+\Ve_2[\Te_{44}]+\Ve_{2}[\Te_{34}]+1.  \]
\end{proposition}
\begin{proof} Similarly to Proposition \ref{propositionK-0}, we begin by making the bootstrap assumption \be \label{bootstrapDK} \frac{a}{\modu} \scaletwoSu{\De K}\lesssim B_1, \ee where $B_1$ is a large constant.
There holds 

\begin{align}
\snabla_4 \De K =& \frac{1}{\al}\De^2 \beta + (\eta,\etabar)\De \beta + \beta \De(\eta,\etabar) + \frac{1}{\al}(\chihat \De^2 \etabar + \De \chihat \De \etabar) +\etabar^2 \De (\chihat, \tr\chi) +2(\chihat,\tr\chi) \hsp \etabar \De \etabar \notag \\+& \frac{1}{\al}(\tr\chi \De^2\etabar + \De \tr\chi \De \etabar) +\chihat\De \Te + \Te \De\chihat + \tr\chi \De\Te_{34}+\Te_{34}\De \tr\chi +\tr\chibar \De \Te_{44}+ \Te_{44}\De \tr\chibar \notag \\ +& \De\big(\frac{1}{\modu}\De \Te_{44}\big) +\eta \De \Te_4 +\Te_4 \De \eta +\omega \De \Te_{34}+\Te_{34}\De \omega + \De^2 \Te_{34}+\tr\chi \De K +K \De \tr\chi \notag \\ +& \frac{1}{\al}\modu (\eta,\etabar)\De K+ \modu (\eta,\etabar)^2 K +\modu \sigma K+\omega \De K +\frac{1}{\al}(\chihat,\tr\chi)\De K + (\chihat,\tr\chi)\etabar K\notag \\+&\al(\beta,\Te_{4}) K :=F.
\end{align}Consequently, 

\be \frac{a}{\modu}\scaletwoSu{\De K}\lesssim \frac{a}{\modu}\lVert \De K \rVert_{\mathcal{L}^2_{(sc)}(S_{u,0})} +\frac{a}{\modu}\intubar \scaletwoSuubarprime{F}\dubarprime. \ee It is easy to check that $\De K =\frac{2}{\modu^2 \Omega^2}$ when $\De=\modu \snabla_3$ and $\De K=0$ otherwise, whence \[\frac{a}{\modu} \lVert \De K \rVert_{\mathcal{L}^2_{(sc)}(S_{u,0})} \lesssim 1.\]Using the results of Sections \ref{sectionO14}, \ref{section0infty} and \ref{section40curv} together with the bootstrap assumption \eqref{bootstrapDK} it follows that 

\begin{align}
\frac{a}{\modu}\intubar &\big(\frac{1}{\al}\scaletwoSuubarprime{\modu(\eta,\etabar)\De K}+ \scaletwoSuubarprime{\modu (\eta,\etabar)^2 K} +\scaletwoSuubarprime{\modu \sigma K}+\scaletwoSuubarprime{\omega \De K}\notag \\ +&\frac{1}{\al}\scaletwoSuubarprime{(\chihat,tr\chi)\De K} +\scaletwoSuubarprime{(\chihat,\tr\chi)\etabar K}+\al \scaletwoSuubarprime{(\beta,\Te_4) K}\big)\dubarprime \lesssim 1. \end{align}
For the rest of the terms, we have:

\be \frac{a}{\modu}\intubar \frac{1}{\al}\scaletwoSuubarprime{\De^2 \beta} \dubarprime \lesssim \frac{\al}{\modu}\scaletwoHu{\De^2 \beta}\lesssim \frac{\al}{\modu}\mathcal{R}_2[\beta]\lesssim \frac{\al R}{\modu}\lesssim 1. \ee The most dangerous terms are:

\be \frac{a}{\modu}\intubar \scaletwoSuubarprime{\frac{1}{\modu}\De^2 \Te_{44}}\dubarprime \lesssim \scaletwoHu{\De^2 \Te_{44}}\lesssim \frac{a}{\modu}\Ve_2[\Te_{44}], \ee
\begin{align} &\frac{a}{\modu} \intubar \scaletwoSuubarprime{\frac{1}{\modu \Omega^2} \De \Te_{44}}\dubarprime \notag \\ \lesssim& \frac{a}{\modu} \intubar \big(\scaletwoSuubarprime{\frac{1}{\modu}\Omegaterm \De\Te_{44}} + \scaletwoSuubarprime{\frac{1}{\modu}\De\Te_{44}}\big)\dubarprime \lesssim \scaletwoHu{\De \Te_{44}}\notag \\ \lesssim& \frac{a}{\modu}\Ve_1[\Te_{44}]\lesssim 1.\end{align}

\be \intubar \frac{a}{\modu}\scaletwoSuubarprime{\De^2 \Te_{34}}\dubarprime \lesssim \frac{a}{\modu}\scaletwoHu{\De^2 \Te_{34}}\lesssim \frac{a}{\modu} \Ve_2[\Te_{34}], \ee

\be \frac{a}{\modu} \intubar \scaletwoSuubarprime{\tr\chibar \De \Te_{44}}\dubarprime \lesssim \scaletwoHu{\De \Te_{44}}\lesssim \frac{a}{\modu} \Ve_1[\De\Te_{44}], \ee

\be \intubar \scaletwoSuubarprime{\Te_{44}\De\tr\chibar} \lesssim \frac{\modu}{a}\scaletwoSu{\Te_{44}}\lesssim 1. \ee The rest of the terms are bounded above by $1$ using the bootstrap assumptions. For $\De \big(K-\frac{1}{\modu^2}\big)$

\begin{align}
\snabla_4\De\big( K-\frac{1}{\modu^2} \big) =& \frac{1}{\al}\De^2 \beta + (\eta,\etabar)\De \beta + \beta \De(\eta,\etabar) + \frac{1}{\al}(\chihat \De^2 \etabar + \De \chihat \De \etabar) +\etabar^2 \De (\chihat, \tr\chi) +2(\chihat,\tr\chi) \hsp \etabar \De \etabar \notag \\+& \frac{1}{\al}(\tr\chi \De^2\etabar + \De \tr\chi \De \etabar) +\chihat\De \Te + \Te \De\chihat + \tr\chi \De\Te_{34}+\Te_{34}\De \tr\chi +\tr\chibar \De \Te_{44}+ \Te_{44}\De \tr\chibar \notag \\ +& \De\big(\frac{1}{\modu}\De \Te_{44}\big) +\eta \De \Te_4 +\Te_4 \De \eta +\omega \De \Te_{34}+\Te_{34}\De \omega + \De^2 \Te_{34}+\tr\chi \De K +K \De \tr\chi \notag \\ +& \frac{1}{\al}\modu (\eta,\etabar)\De \big(K-\frac{1}{\modu^2}\big)+ \modu (\eta,\etabar)^2 \big(K-\frac{1}{\modu^2}\big) +\modu \sigma \big(K-\frac{1}{\modu^2}\big)+\omega \De\big(K-\frac{1}{\modu^2}\big) \notag \\ +&\frac{1}{\al}(\chihat,\tr\chi)\De \big(K-\frac{1}{\modu^2}\big) + (\chihat,\tr\chi)\etabar \big(K-\frac{1}{\modu^2}\big)+\al(\beta,\Te_{4})\big(K-\frac{1}{\modu^2}\big).
\end{align}The process for bounding $\scaletwoSu{\De\big(K-\frac{1}{\modu^2}\big)}$ is exactly the same as for $\scaletwoSu{\De K}$. The only difference is that, since $K-\frac{1}{\modu^2}\equiv 0$ on the entire $\Hbar_0$, we have $\lVert K-\frac{1}{\modu^2}\rVert_{\mathcal{L}^2_{(sc)}(S_{u,0})} =0$ and hence there is no longer a need for an $\frac{a}{\modu}$-term in front of our bootstrap assumption. Instead, we begin with the bootstrap assumption 
\be \label{bootstrapK-1} \scaletwoSu{\De\big(K-\frac{1}{\modu^2}\big)}\leq \tilde{B}_1. \ee This concludes the proof.
\end{proof}
\begin{remark}
    Essentially, however, we are only going to need control only on the $(\al\snabla)-$derivatives of $K-\frac{1}{\modu^2}$. 
\end{remark} Let us recall, at this point, the definitions of divergence and curl for a symmetric, covariant tensor $\phi$ of arbitrary rank:

\[ (\slashed{\div} \phi)_{A_{1}\dots A_r}   = \snabla^B \phi_{BA_1\dots A_r},   \]\[ (\slashed{\curl} \phi)_{A_1\dots A_r}= \slashed{\epsilon}^{BC}\snabla_B \phi_{CA_1\dots A_r}.\]Let us, finally, recall the definition for the trace:

\[ (\tr\hsp \phi)_{A_1\dots A_{r-1}} = \gslash^{BC}\phi_{BCA_1\dots A_{r-1}}.       \]We are ready to state the main elliptic estimate for Hodge systems.

\begin{lemma}[Main Elliptic Estimate for Hodge Systems] \label{mainellipticlemma} We continue to work under the assumptions of Theorem \ref{mainone} and the bootstrap assumptions \eqref{boundsbootstrap}. Let $\phi$ be a totally symmetric, $(r+1)-$covariant tensor field on a metric $2-$sphere $(
\mathbb{S}^2,\gamma)$ satisfying \[ \slashed{\div} \phi = f , \hspace{3mm} 
\slashed{\curl} \phi = g, \hspace{3mm} \tr\hsp  \phi = h. \]Then, there holds

\[  \scaletwoSu{(\al\snabla)^3 \phi}\lesssim  \al \sum_{j=0}^{2} \scaletwoSu{(\al \snabla)^j(f,g)}+  \sum_{j=0}^{2}\scaletwoSu{(\al \snabla)^j (\phi,h)}. \]
\end{lemma}
\begin{proof}
Recall the following identity from Chapter 7 in \cite{C09}, that for $\phi, f,g$ and $h$ as above, there holds 

\be \label{ellipticidentity} \int_{S_{u,\ubar}} \big( \lvert \snabla \phi \rvert^2 + (r+1)K \lvert \phi \rvert^2\big) \text{d} \mu_{\gslash} = \int_{S_{u,\ubar})} \big( \lvert f \rvert^2 + \lvert g \rvert^2 + r \hsp  K \hsp \lvert h \rvert^2 \big)  \text{d} \mu_{\gslash}.    \ee To obtain the estimate, we proceed iteratively. Recall, from \cite{C09}, that the symmetrized angular derivative of $\phi$, defined by 

\[  (\snabla\phi)^s_{BA_1\dots A_{r+1}} := \frac{1}{r+2}\big(\snabla_B \phi_{A_1\dots A_r}+\sum_{i=1}^{r+1}\snabla_{A_i}\phi_{A_1\dots <A_i>B\dots A_{r+1}}\big),   \]satisfies the div-curl system

\[ \slashed{\div} (\snabla \phi)^s = (\snabla f)^s - \frac{1}{r+2}(\Hodge{\snabla}g)^s +(r+1)K\phi - \frac{2K}{r+1}(\gslash\otimes^s h),    \]\[ \slashed{\curl} (\snabla \phi)^s = \frac{r+1}{r+2}(\snabla g)^s +(r+1)\hsp K(\Hodge{\phi})^s,   \]\[\tr (\snabla\phi)^s = \frac{2}{r+2}f +\frac{r}{r+2}(\snabla h)^s,   \]where \[ (\gslash \otimes^s h)_{A_1\dots A_{r+1}}= \gslash_{A_iA_j} \sum_{i<j=1,\dots,r+1}h_{A_1\dots <A_i>\dots<A_j>\dots A_{r+1}}
 \]and \[ (\Hodge{\phi})^s_{A_1\dots A_{r+1}} := \frac{1}{r+1}\sum_{i=1}^{r+1}{\slashed{\epsilon}_{A_i}}^B \phi_{A_1\dots<A_i>B\dots A_{r+1}}.    \]Using the main elliptic identity \eqref{ellipticidentity}, we obtain
 \begin{align} \twoSu{\snabla^2 \phi}^2 \lesssim& \twoSu{\snabla f}^2+\twoSu{\snabla g}^2 + \oneSu{K(\lvert \snabla \phi\rvert^2 +\lvert f \rvert^2 +\lvert \snabla h \rvert^2}\notag \\ +&\twoSu{K\phi}^2+\twoSu{Kh}^2.  \end{align}Iterating this procedure and using \eqref{ellipticidentity}, we hence have 

 \begin{align}
    \twoSu{\snabla^3 \phi}^2\lesssim& \twoSu{\snabla^2 f}^2+\twoSu{\snabla^2 g}^2 +\twoSu{(\phi,h)\snabla K}^2 +\twoSu{K(\snabla \phi,\snabla h)}^2 \notag \\ +& \twoSu{Kf}^2 +\oneSu{K(\snabla f)^2}+\oneSu{K(\snabla g)^2}+\oneSu{K(\snabla^2\phi)^2} \notag \\ +&\oneSu{K(\snabla^2 h)^2}+\oneSu{K^3(\phi^2,h^2)}.
 \end{align}By decomposing $K=K-\frac{1}{\modu^2}+\frac{1}{\modu^2}$ and using H\"older's inequality, we obtain

\begin{align}
\twoSu{\snabla^3 \phi}^2\lesssim& \twoSu{\snabla^2 f}^2 +\twoSu{\snabla^2 g}^2 +\frac{1}{\modu^4}\twoSu{\snabla \phi,\snabla h}^2 \notag \\ +& \frac{1}{\modu^2}\twoSu{\snabla^2 \phi, \snabla f,\snabla g,\snabla^2 h}^2 + \frac{1}{\modu^6}\twoSu{\phi,h}^2 +\frac{1}{\modu^4}\twoSu{f}^2 \notag \\ +& \twoSu{\snabla(K-\frac{1}{\modu^2})}^2\inftySu{(\phi,h)}^2 +\fourSu{K-\frac{1}{\modu^2}}^2 \fourSu{\snabla \phi,\snabla h,f}^2 \notag \\ +& \modu^{\frac{1}{2}}\twoSu{K-\frac{1}{\modu^2}}\twoSu{\snabla^2\phi,\snabla f,\snabla g,\snabla^2 h}^2 \notag \\+& \modu^{\frac{1}{2}}\fourSu{K-\frac{1}{\modu^2}}^3 \inftySu{\phi,h}^2,
\end{align}which gives the desired result after translation to scale-invariant norms.

\end{proof}
\begin{proposition}
There exists a constant $C(\mathcal{R},\mathcal{V})$, depending only on the curvature and energy norms, such that there holds

\[  \scaletwoHu{\De^2 \snabla \tr\chi}\lesssim C(\mathcal{R},\Ve) , \hsp \frac{1}{\al}\scaletwoHu{\De^2\snabla \chihat}\lesssim C(\mathcal{R},\Ve).   \]
\end{proposition}

\begin{proof}We begin by making the auxiliary bootstrap assumptions

\be \label{elliptictrchichihatbootstrap} \scaletwoHu{\De^2 \snabla \tr\chi} \leq O_3, \hsp \frac{1}{\al} \scaletwoHbaru{\De^2\snabla\chihat}\leq O_3, \ee where $O_3\leq a^{\frac{1}{320}}$.
Recall the transport equation for $\tr\chi$:

\[ \snabla_4 \tr\chi =-\frac{1}{2}(\tr\chi)^2 -\lvert \chihat \rvert^2-\Te_{44}    \]Commuting this equation with $\snabla$ and then with $\De^i$, $i=0,1,2$ yields 

\begin{align} \label{disnablatrchi}
\snabla_4 \De^i\snabla \tr\chi =& \sumif \dione\psi_g^{i_2} \dit (\tr\chi,\chihat)\dif\snabla(\tr\chi,\chihat)\notag \\ +&\sumit \dione\psi_g^{i_2}\dit\snabla\Te_{44} \notag \\ +& \sumifi \dione\psi_g^{i_2}\dit(\chihat,\tr\chi)\dif\etabar \difi\tr\chi \notag \\ +&\sumif \dione\psi_g^{i_2}\dit (\beta,\Te_{4})\dif\tr\chi \notag \\+&\frac{1}{\al} \sumifim\dione \psi_g^{i_2} \dit\modu\dif(\eta,\etabar)\difi(\al\snabla)\snabla \tr\chi \notag \\ +&\sumifm \dione \psi_g^{i_2}\dit(\tr\chi,\chihat)\dif(\al\snabla)\tr\chi \notag \\+& \al \sumifim \dione \psi_g^{i_2} \dit (\chihat,\tr\chi)\dif \etabar \difi \snabla \tr\chi \notag \\ +& \al \sumifm \dione \psi_g^{i_2}\dit(\beta,\Te_4)\dif\snabla \tr\chi \notag \\+& \sumifim \dione \psi_g^{i_2}\dit \modu \dif \sigma \difi \snabla\tr\chi \notag \\ +& \sum_{i_1+i_2+i_3+i_4+i_5+i_6=i-1} \dione\psi_g^{i_2}\dit \modu \dif(\eta,\etabar)\difi (\eta,\etabar)\mathcal{D}^{i_6}\snabla \tr\chi \notag \\ :=& T_1+\dots+T_{10}.
\end{align}We therefore have, for $i=0,1,2$:

\be \scaletwoSu{\De^i\snabla \tr\chi} \lesssim \sum_{j=0}^{10}\intubar \scaletwoSuubarprime{T_j}\dubarprime. \ee
Of the terms on the right-hand side, the most borderline ones are essentially the ones with top-order derivatives. Beginning with the first one, we have:

\begin{align}
&\intubar \sumif \scaletwoSuubarprime{\dione (\psi_g,\modu\omegabar)^{i_2}\dit(\tr\chi,\chihat)\dif\snabla(\tr\chi,\chihat)} \dubarprime \notag \\ \lesssim& \intubar \big(\scaletwoSuubarprime{\De^2(\tr\chi,\chihat)\snabla(\tr\chi,\chihat)} +\scaletwoSuubarprime{(\tr\chi,\chihat)\De^2\snabla(\tr\chi,\chihat)}\big)\dubarprime \notag \\ +&\intubar \big(\scaletwoSuubarprime{\De\snabla(\tr\chi,\chihat)\De(\tr\chi,\chihat)} +\scaletwoSuubarprime{(\psi_g,\modu\omegabar)(\tr\chi,\chihat)\De\snabla(\tr\chi,\chihat)}\big)\dubarprime \notag \\ +&\intubar \big(\scaletwoSuubarprime{(\psi_g,\modu\omegabar)\De(\tr\chi,\chihat)\snabla(\tr\chi,\chihat)} +\scaletwoSuubarprime{\De(\psi_g,\modu\omegabar)(\tr\chi,\chihat)\snabla(\tr\chi,\chihat)}\big)\dubarprime \notag  \\ +&\intubar\big(\scaletwoSuubarprime{(\psi_g,\modu\omegabar)(\tr\chi,\chihat)\snabla(\tr\chi,\chihat)} + \scaletwoSuubarprime{\De(\tr\chi,\chihat)\snabla(\tr\chi,\chihat)}\big)\dubarprime \notag \\ +&\intubar\big(\scaletwoSuubarprime{(\tr\chi,\chihat)\De\snabla(\tr\chi,\chihat)} +\scaletwoSuubarprime{(\tr\chi,\chihat)\snabla(\tr\chi,\chihat)}\big)\dubarprime. \label{sevenfortyseven}
\end{align}Here we have simply conditioned on whether $i=0,1,2$. For the first term herein, there holds:

\begin{align}
&\intubar \scaletwoSuubarprime{\De^2(\tr\chi,\chihat)\snabla(\tr\chi,\chihat)}\dubarprime \notag \\\lesssim& \frac{1}{\modu}\sup_{0\leq \ubar^{\prime\prime}\leq 1}\lVert \De^2(\tr\chi,\chihat)\rVert_{\mathcal{L}^2_{(sc)}(S_{u,\ubar^{\prime\prime}})} \intubar \scaleinftySuubarprime{\snabla(\tr\chi,\chihat)} \dubarprime \notag \\ \lesssim& \frac{1}{\modu}\sup_{0\leq \ubar^{\prime\prime}\leq 1}\lVert \De^2(\tr\chi,\chihat)\rVert_{\mathcal{L}^2_{(sc)}(S_{u,\ubar^{\prime\prime}})}\sum_{j=0}^2\intubar\scaletwoSuubarprime{(\al\snabla)^j\snabla(\tr\chi,\chihat)}\dubarprime \notag \\  \lesssim& \frac{\al}{\modu}(\mathcal{R}_2[\alpha]+1)(\mathcal{R}_2[\alpha]+1+O_3) + \frac{a}{\modu}(\mathcal{R}_2[\alpha]+1)\big(\frac{1}{\al}\intubar \scaletwoSuubarprime{a\snabla^3 \chihat} \dubarprime\big). 
\end{align}Given the $\mathcal{L}^2_{(sc)}$ bounds on up to two derivatives of Ricci coefficients obtained in the previous Sections, we have 

\begin{align} \intubar \scaletwoSuubarprime{\De^2(\tr\chi,\chihat)\snabla(\tr\chi,\chihat)}\dubarprime \lesssim& \frac{\al (R+V)}{\modu} \intubar \scaletwoSuubarprime{a\snabla^3 \tr\chi}\dubarprime \notag \\ +&\frac{a C(\mathcal{R},\mathcal{V})}{\modu}\big(\frac{1}{\al}\intubar\scaletwoSuubarprime{a\snabla^3 \chihat}\dubarprime\big)\notag \\ \lesssim& \frac{a C(\mathcal{R},\mathcal{V})}{\modu}\big(\frac{1}{\al}\intubar\scaletwoSuubarprime{a\snabla^3 \tr\chi}\dubarprime\big)+1. \end{align}
For the second term, there holds:

\begin{align}
&\intubar \scaletwoSuubarprime{ (\tr\chi,\chihat)\De^2\snabla(\tr\chi,\chihat)}\dubarprime \notag \\ \lesssim& \frac{a C(\mathcal{R},\mathcal{V})}{\modu}\big(\frac{1}{\al}\intubar \scaletwoSuubarprime{\De^2\snabla \chihat}\dubarprime\big) +\frac{\al C(\mathcal{R},\mathcal{V})}{\modu}\big(\intubar \scaletwoSuubarprime{\De^2\snabla \tr\chi}\dubarprime\big)+1\notag  \\ \lesssim& \frac{a C(\mathcal{R},\mathcal{V})}{\modu}\big(\frac{1}{\al}\intubar \scaletwoSuubarprime{\De^2\snabla \chihat}\dubarprime
\big)+1, \label{sevenfifty}
\end{align}where we have used the bootstrap assumptions \eqref{boundsbootstrap}, \eqref{elliptictrchichihatbootstrap} and the results of the previous sections. The rest of the terms in \eqref{sevenfortyseven} are controlled above by $1$, using the bootstrap assumptions \eqref{boundsbootstrap} and the results of Sections \ref{sectionO04},\ref{sectionO14}, \ref{section0infty} and \ref{sectionO22}. We finally arrive at

\begin{align}&\sum_{i=0}^2\intubar \sumif  
\scaletwoSu{\De^i (\modu\omegabar)^{i_2}\dit(\tr\chi,\chihat)\dif\snabla(\tr\chi,\chihat)} \dubarprime \notag \\ \lesssim& \frac{a}{\modu}(\mathcal{R}_2[\alpha]+1)\big(\frac{1}{\al}\intubar \scaletwoSuubarprime{a\snabla^3 \chihat}\dubarprime\big)+1. \label{sevenfiftyone}
\end{align} For the second term in \eqref{disnablatrchi}, we have

\be \sum_{i=0}^2 \sumit \intubar \scaletwoSuubarprime{\db \snabla\Te_{44}}\lesssim 1, \ee where we have used the bootstrap assumptions \eqref{boundsbootstrap} involving the energy norms for Vlasov.  The rest of the terms in \eqref{disnablatrchi} are similarly bounded above by 1, using the bootstrap assumptions \eqref{boundsbootstrap} and the results of Sections \ref{sectionO04}-\ref{section0infty} and \ref{sectionO22}. We therefore arrive at the important inequality

\be \intubar \scaletwoSuubarprime{\De^2\snabla\tr\chi}\dubarprime \lesssim \frac{a}{\modu}(\mathcal{R}_2[\alpha]+1)\big(\frac{1}{\al}\intubar \scaletwoSuubarprime{a\snabla^3 \chihat}\dubarprime\big)+1. \label{ellipticchihattrchiintermediateinequality} \ee We are, thus, inclined to obtain control on $\frac{1}{\al}\intubar \scaletwoSuubarprime{\De^2\snabla\chihat}\dubarprime.$ Recall the elliptic equation for $\chihat$:

\be \sdiv \chihat = \frac{1}{2}\snabla \tr\chi -\chihat\cdot (\eta-\etabar)-\frac{1}{2}\tr\chi \gslash\cdot (\eta-\etabar)-\beta-\frac{1}{2}\Te_{4}. \ee Using the elliptic Proposition \ref{mainellipticlemma}, we obtain

\begin{align}
\scaletwoSuubarprime{a\snabla^3 \chihat}\lesssim& \frac{1}{\al}\sum_{0\leq i\leq 3}\scaletwoSuubarprime{(\al\snabla)^3\tr\chi} + \sum_{0\leq i \leq 2}\scaletwoSuubarprime{(\al\snabla)^i\beta} \notag \\ +& \sum_{0\leq i \leq 2}\scaletwoSuubarprime{(\al\snabla)^i\Te_4} +\sum_{\substack{0\leq i\leq 2\\ i_1+i_2=i}}\scaletwoSuubarprime{a^{\frac{i}{2}}\snabla^{i_1}(\eta,\etabar)\snabla^{i_2}\chihat } \notag \\ +& \sum_{0\leq i \leq 2}\frac{1}{\al}\scaletwoSuubarprime{(\al\snabla)^i\chihat}. \label{chihatinbetween}
\end{align} Integrating along $H_u$ and dividing by $\al$, we have 

\be \frac{1}{\al}\intubar \scaletwoSuubarprime{a\snabla^3\chihat}\dubarprime \lesssim \frac{1}{\al}\intubar \scaletwoSuubarprime{a\snabla^3 \tr\chi}\dubarprime +1. \label{sevenfiftyfour} \ee Plugging this back to 
\eqref{sevenfiftyone} and then again back to \eqref{sevenfiftyfour}, we obtain the result in the case $\De^2= (\al\snabla)^2$. If $\snabla_4$ appears in $\De^2$, then we can use the structure equation for $\chihat$ to reduce the number of derivatives by $1$. This only produces terms that have been estimated at previous steps. Similarly, if $\De^2 = (\al\snabla)\modu\snabla_3$, then we use the commutation formula to estimate 

\begin{align} &\intubar \scaletwoSuubarprime{(\al\snabla)\modu\snabla_3 \snabla \chihat}\dubarprime \lesssim \intubar \big(\scaletwoSuubarprime{\al\snabla^2\modu\snabla_3 \chihat} + \scaletwoSuubarprime{\al\snabla \modu [\snabla_3,\snabla]\chihat}\big) \dubarprime \notag\\\lesssim& \intubar \scaletwoSuubarprime{\al\snabla^2(\modu\snabla_3)\chihat}\dubarprime +\text{terms that have been estimated at previous steps}.  
\end{align}To estimate $\intubar \scaletwoSuubarprime{\al\snabla^2(\modu\snabla_3)\chihat}\dubarprime$, we use the commutation formula for $[ \snabla_3,\sdiv]$ along with Proposition \ref{mainellipticlemma}:

\begin{align} &\sdiv(\modu\snabla_3)\chihat \notag \\ =& \modu\snabla_3 \sdiv \chihat +\modu [\snabla_3,\sdiv]\chihat = \modu \snabla_3 \big( \frac{1}{2}\snabla \tr\chi -\chihat\cdot (\eta-\etabar)-\frac{1}{2}\tr\chi \gslash\cdot (\eta-\etabar)-\beta-\frac{1}{2}\Te_{4}\big) \notag \\ +& \modu\big(  -\frac{1}{2}\tr\chibar \hsp \sdiv \chihat -\chibarhat \cdot \snabla \chihat -(\betabar+\frac{1}{2}\Te_3) \chihat +\frac{1}{2}(\eta+\etabar)\snabla_3 \chihat -(\chihat+\frac{1}{2}\tr\chi) \hsp \etabar \hsp \chihat + (\chibarhat+\frac{1}{2}\tr\chibar)\hsp \eta \hsp \chihat \big).\end{align} Hence

\begin{align}
\intubar \scaletwoSuubarprime{\al\snabla^2(\modu\snabla_3)\chihat}\lesssim& \sum_{j=0}^1 \intubar \scaletwoSuubarprime{(\al\snabla)^j(\modu\snabla_3)\snabla \tr\chi} \dubarprime \notag \\ +&\frac{1}{\al}\sum_{j=0}^1\intubar \scaletwoSuubarprime{(\al\snabla)^j(\modu\snabla_3)\chihat}\notag + \text{lower order terms}.
\end{align}Putting everything together, for $\mathcal{D}^2=(\al\snabla)\modu\snabla_3$, we arrive at

\be \frac{1}{\al}\intubar \scaletwoSuubarprime{\De^2\snabla\chihat}\dubarprime \lesssim \frac{1}{\al}\intubar \scaletwoSuubarprime{\De^2 \snabla \tr\chi}\dubarprime +1. \ee Proceeding in the same way, we can show that the above inequality holds for any choice of $\De^2=\De_1\De_2$. We therefore obtain that there exists a constant $C(\mathcal{R},\Ve)$ such that 

\be \frac{1}{\al}\intubar \scaletwoSuubarprime{\De^2 \snabla\chihat}\dubarprime \lesssim C(\mathcal{R},\Ve). \ee Putting this back to \eqref{ellipticchihattrchiintermediateinequality}, we arrive at the result.
\end{proof}
\noindent Elliptic control on $\chihat, \tr\chi$, in turn, allows to obtain control on the elliptic norm of $\eta$. This is the content of the next Proposition.

\begin{proposition}
Under the assumptions of Theorem \ref{mainone} and the bootstrap bounds \eqref{boundsbootstrap}, there exists a constant $C$ depending only the curvature norms $\mathcal{R},\underline{\mathcal{R}}$, such that there holds 

\[ \scaletwoHu{\De^2 \snabla \eta}\lesssim C(\mathcal{R},\underline{\mathcal{R}}). \]

\end{proposition}

\begin{proof}
Let us introduce the quantity 

\[\mu := - \slashed{\div}\eta -\rho + \frac{1}{2}\chihat\cdot \chibarhat +\frac{1}{4}\tr\chi \tildetr.  \]
We wish to obtain a $\snabla_4$ evolution equation for $\mu$. 
For a $1-$form U, recall the commutation formula (see for example \cite{AnAth}):

\be [\snabla_4,\slashed{\div}]U= -\frac{1}{2}\tr\chi \hsp  \slashed{\div}U -\chihat \cdot \snabla U -(\beta-\frac{1}{2}\Te_{4})\cdot U +\frac{1}{2}(\eta+\etabar)\snabla_4 U -\etabar \hsp \chibarhat\hsp  U +\frac{1}{2}\tr\chi\hsp\etabar \hsp U. \ee

We explicitly compute:

\begin{align} \snabla_4 \sdiv \eta =&\sdiv(\snabla_4 \eta)+[\snabla_4,\sdiv]\eta = \sdiv\big(\frac{1}{2}\tr\chi(\etabar-\eta)+\chihat\cdot(\etabar-\eta)-\beta-\frac{1}{2}\Te_{4}\big)   \notag \\ -&\frac{1}{2}\tr\chi \hsp \sdiv \eta -\chihat \cdot \snabla \eta -(\beta-\frac{1}{2}\Te_4)\eta +\frac{1}{2}(\eta+\etabar)\snabla_4 \eta - \etabar \hsp \chibarhat \hsp \eta +\frac{1}{2}\tr\chi \hsp \etabar \hsp \eta.  \label{muintermediate1} \end{align}Moreover, 
\be \snabla_4 \rho =- \frac{3}{2}\tr\chi \rho + \slashed{\div}\beta -\frac{1}{2}\chibarhat^{\sharp}\cdot \alpha + (\etabar,\beta) - \frac{1}{4}\big(\snabla_3 \Te_{44} -\snabla_4 \Te_{34} +2\omegabar\hsp \Te_{44} +\bcancel{\omega \Te_{34}}+2(\etabar-\eta)\cdot \Te_4\big). \label{muintermediate2}\ee For the rest of the terms, we have

\begin{align} \snabla_4(\frac{1}{2}\chihat \cdot \chibarhat) =& \frac{1}{2}\chihat^{AB}  \big( -\frac{1}{2}\tr\chi \hsp \chibarhat_{AB} +(\snabla_{A}\etabar_B+\snabla_B \etabar_A -\sdiv \etabar \hsp \gslash_{AB}) -\frac{1}{2}\tr\chibar \chihat_{AB} +2\etabar_A \etabar_B - \lvert \etabar\rvert^2 \gslash_{AB} +\slashed{T}_{AB}\big)\notag \\ +& \frac{1}{2} \hsp\chibarhat^{AB}\big(-\tr\chi \hsp \chihat_{AB}-\alpha_{AB}, \big) \label{muintermediate3}.\end{align}
Finally,

\begin{align}
\snabla_4\big(\frac{1}{4}\tr\chi \tildetr\big) = \frac{1}{4}\tildetr \big(-\frac{1}{2} (\tr\chi)^2 - \lvert \chihat \rvert^2 -\slashed{T}_{44}\big) +\frac{1}{4} \tr\chi\big( 2 \div \etabar +2\lvert \etabar \rvert^2 + \chihat\cdot \chibarhat +\frac{1}{2}\tr\chi \tr\chibar +\rho-\frac{1}{2}\Te_{34}\big) \label{muintermediate4}
\end{align}

\noindent After a lengthy but straightforward calculation, adding $-$\eqref{muintermediate1}$-$ \eqref{muintermediate2}+\eqref{muintermediate3}+\eqref{muintermediate4}, the terms \[\sdiv \beta, -\frac{1}{2}\hsp \chibarhat^{\sharp}\cdot \alpha,\] as well as \[\frac{1}{2}\tr\chi \hsp \sdiv \etabar, \chihat \cdot \snabla \etabar\] crucially cancel out. We finally obtain

\begin{align} 
\snabla_4\mu +\tr\chi \mu=& \frac{1}{2}(\etabar-\eta)\snabla\tr\chi -2\chihat \cdot \snabla \eta +(\eta-\etabar)\hsp \sdiv \chihat -\frac{1}{2}\sdiv \Te_4 -(\beta-\frac{1}{2}\Te_4)\eta \notag \\  +&\frac{1}{2}(\eta+\etabar)\big(-\chi\cdot(\eta-\etabar)-\beta-\frac{1}{2}\Te_4\big)-\etabar \hsp \chibarhat \hsp \eta +\frac{1}{2}\tr\chi \hsp  \etabar\hsp  \eta +\frac{1}{4}\tr\chi \rho \notag \\+&(\etabar,\beta) -\frac{1}{4}\big(\snabla_3 \Te_{44}-\snabla_4 \Te_{34}+2\omegabar \hsp \Te_{44}+2(\etabar-\eta)\cdot \Te_{4}\big) -\frac{1}{4}\tr\chibar \lvert \chihat \rvert^2 \notag  \\ +& \chihat \cdot \etabar \cdot \etabar +\frac{1}{2}\chihat\cdot \Te +\frac{1}{8}\tr\chi \tildetr \tr\chi +\frac{1}{4}\tildetr (-\lvert \chihat \rvert^2-\Te_{44}) +\frac{1}{2}\tr\chi \lvert \etabar \rvert^2 \notag \\+& \frac{1}{8}\tr\chi^2 \tr\chibar -\frac{1}{8}\tr\chi \Te_{34}:=M_0.
\end{align}
By bounding

\be \scaletwoSu{\mu} \lesssim \intubar \scaletwoSuubarprime{M_0}\dubarprime , \ee using the bootstrap assumptions, we get 

\be \scaletwoSu{\mu}\lesssim 1. \ee 

\noindent Commuting with $\mathcal{D}$, we get the equation (for $i=1,2$)
\begin{align} 
\label{nabla4dim}
\snabla_4 \De^i \mu =&    \sumit \dione (\modu \omegabar)^{i_2}\dit M_0 \nonumber \\ \notag +& \frac{1}{\al} \sum_{i_1+i_2+i_3+i_4+i_5=i-1}\mathcal{D}^{i_1}\psi_g^{i_2}\mathcal{D}^{i_3}\modu \mathcal{D}^{i_4}(\eta,\etabar) \mathcal{D}^{i_5+1}\mu \\ \nonumber +&\sum_{i_1+i_2+i_3+i_4+1=i}\mathcal{D}^{i_1}\psi_g^{i_2}\mathcal{D}^{i_3}(\tr\chi,\chihat)\mathcal{D}^{i_4+1}\mu \\ \nonumber &+ \al \sumifim \dione \psi_g^{i_2} \dit(\chihat,\tr\chi)\dif \etabar \difi \mu \\ \nonumber +& \al \sumifm \dione \psi_g^{i_2}\dit (\beta,\slashed{T}_{4}) \dif \mu  \\  \nonumber+&\sum_{i_1+i_2+i_3+i_4+i_5=i-1}\mathcal{D}^{i_1}\psi_g^{i_2}\mathcal{D}^{i_3}\modu \mathcal{D}^{i_4}\sigma\mathcal{D}^{i_5}\mu \\ +& \sum_{i_1+\dots+i_6=i-1}\mathcal{D}^{i_1}\psi_g^{i_2}\mathcal{D}^{i_3}\modu \mathcal{D}^{i_4}(\eta,\etabar)\mathcal{D}^{i_5}(\eta,\etabar)\mathcal{D}^{i_6}\mu. 
\end{align} We proceed on a term by term basis. For the first term (keeping in mind that the very first operator from $\De_{i_4+1}$ acting on $\tr\chi$ is $\al\snabla$),

\begin{align}
\frac{1}{\al}\intubar   \ScaletwoSuubarprime{\sumifm \dione (\psi_g,\modu\omega)^{i_2}\dit(\eta,\etabar)\mathcal{D}^{i_4+1}\tr\chi}\dubarprime,
\end{align}we distinguish cases. If $i_4=i-1$, then we control the term by a bootstrap on the elliptic norm of $\tr\chi$ and use the improvements on $\scaleinftySu{\eta,\etabar}$ that we obtained in Section \ref{section0infty} (note that in the non-borderline cases, we can just as well use the bootstrap assumptions instead of the improvements):

\be \frac{1}{\al}\intubar \scaletwoSuubarprime{(\eta,\etabar)\De^i \tr\chi}\lesssim \frac{O_{\text{ell}}[\tr\chi]C(\mathcal{R},\underline{\mathcal{R}},\Ve, \underline{\Ve})}{\modu}\lesssim 1.
\ee If $i_4=i-2=1$, then we bound 
\begin{align}
&\frac{1}{\al}\intubar \big(\scaletwoSuubarprime{\De(\eta,\etabar)\De^2\tr\chi}+\scaletwoSuubarprime{(\psi_g,\modu\omegabar)(\eta,\etabar)\De^2 \tr\chi}\big) \dubarprime \notag \\ \lesssim& \frac{1}{\al}\frac{1}{\modu}O \intubar \scalefourSuubarprime{\De^2\tr\chi} \dubarprime + \frac{O^3}{\modu^2}
\notag \\  \lesssim& \frac{O}{\modu} \big( \scaletwoHu{\De\snabla \tr\chi}^{\frac{1}{2}}\scaletwoHu{\De^2 \snabla \tr\chi}^{\frac{1}{2}}+ \scaletwoHu{\De\snabla \tr\chi}\big)  +1 \notag \\ \lesssim& \frac{O}{\modu}(O^{\frac{1}{2}}O_{\text{ell}}^{\frac{1}{2}} +O)+1 \lesssim 1.\end{align} If $i_4=i-2=0,$ then we bound 

\begin{align}
&\frac{1}{\al} \intubar \big( \scaletwoSuubarprime{\De ^2(\eta,\etabar) (\al\snabla) \tr\chi} +\scaletwoSuubarprime{(\psi_g,\modu\omegabar)\De(\eta,\etabar)(\al\snabla)\tr\chi} \big)\dubarprime \notag \\ +&\frac{1}{\al}\intubar \scaletwoSuubarprime{\De(\psi_g,\modu\omegabar)(\eta,\etabar)(\al\snabla)\tr\chi)}\big) \dubarprime \lesssim 1,
\end{align}where we have bounded $\De^2(\eta,\etabar)$ in $\mathcal{L}^2_{(sc)}$ and the $(\al\snabla)\tr\chi$ in $\mathcal{L}^{\infty}_{(sc)}$ since, using Sobolev embedding, we have

\[ \frac{1}{\al}\scaleinftySuubarprime{\al\snabla \tr\chi}\lesssim \frac{1}{\al} \sum_{i=0}^2\scaletwoSuubarprime{(\al\snabla)^{1+i}\tr\chi}.        \]The terms $\scaletwoSuubarprime{\snabla \tr\chi}, \scaletwoSuubarprime{\al\snabla^2 \tr\chi}$ are controlled by estimates in the previous sections, whereas the term 
\begin{align} &\frac{1}{\modu}\intubar\scaletwoSuubarprime{\De^2(\eta,\etabar)}\scaletwoSuubarprime{a\snabla^3\tr\chi}\dubarprime \notag \\ \lesssim& \frac{1}{\modu}\sup_{0\leq \ubar^{\prime}\leq 1} \scaletwoSuubarprime{\De^2(\eta,\etabar)}\intubar \lVert a\snabla^3\tr\chi\rVert_{\mathcal{L}^2_{(sc)}(S_{u,\ubar^{\prime\prime}})}\text{d}\ubar^{\prime\prime}  \notag  \end{align}is controlled by the elliptic control we have on $\De^2 \snabla \tr\chi$. For the second term in $\db M_0$, similarly, we have

\begin{align} &\sum_{i=1}^2 \sumif \frac{1}{\al}\intubar \scaletwoSuubarprime{\db \chihat \mathcal{D}^{i_4}\snabla \eta}\dubarprime  \lesssim \intubar \scaletwoSuubarprime{\chihat \De^2 \snabla \eta}\dubarprime +1 \notag \\ \lesssim& \frac{\al O }{\modu}\intubar \scaletwoSuubarprime{\snabla \De^2 \eta}\dubarprime+1 \lesssim 1. \end{align}This term will be absorbed by Gr\"onwall's inequality in the end, so can be ignored for all purposes. Note that here there is no need to use the improved estimate from Remark \ref{chihatinfinityr} for $\scaleinftySu{\chihat}$. as the bootstrap assumption only suffices. Similarly, the rest of the terms in \[  \intubar \scaletwoSuubarprime{\db M_0}\dubarprime  \]are bounded above by $1.$  Finally, for the rest of the terms in \eqref{nabla4dim}, we see that they are also bounded above by 1 using the bootstrap assumptions. Ultimately, we obtain

\be \sum_{i=0}^2 \intubar \scaletwoSuubarprime{\De^2 \mu}\dubarprime \lesssim 1. \ee We now recall the div-curl system :

\[ \sdiv \eta = -\mu-\rho + \frac{1}{2}\chihat\cdot \chibarhat + \frac{1}{4}\tr\chi \hsp \tildetr,    \] \[  \scurl \eta = \sigma + \frac{1}{2}\chibarhat \wedge \chihat.   \]Making use of Proposition \ref{mainellipticlemma}, we thus obtain

\begin{align}
\scaletwoSu{a\snabla^3 \eta}\lesssim& \sum_{i=0}^2\big( \scaletwoSu{(\al \snabla)^i \mu} + \scaletwoSu{(\al \snabla)^i(\rho,\sigma)}+\scaletwoSu{(\al\snabla)^i(\chihat \cdot \chibarhat)} \notag \\ +&\scaletwoSu{(\al\snabla)^i(\tr\chi \hsp \tildetr)}\big) +\sum_{i\leq 1}\scaletwoSu{(\al\snabla)^i \eta}.
\end{align}Integrating along the $\ubar-$direction  and raising to the second power, we arrive at \[ \frac{a}{\modu}\scaletwoHu{a \snabla^3 \eta} \lesssim 1+\mathcal{R}.   \]  This proves the claim of the Proposition for the case $\De^2=(\al\snabla)^2$. For a general $\mathcal{D}^2,$ we obtain a new div-curl system of the form

\[ \sdiv \De^2 \eta  = \De^2(-\mu-\rho+\frac{1}{2}\chihat\cdot \chibarhat +\frac{1}{4}\tr\chi\tildetr)+[\sdiv, \De^2]\eta = \De^2(-\mu-\rho+\frac{1}{2}\chihat\cdot \chibarhat +\frac{1}{4}\tr\chi\tildetr)+\hsp \text{terms of order $\leq 2$}, \]\[ \scurl \De^2 \eta = \De^2(\sigma +\frac{1}{2}\chibarhat \wedge \chihat) +\hsp \text{terms of order $\leq 2$}.   \]As such, there holds 

\be \scaletwoSu{\snabla \De^2 \eta}\lesssim \scaletwoSu{\De^2(\mu,\rho, \sigma,\chihat\cdot \chibarhat,\tr\chi\hsp \tildetr)}+ \hsp \text{lower-order terms}. \ee Finally,

\begin{align}\notag  \scaletwoSu{\De^2\snabla \eta}\lesssim& \scaletwoSu{\snabla \De^2 \eta} + \scaletwoSu{[\snabla,\De^2]\eta}\notag \\ \lesssim& \scaletwoSu{\De^2(\mu,\rho, \sigma,\chihat\cdot \chibarhat,\tr\chi\hsp \tildetr)}+ \hsp \text{lower-order terms}. \end{align} Integrating along the $\ubar-$direction and raising to the second power, we arrive at \[ \frac{a}{\modu}\scaletwoSu{\De^2 \snabla \eta}\lesssim 1 +\mathcal{R}.   \]

\end{proof}

\section{Estimates for the Vlasov matter}
\label{vlasov}
\noindent We begin the discussion on the Vlasov stress-energy tensor components $T_{\mu\nu}$, for up to three derivatives. These components, being functions of spacetime themselves, have well-defined norms with respect to the spacetime derivatives $\mathcal{D}$. However, from the point of view of the mass shell, they are non-local objects, in that they appear integrated over the momentum support. After having established the decay rates for the momentum components involved in this integral, we need to estimate the growth of certain derivatives of $f$  on the mass shell. This essentially translates to a problem of controlling several error terms, arising from the non-commutativity of the vector fields in question with the geodesic spray $X$. \vspace{3mm}

\noindent In this section, we control these Vlasov objects. First, we define the Vlasov norms that we want to control. Before we proceed to do so, we need to define the vertical and horizontal lifts associated with the tangent bundle. As $f$ lives on $\mathcal{P}\subset T\mathcal{M}$, to estimate derivatives of $f$ we use vector fields $\in \Gamma(TT\mathcal{M})$, much like the geodesic spray. There is a natural way to "lift" vector fields from $\Gamma(T\mathcal{M})$ to $\Gamma(TT\mathcal{M})$ (or $\Gamma(TT\mathcal{P}))$:

\vspace{3mm} 
\begin{definition}Given $v\in T_x\mathcal{M},$ its \textbf{vertical lift} at $p\in T_x\mathcal{M}$, $\text{Ver}_{(x,p)}(v)$, is defined to be the vector tangent to the curve $c_{(x,p),V}:(-\varepsilon,\varepsilon)\to T\mathcal{M}$ given by \[ c_{(x,p),v}(s)=(x,p+sv),   \]at $s=0$.
\end{definition}
\vspace{3mm}

\noindent Similarly, we introduce the notion of a \textit{horizontal lift} of a vector field on $\mathcal{M}$.

\begin{definition} Let $c: (-\varepsilon, \varepsilon) \to \mathcal{M}$ be a curve on $\mathcal{M}$ with $c(0)=x, c^\prime(0)=v$. Extend $p$ to a vector field along $c$ by parallel transport using the Levi-Civita connection of $g$ on $\mathcal{M}$. The \textbf{horizontal lift} of $v$ at $(x,p)$ denoted $\text{Hor}_{(x,p)}(v)$ is then defined to be the tangent vector to the curve $c_{(x,p),H}:(-\varepsilon,\varepsilon)\to T\mathcal{M}$ given by $c_{(x,p),H}(s)=(c(s),p)$ at $s=0$,\[ \text{Hor}_{(x,p)}(v)= c_{(x,p),H}^\prime (0).\]
\end{definition}

\vspace{3mm}

\noindent Given the coordinates
 $p^1,\dots,p^4$ on $T_x\mathcal{M}$ conjugate to $e_1,\dots,e_4$, the double null frame on $\mathcal{M}$, one has a frame for $T\mathcal{M}$ given by $(e_1,\dots,e_4,p^1,\dots p^4)$. If $v\in T_x\mathcal{M}$ is written as $v=v^{\mu}e_{\mu}$, then
 \[ \text{Ver}_{(x,p)}(v)=v^{\mu}\partial_{p^{\mu}},   \]with a similar expression for $\text{Hor}_{(x,p)}(v):$ \[ \text{Hor}_{(x,p)}(v)= v^{\mu}e_{\mu}-v^{\mu}p^{\nu}\Gamma^{\lambda}_{\mu\nu}\partial_{p^{\lambda}}.\] 

 \vspace{3mm}
\noindent We note that, after restricting $X$ to $\mathcal{P}\subset T\mathcal{M}$, the Vlasov equation rewrites as \[\restri{X}{\mathcal{P}}f = p^{\mu}e^{\mu}(f)-\Gamma^{\hat{\lambda}}_{\mu\nu}p^{\mu}p^{\nu}\partial_{\pbar^{\hat{\lambda}}}f=0.\]Here $\hat{\lambda}$ runs over $1,2,3$ and $\pbar^1,\pbar^2,\pbar^3$ denote the restrictions of $p^1,p^2,p^3$ to $\mathcal{P}$, while $\partial_{\pbar^1},\partial_{\pbar^2},\partial_{\pbar^3}$ denote the partial derivatives with respect to the restricted coordinate system $(\pbar^1,\pbar^2,\pbar^3)$. It is easy to see that \[ \partial_{\pbar^A} = \partial_{p^A}+ \frac{\slashed{g}_{AB}p^B}{2p^3}\partial_{p^4}, \hspace{2mm} \partial_{\pbar^3}=\partial_{p^3}-\frac{p^4}{p^3}\partial_{p^4}.    \]} Finally, notice that 

\be X_{(x,p)}= Hor_{(x,p)}(p). \ee
Recall the vector fields $V_1,\dots V_6,$
\begin{gather}
V_{(A)}:= Hor(e_A)- \frac{p^3}{\lvert u \rvert} \partial_{\pbar^A}, \hsp V_{(3)} := Hor(e_4), \hspace{3mm} V_{(4)} := p^3 \partial_{\pbar^3} - \lvert u \rvert Hor(e_3), \hspace{3mm} V_{(4+A)}:= \frac{p^3}{\lvert u \rvert^2}\partial_{\pbar^A},
\end{gather}together with the vector field \be V_{(0)}:= \lvert u \rvert Hor(e_3). \ee
\noindent 
We further introduce a frame $F$ on the mass shell $\mathcal{P}$ given by 
\begin{align}
\label{eq:frame}
F_A=e_A, \hsp\hsp F_3=\modu \partial_u, \hsp \hsp F_4=\partial_v, F_{4+A}=\frac{p^3}{\modu}\partial_{\pbar^A}, F_7=p^3 \partial_{\pbar^3}. \end{align}This will be the frame in which the commutators will be calculated. Introduce, for convenience, the set \be \label{tildeV}\tilde{\mathcal{V}}:=\{ V_{(A)},V_{(3)},V_{(4)}, \modu V_{(4+A)}, \modu V_{(0)}\}, \ee where $A=1,2$. Also, define for $k_{i}\in \mathbb{Z}_{\geq 0}$, the total Vlasov norms:
\begin{align}
\mathfrak{V}_{0}:=\sup_{u,\ubar}\sup_{S_{u,\ubar}}\sup_{\mathcal{P}_{x}}f, \end{align}
\begin{align} 
\mathfrak{V}_{1} := \sum_{i=1}^4\sup_{u,\ubar}\sup_{S_{u,\ubar}}\int_{\mathcal{P}_{x}}|V_{(i)}f|^{2}\sqrt{\det\gslash}\frac{dp^{1}dp^{2}dp^{3}}{p^{3}}
+\sum_{i\in\{0,5,6\}}\sup_{u,\ubar}\sup_{S_{u,\ubar}}\int_{\mathcal{P}_{x}}||u|V_{(i)}f|^{2}\sqrt{\det\gslash}\frac{dp^{1}dp^{2}dp^{3}}{p^{3}},\end{align}
\be
\mathfrak{V}_{2}:= \sum_{\tilde{V}_{(1)},\tilde{V}_{(2)} \in \tilde{\mathcal{V}}} \sup_{u,\ubar}\int_{S_{u,\ubar}} \int_{\mathcal{P}_{x}}\lvert \tilde{V}_{(1)}\tilde{V}_{(2)}f \rvert^2 \sqrt{\det \gslash}\frac{dp^1 dp^2 dp^3}{p^3}\hsp \sqrt{\det\gslash}d\theta^1d\theta^2,\ee 
\be
\mathfrak{V}_{3}:= \sum_{\tilde{V}_{(1)},\tilde{V}_{(2)}, \tilde{V}_{(3)} \in \tilde{\mathcal{V}}} \sup_{u} \int_0^1\int_{S_{u,\ubar^{\prime}}} \int_{\mathcal{P}_{x}}\lvert \tilde{V}_{(1)}\tilde{V}_{(2)}\tilde{V}_{(3)} f \rvert^2 \sqrt{\det \gslash}\frac{dp^1 dp^2 dp^3}{p^3}\hsp \sqrt{\det\gslash}d\theta^1d\theta^2\dubarprime.\ee 
\noindent Once we have controlled these norms, we can easily estimate derivatives of $T_{\mu\nu}$ by the following relations
 between\footnote{Herein by $e_{\mu}, \mu=1,2,3,4$ we mean $e_{\mu}$ extended as a vector field $\in \Gamma(TT\mathcal P)$.} the $e_{\mu}$ and the vector fields $V$ (here $A=1,2$):
  
\begin{align} e_A =& V_{(A)}+\modu^2\bigg(\frac{\Gammaslash^B_{AC}p^C}{p^3}+{\chibarhat_{A}}^B +\frac{{\chi_A}^B p^4}{p^3}\bigg)V_{(4+B)} \notag \\ +& \frac{\modu^2}{2}\tildetr V_{(4+A)} +\bigg(\frac{1}{2}\frac{\chi_{AB}p^B}{p^3}-\etabar_A\bigg)(V_{(4)}+V_{(0)}), \end{align}
\be \modu e_3 = \big(1+\modu \frac{\eta_A p^A}{p^3}+\modu \hsp \omegabar \big)V_{(0)}+\modu^2\bigg(\frac{({\chibar_B}^A-e_B(b^A))p^B +2\eta^A \hsp p^4}{p^3}\bigg)\modu V_{(4+A)} +\modu \hsp \omegabar \hsp  V_{4}, \ee
\be
e_4 = V_{(3)} + \modu^2\left( \frac{{\chi_B}^A p^B}{p^3} + 2\hsp \etabar^A \right)V_{(4+A)},
\ee
\be
\frac{\partial}{\partial \pbar^{A}}=\frac{|u|^{2}}{p^{3}}V_{(4+A)},~\frac{\partial}{\partial \pbar^{3}}=V_{(4)}+V_{(0)}.
\ee
 
\noindent In order to estimate these $\mathfrak{V}-$norms, we integrate the transport equations 
\begin{align}
 X[V^{(i)}f]=E_i,   
\end{align} for $i=1,2,3$, where by $V^{(i)}$ we denote the presence of $n$ $V-$vector fields and $E^{(i)}$ are the corresponding error terms.
For this, we need the following important integration lemma, which is essentially a version of the Reynolds transport theorem adapted to the mass shell.

\begin{lemma}[Transport inequality on the mass shell]
\label{lem:transport}
Assume that the bootstrap assumptions \eqref{boundsbootstrap} and \eqref{elliptictrchichihatbootstrap} hold. Let \( \mathcal{P} \subset TM \) denote the future mass shell in the tangent bundle of a Lorentzian manifold \( (\mathcal{M}, g) \), equipped with the canonical Sasaki metric \( \widetilde{g} \). Let \( \varphi: \mathcal{P} \to \mathbb{R} \) be a smooth scalar function, and consider the squared modulus \( |\varphi|^2 \) on \( \mathcal{P} \). We define the following scalar functionals on level sets of the foliation \( \{S_{u,\ubar}\} \) and on the null foliation of \( \mathcal{M} \):
\be 
A(s) := \int_{\pi^{-1}(x(s))} |\varphi|^2 \sqrt{\det \slashed{g}} \, \frac{dp^1 dp^2 dp^3}{p^3}, \label{eq:def-A}\ee \be
B(s) := \int_{S_{u,\ubar}} \left( \int_{\pi^{-1}(x(s))} |\varphi|^2 \sqrt{\det \slashed{g}} \, \frac{dp^1 dp^2 dp^3}{p^3} \right) \sqrt{\det \slashed{g}} \, dS_{u,\ubar}(x), \label{eq:def-B}\ee 
\be
C(s) := \int_0^1 \int_{S_{u,\ubar}} \left( \int_{\pi^{-1}(x(s))} |\varphi|^2 \sqrt{\det \slashed{g}} \, \frac{dp^1 dp^2 dp^3}{p^3} \right) \sqrt{\det \slashed{g}} \, dS_{u,\ubar}(x) \, d\ubar, \label{eq:def-C}
\ee

\noindent where \( \pi: \mathcal{P} \to \mathcal{M} \) denotes the basepoint projection, \( \slashed{g} \) denotes the induced Riemannian metric on the 2-spheres \( S_{u,\ubar} \), and \( dS_{u,\ubar} \) is the Riemannian measure associated to \( \slashed{g} \). The variable \( s \in \mathbb{R} \) is an affine parameter along the integral curves of the geodesic spray vector field \( X \in \Gamma(T\mathcal{P}) \), i.e., \( \gamma(s) = \exp_s(x_0, p_0) \) solves \( \frac{d\gamma}{ds} = X(\gamma(s)) \). Then the following integral inequalities hold along the geodesic flow:
\be 
A(s) \lesssim A(s_\infty) + \int_{s_\infty}^{s} \left( \int_{\pi^{-1}(x(s'))} \varphi \cdot X[\varphi] \sqrt{\det \slashed{g}} \, \frac{dp^1 dp^2 dp^3}{p^3} \right) ds', \label{eq:transport-A}\ee
\be
B(s) \lesssim B(s_\infty) + \int_{s_\infty}^{s} \int_{S_{u,\ubar}} \left( \int_{\pi^{-1}(x)} \varphi \cdot X[\varphi] \sqrt{\det \slashed{g}} \, \frac{dp^1 dp^2 dp^3}{p^3} \right) \sqrt{\det \slashed{g}} \, dS_{u,\ubar}(x) \, ds', \label{eq:transport-B}\ee \be
C(s) \lesssim C(s_\infty) + \int_{s_\infty}^{s} \int_0^1 \int_{S_{u,\ubar}} \left( \int_{\pi^{-1}(x)} \varphi \cdot X[\varphi] \sqrt{\det \slashed{g}} \, \frac{dp^1 dp^2 dp^3}{p^3} \right) \sqrt{\det \slashed{g}} \, dS_{u,\ubar}(x) \, d\ubar \, ds'. \label{eq:transport-C}
\ee

\noindent Here, again, $X$ is the geodesic spray, i.e. the Hamiltonian vector field associated with the kinetic Hamiltonian \( H(x,p) = \frac{1}{2} g_x(p,p) \) and the implied constants depend only on the dimension of the spacetime.
\end{lemma}

\begin{proof}
Let $X$ denote the generator of the geodesic flow on the mass shell $\mathcal{P}_x \subset T_x\mathcal{M}$, defined by $X = p^\mu \partial_{x^\mu} - \Gamma^{\hat{\lambda}}_{\mu\nu} p^\mu p^\nu \partial_{\pbar^{\hat{\lambda}}}$, and consider the evolution of the quadratic observable
\be
A(s) := \int_{\mathcal{P}_{x}} |\varphi|^2 \sqrt{\det \gslash} \frac{dp^1 dp^2 dp^3}{p^3},
\ee
along the integral curves of $X$, where $\varphi$ is compactly supported in $\mathcal{P}_{x}$ and smooth. By applying the transport theorem on the evolving hypersurface $\mathcal{P}_{x}$ embedded in $T_xM$, we compute the total derivative of $A(s)$ along the flow:
\begin{align}
\frac{d}{ds} A(s)
&= \int_{\mathcal{P}_{x}} X(|\varphi|^2) \sqrt{\det \gslash} \frac{dp^1 dp^2 dp^3}{p^3}
+ \int_{\mathcal{P}_{x}} |\varphi|^2 \frac{1}{\mu} X[\mu] \sqrt{\det \gslash} \frac{dp^1 dp^2 dp^3}{p^3}, \label{eq:transport}
\end{align}
where $\mu=\frac{\sqrt{\det\gslash}}{p^{3}}$ and 
\[
\frac{1}{\mu} X[\mu] = \left( \Gamma^\alpha_{\mu\alpha} p^\mu + \Gamma^3_{\mu\nu} \frac{p^\mu p^\nu}{p^3} \right).
\]
Consequently, we obtain
\begin{align}
\frac{d}{ds} A(s) 
&= 2 \int_{\mathcal{P}_{x}} \varphi X[\varphi] \sqrt{\det \gslash} \frac{dp^1 dp^2 dp^3}{p^3}
+ \int_{\mathcal{P}_{x}} \left( \Gamma^\alpha_{\mu\alpha} p^\mu + \Gamma^3_{\mu\nu} \frac{p^\mu p^\nu}{p^3} \right) |\varphi|^2 \sqrt{\det \gslash} \frac{dp^1 dp^2 dp^3}{p^3}. \label{eq:A_derivative}
\end{align}

\noindent We now analyze the two error terms arising from the divergence computation. First, we expand the contraction:
\begin{align}\notag
\Gamma^\alpha_{\mu\alpha} p^\mu 
&= \Gamma^A_{3A} p^3 + \Gamma^A_{4A} p^4 + \Gamma^A_{BA} p^B + \Gamma^4_{34} p^3 + \Gamma^4_{44} p^4 + \Gamma^4_{A4} p^A \\ &+
 \Gamma^3_{33} p^3 + \Gamma^3_{43} p^4 + \Gamma^3_{A3} p^A.
\end{align}
By introducing the standard null structure coefficients on the mass shell, and expressing the angular connection coefficients $\Gamma^A_{BA}$ in terms of $\gslash$, we rewrite the above as
\begin{align}
\Gamma^\alpha_{\mu\alpha} p^\mu 
= \left( \tr \chibar - e_A(b^A) \right) p^3 + \tr \chi \, p^4 + \Gammaslash^A_{BA} p^B + 2 \omegabar p^3 + (\etabar_A - \eta_A) p^A,
\end{align}
which we abbreviate as
\[
\Gamma^\alpha_{\mu\alpha} p^\mu = \left( \tr \chibar - e_A(b^A) \right) p^3 + r,
\]
with a remainder term $r$ satisfying
\[
|r| \lesssim \frac{a^{1/2} \Gamma}{|u|^2},
\]
using the bootstrap assumptions and $L^\infty$ estimates for the Ricci coefficients. The leading term $\tr \chibar \, p^3$ decays only like $|u|^{-1}$, and thus is not integrable in the $u$-direction. However, \textit{we exploit a favorable cancellation} by rewriting:
\begin{align} \notag
\int_{\mathcal{P}_{x}} \Gamma^\alpha_{\mu\alpha} p^\mu |\varphi|^2 \sqrt{\det \gslash} \frac{dp^1 dp^2 dp^3}{p^3}
&= \int_{\mathcal{P}_{x}} \left( \tr \chibar + \frac{2}{|u|} - e_A(b^A) \right) p^3 |\varphi|^2 \sqrt{\det \gslash} \frac{dp^1 dp^2 dp^3}{p^3} \\
 &- \frac{2}{|u|} \int_{\mathcal{P}_{x}} p^3 |\varphi|^2 \sqrt{\det \gslash} \frac{dp^1 dp^2 dp^3}{p^3}
+ \int_{\mathcal{P}_{x}} r |\varphi|^2 \sqrt{\det \gslash} \frac{dp^1 dp^2 dp^3}{p^3}.
\end{align}
The cancellation of the leading borderline term now yields
\[
\left| \int_{\mathcal{P}_{x}} \Gamma^\alpha_{\mu\alpha} p^\mu |\varphi|^2 \sqrt{\det \gslash} \frac{dp^1 dp^2 dp^3}{p^3} \right|
\lesssim \frac{a^{1/2} \Gamma}{|u|^2} \int_{\mathcal{P}_{x}} |\varphi|^2 \sqrt{\det \gslash} \frac{dp^1 dp^2 dp^3}{p^3}.
\]
Notice that the appearance of a decay term is essentially a consequence of the fact that the measure of the spheres $S_{u,\ubar}$ is strictly decreasing along the incoming direction.
\noindent Next, we estimate the term $\Gamma^3_{\mu\nu} p^\mu p^\nu / p^3$. Using the null frame decomposition of the connection coefficients, we compute
\begin{align} \notag 
\Gamma^3_{\mu\nu} p^\mu p^\nu 
&= \Gamma^3_{44} (p^4)^2 + 2\Gamma^3_{43} p^4 p^3 + 2\Gamma^3_{4A} p^4 p^A + \Gamma^3_{33} (p^3)^2 + 2\Gamma^3_{3A} p^3 p^A + \Gamma^3_{AB} p^A p^B \\
&= 2\omegabar (p^3)^2 + (\eta_A - \etabar_A) p^A p^3 + \chi_{AB} p^A p^B.
\end{align}
Therefore, applying the pointwise estimates for Ricci coefficients, we obtain
\[
\left| \Gamma^3_{\mu\nu} p^\mu p^\nu \right| \lesssim \frac{a \Gamma}{|u|^3} + \frac{a^{1/2} \Gamma}{|u|^2}.
\]

\noindent Substituting these bounds into \eqref{eq:A_derivative}, we conclude that
\begin{align}
\frac{d}{ds} A(s) 
&\leq 2 \int_{\mathcal{P}_{x}} \varphi X[\varphi] \sqrt{\det \gslash} \frac{dp^1 dp^2 dp^3}{p^3} + \frac{a^{1/2} \Gamma}{|u|^2} A(s) + \frac{a \Gamma}{|u|^3} A(s).
\end{align}
Applying Grönwall's inequality and invoking Proposition~\ref{integrate}, together with the bootstrap assumptions and the smoothness of $\varphi$, we deduce that
\[
A(s) \lesssim A(s_\infty) + \int_{s_\infty}^{s} \left( \int_{\mathcal{P}_{x}} \varphi X[\varphi] \sqrt{\det \gslash} \frac{dp^1 dp^2 dp^3}{p^3} \right) ds'.
\]
The estimates for the associated functionals $B(s)$ and $C(s)$ proceed analogously and are omitted for brevity.
\end{proof}

\noindent Now we want to apply these transport inequalities systematically to control up to three derivatives of the Vlasov distribution function. Before moving to the estimates, we recall one more remark.
\begin{remark}
 One can bound the spacetime (or over $S_{u,\ubar}$) $L^{2}$-norm of the Vlasov stress-energy tensor components by the $L^{2}_{x}L^{2}_{\mathcal{P}_{x}}$ norm of the same tensor components, using H\"older's inequality and the decay properties of the momentum support obtained in subsection \ref{subsectiondecaymomentum}. Indeed, there holds  
 \begin{align}
&\big\lVert\int_{\mathcal{P}_{x}}\varphi \sqrt{\det(\gslash)}\frac{dp^{1}dp^{2}dp^{3}}{p^{3}} \big\rVert^{2}_{L^{2}(S_{u,\ubar})}\nonumber \\ =& \int_{S_{u,\ubar}} \nonumber   \left(\int_{\mathcal{P}_{x}}\varphi \sqrt{\det(\gslash)}\frac{dp^{1}dp^{2}dp^{3}}{p^{3}} \right)\left(\int_{\mathcal{P}_{x}}\varphi \sqrt{\det(\gslash)}\frac{dp^{1}dp^{2}dp^{3}}{p^{3}}\right)\sqrt{\det(\gslash)}\text{d}\theta^1\text{d}\theta^2.
\end{align}
Now we can use Cauchy-Schwarz for the integral to get: 
\begin{align}\nonumber 
\int_{\mathcal{P}_{x}}\varphi \sqrt{\det(\gslash)}\frac{dp^{1}dp^{2}dp^{3}}{p^{3}}\leq& \left(\int_{\mathcal{P}_{x}}\varphi^{2}\sqrt{\det(\gslash)}\frac{dp^{1}dp^{2}dp^{3}}{p^{3}}  \right)^{\frac{1}{2}}\left(\int_{\mathcal{P}_x}  \sqrt{\det(\gslash)}\frac{dp^{1}dp^{2}dp^{3}}{p^{3}}  \right)^{\frac{1}{2}}\\\nonumber 
\lesssim& \frac{1}{|u|} \left(\int_{\mathcal{P}_x}\varphi^{2}\sqrt{\det(\gslash)}\frac{dp^{1}dp^{2}dp^{3}}{p^{3}}  \right)^{\frac{1}{2}}
\end{align}
Therefore, we have under the decay of the momentum support \eqref{momentum},
\begin{eqnarray}
||\int_{P}\varphi \sqrt{\det(\gslash)}\frac{dp^{1}dp^{2}dp^{3}}{p^{3}} ||^{2}_{L^{2}(S_{u,\ubar})}\lesssim \frac{1}{|u|^{2}}\int_{S_{u,\ubar}}\int_{p}  \varphi^{2}\sqrt{\det(\gslash)}\frac{dp^{1}dp^{2}dp^{3}}{p^{3}}.    
\end{eqnarray}
\end{remark}
Now we state one of the most vital aspects of the Vlasov estimates, namely that of the expressions for the commutators $[X,V]$ for each of the $V-$vector fields.
\begin{proposition}
\label{commutator}
Recall the following set of vector fields defined on the mass-shell 
\begin{gather}
V_{(A)}:= Hor(e_A)- \frac{p^3}{\lvert u \rvert} \partial_{\pbar^A}, \hsp V_{(3)} := Hor(e_4), \hspace{3mm} V_{(4)} := p^3 \partial_{\pbar^3} - \lvert u \rvert Hor(e_3), \hspace{3mm} V_{(4+A)}:= \frac{p^3}{\lvert u \rvert^2}\partial_{\pbar^A},
\end{gather}together with the vector field \be V_{(0)}:= \lvert u \rvert Hor(e_3). \ee 
Let $X$ be the geodesic spray. The following commutation relations hold: 
\begin{align}
[X,V_{(A)}]=& \Bigg[\big[-p^3 \hsp e_A(b^B)+\big(\Gamma^B_{A\nu}p^{\nu}+\frac{p^3}{\modu}{\delta_A}^B\big)\big]\big(\frac{\frac{1}{2}\chi_{BC}p^C}{p^3}-\etabar_B\big) \notag\\ +&\frac{1}{\modu}\big(2\hsp e_A(\log \Omega)\hsp p^3 + \Gamma^3_{A\nu}p^{\nu}\big) \big(1+\modu \frac{\eta_C p^C}{p^3}+|u|\omegabar\big) \notag \\ +& \bigg((e_A(\Gamma^3_{\mu\nu})-e_\mu (\Gamma^3_{A\nu}))p^{\mu}\hsp p^{\nu} +\Gamma^C_{\mu\nu}p^{\mu}p^{\nu}\Gamma^{3}_{A\lambda}\partial_{\pbar^C}(p^{\lambda}) + \Gamma^3_{\mu\nu}p^{\mu}p^{\nu}\Gamma^3_{A\lambda} \partial_{\pbar^3}(p^{\lambda}) \notag \\ -&\Gamma^C_{A\lambda}p^{\lambda}\Gamma^3_{\mu\nu}\partial_{\pbar^C}(p^{\mu}p^{\nu})-\Gamma^3_{A\lambda}p^{\lambda} \Gamma^3_{\mu\nu}\partial_{\pbar^3}(p^{\mu}p^{\nu}) +\frac{p^3}{\modu}\Gamma^3_{\mu\nu}\partial_{\pbar^A}(p^{\mu}p^{\nu}) \bigg)\frac{1}{p^3}\Bigg]\mathbf{V_{(0)}} \notag \\ +&\big[ -p^3 \hsp e_A(b^B)+(\Gamma^B_{A\nu}p^{\nu}+\frac{p^3}{\modu}{\delta_A}^B)    \big]  \mathbf{V_{(B)}} \notag \\ +& \Omega^{2}\bigg[\frac{2}{\Omega^{2}}(\eta_{A}+\etabar_{A})p^{4}+  \frac{1}{\Omega^2}\bigg(\big[\Gammaslash^C_{AB}p^B +\big({\chibarhat_{A}}^C + \frac{1}{2}\tildetr {\delta_A}^C\big)p^3 +{\chi_A}^C p^4\big] \frac{\gslash_{CD}p^D}{2p^3} \bigg) \notag \\ -& \frac{1}{\Omega^2}\big(\frac{1}{2}\chi_{AB}p^B -\etabar_A \hsp p^3\big)\frac{p^4}{p^3} \bigg] \mathbf{V_{(3)}} \notag \\ +& \Bigg[\big[-p^3 \hsp e_A(b^B)+\big(\Gamma^B_{A\nu}p^{\nu}+\frac{p^3}{\modu}{\delta_A}^B\big)\big]\big(\frac{\frac{1}{2}\chi_{BC}p^C}{p^3}-\etabar_B\big) \notag\\ +&\big(\chi_{AB}p^{B}-\etabar_{A}p^{3} \big) |u|\omegabar  \notag \\+& \bigg((e_A(\Gamma^3_{\mu\nu})-e_\mu (\Gamma^3_{A\nu}))p^{\mu}\hsp p^{\nu} +\Gamma^C_{\mu\nu}p^{\mu}p^{\nu}\Gamma^{3}_{A\lambda}\partial_{\pbar^C}(p^{\lambda}) + \Gamma^3_{\mu\nu}p^{\mu}p^{\nu}\Gamma^3_{A\lambda} \partial_{\pbar^3}(p^{\lambda}) \notag \\ -&\Gamma^C_{A\lambda}p^{\lambda}\Gamma^3_{\mu\nu}\partial_{\pbar^C}(p^{\mu}p^{\nu})-\Gamma^3_{A\lambda}p^{\lambda} \Gamma^3_{\mu\nu}\partial_{\pbar^3}(p^{\mu}p^{\nu}) +\frac{p^3}{\modu}\Gamma^3_{\mu\nu}\partial_{\pbar^A}(p^{\mu}p^{\nu}) \bigg)\frac{1}{p^3}\Bigg]\mathbf{V_{(4)}} \notag \\
\\ +& \bigg[ \big[-p^3 e_A(b^C)+\big(\Gamma^C_{A\nu}p^{\nu}+ \frac{p^3}{\modu}{\delta_A}^C \big) \big]\hsp \modu^2 \hsp \bigg(\frac{\Gammaslash^B_{CD}p^D}{p^3}+{\chibarhat}_C^B+\frac{1}{2}(\tr\chibar+\frac{2}{\modu}){\delta_C}^B+ \frac{{\chi_C}^B p^4}{p^3}\bigg) \notag \\ +& \big(\chi_{AB}p^{B}-\etabar_{A}p^{3} \big)\hsp \modu^2 \hsp \big(\frac{({\chibar_C}^B-e_{C}(b^{B}))\hsp p^C+2\eta^B\hsp p^4}{p^3}\big) \notag \\ +& \big(2(\eta_A+\etabar_A) p^4 + \bigg(\big[\Gammaslash^C_{AB}p^B +\big({\chibarhat_{A}}^C + \frac{1}{2}\tildetr {\delta_A}^C\big)p^3 +{\chi_A}^C p^4\big] \frac{\gslash_{CD}p^D}{2p^3} \bigg) \notag \\ -& \big(\frac{1}{2}\chi_{AB}p^B -\etabar_A \hsp p^3\big)\frac{p^4}{p^3}\big) \hsp \modu^2 \hsp \big(\frac{({\chi_C}^B)\hsp p^C}{p^3} +2\hsp \etabar^B\big) \notag \\ +& \bigg( {(e_A(\Gamma^B_{\mu\nu})-e_{\mu}(\Gamma^B_{A\nu}))\hsp p^{\mu}p^{\nu}}+\Gamma^C_{\mu\nu} \hsp p^{\mu}\hsp p^{\nu} \partial_{\pbar^C}(p^\lambda)\Gamma^B_{A\lambda} + \Gamma^3_{\mu\nu}p^{\mu}p^{\nu} \partial_{\pbar^3}(p^\lambda) \Gamma^B_{A\lambda} \notag \\ &{- {\Gamma^C_{A\lambda}p^{\lambda}\Gamma^B_{\mu\nu} \partial_{\pbar^C}(p^{\mu}p^{\nu})}}-\Gamma^3_{A\lambda}p^{\lambda}\Gamma^B_{\mu\nu}\partial_{\pbar^3}(p^{\mu}p^{\nu}) +\big({-\frac{(p^3)^2}{\modu^2}\frac{1}{\Omega^2} {\delta_A}^B} +\frac{{\delta_A}^B}{\modu}\Gamma^3_{\mu\nu}p^{\mu}p^{\nu}\big)\notag \\ &- {\frac{p^3}{\modu} \partial_{\pbar^A}(p^{\mu}p^{\nu})\Gamma^B_{\mu\nu}}\bigg)\hsp \frac{\modu^2}{p^3}\Bigg]\mathbf{V_{(4+B)}}, \notag 
\end{align}

\begin{align}\nonumber
    [X,V_{(3)}]=& \Bigg[ \big[2\eta^{A}p^{3}+\chihat^{A}_{D}p^{D}+\frac{1}{2}\tr\chi p^{A} \big]\big(\frac{\chi_{AC}p^C}{2p^3} - \etabar_A\big) \\ \nonumber +&\bigg( \left( {R^3}_{\nu 4\mu}- \Gamma^{\epsilon}_{\mu\nu}\Gamma^{3}_{4\epsilon} +\Gamma^{\epsilon}_{4\nu}\Gamma^3_{\mu \epsilon} +(\Gamma^{\epsilon}_{4\mu}-\Gamma^{\epsilon}_{\mu 4})\Gamma^3_{\epsilon \nu}\right)p^{\mu}p^{\nu} + \Gamma^C_{\mu\nu}p^{\mu}p^{\nu} \Gamma^3_{4\lambda}\partial_{\pbar^C}(p^{\lambda}) \notag \\ +&\Gamma^3_{\mu\nu}p^{\mu}p^{\nu} \Gamma^3_{4\lambda}\partial_{\pbar^3}(p^{\lambda})  -\Gamma^C_{4\lambda}\hsp \Gamma^3_{\mu\nu}\hsp p^{\lambda}\hsp \partial_{\pbar^C}(p^{\mu}p^{\nu}) -\Gamma^3_{4\lambda}\hsp \Gamma^3_{\mu\nu}\hsp p^{\lambda}\hsp \partial_{\pbar^3}(p^{\mu}p^{\nu})\bigg)\frac{1}{p^3}\Bigg] \bf{V_{(0)}} \\ \nonumber +& \big[2\eta^{B}p^{3}+\chihat^{B}_{A}p^{A}+\frac{1}{2}\tr\chi p^{B}  \big] \bf{V_{(B)}} \\ \nonumber +&\bigg[-(\eta_{A}+\etabar_{A})p^{A}-[\omegabar-\omegabar(1-\Omega^{-2})]p^{3}+(\chi^{C}_{B}p^{B}\frac{\gslash_{CD}p^{D}}{2p^{3}}+2\etabar_{D}p^{D})   \bigg] \bf{V_{(3)}} \\ \nonumber +& \Bigg[   \big[ 2\eta^{B}p^{3}+\chihat^{B}_{A}p^{A}+\frac{1}{2}\tr\chi p^{B}\big]\big(\frac{\chi_{AC}p^C}{2p^3} - \etabar_A\big) \\ \nonumber +&\bigg( \left( {R^3}_{\nu 4\mu}- \Gamma^{\epsilon}_{\mu\nu}\Gamma^{3}_{4\epsilon} +\Gamma^{\epsilon}_{4\nu}\Gamma^3_{\mu \epsilon} +(\Gamma^{\epsilon}_{4\mu}-\Gamma^{\epsilon}_{\mu 4})\Gamma^3_{\epsilon \nu}\right)p^{\mu}p^{\nu} + \Gamma^C_{\mu\nu}p^{\mu}p^{\nu} \Gamma^3_{4\lambda}\partial_{\pbar^C}(p^{\lambda}) \notag \\ +&\Gamma^3_{\mu\nu}p^{\mu}p^{\nu} \Gamma^3_{4\lambda}\partial_{\pbar^3}(p^{\lambda})  -\Gamma^C_{4\lambda}\hsp \Gamma^3_{\mu\nu}\hsp p^{\lambda}\hsp \partial_{\pbar^C}(p^{\mu}p^{\nu}) -\Gamma^3_{4\lambda}\hsp \Gamma^3_{\mu\nu}\hsp p^{\lambda}\hsp \partial_{\pbar^3}(p^{\mu}p^{\nu})\bigg)\frac{1}{p^3}\Bigg] \bf{V_{(4)}} \\ \nonumber +& \Bigg[ \big[ 2\eta^{B}p^{3}+\chihat^{B}_{A}p^{A}+\frac{1}{2}\tr\chi p^{B}  \big] \modu^2 \big( \frac{\Gammaslash^B_{AC}p^C}{p^3}+{\chibarhat_A}^B +\frac{1}{2}\tildetr {\delta_A}^B+ {\chi_A}^B \frac{p^4}{p^3}\big)\\ \nonumber +& \big[-(\eta_{A}+\etabar_{A})p^{A}-[\omegabar-\omegabar(1-\Omega^{-2})]p^{3}+(\chi^{C}_{B}p^{B}\frac{\gslash_{CD}p^{D}}{2p^{3}}+2\etabar_{D}p^{D}) \big] |u|^{2}(\frac{\chi^{A}_{B}p^{B}}{p^{3}}+2\etabar^{A})
    \\ \nonumber +& \bigg(  \left( {R^B}_{\nu 4\mu}- \Gamma^{\epsilon}_{\mu\nu}\Gamma^{B}_{4\epsilon} +\Gamma^{\epsilon}_{4\nu}\Gamma^B_{\mu \epsilon} +(\Gamma^{\epsilon}_{4\mu}-\Gamma^{\epsilon}_{\mu 4})\Gamma^B_{\epsilon \nu}\right)p^{\mu}p^{\nu} + \Gamma^C_{\mu\nu}p^{\mu}p^{\nu} \Gamma^B_{4\lambda}\partial_{\pbar^C}(p^{\lambda}) \notag \\ +&\Gamma^3_{\mu\nu}p^{\mu}p^{\nu} \Gamma^B_{4\lambda}\partial_{\pbar^3}(p^{\lambda})  -\Gamma^C_{4\lambda}\hsp \Gamma^B_{\mu\nu}\hsp p^{\lambda}\hsp \partial_{\pbar^C}(p^{\mu}p^{\nu}) -\Gamma^3_{4\lambda}\hsp \Gamma^B_{\mu\nu}\hsp p^{\lambda}\hsp \partial_{\pbar^3}(p^{\mu}p^{\nu})  \bigg) \frac{\modu^2}{p^3}\Bigg]\bf{V_{(4+B)}},
\end{align}

\begin{align}\nonumber
    [X,V_{(4)}]=& \Bigg[ \modu \big[-2p^{4}(\etabar^{A}-\eta^{A})-\Gamma^{A}_{3\nu}p^{\nu}\big]\big(\frac{\chi_{AC}p^C}{2p^3} - \etabar_A\big) \\ \nonumber +& \big(1+\modu \frac{\eta_C p^C}{p^3}+|u|\omegabar\big)\big(\eta_{B}p^{B}+\omegabar p^{3}) \\ \nonumber  +&\bigg( -\Gamma^3_{\mu\nu}p^{\mu}p^{\nu} -\frac{p^3}{\Omega^2} \Gamma^3_{3\nu}p^{\nu} + \modu \big( e_{\mu}(\Gamma^3_{3\nu})p^{\mu}p^{\nu}- \Gamma^C_{\mu\nu}p^{\mu}p^{\nu} \Gamma^3_{3\lambda}\partial_{\pbar^C}(p^{\lambda}) -\Gamma^3_{\mu\nu}p^{\mu}p^{\nu}\Gamma^3_{3\lambda}\partial_{\pbar^3}(p^{\lambda})\big)\notag \\ \nonumber +& p^3\Gamma^3_{\mu\nu}\partial_{\pbar^3}(p^{\mu}p^{\nu})- \modu \big( e_3(\Gamma^3_{\mu\nu}) p^{\mu}p^{\nu}- \Gamma^C_{3\lambda}p^{\lambda} \Gamma^3_{\mu\nu}\partial_{\pbar^C}(p^{\mu}p^{\nu}) - \Gamma^3_{3\lambda}p^{\lambda}\Gamma^3_{\mu\nu}\partial_{\pbar^3}(p^{\lambda})\big)\bigg)\frac{1}{p^3}\Bigg] \bf{V_{(0)}} \\ \nonumber +& \modu \big[-2p^{4}(\etabar^{A}-\eta^{A})-\Gamma^{A}_{3\nu}p^{\nu} \big] \bf{V_{(B)}} \\ \nonumber +&\bigg[p^{4}+|u|(-\omegabar p^{4}-(\Gamma^{A}_{3B}p^{B}+2\eta^{A}p^{4})\frac{\gslash_{AD}p^{D}}{p^{3}})+\frac{|u|}{p^{3}}(\eta_{A}p^{A}+\omegabar p^{3})p^{4}   \bigg] \bf{V_{(3)}} \\ \nonumber +& \Bigg[   \modu \big[ -2p^{4}(\etabar^{A}-\eta^{A})-\Gamma^{B}_{3\nu}p^{\nu} \big] \big(\frac{\chi_{AC}p^C}{2p^3} - \etabar_A\big) \\ \nonumber &+|u|\omegabar(\etabar_{A}p^{A}+\omegabar p^{3})\\ \nonumber +&\bigg( -\Gamma^3_{\mu\nu}p^{\mu}p^{\nu} -\frac{p^3}{\Omega^2} \Gamma^3_{3\nu}p^{\nu} + \modu \big( e_{\mu}(\Gamma^3_{3\nu})p^{\mu}p^{\nu}- \Gamma^C_{\mu\nu}p^{\mu}p^{\nu} \Gamma^3_{3\lambda}\partial_{\pbar^C}(p^{\lambda}) -\Gamma^3_{\mu\nu}p^{\mu}p^{\nu}\Gamma^3_{3\lambda}\partial_{\pbar^3}(p^{\lambda})\big)\notag \\ \nonumber +& p^3\Gamma^3_{\mu\nu}\partial_{\pbar^3}(p^{\mu}p^{\nu})- \modu \big( e_3(\Gamma^3_{\mu\nu}) p^{\mu}p^{\nu}- \Gamma^C_{3\lambda}p^{\lambda} \Gamma^3_{\mu\nu}\partial_{\pbar^C}(p^{\mu}p^{\nu}) - \Gamma^3_{3\lambda}p^{\lambda}\Gamma^3_{\mu\nu}\partial_{\pbar^3}(p^{\lambda})\big)\bigg)\frac{1}{p^3}\Bigg] \bf{V_{(4)}}  \\ \nonumber +& \Bigg[ \modu \big[-2p^{4}(\etabar^{A}-\eta^{A})-\Gamma^{A}_{3\nu}p^{\nu} \big] \modu^2 \big( \frac{\Gammaslash^B_{AC}p^C}{p^3}+{\chibarhat_A}^B +\frac{1}{2}\tildetr {\delta_A}^B + {\chi_A}^B \frac{p^4}{p^3}\big)\\ \nonumber +& \bigg(\eta_{A}p^{A}+\omegabar p^{3} \bigg) \modu^3 \frac{({\chibar_C}^B-e_{C}(p^{B})) p^C + 2\eta^B p^4}{p^3} \\ \nonumber +& \big( p^{4}+|u|(-\omegabar p^{4}-(\Gamma^{A}_{3B}p^{B}+2\eta^{A}p^{4})\frac{\gslash_{AD}p^{D}}{p^{3}})+\frac{|u|}{p^{3}}(\eta_{A}p^{A}+\omegabar p^{3})\big)\modu^2 \hsp \big(\frac{{\chi_C}^B\hsp p^C}{p^3} +2\hsp \etabar^B\big) \\ \nonumber +& \bigg(  \modu ( {R^B}_{\nu \mu 3}- \Gamma^{\epsilon}_{3\nu}\Gamma^{B}_{\mu\epsilon} +\Gamma^{\epsilon}_{\mu\nu}\Gamma^B_{3 \epsilon} +(\Gamma^{\epsilon}_{\mu3}-\Gamma^{\epsilon}_{3 \mu})\Gamma^B_{\epsilon \nu}) - \modu \Gamma^C_{\mu\nu}p^{\mu}p^{\nu}\partial_{\pbar^C}(p^{\lambda}) \Gamma^B_{3\lambda}\notag \\ -& \modu \Gamma^3_{\mu\nu}p^{\mu}p^{\nu}\partial_{\pbar^3}(p^{\lambda})\Gamma^3_{3\lambda} - \frac{p^3}{\Omega^2}\Gamma^B_{3\nu}p^{\nu} + \Gamma^B_{\mu\nu}p^3 \partial_{\pbar^3}(p^{\mu}p^{\nu}) + \modu \Gamma^B_{\mu\nu} \Gamma^C_{3\lambda}p^{\lambda}\partial_{\pbar^C}(p^{\mu}p^{\nu}) \notag \\ +& \modu \Gamma^{B}_{\mu\nu}\Gamma^3_{3\lambda}p^{\lambda}\partial_{\pbar^3}(p^{\mu}p^{\nu}) \bigg) \frac{\modu^2}{p^3}\Bigg]\bf{V_{(4+B)}},
\end{align}

\begin{align}\nonumber
    [X,V_{(4+B)}]=& \Bigg[-\frac{p^3}{\modu^2} {\delta_B}^A   \big(\frac{\chi_{AC}p^C}{2p^3} - \etabar_A\big) \nonumber  \nonumber -\frac{1}{\modu^2}\Gamma^3_{\mu\nu}\partial_{\pbar^B}(p^{\mu}p^{\nu})   \Bigg] \bf{V_{(0)}} \\ \nonumber -& \frac{p^3}{\modu^2} {\delta_B}^A \bf{V_{(A)}}\\\nonumber  -&\frac{\gslash_{BC}p^C}{2\modu^2} \bf{V_{(3)}} \\ \nonumber +&\Bigg[ -\frac{p^3}{\modu^2}{\delta_B}^A \bigg(\frac{\chi_{AC}p^C}{2p^3}-\etabar_A\bigg)  -\frac{1}{\modu^2}\Gamma^3_{\mu\nu}\partial_{\pbar^B}(p^{\mu}p^{\nu})
   \Bigg] \bf{V_{(4)}} \\ \nonumber +& \Bigg[ -\frac{p^3}{\modu^2}{\delta_B}^C \hsp \modu^2 \big( \frac{\Gammaslash^A_{CD}p^D}{p^3}+{\chibarhat_C}^A+{\chi_C}^A\frac{p^4}{p^3}+ \frac{1}{2}(\tr\chibar+\frac{2}{\modu}) {\delta_C}^A \big) \\ \nonumber -& \frac{\gslash_{BC}p^C}{2}\bigg(\frac{{\chi_C}^Ap^C}{p^3}+2\etabar^A\bigg) \\ +& \bigg[{\delta_B}^A \big((-\Gamma^3_{\mu\nu}p^{\mu}p^{\nu} + \frac{2(p^3)^2}{\modu^3}\big)+\Gamma^A_{\mu\nu}p^3 \partial_{\pbar^B}(p^{\mu}p^{\nu}) \bigg]
         \Bigg]\bf{V_{(4+A)}}.
\end{align}
\end{proposition}


\begin{remark}
  It is of importance to include $[X,V_{(0)}]$, for the sake of higher derivative estimates. However, 
\be [X,V_{(0)}]=-[X,V_{(4)}]+[X,p^3\partial_{\pbar^3}]. \ee 
Therefore, we omit its explicit expression.
\end{remark}

\begin{proof}
The proof is a consequence of direct calculations, keeping careful track of all the cancellations. Recall \[ X= p^\mu e_{\mu}- \Gamma^A_{\mu\nu}p^{\mu}p^{\nu}\partial_{\pbar^A} - \Gamma^3_{\mu\nu}p^{\mu}p^{\nu}\partial_{\pbar^3}.  \] We perform the commutation calculations in the holonomic $C$-basis\footnote{This was deemed most practical, because in this holonomic basis the commutators enjoy the simple identity \be \notag [X,V]^{i}= X(V^i)-V(X^i), \ee where $F^i$ denotes the $i-$th component of a vector field $F$.} $\{\frac{\partial}{\partial\theta^{A}},\frac{\partial}{\partial u},\frac{\partial}{\partial\ubar},\frac{\partial}{\partial \pbar^{A}},\frac{\partial}{\partial \pbar^{3}}\}_{A=1}^{2}$ 
In this $C$-basis, the components of the vector field $X$ read\footnote{Recall the choice of frame \eqref{ourframe}.}
\be X = \big(p^3 b^A +p^A, p^3, \frac{p^4}{\Omega^2}, -\Gamma^A_{\nu\lambda}p^{\nu}p^{\lambda}, -\Gamma^3_{\nu\lambda}p^{\nu}p^{\lambda}\big). \ee
First, consider the vector field $V_{(A)}$. We have $V_{(A)}= \Hor(e_A)-\frac{p^3}{\modu}\partial_{\pbar^A},$ hence \[ V_{(A)}= \big({\delta_A}^B, 0,0, -(\Gamma^B_{A\nu}p^\nu+\frac{p^3}{\modu}{\delta_A}^B), -\Gamma^3_{A\nu}p^{\nu}\big).   \]For the first two components, there holds

\begin{align}\notag 
[X,V_{(A)}]^B =& X({\delta_A}^B) -V_{(A)}(p^3b^B+p^B) = -p^3 e_A(b^B) + \big(\Gamma^B_{A\nu}p^{\nu}+\frac{p^3}{\modu} {\delta_A}^B\big) -\Gamma^3_{A\nu}p^{\nu} b^B\\ =& \Gammaslash^B_{AC}p^C + {\chi_A}^B p^4 +\big({\chibarhat_A}^B- e_A(b^B) +\frac{1}{2}\tildetr {\delta_{A}}^B\big)p^3 - \f12 \chi_{AC}p^Cb^B +\etabar_A p^3 b^B.
\end{align} For the third component, there holds

\begin{align} [X,V_{(A)}]^3 = -\Hor(e_A)(p^3) = \Gamma^3_{A\nu}p^\nu = \chi_{AB}p^B -\etabar_A p^3. \end{align}
For the fourth component, 
\begin{align}
[X,V_{(A)}]^4 =& -\Hor(e_A)\big(\frac{p^4}{\Omega^2}\big)  = \frac{2}{\Omega^2} e_A(\log \Omega) p^4 +\frac{1}{\Omega^2}\big(\Gamma^C_{A\nu}p^{\nu} \frac{\gslash_{CD}p^D}{2p^3} -\Gamma^3_{A\nu}p^{\nu}\frac{p^4}{p^3}+\frac{p^3}{\modu}\frac{\gslash_{AB}p^B}{2p^3}\big) \notag \\ =& \frac{2}{\Omega^2}(\eta_A+\etabar_A) p^4 + \frac{1}{\Omega^2}\bigg(\big[\Gammaslash^C_{AB}p^B +\big({\chibarhat_{A}}^C + \frac{1}{2}\tildetr {\delta_A}^C\big)p^3 +{\chi_A}^C p^4\big] \frac{\gslash_{CD}p^D}{2p^3} \bigg) \notag \\ -& \frac{1}{\Omega^2}\big(\frac{1}{2}\chi_{AB}p^B -\etabar_A \hsp p^3\big)\frac{p^4}{p^3}.    \end{align}For the fifth and sixth term, we have the following:

 \begin{align} [X, Hor(e_A)-\frac{p^3}{\modu}\partial_{\pbar^A}]^{4+B}&= -X(\Gamma^B_{A\nu}p^{\nu}) +Hor(e_A)(\Gamma^B_{\mu\nu}p^{\mu}p^{\nu}) -X \left(\frac{p^3}{\lvert u \rvert}{\delta_A}^B\right) -\frac{p^3}{\modu}\partial_{\pbar^A}(\Gamma^B_{\mu\nu}p^{\mu}p^{\nu}) \notag \\ &= \underbrace{(e_A(\Gamma^B_{\mu\nu})-e_{\mu}(\Gamma^B_{A\nu}))\hsp p^{\mu}p^{\nu}}_\text{I}+\Gamma^C_{\mu\nu} \hsp p^{\mu}\hsp p^{\nu} \partial_{\pbar^C}(p^\lambda)\Gamma^B_{A\lambda} + \Gamma^3_{\mu\nu}p^{\mu}p^{\nu} \partial_{\pbar^3}(p^\lambda) \Gamma^B_{A\lambda} \notag \\ &\underbrace{-\Gamma^C_{A\lambda}p^{\lambda}\Gamma^B_{\mu\nu} \partial_{\pbar^C}(p^{\mu}p^{\nu})}_\text{II}-\Gamma^3_{A\lambda}p^{\lambda}\Gamma^B_{\mu\nu}\partial_{\pbar^3}(p^{\mu}p^{\nu}) +\left(\underbrace{-\frac{(p^3)^2}{\modu^2}{\delta_A}^B}_\text{I} +\frac{{\delta_A}^B}{\modu}\Gamma^3_{\mu\nu}p^{\mu}p^{\nu}\right)\notag \\&\underbrace{-\frac{p^3}{\modu} \partial_{\pbar^A}(p^{\mu}p^{\nu})\Gamma^B_{\mu\nu}}_\text{II}.\end{align}Each pair of terms in brackets labelled (I) and (II) leads to a cancellation that improves the order of decay. Indeed, the first term in bracket (II), is $o(\modu^{-3})-{\chibar_A}^C\hsp p^3 \hsp \Gamma^B_{\mu\nu}\hsp \partial_{\pbar^C}(p^{\mu}p^{\nu}) = -\frac{1}{2}\tr\chibar\hsp p^3\hsp \Gamma^B_{\mu\nu}\hsp \partial_{\pbar^A}({p^{\mu}p^{\nu}})+o(\modu^{-3})$. By itself, this term is only $o(\modu^{-2})$. When combined, however, with the second term in bracket $(II)$, the whole bracket becomes
 \[   (II) =   -\frac{1}{2}p^3 \hsp \big(\tr\chibar+\frac{2}{\modu}\big)\Gamma^B_{\mu\nu}\partial_{\pbar^A}(p^{\mu}p^{\nu})+ o(\modu^{-3})=o(\modu^{-3}).            \]As for the terms in bracket (I), it is here where not only the issue of integrability, but also the issue of regularity in the terms is overcome. Indeed, notice that

\begin{equation}
     e_A(\Gamma^B_{\mu\nu})-e_{\mu}(\Gamma^B_{A\nu}) = {R^B}_{\nu A \mu}- \Gamma^{\epsilon}_{\mu\nu}\Gamma^{B}_{A\epsilon} +\Gamma^{\epsilon}_{A\nu}\Gamma^B_{\mu \epsilon} +(\Gamma^{\epsilon}_{A\mu}-\Gamma^{\epsilon}_{\mu A})\Gamma^B_{\epsilon \nu}.
\end{equation}In other words, throughout the entire argument, \textbf{derivatives of Christoffel symbols will always appear in differences that give rise to Riemann curvature components (up to terms of the form $\Gamma\Gamma$, which are harmless, both in terms of integrability and regularity)}.

\vspace{3mm}
\par\noindent As a consequence, the first term in $(I)$ is $o(\modu^{-3}) + \Gamma^{\epsilon}_{A\nu}\Gamma^B_{\mu\epsilon}p^{\mu}p^{\nu}$, where the latter fails to be $o(\modu^{-3})$ precisely when $\mu=\nu=3$. In that case, the worst term is \[   \Gamma^C_{A3}\Gamma^{B}_{3C} p^3 p^3 \hspace{1mm} \text{(summation over C=1,2)} =\frac{1}{4}(\tr\chibar)^2 (p^3)^2 =o(\modu^{-2}).   \]When combined, however, with the second term in bracket (I), the whole bracket becomes \[ (I) = \frac{1}{4} (p^3)^2 \big( \tr\chibar + \frac{2}{\modu} \big) \big(\tr\chibar - \frac{2}{\modu }\big){\delta_A}^B +o(\modu^{-3}) =o(\modu^{-3}).       \]All the other terms in the commutator are also $o(\modu^{-3})$. As a consequence \be [X,V_{(A)}]^{4+B} = o(\modu^{-3}) \hspace{3mm} \text{(in the C basis)}, \ee
which is $o(\modu^{-2})$ in the $F-$basis, where $F_{4+A}= \frac{1}{\modu}\partial_{\pbar^A}$. 

\vspace{3mm}
\noindent For the seventh term, we also have

 \begin{align} [X, Hor(e_A)-\frac{p^3}{\modu}\partial_{\pbar^A}]^{7}&= -X(\Gamma^3_{A\nu}p^{\nu}) +Hor(e_A)(\Gamma^3_{\mu\nu}p^{\mu}p^{\nu}) -\frac{p^3}{\modu}\partial_{\pbar^A}(\Gamma^3_{\mu\nu}p^{\mu}p^{\nu}) \notag \\ &= \underbrace{(e_A(\Gamma^3_{\mu\nu})-e_{\mu}(\Gamma^3_{A\nu}))\hsp p^{\mu}p^{\nu}}_\text{I}+\Gamma^C_{\mu\nu} \hsp p^{\mu}\hsp p^{\nu} \partial_{\pbar^C}(p^\lambda)\Gamma^3_{A\lambda} + \Gamma^3_{\mu\nu}p^{\mu}p^{\nu} \partial_{\pbar^3}(p^\lambda) \Gamma^3_{A\lambda} \notag \\ &\underbrace{-\Gamma^C_{A\lambda}p^{\lambda}\Gamma^3_{\mu\nu} \partial_{\pbar^C}(p^{\mu}p^{\nu})}_\text{II}-\Gamma^3_{A\lambda}p^{\lambda}\Gamma^3_{\mu\nu}\partial_{\pbar^3}(p^{\mu}p^{\nu}) \underbrace{-\frac{p^3}{\modu} \partial_{\pbar^A}(p^{\mu}p^{\nu})\Gamma^3_{\mu\nu}}_\text{II}= o(\modu^{-2}).\end{align}

\noindent This follows similarly after writing \[e_A(\Gamma^3_{\mu\nu})-e_{\mu}(\Gamma^3_{A\nu})={R^3}_{\nu A\mu} -\Gamma^{\epsilon}_{\mu\nu}\Gamma^3_{A\epsilon} +\Gamma^{\epsilon}_{A\nu}\Gamma^3_{\mu\epsilon}+(\Gamma^{\epsilon}_{A\mu}-\Gamma^{\epsilon}_{\mu A})\Gamma^3_{\epsilon\nu}.\] and estimating term by term.    We conclude that

\begin{equation} \label{VA}
\lvert X V_{A}f(x,p)\rvert = \lvert \hsp [X, V_{(A)}](f)(x,p) \rvert \leq \frac{1}{\modu^{2}}\sum_{j=1}^7 \lvert F_j(f)(x,p)\rvert \leq \frac{V}{\modu^{2}}. 
\end{equation}
For $V_{(3)} = Hor(e_4)$, we have

\be \Hor(e_4) = \big(0,0,\frac{1}{\Omega^2}, -\Gamma^A_{4\nu}p^{\nu},-\Gamma^3_{4\nu}p^{\nu}\big). \ee Hence, we estimate

\begin{align}\notag
[X,V_{(3)}]^B =& -\Hor(e_4)(p^3 b^B +p^B)= -p^3 e_4(b^B) + \Gamma^B_{4\nu}p^{\nu} +  \Gamma^3_{4\nu}p^{\nu}b^B \notag \\ =& -p^3 \big( 2(\etabar^B - \eta^B) +{\chihat_C}^{B}b^C+ \frac{1}{2}\tr\chi \hsp  b^B - {\chi_C}^B b^C \big) + {\chi_A}^B p^A +2 \etabar^B p^3 \notag \\ =& 2\eta^B \hsp p^3 +{\chihat_{A}}^B p^A +\frac{1}{2}\tr\chi \hsp p^B,
\end{align}
\begin{align}
[X,V_{(3)}]^3 = -\Hor(e_4)(p^3)= \Gamma^3_{4\nu}p^{\nu} = 0.
\end{align}For the fourth component, we have \begin{align} [X,V_{(3)}]^4 =& X\big(\frac{1}{\Omega^2}\big)  -\Hor(e_4)\big(\frac{p^4}{\Omega^2}\big) = p^A e_A\big(\frac{1}{\Omega^2}\big)+ p^3 \hsp e_3\big(\frac{1}{\Omega^2}\big) +\bcancel{p^4 e_4\big(\frac{1}{\Omega^2}\big)}\notag \\ -&\bcancel{p^4 e_4\big(\frac{1}{\Omega^2}\big)}+\frac{1}{\Omega^2}\big(\Gamma^C_{4\nu}p^\nu \frac{\gslash_{CD}p^D}{2p^3}-\Gamma^3_{4\nu}p^{\nu}\frac{p^4}{p^3}\big)= -\frac{1}{\Omega^2} (\eta_A+\etabar_A)p^A \\ -&\frac{1}{\Omega^2}\big[\omegabar-\omegabar\big(1-\frac{1}{\Omega^2}\big)\big]\hsp p^3 +\frac{1}{\Omega^2} \big({\chi_B}^C \hsp p^B \frac{\gslash_{CD}p^D}{2p^3} + 2 \etabar_D p^D\big). \end{align}

\noindent Moving on, we have

\begin{align}
[X, V_{(3)}]^{4+B} =& X(-\Gamma^B_{4\nu}p^{\nu}) + Hor(e_4)(\Gamma^B_{\mu\nu}p^{\mu}p^{\nu}) \notag \\ =& (e_4(\Gamma^B_{\mu\nu})-e_{\mu}(\Gamma^B_{4\nu}))p^{\mu}p^{\nu}+ \Gamma^C_{\mu\nu}p^{\mu}p^{\nu} \Gamma^B_{4\lambda}\partial_{\pbar^C}(p^{\lambda})+\Gamma^3_{\mu\nu}p^{\mu}p^{\nu} \Gamma^B_{4\lambda}\partial_{\pbar^3}(p^{\lambda}) \notag \\ -&\Gamma^C_{4\lambda}\hsp \Gamma^B_{\mu\nu}\hsp p^{\lambda}\hsp \partial_{\pbar^C}(p^{\mu}p^{\nu}) -\Gamma^3_{4\lambda}\hsp \Gamma^B_{\mu\nu}\hsp p^{\lambda}\hsp \partial_{\pbar^3}(p^{\mu}p^{\nu}) \notag \\ =& \left( {R^B}_{\nu 4\mu}- \Gamma^{\epsilon}_{\mu\nu}\Gamma^{B}_{4\epsilon} +\Gamma^{\epsilon}_{4\nu}\Gamma^B_{\mu \epsilon} +(\Gamma^{\epsilon}_{4\mu}-\Gamma^{\epsilon}_{\mu 4})\Gamma^B_{\epsilon \nu}\right)p^{\mu}p^{\nu} + \Gamma^C_{\mu\nu}p^{\mu}p^{\nu} \Gamma^B_{4\lambda}\partial_{\pbar^C}(p^{\lambda}) \notag \\ +&\Gamma^3_{\mu\nu}p^{\mu}p^{\nu} \Gamma^B_{4\lambda}\partial_{\pbar^3}(p^{\lambda})  -\Gamma^C_{4\lambda}\hsp \Gamma^B_{\mu\nu}\hsp p^{\lambda}\hsp \partial_{\pbar^C}(p^{\mu}p^{\nu}) -\Gamma^3_{4\lambda}\hsp \Gamma^B_{\mu\nu}\hsp p^{\lambda}\hsp \partial_{\pbar^3}(p^{\mu}p^{\nu})  =o(\modu^{-3}). \end{align}
Finally, we have

\begin{align}  [X,V_{(3)}]^{7} =& X(-\Gamma^3_{4\lambda}p^{\lambda}) +Hor(e_4)(\Gamma^3_{\mu\nu}p^{\mu}p^{\nu}) \notag \\ =& (e_4(\Gamma^3_{\mu\nu})-e_{\mu}(\Gamma^3_{4\nu}))p^{\mu}p^{\nu} +  \Gamma^C_{\mu\nu}p^{\mu}p^{\nu} \Gamma^3_{4\lambda}\partial_{\pbar^C}(p^{\lambda})+\Gamma^3_{\mu\nu}p^{\mu}p^{\nu} \Gamma^3_{4\lambda}\partial_{\pbar^3}(p^{\lambda}) \notag \\ -&\Gamma^C_{4\lambda}\hsp \Gamma^3_{\mu\nu}\hsp p^{\lambda}\hsp \partial_{\pbar^C}(p^{\mu}p^{\nu}) -\Gamma^3_{4\lambda}\hsp \Gamma^3_{\mu\nu}\hsp p^{\lambda}\hsp \partial_{\pbar^3}(p^{\mu}p^{\nu})  \notag \\  =& \left( {R^3}_{\nu 4\mu}- \Gamma^{\epsilon}_{\mu\nu}\Gamma^{3}_{4\epsilon} +\Gamma^{\epsilon}_{4\nu}\Gamma^3_{\mu \epsilon} +(\Gamma^{\epsilon}_{4\mu}-\Gamma^{\epsilon}_{\mu 4})\Gamma^3_{\epsilon \nu}\right)p^{\mu}p^{\nu} + \Gamma^C_{\mu\nu}p^{\mu}p^{\nu} \Gamma^3_{4\lambda}\partial_{\pbar^C}(p^{\lambda}) \notag \\ +&\Gamma^3_{\mu\nu}p^{\mu}p^{\nu} \Gamma^3_{4\lambda}\partial_{\pbar^3}(p^{\lambda})  -\Gamma^C_{4\lambda}\hsp \Gamma^3_{\mu\nu}\hsp p^{\lambda}\hsp \partial_{\pbar^C}(p^{\mu}p^{\nu}) -\Gamma^3_{4\lambda}\hsp \Gamma^3_{\mu\nu}\hsp p^{\lambda}\hsp \partial_{\pbar^3}(p^{\mu}p^{\nu}).  \end{align}

\noindent Moving on to $V_{(4)}$, we have
\be V_{(4)}= \big(-\modu b^A, -\modu, 0, \modu \Gamma^A_{3\nu}p^{\nu}, \modu \Gamma^3_{3\nu}p^{\nu}+p^3\big). \ee Therefore,
\begin{align}
[X,V_{(4)}]^B =& X(-\modu b^B)-p^3 \partial_{\pbar^3}(p^3 b^B+p^B) +\modu \notag \Hor(e_3)(p^3 b^B+p^B) \\ =& -p^3 e_3(\modu b^B) -\modu p^4 e_4(b^B) - p^3 b^B +\modu p^3 e_3(b^B) -\modu (\Gamma^B_{3\nu}p^{\nu}+ \Gamma^3_{3\nu}p^{\nu}b^B) \notag \\ =& -2 \modu p^4(\etabar^B-\eta^B) -\modu (\Gamma^B_{3\nu}p^{\nu}+\Gamma^3_{3\nu}p^{\nu}b^B).
\end{align}Note two important things here. The first is the cancellation of the terms $-p^3 e_3(\modu b^B)-p^3 b^B +\modu p^3 e_3(b^B)$. It is never the case that either one of $e_3(b^B)$ or $e_4(\Omega)$, objects for which we have no structure equations and whose appearance would create a loss of derivatives, can be part of a commutation term. Finally, note that $\Gamma^B_{3A}$ appears whole and $e_A(b^B)$ does not appear by itself, as was clearly promised in the introduction. In addition, the lower order term $\Gamma^{3}_{3\nu}p^{\nu}b^{B}$ containing the shift vector field cancels out while expressed in the $V$ basis. We stress this point that working with the appropriate basis is extremely crucial. Moving on, we have

\begin{align}\notag 
[X,V_{(4)}]^3 =& X(-\modu)-(p^3\partial_{\pbar^3}-\modu\Hor(e_3))(p^3) = p^3-p^3 +\modu \Hor(e_3)(p^3) \\ =& \modu \Gamma^3_{3\nu}p^{\nu} = \modu \eta_A p^A + \modu\hsp  \omegabar \hsp p^3,
\end{align}

\begin{align}
[X,V_{(4)}]^4 =& -(p^3 \partial_{\pbar^3}-\modu\Hor(e_3))\big(\frac{p^4}{\Omega^2}\big) = \frac{p^4}{\Omega^2} +\frac{\modu}{\Omega^2}\bigg[ -\omegabar \hsp p^4 -(\Gamma^A_{3B}p^B +2 \eta^A p^4)\frac{\gslash_{AD}p^D}{2p^3}\bigg]\notag \\ +&\frac{\modu}{\Omega^2} (\eta_A p^A +\omegabar \hsp p^3)\cdot \frac{p^4}{p^3}.
\end{align}
For the next two components, we have

\begin{align} [X, V_{(4)}]^{4+B} =& X(\modu \Gamma^B_{3\nu}p^{\nu}) - \Gamma^B_{\mu\nu}p^3 \partial_{\pbar^3}(p^{\mu}p^{\nu}) + \modu e_3(\Gamma^B_{\mu\nu})p^{\mu}p^{\nu}\notag \\ -& \modu \Gamma^B_{\mu\nu} \Gamma^C_{3\lambda}p^{\lambda} \partial_{\pbar^C}(p^{\mu}p^{\nu}) - \modu \Gamma^B_{\mu\nu} \Gamma^3_{3\lambda}p^{\lambda} \partial_{\pbar^3}(p^{\mu}p^{\nu}) \notag \\=& \modu \big( \left(e_{\mu}(\Gamma^B_{3\nu})-e_3(\Gamma^B_{\mu\nu})\right) p^{\mu}p^{\nu} -\Gamma^C_{\mu\nu}p^{\mu}p^{\nu}\partial_{\pbar^C}\left(p^{\lambda}\right) \Gamma^B_{3\lambda} -\Gamma^3_{\mu\nu}p^{\mu}p^{\nu}\partial_{\pbar^3}\left(p^{\lambda} \right) \Gamma^3_{3\lambda}\big)
\notag \\ \notag -& \frac{p^3}{\Omega^2}\Gamma^B_{3\nu}p^{\nu} + \Gamma^B_{\mu\nu}p^3 \partial_{\pbar^3}(p^{\mu}p^{\nu}) + \modu \Gamma^B_{\mu\nu} \left(\Gamma^C_{3\lambda}p^{\lambda}\partial_{\pbar^C}(p^{\mu}p^{\nu}) +\Gamma^3_{3\lambda}p^{\lambda}\partial_{\pbar^3}(p^{\mu}p^{\nu})\right)\\ =& \modu ( {R^B}_{\nu \mu 3}- \Gamma^{\epsilon}_{3\nu}\Gamma^{B}_{\mu\epsilon} +\Gamma^{\epsilon}_{\mu\nu}\Gamma^B_{3 \epsilon} +(\Gamma^{\epsilon}_{\mu3}-\Gamma^{\epsilon}_{3 \mu})\Gamma^B_{\epsilon \nu}) - \modu \Gamma^C_{\mu\nu}p^{\mu}p^{\nu}\partial_{\pbar^C}(p^{\lambda}) \Gamma^B_{3\lambda}\notag \\ -& \modu \Gamma^3_{\mu\nu}p^{\mu}p^{\nu}\partial_{\pbar^3}(p^{\lambda})\Gamma^3_{3\lambda} - \frac{p^3}{\Omega^2}\Gamma^B_{3\nu}p^{\nu} + \Gamma^B_{\mu\nu}p^3 \partial_{\pbar^3}(p^{\mu}p^{\nu}) + \modu \Gamma^B_{\mu\nu} \Gamma^C_{3\lambda}p^{\lambda}\partial_{\pbar^C}(p^{\mu}p^{\nu}) \notag \\ +& \modu \Gamma^3_{3\lambda}p^{\lambda}\partial_{\pbar^3}(p^{\mu}p^{\nu}),
\end{align}

\begin{align}
[X, V_{(4)}]^7 =& X(p^3 +\modu \Gamma^3_{3\nu}p^{\nu}) +p^3\partial_{\pbar^3}(\Gamma^3_{\mu\nu}p^{\mu}p^{\nu}) -\modu Hor(e_3)(\Gamma^3_{\mu\nu}p^{\mu}p^{\nu}) \notag \\ =& -\Gamma^3_{\mu\nu}p^{\mu}p^{\nu} -\frac{p^3}{\Omega^2} \Gamma^3_{3\nu}p^{\nu} + \modu \big( e_{\mu}(\Gamma^3_{3\nu})p^{\mu}p^{\nu}- \Gamma^C_{\mu\nu}p^{\mu}p^{\nu} \Gamma^3_{3\lambda}\partial_{\pbar^C}(p^{\lambda}) -\Gamma^3_{\mu\nu}p^{\mu}p^{\nu}\Gamma^3_{3\lambda}\partial_{\pbar^3}(p^{\lambda})\big)\notag \\ +& p^3\Gamma^3_{\mu\nu}\partial_{\pbar^3}(p^{\mu}p^{\nu})- \modu \big( e_3(\Gamma^3_{\mu\nu}) p^{\mu}p^{\nu}- \Gamma^C_{3\lambda}p^{\lambda} \Gamma^3_{\mu\nu}\partial_{\pbar^C}(p^{\mu}p^{\nu}) - \Gamma^3_{3\lambda}p^{\lambda}\Gamma^3_{\mu\nu}\partial_{\pbar^3}(p^{\lambda})\big) .
\end{align} For the last two vector fields $V_{(4+A)}= \frac{p^3}{\modu^2}\partial_{\pbar^A}$, we have

\be V_{(4+A)}= (0,0,0,\frac{p^3}{\modu^2},0), \ee hence direct computation yields

\be [X,V_{(4+A)}]^B = -\frac{p^3}{\modu^2}\partial_{\pbar^A}(p^3 b^B+p^B) = -\frac{p^3}{\modu^2}{\delta_A}^B, \ee 

\be [X,V_{(4+A)}]^3 = -\frac{p^3}{\modu^2}\partial_{\pbar^A}(p^3) =0, \ee

\be [X,V_{(4+A)}]^4 = -\frac{p^3}{\modu^2}\partial_{\pbar^A}\big(\frac{p^4}{\Omega^2}\big) = -\frac{p^3}{\modu^2}\frac{\gslash_{AD}p^D}{2\hsp \Omega^2\hsp p^3} , \ee

\be [X,V_{(4+A)}]^{4+B} = X\bigg(\frac{p^3}{\modu^2}{\delta_A}^B\bigg) +\frac{p^3}{\modu^2}\Gamma^B_{\mu\nu}\partial_{\pbar^A}(p^{\mu}p^{\nu}) = \frac{{\delta_A}^B}{\modu^2}\big(-\Gamma^3_{\mu\nu}p^{\mu}p^{\nu} +\frac{2 \hsp (p^3)^2}\modu\big) +\Gamma^B_{\mu\nu} \frac{p^3}{\modu^2}\partial_{\pbar^A}(p^{\mu}p^{\nu}).\ee     

\be [X, V_{(4+A)}]^7 = -\frac{p^3}{\modu^2}\partial_{\pbar^A}(\Gamma^3_{\mu\nu}p^{\mu}p^{\nu}) = o(\modu^{-2}). \ee
\noindent Notice the disappearance of the $\frac{\partial}{\partial u}$ component and hence no extra shift $b$ term.  Now, use the following relation between the $C$ basis (lifted up to $\mathcal{P}$) and the $V$ basis, for $A=1,2$: 
\begin{align} e_A =& V_{(A)}+\modu^2\bigg(\frac{\Gammaslash^B_{AC}p^C}{p^3}+{\chihat_{A}}^B +\frac{{\chi_A}^B p^4}{p^3}\bigg)V_{(4+B)} \notag \\ +& \frac{\modu^2}{2}\tildetr V_{(4+A)} +\bigg(\frac{1}{2}\frac{\chi_{AB}p^B}{p^3}-\etabar_A\bigg)(V_{(4)}+V_{(0)}), \end{align}
\be \modu e_3 = \big(1+\modu \frac{\eta_A p^A}{p^3}+\modu \hsp \omegabar \big)V_{(0)}+\modu^2\bigg(\frac{({\chibar_B}^A-e_B(b^A))p^B +2\eta^A \hsp p^4}{p^3}\bigg)\modu V_{(4+A)} +\modu \hsp \omegabar \hsp  V_{4}, \ee
\be
e_4 = V_{(3)} + \modu^2\left( \frac{{\chi_B}^A p^B}{p^3} + 2\hsp \etabar^A \right)V_{(4+A)},
\ee
\be
\frac{\partial}{\partial \pbar^{A}}=\frac{|u|^{2}}{p^{3}}V_{(4+A)},~\frac{\partial}{\partial \pbar^{3}}=V_{(4)}+V_{(0)}.
\ee
This completes the proof of the commutation calculus. 
\end{proof}

\begin{remark}[Appearance of curvature]
First, note the most important observation made in this study. As pointed out by Taylor \cite{taylor}, there appears to be a significant challenge in closing the regularity while deriving higher order estimates for the Vlasov distribution function $f$. Recall that the main ingredient that one possesses to estimate $f$ is the transport equation
\begin{align}
 X[f]=0.   
\end{align}
Consider estimating the third derivative of $ f$; this is the desired regularity in the present framework. Taylor \cite{taylor} observed that commutation of this equation with a vector field $V$ introduces $V[\Gamma]$, introducing $V[e(b)]$, which requires control of four derivatives of the shift vector field $b$. Consequently, this would require control of four derivatives of the connection coefficients since $b$ verifies an equation of the type 
\begin{align}
 \snabla_{4}b\sim (\eta,\etabar)+\chi\cdot b.   
\end{align}
But this is impossible to close since the best we can do for the connection coefficients is to control them on null hypersurfaces up to three derivatives by means of elliptic estimates. This appears to be an obstruction to directly commuting the transport equation, even though heuristically speaking, four derivatives of $b$ are at the same level as two derivatives of curvature-note components of $b$ are essentially metric components. A similar phenomenon occurs for the connection coefficients $\Gammaslash$ tangent to the $2-$spheres.  However, Taylor's observation ignores the full structure of the error terms-which he points out. We make the vital observation that the derivatives of the connection coefficients always appear in a combination that produces the curvature. For example, in the expression of $[X,V_{(4)}]$, one obtains $e_{\mu}(\Gamma^{3}_{3\nu})-e_{3}(\Gamma^{3}_{\mu\nu})$. Note that this combination precisely yields the curvature through the elementary identity 
\begin{align}
e_{\mu}(\Gamma^{3}_{3\nu})-e_{3}(\Gamma^{3}_{\mu\nu})=R^{3}~_{\nu \mu 3}-\Gamma^{\lambda}_{3\nu}\Gamma^{3}_{\mu\lambda}+\Gamma^{\lambda}_{\mu\nu}\Gamma^{3}_{3\lambda}+(\Gamma^{\lambda}_{\mu 3}-\Gamma^{\lambda}_{3\mu})\Gamma^{3}_{\lambda\nu}.
\end{align}
Therefore, the undesirable terms of the type  $V(e(b))$ or $V(\Gammaslash)$ do not appear, and terms of the type $\Gamma \hsp\Gamma$ enjoy a higher decay property, or more precisely, integrable decay. This is pivotal throughout the present work.
    
\end{remark}

\begin{remark}
We note another vital point about the structure of the commutators in proposition \ref{commutator}. Notice that the shift vector field $b$ always appears in the combination $e_{A}(b^{B})p^{A}+\Gamma^{B}_{3A}p^{A}$ or as $-e_{A}(b^{B})p^{A}+\Gamma^{B}_{A3}p^{A}$ (or instead of $p^A$, any other momentum variable $p^{\alpha}$ appears in multiplication). Note that the term of the first type $e_{A}(b^{B})p^{A}+\Gamma^{B}_{3A}p^{A}$  simplifies to $\chibar^{B}_{A}\hsp p^A$ while the second term $-e_{A}(b^{B})p^{A}+\Gamma^{B}_{A3}p^{A}$ simply yields $\Gamma^{B}_{3A}p^{A}$. Luckily, we have separate $\snabla_{4}$ equations for the trace and tracefree part of $\Gamma^{B}_{A3}$. In addition, $\snabla_{4}b$ term can be expressed in terms of the Ricci coefficients by means of the transport equation for $b$.  
\end{remark}

\subsection{First derivative control of Vlasov}\label{firstvlasovs}
Now we prove the first-order energy estimates for the Vlasov field. Motivated by the  expression for the total first-order Vlasov norm $\mathfrak{V}_{1}$, we define certain $u-$dependent Vlasov norms. The first one is as follows:
\begin{eqnarray}
\label{eq:first} \notag 
&&\mathfrak{V}_{1}(u):=\frac{|u|^{2}}{a^{2}}\Bigg(\sum_{A=1,2}\sup_{\ubar}\sup_{S_{u,\ubar}}\int_{\mathcal{P}_{x}}|V_{A}f|^{2}\sqrt{\det\gslash}\frac{dp^{1}dp^{2}dp^{3}}{p^{3}}+\sup_{\ubar}\sup_{S_{u,\ubar}}\int_{\mathcal{P}_{x}}|V_{4}f|^{2}\sqrt{\det\gslash}\frac{dp^{1}dp^{2}dp^{3}}{p^{3}}\\\nonumber &&+\sup_{\ubar}\sup_{S_{u,\ubar}}\int_{\mathcal{P}_{x}}|V_{3}f|^{2}\sqrt{\det\gslash}\frac{dp^{1}dp^{2}dp^{3}}{p^{3}}+ \sum_{A=1,2}\sup_{\ubar}\sup_{S_{u,\ubar}}\int_{\mathcal{P}_{x}}|\modu V_{4+A}f|^{2}\sqrt{\det\gslash}\frac{dp^{1}dp^{2}dp^{3}}{p^{3}}\\\nonumber 
&&+\sup_{\ubar}\sup_{S_{u,\ubar}}\int_{\mathcal{P}_{x}}||u|V_{0}f|^{2}\sqrt{\det\gslash}\frac{dp^{1}dp^{2}dp^{3}}{p^{3}} \Bigg).  
\end{eqnarray}
In addition to these uniform norms, we also need improved $L^{2}(S_{u,\ubar})$-estimates of the $Vf$. We thus introduce:
\begin{align}
\label{eq:firstL2}
&\mathfrak{V}^{s}_{1}(u) \nonumber \\\nonumber :=&\frac{1}{a^{2}}\Bigg(\sum_{A=1,2}\sup_{\ubar}\int_{S_{u,\ubar}}\Bigg(\int_{\mathcal{P}_{x}}|V_{A}f|^{2}\sqrt{\det\gslash}\frac{dp^{1}dp^{2}dp^{3}}{p^{3}}\Bigg)dS_{u,\ubar}+\sup_{\ubar}\int_{S_{u,\ubar}}\int_{\mathcal{P}_{x}}|V_{4}f|^{2}\sqrt{\det\gslash}\frac{dp^{1}dp^{2}dp^{3}}{p^{3}}dS_{u,\ubar}\\\nonumber +&\sup_{\ubar}\int_{S_{u,\ubar}}\int_{\mathcal{P}_{x}}|V_{3}f|^{2}\sqrt{\det\gslash}\frac{dp^{1}dp^{2}dp^{3}}{p^{3}}dS_{u,\ubar}+ \sum_{A=1,2}\sup_{\ubar}\int_{S_{u,\ubar}}\int_{\mathcal{P}_{x}}|uV_{4+A}f|^{2}\sqrt{\det\gslash}\frac{dp^{1}dp^{2}dp^{3}}{p^{3}}dS_{u,\ubar}\\ 
+&\sup_{\ubar}\int_{S_{u,\ubar}}\int_{\mathcal{P}_{x}}||u|V_{0}f|^{2}\sqrt{\det\gslash}\frac{dp^{1}dp^{2}dp^{3}}{p^{3}}dS_{u,\ubar} \Bigg).    
\end{align}
Also let us define $\mathfrak{V}_{1}(u_{\infty})=\mathcal{V}^{0}$ and $\mathfrak{V}^{s}_{1}(u_{\infty})=\mathcal{V}^{0}_{s}$
\begin{remark}
Note that the assumption on the initial data in Theorem \ref{mainone} implies control of both $\mathcal{V}_{1}$ and $\mathcal{V}^{S}_{1}$ at $u=u_{\infty}$.      
\end{remark}
\begin{proposition}[First derivative estimates for the distribution function]
\label{firstvlasov}
Let \( f \colon \mathcal{P} \to \mathbb{R}_{\geq 0} \) denote the distribution function associated to a solution of the Einstein--Vlasov system on a globally hyperbolic spacetime \( (\mathcal{M}, g) \), and let \( \mathfrak{V}_{1} \) denote the $L^{2}(\mathcal{P}_{x})$-norm of the first-order geometric derivatives of \( f \) along the frame-adapted vector fields $V$ as defined in \eqref{eq:first}. Let \( \mathcal{V}^{0} \) (\( \mathcal{V}^{0}_{s} \), respectively) denote the initial norm of $\mathfrak{V}_{1}$ ($\mathfrak{V}^{s}_{1}$, respectively)  on the null hypersurface $H_{u_{\infty}}$. Then, under the bootstrap assumptions described in Section~\ref{section:s31}, the following pointwise bound holds:
\[
\mathfrak{V}_{1} \lesssim 1 + \mathcal{V}^{0}, \mathfrak{V}^{s}_{1}\lesssim \mathcal{V}^{0}_{s}.
\]
\end{proposition}
\begin{proof}
We prove these estimates by using the transport inequality \ref{lem:transport} on the mass-shell. But before doing so, let us make the following bootstrap assumption on the Vlasov field. Assume that $\mathfrak{V}_{1}(u)$ and $\mathfrak{V}^{s}_{1}(u)$ verifies the following bootstrap assumptions 
\begin{align}
\label{eq:intboot1}
 \mathfrak{V}_{1}\leq \mathcal{V},~\mathfrak{V}^{s}_{1}\leq \mathcal{V}_{1}   
\end{align}
where $\mathcal{V}$ ($\mathcal{V}_{1}$, respectively) is large enough so that $\mathcal{V}^{0}\ll \mathcal{V}$ ($\mathcal{V}^{0}_{s}\ll \mathcal{V}_{1}$, respectively) but also $\mathcal{V}\ll a^{\frac{1}{10}}$ ($\mathcal{V}_{1}\ll a^{\frac{1}{10}}$, respectively). We only consider the cases for the vector fields that contain terms that are borderline dangerous. The estimates for the remaining vector fields follow in the exact same manner. More precisely, we consider one estimate in the unweighted category (e.g., for the vector field $V_{(A)}$) and one estimate for the weighted category (e.g., for the vector field $V_{(4+A)}$). Now we use the transport inequality from Lemma \ref{lem:transport} to the entity $\frac{a^{2}}{|u|^{2}}\mathfrak{V}_{1}$. First, consider $\mathfrak{V}^{1,2}_{1}:=\sum_{A=1,2}\int_{\mathcal{P}_{x}}|V_{A}f|^{2}\sqrt{\det\gslash}\frac{dp^{1}dp^{2}dp^{3}}{p^{3}}$. Now recall the transport lemma \ref{lem:transport} and set $A(s):=\mathfrak{V}^{1,2}_{1}(s)$ to yield 
\begin{align}
\mathcal{V}^{1,2}_{1}(s)\lesssim \mathcal{V}^{1,2}_{1}(s_{\infty})+\int_{s_{\infty}}^{s}\Bigg(\int_{p}V_{A}f X[V_{A}f]\sqrt{\det(\gslash)}\frac{dp^{1}dp^{2}dp^{3}}{p^{3}}\Bigg)ds^{'}      
\end{align}
and estimate the error term involving $X[V_{A}f]$. \begin{eqnarray}
[X,V_{(A)}]f=\sum_{J=0}^{6}C^{J}_{A}V_{J}f    
\end{eqnarray}
The coefficients $C^{J}_{A}$ of the error terms can be obtained from the proposition \ref{commutator}. First, recall the expression for $C^{0}_{B},~B=1,2$
\begin{align}
C^{0}_{B}&=\big[-p^3 \hsp e_A(b^B)+\big(\Gamma^B_{A\nu}p^{\nu}+\frac{p^3}{\modu}{\delta_A}^B\big)\big]\big(\frac{\frac{1}{2}\chi_{BC}p^C}{p^3}-\etabar_B\big) \notag\\ +&\frac{1}{\modu}\big(2\hsp e_A(\log \Omega)\hsp p^3 + \Gamma^3_{A\nu}p^{\nu}\big) \big(1+\modu \frac{\eta_C p^C}{p^3}+|u|\omegabar\big) \notag \\ +& \bigg((e_A(\Gamma^3_{\mu\nu})-e_\mu (\Gamma^3_{A\nu}))p^{\mu}\hsp p^{\nu} +\Gamma^C_{\mu\nu}p^{\mu}p^{\nu}\Gamma^{3}_{A\lambda}\partial_{\pbar^C}(p^{\lambda}) + \Gamma^3_{\mu\nu}p^{\mu}p^{\nu}\Gamma^3_{A\lambda} \partial_{\pbar^3}(p^{\lambda}) \notag \\ -&\Gamma^C_{A\lambda}p^{\lambda}\Gamma^3_{\mu\nu}\partial_{\pbar^C}(p^{\mu}p^{\nu})-\Gamma^3_{A\lambda}p^{\lambda} \Gamma^3_{\mu\nu}\partial_{\pbar^3}(p^{\mu}p^{\nu}) +\frac{p^3}{\modu}\Gamma^3_{\mu\nu}\partial_{\pbar^A}(p^{\mu}p^{\nu}) \bigg)\frac{1}{p^3}.
\end{align}

\noindent Fix a double null foliation $(u,\ubar,\theta^A)$ and the corresponding sphere $S_{u,\ubar}$ with induced metric $\gslash_{AB}$. We use frame vectorfields $e_3,e_4,e_A$ as in the main text. Greek indices $\mu,\nu,\dots$ run over $\{1,2,3,4\}$; capital Latin indices $A,B,C$ run over $\{1,2\}$ and are raised and lowered with $\gslash$. The momentum fibre coordinates are $(p^1,p^2,p^3)$ with $p^4=p^4(p^1,p^2,p^3)$ determined by the mass-shell constraint; we abbreviate $\pslash=(p^1,p^2)$. The momentum integration on each fibre is
\[
\int_{\mathcal{P}} \phi(p)\,dp := \int_{\mathbb R^3} \phi(p)\sqrt{\det\gslash}\,\frac{dp^1dp^2dp^3}{p^3}\,,
\]
as in the body of the paper. We assume the bootstrap/pointwise bounds on the Ricci coefficients $\psi$ and curvature components $\Psi$ recorded elsewhere.

\noindent Under the hypotheses above, there exists a finite collection of smooth scalar functions $\{f^i_j\}$ of the indicated geometric variables such that, schematically,
\begin{equation}\label{C0-decomp}
\begin{split}
C^{0}_{B}
&=\sum_{i} f^{i}_{1}(\psi)\, f^{i}_{2}(\Psi)\, f^{i}_{3}(p^{4})\, f^{i}_{4}(\pslash)\, f^{i}_{5}(\gslash)\, f^{i}_{6}(\gslash^{-1}) \\
&\quad + \;\gslash\pslash\Bigg( p^{4}\int_{P} f\,p^{3}\,dp \;+\; p^{3}\int_{P} f\,p^{4}\,dp \;+\; \pslash\int_{P} f\,\gslash\pslash\,dp \Bigg)\\
&\quad +\; p^{4}p^{3}\int_{P} f\,\gslash\pslash\,dp 
\;+\; p^{3}p^{3}\int_{P} f\,p^{4}\gslash\pslash\,dp
\;+\; \pslash\,p^{3}\int_{P} f\,\gslash\pslash\,\gslash\pslash\,dp,
\end{split}
\end{equation}
where each $f^{i}_{1}(\psi)$ is a polynomial in the Ricci coefficients
\[
\psi\in\{\chihat,\tr\chi,\chibarhat,\tr\chibar,\omegabar,\eta,\etabar\},
\]
each $f^{i}_{2}(\Psi)$ is a polynomial in the Weyl components
\[
\Psi\in\{\alpha,\beta,\rho,\sigma,\betabar,\alphabar\},
\]
and the remaining $f^{i}_{j}$ are polynomial functions of the indicated fiber or metric variables.

\noindent Moreover, each term in \eqref{C0-decomp} satisfies the pointwise decay bounds (for $|u|$ large)
\[
\big| \text{(each algebraic term in the first line)} \big| \;\lesssim\; \Gamma\,|u|^{-2}\cdot \mathcal{P}(p^3,p^4,\pslash),
\]
and every term involving a momentum integral is bounded by
\[
\big| \text{(each integral term)}\big| \;\lesssim\; \Gamma\,|u|^{-2}\cdot \mathcal{Q}(p^3,p^4,\pslash),
\]
where $\mathcal{P}$ and $\mathcal{Q}$ are explicit polynomial factors in the fiber variables which are uniformly integrable in $p$ given the decay of $f$ on the mass shell. In particular, all terms on the right-hand side of \eqref{C0-decomp} are $u$-integrable and therefore do not cause loss of regularity in the energy estimates.

Differentiate the connection coefficients in antisymmetric combinations and commute derivatives exactly as in the introduction. For example, for fixed indices $A,\mu,\nu$,
\[
e_A(\Gamma^{3}_{\mu\nu}) - e_\mu(\Gamma^{3}_{A\nu})
= R^{3}{}_{\nu A\mu}
- \Gamma^{\lambda}{}_{\mu\nu}\Gamma^{3}{}_{A\lambda}
+ \Gamma^{\lambda}{}_{A\nu}\Gamma^{3}{}_{\mu\lambda}
+(\Gamma^{\lambda}{}_{A\mu}-\Gamma^{\lambda}{}_{\mu A})\Gamma^{3}{}_{\lambda\nu}.
\]
The first term on the right is a curvature component and the remaining terms are quadratic in the connection coefficients. Using the Ricci decomposition of $R$ into Weyl plus trace parts, the curvature term is expressed as a combination of Weyl components $\Psi$ and matter stress-energy (which yields the momentum integrals $\int_P f\cdot\,$). This produces the schematic factor $f^i_2(\Psi)$ and the momentum-integral terms appearing in \eqref{C0-decomp}.

Expanding the quadratic connection terms and collecting factors yields finite sums of products, each of which is a polynomial in the Ricci coefficients $\psi$, possibly multiplied by smooth functions of $\gslash,\gslash^{-1},\pslash,p^4$. This produces the first line of \eqref{C0-decomp}.

\noindent The matter terms arising from the Ricci contraction of $R$ contribute expressions of the form (schematically)
\[
\int_P \Phi(\psi,\gslash)\, f(p)\, \mathfrak{m}(p)\,dp,
\]
where $\mathfrak m(p)=\sqrt{\det\gslash}\,p^{-3}$ is the geometric fibre measure and $\Phi$ depends smoothly on $\psi$ and $\gslash$. After pulling the fibre factors outside and expanding in $\pslash,p^3,p^4$ one obtains the integrals that appear in the second and third lines of \eqref{C0-decomp}.

\noindent The potentially non-integrable expression,
\[
-p^{3}e_A(b^B) + \Gamma^{B}{}_{A\nu}p^\nu + \frac{p^{3}}{|u|}\delta^B_A,
\]
is algebraically rearranged by expanding $\Gamma^{B}{}_{A\nu}p^\nu$ into its frame components:
\[
\Gamma^{B}{}_{A\nu}p^\nu
= \Gamma^{B}{}_{A4}p^4 + \Gammaslash^{B}{}_{AC}p^{C} + \Gamma^{B}{}_{A3}p^3.
\]
Recalling the definition
\[
\mathcal G^{A}{}_{B} := \Gamma^{A}{}_{3B} + e_{A}(b^{B}) \quad\text{(equivalently }\Gamma^{A}{}_{3B}=\Gamma^{A}{}_{B3}-e_A(b^B)\text{)},
\]
one obtains the identity
\[
-p^{3}e_A(b^B) + \Gamma^{B}{}_{A\nu}p^\nu + \frac{p^{3}}{|u|}\delta^B_A
= \Gamma^{B}{}_{A4}p^4 + \Gammaslash^{B}{}_{AC}p^C + \widehat{\mathcal G}^{B}{}_{A} + \tfrac12\Big(\tr\mathcal G + \tfrac{2}{|u|}\Big)\delta^B_A,
\]
where $\widehat{\mathcal G}=\mathcal G-\tfrac{1}{2}(\tr\mathcal G)\,\mathrm{Id}$ denotes the trace-free part of $\mathcal G$.

\noindent By the bootstrap/pointwise estimates on the Ricci coefficients, one has
\be \label{eq:cancellation}
\big|\tr\mathcal G + \tfrac{2}{|u|}\big| \;\lesssim\; \Gamma\,|u|^{-2}.
\ee
Thus the dangerous $\tfrac{p^{3}}{|u|}$ contribution is exactly cancelled, leaving only $\lesssim\Gamma|u|^{-2}$. The remaining terms inherit at least $|u|^{-2}$ decay from the bootstrap hypotheses. Therefore, every term is $u$-integrable and does not produce the non-integrable obstruction originally feared.

\noindent Collecting the estimates yields the decomposition \eqref{C0-decomp} together with the stated decay bounds. Consequently, the coefficient $C^{0}_{B}$ is composed of finitely many algebraic polynomial-type terms and momentum-integral terms, each of which decays sufficiently rapidly in $|u|$ (and is uniformly integrable in the fibre variable) to avoid loss of regularity in the subsequent energy estimates.

\noindent We need to be careful here since we can not use the scale-invariant Holder inequality in estimating the product terms, since we do not know the signature of the Vlasov momentum variables $\{p^{4},p^{3},p^{A},p^{B}\}$. Therefore, we convert the scale-invariant estimates of the Ricci coefficients and the Weyl curvature components to the unscaled form and use these to control the Vlasov norms. First, note the following for a Ricci coefficient belonging to the category $\psi_{g}$   
\begin{eqnarray}
||\omegabar||_{L^{\infty}(S_{u,\ubar})}\lesssim \frac{a\Gamma}{|u|^{3}},~ ||\tr\chi||_{L^{\infty}(S_{u,\ubar})}\lesssim \frac{\Gamma}{|u|},~||\eta||_{L^{\infty}(S_{u,\ubar})}\lesssim \frac{a^{\frac{1}{2}}}{|u|^{2}}\Gamma,~||\etabar||_{L^{\infty}(S_{u,\ubar})}\lesssim \frac{a^{\frac{1}{2}}}{|u|^{2}}\Gamma.   
\end{eqnarray}
and the following pointwise estimates (of the tensor norm) 
\begin{eqnarray}
\nonumber |\psi_{g}|\lesssim \frac{\Gamma}{|u|},~|\gslash^{-1}|\lesssim 1,~|\chibarhat|\lesssim \frac{a^{\frac{1}{2}}\Gamma}{|u|^{2}},~|\widetilde{\tr\chibar}|\lesssim \frac{\Gamma}{|u|^{2}},~|\chihat|\lesssim \frac{\Gamma a^{\frac{1}{2}}}{|u|},~|p^{4}|\lesssim \frac{1}{|u|^{2}},~|p^{3}|\lesssim 1,~|\Gammaslash|\lesssim \frac{a^{\frac{1}{2}}}{|u|^{2}},\\\nonumber 
 |\pslash|\lesssim \frac{1}{|u|},~|\gslash|\lesssim 1,~|f|\lesssim a, |p^{1,2}|\lesssim \frac{1}{|u|^{2}},~|\alpha|\lesssim \frac{\mathcal{R} a^{\frac{1}{2}}}{|u|},~|\beta|\lesssim \frac{a^{\frac{1}{2}}\mathcal{R}}{|u|^{2}},~|\betabar|\lesssim \frac{a^{\frac{3}{2}}\mathcal{R}}{|u|^{4}},~|(\rho,\sigma)|\lesssim \frac{a\mathcal{R}}{|u|^{3}},~|\alphabar|\lesssim \frac{a^{2}\mathcal{R}}{|u|^{5}}.
\end{eqnarray}
This yields the following component-wise and point-wise estimates 
\begin{align} \label{boundsbootstrapnew1}
\lvert \chihat_{AB} \nonumber\rvert \lesssim \Gamma\al\lvert u \rvert, \hspace{3mm} \lvert \tr \chi \rvert \lesssim \frac{\Gamma}{\lvert u\rvert}, \hspace{3mm} \lvert \chibarhat_{AB} \rvert \lesssim \al \Gamma , \hspace{3mm} \lvert \tr\chibar \rvert \lesssim \frac{\Gamma}{\lvert u \rvert}, \\ \lvert {\chibarhat_A}^B \rvert \lesssim \frac{\al\Gamma}{\modu^2}, \hspace{3mm} \lvert \etabar_A \rvert \lesssim \frac{\al\Gamma}{\modu}, \hspace{3mm} \lvert {\chihat_{A}}^B \rvert \lesssim \frac{\al\Gamma}{\modu},\hspace{3mm} \lvert e_A(b^B) \rvert \lesssim \frac{\Gamma}{\lvert u \rvert^2}, \\ \nonumber \lvert \eta^A \rvert \lesssim \frac{\al\Gamma}{\lvert u \rvert^3}, \hspace{3mm} \lvert \omega \rvert \lesssim \frac{\Gamma}{\modu}, |\beta_{A}|\lesssim \frac{a^{\frac{1}{2}}\mathcal{R}}{|u|},~|\alpha_{AB}|\lesssim a^{\frac{1}{2}}|u|\mathcal{R},~ |\betabar_{A}|\lesssim \frac{a^{\frac{3}{2}}\mathcal{R}}{|u|^{3}},\\
|(\rho,\sigma)|\lesssim \frac{a\mathcal{R}}{|u|^{3}},~|\alphabar_{AB}|\lesssim \frac{a^{2}\mathcal{R}}{|u|^{3}}
\end{align}
along with the basic estimates 
\be   \frac{c}{\lvert u \rvert^2} \leq  \gslash^{AB}(u,\ubar, \theta^1, \theta^2) \leq \frac{C}{\lvert u \rvert^2}, \hspace{3mm} c \hsp \lvert u \rvert^2 \leq \gslash_{AB}(u,\ubar, \theta^1, \theta^2) \leq C \lvert u \rvert^2. \ee 
We will use the component-wise estimates instead of tensor norms. This is because we estimate the components of several derivatives of the Vlasov distribution function $f$. For example, we will estimate $V_{A}V_{B}f$ component-wise (in the frame that we are explicitly working) and while the tensor norm of the tensor $\xi_{AB}:=\int_{\mathcal{P}_{x}}V_{A}V_{B}f$ would be defined as follows 
\begin{align}
 ||\xi||:=\sqrt{\gslash^{AC}\gslash^{BD}\xi_{AB}\xi_{CD}}.   
\end{align}
Using these facts, it is easy to observe that the term $\sum f^{i}_{1}(\psi)f^{i}_{2}(\Psi)f^{i}_{3}(p^{4})f^{i}_{4}(\pslash)f^{i}_{5}(\gslash)f^{i}_{6}(\gslash^{-1})$ enjoys integrable decay and more precisely verifies a bound of the following type 
\begin{align}
|\sum f^{i}_{1}(\psi)f^{i}_{2}(\Psi)f^{i}_{3}(p^{4})f^{i}_{4}(\pslash)f^{i}_{5}(\gslash)f^{i}_{6}(\gslash^{-1})|\lesssim \frac{a^{\frac{1}{2}}\Gamma}{|u|^{2}}
\end{align}
where we have noted the worst possible decay. In addition, the stress-energy tensor part of $C^{0}_{B}$ decays faster and obeys 
\begin{align} \notag
|\gslash\pslash\left(
p^{4}\int_{p}fp^{3}dp + 
p^{3}\int_{p}fp^{4}dp + 
\pslash\int_{p}f\gslash\pslash dp
\right)+ p^{4}p^{3}\int_{p}f\gslash \pslash dp 
+ p^{3}p^{3}\int_{p}f p^{4}\gslash \pslash dp 
+ \pslash p^{3}\int_{P}f\gslash \pslash \gslash\pslash dp|\\
\lesssim \frac{a}{|u|^{4}}.
\end{align}
Therefore, the decay estimate verifies by $C^{0}_{B}$ reads 
\begin{align}
 |C^{0}_{B}|\lesssim  \frac{a^{\frac{1}{2}}(\Gamma,\mathcal{R})}{|u|^{2}}.  
\end{align}
We will proceed in an exactly similar manner to extract the decay estimates for the remaining error coefficients. Recall the following expression for the error coefficient $C^{B}_{A}$ 
\begin{align}
C^{B}_{A}= -p^3 \hsp e_A(b^B)+(\Gamma^B_{A\nu}p^{\nu}+\frac{p^3}{\modu}{\delta_A}^B) 
\end{align}
which, once again, exhibits the cancellation property, yielding integrable decay as in \eqref{eq:cancellation}. Ultimately, this verifies 
\begin{align}
|C^{B}_{A}|\lesssim \frac{a^{\frac{1}{2}}\Gamma}{|u|^{2}}.    
\end{align}
This cancellation, resulting in an integrable decay feature, is ubiquitous in this study. Next we consider
\begin{align}
 C^{4}_{A}&=\big[-p^3 \hsp e_A(b^B)+\big(\Gamma^B_{A\nu}p^{\nu}+\frac{p^3}{\modu}{\delta_A}^B\big)\big]\big(\frac{\frac{1}{2}\chi_{BC}p^C}{p^3}-\etabar_B\big) \notag\\ +&\big(\chi_{AB}p^{B}-\etabar_{A}p^{3} \big) |u|\omegabar  \notag \\+& \bigg((e_A(\Gamma^3_{\mu\nu})-e_\mu (\Gamma^3_{A\nu}))p^{\mu}\hsp p^{\nu} +\Gamma^C_{\mu\nu}p^{\mu}p^{\nu}\Gamma^{3}_{A\lambda}\partial_{\pbar^C}(p^{\lambda}) + \Gamma^3_{\mu\nu}p^{\mu}p^{\nu}\Gamma^3_{A\lambda} \partial_{\pbar^3}(p^{\lambda}) \notag \\ -&\Gamma^C_{A\lambda}p^{\lambda}\Gamma^3_{\mu\nu}\partial_{\pbar^C}(p^{\mu}p^{\nu})-\Gamma^3_{A\lambda}p^{\lambda} \Gamma^3_{\mu\nu}\partial_{\pbar^3}(p^{\mu}p^{\nu}) +\frac{p^3}{\modu}\Gamma^3_{\mu\nu}\partial_{\pbar^A}(p^{\mu}p^{\nu}) \bigg)\frac{1}{p^3}.   
\end{align}
Once again notice the appearance of the factor $-p^3 \hsp e_A(b^B)+\big(\Gamma^B_{A\nu}p^{\nu}+\frac{p^3}{\modu}{\delta_A}^B\big)$ that yields cancellation and integrable decay (here it does not even require the cancellation property since it is multiplied by a factor that forces the whole expression to be integrable-however cancellation yields an improved decay rate). The derivatives of the connection produce Riemann curvature as promised  
\begin{align}
e_A(\Gamma^3_{\mu\nu})-e_\mu (\Gamma^3_{A\nu})= R^{3}~_{\nu A\mu}-\Gamma^{\lambda}_{\mu\nu}\Gamma^{3}_{A\lambda}+\Gamma^{\lambda}_{A\nu}\Gamma^{3}_{\mu\lambda}+(\Gamma^{\lambda}_{A\mu}-\Gamma^{\lambda}_{\mu A})\Gamma^{3}_{\lambda \nu}.    
\end{align}

More explicitly, recalling the notation from the expression of $C^{0}_{A}$, we write 
\begin{align}
    C^{4}_{A}&=\sum f^{i}_{1}(\psi)f^{i}_{2}(\Psi)f^{i}_{3}(p^{4})f^{i}_{4}(\pslash)f^{i}_{5}(\gslash)f^{i}_{6}(\gslash^{-1})+\gslash\pslash \left( 
    p^{4}\int_{p}f p^{3}\,dp 
  + p^{3}\int_{p}f p^{4}\,dp 
  + \pslash\int_{p}f \gslash \pslash\, dp 
\right) \nonumber \\
&+ p^{4}p^{3}\int_{p}f \gslash \pslash\,dp 
+ p^{3}p^{3}\int_{p}f p^{4} \gslash \pslash\,dp 
+ \pslash p^{3} \int_{p}f \gslash \pslash \gslash \pslash\ dp.
\end{align}
Here we note an important null structure of the Weyl curvature terms appearing in \[\sum f^{i}_{1}(\psi)f^{i}_{2}(\Psi)f^{i}_{3}(p^{4})f^{i}_{4}(\pslash)f^{i}_{5}(\gslash)f^{i}_{6}(\gslash^{-1}).\] These terms read 
\begin{align}
 \betabar_{A} p^{3}p^{3} 
+ \beta_{A} p^{4}p^{3} 
+ \beta_{A} \pslash_{C}\pslash^{C}  
+ \alpha_{AB} p^{4}\pslash^{B} 
+ \alphabar_{AB} \pslash^{B} p^{3}.   
\end{align}
The components $\alpha_{AB}$ are the most dangerous terms since they behave like $a^{\frac{1}{2}}|u|$. But note that it always appears multiplied by $p^{4}$ and $\pslash^{B}$ both of which decay like $|u|^{-2}$. On the other hand, $p^{3}$does appear with $\betabar$ and $\alphabar$, but they exhibit superior integrable decay by themselves. In summary, the connection/curvature components with less decay always appear with $p^{4}$ and/or $\pslash^{A}$ while the dangerous momentum components $p^{3}$ (that enjoys no decay) appear with connection/curvature components with superior integrable decay.  Every term, therefore, enjoys integrability. This is crucial. Stress-energy tensor components exhibit integrable decay as well. More precisely, 
\begin{align}
&\big|\gslash\pslash \left( 
    p^{4}\int_{p}f p^{3}\,dp 
  + p^{3}\int_{p}f p^{4}\,dp 
  + \pslash\int_{p}f \gslash \pslash\, dp 
\right) 
+ p^{4}p^{3}\int_{p}f \gslash \pslash\,dp 
+ p^{3}p^{3}\int_{p}f p^{4} \gslash \pslash\,dp 
+ \pslash p^{3} \int_{p}f \gslash \pslash \gslash \pslash\ dp\big|\nonumber \\  
\lesssim& \frac{a}{|u|^{4}}
\end{align}
and $C^{4}_{A}$ verifies 
\begin{align}
 |C^{4}_{A}|\lesssim \frac{a^{\frac{3}{2}}\mathcal{R}}{|u|^{3}}+\frac{a^{\frac{1}{2}}\Gamma}{|u|^{2}}.  
\end{align}
Now note the next coefficient is $C^{4+B}_{A}$. We write the complete expression because we want to point out a few vital cancellations that are responsible for yielding integrable decay. Direct computations yield 
\begin{align*}
C^{4+B}_{A} &= |u|^{2}\,\chi^{B}_{A} 
 \Big( \omegabar\,p^{4}p^{3} + \etabar\,p^{4}p^{4} + \chibar_{CD}p^{C}p^{D} \Big) \\
&\quad + |u|^{2} \Big( \etabar_{A}p^{A}p^{3} + \eta_{A}p^{A}p^{3} + \chi_{CD}p^{C}p^{D} \Big)\chibar^{B}_{A} \\
&\quad + |u|^{2} \Big( \etabar^{C}p^{4}p^{3} + \eta^{C}p^{4}p^{3} 
       + 2\chibar^{C}_{A}p^{A}p^{3} + \Gammaslash^{C}_{EF}p^{E}p^{F} \Big)\Gammaslash^{B}_{AC} \\
&\quad + |u|^{2}\Big(
     \etabar_{A}\eta^{B}p^{4}p^{3} 
   + \etabar_{A}\chi^{B}_{C}p^{4}p^{C} 
   + \chibar_{AC}\eta^{B}p^{C}p^{3} 
   + \chibar_{AC}\chi^{B}_{D}p^{C}p^{D} \Big) \\
&\quad + |u|^{2}\Big(
     \etabar_{A}\etabar^{B}p^{4}p^{3}
   + \chi_{AC}\etabar^{B}p^{4}p^{C}
   + \etabar_{A}\chibar^{B}_{C}p^{C}p^{3}
   + \chi_{AC}\chibar^{B}_{D}p^{C}p^{D} \Big) \\
&\quad + |u|^{2}\Big(
     \chibar^{C}_{A}\chibar^{B}_{C}p^{3}p^{3}
   + \chi^{C}_{A}\chi^{B}_{C}p^{4}p^{4}
   + \chi^{C}_{A}\chibar^{B}_{C}p^{4}p^{3}
   + \chibar^{C}_{A}\chi^{B}_{C}p^{4}p^{3} \Big) \\
&\quad + |u|^{2}\Big(
     \chi^{C}_{A}\Gammaslash^{B}_{DC}p^{4}p^{D}
   + \Gammaslash^{C}_{AD}\chi^{B}_{C}p^{4}p^{D}
   + \chibar^{C}_{A}\Gammaslash^{B}_{DC}p^{D}p^{C}
   + \Gammaslash^{C}_{AD}\chibar^{B}_{C}p^{D}p^{C}
   + \Gammaslash^{C}_{AD}\Gammaslash^{B}_{EC}p^{E}p^{D} \Big) \\
&\quad + |u|^{2}\Bigg[ 
   \etabar_{A}\etabar^{B}p^{4}p^{3} 
 + \etabar_{A}\chi^{B}_{C}p^{4}p^{C} 
 + \chibar_{AC}\chi^{B}_{D}p^{C}p^{D} 
 + \etabar_{A}\eta^{B}p^{4}p^{3} \\
&\qquad\quad 
 + \etabar_{A}\chibar^{B}_{C}p^{C}p^{3} 
 + \chi_{AC}\eta^{B}p^{4}p^{C} 
 + \chi_{AC}\chibar^{B}_{D}p^{C}p^{D} 
 + \chi^{C}_{A}\chi^{B}_{C}p^{4}p^{4} 
 + \chi^{C}_{A}\chibar^{B}_{C}p^{4}p^{3} \\
&\qquad\quad 
 + \underbrace{\chibar^{C}_{A}\chibar^{B}_{C}p^{3}p^{3}}_{V}
 + \chibar^{C}_{A}\Gammaslash^{B}_{CD}p^{D}p^{3} 
 + \chi^{C}_{A}\Gammaslash^{B}_{CD}p^{4}p^{D} 
 + \Gammaslash^{C}_{AD}\Gammaslash^{B}_{CE}p^{E}p^{D} \\
&\qquad\quad 
 + \etabar_{A}\etabar^{B}p^{4}p^{3} 
 + \etabar_{A}\big(\chibar^{B}_{C} - e_{C}(b^{B})\big)p^{4}p^{C} 
 + \chibar_{CA}\etabar^{B}p^{C}p^{3} 
 + \chibar_{CA}\chi^{B}_{D}p^{C}p^{D} \\
&\qquad\quad 
 + \eta_{A}\eta^{B}p^{4}p^{3} 
 + \eta^{B}\chi_{CA}p^{4}p^{C} 
 + \eta_{A}\chibar^{B}_{C}p^{C}p^{3} 
 + \chi_{CA}\Gammaslash^{B}_{CD}p^{C}p^{D} \\
&\qquad\quad 
 + \big(\chibar^{C}_{A} - e_{A}(b^{C})\big)\chi^{B}_{C}p^{4}p^{4} 
 + \big(\chibar^{C}_{A} - e_{A}(b^{C})\big)\chibar^{B}_{C}p^{4}p^{3} 
 + \chibar^{C}_{A}\chi^{B}_{C}p^{4}p^{3} \\
&\qquad\quad 
 + \big(\chibar^{C}_{A} - e_{A}(b^{C})\big)\Gammaslash^{B}_{CD}p^{4}p^{D} 
 + \chibar^{C}_{A}\Gammaslash^{B}_{CD}p^{D}p^{3} 
 + \Gammaslash^{B}_{DA}\Gammaslash^{B}_{CE}p^{E}p^{D} \\
&\qquad\quad 
 + \Big( \Gammaslash^{B}_{AC} + \tfrac{\gslash_{CD}p^{D}}{2p^{3}}\chi^{B}_{A} \Big)
   \Big( 2(\etabar^{C} + \eta^{C})p^{4}p^{3} 
       + \Gammaslash^{C}_{AB}p^{A}p^{B} 
       + 2\chibar^{C}_{A}p^{A}p^{3} \Big) \\
&\qquad\quad 
 + \Big( \chibar^{B}_{A} - \tfrac{p^{4}}{p^{3}}\chi^{B}_{A} \Big)
   \Big( \etabar_{A}p^{4}p^{A} + \zeta_{A}p^{4}p^{A} 
       + (\eta_{A}-\etabar_{A})p^{A}p^{3} 
       + \chi_{AB}p^{A}p^{B} + \omegabar p^{4}p^{3} \Big) \\
&\qquad\quad 
 + \Big( \chi^{C}_{A}p^{4} 
       + \underbrace{\chibar^{C}_{A}p^{3}}_{I} 
       + \Gammaslash^{C}_{AB}p^{B} \Big)
   \Big( \chi^{B}_{C}p^{4} 
       + \underbrace{(\chibar^{B}_{C} - e_{C}(b^{B}))p^{3}}_{II} 
       + \Gammaslash^{B}_{DC}p^{D}
       + \tfrac{\gslash_{CD}p^{D}}{2p^{3}}
         \big(\etabar^{B}p^{3} + \eta^{B}p^{3} + \chi^{B}_{C}p^{C}\big) \Big) \\
&\qquad\quad 
 + \underbrace{\tfrac{p^{3}}{|u|}}_{III}
   \Big( \chi^{B}_{A}p^{4} 
       + \underbrace{\chibar^{B}_{A}p^{3}}_{IV} 
       + \tfrac{\gslash_{AD}p^{D}}{2p^{3}}
         \big(\etabar^{B}p^{3} + \eta^{B}p^{3} + \chi^{B}_{C} \big) \Big)
\Bigg]\\
&+\bigg( {(e_A(\Gamma^B_{\mu\nu})-e_{\mu}(\Gamma^B_{A\nu}))\hsp p^{\mu}p^{\nu}}+\Gamma^C_{\mu\nu} \hsp p^{\mu}\hsp p^{\nu} \partial_{\pbar^C}(p^\lambda)\Gamma^B_{A\lambda} + \Gamma^3_{\mu\nu}p^{\mu}p^{\nu} \partial_{\pbar^3}(p^\lambda) \Gamma^B_{A\lambda} \notag \\ &{- {\Gamma^C_{A\lambda}p^{\lambda}\Gamma^B_{\mu\nu} \partial_{\pbar^C}(p^{\mu}p^{\nu})}}-\Gamma^3_{A\lambda}p^{\lambda}\Gamma^B_{\mu\nu}\partial_{\pbar^3}(p^{\mu}p^{\nu}) +\big({-\frac{(p^3)^2}{\modu^2}\frac{1}{\Omega^2} {\delta_A}^B} +\frac{{\delta_A}^B}{\modu}\Gamma^3_{\mu\nu}p^{\mu}p^{\nu}\big)\notag \\ &- {\frac{p^3}{\modu} \partial_{\pbar^A}(p^{\mu}p^{\nu})\Gamma^B_{\mu\nu}}\bigg)\hsp \frac{\modu^2}{p^3}.
\end{align*}
We first emphasize a crucial structural feature of the expression for $C^{4+B}_{A}$.  
In the expansions above, the most singular contributions arise from terms quadratic in 
$\tr\chibar$, specifically of the schematic form $\tr\chibar\cdot \tr\chibar$.  
These appear within the combinations labeled $I, II, III, IV$, and $V$.  
A direct inspection reveals that these contributions cancel \emph{exactly}.  
This cancellation is indispensable: without it, one would be left with non--integrable terms 
in the energy hierarchy, thereby obstructing any possibility of closing the estimates.  

Next, we recall that derivatives of the connection coefficients never appear in isolation but are 
always absorbed into curvature components. Indeed, for the Levi--Civita connection one has the identity
\begin{align}
e_A(\Gamma^{B}_{\mu\nu}) - e_{\mu}(\Gamma^{B}_{A\nu})
   &= R^{B}{}_{\nu A\mu}
      - \Gamma^{\lambda}_{\mu\nu}\Gamma^{B}_{A\lambda}
      + \Gamma^{\lambda}_{A\nu}\Gamma^{B}_{\mu\lambda}
      + (\Gamma^{\lambda}_{A\mu}-\Gamma^{\lambda}_{\mu A})\Gamma^{B}_{\lambda\nu}.
\end{align}
Thus, apparent derivatives of connection coefficients reorganize into curvature together with lower-order quadratic terms, thereby precluding any loss of derivatives.  

In the context of the energy estimates, it is essential to control the contribution of $C^{4+B}_{A}$ 
in a manner that ensures integrability. More precisely, since the natural quantity appearing in the 
transport operator is $|u|V_{4+A}f$, it is necessary that the rescaled coefficient $|u|^{-1}C^{4+B}_{A}$ 
is locally integrable. A detailed inspection of the structure of $C^{4+B}_{A}$, combined with the 
pointwise estimates for the Ricci coefficients and curvature components yield the bound
\begin{align}
|u|^{-1}\,\big|C^{4+B}_{A}\big|
   \;\lesssim\; \frac{a^{1/2}(\Gamma,\mathcal
   {R})}{|u|^{2}}.
\end{align}
This estimate is precisely at the threshold required to guarantee integrability along the null 
hypersurfaces and thus suffices for the closure of the weighted energy hierarchy. The last error coefficient $C^{3}_{A}$ reads 
\begin{align}
C^{3}_{A}=\bigg[2(\eta_{A}+\etabar_{A})p^{4}+  \bigg(\big[\Gammaslash^C_{AB}p^B +\big({\chibarhat_{A}}^C + \frac{1}{2}\tildetr {\delta_A}^C\big)p^3 +{\chi_A}^C p^4\big] \frac{\gslash_{CD}p^D}{2p^3} \bigg) \notag -\big(\frac{1}{2}\chi_{AB}p^B -\etabar_A \hsp p^3\big)\frac{p^4}{p^3} \bigg].    
\end{align}
This term exhibits superior integrable decay without the necessity of any cancellation and is easy to estimate. It reads
\begin{align}
|C^{3}_{A}|\lesssim \frac{a^{\frac{1}{2}}\Gamma}{|u|^{2}}.    
\end{align}
Let the hypotheses and notation of the preceding sections hold and assume the bootstrap bound displayed as \eqref{eq:intboot1}. We will prove the following bound
\begin{equation}\label{eq:V12-rescaled}
\frac{|u|^{2}}{a^{2}}\,\mathfrak V^{1,2}_1(u)
\lesssim
\frac{|u_\infty|^{2}}{a^{2}}\,\mathfrak V^{1,2}_1(u_\infty)
+ (\Gamma,\mathcal R)\,\mathcal V\,\frac{a^{1/2}}{|u|}.
\end{equation}
Consequently, under the bootstrap assumption \eqref{eq:intboot1} there exists $C'>0$ such that
\begin{equation}\label{eq:V12-unif}
\frac{|u|^{2}}{a^{2}}\,\mathfrak V^{1,2}_1(u)
\lesssim
\frac{|u_\infty|^{2}}{a^{2}}\,\mathfrak V^{1,2}_1(u_\infty)
+1
\end{equation}
We begin by collecting the error estimates obtained previously. There exists an absolute constant $C>0$ for which the following differential inequality holds (for $s$ in the corresponding parameter interval):
\begin{equation}\label{eq:err-ineq}
\mathfrak V^{1,2}_1(s)
\lesssim 
\mathfrak V^{1,2}_1(s_\infty)
+
\int_{s_\infty}^s
|\mathcal E(s')|\,ds',
\end{equation}
with
\begin{multline}\label{eq:E-def}
|\mathcal E(s')|
\lesssim \Bigg(
\frac{a^{1/2}\Gamma}{|u|^{2}}\int_{P}|\mathbf V_A f|^2
+\frac{a^{1/2}\Gamma\mathcal R}{|u|^{3}}\int_{P}\bigl|\,|u|\,\mathbf V_0 f\bigr|^2\\
+\Bigl[\frac{a^{3/2}\mathcal R}{|u|^{3}}+\frac{a^{1/2}\Gamma}{|u|^{2}}\Bigr]\int_{P}|\mathbf V_4 f|^2
+\frac{a^{1/2}\Gamma}{|u|^{2}}\int_{P}|\mathbf V_3 f|^2
+\frac{(\Gamma,\mathcal R)\,a^{1/2}}{|u|^{2}}\int_{P}\bigl|\,|u|\,\mathbf V_{4+B} f\bigr|^2
\Bigg).
\end{multline}

\noindent By the bootstrap control of $\mathfrak V_1$ we have the moment bounds (uniform in $s$)
\begin{equation}\label{eq:moments}
\int_{P}|\mathbf V_A f|^2 \lesssim \frac{a^{2}\mathcal V}{|u|^{2}},\quad
\int_{P}\bigl|\,|u|\,\mathbf V_{4+A} f\bigr|^2 \lesssim \frac{a^{2}\mathcal V}{|u|^{2}},\quad
\int_{P}\bigl|\,|u|\,\mathbf V_{0} f\bigr|^2 \lesssim \frac{a^{2}\mathcal V}{|u|^{2}},
\end{equation}
and
\[
\int_{P}|\mathbf V_4 f|^2 \lesssim \frac{a^{2}\mathcal V}{|u|^{2}},\qquad
\int_{P}|\mathbf V_3 f|^2 \lesssim \frac{a^{2}\mathcal V}{|u|^{2}}.
\]
Inserting these bounds into \eqref{eq:E-def} and collecting like terms yields (after absorbing absolute constants)
\begin{equation}\label{eq:E-simplified}
|\mathcal E(s')| \le (\Gamma,\mathcal R)\,\mathcal V\left(\frac{a^{5/2}}{|u|^{4}} + \frac{a^{7/2}}{|u|^{5}}\right).
\end{equation}
Next we change variables from $s$ to $u$ using Proposition~\ref{integrate}, which gives $\dfrac{d}{ds}=p^3\dfrac{d}{du}$. we deduce from \eqref{eq:err-ineq} and \eqref{eq:E-simplified} that
\[
\mathfrak V^{1,2}_1(u)
\lesssim
\mathfrak V^{1,2}_1(u_\infty)
+ (\Gamma,\mathcal R)\,\mathcal V
\int_{u_\infty}^{u}\!\left(\frac{a^{5/2}}{|u'|^{4}} + \frac{a^{7/2}}{|u'|^{5}}\right)\,du'.
\]
Carrying out the elementary integrals in $u'$ yields
\[
\mathfrak V^{1,2}_1(u)
\le
\mathfrak V^{1,2}_1(u_\infty)
+ (\Gamma,\mathcal R)\,\mathcal V\left(\frac{a^{5/2}}{|u|^{3}} + \frac{a^{7/2}}{|u|^{4}}\right).
\]
\noindent Multiplying both sides by $\dfrac{|u|^{2}}{a^{2}}$ we obtain
\[
\frac{|u|^{2}}{a^{2}}\,\mathfrak V^{1,2}_1(u)
\le
\frac{|u_\infty|^{2}}{a^{2}}\,\mathfrak V^{1,2}_1(u_\infty)
+ (\Gamma,\mathcal R)\,\mathcal V\left(\frac{a^{1/2}}{|u|} + \frac{a^{3/2}}{|u|^{2}}\right).
\]
Since $a$ is fixed in our regime and $|u|^{-1}\ge |u|^{-2}$ for the relevant $u$, the right-most bracket is bounded by $\lesssim \,a^{1/2}/|u|$, which proves \eqref{eq:V12-rescaled}. Finally, invoking the bootstrap assumption \eqref{eq:intboot1} to control the left-hand side uniformly allows us to absorb the remaining term and conclude \eqref{eq:V12-unif}.

Now we show the estimates for a vector field in the weighted category, namely $V_{4+A}$. Consider $V_{4+A}f$ i.e., we want to estimate 
\begin{eqnarray}
\mathfrak{V}^{4+A}_{1}:=\int_{\mathcal{P}_{x}}|V_{4+A}f|^{2}\sqrt{\det\gslash}\frac{dp^{1}dp^{2}dp^{3}}{p^{3}}.    
\end{eqnarray}
In reality, we want to estimate $|u|^{2}\mathfrak{V}^{4+A}_{1}$, but do so indirectly.
Recall lemma \ref{lem:transport} and set $A(s):=\mathfrak{V}^{4+A}_{1}$ and obtain 
\begin{align}
 \mathfrak{V}^{4+A}_{1}(s)\lesssim \mathfrak{V}^{4+A}_{1}(s_{\infty})+\int_{s_{\infty}}^{s}\Bigg(\int_{p}V_{4+A}f X[V_{4+A}f]\sqrt{\det(\gslash)}\frac{dp^{1}dp^{2}dp^{3}}{p^{3}}\Bigg)ds^{'}.   
\end{align}
We control the error term. Recall the commutation formula from proposition \ref{commutator}
\begin{align}\nonumber
    [X,V_{(4+B)}]=& \Bigg[-\frac{p^3}{\modu^2} {\delta_B}^A   \big(\frac{\chi_{AC}p^C}{2p^3} - \etabar_A\big) \nonumber  \nonumber -\frac{1}{\modu^2}\Gamma^3_{\mu\nu}\partial_{\pbar^B}(p^{\mu}p^{\nu})   \Bigg] \bf{V_{(0)}} \\ \nonumber -& \frac{p^3}{\modu^2} {\delta_B}^A \bf{V_{(A)}}\\\nonumber  -&\frac{\gslash_{BC}p^C}{2\modu^2} \bf{V_{(3)}} \\ \nonumber +&\Bigg[ -\frac{p^3}{\modu^2}{\delta_B}^A \bigg(\frac{\chi_{AC}p^C}{2p^3}-\etabar_A\bigg)  -\frac{1}{\modu^2}\Gamma^3_{\mu\nu}\partial_{\pbar^B}(p^{\mu}p^{\nu})
   \Bigg] \bf{V_{(4)}} \\ \nonumber +& \Bigg[ -\frac{p^3}{\modu^2}{\delta_B}^C \hsp \modu^2 \big( \frac{\Gammaslash^A_{CD}p^D}{p^3}+{\chibarhat_C}^A+{\chi_C}^A\frac{p^4}{p^3}+ \frac{1}{2}(\tr\chibar+\frac{2}{\modu}) {\delta_C}^A \big) \\ \nonumber -& \frac{\gslash_{BC}p^C}{2}\bigg(\frac{{\chi_C}^Ap^C}{p^3}+2\etabar^A\bigg) \\ +& \bigg[{\delta_B}^A \big((-\Gamma^3_{\mu\nu}p^{\mu}p^{\nu} + \frac{2(p^3)^2}{\modu^3}\big)+\Gamma^A_{\mu\nu}p^3 \partial_{\pbar^B}(p^{\mu}p^{\nu}) \bigg]
         \Bigg]\bf{V_{(4+A)}}.
\end{align}
This expression is relatively simple, and estimating the error coefficients is straightforward
\begin{align}
 |C^{0}_{4+B}|= \big|-\frac{p^3}{\modu^2} {\delta_B}^A   \big(\frac{\chi_{AC}p^C}{2p^3} - \etabar_A\big)  -\frac{1}{\modu^2}\Gamma^3_{\mu\nu}\partial_{\pbar^B}(p^{\mu}p^{\nu})\big|\lesssim \frac{a^{\frac{1}{2}}\Gamma}{|u|^{3}},\\
 |C^A_{4+B}| = |-\frac{p^3}{\modu^2} \delta^A_B| \lesssim \modu^{-2}, 
\qquad
|C^3_{4+B}| = |-\frac{\gslash_{BC}p^C}{2\modu^2}|\lesssim |u|^{-2},\\
|C^{4}_{4+B}|=|-\frac{p^3}{\modu^2}{\delta_B}^A \bigg(\frac{\chi_{AC}p^C}{2p^3}-\etabar_A\bigg)  -\frac{1}{\modu^2}\Gamma^3_{\mu\nu}\partial_{\pbar^B}(p^{\mu}p^{\nu})|\lesssim \frac{a^{\frac{1}{2}}\Gamma}{|u|^{3}},
\end{align}
and lastly 
\begin{align}
C^{4+A}_{4+B} &= 
- \delta_B^C \left( \frac{p^3}{\modu^2} \right)
\left(
\modu^2 \left( \frac{\Gammaslash^A_{CD}p^D}{p^3} + \chibarhat_C^A + \chi_C^A \frac{p^4}{p^3}
+ \frac{1}{2} \left( \tr\chibar + \frac{2}{\modu} \right) \delta_C^A \right)
\right) \notag \\
&\quad
- \frac{\gslash_{BC}p^C}{2} \left( \frac{(\chi_C^A - e_C(b^A))p^C}{p^3} + 2\etabar^A \right) \notag \\
&\quad
+ \delta_B^A \left( -\Gamma^3_{\mu\nu} p^\mu p^\nu + \frac{2(p^3)^2}{\modu^3} \right)
+ \Gamma^A_{\mu\nu} p^3 \, \partial_{\pbar^B}(p^\mu p^\nu).
\end{align}
We explicitly calculate
\begin{align}
\Gamma^3_{\mu\nu} p^\mu p^\nu 
= \omegabar (p^3)^2 + (\eta_{A}-\etabar_{A})p^{A}p^{3}+\chi_{AB}p^{A}p^{B}
\end{align}
and
\begin{eqnarray}
 \Gamma^{A}_{\mu\nu}p^{3}\frac{\partial}{\partial \pbar^{B}}(p^{\mu}p^{\nu})=\chi^{A}_{B}p^{4}p^{3}+\underbrace{\Bigg(\chibar^{A}_{B}-e_{A}(b^{B})\Bigg)p^{3}p^{3}}_{A}+\Gammaslash^{A}_{BC}p^{C}p^{3}+\chi^{A}_{C}P^{C}+\etabar^{A}p^{3}.   
\end{eqnarray}
There are several points to be noted here. First, note that $e_{A}(b^{B})$ appears with $\chibar^{A}_{B}$ and therefore the form $\Gamma^{B}_{3A}$ and can be controlled through the $\snabla_{4}$ evolution equation. But more importantly, in the trace-tracefree decomposition of $\Gamma^{A}_{3B}=\chibar^{A}_{B}-e_{A}(b^{B})$, the non-integrable term $\tr\chibar-e_{A}(b^{A})$ needs to be taken care of. But notice that this term appears with a good sign since $\tr\chibar-e_{A}(b^{A})=-\frac{2}{|u|}+(\tr\chibar+\frac{2}{|u|}-e_{A}(b^{A}))$ and therefore the non-integrable part in term $A$ is estimated as  
\begin{align}
\notag &\int_{\mathcal{P}}V_{4+A}f\Bigg(\tr\chibar-e_{A}(b^{A})\Bigg) p^{3}p^{3}V_{4+B}f \sqrt{\gslash}\frac{dp^{1}dp^{2}dp^{3}}{p^{3}}\\ 
 =&-\frac{2}{|u|}\int_{P}|V_{4+A}f|^{2}(p^{3})^{2}\sqrt{g}\frac{dp^{1}dp^{2}dp^{3}}{p^{3}}+\Bigg(\tr\chibar+\frac{2}{|u|}-e_{A}(b^{A})\Bigg)\int_{P}|V_{4+A}f|^{2}(p^{3})^{2}\sqrt{g}\frac{dp^{1}dp^{2}dp^{3}}{p^{3}}.
\end{align}
Finally, we collect every term, yielding
\begin{align}\notag
   \mathfrak{V}^{4+A}_{1}(s)\lesssim \mathfrak{V}^{4+A}_{1}(s_{\infty})+ &\int_{s_{\infty}}^{s}\Bigg(\frac{1}{|u|^{2}}\int_{P}|V_{A}f||V_{4+B}f|dP+\frac{a^{\frac{1}{2}}\Gamma}{|u|^{3}}\int_{P}|V_{0}f||V_{4+B}f|+\frac{a^{\frac{1}{2}}\Gamma}{|u|^{3}}\int_{P}|V_{4}f||V_{4+B}f|\\\nonumber 
&+\frac{1}{|u|^{2}}\int_{P}|V_{3}f||V_{4+B}f|+\frac{a^{\frac{1}{2}}\Gamma}{|u|^{2}}\int_{P}|V_{4+A}f|^{2}dP\Bigg)ds^{'}\\\nonumber 
&\lesssim \int_{s_{\infty}}^{s}\Bigg(\frac{1}{|u|^{2}}\int_{P}|V_{A}f||V_{4+B}f|dP+\frac{a^{\frac{1}{2}}\Gamma}{|u|^{3}}\int_{P}|V_{0}f||V_{4+B}f|+\frac{a^{\frac{1}{2}}\Gamma}{|u|^{3}}\int_{P}|V_{4}f||V_{4+B}f|\\\nonumber 
&+\frac{1}{|u|^{2}}\int_{P}|V_{3}f||V_{4+B}f|\\ 
&+\frac{1}{|u|^{3}}\int_{P}|V_{3}f||u V_{4+B}f|dP+\frac{a^{\frac{1}{2}}\Gamma}{|u|^{4}}\int_{P}|u V_{4+A}f|^{2}dP\Bigg)ds^{'}.
\end{align}
First, note that we estimate  $V_{4+A}f$ at the very end, by which time estimates have already been obtained for the rest of the vector fields i.e.,
\begin{align}
\int_{P}|V_{A}f|^{2}dp\lesssim \frac{a^{2}\mathcal{V}^{0}}{|u|^{2}},\int_{P}||u|V_{0}f|^{2}dp\lesssim \frac{a^{2}\mathcal{V}^{0}}{|u|^{2}},~\int_{P}|V_{4}f|^{2}dp\lesssim \frac{a^{2}\mathcal{V}^{0}}{|u|^{2}},~\int_{P}|V_{3}f|^{2}dp\lesssim \frac{a^{2}\mathcal{V}^{0}}{|u|^{2}}    
\end{align}
and we use the boot-strap on the $|u|V_{4+A}f$,
\begin{align}
\int_{P}||u|V_{4+A}f|^{2}dp\lesssim \frac{a^{2}\mathcal{V}}{|u|^{2}},
\end{align}
which yields
\begin{align}
 \mathfrak{V}^{4+A}(s)\lesssim \mathfrak{V}^{4+A}(s_{\infty})+\int_{s_{\infty}}^{s}\bigg[\frac{a\mathcal{V}^{0}}{|u|^{3}}\sqrt{\mathfrak{V}^{4+A}}+\frac{a^{\frac{3}{2}}\Gamma\mathcal{V}^{0}}{|u|^{4}}\sqrt{\mathfrak{V}^{4+A}}+\frac{a^{\frac{5}{2}}\Gamma\mathcal{V}}{|u|^{6}}\bigg]ds^{'}.
\end{align}
Now recall $\frac{d}{ds}=p^{3}\frac{d}{du}$ (Proposition \ref{integrate}) to obtain, under the bootstrap,
\begin{align}
\frac{|u|^{4}}{a^{2}}\mathfrak{V}^{4+A}(s)\lesssim \frac{|u_{\infty}|^{4}}{a^{2}}\mathfrak{V}^{4+A}(s_{\infty})+\mathcal{V}^{0}+1.
\end{align}
Therefore, all the estimates collectively yield 
\begin{align}
\mathfrak{V}(u)\lesssim \mathcal{V}^{0}+1
\end{align}
uniformly in $u$,
which completes the proof. For the estimates for $\mathfrak{V}^{s}(u)$, we use the second transport inequality \eqref{eq:transport-B} in Lemma \ref{lem:transport} and proceed in exactly the same fashion. For brevity, we do not repeat the computations since it would be repetitive. This concludes the proof of the estimates for the first derivatives of the distribution function $f$. 
\end{proof}

\subsection{Second Derivative control of Vlasov}
Now we move on to the second derivative estimates for the Vlasov field. To this end, we need to compute the second commutator. Let \( f \) be a smooth scalar function on phase space and let \( X \), \( V_A \), and \( V_B \) denote smooth vector fields satisfying the necessary regularity and geometric compatibility assumptions (e.g., adapted to a null foliation or a geometric frame bundle over spacetime). Then the following elementary identity holds:
\begin{equation} \label{eq:basic-product}
X(V_A V_B f) = V_A \left( X(V_B f) \right) - [V_A, X](V_B f).
\end{equation}

\noindent We now compute each term on the right-hand side of \eqref{eq:basic-product} in terms of the frame \(\{V_I\}_{I=0}^{6}\), assumed to span the relevant tangent bundle. Specifically, we expand the commutator \([X, V_A]\) as follows:
\begin{equation} \label{eq:commutator}
[X, V_A] = \sum_{I=0}^{6} C^I_A V_I,
\end{equation}
for some smooth scalar coefficient functions \( C^I_A \), depending on spacetime position and momentum variables.

\noindent Applying \eqref{eq:commutator}, we obtain the action of the commutator on \( V_B f \):
\begin{equation} \label{eq:commutator-on-VB}
[X, V_A](V_B f) = \sum_{I=0}^{6} C^I_A V_I (V_B f).
\end{equation}

\noindent Next, we compute the full second-order derivative term \( V_A(X(V_B f)) \). Using the product rule and the expansion of \( X(V_B) \), we write:
\begin{equation} \label{eq:VA-on-XVB}
V_A \left( X(V_B f) \right) = \sum_{I=0}^{6} C^I_B V_A \left( V_I f \right) + \sum_{I=0}^{6} V_A(C^I_B) V_I f.
\end{equation}

\noindent Combining \eqref{eq:commutator-on-VB} and \eqref{eq:VA-on-XVB} into \eqref{eq:basic-product}, we obtain the following identity:
\begin{equation} \label{eq:full-commuted}
X(V_A V_B f) \sim \sum_{I=0}^{6} C^I_A V_I V_B f + \sum_{I=0}^{6} V_A(C^I_B) V_I f.
\end{equation}

\noindent This identity expresses the commuted second-order operator \( X(V_A V_B f) \) entirely in terms of the frame \( \{V_I\} \), the coefficient functions \( C^I_A \), and their derivatives. Now we already have explicit expressions for the $C^{I}_{A}$. At the level of the second derivative of the Vlasov field, we therefore need to estimate the first derivatives of $C^{I}_{B}$, which would involve first derivatives of the Weyl curvature, first derivatives of the connection coefficients, and the first derivatives of the Vlasov matter stress-energy tensor. We need a systematic way to compute the typical term $V_{A}(C^{I}_{B})$. The following lemma accomplishes it. 
\begin{lemma}
\label{action}
Let \(\psi = (\psi_0, \dots, \psi_k)\) denote a finite collection of smooth scalar or tensor fields (on $S_{u,\ubar}$) depicting the Ricci coefficients of the Weyl curvature components and let \(\pslash\) denote the pair \((p^1, p^2)\) of angular momentum components, \(p^4\) the mass shell component determined by the constraint relation, and \(\gslash = \gslash_{AB}\) the induced metric on the sphere \(S_{u,\ubar}\). Consider an entity of the form
\[
G := f_1(\psi)\cdot\, f_2(\pslash)\cdot\, f_3(p^4)\cdot \, f_4(\gslash),
\]
where \(f_1, f_2, f_3, f_4\) are smooth functions of their respective arguments. Then the action of the derivative \(V_D\) satisfies the identity:
\begin{align}
\label{eq:fix}
V_D(G) &= \Bigg[
\left( \sum_{i=0}^k \frac{\partial f_1}{\partial \psi_i} \nablasl \psi_i \right) f_4(\gslash)
+ \left( \sum_{i=0}^k \frac{\partial f_1}{\partial \psi_i} \Gammaslash\, \psi_i \right) f_4(\gslash)
+ f_1(\psi) f_4'(\gslash)\, \Gammaslash\, \gslash
\Bigg] f_2(\pslash) f_3(p^4) \notag \\
&\quad + f_1(\psi) f_4(\gslash) \Bigg[
\left( \chihat\, p^4 + \tr\chi\, p^4 + \chibarhat\, p^3 + \tr\chibar\, p^3 + \Gammaslash\, \pslash \right)
\left( f_2'(\pslash) f_3(p^4) + f_2(\pslash)\, \gslash\, \pslash\, f_3'(p^4) \right) \notag \\
&\qquad + \left( \etabar\, p^3 + \chihat\, \pslash + \tr\chi\, \pslash \right) f_2(\pslash)\, f_3'(p^4)\, p^4 
+ \frac{p^3}{|u|} \left( f_2'(\pslash) f_3(p^4) + \gslash\, \pslash\, f_2(\pslash)\, f_3'(p^4) \right)
\Bigg]_{D}.
\end{align}
\noindent Moreover, if each \(\psi_i\) is a scalar function (i.e., independent of the angular variables), then the term involving \(\Gammaslash\, \psi_i\) vanishes identically:
\[
\sum_{i=0}^k \frac{\partial f_1}{\partial \psi_i} \Gammaslash\, \psi_i = 0.
\]
The subscript in equation \ref{eq:fix} denotes the index that is not summed over i.e., for an entity $\alpha^{A}_{D}\eta_{A}$ we write $[\alpha\eta]_{D}$. 
\end{lemma}
\begin{proof}
 Straightforward computation using the definition of the vector fields $V$.  
\end{proof}
Equipped with Lemma \ref{action}, we proceed to estimate the second derivative of the Vlasov field. First, recall the definition of the norm that we want to control (together with the definition of the set $\tilde{V}$ in \eqref{tildeV}): 
\begin{eqnarray}
\mathfrak{V}_{2}:=\frac{1}{a^{2}}\sum_{V_{(1)},V_{(2)}\in \tilde{V}}\sup_{u,\ubar}\int_{S_{u,\ubar}}\nonumber \int_{\mathcal{P}_{x}}|\tilde{V}_{(1)}\tilde{V}_{(2)} f|^{2}\sqrt{\det\gslash}\frac{dp^{1}dp^{2}dp^{3}}{p^{3}}\sqrt{\det\gslash}\hsp d^{2}\theta.
\end{eqnarray}
The following proposition accomplishes it:
\begin{proposition}[Second derivative estimates for the Vlasov distribution function]
\label{prop:second_derivative_estimate}
Let \( f: \mathcal{M} \times \mathbb{R}^3 \to \mathbb{R} \) denote a classical solution to the Vlasov equation posed on a smooth, globally hyperbolic Lorentzian manifold \( (\mathcal{M}, g) \), foliated by a double null coordinate system \( (u, \ubar, \theta^1, \theta^2) \). Suppose that \( f \) satisfies the transport equation along the null geodesic flow and that all relevant Ricci coefficients, curvature components, and geometric quantities satisfy the bootstrap assumptions described in Section \ref{bootstrap}.

\noindent Let \( \mathfrak{V}_2 \) denote the second-order energy quantity associated to \( f \), defined by $\mathfrak{V}$,
built from commutation vector fields adapted to the geometric foliation as follows 
\begin{eqnarray}
\mathfrak{V}_{2}:=\frac{1}{a^{2}}\sum_{V_{(1)},V_{(2)}\in \tilde{V}}\sup_{u,\ubar}\int_{S_{u,\ubar}}\nonumber \int_{\mathcal{P}_{x}}|\tilde{V}_{(1)}\tilde{V}_{(2)} f|^{2}\sqrt{\det\gslash}\frac{dp^{1}dp^{2}dp^{3}}{p^{3}}\sqrt{\det\gslash}\hsp d^{2}\theta.
\end{eqnarray}

\noindent Denote, moreover, the initial data expression
\[
\mathcal{I}^0_{2} :=\mathfrak{V}_{2}:=\frac{1}{a^{2}}\sum_{V_{(1)},V_{(2)}\in \tilde{V}}\sup_{\ubar}\int_{S_{u_{\infty},\ubar}}\nonumber \int_{\mathcal{P}_{x}}|\tilde{V}_{(1)}\tilde{V}_{(2)} f|^{2}\sqrt{\det\gslash}\frac{dp^{1}dp^{2}dp^{3}}{p^{3}}\sqrt{\det\gslash}\hsp d^{2}\theta<\infty.
\]
Then the following estimates hold
\begin{equation}
\label{eq:second_derivative_estimate}
\mathfrak{V}_2 \lesssim 1 + \mathcal{I}^0_{2}.
\end{equation}
where the involved constants are purely numerical.
\end{proposition}
\begin{proof} 
First, we make the bootstrap assumption  \be \label{V2bootstrap} \mathfrak{V}_{2}\leq \mathcal{V}, \ee
where $\mathcal{V}$ is large enough so that 
 $\mathcal{I}^{0}\ll \mathcal{V}$
but also $\mathcal{V}\leq a^{\frac{1}{320}}$. Once again, we will only provide a detailed proof for the borderline terms. We consider $V_{A}V_{B}f$. More precisely, we want to estimate 
\begin{eqnarray}
 \mathfrak{V}^{1}_{2}:=\int_{S_{u,\ubar}}\int_{\mathcal{P}_{x}}|V_{A}V_{B}f|^{2}\sqrt{\det\gslash}\hsp \frac{dp^{1}dp^{2}dp^{3}}{p^{3}}\sqrt{\det\gslash}\hsp d^{2}\theta.   
\end{eqnarray}
First, recall the elementary identity 
\begin{eqnarray}
  X(V_{A}V_{B}f)=V_{A}(X(V_{B}f))-[V_{A},X](V_{B}f).  
\end{eqnarray}
Now we can explicitly compute both $V_{A}(X(V_{B}))$ and $[V_{A},X](V_{B}f)$. Indeed,
\begin{eqnarray}
 [X,V_{A}]V_{B}f=\sum_{I=0}^{6}C^{I}_{A}V_{I}V_{B}f,   
\end{eqnarray}
and 
\begin{eqnarray}
 V_{A}(X(V_{B}f))=\sum_{I=0}^{6}C^{I}_{A}V_{I}V_{B}f+\sum_{I=0}^{6}V_{A}(C^{I}_{B})V_{I}f.   
\end{eqnarray}
The new terms that we need to estimate now are the $V_{A}(C^{I}_{B})$.
First, recall the transport inequality from lemma \ref{lem:transport} and set $B:=\mathfrak{V}^{1}_{2}$,
\begin{align}
\mathfrak{V}^{1}_{2}(s) &\lesssim \mathfrak{V}^{1}_{2}(s_{\infty}) + \int_{s_\infty}^{s} \int_{S_{u,\ubar}} \left( \int_{\pi^{-1}(x)} V_{A}V_{B}f \cdot X[ V_{A}V_{B}f] \sqrt{\det \slashed{g}} \, \frac{dp^1 dp^2 dp^3}{p^3} \right) \sqrt{\det \slashed{g}} \, dS_{u,\ubar}(x) \, ds', \label{eq:transport-B1} 
\end{align}
where $s_{\infty}=s(u_{\infty})$. The first thing that we need to do is evaluate the error term $X[V_{A}V_{B}f]$. Note that the only non-trivial computations that we have to execute are $V_{D}(C^{I}_{B})$ for $I=A,4,3,4+A,0$. To this end, we will use the lemma \ref{action}. The idea here is to keep track of the upper and lower indices since they result in the appearance of $\gslash^{-1}$ and $\gslash$ in the component-wise point-wise norm. Then we can use the table \ref{boundsbootstrapnew1} to control them. Schematically, we show a few terms. First, we have $V_{D}(C^{0}_{A})$ which contains both the indices downstairs. Using Lemma \ref{action}, we obtain   
\begin{align}
\label{eq:V1}
V_{D}(C^{0}_{A}) \supset  V_D(\psi_g \psi_g p^4)_{A} 
&\sim \bigg[(\nablasl \psi_g)\, \psi_g\, p^4 
+ \Gammaslash\, \psi_g\, \psi_g\, p^4 \notag 
+ (\chihat\, p^4 + \tr\chi\, p^4 + \chibarhat\, p^3 + \tr\chibar\, p^3 + \Gammaslash\, \pslash)\, \psi_g\, \psi_g\, \gslash\, \pslash \notag \\
&\quad + (\etabar + \chihat\, \pslash + \tr\chi\, \pslash)\, \psi_g\, \psi_g\, p^4 
+ |u|^{-1}\, \psi_g\, \psi_g\, \gslash\, \pslash\bigg]_{DA},\\
V_{D}(C^{0}_{A}) \supset V_D(\psi_g\, \psi_g\, \pslash^2)_{A} 
&\sim \bigg[(\nablasl \psi_g)\, \psi_g\, \pslash^2 
+ \Gammaslash\, \psi_g\, \psi_g\, \pslash^2 \notag \\
&\quad + \psi_g\, \psi_g \left[ (\chihat\, p^4 + \tr\chi\, p^4 + \chibarhat\, p^3 + \tr\chibar\, p^3) - |u|^{-1} \pslash \right]\bigg]_{DA},\\
V_{D}(C^{0}_{A}) \supset V_D(\psi_g\, \tr\chibar\, \pslash)_{A} 
&\sim \bigg[\left[ (\nablasl \psi_g)\, \tr\chibar + (\nablasl \tr\chibar)\, \psi_g + \tr\chibar\, \Gammaslash\, \psi_g \right]\, \pslash \notag \\
&\quad + \psi_g\, \tr\chibar \left( \chihat\, p^4 + \tr\chi\, p^4 + \chibarhat\, p^3 + \tr\chibar\, p^3 + \Gammaslash\, \pslash + |u|^{-1} \right)\bigg]_{DA},\\
V_{D}(C^{0}_{A}) \supset V_D\left( \psi_g\, \gslash\, \pslash\, p^4 \right)_{A}
&\sim  \bigg[\left[ (\nablasl \psi_g)\, \gslash + \Gammaslash\, \psi_g + \psi_g\, \Gammaslash\, \gslash \right]\, \pslash\, p^4 \notag \\
&\quad + \psi_g\, \gslash \Big[ (\chihat\, p^4 + \tr\chi\, p^4 + \chibarhat\, p^3 + \tr\chibar\, p^3 + \Gammaslash\, \pslash)(p^4 + \pslash\, \gslash\, \pslash) \notag \\
&\quad + (\etabar\, p^3 + \chihat\, \pslash + \tr\chi\, \pslash)\, \pslash\, p^4 
+ p^3\, |u|^{-1}(p^4 + \gslash\, \pslash^2) \Big]\bigg]_{DA},\\
V_{D}(C^{0}_{A}) \supset V_D\left( \psi_g\, \gslash\, \pslash^3 \right)_{A}
&\sim \bigg[ \left[ (\nablasl \psi_g)\, \gslash + \Gammaslash\, \psi_g\, \gslash \right]\, \pslash^3 + \psi_g\, \gslash \left( \chihat\, p^4 + \tr\chi\, p^4 + \chibarhat\, p^3 + \tr\chibar\, p^3 + \Gammaslash\, \pslash \right)\, \pslash^2 \notag \\
&\quad + \psi_g\, \gslash\, p^3\, |u|^{-1}\, \pslash^2\bigg]_{DA}.
\end{align}
Notice the decay structure of every term present. 
Similarly, there are additional terms in $V_{D}(C^{0}_{A})$ with exactly similar schematics, but with better or equal decay bounds. Therefore, we do not include these for the sake of clarity. Next we consider $V_{D}(C^{B}_{A})$. There holds

\begin{align}\notag 
V_D\bigl(C^{B}_{A}\bigr)
&\sim
\Bigg[
\,
\nablasl(\chibarhat,\;\widetilde{\tr\chibar})
\;+\;
\Gammaslash(\chibarhat,\;\widetilde{\tr\chibar}) 
+ \bigl(\nablasl\chibarhat+\Gammaslash\chibarhat\bigr)\,p^{4}
\\[4pt] \notag 
&\qquad
+ \chibarhat\Bigl[
\bigl(\chihat\,p^{4}+\tr\chi\,p^{4}+\chibarhat\,p^{3}+\tr\chibar\,p^{3}+\Gammaslash\,\pslash\bigr)\,\gslash\,\pslash \\[4pt] \notag 
&\qquad+ \bigl(\etabar\,p^{3}+\chihat\,\pslash+\tr\chi\,\pslash\bigr)\,p^{4}
+ |u|^{-1}\,\gslash\,\pslash
\Bigr] + \bigl(\nablasl\psi_{g}+\Gammaslash\psi_{g}\bigr)\,p^{4}
\\[4pt] \notag 
&\qquad
+ \psi_{g}\Bigl[
\bigl(\chihat\,p^{4}+\tr\chi\,p^{4}+\chibarhat\,p^{3}+\tr\chibar\,p^{3}+\Gammaslash\,\pslash\bigr)\,\gslash\,\pslash \\[4pt]
&\qquad 
+ \bigl(\etabar\,p^{3}+\chihat\,\pslash+\tr\chi\,\pslash\bigr)\,p^{4}
+ |u|^{-1}\,\gslash\,\pslash
\Bigr]
\,
\Bigg]^{B}_{\;DA}\!,
\end{align}

where, once again, we omitted the terms that reappear. In particular, this will be the recurring theme in estimating the error terms with the same decay structure. The next error term $V_{D}(C^{4}_{A})$ has the following semi-schematic structure
\begin{align}
 & V_{D}(C^{4}_{A})\sim \Bigg[\bigg[\big( \chihat\, \nablasl \psi_g + \psi_g\, \nablasl \chihat
      + \chihat\, \Gammaslash\, \psi_g + \psi_g\, \Gammaslash\, \chihat \big)
    p^4\, \pslash \notag\\
&\quad + \psi_g\, \chihat \Big[
   \big( \chihat p^4 + \tr\chi\, p^4 + \chibarhat\, p^3 + \tr\chibar\, p^3
        + \Gammaslash\, \pslash \big)(p^4 + \gslash\, \pslash^2) \notag\\ \notag
&\qquad\quad + \big( \etabar\, p^3 + \chihat\, \pslash + \tr\chi\, \pslash \big)
   + p^3\, |u|^{-1}(p^4 + \gslash \hsp  \pslash^2) \Big]\bigg]\\
   &\quad +\bigg[\big( 2\hsp \psi_g\, \nablasl \psi_g\, \gslash
       + 2\psi_g\, \Gammaslash\, \psi_g\, \gslash
       + \psi_g^2\, \Gammaslash\, \gslash \big)\pslash\, p^4 \notag\\
&\quad + \psi_g^2 \hsp  \gslash \Big[
   (\chihat \hsp  p^4 + \tr\chi\, p^4 + \chibarhat\, p^3 + \tr\chibar\, p^3
      + \Gammaslash\, \pslash)(p^4 + \gslash\, \pslash^2) \notag\\ \notag 
&\qquad\quad + (\etabar\, p^3 + \chihat\, \pslash + \tr\chi\, \pslash)
   + p^3\, |u|^{-1} (p^4 + \gslash\, \pslash^2) \Big] \bigg]\\
   &+\bigg[\big( \psi\, \nablasl \psi_g + \psi_g\, \nablasl \psi
       + \psi\, \Gammaslash\, \psi_g + \psi_g\, \Gammaslash\, \psi \big)
      \gslash\, \pslash^2 \notag\\ \notag 
&\quad + \psi_g\, \psi\, \gslash \Big[
   (\chihat \hsp p^4 + \tr\chi\, p^4 + \chibarhat \hsp p^3 + \tr\chibar \hsp  p^3
      + \Gammaslash\, \pslash)\,\pslash
   + p^3 \hsp |u|^{-1} \hsp \pslash \Big]\bigg]\\ \notag 
   &+\Big[
(\nablasl \widetilde{\tr\chibar} + \Gammaslash \hsp \widetilde{\tr\chibar}) \hsp \gslash
+ \widetilde{\tr\chibar}\hsp \Gammaslash \hsp \gslash
\Big] \pslash
+ \widetilde{\tr\chibar} \hsp  \gslash
\Big[
(\chihat \hsp p^{4} + \tr\chi \hsp p^{4} + \chibarhat \hsp p^{3} + \tr\chibar  \hsp p^{3} + \Gammaslash \hsp \pslash) + p^{3} \hsp |u|^{-1}
\Big] \\ \notag 
& + (\nablasl \chibarhat + \Gammaslash \hsp \chibarhat) \hsp \pslash
+ \chibarhat \Big[
(\chihat \hsp p^{4} + \tr\chi \hsp p^{4} + \chibarhat \hsp p^{3} + \tr\chibar \hsp p^{3} + \Gammaslash \hsp  \pslash) + p^{3} \hsp |u|^{-1}
\Big] \\ \notag 
& + (\nablasl \chihat + \Gammaslash \hsp \chihat) \hsp \pslash p^{4} \\ \notag 
& + \chihat \Big[
(\chihat \hsp p^{4} + \tr\chi \hsp  p^{4} + \chibarhat \hsp p^{3} + \tr\chibar \hsp p^{3} + \Gammaslash \hsp \pslash)(p^{4} + \gslash \hsp \pslash \hsp \pslash ) 
+ (\etabar \hsp p^{3} + \chihat  \hsp \pslash + \tr\chi \hsp \pslash)\hsp \pslash \hsp p^{4} + p^{3} \hsp |u|^{-1} (p^{4} + \gslash \hsp \pslash\hsp \pslash)
\Big] \\ \notag 
& + \Big[
(\nablasl \psi_{g} + \Gammaslash \hsp \psi_{g}) \hsp \gslash + \psi_{g} \hsp \Gammaslash \hsp \gslash
\Big] \pslash \hsp p^{4} \\
\notag  & + \psi_{g} \hsp \gslash \Big[
(\chihat \hsp^{4} + \tr\chi \hsp p^{4} + \chibarhat \hsp p^{3} + \tr\chibar \hsp p^{3} + \Gammaslash \hsp \pslash)(p^{4} + \gslash \hsp\pslash\hsp \pslash) 
+ (\etabar \hsp p^{3} + \chihat \hsp \pslash + \tr\chi \hsp \pslash) \pslash \hsp p^{4} + p^{3} \hsp |u|^{-1} \hsp (p^{4} + \gslash \hsp \pslash\hsp \pslash)
\Big] \\ \notag 
& + \Big[
(\nablasl \Gammaslash + \Gammaslash \hsp \Gammaslash) \hsp \gslash + \Gammaslash \hsp \Gammaslash \hsp \gslash
\Big]\hsp \pslash \hsp \pslash \\ \notag 
& + \Gammaslash \hsp \gslash \hsp \Big[
(\chihat \hsp p^{4} + \tr\chi \hsp p^{4} + \chibarhat \hsp p^{3} + \tr\chibar \hsp p^{3} + \Gammaslash \hsp \pslash)(2\hsp \pslash) + p^{3} \hsp |u|^{-1} \hsp (2\hsp  \pslash)
\Big] \\ \notag 
& + (\nablasl \psi_{g} + \Gammaslash \hsp \psi_{g}) \hsp p^{4} 
+ \psi_{g} \Big[
(\etabar \hsp p^{3} + \chihat \hsp \pslash + \tr\chi \hsp \pslash) \hsp p^{4} + p^{3} \hsp |u|^{-1} \hsp p^{4}
\Big] \\ \notag 
& + (\nablasl \chibarhat + \Gammaslash \hsp \chibarhat) \hsp \pslash \hsp p^{4} \\ \notag 
& + \chibarhat \hsp \Big[
(\chihat \hsp p^{4} + \tr\chi \hsp p^{4} + \chibarhat \hsp p^{3} + \tr\chibar \hsp p^{3} + \Gammaslash \hsp \pslash)(p^{4} + \gslash \hsp \pslash\hsp \pslash) \\ \notag 
& \quad + (\etabar \hsp p^{3} + \chihat \hsp \pslash + \tr\chi \hsp \pslash) \pslash \hsp p^{4} + p^{3} \hsp |u|^{-1} (p^{4} + \gslash \hsp \pslash \hsp \pslash)
\Big] \\ \notag 
& + \Big[
(\nablasl \psi_{g} + \Gammaslash \hsp \psi_{g}) \hsp \gslash + \psi_{g}\hsp \Gammaslash \hsp \gslash
\Big] \hsp \pslash \hsp p^{4} \\ \notag 
& + \psi_{g} \hsp \gslash \Big[
(\chihat \hsp p^{4} + \tr\chi \hsp p^{4} + \chibarhat \hsp p^{3} + \tr\chibar \hsp p^{3} + \Gammaslash \hsp \pslash)(p^{4} + \gslash \hsp  \pslash \hsp \pslash) \\ 
& \quad + (\etabar \hsp p^{3} + \chihat \hsp \pslash + \tr\chi \hsp \pslash)\hsp  \pslash \hsp p^{4} + p^{3}\hsp  |u|^{-1} \hsp (p^{4} + \gslash \hsp  \pslash \hsp \pslash)
\Big]\Bigg]_{DA}.
\end{align}
Lastly, we want to write down the error term $V_{D}(C^{4+B}_{A})$ schematically. Once again, we only note the non-repeating terms. The error terms read
\begin{align}
\label{eq:V4} \notag 
& V_D[C^{4+B}_{A}]\\ \notag 
\sim\, & \Bigg[ |u|^2 \Big[
(\nablasl \chihat \hsp \etabar + \chihat \hsp \nablasl \etabar + \nablasl \tr\chi \hsp \etabar + \tr\chi \hsp \nablasl \etabar)(p^4)^2 \\ \notag 
& + (\nablasl \chihat \hsp \omegabar + \chihat \hsp \nablasl \omegabar + \nablasl \tr\chi \hsp \omegabar + \tr\chi \hsp \nablasl \omegabar)\hsp p^4 \\ \notag 
& + (\nablasl \chihat \cdot \chibarhat + \chihat \cdot \nablasl \chibarhat + \nablasl \tr\chi \cdot \chibarhat + \tr\chi \cdot \nablasl \chibarhat)\hsp  \pslash \hsp \pslash \\ \notag 
& + (\nablasl \chihat \cdot \gslash \tr\chibar + \chihat \cdot \nablasl(\gslash \tr\chibar) + \nablasl \tr\chi \cdot \gslash \tr\chibar + \tr\chi \cdot \nablasl(\gslash \tr\chibar))\hsp  \pslash \hsp \pslash \\ \notag 
&+\left[
\left( \nablasl \psi_g \cdot \tr\chi + \psi_g \cdot \nablasl \tr\chi \right)
+ \left( \Gammaslash \psi_g \cdot \tr\chi + \Gammaslash \psi_g \cdot \tr\chi \right)
\right]\hsp  \pslash \hsp p^4 \\
\notag  & + \psi_g \tr\chi \cdot 
\Big[
(\chihat \hsp p^4 + \tr\chi  \hsp p^4 + \chibarhat \hsp p^3 + \tr\chibar \hsp p^3 + \Gammaslash \hsp \pslash) \cdot (p^4 + \gslash\hsp \pslash \hsp \pslash) \\ \notag 
&\qquad + (\etabar \hsp  p^3 + \chihat \hsp  \pslash + \tr\chi \hsp  \pslash) \cdot \pslash  \hsp  p^4+ p^3 \hsp  |u|^{-1} \cdot (p^4 + \gslash\hsp  \pslash\hsp  \pslash)
\Big]\\ \notag 
&+\left[
\nablasl \chihat \cdot \gslash + \Gammaslash \hsp \chihat \cdot \gslash + \chihat \cdot \Gammaslash \hsp  \gslash
\right] \cdot \pslash \hsp p^4 \\ \notag 
& + \chihat \cdot \gslash \cdot
\Big[
(\chihat \hsp  p^4 + \tr\chi \hsp  p^4 + \chibarhat \hsp  p^3 + \tr\chibar \hsp  p^3 + \Gammaslash \hsp \pslash)(p^4 + \gslash \hsp \pslash) \\ \notag 
&\qquad + (\etabar \hsp  p^3 + \chihat \hsp  \pslash + \tr\chi \hsp \pslash) \cdot \pslash \hsp p^4 + p^3 |u|^{-1} (p^4 + \gslash \hsp  \pslash)
\Big]\\ \notag 
&+|u|^2 \Bigg\{ \left( \nablasl \chihat + \nablasl \tr\chi + \Gammaslash\hsp  \chihat + \Gammaslash \hsp  \tr\chi \right) \cdot \left[\omegabar \hsp  p^4 + \etabar \hsp (p^4)^2 +(\chibarhat + \gslash \hsp  \tr\chibar) \pslash\hsp  \pslash \hsp \right] \\ \notag 
&\quad + (\chihat + \tr\chi) \cdot \Big[
\left( \nablasl \omegabar + \Gammaslash \hsp \omegabar \right) p^4 + \left( \nablasl \etabar + \Gammaslash \hsp \etabar \right) (p^4)^2 \\ 
&\qquad + \left( \nablasl \chibarhat + \Gammaslash \hsp  \chibarhat \right) \pslash \hsp \pslash + \left( \nablasl \tr\chibar + \Gammaslash \hsp  \tr\chibar \right) \hsp \gslash \hsp \pslash \hsp \pslash + \tr\chibar \hsp \cdot \Gammaslash \hsp \pslash \hsp  \pslash
\Big] \Bigg\}\Bigg]^{B}_{DA}
\end{align}
Once we have the schematics, we are in a position to estimate the energy $\mathfrak{V}^{1}_{2}$. The important point to note here is that we need to estimate each type of term separately and carefully. Recall the transport inequality for $\mathfrak{V}^{1}_{2}$ and perform the change of variable $s\mapsto u(s)$ together with Proposition \ref{integrate} : 
\begin{align}
\mathfrak{V}^{1}_{2}(u)\lesssim \mathfrak{V}^{1}_{2}(u_{\infty})+\intu \int_{S_{u^{\prime},\ubar}}\big(\int_{\mathcal{P}_{x}}V_{A}V_{B}f\hsp X[V_{A}V_{B}f] \hsp \text{d}\mathcal{P}_{x}\big)\sqrt{\det\gslash} \hsp \text{d}\theta^1 \text{d}\theta^2 \text{d}u^{\prime}.
 \end{align}
Now recall that the error term contains two types of terms: 
\begin{eqnarray}
 X[V_{A}V_{B}f]\sim \sum_{I=0}^{6}C^{I}_{A}V_{I}V_{B}f+\sum_{I=0}^{6}V_{A}(C^{I}_{B})V_{I}f,   
\end{eqnarray}
hence substituting in the error integral yields
\begin{align}
&\int_{u_{\infty}}^{u}\int_{S_{u^{\prime},\ubar}}\big(\int_{\mathcal{P}_{x}}V_{A}V_{B}f\hsp X[V_{A}V_{B}f]\hsp \text{d}\mathcal{P}_{x}\big)\hsp \sqrt{\det\gslash} \hsp \text{d}\theta^1 \text{d}\theta^2 \text{d}u^{\prime} \notag \\
=&\int_{u_{\infty}}^{u}\int_{S_{u^{\prime},\ubar}}\big(\int_{\mathcal{P}_{x}}V_{A}V_{B}f\Bigg[\sum_{I=0}^{6}C^{I}_{A}V_{I}V_{B}f+\sum_{I=0}^{6}V_{A}(C^{I}_{B})V_{I}f\Bigg] \text{d}\mathcal{P}_{x}\big)\sqrt{\det\gslash}\hsp \text{d}\theta^1 \text{d}\theta^2 \text{d}u^{\prime}.
\end{align}
Now for the first term 
\begin{align}
 \int_{u_{\infty}}^{u}\int_{S_{u^{\prime},\ubar}}\big(\int_{\mathcal{P}_{x}}V_{A}V_{B}f\Big[\sum_{I=0}^{6}C^{I}_{A}V_{I}V_{B}f\Big] d\mathcal{P}_{x}\big)\sqrt{\det\gslash}\hsp \text{d}\theta^1  \text{d}\theta^2  \text{d}u^{\prime},   
\end{align}
the estimates follow in a similar fashion as the first derivative estimates since we already have the estimates for $C^{I}_{A}$. We control this term as follows 
\begin{align} \notag 
&\int_{u_{\infty}}^{u}\int_{S_{u^{\prime},\ubar}}\big(\int_{\mathcal{P}_{x}}V_{A}V_{B}f\Big[\sum_{I=0}^{6}C^{I}_{A}V_{I}V_{B}f\Big] d\mathcal{P}_{x}\big)\sqrt{\det\gslash}\hsp \text{d}\theta^1  \text{d}\theta^2  \text{d}u^{\prime}\\\nonumber 
=&\int_{u_{\infty}}^{u}\int_{S_{u^{\prime},\ubar}}\big(\int_{\mathcal{P}_{x}}V_{A}V_{B}f\Big[\sum_{I\in \{0,1,2,3,4\}}C^{I}_{A}V_{I}V_{B}f\Big] d\mathcal{P}_{x}\big)\sqrt{\det\gslash}\hsp \text{d}\theta^1  \text{d}\theta^2  \text{d}u^{\prime}\\\nonumber 
+&\int_{u_{\infty}}^{u}\int_{S_{u^{\prime},\ubar}}\big(\int_{\mathcal{P}_{x}}V_{A}V_{B}f\Big[\sum_{D=1}^{2}C^{4+D}_{A}V_{4+D}V_{B}f\Big] d\mathcal{P}_{x}\big)\sqrt{\det\gslash}\hsp \text{d}\theta^1  \text{d}\theta^2  \text{d}u^{\prime}\\\nonumber 
\lesssim &\sum_{I\in \{0,1,2,3,4\}}\int_{u_{\infty}}^{u}\left(\frac{C\Gamma a^{\frac{1}{2}}}{|u|^{3}}+\frac{a^{\frac{1}{2}}}{|u|^{2}}+\frac{C\Gamma }{|u|^{2}}\right)\int_{S_{u^{\prime},\ubar}}\big(\int_{\mathcal{P}_{x}}V_{A}V_{B}f V_{I}V_{B}f d\mathcal{P}_{x}\big)\sqrt{\det\gslash}\hsp \text{d}\theta^1  \text{d}\theta^2  \text{d}u^{\prime}\\ 
+&\int_{u_{\infty}}^{u}\frac{\Gamma a^{\frac{1}{2}}}{|u|^{2}}\int_{S_{u^{\prime},\ubar}}\big[\int_{\mathcal{P}_{x}}|u|V_{A}V_{B}f V_{4+D}V_{B}f d\mathcal{P}_{x}\big]\sqrt{\det\gslash}\hsp \text{d}\theta^1  \text{d}\theta^2  \text{d}u^{\prime}.
\end{align}
Now recall that we have made the bootstrap assumption \eqref{V2bootstrap} at the start, namely
\begin{align}
\mathfrak{V}_{2}:=\frac{1}{a^{2}}\sum_{V_{(1)},V_{(2)}\in \tilde{V}}\sup_{u,\ubar}\int_{S_{u,\ubar}}\nonumber \int_{\mathcal{P}_{x}}|\tilde{V}_{(1)}\tilde{V}_{(2)} f|^{2}\sqrt{\det\gslash}\frac{dp^{1}dp^{2}dp^{3}}{p^{3}}\sqrt{\det\gslash}\hsp d^{2}\theta \leq \mathcal{V},   
\end{align}
therefore the terms of the first type belonging to the error terms yield
\begin{align}
\bigg |\int_{u_{\infty}}^{u}\int_{S_{u^{\prime},\ubar}}\big(\int_{\mathcal{P}_{x}}V_{A}V_{B}f\Big[\sum_{I=0}^{6}C^{I}_{A}V_{I}V_{B}f\Big] d\mathcal{P}_{x}\big )\sqrt{\det\gslash}\hsp \text{d}\theta^1 \text{d}\theta^2 \duprime \bigg|
\lesssim \frac{\Gamma\mathcal{V}a^{\frac{5}{2}}}{|u|}.
\end{align}
Now consider the remaining terms 
\begin{align}
 \nonumber &\sum_{I=0}^{6}\int_{u_{\infty}}^{u}\int_{S_{u^{\prime},\ubar}}\Big[\int_{\mathcal{P}_{x}}V_{A}V_{B}f V_{A}(C^{I}_{B})V_{I}f d\mathcal{P}_{x}\Big]\sqrt{\det\gslash}\hsp \text{d}\theta^1 \text{d}\theta^2 \duprime\\\nonumber 
 =&\sum_{I=0}^{4}\int_{u_{\infty}}^{u}\int_{S_{u^{\prime},\ubar}}\Big[\int_{\mathcal{P}_{x}}V_{A}V_{B}f V_{A}(C^{I}_{B})V_{I}f d\mathcal{P}_{x}\Big]\sqrt{\det\gslash}\hsp \text{d}\theta^1 \text{d}\theta^2 \duprime\\
 +&\sum_{D=1}^{2}\int_{u_{\infty}}^{u}\int_{S_{u^{\prime},\ubar}}\Big[\int_{\mathcal{P}_{x}}V_{A}V_{B}f V_{A}(C^{4+D}_{B})V_{4+D}f d\mathcal{P}_{x}\Big]\sqrt{\det\gslash}\hsp \text{d}\theta^1 \text{d}\theta^2 \duprime.
\end{align}
To estimate these terms systematically, we adopt the following strategy. The error term $V_{A}(C^{I}_{B})$ schematically reads as follows¨ 
\begin{eqnarray}
V_{A}(C^{I}_{B})\sim \nablasl\psi \cdot\hsp  G(\psi,\pslash,p^{4})+\text{terms~involving~1~derivative~of~stress tensor}+ l.o.t.,  
\end{eqnarray}
where $l.o.t$ are the terms that do not contain derivatives of $\psi$ and $G$ is a polynomial function of its entries and does not contain any derivatives. Here, $\psi$ belongs to the Ricci coefficients-Weyl curvature set $\{\chihat,\tr\chi,\chibarhat,\tr\chibar,\eta,\etabar,\omegabar,\alpha,\alphabar,\beta,\betabar,\rho,\sigma\}$. First, we consider the derivative terms
\begin{align}
 \int_{u_{\infty}}^{u}\int_{S_{u^{\prime},\ubar}}\int_{\mathcal{P}_{x}}V_{A}V_{B}f \hsp G(\psi,\pslash,p^{4})\nablasl \psi V_{B}fd\mathcal{P}_{x}\sqrt{\det\gslash}\hsp \text{d}\theta^1 \text{d}\theta^2 \duprime  .  
\end{align}
We control these in the following way, 
\begin{align} \notag
&\big|\int_{u_{\infty}}^{u}\int_{S_{u^{\prime},\ubar}}\int_{\mathcal{P}_{x}}V_{A}V_{B}f \hsp G(\psi,\pslash,p^{4})\nablasl \psi \hsp  V_{B}f\hsp \text{d}\mathcal{P}_{x}\sqrt{\det\gslash}\hsp \text{d}\theta^1 \text{d}\theta^2 \duprime\big|\\\nonumber 
\leq &\int_{u_{\infty}}^{u}\sup_{\ubar^{\prime},S_{u^{\prime},\ubar^{\prime}},\mathcal{P}_{x}}|G(\psi,\pslash,p^{4})|\Big[\int_{S_{u^{\prime},\ubar^{\prime}}}|\nablasl \psi|||V_{A}V_{B}f||_{L^{2}_{\mathcal{P}_{x}}}||V_{B}f||_{L^{2}_{\mathcal{P}_{x}}}\sqrt{\det\gslash}\Big]\text{d}u^{\prime}  \\\nonumber 
\lesssim &\int_{u_{\infty}}^{u}\sup_{\ubar^{\prime},S_{u^{\prime},\ubar^{\prime}},\mathcal{P}_{x}}|G(\psi,\pslash,p^{4})|\Big[||\nablasl\psi||_{L^{2}({S_{u^{\prime},\ubar^{\prime}}})}\int_{S_{u^{\prime},\ubar}}\int_{\mathcal{P}_{x}}|V_{A}V_{B}f|^{2}d\mathcal{P}_{x}\sqrt{\det\gslash}\Big]\text{d}u^{\prime} \\ 
\lesssim &\int_{u_{\infty}}^{u}\sup_{\ubar,S_{u^{\prime},\ubar^{\prime}},\mathcal{P}_{x}}|G(\psi,\pslash,p^{4})|\sup_{\ubar}||\nablasl\psi||_{L^{2}({S_{u^{\prime},\ubar^{\prime}}})}\Big[\int_{S_{u^{\prime},\ubar^{\prime}}}\int_{\mathcal{P}_{x}}|V_{A}V_{B}f|^{2}d\mathcal{P}_{x}\sqrt{\det\gslash}\Big]\text{d}u^{\prime} ,
\end{align}
where we have used the fact that $||V_{B}f||_{L^{2}_{\mathcal{P}_{x}}}$ is uniformly bounded by invoking the estimates, from Proposition \ref{firstvlasov}. The idea is that we use the bootstrap on \[\frac{1}{a^{2}}\Bigg[\int_{S_{u^{\prime},\ubar}}\int_{\mathcal{P}_{x}}|V_{A}V_{B}f|^{2}d\mathcal{P}_{x}\sqrt{\det\gslash}\Bigg],\] uniformly over $u$, which leaves 
\begin{align}\notag
&\big|\int_{u_{\infty}}^{u}\int_{S_{u^{\prime},\ubar}}\int_{\mathcal{P}_{x}}V_{A}V_{B}f \hsp G(\psi,\pslash,p^{4})\hsp \nablasl \psi \hsp  V_{B}f\hsp d\mathcal{P}_{x}\sqrt{\det\gslash}\hsp \text{d}\theta^1 \text{d}\theta^2 \duprime \big| \\ \lesssim& \hsp a^{2}\mathcal{V}\int_{u_{\infty}}^{u}\hsp \sup_{\ubar^{\prime},S_{u^{\prime},\ubar^{\prime}},\mathcal{P}_{x}}|G(\psi,\pslash,p^{4})|\sup_{\ubar^{\prime}}||\nablasl\psi||_{L^{2}({S_{u^{\prime},\ubar^{\prime}}})} \text{d}u^{\prime}.  
\end{align}
Also, recall that the lower order terms appearing here can be estimated in $L^{\infty}(S_{u,\ubar})$. Now we use the explicit expressions for the error terms $\{V_{D}(C^{I}_{A})\}_{I=0}^{7}$. We estimate: 
\paragraph{\( \nablasl \psi_{g} \cdot \psi_{g}\cdot   p^{4} \quad\Big( G(\psi, \pslash, p^4) = \psi_{g} \cdot p^{4} \Big) \)}
\begin{align}\notag
\int_{u_{\infty}}^{u} \sup_{\ubar,\, S_{u', \ubar},\, \mathcal{P}_{x}} 
|G(\psi, \pslash, p^4)| \cdot 
\sup_{\ubar'} \left\| \nablasl \psi \right\|_{L^{2}(S_{u', \ubar'})} 
\, \mathrm{d}u'
&\lesssim 
\int_{u_{\infty}}^{u}
\left(
\frac{a^{1/2} \Gamma}{|u'|} \cdot \frac{\Gamma}{|u'|} \cdot \frac{1}{|u'|^2}
+
\frac{a \Gamma}{|u'|^2} \cdot \frac{\Gamma}{|u'|} \cdot \frac{1}{|u'|^2}
\right)
\, \mathrm{d}u' \\
&\lesssim 
\frac{a^{1/2} \Gamma^2}{|u|^3} 
+ \frac{a \Gamma^2}{|u|^4}, \notag
\end{align}
$\nablasl \psi_g \cdot \psi_g \cdot \pslash \cdot \pslash$
\begin{align*}
\int_{u_\infty}^{u} \sup_{\ubar, S_{u', \ubar}, \mathcal{P}_{x}} |G(\psi, \pslash, p^4)|
\cdot \sup_{\ubar} \|\nablasl \psi\|_{L^2(S_{u', \ubar})} \, du' 
&\lesssim \int_{u_\infty}^{u} \left( \frac{a^{1/2} \Gamma}{|u'|} \cdot \frac{\Gamma}{|u'|} \cdot \frac{1}{|u'|^2} 
+ \frac{a \Gamma}{|u'|^2} \cdot \frac{\Gamma}{|u'|} \cdot \frac{1}{|u'|^2} \right) du' \\
&\lesssim \frac{a^{1/2} \Gamma^2}{|u|^3} + \frac{a \Gamma^2}{|u|^4},
\end{align*}
$\nablasl \psi_g \cdot \chibarhat \cdot \pslash + \nablasl \chibarhat \cdot \psi_g \cdot \pslash$
\begin{align*}
\int_{u_\infty}^{u} \sup_{\ubar, S_{u', \ubar}, \mathcal{P}_{x}} |G(\psi, \pslash, p^4)| 
\cdot \sup_{\ubar} \|\nablasl(\psi_g, \chibarhat)\|_{L^2(S_{u', \ubar})} \, du'
&\lesssim \int_{u_\infty}^{u} \left( \frac{a^{1/2} \Gamma}{|u'|^2} \cdot \frac{\Gamma}{|u'|} 
+ \frac{\Gamma}{|u'|} \cdot \frac{\Gamma a^{1/2}}{|u'|^2} \right) du' \\
&\lesssim \frac{a^{\frac{1}{2}} \Gamma^{2}}{|u|^2},
\end{align*}
$(\nablasl \psi_g \cdot \tr\chibar + \nablasl \tr\chibar \cdot \psi_g) \cdot \pslash$
\begin{align*}
\int_{u_\infty}^{u} \sup_{\ubar, S_{u', \ubar}, \mathcal{P}_{x}} |G(\psi, \pslash, p^4)| 
\cdot \sup_{\ubar} \|\nablasl \psi\|_{L^2(S_{u', \ubar})} \, du' 
&\lesssim \int_{u_\infty}^{u} \left( \frac{a^{1/2} \Gamma}{|u'|^2} \cdot \frac{\Gamma}{|u'|} 
+ \frac{\Gamma}{|u'|} \cdot \frac{a^{1/2} \Gamma}{|u'|^2} \right) du' \\
&\lesssim \frac{a^{1/2} \Gamma^2}{|u|^2}, 
\end{align*}
 $(\nablasl \psi_g) \cdot \gslash \cdot \pslash \cdot p^4$
\begin{align*}
\int_{u_\infty}^{u} \sup_{\ubar, S_{u', \ubar}, \mathcal{P}_{x}} |G(\psi, \pslash, p^4)| 
\cdot \sup_{\ubar} \|\nablasl \psi\|_{L^2(S_{u', \ubar})} \, du'
&\lesssim \int_{u_\infty}^{u} \left( \frac{1}{|u'|^2} \cdot \frac{a^{1/2} \Gamma}{|u'|} \right) du' \lesssim \frac{a^{1/2} \Gamma}{|u|^2},
\end{align*}
 $\nablasl \psi_g \cdot \gslash \cdot \pslash \cdot \pslash \cdot \pslash$
\begin{align*}
\int_{u_\infty}^{u} \sup_{\ubar, S_{u', \ubar}, \mathcal{P}_{x}} |G(\psi, \pslash, p^4)|
\cdot \sup_{\ubar} \|\nablasl \psi\|_{L^2(S_{u', \ubar})} \, du'
&\lesssim \int_{u_\infty}^{u} \left( \frac{1}{|u'|^3} \cdot \frac{a^{1/2} \Gamma}{|u'|} \right) du'\lesssim \frac{a^{1/2} \Gamma}{|u|^3},
\end{align*}
$\nablasl \psi_g \cdot \psi_g \cdot p^4 \cdot p^4$
\begin{align*}
\int_{u_\infty}^{u} \sup_{\ubar, S_{u', \ubar}, \mathcal{P}_{x}} |G(\psi, \pslash, p^4)| 
\cdot \sup_{\ubar} \|\nablasl \psi\|_{L^2(S_{u', \ubar})} \, du'
&\lesssim \int_{u_\infty}^{u} \left( \frac{\Gamma^2}{|u'|^4} \cdot \frac{a^{1/2}}{|u'|} \right) du'\lesssim \frac{a^{1/2} \Gamma^2}{|u|^4},
\end{align*}
$\psi_g \cdot \nablasl \psi_g \cdot \pslash \cdot p^4$
\begin{align*}
\int_{u_\infty}^{u} \sup_{\ubar, S_{u', \ubar}, \mathcal{P}_{x}} |G(\psi, \pslash, p^4)|
\cdot \sup_{\ubar} \|\nablasl \psi\|_{L^2(S_{u', \ubar})} \, du'
&\lesssim \int_{u_\infty}^{u} \left( \frac{\Gamma}{|u'|} \cdot \frac{1}{|u'|^3} \cdot \frac{a^{1/2} \Gamma}{|u'|} \right) du'\lesssim \frac{a^{1/2} \Gamma^2}{|u|^4},
\end{align*}
$\psi_g \cdot \nablasl \tr\chibar + \tr\chibar \cdot \nablasl \psi_g$
\begin{align*}
\int_{u_\infty}^{u} \sup_{\ubar, S_{u', \ubar}, \mathcal{P}_{x}} |G(\psi, \pslash, p^4)| 
\cdot \sup_{\ubar} \|\nablasl \psi\|_{L^2(S_{u', \ubar})} \, du' 
&\lesssim \int_{u_\infty}^{u} \left( \frac{\Gamma}{|u'|} \cdot \frac{a^{1/2} \Gamma}{|u'|} \right) du'\lesssim \frac{a^{1/2} \Gamma^2}{|u|}.
\end{align*}
$\nablasl(\chibarhat, \widetilde{\tr \chibar})$
\begin{align*}
\int_{u_\infty}^{u} \sup_{\ubar, S_{u', \ubar}, \mathcal{P}_{x}} |G(\psi, \pslash, p^4)|
\cdot \sup_{\ubar} \|\nablasl \psi\|_{L^2(S_{u', \ubar})} \, du'
&\lesssim \int_{u_\infty}^{u} \frac{a^{1/2} \Gamma}{|u'|^{2}} \, du' \lesssim \frac{a^{1/2} \Gamma}{|u|},
\end{align*}
$\nablasl\chibarhat \cdot p^4$

\begin{equation*}
\int_{u_{\infty}}^{u} \sup_{\ubar, S_{u',\ubar}, \mathcal{P}_{x}} |G(\psi, \pslash, p^4)| \sup_{\ubar} \left\| \nablasl \psi \right\|_{L^2(S_{u',\ubar'})} \, du' 
\lesssim \frac{a^{1/2} \Gamma}{|u|^3},
\end{equation*}
$\nablasl \psi_{g} \cdot p^4$

\begin{equation*}
\int_{u_{\infty}}^{u} \sup_{\ubar, S_{u',\ubar}, \mathcal{P}_{x}} |G(\psi, \pslash, p^4)| \sup_{\ubar} \left\| \nablasl \psi \right\|_{L^2(S_{u',\ubar'})} \, du' 
\lesssim \int_{u_{\infty}}^{u} \frac{\Gamma}{|u'|^3} \, du' 
\lesssim \frac{\Gamma}{|u|^2}
\end{equation*}
$\nablasl \alpha \cdot \pslash$
\begin{equation*}
\int_{u_{\infty}}^{u} \sup_{\ubar, S_{u',\ubar}, \mathcal{P}_{x}} |G(\psi, \pslash, p^4)| \sup_{\ubar} \left\| \nablasl \psi \right\|_{L^2(S_{u',\ubar'})} \, du' 
\lesssim \int_{u_{\infty}}^{u} \frac{\Gamma a^{1/2}}{|u'|^2} \, du' 
\lesssim \frac{\Gamma a^{1/2}}{|u|}
\end{equation*}
$\nablasl \betabar$
\begin{equation*}
\int_{u_{\infty}}^{u} \sup_{\ubar, S_{u',\ubar}, \mathcal{P}_{x}} |G(\psi, \pslash, p^4)| \sup_{\ubar} \left\| \nablasl \psi \right\|_{L^2(S_{u',\ubar'})} \, du' 
\lesssim \int_{u_{\infty}}^{u} \frac{a^{3/2} \Gamma}{|u'|^4} \, du' 
\lesssim \frac{a^{3/2} \Gamma}{|u|^3}
\end{equation*}
$\nablasl \psi_{g} \cdot \chihat \cdot \pslash \cdot p^4$
\begin{equation*}
\int_{u_{\infty}}^{u} \sup_{\ubar, S_{u',\ubar}, \mathcal{P}_{x}} |G(\psi, \pslash, p^4)| \sup_{\ubar} \left\| \nablasl \psi \right\|_{L^2(S_{u',\ubar'})} \, du' 
\lesssim \int_{u_{\infty}}^{u} \frac{a^{1/2} \Gamma^2}{|u'|^5} \, du' 
\lesssim \frac{a^{1/2} \Gamma^2}{|u|^4},
\end{equation*}
$\nablasl \psi_g \cdot \psi_g \cdot \pslash \cdot \pslash$
\begin{equation*}
\int_{u_{\infty}}^{u} \sup_{\ubar, S_{u', \ubar}, \mathcal{P}_{x}} |G(\psi, \pslash, p^4)| \cdot \sup_{\ubar} \|\nablasl \psi\|_{L^2(S_{u', \ubar'})} \, du'
\lesssim \int_{u_{\infty}}^{u} \frac{\Gamma^2}{|u'|^4} \, du'
\lesssim \frac{\Gamma^2}{|u|^3}
\end{equation*}
$\nablasl \psi_g \cdot \widetilde{\tr\chibar}$
\begin{equation*}
\int_{u_{\infty}}^{u} \sup_{\ubar, S_{u', \ubar}, \mathcal{P}_{x}} |G(\psi, \pslash, p^4)| \cdot \sup_{\ubar} \|\nablasl \psi\|_{L^2(S_{u', \ubar'})} \, du'
\lesssim \int_{u_{\infty}}^{u} \frac{\Gamma^2}{|u'|^3} \, du'
\lesssim \frac{\Gamma^2}{|u|^2}
\end{equation*}
$\nablasl \widetilde{\tr\chibar} \cdot \psi_g$
\begin{equation*}
\int_{u_{\infty}}^{u} \sup_{\ubar, S_{u', \ubar}, \mathcal{P}_{x}} |G(\psi, \pslash, p^4)| \cdot \sup_{\ubar} \|\nablasl \psi\|_{L^2(S_{u', \ubar'})} \, du'
\lesssim \int_{u_{\infty}}^{u} \frac{\Gamma^2}{|u'|^3} \, du'
\lesssim \frac{\Gamma^2}{|u|^2}
\end{equation*}
$\nablasl \chihat \cdot \pslash \cdot p^4$
\begin{equation*}
\int_{u_{\infty}}^{u} \sup_{\ubar, S_{u', \ubar}, \mathcal{P}_{x}} |G(\psi, \pslash, p^4)| \cdot \sup_{\ubar} \|\nablasl \psi\|_{L^2(S_{u', \ubar'})} \, du'
\lesssim a^{1/2} \Gamma \int_{u_{\infty}}^{u} \frac{1}{|u'|^4} \, du'
\lesssim \frac{a^{1/2} \Gamma}{|u|^3}
\end{equation*}
$\chibarhat \cdot \nablasl \chihat$
\begin{equation*}
\int_{u_{\infty}}^{u} \sup_{\ubar, S_{u', \ubar}, \mathcal{P}_{x}} |G(\psi, \pslash, p^4)| \cdot \sup_{\ubar} \|\nablasl \psi\|_{L^2(S_{u', \ubar'})} \, du'
\lesssim a \Gamma^2 \int_{u_{\infty}}^{u} \frac{1}{|u'|^3} \, du'
\lesssim \frac{a \Gamma^2}{|u|^2}
\end{equation*}
$\nablasl \chibarhat \cdot \chihat$
\begin{equation*}
\int_{u_{\infty}}^{u} \sup_{\ubar, S_{u', \ubar}, \mathcal{P}_{x}} |G(\psi, \pslash, p^4)| \cdot \sup_{\ubar} \|\nablasl \psi\|_{L^2(S_{u', \ubar'})} \, du'
\lesssim a \Gamma^2 \int_{u_{\infty}}^{u} \frac{1}{|u'|^3} \, du'
\lesssim \frac{a \Gamma^2}{|u|^2},
\end{equation*}
$\nablasl \psi_g \cdot \psi_g \cdot p^4$
\begin{equation*}
\int_{u_{\infty}}^{u} \sup_{\ubar, S_{u', \ubar}, \mathcal{P}_{x}} |G(\psi, \pslash, p^4)| \cdot \sup_{\ubar} \|\nablasl \psi\|_{L^2(S_{u', \ubar'})} \, du'
\lesssim \frac{\Gamma^2}{|u|^3},
\end{equation*}
$\nablasl \chihat \cdot \widetilde{\tr\chibar} \cdot \pslash$
\begin{equation*}
\int_{u_{\infty}}^{u} \sup_{\ubar, S_{u', \ubar}, \mathcal{P}_{x}} |G(\psi, \pslash, p^4)| \cdot \sup_{\ubar} \|\nablasl \psi\|_{L^2(S_{u', \ubar'})} \, du'
\lesssim \frac{a^{1/2} \Gamma^2}{|u|^3},
\end{equation*}
$\nablasl \chihat \cdot \chihat \cdot p^4 \cdot \pslash$
\begin{equation*}
\int_{u_{\infty}}^{u} \sup_{\ubar, S_{u', \ubar}, \mathcal{P}_{x}} |G(\psi, \pslash, p^4)| \cdot \sup_{\ubar} \|\nablasl \psi\|_{L^2(S_{u', \ubar'})} \, du'
\lesssim \int_{u_{\infty}}^{u} \frac{a \Gamma^2}{|u'|^5} \, du'
\lesssim \frac{a \Gamma^2}{|u|^4}
\end{equation*}
$\nablasl \psi_g \cdot \chihat \cdot \pslash \cdot \pslash$
\begin{equation*}
\int_{u_{\infty}}^{u} \sup_{\ubar, S_{u', \ubar}, \mathcal{P}_{x}} |G(\psi, \pslash, p^4)| \cdot \sup_{\ubar} \|\nablasl \psi\|_{L^2(S_{u', \ubar'})} \, du'
\lesssim \int_{u_{\infty}}^{u} \frac{a^{1/2} \Gamma^2}{|u'|^4} \, du'
\lesssim \frac{a^{1/2} \Gamma^2}{|u|^3}
\end{equation*}
$\nablasl \chihat \cdot \psi_g \cdot \pslash \cdot p^4$
\begin{equation*}
\int_{u_{\infty}}^{u} \sup_{\ubar, S_{u', \ubar}, \mathcal{P}_{x}} |G(\psi, \pslash, p^4)| \cdot \sup_{\ubar} \|\nablasl \psi\|_{L^2(S_{u', \ubar'})} \, du'
\lesssim \int_{u_{\infty}}^{u} \frac{\Gamma^2}{|u'|^5} \, du'
\lesssim \frac{\Gamma^2}{|u|^4}
\end{equation*}
$\nablasl \beta \cdot p^4$
\begin{equation*}
\int_{u_{\infty}}^{u} \sup_{\ubar, S_{u', \ubar}, \mathcal{P}_{x}} |G(\psi, \pslash, p^4)| \cdot \sup_{\ubar} \|\nablasl \psi\|_{L^2(S_{u', \ubar'})} \, du'
\lesssim \int_{u_{\infty}}^{u} \frac{a^{1/2} \Gamma}{|u'|^4} \, du'
\lesssim \frac{a^{1/2} \Gamma}{|u|^3}
\end{equation*}
$\nablasl \beta \cdot \pslash \cdot \pslash$
\begin{equation*}
\int_{u_{\infty}}^{u} \sup_{\ubar, S_{u', \ubar}, \mathcal{P}_{x}} |G(\psi, \pslash, p^4)| \cdot \sup_{\ubar} \|\nablasl \psi\|_{L^2(S_{u', \ubar'})} \, du'
\lesssim \int_{u_{\infty}}^{u} \frac{a^{1/2} \Gamma}{|u'|^4} \, du'
\lesssim \frac{a^{1/2} \Gamma}{|u|^3}
\end{equation*}
$\nablasl \alpha \cdot \pslash$
\begin{equation*}
\int_{u_{\infty}}^{u} \sup_{\ubar, S_{u', \ubar}, \mathcal{P}_{x}} |G(\psi, \pslash, p^4)| \cdot \sup_{\ubar} \|\nablasl \psi\|_{L^2(S_{u', \ubar'})} \, du'
\lesssim \int_{u_{\infty}}^{u} \frac{a^{1/2} \Gamma}{|u'|^2} \, du'
\lesssim \frac{a^{1/2} \Gamma}{|u|},
\end{equation*}
$\nablasl \alpha\cdot \pslash\cdot p^4$
\begin{align*}
&\int_{u_{\infty}}^{u} \sup_{\ubar,\, S_{u',\ubar},\, \mathcal{P}_{x}} |G(\psi,\pslash,p^4)| \sup_{\ubar} \| \nablasl \psi \|_{L^2(S_{u',\ubar})} \, du' \lesssim \int_{u_{\infty}}^{u} \frac{a^{\frac{1}{2}} \Gamma}{|u'|^4} \, du' \lesssim \frac{a^{\frac{1}{2}} \Gamma}{|u|^3}.
\end{align*}
$\nablasl \alphabar\cdot \pslash$
\begin{align*}
&\int_{u_{\infty}}^{u} \sup_{\ubar,\, S_{u',\ubar},\, \mathcal{P}_{x}} |G(\psi,\pslash,p^4)| \sup_{\ubar} \| \nablasl \psi \|_{L^2(S_{u',\ubar})} \, du'\lesssim \int_{u_{\infty}}^{u} \frac{a^2 \Gamma}{|u'|^6} \, du' \lesssim \frac{a^2 \Gamma}{|u|^5}.
\end{align*}
$\nablasl \widetilde{\tr \chibar}\cdot \pslash$
\begin{align*}
&\int_{u_{\infty}}^{u} \sup_{\ubar,\, S_{u',\ubar},\, \mathcal{P}_{x}} |G(\psi,\pslash,p^4)| \sup_{\ubar} \| \nablasl \psi \|_{L^2(S_{u',\ubar})} \, du' \lesssim \Gamma \int_{u_{\infty}}^{u} \frac{1}{|u'|^3} \, du' \lesssim \frac{\Gamma}{|u|^2}.
\end{align*}
$\nablasl \chibarhat\cdot \pslash$
\begin{align*}
&\int_{u_{\infty}}^{u} \sup_{\ubar,\, S_{u',\ubar},\, \mathcal{P}_{x}} |G(\psi,\pslash,p^4)| \sup_{\ubar} \| \nablasl \psi \|_{L^2(S_{u',\ubar})} \, du'
\lesssim a^{\frac{1}{2}} \Gamma \int_{u_{\infty}}^{u} \frac{1}{|u'|^3} \, du' \lesssim \frac{a^{\frac{1}{2}} \Gamma}{|u|^2}.
\end{align*}
$\nablasl \chihat\cdot \pslash\cdot p^4$
\begin{align*}
&\int_{u_{\infty}}^{u} \sup_{\ubar,\, S_{u',\ubar},\, \mathcal{P}_{x}} |G(\psi,\pslash,p^4)| \sup_{\ubar} \| \nablasl \psi \|_{L^2(S_{u',\ubar})} \, du'
\lesssim a^{\frac{1}{2}} \Gamma \int_{u_{\infty}}^{u} \frac{1}{|u'|^4} \, du' \lesssim \frac{a^{\frac{1}{2}} \Gamma}{|u|^3}.
\end{align*}
$\nablasl \psi_{g}\cdot \pslash\cdot p^4$
\begin{align*}
&\int_{u_{\infty}}^{u} \sup_{\ubar,\, S_{u',\ubar},\, \mathcal{P}_{x}} |G(\psi,\pslash,p^4)| \sup_{\ubar} \| \nablasl \psi \|_{L^2(S_{u',\ubar})} \, du'
\qquad\lesssim \Gamma \int_{u_{\infty}}^{u} \frac{1}{|u'|^4} \, du' \lesssim \frac{\Gamma}{|u|^3}.
\end{align*}
Now we estimate the following term
\begin{align}
 \sum_{D=1}^{2}\int_{u_{\infty}}^{u}\int_{S_{u^{\prime},\ubar}}\Bigg[\int_{\mathcal{P}_{x}}V_{A}V_{B}f \hsp V_{A}(C^{4+D}_{B})V_{4+D}f \text{d}\mathcal{P}_{x}\Bigg]\sqrt{\det\gslash}\hsp \text{d}\theta^1 \text{d}\theta^2 \duprime.   
\end{align}
Again, we have the following form 
\begin{align}
 \int_{u_{\infty}}^{u}\int_{S_{u^{\prime},\ubar}}\int_{\mathcal{P}_{x}}V_{A}V_{B}f \hsp G(\psi,\pslash,p^{4})\hsp \nablasl \psi \hsp |u^{\prime}|^{2}\hsp V_{B+4}f\hsp \text{d}\mathcal{P}_{x}\hsp \sqrt{\det\gslash}\text{d}\theta^1 \text{d}\theta^2 \duprime. 
\end{align}
We control this term in the following way ,
\begin{align}\notag
&\big|\int_{u_{\infty}}^{u}\int_{S_{u^{\prime},\ubar}}\int_{\mathcal{P}_{x}}V_{A}V_{B}f \hsp G(\psi,\pslash,p^{4})\hsp \nablasl \psi\hsp  |u^{\prime}|^{2}\hsp V_{B+4}f\hsp \text{d}\mathcal{P}_{x}\sqrt{\det\gslash}\hsp \text{d}\theta^1 \text{d}\theta^2 \duprime\big |\\\nonumber 
\leq &\int_{u_{\infty}}^{u}\sup_{\ubar^{\prime},S_{u^{\prime},\ubar},\mathcal{P}_{x}}|G(\psi,\pslash,p^{4})|\Bigg[\int_{S_{u^{\prime},\ubar}}|\nablasl \psi|||V_{A}V_{B}f||_{L^{2}_{\mathcal{P}_{x}}}|u^{\prime}|||u^{\prime}V_{B+4}f||_{L^{2}_{\mathcal{P}_{x}}}\sqrt{\det\gslash}\hsp \text{d}\theta^1 \text{d}\theta^2 \Bigg]\text{d}u^{\prime}  \\\nonumber 
\lesssim &\int_{u_{\infty}}^{u}\sup_{\ubar^{\prime},S_{u^{\prime},\ubar},\mathcal{P}_{x}}|G(\psi,\pslash,p^{4})|\Bigg[||\nablasl\psi||_{L^{2}({S_{u^{\prime},\ubar^{\prime}}})}|u^{\prime}|\int_{S_{u^{\prime},\ubar}}\int_{\mathcal{P}_{x}}|V_{A}V_{B}f|^{2}d\mathcal{P}_{x}\sqrt{\det\gslash} \hsp \text{d}\theta^1 \text{d}\theta^2 \Bigg]\text{d}u^{\prime} \\ 
\lesssim &\int_{u_{\infty}}^{u}\sup_{\ubar,S_{u^{\prime},\ubar},\mathcal{P}_{x}}|G(\psi,\pslash,p^{4})|\sup_{\ubar}||\nablasl\psi||_{L^{2}({S_{u^{\prime},\ubar^{\prime}}})}|u^{\prime}|\Bigg[\int_{S_{u^{\prime},\ubar}}\int_{\mathcal{P}_{x}}|V_{A}V_{B}f|^{2}\text{d}\mathcal{P}_{x}\sqrt{\det\gslash}\hsp \text{d}\theta^1 \text{d}\theta^2 \Bigg]\text{d}u^{\prime} ,
\end{align}
    where we have used that fact that $||\modu \hsp V_{B+4}f||_{L^{2}_{\mathcal{P}_{x}}}$ is uniformly bounded, by the first order estimate from Proposition \ref{firstvlasov}. The idea is that we use the bootstrap on $\Bigg[\int_{S_{u^{\prime},\ubar}}\int_{\mathcal{P}_{x}}|V_{A}V_{B}f|^{2}d\mathcal{P}_{x}\sqrt{\det\gslash}\hsp \text{d}\theta^1 \text{d}\theta^2 \duprime \Bigg]$, uniformly over $u$, which leaves 
\begin{align}
\notag &\big|\int_{u_{\infty}}^{u}\int_{S_{u^{\prime},\ubar}}\int_{\mathcal{P}_{x}}V_{A}V_{B}f \hsp G(\psi,\pslash,p^{4})\hsp \nablasl \psi\hsp  |u^{\prime}|^{2}\hsp V_{B+4}f\hsp \text{d}\mathcal{P}_{x}\sqrt{\det\gslash}\hsp \text{d}\theta^1 \text{d}\theta^2 \duprime\big |\\ \lesssim& a^{2}\mathcal{V}\int_{u_{\infty}}^{u}\hsp \sup_{\ubar,S_{u^{\prime},\ubar},\mathcal{P}_{x}}|G(\psi,\pslash,p^{4})|\sup_{\ubar}||\nablasl\psi||_{L^{2}({S_{u^{\prime},\ubar^{\prime}}})}|u^{\prime}| \text{d}u^{\prime} .   
\end{align}
We estimate these terms in a similar fashion, and therefore we do not present the term-by-term estimates, but rather the final one (i.e., the worst possible decay). Collecting all the terms together with Gr\"onwall inequality, we obtain 
\begin{align}
\mathfrak{V}^{1}_{2}(u)\lesssim \mathfrak{V}^{1}_{2}(u_{\infty})+\frac{a^{\frac{5}{2}}\Gamma^{2}\mathcal{V}}{|u|}, \end{align}or equivalently \begin{align} \frac{1}{a^{2}}\mathfrak{V}^{1}_{2}(u)\lesssim \frac{1}{a^{2}}\mathfrak{V}^{1}_{2}(u_{\infty})+\frac{a^{\frac{1}{2}}\Gamma^{2}\mathcal{V}}{|u|}.
\end{align}

\noindent The other derivative estimates follow in a similar fashion. In fact, the estimates are exactly the same with repeated control over the Ricci coefficients and the Weyl curvature components. Therefore, we do not present the readers with repeated computations. This completes the second derivative estimates for the Vlasov field. 
\end{proof}
\subsection{Third Derivative Control of Vlasov}\label{thirdvlasovs}
Now we move on to the third derivative estimates. 
First, we illustrate how the regularity level is controlled, since integrability is not the major issue at higher orders (it having been resolved at first order and with the understanding that extra $\tilde{V}-$derivatives preserve it). First, recall the definition of the norm that we wish to control 
for $k_{i}\in \mathbb{Z}_{\geq 0}$
\begin{align}
 \mathfrak{V}_{3}:=\frac{1}{a^{2}}\sum_{\tilde{V}_{(1)},\tilde{V}_{(2)}, \tilde{V}_{(3)}\in \tilde{V}}\sup_{u}\int_0^1 \int_{S_{u,\ubar^{\prime}}}\nonumber \int_{\mathcal{P}_{x}}|\tilde{V}_{(1)}\tilde{V}_{(2)} \tilde{V}_{(3)} f|^{2}\sqrt{\det\gslash}\hsp\frac{dp^{1}dp^{2}dp^{3}}{p^{3}}\sqrt{\det\gslash}\hsp d^{2}\theta \dubarprime.
\end{align}
\begin{proposition}
\label{thirdvlasov}
 Let \( f: \mathcal{P} \to \mathbb{R} \) denote a classical solution to the Vlasov equation posed on a smooth, globally hyperbolic Lorentzian manifold \( (\mathcal{M}, g) \), foliated by a double null coordinate system \( (u, \ubar, \theta^1, \theta^2) \). Suppose that \( f \) satisfies the transport equation along the null geodesic flow and that all relevant Ricci coefficients, curvature components, and geometric quantities satisfy the bootstrap assumptions described in Section \ref{bootstrapassumption}. Let \( \mathfrak{V}_2 \) denote the second-order energy quantity associated to \( f \), defined by $\mathfrak{V}$,
built from commutation vector fields adapted to the geometric foliation as follows 
\begin{eqnarray}
\mathfrak{V}_{3}:=\frac{1}{a^{2}}\sum_{\tilde{V}_{(1)},\tilde{V}_{(2)}, \tilde{V}_{(3)}\in \tilde{V}}\sup_{u}\int_0^1 \int_{S_{u,\ubar^{\prime}}}\nonumber \int_{\mathcal{P}_{x}}|\tilde{V}_{(1)}\tilde{V}_{(2)} \tilde{V}_{(3)} f|^{2}\sqrt{\det\gslash}\hsp \frac{dp^{1}dp^{2}dp^{3}}{p^{3}}\sqrt{\det\gslash}\hsp d^{2}\theta \dubarprime.
\end{eqnarray}

\noindent Assume, furthermore, that the initial data for \( f \) obeys the bound
\[
\mathcal{I}^0_{3} :=\mathfrak{V}_{3}:=\frac{1}{a^{2}}\sum_{\tilde{V}_{(1)},\tilde{V}_{(2)}, \tilde{V}_{(3)}\in \tilde{V}}\sup_{u}\int_0^1 \int_{S_{u_{\infty},\ubar^{\prime}}}\nonumber \int_{\mathcal{P}_{x}}|\tilde{V}_{(1)}\tilde{V}_{(2)} \tilde{V}_{(3)} f|^{2}\sqrt{\det\gslash}\frac{dp^{1}dp^{2}dp^{3}}{p^{3}}\sqrt{\det\gslash}\hsp d^{2}\theta \dubarprime
 < \infty.
\]
Then the following estimates hold
\begin{equation}
\label{eq:third_derivative_estimate}
\mathfrak{V}_3 \lesssim \left(1 + \mathcal{I}^0_{3}\right).
\end{equation}
where the involved constants are purely numerical.
\end{proposition}

\begin{proof}
First, we make the bootstrap assumption that 
\begin{align}
\mathfrak{V}_{3}\leq \mathcal{V},
\end{align}
where $\mathcal{V}$ is large enough so that 
$\mathcal{I}^{0}_{3}\ll \mathcal{V}$   
but also $\mathcal{V}\leq a^{\frac{1}{320}}$. Under the bootstrap assumption, we have 
\begin{align}
\sum_{\tilde{V}_{(1)},\tilde{V}_{(2)}, \tilde{V}_{(3)}\in \tilde{V}}\sup_{u}\int_0^1 \int_{S_{u,\ubar^{\prime}}}\nonumber \int_{\mathcal{P}_{x}}|\tilde{V}_{(1)}\tilde{V}_{(2)} \tilde{V}_{(3)} f|^{2}\sqrt{\det\gslash}\frac{dp^{1}dp^{2}dp^{3}}{p^{3}}\sqrt{\det\gslash}\hsp d^{2}\theta \dubarprime \leq a^2 \mathcal{V}.   
\end{align}
Now, we outline the strategy to close the regularity estimates using the generalized H\"older inequality on the mass shell. This is the most important aspect of the third derivative estimates. What then remains are tedious but straightforward computations. Therefore, we keep the tedious estimates as short as possible for clarity. 
For $I,J,K \in \{0,\dots,6\}$, recall the commutation formula  
\begin{equation} \label{eq:comm-triple}
    X\big[ V_{I} V_{J} V_{K} f \big] 
    = V_{I} \big( X[V_{J} V_{K} f] \big) 
      - [V_{I}, X] \, V_{J} V_{K} f.
\end{equation}
Using the expression for the double commutator, we have  
\begin{align}
    X\big[ V_{I} V_{J} V_{K} f \big] 
    &= V_{I} \big( X[V_{J} V_{K} f] \big) 
       - [V_{I}, X] \, V_{J} V_{K} f \notag \\
    &= \sum_{L=0}^{6} C^{L}_{I} \, V_{L} V_{J} V_{K} f 
       + V_{I} \Bigg[ \sum_{L=0}^{6} C^{L}_{J} \, V_{L} V_{K} f 
       + \sum_{L=0}^{6} V_{J}(C^{L}_{K}) \, V_{L} f \Bigg].
\end{align}
Since each $V_{I}$ is a derivation on the mass shell, it satisfies Leibniz's rule.  
Therefore,  
\begin{equation}
\label{eq:triple-decomp}
    X\big[ V_{I} V_{J} V_{K} f \big]
    = \underbrace{\sum_{L=0}^{6} C^{L}_{I} \, V_{L} V_{J} V_{K} f}_{\mathrm{I}}
    + \underbrace{\sum_{L=0}^{6} V_{I} V_{J}(C^{L}_{K}) \, V_{L} f}_{\mathrm{II}}
    + \underbrace{\sum_{L=0}^{6} V_{J}(C^{L}_{K}) \, V_{I} V_{L} f}_{\mathrm{III}}.
\end{equation}

\noindent The term $\mathrm{I}$ is straightforward to handle (analogous to the previous two cases) and will not be discussed in detail here.  
The term $\mathrm{II}$ is slightly more intricate; we will treat only the most dangerous contributions.  
The term $\mathrm{III}$ is the most involved and will be analyzed fully.

\noindent Consider the quantity  
\begin{equation}
    \mathfrak{V}^{1}_{3}
    := \int_{0}^{1} \int_{S_{u, \ubar'}} 
       \bigg( \int_{\mathcal{P}_{x}} \big| V_{A} V_{B} V_{C} f \big|^{2} \, \text{d}\mathcal{P}_{x} \bigg) 
       \sqrt{\det \gslash} \text{d}\theta^1 \text{d}\theta^2  \dubarprime,
\end{equation}
which, under the bootstrap assumption, verifies $\mathfrak{V}^{1}_{3}\leq a^{2}\mathcal{V}$. 
Applying the transport inequality \eqref{lem:transport} to $\mathfrak{V}^{1}_{3}$ yields  
\begin{align}
    \mathfrak{V}^{1}_{3}(u) 
    &\lesssim \mathfrak{V}^{1}_{3}(u_{\infty}) 
      + \int_{u_{\infty}}^{u} \int_{0}^{1} \int_{S_{u', \ubar'}} 
        \bigg( \int_{\mathcal{P}_{x}} V_{A} V_{B} V_{C} f \,
        X\big[ V_{A} V_{B} V_{C} f \big] \, \text{d}\mathcal{P}_{x} \bigg) 
        \sqrt{\det \gslash} \hsp \text{d}\theta^1 \text{d}\theta^2 \dubarprime \duprime.
\end{align}
Thus, it remains to control the error terms arising from $X[ V_{A} V_{B} V_{C} f ]$.  
As an example, we consider  
\begin{equation}
    \sum_{J=0}^{6} \int_{u_{\infty}}^{u} \int_{0}^{1} \int_{S_{u', \ubar'}} 
    \bigg( \int_{\mathcal{P}_x} V_{A} V_{B} V_{C} f \,
    V_{B}(C^{J}_{C}) \, V_{A} V_{J} f \, dP \bigg)
    \sqrt{\det \gslash} \hsp \text{d}\theta^1 \text{d}\theta^2 \dubarprime \duprime.
\end{equation}

\noindent The term $V_{B}(C^{0}_{C})$ has the schematic form, for $\psi$ being either a Ricci coefficient or a Weyl curvature component,
\begin{equation}
V_{B}(C^{0}_{C}) \sim \nablasl \psi \, f_{1}(\psi) \, f_{2}(\pslash) \, f_{3}(p^{4}) + \text{l.o.t.}
\end{equation}

\noindent We first control the terms containing the top-order derivative of Vlasov, the second derivative of Vlasov, and the first derivative of Ricci coefficients and Weyl curvature components.  
These terms have the schematic structure
\begin{align}
\label{eq:ultimate1}
&\int_{u_{\infty}}^{u} \int_{0}^{1} \int_{S_{u,u'}} 
\left( \int_{\mathcal{P}_{x}} V_{A}V_{B}V_{C}f \hsp \nablasl \psi \, f_{1}(\psi) \, f_{2}(\pslash) \, f_{3}(p^{4}) \, V_{A}V_{0}f \, \text{d}\mathcal{P}_x \right) 
\sqrt{\det\gslash}\hsp \text{d}\theta^1 \text{d}\theta^2 \dubarprime \duprime  \notag \\
 \lesssim& \int_{u_{\infty}}^{u} \int_{0}^{1} \int_{S_{u',\ubar}} 
|\nablasl \psi| \, |f_{1}(\psi)| \, 
\sup_{\mathcal{P}_{x}} |f_{2}(\pslash)| \, \sup_{\mathcal{P}_{x}} |f_{3}(p^{4})| \,
\| V_{A}V_{B}V_{C}f \|_{L^{2}(\mathcal{P}_{x})} \, 
\| V_{A}V_{0}f \|_{L^{2}(\mathcal{P}_{x})} \notag \\
&\hspace{6cm} \times \sqrt{\det\gslash}\hsp \text{d}\theta^1 \text{d}\theta^2 \dubarprime \duprime  \notag \\
\lesssim& \int_{u_{\infty}}^{u} \int_{0}^{1} 
\sup_{\mathcal{P}_{x}} |f_{2}(\pslash)| \, \sup_{\mathcal{P}_{x}} |f_{3}(p^{4})| \,
\|\nablasl \psi\|_{L^{4}(S_{u',\ubar'})} \,
\|f_{1}(\psi)\|_{L^{\infty}(S_{u',\ubar'})} \notag \\
&\hspace{2cm} \times 
\|V_{A}V_{B}V_{C}f\|_{L^{2}(S_{u',\ubar'})L^{2}(\mathcal{P}_{x})} \,
\|V_{A}V_{0}f\|_{L^{4}(S_{u',\ubar'})L^{2}(\mathcal{P}_{x})} 
\, \dubarprime \duprime \notag \\
\quad \lesssim& \int_{u_{\infty}}^{u} 
\sup_{\ubar} \|\nablasl \psi\|_{L^{4}(S_{u',\ubar})} \,
\sup_{\ubar} \|f_{1}(\psi)\|_{L^{\infty}(S_{u',\ubar})} \,
\sup_{\mathcal{P}_{x}} |f_{2}(\pslash)| \,
\sup_{\mathcal{P}_{x}} |f_{3}(p^{4})| \notag \\
&\hspace{2cm} \times \int_{0}^{1} 
\|V_{A}V_{B}V_{C}f\|_{L^{2}(S_{u',\ubar'})L^{2}(\mathcal{P}_{x})} \,
\|V_{A}V_{0}f\|_{L^{4}(S_{u',\ubar'})L^{2}(\mathcal{P}_{x})} 
\, \dubarprime \duprime.
\end{align}
Now we claim that the integral 
\begin{align}
\label{eq:border1}
\int_{0}^{1}||V_{A}V_{B}V_{C}f||_{L^{2}(S_{u^{\prime},\ubar^{\prime}})L^{2}(\mathcal{P}_{x})}||V_{A}V_{0}f||_{L^{4}(S_{u^{\prime},\ubar^{\prime}})L^{2}(\mathcal{P}_{x})}\text{d}\ubar^{\prime}    
\end{align}
can be controlled by means of the boot-strap assumptions on the top order norm of $f$ and the lower order norms estimated earlier. Note that $||V_{A}V_{0}f||_{L^{2}(\mathcal{P}_{x})}$ is a function on spacetime and particular on $S_{u,\ubar}$. Therefore, we can apply Sobolev embedding on the sphere. Recall the Sobolev embedding for a function $\mathcal{G}$ on $S_{u,\ubar}$, 
\begin{eqnarray}
||\mathcal{G}||_{L^{4}(S_{u,\ubar})}\lesssim |u|^{\frac{1}{2}}||\nablasl \mathcal{G}||_{L^{2}(S_{u,\ubar})}+|u|^{-\frac{1}{2}}||\mathcal{G}||_{L^{2}(S_{u,\ubar})}.    
\end{eqnarray}
Applying this to the function $\mathcal{G}=||V_{A}V_{0}f||_{L^{2}(\mathcal{P}_{x})}$, we obtain 
\begin{align}
||V_{A}V_{0}f||_{L^{4}(S_{u^{\prime},\ubar^{\prime}})L^{2}(\mathcal{P}_{x})}\lesssim& |u|^{\frac{1}{2}} \left[\int_{S_{u,\ubar}}\left(\nablasl \left(\int_{\mathcal{P}_{x}}|V_{A}V_{0}f|^{2} d\mathcal{P}_{x}\right)^{\frac{1}{2}} \right)^{2} \right]^{\frac{1}{2}}\nonumber\\ \notag  +&|u|^{-\frac{1}{2}}\Bigg[\int_{S_{u,\ubar}}\left(\int_{\mathcal{P}_{x}}|V_{A}V_{0}f|^{2} d\mathcal{P}_{x}\right)\Bigg]^{\frac{1}{2}} \\ 
\lesssim& |u|^{\frac{1}{2}} \left[\int_{S_{u,\ubar}}\left(\nablasl \left(\int_{\mathcal{P}_{x}}|V_{A}V_{0}f|^{2} d\mathcal{P}_{x}\right)^{\frac{1}{2}} \right)^{2} \right]^{\frac{1}{2}}+a|u|^{-1}\mathcal{I}^{0}_{2},
\end{align}
by using the second derivative estimates from Proposition \ref{prop:second_derivative_estimate}. Now, note the following computations 
\begin{eqnarray}
\nablasl\left(\int_{\mathcal{P}_{x}}|V_{A}V_{0}f|^{2} \right)^{\frac{1}{2}}\nonumber=\frac{\int_{P}V_{A}V_{0}f \nablasl V_{A}V_{0}f}{\sqrt{\left(\int_{\mathcal{P}_{x}}|V_{A}V_{0}f|^{2} \right)}}+|\Gammaslash|||V_{A}V_{0}f||_{L^{2}(\mathcal{P}_{x})} \lesssim ||\nablasl V_{A}V_{0}f||_{L^{2}(\mathcal{P}_{x})}+ |\Gammaslash|||V_{A}V_{0}f||_{L^{2}(\mathcal{P}_{x})},  
\end{eqnarray}
and recall the following relation between the derivatives 
\begin{align}  e_A =& V_{(A)}  + \modu^2 \left( \frac{\Gammaslash^C_{AB}\hsp p^B}{p^3} + {\chibarhat_A}^C +\frac{{\chi_A}^C \hsp p^4}{p^3}\right) V_{(4+C)}  \notag \\ &+\frac{\modu^2}{2} (\tr\chibar+\frac{2}{\modu}) V_{(4+A)} +\left( \frac{\frac{1}{2}\chi_{AB}\hsp p^B}{p^3} -\etabar_A\right)(V_{(4)}+V_{(0)}) \label{expressionVA-2} \end{align}    
and the fact that, schematically,
\begin{align}
&\nablasl V_{A}V_{0}f \sim\notag \\ \sim &V_{D}V_{A}V_{0}f+|u|^{2} \hsp (\Gammaslash\hsp \pslash+\chibarhat+\chi \hsp p^{4})\hsp V_{4+D}V_{A}V_{0}f+|u|^{2}\hsp \widetilde{\tr\chibar}\hsp V_{4+D}V_{A}V_{0}f+(\chi\hsp \pslash+\etabar)(V_{4}V_{A}V_{0}f+V_{0}V_{A}V_{0}f).     
\end{align}
Finally, this yields 
\begin{align}
\notag &|u|^{\frac{1}{2}} \left[\int_{S_{u,\ubar}}\left(\nablasl \left(\int_{\mathcal{P}_{x}}|V_{A}V_{0}f|^{2} d\mathcal{P}_{x}\right)^{\frac{1}{2}} \right)^{2} \sqrt{\det\gslash}\hsp \text{d}\theta^1 \text{d}\theta^2 \right]^{\frac{1}{2}}\\ \notag
\lesssim& |u|^{\frac{1}{2}}\Bigg[\int_{S_{u,\ubar}}\Bigg(||V_{D}V_{A}V_{0}f+|u|^{2}(\Gammaslash\pslash+\chibarhat+\chi p^{4})V_{4+D}V_{A}V_{0}f+|u|^{2}\widetilde{\tr\chibar}V_{4+D}V_{A}V_{0}f\\ \notag +&(\chi\pslash+\etabar)(V_{4}V_{A}V_{0}f+V_{0}V_{A}V_{0}f)||_{L^{2}(\mathcal{P}_{x})}\Bigg)^{2}\sqrt{\det\gslash}\hsp \text{d}\theta^1 \text{d}\theta^2 \Bigg]^{\frac{1}{2}}\\ \notag
\lesssim& |u|^{\frac{1}{2}}\Bigg[\int_{S_{u,\ubar}}\Bigg(||V_{D}V_{A}V_{0}f||^{2}_{L^{2}(\mathcal{P}_{x})}+a^{2}\Gamma^{2}||V_{4+D}V_{A}V_{0}f||^{2}_{L^{2}(\mathcal{P}_{x})}+\Gamma^{2}||V_{4+D}V_{A}V_{0}f||^{2}_{L^{2}(\mathcal{P}_{x})}\\ \notag
+&\frac{a\Gamma^{2}}{|u|^{4}}||V_{4}V_{A}V_{0}f||_{L^{2}(\mathcal{P}_{x})}\Bigg)\sqrt{\det\gslash} \Bigg]^{\frac{1}{2}}\\ \notag
\lesssim& |u|^{\frac{1}{2}}\Bigg[||V_{D}V_{A}V_{0}f||_{L^{2}(S_{u,\ubar})L^{2}(\mathcal{P}_{x})}+a\hsp \Gamma \hsp ||V_{4+D}V_{A}V_{0}f||_{L^{2}(S_{u,\ubar})L^{2}(\mathcal{P}_{x})}\\+&\Gamma \hsp ||V_{4+D}V_{A}V_{0}f||_{L^{2}(S_{u,\ubar})L^{2}(\mathcal{P}_{x})}+\frac{a^{\frac{1}{2}}\Gamma}{|u|^{2}}||V_{4}V_{A}V_{0}f||_{L^{2}(S_{u,\ubar})L^{2}(\mathcal{P}_{x})}\Bigg].
\end{align}
Therefore, the integral in \eqref{eq:border1} is estimated as
\begin{align}
&\int_{0}^{1} 
   \| V_{A} V_{B} V_{C} f \|_{L^{2}(S_{u',\ubar'}) L^{2}(\mathcal{P}_{x})} 
   \| V_{A} V_{0} f \|_{L^{4}(S_{u',\ubar'}) L^{2}(\mathcal{P}_{x})} 
   \, \text{d}\ubar' \notag \\
\quad \lesssim& |u|^{\frac12} \int_{0}^{1} 
   \| V_{A} V_{B} V_{C} f \|_{L^{2}(S_{u',\ubar'}) L^{2}(\mathcal{P}_{x})} \notag \\
&\qquad \times \Big[ 
       \| V_{D} V_{A} V_{0} f \|_{L^{2}(S_{u',\ubar'}) L^{2}(\mathcal{P}_{x})} 
     + a \Gamma \| V_{4+D} V_{A} V_{0} f \|_{L^{2}(S_{u',\ubar'}) L^{2}(\mathcal{P}_{x})} \notag \\
&\qquad\quad 
     + \Gamma \| V_{4+D} V_{A} V_{0} f \|_{L^{2}(S_{u',\ubar'}) L^{2}(\mathcal{P}_{x})} 
     + \frac{a^{\frac12} \Gamma}{|u|^{2}} 
       \| V_{4} V_{A} V_{0} f \|_{L^{2}(S_{u',\ubar'}) L^{2}(\mathcal{P}_{x})} 
   \Big] \, d\ubar' \notag \\
&\qquad + \int_{0}^{1} 
     \| V_{A} V_{B} V_{C} f \|_{L^{2}(S_{u',\ubar'}) L^{2}(\mathcal{P}_{x})} 
     \, d\ubar' \notag \\
\quad \lesssim& |u|^{\frac12} \int_{0}^{1} \Big[ 
       \| V_{A} V_{B} V_{C} f \|^{2}_{L^{2}(S_{u',\ubar'}) L^{2}(\mathcal{P}_{x})} 
     + \frac{1}{|u|^{2}}\| V_{D} V_{A} |u|V_{0} f \|^{2}_{L^{2}(S_{u',\ubar'}) L^{2}(\mathcal{P}_{x})} \notag \\
&\qquad 
     + \frac{a \Gamma}{|u|^{2}} 
       \| |u| V_{4+D} V_{A} |u| V_{0} f \|^{2}_{L^{2}(S_{u',\ubar'}) L^{2}(\mathcal{P}_{x})} \notag \\
&\qquad 
     + \frac{\Gamma}{|u|^{2}} 
       \| |u| V_{4+D} V_{A} |u|V_{0} f \|^{2}_{L^{2}(S_{u',\ubar'}) L^{2}(\mathcal{P}_{x})} 
     + |u|^{-\frac12} 
       \| V_{A} V_{B} V_{C} f \|^{2}_{L^{2}(S_{u',\ubar'}) L^{2}(\mathcal{P}_{x})} 
   \Big] \, d\ubar' \notag \\
\quad \lesssim& |u|^{\frac12} \Big[ 
       a^{2}\mathcal{V} 
     + \frac{a^{3} \Gamma \mathcal{V}}{|u|^{2}} 
     + \frac{\Gamma a^{2}\mathcal{V}}{|u|^{2}} 
     + a^{2}|u|^{-\frac12} \mathcal{V} 
   \Big] 
   \lesssim a^{2}|u|^{\frac12}\mathcal{V}. 
\end{align}
Therefore, the ultimate estimate for \eqref{eq:ultimate1} reads 
\begin{align}
&\int_{u_{\infty}}^{u} \int_{0}^{1} \int_{S_{u^{\prime}, \ubar'}} 
 \bigg( \int_{\mathcal{P}_{x}} V_{A} V_{B} V_{C} f \hsp \nablasl \psi \hsp
 f_{1}(\psi)\hsp  f_{2}(\pslash)\hsp  f_{3}(p^{4}) \hsp
 V_{A} V_{0} f \hsp \text{d}\mathcal{P}_x \bigg)
 \sqrt{\det \gslash} \hsp \text{d}\theta^1 \text{d}\theta^2 \text{d}\ubar' \, \duprime \nonumber \\
\quad \lesssim& 
\int_{u_{\infty}}^{u} 
\sup_{\ubar} \| \nablasl \psi \|_{L^{4}(S_{u',\ubar})} \,
\sup_{\ubar} \| f_{1}(\psi) \|_{L^{\infty}(S_{u',\ubar})} \,
\sup_{\mathcal{P}_{x}} |f_{2}(\pslash)| \,
\sup_{\mathcal{P}_{x}} |f_{3}(p^{4})| \nonumber \\
&\qquad\qquad \times 
\int_{0}^{1} 
\| V_{A} V_{B} V_{C} f \|_{L^{2}(S_{u',\ubar'}) L^{2}(\mathcal{P}_{x})} \,
\| V_{A} V_{0} f \|_{L^{4}(S_{u',\ubar'}) L^{2}(\mathcal{P}_{x})} \,
\dubarprime\duprime  \nonumber \\
\quad \lesssim& 
\int_{u_{\infty}}^{u} 
\sup_{\ubar} \| \nablasl \psi \|_{L^{4}(S_{u',\ubar})} \,
\sup_{\ubar} \| f_{1}(\psi) \|_{L^{\infty}(S_{u',\ubar})} \,
\sup_{\mathcal{P}_{x}} |f_{2}(\pslash)| \,
\sup_{\mathcal{P}_{x}} |f_{3}(p^{4})| \, 
|u'|^{\frac{1}{2}} a^{2} \mathcal{V}
 \, \duprime \nonumber \\
\quad =& 
a^{2}\mathcal{V}
\int_{u_{\infty}}^{u} 
\sup_{\ubar} \| \nablasl \psi \|_{L^{4}(S_{u',\ubar})} \,
\sup_{\ubar} \| f_{1}(\psi) \|_{L^{\infty}(S_{u',\ubar})} \,
\sup_{\mathcal{P}_{x}} |f_{2}(\pslash)| \,
\sup_{\mathcal{P}_{x}} |f_{3}(p^{4})| \,
|u'|^{\frac{1}{2}} \, \duprime.
\end{align}
Now we systematically control these terms, by recalling the commutation formula derived in the previous Section.  Using the explicit expressions for the error terms $V_{I}(C^{J}_{K})$ as in \ref{eq:V1}-\ref{eq:V4}, we estimate the following:
\begin{align}
a^{2}\mathcal{V}\int_{u_{\infty}}^{u}\sup_{\ubar}||\nablasl\psi||_{L^{4}(S_{u^{\prime},\ubar})}\sup_{\ubar}||f_{1}(\psi)||_{L^{\infty}(S_{u^{\prime},\ubar})}\sup_{\mathcal{P}_{x}}|f_{2}(\pslash)|\sup_{\mathcal{P}_{x}}|f_{3}(p^{4})| |u^{\prime}|^{\frac{1}{2}}\text{d}u^{\prime}.
\end{align}
First, consider the worst term. For $\psi=\alpha,~f_{1}=1,f_{2}=\pslash, f_{3}=1$, we have 
\begin{eqnarray}
a^{2}\mathcal{V}\int_{u_{\infty}}^{u}\sup_{\ubar}||\nablasl\alpha||_{L^{4}(S_{u^{\prime},\ubar})}\sup_{\mathcal{P}_{x}}|\pslash| |u^{\prime}|^{\frac{1}{2}}\text{d}u^{\prime}     
\end{eqnarray}
To estimate $||\snabla\alpha||_{L^{4}}$, we first use the scale-invariant codimension-1 trace inequality \eqref{codimension12}: 
\begin{align}
||\nablasl\alpha||_{L^{4}_{sc}(S_{u,\ubar})}\lesssim ||\nablasl\alpha||_{L^{4}_{sc}(S_{u,0})}+||\nabla_{4}\nablasl\alpha||^{\frac{1}{2}}_{L^{2}_{sc}(H)}\Bigg(||\nablasl\alpha||^{\frac{1}{2}}_{L^{4}_{sc}(H)}+||a^{\frac{1}{2}}\nablasl^{2}\alpha||^{\frac{1}{2}}_{L^{2}_{sc}(H)} \Bigg).    
\end{align}
Note that the right-hand side is under control in terms of $\Gamma$ by bootstrap. Therefore, noting the scaling 
\begin{align}
 ||\nablasl\alpha||_{L^{4}(S_{u,\ubar})}\lesssim \frac{a^{\frac{1}{2}}}{|u|^{\frac{3}{2}}}\Gamma,   
\end{align}
hence 
\begin{align} \notag 
&a^{2}\mathcal{V}\int_{u_{\infty}}^{u}\sup_{\ubar}||\nablasl\alpha||_{L^{4}(S_{u^{\prime},\ubar})}\hsp \sup_{\mathcal{P}_{x}}|\pslash| \hsp |u^{\prime}|^{\frac{1}{2}}\text{d}u^{\prime}\lesssim a^{\frac{5}{2}}\mathcal{V}\Gamma \int_{u_{\infty}}^{u}|u^{\prime}|^{-\frac{3}{2}}|u^{\prime}|^{-1}|u^{\prime}|^{\frac{1}{2}}\text{d}u^{\prime} \\
\lesssim& \frac{a^{\frac{5}{2}}\mathcal{V}\Gamma}{|u|}\lesssim a^2,
\end{align}
Now consider the curvature $\betabar$. We observe that we need to estimate the following
\begin{eqnarray}
a^{2}\mathcal{V}\int_{u_{\infty}}^{u}\sup_{\ubar}||\nablasl\betabar||_{L^{4}(S_{u^{\prime},\ubar})} |u^{\prime}|^{\frac{1}{2}}\text{d}u^{\prime}     
\end{eqnarray}
We control  $||\nablasl\betabar||_{L^{4}(S_{u^{\prime},\ubar})}$ by the codimension-1 trace inequality \ref{codimension12}
\begin{eqnarray}
||\nablasl\betabar||_{L^{4}_{sc}(S_{u,\ubar})}\lesssim ||\nablasl\betabar||_{L^{4}_{sc}(S_{u,0})}+||\snabla_{4}\nablasl\betabar||^{\frac{1}{2}}_{L^{2}_{sc}(H)}\Bigg(||\nablasl\betabar||^{\frac{1}{2}}_{L^{4}_{sc}(H)}+||a^{\frac{1}{2}}\nablasl^{2}\betabar||^{\frac{1}{2}}_{L^{2}_{sc}(H)} \Bigg)    
\end{eqnarray}
and by the bootstrap assumptions \eqref{boundsbootstrap}, we have 
\begin{eqnarray}
||\nablasl\betabar||_{L^{4}_{sc}(S_{u,\ubar})}\lesssim a^{\frac{1}{4}}\Gamma \implies ||\nablasl\betabar||_{L^{4}(S_{u,\ubar})}\lesssim \frac{a^{2+\frac{1}{4}}}{|u|^{4+\frac{1}{2}}}\Gamma,   
\end{eqnarray}
since $s_{2}(\betabar)=\frac{3}{2}$, $s_{2}(\nablasl \betabar)=2$, and $||\nablasl\betabar||_{L^{4}(S_{u,\ubar})}=\frac{a^{2}}{|u|^{4+\frac{1}{2}}}||\nablasl\betabar||_{L^{4}_{sc}(S_{u,\ubar})}$. Therefore 
\begin{eqnarray}
a^{2}\mathcal{V}\int_{u_{\infty}}^{u}\sup_{\ubar}||\nablasl\betabar||_{L^{4}(S_{u^{\prime},\ubar})} |u^{\prime}|^{\frac{1}{2}}\text{d}u^{\prime} \lesssim a^{2}\mathcal{V}\hsp \Gamma\int_{u_{\infty}}^{u}\frac{a^{2+\frac{1}{4}}}{|u^{\prime}|^{4+\frac{1}{2}}}|u^{\prime}|^{\frac{1}{2}}\text{d}u^{\prime}
\lesssim \frac{a^{4+\frac{1}{4}}\mathcal{V}\Gamma}{|u|^{3}}.  
\end{eqnarray}
Similarly,
\begin{align}
a^{2}\mathcal{V}\int_{u_{\infty}}^{u}\sup_{\ubar}||\nablasl\beta||_{L^{4}(S_{u^{\prime},\ubar})}\sup_{\mathcal{P}_{x}}|\pslash\hsp \pslash| |u^{\prime}|^{\frac{1}{2}}\text{d}u^{\prime}\lesssim  a^{2}\mathcal{V}\hsp \Gamma\int_{u_{\infty}}^{u}\frac{a^{1+\frac{1}{4}}}{|u^{\prime}|^{2+\frac{1}{2}}}\frac{1}{|u^{\prime}|^{2}}|u^{\prime}|^{\frac{1}{2}}\text{d}u^{\prime}
\lesssim \frac{a^{3+\frac{1}{4}}\mathcal{V}\Gamma}{|u|^{3}}.
\end{align}
Finally, we need to control the following involving $\alphabar$ that has the highest decay
\begin{align}
a^{2}\mathcal{V}\int_{u_{\infty}}^{u}\sup_{\ubar}||\nablasl\alphabar||_{L^{4}(S_{u^{\prime},\ubar})}\sup_{\mathcal{P}_{x}}|\pslash|\hsp  |u^{\prime}|^{\frac{1}{2}}\text{d}u^{\prime}    
\end{align}
We invoke the following codimension-1 trace inequality \ref{codimension11} for $\alphabar$, 
\begin{eqnarray}
 ||\nablasl\alphabar||_{L^{4}_{sc}(S_{u,\ubar})}\lesssim ||\nablasl\alphabar||_{L^{4}_{sc}(S_{u_{\infty},\ubar})}+||a^{\frac{1}{2}}\nablasl_{3}\nablasl\alphabar||^{\frac{1}{2}}_{L^{2}_{sc}(\Hbar)}\left(a^{-\frac{1}{4}}||\nablasl\alphabar||^{\frac{1}{2}}_{L^{2}_{sc}(\Hbar)}+||\nablasl\nablasl\alphabar||^{\frac{1}{2}}_{L^{2}_{sc}(\Hbar)}\right).     
\end{eqnarray}
Now, from the bootstrap, we have control over the following $\alphabar$ norms 
\begin{align}
 \scaletwoHbaru{\mathcal{D}\alphabar}\lesssim \Gamma,\scaletwoHbaru{\mathcal{D}^2\alphabar}\lesssim \Gamma   
\end{align}
for $\mathcal{D}\in \{\snabla_{4},|u|\snabla_{3},a^{\frac{1}{2}}\nablasl\}$. Therefore, this translates to 
\begin{eqnarray}
 ||\nablasl\alphabar||_{L^{4}_{sc}(S_{u,\ubar})}\lesssim ||\nablasl\alphabar||_{L^{4}_{sc}(S_{u_{\infty}})}+\frac{\Gamma}{a}   
\end{eqnarray}
or in the ordinary norm 
\begin{eqnarray}
||\nablasl\alphabar||_{L^{4}(S_{u,\ubar})}\lesssim \frac{a^{\frac{5}{2}}}{|u|^{5+\frac{1}{2}}}\Gamma    
\end{eqnarray}
and therefore 
\begin{eqnarray}
 a^{2}\mathcal{V}\int_{u_{\infty}}^{u}\sup_{\ubar}||\nablasl\alphabar||_{L^{4}(S_{u^{\prime},\ubar})}\sup_{\mathcal{P}_{x}}|\pslash| \hsp  |u^{\prime}|^{\frac{1}{2}}\text{d}u^{\prime}\lesssim a^{2}\mathcal{V}\Gamma\int_{u_{\infty}}^{u}\frac{a^{\frac{5}{2}}}{|u|^{5+\frac{1}{2}}}|u^{\prime}|^{-1}|u^{\prime}|^{\frac{1}{2}}\text{d}u^{\prime}
 \lesssim \frac{a^{2+\frac{5}{2}}\mathcal{V}\Gamma}{|u|^{4}}. 
\end{eqnarray}
The derivatives of the Ricci coefficients are estimated as follows (one again, we do not repeat estimates that appear twice or more):
\begin{align}
&a^{2} \mathcal{V}\int_{u_{\infty}}^{u}\sup_{\ubar}||\nablasl\widetilde{\tr\chibar}||_{L^{4}(S_{u^{\prime},\ubar})}\sup_{\ubar}||\psi_{g}||_{L^{\infty}(S_{u^{\prime},\ubar})}|u^{\prime}|^{\frac{1}{2}}\text{d}u^{\prime}\\
 \lesssim& a^{2}\mathcal{V}\int_{u_{\infty}}^{u}\frac{\Gamma}{|u^{\prime}|^{2+\frac{1}{2}}}\frac{\Gamma}{|u^{\prime}|}|u^{\prime}|^{\frac{1}{2}}\lesssim \frac{a^{2}\mathcal{V}\hsp \Gamma^{2}}{|u|^{2}}.
\end{align}
This completes the estimates for the term 
\begin{align}
a^{2}\mathcal{V}\int_{u_{\infty}}^{u}\sup_{\ubar}||\nablasl\psi||_{L^{4}(S_{u^{\prime},\ubar})}\sup_{\ubar}||f_{1}(\psi)||_{L^{\infty}(S_{u^{\prime},\ubar})}\sup_{\mathcal{P}_{x}}|f_{2}(\pslash)|\sup_{\mathcal{P}_{x}}|f_{3}(p^{4})| |u^{\prime}|^{\frac{1}{2}}\text{d}u^{\prime} .
\end{align}
Now consider the next term 
\begin{align}
\int_{u_{\infty}}^{u} \int_{0}^{1}\int_{S_{u^{\prime},\ubar^{\prime}}}\Bigg(\int_{p}V_{A}V_{B}V_{C}f V_{B}(C^{D}_{C})V_{A}V_{D}f dP\Bigg)\sqrt{\det\gslash} \hsp\text{d}\theta^1 \text{d}\theta^2\text{d}\ubar^{\prime}   \duprime.
\end{align}
Similarly, one has the following schematic structure for a $\psi$ belonging to the set of connection of Weyl curvature components: 
\begin{align}
 V_{B}(C^{D}_{C})\sim \nablasl\psi \cdot f_{1}(\psi)\cdot f_{2}(\pslash)\cdot f_{3}(p^{4})+l.o.t.   
\end{align}
We estimate the integral in a similar fashion,
\begin{align} \notag
&\int_{u_{\infty}}^{u} \int_{0}^{1}\int_{S_{u^{\prime},u^{\prime}}}\Bigg(\int_{p}V_{A}V_{B}V_{C}f\hsp \nablasl \psi \hsp f_{1}(\psi) \hsp f_{2}(\pslash)\hsp f_{3}(p^{4}) \hsp V_{A}V_{D}f \text{d}\mathcal{P}_{x}\Bigg)\sqrt{\det\gslash} \hsp \text{d}\theta^1 \text{d}\theta^2 \dubarprime\duprime \\\nonumber 
 \lesssim& 
 \int_{u_{\infty}}^{u}\int_{0}^{1}\int_{S_{u,\ubar}}|\nablasl\psi||f_{1}(\psi)|\sup_{\mathcal{P}_{x}}|f_{2}(\pslash)|\sup_{\mathcal{P}_{x}}|f_{3}(p^{4})| ||V_{A}V_{B}V_{C}f||_{L^{2}(\mathcal{P}_{x})}||V_{A}V_{D}f||_{L^{2}(\mathcal{P}_{x})}\sqrt{\det\gslash}\hsp \text{d}\theta^1 \text{d}\theta^2\dubarprime\duprime  \\\nonumber 
 \lesssim& \int_{u_{\infty}}^{u}\int_{0}^{1}\sup_{\mathcal{P}_{x}}|f_{2}(\pslash)|\sup_{\mathcal{P}_{x}}|f_{3}(p^{4})||\nablasl\psi||_{L^{4}(S_{u^{\prime},\ubar^{\prime}})}||f_{1}(\psi)||_{L^{\infty}(S_{u^{\prime},\ubar^{\prime}})}\\ &\hspace{3mm}\times||V_{A}V_{B}V_{C}f||_{L^{2}(S_{u^{\prime},\ubar^{\prime}})L^{2}(\mathcal{P}_{x})}||V_{A}V_{D}f||_{L^{4}(S_{u^{\prime},\ubar^{\prime}})L^{2}(\mathcal{P}_{x})}\text{d}\ubar^{\prime} \text{d}u^{\prime} \\\nonumber 
 \lesssim& \int_{u_{\infty}}^{u}\sup_{\ubar}||\nablasl\psi||_{L^{4}(S_{u^{\prime},\ubar})}\sup_{\ubar}||f_{1}(\psi)||_{L^{\infty}(S_{u^{\prime},\ubar})}\sup_{\mathcal{P}_{x}}|f_{2}(\pslash)|\sup_{\mathcal{P}_{x}}|f_{3}(p^{4})|\\ &\hspace{3mm}\times\int_{0}^{1}||V_{A}V_{B}V_{C}f||_{L^{2}(S_{u^{\prime},\ubar^{\prime}})L^{2}(\mathcal{P}_{x})}||V_{A}V_{D}f||_{L^{4}(S_{u^{\prime},\ubar^{\prime}})L^{2}(\mathcal{P}_{x})}\text{d}\ubar^{\prime} \text{d}u^{\prime} .    
\end{align}
Now we claim that the integral 
\begin{align}
\label{eq:border}
\int_{0}^{1}||V_{A}V_{B}V_{C}f||_{L^{2}(S_{u^{\prime},\ubar^{\prime}})L^{2}(\mathcal{P}_{x})}||V_{A}V_{D}f||_{L^{4}(S_{u^{\prime},\ubar^{\prime}})L^{2}(\mathcal{P}_{x})}\text{d}\ubar^{\prime}    
\end{align}
can be controlled by means of the bootstrap assumptions on the top order norm of $f$ and the lower order norms estimated earlier. The philosophy is exactly the same as before. Note that $||V_{A}V_{D}f||_{L^{2}(\mathcal{P}_{x})}$ is a function on spacetime and particular on $S_{u,\ubar}$. Therefore, we can apply the Sobolev embedding result on the sphere. Recall the Sobolev embedding for a function $\mathcal{G}$ on $S_{u,\ubar}$: 
\begin{eqnarray}
||\mathcal{G}||_{L^{4}(S_{u,\ubar})}\lesssim |u|^{\frac{1}{2}}||\nablasl \mathcal{G}||_{L^{2}(S_{u,\ubar})}+|u|^{-\frac{1}{2}}||\mathcal{G}||_{L^{2}(S_{u,\ubar})}.    
\end{eqnarray}
Applying this to the function $\mathcal{G}=||V_{A}V_{D}f||_{L^{2}(\mathcal{P}_{x})}$ yields
\begin{align}&
||V_{A}V_{D}f||_{L^{4}(S_{u^{\prime},\ubar^{\prime}})L^{2}(\mathcal{P}_{x})}\notag \\\lesssim& |u|^{\frac{1}{2}} \left[\int_{S_{u,\ubar}}\left(\nablasl \left(\int_{\mathcal{P}_{x}}|V_{A}V_{D}f|^{2} d\mathcal{P}_{x}\right)^{\frac{1}{2}} \right)^{2} \right]^{\frac{1}{2}}\nonumber+|u|^{-\frac{1}{2}}\Bigg[\int_{S_{u,\ubar}}\left(\int_{\mathcal{P}_{x}}|V_{A}V_{D}f|^{2} d\mathcal{P}_{x}\right)\Bigg]^{\frac{1}{2}} \\ 
\lesssim& |u|^{\frac{1}{2}} \left[\int_{S_{u,\ubar}}\left(\nablasl \left(\int_{\mathcal{P}_{x}}|V_{A}V_{D}f|^{2} d\mathcal{P}_{x}\right)^{\frac{1}{2}} \right)^{2} \right]^{\frac{1}{2}}+a|u|^{-\frac{1}{2}}\sqrt{\mathcal{I}^{0}_{2}},
\end{align}
by using the second derivative estimates from the previous Proposition \ref{prop:second_derivative_estimate}. We further  compute \begin{align}
\nablasl\left(\int_{\mathcal{P}_{x}}|V_{A}V_{D}f|^{2} \right)^{\frac{1}{2}}=\frac{\int_{P}V_{A}V_{D}f \nablasl V_{A}V_{D}f}{\sqrt{\left(\int_{\mathcal{P}_{x}}|V_{A}V_{D}f|^{2} \right)^{\frac{1}{2}}}}+\Gammaslash|||V_{A}V_{D}f||_{L^{2}(\mathcal{P}_{x})}\lesssim ||\nablasl V_{A}V_{D}f||_{L^{2}(\mathcal{P}_{x})}+\Gammaslash \hsp |||V_{A}V_{D}f||_{L^{2}(\mathcal{P}_{x})}. 
\end{align}
Recalling the following relation between the derivatives, 
\begin{align}  e_A =& V_{(A)}  + \modu^2 \left( \frac{\Gammaslash^C_{AB}\hsp p^B}{p^3} + {\chibarhat_A}^C +\frac{{\chi_A}^C \hsp p^4}{p^3}\right) V_{(4+C)}  \notag \\ &+\frac{\modu^2}{2} (\tr\chibar+\frac{2}{\modu}) V_{(4+A)} +\left( \frac{\frac{1}{2}\chi_{AB}\hsp p^B}{p^3} -\etabar_A\right)(V_{(4)}+V_{(0)}), \label{expressionVA-1} \end{align}    
we establish the schematic expression 
\begin{align}
&||\nablasl V_{A}V_{E}f||_{L^{2}(\mathcal{P}_{x})}\sim ||V_{D}V_{A}V_{E}f+|u|^{2}(\Gammaslash\pslash+\chibarhat+\chi p^{4})V_{4+D}V_{A}V_{E}f+|u|^{2}\widetilde{\tr\chibar}V_{4+D}V_{A}V_{E}f \notag\\ &\hspace{72mm}+(\chi\pslash+\etabar)(V_{4}V_{A}V_{E}f+V_{0}V_{A}V_{E}f)||_{L^{2}(\mathcal{P}_{x})}.     
\end{align}
Putting together all these basic computations, we obtain 
\begin{align}
&|u|^{\frac{1}{2}} \left[\int_{S_{u,\ubar}}\left(\nablasl \left(\int_{\mathcal{P}_{x}}|V_{A}V_{E}f|^{2} d\mathcal{P}_{x}\right)^{\frac{1}{2}} \right)^{2} \hsp \sqrt{\det\gslash}\hsp\text{d}\theta^1\text{d}\theta^2 \right]^{\frac{1}{2}}\\ \notag
\lesssim& |u|^{\frac{1}{2}}\Bigg[||V_{D}V_{A}V_{E}f||_{L^{2}(S_{u,\ubar})L^{2}(\mathcal{P}_{x})}+a\Gamma||V_{4+D}V_{A}V_{E}f||_{L^{2}(S_{u,\ubar})L^{2}(\mathcal{P}_{x})}\\+&\Gamma||V_{4+D}V_{A}V_{E}f||_{L^{2}(S_{u,\ubar})L^{2}(\mathcal{P}_{x})}+\frac{a^{\frac{1}{2}}\Gamma}{|u|^{2}}||V_{4}V_{A}V_{E}f||_{L^{2}(S_{u,\ubar})L^{2}(\mathcal{P}_{x})}\Bigg].
\end{align}
Notice that this controls the top order terms in the Vlasov estimates. The vital point to note here that without uniform estimate on the first derivative $Vf$ in $L^{2}(\mathcal{P}_{x})$, this strategy fails. This, in turn, necessitates $L^{\infty}$ control of the Weyl curvature. Therefore, at least two derivatives of the curvature on the null cone $H$ (and/or $\Hbar$) need to be controlled. Finally, the integral in \eqref{eq:border} is estimated as
\begin{align} \notag
&\int_{0}^{1}||V_{A}V_{B}V_{C}f||_{L^{2}(S_{u^{\prime},\ubar^{\prime}})L^{2}(\mathcal{P}_{x})}||V_{A}V_{E}f||_{L^{4}(S_{u^{\prime},\ubar^{\prime}})L^{2}(\mathcal{P}_{x})}\text{d}\ubar^{\prime} \\
\notag \lesssim& |u|^{\frac{1}{2}}\int_{0}^{1}\Bigg[||V_{A}V_{B}V_{C}f||^{2}_{L^{2}(S_{u^{\prime},\ubar^{\prime}})L^{2}(\mathcal{P}_{x})}+||V_{D}V_{A}V_{E}f||^{2}_{L^{2}(S_{u,\ubar})L^{2}(\mathcal{P}_{x})}+\frac{a^{2}\Gamma^{2}}{|u|^{2}}||u V_{4+D}V_{A}V_{E}f||^{2}_{L^{2}(S_{u,\ubar})L^{2}(\mathcal{P}_{x})}\\ \notag+&\frac{\Gamma^{2}}{|u|^{2}}||u V_{4+D}V_{A}V_{E}f||^{2}_{L^{2}(S_{u,\ubar})L^{2}(\mathcal{P}_{x})}+\frac{a\Gamma}{|u|^{4}}||V_{4}V_{A}V_{E}f||^{2}_{L^{2}(S_{u,\ubar})L^{2}(\mathcal{P}_{x})}+|u|^{-\frac{1}{2}}||V_{A}V_{B}V_{C}f||^{2}_{L^{2}(S_{u^{\prime},\ubar^{\prime}})L^{2}(\mathcal{P}_{x})}\Bigg]\text{d}\ubar^{\prime} \notag \\ 
\lesssim& |u|^{\frac{1}{2}}\Bigg[\mathcal{V}+\frac{a^{2}\Gamma^{2}\mathcal{V}}{|u|^{2}}+\frac{\Gamma^{2}\mathcal{V}}{|u|^{2}}+|u|^{-\frac{1}{2}}\mathcal{V}\Bigg]\lesssim |u|^{\frac{1}{2}}\Gamma^{2}\mathcal{V}.
\end{align}
Recalling back the estimates that we wanted to obtain, 
\begin{align} \notag 
&\int_{u_{\infty}}^{u} \int_{0}^{1}\int_{S_{u,u^{\prime}}}\Bigg(\int_{p}V_{A}V_{B}V_{C}f\nablasl \psi f_{1}(\psi) f_{2}(\pslash)f_{3}(p^{4}) V_{A}V_{E}f dP\Bigg)\sqrt{\det\gslash}\hsp \text{d}\theta^1 \text{d}\theta^2 \dubarprime \duprime \\ \notag 
\lesssim& \int_{u_{\infty}}^{u}\sup_{\ubar}||\nablasl\psi||_{L^{4}(S_{u^{\prime},\ubar})}\sup_{\ubar}||f_{1}(\psi)||_{L^{\infty}(S_{u^{\prime},\ubar})}\sup_{\mathcal{P}_{x}}|f_{2}(\pslash)|\sup_{\mathcal{P}_{x}}|f_{3}(p^{4})|\\ \notag &\hspace{10mm}\times \int_{0}^{1}||V_{A}V_{B}V_{C}f||_{L^{2}(S_{u^{\prime},\ubar^{\prime}})L^{2}(\mathcal{P}_{x})}||V_{A}V_{E}f||_{L^{4}(S_{u^{\prime},\ubar^{\prime}})L^{2}(\mathcal{P}_{x})}\text{d}\ubar^{\prime} \text{d}u^{\prime} \\\nonumber 
\lesssim& \int_{u_{\infty}}^{u}\sup_{\ubar}||\nablasl\psi||_{L^{4}(S_{u^{\prime},\ubar})}\sup_{\ubar}||f_{1}(\psi)||_{L^{\infty}(S_{u^{\prime},\ubar})}\sup_{\mathcal{P}_{x}}|f_{2}(\pslash)|\sup_{\mathcal{P}_{x}}|f_{3}(p^{4})| |u^{\prime}|^{\frac{1}{2}}\Gamma^{2}\mathcal{V}\text{d}u^{\prime} \\
=&\Gamma^{2}\mathcal{V}\int_{u_{\infty}}^{u}\sup_{\ubar}||\nablasl\psi||_{L^{4}(S_{u^{\prime},\ubar})}\sup_{\ubar}||f_{1}(\psi)||_{L^{\infty}(S_{u^{\prime},\ubar})}\sup_{\mathcal{P}_{x}}|f_{2}(\pslash)|\sup_{\mathcal{P}_{x}}|f_{3}(p^{4})| |u^{\prime}|^{\frac{1}{2}}\text{d}u^{\prime} .
\end{align} 
First, note that several estimates repeat, and therefore we do not write the estimates for every term. Collecting all terms together yields 
\begin{align}
&\Gamma^{2}\mathcal{V}\int_{u_{\infty}}^{u}\sup_{\ubar}||\nablasl (\chibarhat,\widetilde{\tr\chibar})||_{L^{4}(S_{u^{\prime},\ubar})}\sup_{\ubar}|| |u^{\prime}|^{\frac{1}{2}}\text{d}u^{\prime} \lesssim\Gamma^{2}\mathcal{V}\int_{u_{\infty}}^{u}\Bigg(\frac{a^{\frac{1}{2}}\Gamma}{|u^{\prime}|^{2+\frac{1}{2}}},\frac{\Gamma}{|u^{\prime}|^{2+\frac{1}{2}}}\Bigg)|u^{\prime}|^{\frac{1}{2}}\text{d}u^{\prime} \notag \\ 
\lesssim& \Gamma^{2}\mathcal{V}\Bigg(\frac{a^{\frac{1}{2}}\Gamma}{|u|},\frac{\Gamma}{|u|}\Bigg).
\end{align}
We now move on to estimating the terms containing two derivatives of Weyl curvature. This ensemble of terms is depicted in the following term in the transport estimates: 
\begin{align}
\sum_{L=0}^{6}\int_{u_{\infty}}^{u}\int_{0}^{1}\int_{S_{u^{\prime},\ubar^{\prime}}}\Bigg[\int_{\mathcal{P}_{x}}V_{A}V_{B}V_{C}fV_{A}V_{B}(C^{L}_{C})V_{L}f \text{d}\mathcal{P}_x\Bigg]\sqrt{\det\gslash}\hsp \text{d}\theta^1 \text{d}\theta^2 \dubarprime\duprime.   
\end{align}
We compute $V_{A}V_{B}(C^{L}_{C})$ for all $L=0,....,6$,
\begin{align}\notag
 &V_{A}V_{B}(C^{0}_{C})\sim \bigg[\nablasl^{2}\psi_{g} \hsp \psi_{g} \hsp p^{4}+\nablasl^{2}\psi_{g} \hsp \psi_{g} \hsp \pslash \hsp \pslash+\nablasl^{2}\psi_{g} \hsp \chibarhat  \hsp \pslash+\nablasl^{2}\psi_{g} \hsp \tr\chibar  \hsp\hsp  \pslash+\nablasl^{2}\tr\chibar  \hsp \psi_{g} \hsp \pslash+\nablasl^{2}\psi_{g} \hsp \pslash  \hsp p^{4}\\ \notag +&\nablasl^{2}\psi_{g} \hsp \pslash \hsp \pslash \hsp \pslash+\nablasl^{2}\psi_{g} \hsp p^{4} \hsp p^{4}+\nablasl^{2}\chibarhat \hsp  p^{4}+\nablasl^{2}\psi_{g} \hsp p^{4}+\nablasl^{2}\alpha \hsp \pslash+\nablasl^{2}\beta \hsp \pslash \hsp \pslash+\nablasl^{2}\alpha \hsp \pslash \hsp p^{4}+\nablasl^{2}\alphabar \hsp  \pslash\\
 +&|\nablasl\psi_{g}|^{2} \hsp (p^{4}+\pslash \hsp \pslash)+\nablasl\psi_{g} \hsp \nablasl\chibarhat  \hsp \pslash+\nablasl\psi_{g} \hsp \nablasl\tr\chibar \hsp \pslash+|\nablasl\psi_{g}|^{2} \hsp p^{4} \hsp p^{4}+\mathcal{E}_{l.o.t}\bigg]_{ABC},
\end{align}
where the $\mathcal{E}_{l.o.t}$ contains terms that are linear in the first derivative. The top derivative terms of $V_{A}V_{B}(C^{0}_{C})$ have the following schematic structure 
\begin{align}
\nablasl^{2}\psi \cdot f_{1}(\psi)\cdot f_{2}(\pslash)\cdot f_{3}(p^{4}), 
\end{align}
for $\psi$ denoting a Ricci coefficient or a Weyl curvature component. If $\psi$ is a Ricci coefficient that belongs to the set $\{\chihat,\tr\chi,\chibarhat,\tr\chibar,\eta,\etabar,\omega=0,\omegabar\}$, then we estimate $\nablasl^{2}\psi$ in $L^{2}(S_{u,\ubar})$,  
\begin{align}
 \notag &\int_{u_{\infty}}^{u}\int_{0}^{1}\int_{S_{u,\ubar}}\Bigg[ \int_{\mathcal{P}_{x}}V_{A}V_{B}V_{C}f \nablasl^{2}\psi f_{1}(\psi)f_{2}(\pslash)f_{3}(p^{4}) V_{0}f \hsp \text{d}\mathcal{P}_{x}\Bigg]\sqrt{\det\gslash}\hsp \text{d}\theta^1 \text{d}\theta^2 \dubarprime \duprime \\
 \lesssim& \int_{u_{\infty}}^{u}\sup_{\ubar}||\nablasl^{2}\psi||_{L^{2}(S_{u^{\prime},\ubar})}\sup_{\ubar}||f_{1}(\psi)||_{L^{\infty}(S_{u^{\prime},\ubar})}\sup_{\mathcal{P}_{x}}|f_{2}(\pslash)|\hsp \sup_{\mathcal{P}_{x}}|f_{3}(p^{4})|\notag \\&\hspace{3cm}\times\Bigg[\int_{0}^{1}\int_{S_{u^{\prime},\ubar^{\prime}}}\int_{\mathcal{P}_{x}}|V_{A}V_{B}V_{C}f|^{2} \text{d}\mathcal{P}_x \text{d}S_{u^{\prime},\ubar^{\prime}}\dubarprime\Bigg]^{\frac{1}{2}} \sup_{\ubar}||V_{0}f||_{L^{\infty}(S_{u^{\prime},\ubar})L^{2}(\mathcal{P}_{x})} \duprime .
\end{align}
Now note that $\sup_{\ubar}|||u|V_{0}f||_{L^{\infty}(S_{u^{\prime},\ubar})L^{2}(\mathcal{P}_{x})}$ is estimated in the previous step and we use bootstrap for $\int_{0}^{1}\int_{S_{u^{\prime},\ubar}}\int_{\mathcal{P}_{x}}|V_{A}V_{B}V_{C}f|^{2}$. Therefore, the final estimates become
\begin{align}
\notag \int_{u_{\infty}}^{u}\int_{0}^{1}\int_{S_{u,\ubar}}\Bigg[ \int_{\mathcal{P}_{x}}V_{A}V_{B}V_{C}f \nablasl^{2}\psi f_{1}(\psi)f_{2}(\pslash)f_{3}(p^{4}) V_{0}f \text{d}\mathcal{P}_x\Bigg]\sqrt{\det\gslash}\hsp \text{d}\theta^1 \text{d}\theta^2 \text{d}\ubar^{\prime} \text{d}u^{\prime} \\
 \lesssim a^{2}\mathcal{V}  \int_{u_{\infty}}^{u}\sup_{\ubar}||\nablasl^{2}\psi||_{L^{2}(S_{u^{\prime},\ubar})}\sup_{\ubar}||f_{1}(\psi)||_{L^{\infty}(S_{u^{\prime},\ubar})}\sup_{\mathcal{P}_{x}}|f_{2}(\pslash)|\sup_{\mathcal{P}_{x}}|f_{3}(p^{4})||u^{\prime}|^{-1} \duprime .
\end{align}
We now collect (without duplication) the terms that contain second derivatives of the Ricci coefficients. We have 
\begin{align}
a^{2}\mathcal{V}  \int_{u_{\infty}}^{u}\sup_{\ubar}||\nablasl^{2}\psi_{g}||_{L^{2}(S_{u^{\prime},\ubar})}\sup_{\ubar}||\psi_{g}||_{L^{\infty}(S_{u^{\prime},\ubar})}\sup_{\mathcal{P}_{x}}|p^{4}||u^{\prime}|^{-1}\\\nonumber 
\lesssim a^{2}\mathcal{V}\int_{u_{\infty}}^{u}\frac{a\Gamma}{|u^{\prime}|^{2}}\frac{\Gamma}{|u^{\prime}|}|u^{\prime}|^{-2}|u^{\prime}|^{-1}\text{d}u^{\prime} \lesssim \frac{a^{3}\Gamma^{2}\mathcal{V}}{|u|^{5}},\end{align}
\begin{align}
a^{2}\mathcal{V} \int_{u_{\infty}}^{u}\sup_{\ubar}||\nablasl^{2}\psi_{g}||_{L^{2}(S_{u^{\prime},\ubar})}\sup_{\ubar}||\psi_{g}||_{L^{\infty}(S_{u^{\prime},\ubar})}\sup_{\mathcal{P}_{x}}|\pslash\pslash||u^{\prime}|^{-1}
\lesssim \frac{a^{3}\Gamma^{2}\mathcal{V}}{|u|^{4}},\end{align}
\begin{align}
&a^{2}\mathcal{V}  \int_{u_{\infty}}^{u}\sup_{\ubar}||\nablasl^{2}\psi_{g}||_{L^{2}(S_{u^{\prime},\ubar})}\sup_{\ubar}||\chibarhat||_{L^{\infty}(S_{u^{\prime},\ubar})}\sup_{\mathcal{P}_{x}}|\pslash| |u^{\prime}|^{-1}\text{d}u^{\prime} \notag \\  \lesssim& 
 a^{2}\mathcal{V}\int_{u_{\infty}}^{u}\frac{a\Gamma}{|u^{\prime}|^{2}}\frac{a^{\frac{1}{2}}\Gamma}{|u^{\prime}|^{2}}|u^{\prime}|^{-1}|u^{\prime}|^{-1} \text{d}u^{\prime} \lesssim \frac{a^{3}\Gamma^{2}\mathcal{V}}{|u|^{5}},\end{align}
\begin{align}
a^{2}\mathcal{V}  \int_{u_{\infty}}^{u}\sup_{\ubar}||\nablasl^{2}\psi_{g}||_{L^{2}(S_{u^{\prime},\ubar})}\sup_{\ubar}||\tr\chibar||_{L^{\infty}(S_{u^{\prime},\ubar})} \sup_{\mathcal{P}_{x}}|\pslash| |u^{\prime}|^{-1}\text{d}u^{\prime}  \lesssim   a^{2}\mathcal{V}\int_{u_{\infty}}^{u}\frac{a\Gamma}{|u^{\prime}|^{2}}\frac{\Gamma}{|u^{\prime}|}|u^{\prime}|^{-2} \text{d}u^{\prime} \lesssim \frac{a^{3}\Gamma^{2}\mathcal{V}}{|u|^{4}},\\
a^{2}\mathcal{V}  \int_{u_{\infty}}^{u}\sup_{\ubar}||\nablasl^{2}\widetilde{\tr\chibar}||_{L^{2}(S_{u^{\prime},\ubar})}\sup_{\ubar}||\psi_{g}||_{L^{\infty}(S_{u^{\prime},\ubar})}\sup_{\mathcal{P}_{x}}|\pslash||u^{\prime}|^{-1}\text{d}u^{\prime}  \lesssim  a^{2}\mathcal{V}\int_{u_{\infty}}^{u}\frac{\Gamma}{|u^{\prime}|^{3}}\frac{\Gamma}{|u^{\prime}|}|u^{\prime}|^{-2} \text{d}u^{\prime} \lesssim \frac{a^{2}\Gamma^{2}\mathcal{V}}{|u|^{5}}, \\ 
a^{2}\mathcal{V}  \int_{u_{\infty}}^{u}\sup_{\ubar}||\nablasl^{2}\widetilde{\tr\chibar}||_{L^{2}(S_{u^{\prime},\ubar})}\sup_{\ubar}||\psi_{g}||_{L^{\infty}(S_{u^{\prime},\ubar})}\sup_{\mathcal{P}_{x}}|\pslash||u^{\prime}|^{-1}\text{d}u^{\prime}  \lesssim  a^{2}\mathcal{V}\int_{u_{\infty}}^{u}\frac{\Gamma}{|u^{\prime}|^{3}}\frac{\Gamma}{|u^{\prime}|}|u^{\prime}|^{-1} \text{d}u^{\prime} \lesssim \frac{a^{2}\Gamma^{2}\mathcal{V}}{|u|^{5}},\end{align}
\begin{align}
 a^{2}\mathcal{V}  \int_{u_{\infty}}^{u}\sup_{\ubar}||\nablasl^{2}\tr\chibar||_{L^{2}(S_{u^{\prime},\ubar})}\sup_{\mathcal{P}_{x}}|p^{4}||u^{\prime}|^{-1}\text{d}u^{\prime}  \lesssim   a^{2}\mathcal{V}\int_{u_{\infty}}^{u}\frac{a^{\frac{1}{2}}\Gamma}{|u^{\prime}|^{3}}|u^{\prime}|^{-2} \text{d}u^{\prime} \lesssim \frac{a^{\frac{5}{2}}\Gamma^{2}\mathcal{V}}{|u|^{4}}.    
\end{align}
We now move on to estimating terms involving the Weyl curvature $\psi \neq \alphabar$: 
\begin{align}
 &\int_{u_{\infty}}^{u}\int_{0}^{1}\int_{S_{u,\ubar}}\Bigg[ \int_{\mathcal{P}_{x}}V_{A}V_{B}V_{C}f \nablasl^{2}\psi f_{1}(\psi)f_{2}(\pslash)f_{3}(p^{4}) V_{0}fdP\Bigg]\sqrt{\det\gslash}\hsp \text{d}\theta^1 \text{d}\theta^2 \dubarprime \duprime \notag \\ \notag
 \lesssim& \int_{u_{\infty}}^{u}\sup_{\ubar}||f_{1}(\psi)||_{L^{\infty}(S_{u^{\prime},\ubar})}\sup_{\mathcal{P}_{x}}|f_{2}(\pslash)|\sup_{\mathcal{P}_{x}}|f_{3}(p^{4})|\\ \notag 
&\hspace{1cm} \times \int_{0}^{1}||\nablasl^{2}\psi||_{L^{2}(S_{u^{\prime},\ubar^{\prime}})}||V_{A}V_{B}V_{C}f||_{L^{2}(S_{u^{\prime},\ubar^{\prime}})L^{2}(\mathcal{P}_{x})}||V_{0}f||_{L^{\infty}(S_{u^{\prime},\ubar^{\prime}})L^{2}(\mathcal{P}_{x})}\text{d}u^{\prime} \text{d}\ubar^{\prime} \\ \notag 
 \lesssim& \int_{u_{\infty}}^{u}\sup_{\ubar}||\nablasl^{2}\psi||_{L^{2}(H_{u^{\prime}})}\sup_{\ubar}||f_{1}(\psi)||_{L^{\infty}(S_{u^{\prime},\ubar})}\sup_{\mathcal{P}_{x}}|f_{2}(\pslash)|\sup_{\mathcal{P}_{x}}|f_{3}(p^{4})|\\
 \notag &\hspace{1cm}\times\Bigg[\int_{0}^{1}\int_{S_{u^{\prime},\ubar}}\int_{\mathcal{P}_{x}}|V_{A}V_{B}V_{C}f|^{2}\Bigg]^{\frac{1}{2}}
 \sup_{\ubar}||V_{0}f||_{L^{\infty}(S_{u^{\prime},\ubar})L^{2}(\mathcal{P}_{x})}\duprime \\ 
 \lesssim &a^{2}\mathcal{V}\int_{u_{\infty}}^{u}\sup_{\ubar}||\nablasl^{2}\psi||_{L^{2}(H_{u^{\prime}})}\sup_{\ubar}||f_{1}(\psi)||_{L^{\infty}(S_{u^{\prime},\ubar})}\sup_{\mathcal{P}_{x}}|f_{2}(\pslash)|\sup_{\mathcal{P}_{x}}|f_{3}(p^{4})||u^{\prime}|^{-1}\text{d}u^{\prime} .
\end{align}
Now estimate 
\begin{align}
\notag &\int_{u_{\infty}}^{u}\int_{0}^{1}\int_{S_{u,\ubar}}\Bigg[ \int_{\mathcal{P}_{x}}V_{A}V_{B}V_{C}f \nablasl^{2}\alpha \pslash V_{0}fdP\Bigg]\sqrt{\det\gslash}\text{d}\ubar^{\prime} \text{d}u^{\prime} \\ \notag
\lesssim& a^{2}\mathcal{V}\int_{u_{\infty}}^{u}\sup_{\ubar}||\nablasl^{2}\alpha||_{L^{2}(H_{u^{\prime}})}\sup_{\ubar}\sup_{\mathcal{P}_{x}}|\pslash||u^{\prime}|^{-1}\text{d}u^{\prime} \\
\lesssim& a^{2}\mathcal{V}\int_{u_{\infty}}^{u}\frac{a^{\frac{1}{2}}\Gamma}{|u^{\prime}|^{2}}|u^{\prime}|^{-2}\text{d}u^{\prime} \lesssim \frac{a^{\frac{5}{2}}\Gamma\mathcal{V}}{|u|^{3}},
\end{align}
\begin{align} \notag 
&\int_{u_{\infty}}^{u}\int_{0}^{1}\int_{S_{u,\ubar}}\Bigg[ \int_{\mathcal{P}_{x}}V_{A}V_{B}V_{C}f \hsp \nablasl^{2}\beta\hsp  \pslash  \hsp \pslash \hsp V_{0}f \hsp \text{d}\mathcal{P}_{x}\Bigg]\sqrt{\det\gslash}\text{d}\ubar^{\prime} \text{d}u^{\prime} \\ \notag
\lesssim& a^{2}\mathcal{V}\int_{u_{\infty}}^{u}\sup_{\ubar}||\nablasl^{2}\beta||_{L^{2}(H_{u^{\prime}})}\sup_{\ubar}\sup_{\mathcal{P}_{x}}|\pslash\hsp \pslash||u^{\prime}|^{-1}\duprime \\
\lesssim& a^{2}\mathcal{V}\int_{u_{\infty}}^{u}\frac{a^{\frac{1}{2}}\Gamma}{|u^{\prime}|^{3}}|u^{\prime}|^{-3}\text{d}u^{\prime} \lesssim \frac{a^{\frac{5}{2}}\Gamma\mathcal{V}}{|u|^{5}}.  
\end{align}
Now, suppose the curvature component $\psi$ is $\alphabar$. The advantage of the term containing $\alphabar$ is that it has the best decay among the weyl curvature components. Then by bootstrap and the first order estimate from Proposition \ref{firstvlasov}, 
\begin{align}
\notag &\int_{u_{\infty}}^{u}\int_{0}^{1}\int_{S_{u,\ubar}}\Bigg[ \int_{\mathcal{P}_{x}}V_{A}V_{B}V_{C}f \hsp \nablasl^{2}\alphabar \hsp f_{1}(\psi)\hsp f_{2}(\pslash)\hsp f_{3}(p^{4}) \hsp V_{0}f \hsp \text{d}\mathcal{P}_{x}\Bigg]\sqrt{\det\gslash}\hsp \text{d}\theta^1 \text{d}\theta^2 \dubarprime \duprime \\ \notag
 \lesssim& \int_{u_{\infty}}^{u}\sup_{\ubar}||f_{1}(\psi)||_{L^{\infty}(S_{u^{\prime},\ubar})}\sup_{\mathcal{P}_{x}}|f_{2}(\pslash)|\sup_{\mathcal{P}_{x}}|f_{3}(p^{4})||u^{\prime}|^{-1}\\ \notag &\hspace{1cm} \int_{0}^{1}||\nablasl^{2}\alphabar||_{L^{2}(S_{u^{\prime},\ubar^{\prime}})}||V_{A}V_{B}V_{C}f||_{L^{2}(S_{u^{\prime},\ubar^{\prime}})L^{2}(\mathcal{P}_{x})}|||u|V_{0}f||_{L^{\infty}(S_{u^{\prime},\ubar^{\prime}})L^{2}(\mathcal{P}_{x})}|u^{\prime}|^{-1}\text{d}u^{\prime} \text{d}\ubar^{\prime} \\ \notag
 \lesssim& a\sqrt{\mathcal{V}}\int_{u_{\infty}}^{u}\sup_{\ubar}||f_{1}(\psi)||_{L^{\infty}(S_{u^{\prime},\ubar})}\sup_{\mathcal{P}_{x}}|f_{2}(\pslash)|\sup_{\mathcal{P}_{x}}|f_{3}(p^{4})|\Bigg(\int_{0}^{1}||\nablasl^{2}\alphabar||^{2}_{L^{2}(S_{u^{\prime},\ubar^{\prime}})}\text{d}\ubar^{\prime} \Bigg)^{\frac{1}{2}}\\\notag &\hspace{2cm} \Bigg(\int_{0}^{1}||V_{A}V_{B}V_{C}f||^{2}_{L^{2}(S_{u^{\prime},\ubar^{\prime}})L^{2}(\mathcal{P}_{x})}\text{d}\ubar^{\prime} \Bigg)^{\frac{1}{2}} |u^{\prime}|^{-1}\duprime \\\nonumber 
 \lesssim& a\sqrt{\mathcal{V}}\sup_{u}\Bigg(\int_{0}^{1}||V_{A}V_{B}V_{C}f||^{2}_{L^{2}(S_{u^{\prime},\ubar^{\prime}})L^{2}(\mathcal{P}_{x})}\text{d}\ubar^{\prime} \Bigg)^{\frac{1}{2}}\Bigg(\int_{u_{\infty}}^{u}|u^{\prime}|^{10}\int_{0}^{1}||\nablasl^{2}\alphabar||^{2}_{L^{2}(S_{u^{\prime},\ubar^{\prime}})}\text{d}\ubar^{\prime} \text{d}u^{\prime} \Bigg)^{\frac{1}{2}}\\ \notag
 &\hspace{2cm}\Bigg(\int_{u_{\infty}}^{u}\frac{\sup_{\ubar}||f_{1}(\psi)||_{L^{\infty}(S_{u^{\prime},\ubar})}\sup_{\mathcal{P}_{x}}|f_{2}(\pslash)|\sup_{\mathcal{P}_{x}}|f_{3}(p^{4})|}{|u^{\prime}|^{12}}\text{d}u^{\prime} \Bigg)^{\frac{1}{2}}\\\nonumber 
 \lesssim& a^{2}\mathcal{V}\Bigg(\int_{0}^{1}\int_{u_{\infty}}^{u}|u^{\prime}|^{10}||\nablasl^{2}\alphabar||^{2}_{L^{2}(S_{u^{\prime},\ubar^{\prime}})}\text{d}u^{\prime} \text{d}\ubar^{\prime} \Bigg)^{\frac{1}{2}}
 \Bigg(\int_{u_{\infty}}^{u}\frac{\sup_{\ubar}||f_{1}(\psi)||_{L^{\infty}(S_{u^{\prime},\ubar})}\sup_{\mathcal{P}_{x}}|f_{2}(\pslash)|\sup_{\mathcal{P}_{x}}|f_{3}(p^{4})|}{a^{-3}|u^{\prime}|^{12}}\text{d}u^{\prime} \Bigg)^{\frac{1}{2}}\\\nonumber 
 \lesssim& a^{2}\mathcal{V}\Bigg(\sup_{\ubar}\int_{u_{\infty}}^{u}|u^{\prime}|^{10}||\nablasl^{2}\alphabar||^{2}_{L^{2}(S_{u^{\prime},\ubar^{\prime}})}\text{d}u^{\prime} \Bigg)^{\frac{1}{2}}\Bigg(\int_{u_{\infty}}^{u}\frac{\sup_{\ubar}||f_{1}(\psi)||_{L^{\infty}(S_{u^{\prime},\ubar})}\sup_{\mathcal{P}_{x}}|f_{2}(\pslash)|\sup_{\mathcal{P}_{x}}|f_{3}(p^{4})|}{a^{-3}|u^{\prime}|^{12}}\text{d}u^{\prime} \Bigg)^{\frac{1}{2}}.
\end{align}
Now recall that $s_{2}(\nablasl^{2}\alphabar)=s_{2}(\alphabar)+1=3$ and therefore the estimates read 
\begin{align}
\notag&\int_{u_{\infty}}^{u}\frac{|u^{\prime}|^{10}}{a^{3}}||\nablasl^{2}\alphabar||^{2}_{L^{2}(S_{u^{\prime},\ubar^{\prime}})}\text{d}u^{\prime} =\int_{u_{\infty}}^{u}\frac{|u^{\prime}|^{10}}{a^{3}} \frac{a^{6}}{|u^{\prime}|^{12}}||||\nablasl^{2}\alphabar||^{2}_{L^{2}_{sc}(S_{u^{\prime},\ubar^{\prime}})}\text{d}u^{\prime} \\
=&\int_{u_{\infty}}^{u}\frac{a}{|u^{\prime}|^{2}}||||a\nablasl^{2}\alphabar||^{2}_{L^{2}_{sc}(S_{u^{\prime},\ubar^{\prime}})}\text{d}u^{\prime} =||(a^{\frac{1}{2}}\nablasl)^{2}\alphabar||^{2}_{L^{2}_{sc}(\Hbar)}
\end{align}
and therefore 
\begin{align} \notag &\int_{u_{\infty}}^{u}\int_{0}^{1}\int_{S_{u,\ubar}}\Bigg[ \int_{\mathcal{P}_{x}}V_{A}V_{B}V_{C}f \hsp \nablasl^{2}\alphabar \hsp f_{1}(\psi)\hsp f_{2}(\pslash)\hsp f_{3}(p^{4}) V_{0}f\hsp \text{d}\mathcal{P}_{x}\Bigg]\sqrt{\det\gslash}\hsp \text{d}\theta^1\text{d}\theta^2\text{d}\ubar^{\prime} \text{d}u^{\prime} \\
  \lesssim& a^{2}\mathcal{V}\Bigg(\sup_{\ubar}||(a^{\frac{1}{2}}\nablasl)^{2}\alphabar||^{2}_{L^{2}_{sc}(\Hbar_{\ubar})}\Bigg)^{\frac{1}{2}}\sup_{\ubar}\Bigg(\int_{u_{\infty}}^{u}\frac{\sup_{\ubar}||f_{1}(\psi)||_{L^{\infty}(S_{u^{\prime},\ubar})}\sup_{\mathcal{P}_{x}}|f_{2}(\pslash)|\sup_{\mathcal{P}_{x}}|f_{3}(p^{4})|}{a^{-3}|u^{\prime}|^{12}}\text{d}u^{\prime} \Bigg)^{\frac{1}{2}}    
\end{align}
Now recall the term that actually contains $\alphabar$ so that we have an explicit form of the functions $f_{1},f_{2}$, and $f_{3}$
\begin{align}
&\int_{u_{\infty}}^{u}\int_{0}^{1}\int_{S_{u,\ubar}}\Bigg[ \int_{\mathcal{P}_{x}}V_{A}V_{B}V_{C}f\hsp \nablasl^{2}\alphabar \hsp \pslash \hsp V_{0}f\hsp \text{d}\mathcal{P}_{x}\Bigg]\sqrt{\det\gslash}\hsp \text{d}\theta^1\text{d}\theta^2 \text{d}\ubar^{\prime} \text{d}u^{\prime} \\ \notag 
\lesssim& a^{2}\mathcal{V}\Bigg(\sup_{\ubar}||(a^{\frac{1}{2}}\nablasl)^{2}\alphabar||^{2}_{L^{2}_{sc}(\Hbar_{\ubar})}\Bigg)^{\frac{1}{2}}\sup_{\ubar}\Bigg(\int_{u_{\infty}}^{u}\frac{\sup_{\mathcal{P}_{x}}|\pslash|}{a^{-3}|u^{\prime }|^{12}}\text{d}u^{\prime } \Bigg)^{\frac{1}{2}}\\
\lesssim& a^{2}\hsp \mathcal{V}\hsp \Gamma \Bigg(\int_{u_{\infty}}^{u}\frac{|u^{\prime}|^{-1}}{a^{-3}|u^{\prime}|^{12}}\text{d}u^{\prime} \Bigg)^{\frac{1}{2}}\lesssim   \mathcal{V}\hsp \Gamma\frac{a^{2+\frac{3}{2}}}{|u|^{6}}.
\end{align}
The remaining terms are estimated in an exactly similar fashion-they verify similar or better estimates. A collection of all the terms yields the desired estimate 
\begin{eqnarray}
&\mathfrak{V}_{3}\lesssim 1+\mathcal{I}^{0}_{3}.
\end{eqnarray}
This completes the proof for the third derivatives of the Vlasov field. 
\end{proof}

\subsection{Estimating the stress-energy tensor $T$ from the estimates of $f$}
\noindent Once the norms $\mathfrak{V}_{1}-\mathfrak{V}_{3}$ are estimated, we are left to estimate the norm of the Vlasov stress-energy tensor. For this, we need to use the expression for $e_{\mu}$ in terms of the vector fields $V$. We will consider the action of derivatives belonging to the following set $\mathcal{D}:=\{\snabla_{4},a^{\frac{1}{2}}\nablasl, |u|\snabla_{3}\}$. First consider $\mathcal{D}=|u|\nabla_{3}$ which is same as $|u|e_{3}$ when $f$ is a function. First, recall the vital expressions 
\begin{itemize}
\item There holds \begin{align}  e_A =& V_{(A)}  + \modu^2 \left( \frac{\Gammaslash^C_{AB}\hsp p^B}{p^3} + {\chibarhat_A}^C +\frac{{\chi_A}^C \hsp p^4}{p^3}\right) V_{(4+C)}  \notag \\ &+\frac{\modu^2}{2} (\tr\chibar+\frac{2}{\modu}) V_{(4+A)} +\left( \frac{\frac{1}{2}\chi_{AB}\hsp p^B}{p^3} -\etabar_A\right)(V_{(4)}+V_{(0)}). \label{expressionVA} \end{align}

\item There holds 
\begin{align} \notag 
\modu e_3 =& \left( 1 + \modu \hsp\frac{\eta_A \hsp p^A}{p^3}\right) V_{(0)}+ \modu^2 \left( \frac{({\chibar_B}^A-e_{B}(b^{A})) \hsp p^B + \eta^A \hsp p^4}{p^3}\right) |u|V_{(4+A)}\\ +& \modu \hsp \frac{\eta_A \hsp p^A}{p^3} V_{(4)}+|u|\omegabar (V_{(4)}+V_{(0)}). \label{expressionV0}
\end{align}
\item There holds 

\begin{align} 
e_4 = V_{(3)} + \modu^2\left( \frac{{\chi_B}^A p^B}{p^3} + 2\hsp \etabar^A \right)V_{(4+A)}
\label{expressionV3}
\end{align}
\end{itemize}
Now, recall that we have the estimates for the $V$ derivatives of $f$. We can use these identities to estimate the ordinary spacetime derivatives of the Vlasov stress-energy tensor. We do this systematically. For example, note that in light of the point-wise estimates on the Ricci coefficients, 
\begin{align}
||u|e_{3}f|\lesssim |V_{0}f|+\frac{\Gamma}{|u|}||u|V_{4+A}f|+\frac{\Gamma}{|u|}|V_{4}f|    
\end{align}
since 
\begin{eqnarray}
1 + \modu \hsp\frac{\eta_A \hsp p^A}{p^3}=1+\frac{a^{\frac{1}{2}}}{|u|^{2}}\lesssim 1\\
\modu^3 \left( \frac{{\chibar_B}^A \hsp p^B + \eta^A \hsp p^4}{p^3}\right)\lesssim \frac{\Gamma}{|u|}+\frac{a^{\frac{1}{2}}}{|u|^{2}}\\
\modu \hsp \frac{\eta_A \hsp p^A}{p^3}\lesssim \frac{a^{\frac{1}{2}}\Gamma}{|u|^{2}}\lesssim 1.
\end{eqnarray}
and the term $|\modu^3 \left( \frac{{\chibar_B}^A \hsp p^B + \eta^A \hsp p^4}{p^3}\right) V_{(4+A)}(f)|\lesssim \frac{\Gamma}{|u|}||u|V_{4+A}f|\lesssim ||u|V_{4+A}f|$. First, we provide the $L^{2}_{sc}(S_{u,\ubar})$ and $L^{4}_{sc}(S_{u,\ubar})$ estimates for the first derivatives of the vlasov field. Recall the following elementary lemma 
\begin{remark}[Action of Derivative on stress-energy tensor]
Suppose the Vlasov stress-energy tensor be $T_{\mu\nu}:=\int_{\mathcal{P}_{x}}f(x,p)p_{\mu}p_{\nu}\frac{\sqrt{\gslash}dp^{1}dp^{2}dp^{3}}{p^3}$. One can explicitly compute the action of the derivatives $\nablasl,\snabla_{4},$ and $\snabla_{3}$ on the stress tensor. For example, observe the following expression for $\snabla_{3}\Te_{3A}$
\begin{align} \notag 
\snabla_{3} &\Te_{3A}=4e_{3}\bigg(\int_{\mathcal{P}_{x}}fp_{A}p^{4}\frac{\sqrt{\det \gslash}dp^{1}dp^{2}dp^{3}}{p^3}\Bigg)-\Gamma^{C}_{3A}T_{3C}\\
=&T[e_{3}(f)]_{3A}+\tr\chibar T[f]_{3A}+[\chibar^{C}_{A}-e_{A}(b^{C})]T_{3C}+T[e_{3}(\log p^{4})f]_{3A}+4\bigg(\int_{\mathcal{P}_{x}}fe_{3}(p_{A})p^{4}\frac{\sqrt{\gslash}dp^{1}dp^{2}dp^{3}}{p^3}\Bigg).
\end{align}
where one uses the fact that $p_{A}=\gslash_{AC}p^{C}$.
One can similarly evaluate the action of derivatives on other stress-energy tensor components.
\end{remark}

\begin{proposition}[First-order Estimates for the Vlasov Stress-Energy Tensor Components]\label{firstL2}

\vspace{3mm}
Let \linebreak $(\mathcal{M},g)$ be a four-dimensional Lorentzian manifold foliated by a double null foliation \(\{S_{u,\ubar}\}_{(u,\ubar)}\), with associated null coordinates \(u,\ubar\), and let \(e_3, e_4\) denote the standard incoming and outgoing normalized null vector fields satisfying \(g(e_3,e_4) = -2\), and \(\{e_A\}_{A=1,2}\) an orthonormal frame tangent to \(S_{u,\ubar}\). Let \(f\colon \mathcal{P} \subset T^*\mathcal{M} \to \mathbb{R}_{\geq 0}\) denote the distribution function of a collisionless ensemble of particles governed by the Vlasov equation, and let \(T_{\alpha\beta}\) denote the associated stress-energy tensor defined by
\[
T_{\alpha\beta}(x) := \int_{\mathcal{P}_x} f(x,p) \, p_\alpha p_\beta \, \frac{\sqrt{\gslash}dp^{1}dp^{2}dp^{3}}{p^3},
\]
where \(p\in \mathcal{P}_x\) lies on the mass shell \(g^{\mu\nu}p_\mu p_\nu = 0\), and the integration is taken with respect to the invariant volume measure on the future-directed null cone at \(x\). Fix \(a \gg 1\), a large parameter governing the hierarchy of estimates, and define the second-order differential operator set
\[
\mathcal{D} := \left\{ |u| \snabla_3,\, a^{1/2} \nablasl,\, \snabla_4 \right\} ,
\]
where \(\nablasl\) denotes the induced Levi-Civita connection on \(S_{u,\ubar}\), and the normalization of vector fields is consistent with the double-null framework.

\noindent Then, for each \(u < 0\), \(\ubar \geq 0\), and for every pair of tensorial components \[(\alpha,\beta)\in \{(4,4), (4,3), (4,A), (A,3), (A,B),\\ (3,3), \}\] the following scale-invariant \(L^2\) estimates hold on the sphere \(S_{u,\ubar}\):
\begin{align}
\frac{|u|}{a} \left\| \mathcal{D} T_{44} \right\|_{\mathcal{L}^2_{(sc)}(S_{u,\ubar})} &\lesssim 1, &
\frac{|u|}{a} \left\| \mathcal{D} T_{43} \right\|_{\mathcal{L}^2_{(sc)}(S_{u,\ubar})} &\lesssim \frac{1}{a}, \notag \\
\frac{|u|}{a} \left\| \mathcal{D} T_{4A} \right\|_{\mathcal{L}^2_{(sc)}(S_{u,\ubar})} &\lesssim \frac{1}{a^{\frac{1}{2}}}, &
\frac{|u|}{a} \left\| \mathcal{D} T_{A3} \right\|_{\mathcal{L}^2_{(sc)}(S_{u,\ubar})} &\lesssim \frac{1}{a^{\frac{3}{2}}}, \label{eq:StressEnergyEstimates1} \\
\frac{|u|}{a} \left\| \mathcal{D} T_{AB} \right\|_{\mathcal{L}^2_{(sc)}(S_{u,\ubar})} &\lesssim \frac{1}{a}, &
\frac{|u|}{a} \left\| \mathcal{D} T_{33} \right\|_{\mathcal{L}^2_{(sc)}(S_{u,\ubar})} &\lesssim \frac{1}{a^{2}}. \notag
\end{align}
The implicit constants in the estimates are numerical. Moreover, we have the improved estimate 
\begin{align}
\frac{|u|^{2}}{a} \left\| T_{44}[|u|e_{3}f] \right\|_{\mathcal{L}^2_{(sc)}(S_{u,\ubar})} &\lesssim 1    
\end{align}
and also for the point-wise norm 
\begin{align}
\frac{|u|^{2}}{a} \left\| T_{44}[|u|e_{3}f] \right\|_{L^\infty_{sc}(S_{u,\ubar})} &\lesssim 1.     
\end{align}
\end{proposition}

\begin{proof}
The improved estimates are easy to see once the previous estimates are obtained, since 
\begin{align}
|u|e_{3}f=\left( 1 + \modu \hsp\frac{\eta_A \hsp p^A}{p^3}\right) V_{(0)}f+ \modu^2 \left( \frac{({\chibar_B}^A-e_{B}(b^{A})) \hsp p^B + \eta^A \hsp p^4}{p^3}\right) |u|V_{(4+A)}f+ \modu \hsp \frac{\eta_A \hsp p^A}{p^3} V_{(4)}f+\omegabar V_{(4)}f    
\end{align}
and verification of the estimates $|V_{0}f|\lesssim a|u|^{-1},~||u|V_{4+A}f|\lesssim a,~|V_{(4)}f|\lesssim a$ implies $||u|e_{3}f|\lesssim \frac{a}{|u|}$ since the coefficients exhibit 
\begin{align}
1 + \big\| \, |u|\,\eta\,\pslash \, \big\|
&\;\lesssim\;
1 + \frac{a^{1/2} \, \Gamma}{|u|^{2}}
\;\lesssim\;
1, 
\\[0.4em]
\modu^2 \left( \frac{({\chibar_B}^A-e_{B}(b^{A})) \hsp p^B + \eta^A \hsp p^4}{p^3}\right)
\, \big| V_{4+A} f \big|
&\;\lesssim\;
\frac{1}{|u|} \, \big\|\, |u| \, V_{4+A}f \,\big\|
\,
\\[0.4em]
\big\|\, |u| \, \eta \, \pslash \,\big\|
&\;\lesssim\;
\frac{a^{1/2} \, \Gamma \, |u|}{|u|^{3}}
\;\lesssim\;
1,
\\[0.4em]
|u| \, \big| \tr\chibar \big|
&\;\lesssim\;
|u| \,
\Bigg(
  \big| \tr\chibar + \tfrac{2}{|u|} \big|
  + \frac{2}{|u|}
\Bigg) \,
\lesssim \frac{\Gamma}{|u|}
\;+\;
1,
\end{align}
the extra $|u|^{-1}$ decay factor appears.
 The estimates for the first derivatives are straightforward. First, recall the improved first-order estimates for $\mathfrak{V}^{S}_{1}$
 \begin{align}
  \mathfrak{V}^{S}_{1}\lesssim 1+\mathcal{I}^{0}_{S},   
 \end{align}
 where $\mathcal{I}^{0}_{S}$ is the initial data for $\mathfrak{V}^{S}_{1}$. Now controlling the $L^{2}(S_{u,\ubar})$ of the first derivatives is a straightforward consequence of the  afore-mentioned expression relating the vector fields $e_{\mu}$ to the $V$-vector fields. Here, we only present the estimates for one stress-energy component. The rest of the estimates follow in an exact similar fashion. We focus on the higher-order estimates that are non-trivial. For $\mathcal{D}=\modu\snabla_3,$
\begin{align}
T_{33} 
&= \int_{\mathcal{P}_{x}} 
    f \, p_{3} p_{3} \,
    \frac{\sqrt{\det \gslash}}{p^{3}}
    \, dp^{1} \, dp^{2} \, dp^{3}, 
\\[0.5em]
\mathcal{D}T_{33} 
&= \int_{\mathcal{P}_{x}}
    \Big( \mathcal{D}f + \tfrac12 |u| \tr\chibar \, f \Big) \, 
    p^{4} p^{4} \,
   \frac{\sqrt{\det \gslash}}{p^{3}}
    \, dp^{1} \, dp^{2} \, dp^{3}
\notag \\[-0.25em]
&\quad
    + 2 \int_{\mathcal{P}_{x}}
        f \, p^{4} \, \mathcal{D}(p^{4}) \,
        \frac{\sqrt{\det \gslash}}{p^{3}}
        \, dp^{1} \, dp^{2} \, dp^{3},
\\[0.5em]
&= \int_{\mathcal{P}_{x}}
    \Bigg(
        \underbrace{
        \big( 1 + |u|\,\eta \,\pslash \big) \, V_{0}f
        + |u|^{3} 
            \Big( (\chibar - \nablasl b + \Gammaslash)\pslash
                   + \eta \, p^{4} \Big) \, V_{4+A}f
        + |u|\, \eta\, \pslash \, V_{4}f
        }_{\text{transport contributions}}
        + \frac12 |u| \tr\chibar \, f
    \Bigg)
    p^{4} p^{4} \,
    \text{d}\mathcal{P}_{x}
\notag \\[-0.25em]
&\quad
    + 2 \int_{\mathcal{P}_{x}}
        f \, p^{4} \, \mathcal{D}(p^{4}) \,
        \frac{\sqrt{\det \gslash}}{p^{3}}
        \, dp^{1} \, dp^{2} \, dp^{3}.
\end{align}
Recalling that
\begin{align}
 e_{3}(p^{4})=\frac{\Bigg(2\chibar _{AB}-\gslash _{BC}e_{A}(b^{C})-\gslash_{AC}e_{B}(b^{C})\Bigg)p^{A}p^{B}}{2p^{3}},   
\end{align}
as well as the estimates on $\Gamma^{A}_{3B}=\chibar^{A}_{B}-e_{B}(b^{A})$, we immediately obtain 
\begin{align}
 |e_{3}p^{4}|\lesssim \frac{\Gamma}{|u|^{3}}.   
\end{align}
The idea is that the momentum variables $p^{\mu}$ and the Ricci coefficients are bounded on the mass-shell in $L^{\infty}$, whenever estimates are available. For example, 
\begin{align}
1 + \big\| \, |u|\,\eta\,\pslash \, \big\|
&\;\lesssim\;
1 + \frac{a^{1/2} \, \Gamma}{|u|^{2}}
\;\lesssim\;
1, 
\\[0.4em]
\modu^2 \left( \frac{({\chibar_B}^A-e_{B}(b^{A})) \hsp p^B + \eta^A \hsp p^4}{p^3}\right)
\, \big| V_{4+A} f \big|
&\;\lesssim\;
\frac{1}{|u|} \, \big\|\, |u| \, V_{4+A}f \,\big\|
\,
\\[0.4em]
\big\|\, |u| \, \eta \, \pslash \,\big\|
&\;\lesssim\;
\frac{a^{1/2} \, \Gamma \, |u|}{|u|^{3}}
\;\lesssim\;
1,
\\[0.4em]
|u| \, \big| \tr\chibar \big|
&\;\lesssim\;
|u| \,
\Bigg(
  \big| \tr\chibar + \tfrac{2}{|u|} \big|
  + \frac{2}{|u|}
\Bigg) \,
\lesssim \frac{\Gamma}{|u|}
\;+\;
1.
\end{align}
Therefore, using the decay of momentum support,
\begin{align} \notag 
\big| \mathcal{D}T_{33} \big|
&\;\lesssim\;
|u|^{-4}
\int_{\mathcal{P}_{x}}
\Big(
   \, |V_{0}f|
   + \big\|\, |u| V_{4+A}f \,\big\|
   + |V_{4}f|
   + |f|
\Big)
\frac{\sqrt{\det\gslash}}{p^{3}}
\, dp^{1} \, dp^{2} \, dp^{3}
\\[0.4em]
&\;\lesssim\;
|u|^{-4}
\Big(
   \, \| V_{0}f \|_{L^{2}(\mathcal{P}_{x})}
   + \big\|\, |u| V_{4+A}f \,\big\|_{L^{2}(\mathcal{P}_{x})}
   + \| V_{4}f \|_{L^{2}(\mathcal{P}_{x})}
   + \| f \|_{L^{2}(\mathcal{P}_{x})}
\Big)
\notag \\[-0.25em] \notag
&\qquad\qquad\quad
\times
\Bigg(
   \int_{\mathcal{P}_{x}}
   \frac{\sqrt{\det\gslash}}{p^{3}}
   \, dp^{1} \, dp^{2} \, dp^{3}
\Bigg)^{\!1/2}
\\[0.4em]
&\;\lesssim\;
|u|^{-5}
\Big(
   \, \| V_{0}f \|_{L^{2}(\mathcal{P}_{x})}
   + \big\|\, |u| V_{4+A}f \,\big\|_{L^{2}(\mathcal{P}_{x})}
   + \| V_{4}f \|_{L^{2}(\mathcal{P}_{x})}
   + \| f \|_{L^{2}(\mathcal{P}_{x})}
\Big),
\end{align}
where H\"older's inequality is applied on the mass shell $\mathcal{P}_{x}$.  
By Minkowski's inequality, we further obtain
\begin{align}
\big\| \mathcal{D}T_{33} \big\|_{L^{2}(S_{u,\ubar})}
&\;\lesssim\;
|u|^{-5}
\Big(
   \, \| V_{0}f \|_{L^{2}(S_{u,\ubar}) L^{2}(\mathcal{P}_{x})}
   + \big\|\, |u| V_{4+A}f \,\big\|_{L^{2}(S_{u,\ubar}) L^{2}(\mathcal{P}_{x})}
\notag \\[-0.25em]
&\qquad
   + \| V_{4}f \|_{L^{2}(S_{u,\ubar}) L^{2}(\mathcal{P}_{x})}
   + \| f \|_{L^{2}(S_{u,\ubar}) L^{2}(\mathcal{P}_{x})}
\Big).
\end{align}
By the estimates on the first derivative of the Vlasov field, we have
\begin{align}
\big\| \mathcal{D}T_{33} \big\|_{L^{4}(S_{u,\ubar})}
&\;\lesssim\;
\frac{a}{|u|^{5}}.
\end{align}
Since $s_{2}(\mathcal{D}T_{33}) = s_{2}(T_{33}) = 2$, it follows that
\begin{align}
\big\| \mathcal{D}T_{33} \big\|_{\mathcal{L}^{2}_{(sc)}(S_{u,\ubar})}
&= \frac{|u|^{4}}{a^{2}} \,
   \big\| \mathcal{D}T_{33} \big\|_{L^{2}(S_{u,\ubar})}
\notag \;\lesssim\;
\frac{1}{a^{2}} \cdot \frac{a}{|u|^{5}}
\;\lesssim\;
\frac{1}{a\,|u|}.
\end{align}
In particular, we automatically obtain
\begin{align}
\frac{|u|}{a} \,
\big\| \mathcal{D}T_{33} \big\|_{\mathcal{L}^{4}_{(sc)}(S_{u,\ubar})}
&\;\lesssim\;
\frac{1}{a^{2}}.
\end{align}
We  now proceed to consider the angular derivative
\[
\mathcal{D} := a^{1/2} \, \nablasl .
\]
We first recall that
\begin{align}
\Te_{33}
&= \int_{\mathcal{P}_{x}} f \, p_{3} p_{3} \,
   \frac{\sqrt{\det\gslash}}{p^{3}}
   \, dp^{1} \, dp^{2} \, dp^{3},
\\[0.4em]
\nablasl \Te_{33}
&= \int_{\mathcal{P}_{x}}
    \big( \mathcal{D} f + \Gammaslash \, f \big)
    \, p^{4} p^{4}
    \frac{\sqrt{\det\gslash}}{p^{3}}
    \, dp^{1} \, dp^{2} \, dp^{3}
\notag \\[0.25em]
&\quad
  + 2 \int_{\mathcal{P}_{x}}
    f \, p^{4} \, e_{A}(p^{4})
    \frac{\sqrt{\det\gslash}}{p^{3}}
    \, dp^{1} \, dp^{2} \, dp^{3}.
\end{align}

\noindent Next, we recall that
\begin{align}
e_{A}
&= V_{(A)}
   + \modu^{2} \left(
       \frac{\Gammaslash^{C}_{AB} \, p^{B}}{p^{3}}
       + {\chibarhat_{A}}^{C}
       + \frac{{\chi_{A}}^{C} \, p^{4}}{p^{3}}
     \right) V_{(4+C)}
\notag \\[0.25em]
&\quad
   + \frac{\modu^{2}}{2}
     \left( \tr\chibar + \frac{2}{\modu} \right)
     V_{(4+A)}
   + \left(
       \frac{\chi_{AB} \, p^{B}}{2\, p^{3}}
       - \etabar_{A}
     \right)
     \big( V_{(4)} + V_{(0)} \big).
\end{align}

Therefore,
\begin{align}
\snabla_A \Te_{33}
&\;\sim\;
  \int_{\mathcal{P}_{x}}
  \Big[
      V_{A} f
    + |u|^{2} \, (\Gammaslash \, \pslash + \chibarhat + p^{4} \chi) \, V_{4+A} f
    + |u|^{2} \, \widetilde{\tr\chibar} \, V_{4+A} f
\notag \\[-0.25em]
&\hspace{6em}
    + (\chi \, \pslash + \etabar) \, \big( V_{4} f + V_{0} f \big)
    + \Gammaslash \, f
  \Big]
  \, p^{4} p^{4}
  \frac{\sqrt{\det\gslash}}{p^{3}}
  \, dp^{1} \, dp^{2} \, dp^{3}.
\end{align}
By applying the pointwise estimates, the decay of the momentum support and Hölder's inequality on the mass–shell $\mathcal{P}_{x}$, we find
\begin{align}
    |\snabla_A \Te_{33}|
    &\lesssim \frac{1}{|u|^{5}} \Big(
        \| V_A f \|_{L^{2}(\mathcal{P}_{x})}
        + \|\, |u| V_{4+A} f \|_{L^{2}(\mathcal{P}_{x})}
        + \| V_{4} f \|_{L^{2}(\mathcal{P}_{x})}
        + \| f \|_{L^{2}(\mathcal{P}_{x})}
    \Big).
\end{align}
Note that
\begin{align}\notag
    |\mathcal{D} \Te_{33}|
    &= a^{\frac12} \sqrt{ \gslash^{AB} \, \nablasl_A \Te_{33} \, \nablasl_B T]\Te_{33} } \\
    &\lesssim \frac{a^{\frac12}}{|u|} \cdot \frac{1}{|u|^{5}}
        \Big(
            \| V_A f \|_{L^{2}(\mathcal{P}_{x})}
            + \|\, |u| V_{4+A} f \|_{L^{2}(\mathcal{P}_{x})}
            + \| V_{4} f \|_{L^{2}(\mathcal{P}_{x})}
            + \| f \|_{L^{2}(\mathcal{P}_{x})}
        \Big).
\end{align}
Hence,
\begin{align}
    \| \mathcal{D} \Te_{33} \|_{L^{2}(S_{u,\ubar})}
    &\lesssim \frac{a^{\frac12}}{|u|^{6}}
    \Big(
        \| V_A f \|_{L^{2}(S_{u,\ubar}) L^{2}(\mathcal{P}_{x})}
        + \|\, |u| V_{4+A} f \|_{L^{2}(S_{u,\ubar}) L^{2}(\mathcal{P}_{x})} \notag \\
    &\quad + \| V_{4} f \|_{L^{2}(S_{u,\ubar}) L^{2}(\mathcal{P}_{x})}
        + \| f \|_{L^{2}(S_{u,\ubar}) L^{2}(\mathcal{P}_{x})}
    \Big).
\end{align}
Using the Vlasov estimates from Proposition~\ref{firstvlasov}, we obtain
\begin{align}
    \| \mathcal{D} \Te_{33} \|_{L^{4}(S_{u,\ubar})}
    \lesssim \frac{a^{\frac32}}{|u|^{6}}.
\end{align}
Since $s_{2}(\nablasl T_{33}) = \frac52$, it follows that
\begin{align}
    \| \mathcal{D} T_{33} \|_{L^{2}(S_{u,\ubar})}
    = \frac{a^{1}}{|u|^{6} a^{\frac52}} \, |u|^{5}
    \lesssim \frac{1}{a |u|}.
\end{align}
Therefore,
\begin{align}
    \frac{|u|}{a} \, \| \mathcal{D} T_{33} \|_{\mathcal{L}^{2}_{(sc)}(S_{u,\ubar})}
    \lesssim \frac{1}{a^{2}}.
\end{align}

\medskip
\noindent For $\mathcal{D}_3 := \snabla_{4}$, recall that
\begin{align}
 e_4 = V_{(3)} + \modu^2\left( \frac{{\chi_B}^A p^B}{p^3} + 2\hsp \etabar^A \right)V_{(4+A)}   
\end{align}
Moreover,
\begin{align}
      \snabla_{4} \Te_{33}
    &\sim \int_{\mathcal{P}_{x}} \big( e_{4} f + \tr\chi \, f \big) \, p^{4} p^{4}
        \frac{\sqrt{\det \gslash}}{p^{3}} \, dp^{1} dp^{2} dp^{3}+\int_{\mathcal{P}_{x}}fe_{4}(p^{4}p^{4})\frac{\sqrt{\det \gslash}}{p^{3}} \, dp^{1} dp^{2} dp^{3}.
\end{align}
Thus,
\begin{align}
    \snabla_{4} \Te_{33}
    &\sim \int_{\mathcal{P}_{x}} \Big(
        V_{3} f
        + |u|^{2} (\chi  \hsp \pslash + \etabar) V_{4+A} f
        + \tr\chi \, f
    \Big)
    p^{4} \hsp p^{4} \frac{\sqrt{\det \gslash}}{p^{3}}
    \, dp^{1} dp^{2} dp^{3}.
\end{align}
The pointwise bounds imply
\begin{align}
    |\snabla_{4} \Te_{33}|
    \lesssim |u|^{-5} \Big(
        \| V_{3} f \|_{L^{2}(\mathcal{P}_{x})}
        + \|\, |u| V_{4+A} f \|_{L^{2}(\mathcal{P}_{x})}
        + \| V_{4} f \|_{L^{2}(\mathcal{P}_{x})}
        + \| V_{0} f \|_{L^{2}(\mathcal{P}_{x})}
        + \| f \|_{L^{2}(\mathcal{P}_{x})}
    \Big).
\end{align}
Therefore,
\begin{align}
    \| \snabla_{4} \Te_{33} \|_{L^{2}(S_{u,\ubar})}
    &\lesssim |u|^{-5} \Big(
        \| V_{3} f \|_{L^{2}(S_{u,\ubar}) L^{2}(\mathcal{P}_{x})}
        + \|\, |u| V_{4+A} f \|_{L^{2}(S_{u,\ubar}) L^{2}(\mathcal{P}_{x})} \notag \\
    &\quad + \| V_{4} f \|_{L^{2}(S_{u,\ubar}) L^{2}(\mathcal{P}_{x})}
        + \| V_{0} f \|_{L^{2}(S_{u,\ubar}) L^{2}(\mathcal{P}_{x})}
        + \| f \|_{L^{2}(S_{u,\ubar}) L^{2}(\mathcal{P}_{x})}
    \Big) \notag \\
    &\lesssim \frac{a}{|u|^{5}}.
\end{align}
Since $s_{2}(\snabla_{4} \Te_{33}) = s_{2}(\Te_{33}) = 2$, we deduce
\begin{align}
    \| \snabla_{4} \Te_{33} \|_{\mathcal{L}^{2}_{(sc)}(S_{u,\ubar})}
    \lesssim \frac{|u|^{4}}{a^{2}} \cdot \frac{a}{|u|^{5}}
    \lesssim \frac{1}{a |u|},
\end{align}
or equivalently,
\begin{align}
    \frac{|u|}{a} \, \| \snabla_{4} T_{33} \|_{\mathcal{L}^{2}_{(sc)}(S_{u,\ubar})}
    \lesssim \frac{1}{a^{2}}.
\end{align}

Now, notice that the remaining estimates follow exactly in a similar fashion. 
\end{proof}

\begin{proposition}[Sharp High-Order Estimates for the Stress-Energy Tensor Components]\label{prop:StressEnergyHighOrder}
Let \linebreak \((\mathcal{M},g)\) be a four-dimensional Lorentzian manifold foliated by a double null foliation \(\{S_{u,\ubar}\}_{(u,\ubar)}\), with associated null coordinates \(u,\ubar\), and let \(e_3, e_4\) denote the standard incoming and outgoing normalized null vector fields satisfying \(g(e_3,e_4) = -2\), and \(\{e_A\}_{A=1,2}\) an orthonormal frame tangent to \(S_{u,\ubar}\). Let \(f\colon \mathcal{P} \subset T^*\mathcal{M} \to \mathbb{R}_{\geq 0}\) denote the distribution function of a collisionless ensemble of particles governed by the Vlasov equation, and let \(T_{\alpha\beta}\) denote the associated stress-energy tensor defined by
\[
T_{\alpha\beta}(x) := \int_{\mathcal{P}_x} f(x,p) \, p_\alpha p_\beta \, \frac{\sqrt{\gslash}dp^{1}dp^{2}dp^{3}}{p^3},
\]
where \(p\in \mathcal{P}_x\) lies on the mass shell (\(g^{\mu\nu}p_\mu p_\nu = 0\)) and the integration is taken with respect to the afore-mentioned measure on the future-directed null cone at \(x\). Fix \(a \gg 1\), a large parameter governing the hierarchy of estimates, and define the second-order differential operator set
\[
\mathcal{D}^2 := \left\{ D_1 D_2 : D_1, D_2 \in \left\{ |u| \snabla_3,\, a^{1/2} \nablasl,\, \snabla_4 \right\} \right\},
\]
where \(\nablasl\) denotes the induced Levi-Civita connection on \(S_{u,\ubar}\), and the normalization of vector fields is consistent with the double-null framework. Then, for each \(u < 0\), \(\ubar \geq 0\), and for every pair of tensorial components\[(\alpha,\beta)\in \{(4,4), (4,3), (4,A), (A,3), (A,B), (3,3)\},\] the following scale-invariant \(L^2\)-estimates hold on the sphere \(S_{u,\ubar}\):
\begin{align}
\frac{|u|}{a} \left\| \mathcal{D}^2 T_{44} \right\|_{\mathcal{L}^2_{(sc)}(S_{u,\ubar})} &\lesssim 1, &
\frac{|u|}{a} \left\| \mathcal{D}^2 T_{43} \right\|_{\mathcal{L}^2_{(sc)}(S_{u,\ubar})} &\lesssim \frac{1}{a}, \notag \\
\frac{|u|}{a} \left\| \mathcal{D}^2 T_{4A} \right\|_{\mathcal{L}^2_{(sc)}(S_{u,\ubar})} &\lesssim \frac{1}{a^{\frac{1}{2}}}, &
\frac{|u|}{a} \left\| \mathcal{D}^2 T_{A3} \right\|_{\mathcal{L}^2_{(sc)}(S_{u,\ubar})} &\lesssim \frac{1}{a^{\frac{3}{2}}}, \label{eq:StressEnergyEstimates2} \\
\frac{|u|}{a} \left\| \mathcal{D}^2 T_{AB} \right\|_{\mathcal{L}^2_{(sc)}(S_{u,\ubar})} &\lesssim \frac{1}{a}, &
\frac{|u|}{a} \left\| \mathcal{D}^2 T_{33} \right\|_{\mathcal{L}^2_{(sc)}(S_{u,\ubar})} &\lesssim \frac{1}{a^{2}}. \notag
\end{align}
The implicit constants in the estimates are numerical. In addition, let
\[
\mathcal{D}^3 := \left\{ D_1 D_2 D_{3} : D_1, D_2, D_3 \in \left\{ |u| \snabla_3,\, a^{1/2} \nablasl,\, \snabla_4 \right\} \right\},
\]
where \(\nablasl\) denotes the induced Levi-Civita connection on \(S_{u,\ubar}\), and the normalization of vector fields is consistent with the double-null framework. Then the following estimate holds for the third derivatives of every pair of tensorial components \((\alpha,\beta)\in \{(4,4), (4,3), (4,A), (A,3), (A,B), (3,3)\}\) ,
\begin{align}
\frac{|u|}{a} \left\| \mathcal{D}^3 T_{44} \right\|_{L^2_{sc}(H_{u})} &\lesssim 1, &
\frac{|u|}{a} \left\| \mathcal{D}^3 T_{43} \right\|_{L^2_{sc}(H_{u})} &\lesssim \frac{1}{a}, \notag \\
\frac{|u|}{a} \left\| \mathcal{D}^3 T_{4A} \right\|_{L^2_{sc}(H_{u})} &\lesssim \frac{1}{a^{\frac{1}{2}}}, &
\frac{|u|}{a} \left\| \mathcal{D}^3 T_{A3} \right\|_{L^2_{sc}(H_{u})} &\lesssim \frac{1}{a^{\frac{3}{2}}}, \label{eq:StressEnergyEstimates} \\
\frac{|u|}{a} \left\| \mathcal{D}^3 T_{AB} \right\|_{L^2_{sc}(H_{u})} &\lesssim \frac{1}{a}, &
\frac{|u|}{a} \left\| \mathcal{D}^3 T_{33} \right\|_{L^2_{sc}(H_{u})} &\lesssim \frac{1}{a^{2}}. \notag
\end{align}
uniformly in $u\in [u,u_{\infty}]$. The constants involved in the estimates are numerical.
\end{proposition}

\begin{proof}
The proof is a simple consequence of expressing the vector fields $(a^{\frac{1}{2}}e_{A},e_{4},|u|e_{3})$ in terms of the vector field $V$s. In light of the identities \eqref{expressionVA}-\eqref{expressionV3}, we can successively express $\mathcal{D}^{2}$ in terms of $V^{2}f$. For example, note the following semi-schematic expression for $\mathcal{D}=|u|\nabla_{3}$ 
\begin{align}
\mathcal{D}^{2}f \sim\; & 
\Big( |u|\, \pslash\, \eta + |u|\, \pslash\, \mathcal{D} \eta \Big) V_{(0)}f \notag \\
&+ \left(1 + \modu\, \frac{\eta_A\, p^A}{p^3} \right) 
\Bigg[ \left(1 + \modu\, \frac{\eta_A\, p^A}{p^3} \right) V_{(0)}V_{(0)}f 
+ \modu^3 \left( \frac{({\chibar_B}^{\;A}-e_{A}(b^{B})\, p^B + \eta^A\, p^4}{p^3} \right) V_{(4+A)}V_{(0)}f \notag \\
&\hspace{5em} + \modu\, \frac{\eta_A\, p^A}{p^3} V_{(4)}V_{(0)}f \Bigg] \notag \\
&+ \Big( |u|^3\, \pslash\, \chibar + |u|^3\, \pslash\, \mathcal{D} \chibar 
+ p^4\, \mathcal{D} \eta + |u|\, \chi\, \pslash\, \pslash \Big) V_{(4+A)}f \notag \\
&+ \modu^3 \left( \frac{({\chibar_B}^{\;A}-e_{A}(b^{B}))\, p^B + \eta^A\, p^4}{p^3} \right) 
\Bigg[ \left(1 + \modu\, \frac{\eta_A\, p^A}{p^3} \right) V_{(0)}V_{(4+A)}f \notag \\
&\hspace{5em} + \modu^3 \left( \frac{({\chibar_B}^{\;A}-e_{A}(b^{B})\, p^B + \eta^A\, p^4}{p^3} \right) V_{(4+A)}V_{(4+A)}f 
+ \modu\, \frac{\eta_A\, p^A}{p^3} V_{(4)}V_{(4+A)}f \Bigg] \notag \\
&+ \Big( |u|\, \pslash\, \eta + |u|\, \pslash\, \mathcal{D} \eta \Big) V_{(4)}f \notag \\
&+ \modu\, \frac{\eta_A\, p^A}{p^3} 
\Bigg[ \left(1 + \modu\, \frac{\eta_A\, p^A}{p^3} \right) V_{(0)}V_{(4)}f 
+ \modu^3 \left( \frac{({\chibar_B}^{\;A}-e_{A}(b^{B})\, p^B + \eta^A\, p^4}{p^3} \right) V_{(4+A)}V_{(4)}f \notag \\
&\hspace{5em} + \modu\, \frac{\eta_A\, p^A}{p^3} V_{(4)}V_{(4)}f \Bigg],
\end{align}
and 

\begin{align} \notag
\mathcal{D}^{3} f 
&\sim 
\mathcal{D}\big( \pslash |u| \eta + \pslash |u| \mathcal{D} \eta \big) V_0 f \\ \notag
&\quad + \big( \pslash |u| \eta + \pslash |u| \mathcal{D} \eta \big)
\big[
(1 + |u| \eta \pslash) V_0 V_0 f 
+ |u|^3 (\chibar \pslash + \eta p^4) V_{4+A} V_0 f 
+ |u| \eta \pslash V_4 V_0 f
\big] \\  \notag
&\quad + (1 + \eta \pslash) \mathcal{D}(|u| \eta \pslash) V_0 V_0 f \\  \notag
&\quad + (1 + |u| \eta \pslash)^2 
\big( \pslash |u| \eta + \pslash |u| \mathcal{D} \eta \big)
\big[
(1 + |u| \eta \pslash) V_0 V_0 V_0 f 
+ |u|^3 (\chibar \pslash + \eta p^4) V_{4+A} V_0 V_0 f 
+ |u| \eta \pslash V_4 V_0 V_0 f
\big] \\  \notag
&\quad + \mathcal{D}(|u| \eta \pslash) |u|^3 (\chibar \pslash + \eta \pslash) V_{4+A} V_0 f \\
 \notag &\quad + (1 + |u| \eta \pslash) 
\big( 
|u|^3 \pslash \chibar + |u|^3 \pslash \mathcal{D} \chibar + p^4 \mathcal{D} \eta + |u| (\chibar - e(b)) \pslash \pslash 
\big) V_{4+A} V_0 f \\
 \notag &\quad + (1 + \eta \pslash) |u|^3 (\chibar \pslash + \eta \pslash)
\big[
(1 + |u| \eta \pslash) V_0 V_{4+A} V_0 f 
+ |u|^3 (\chibar \pslash + \eta p^4) V_{4+A} V_{4+A} V_0 f 
+ |u| \eta \pslash V_4 V_{4+A} V_0 f
\big] \\  \notag
&\quad + (1 + 2 |u| \etabar \pslash) \mathcal{D}(|u| \eta \pslash) V_4 V_0 f \\  \notag
&\quad + |u| \eta \pslash (1 + |u| \eta \pslash)
\big[
(1 + |u| \eta \pslash) V_0 V_4 V_0 f 
+ |u|^3 (\chibar \pslash + \eta p^4) V_{4+A} V_4 V_0 f 
+ |u| \eta \pslash V_4 V_4 V_0 f
\big] \\  \notag
&\quad + \mathcal{D}
\big( 
|u|^3 \pslash \chibar + |u|^3 \pslash \mathcal{D} \chibar + p^4 \mathcal{D} \eta + |u| (\chibar - e(b)) \pslash \pslash 
\big) V_{4+A} f \\  \notag
&\quad + 
\big( 
|u|^3 \pslash \chibar + |u|^3 \pslash \mathcal{D} \chibar + p^4 \mathcal{D} \eta + |u| (\chibar - e(b)) \pslash \pslash 
\big) \\  \notag
&\qquad \times \big[
(1 + |u| \eta \pslash) V_0 V_{4+A} f 
+ |u|^3 (\chibar \pslash + \eta p^4) V_{4+A} V_{4+A} f 
+ |u| \eta \pslash V_4 V_{4+A} f
\big] \\  \notag
&\quad + \mathcal{D}(|u| \eta \pslash) |u|^3 (\chibar \pslash + \eta \pslash) V_0 V_{4+A} f \\  \notag
&\quad + (1 + |u| \eta \pslash) 
\big(
|u|^3 \pslash \chibar + |u|^3 \pslash \mathcal{D} \chibar + p^4 \mathcal{D} \eta + |u| (\chibar - e(b)) \pslash \pslash 
\big) V_0 V_{4+A} f \\  \notag
&\quad + (1 + \eta \pslash) |u|^3 (\chibar \pslash + \eta \pslash)
\big[
(1 + |u| \eta \pslash) V_0 V_0 V_{4+A} f 
+ |u|^3 (\chibar \pslash + \eta p^4) V_{4+A} V_0 V_{4+A} f 
+ |u| \eta \pslash V_4 V_0 V_{4+A} f
\big] \\  \notag
&\quad + |u|^3 (\chibar \pslash + \eta \pslash) 
\mathcal{D} \big( |u|^3 (\chibar \pslash + \eta \pslash) \big) V_{4+A} V_{4+B} f \\  \notag
&\quad + \big( |u|^3 (\chibar \pslash + \eta \pslash) \big)^2 
\big[
(1 + |u| \eta \pslash) V_0 V_{4+A} V_{4+B} f 
+ |u|^3 (\chibar \pslash + \eta p^4) V_{4+C} V_{4+A} V_{4+B} f 
+ |u| \eta \pslash V_4 V_{4+A} V_{4+B} f
\big] \\  \notag
&\quad + \mathcal{D}(|u|^3 (\chibar \pslash + \eta \pslash)) |u| \eta \pslash V_4 V_{4+A} f \\  \notag
&\quad + |u|^3 (\chibar \pslash + \eta \pslash) \mathcal{D}(|u| \eta \pslash) V_4 V_{4+A} f \\  \notag
&\quad + |u|^3 (\chibar \pslash + \eta \pslash) |u| \eta \pslash 
\big[
(1 + |u| \eta \pslash) V_0 V_4 V_{4+A} f 
+ |u|^3 (\chibar \pslash + \eta p^4) V_{4+C} V_4 V_{4+A} f 
+ |u| \eta \pslash V_4 V_4 V_{4+A} f
\big] \\  \notag
&\quad + \mathcal{D} \big( \pslash |u| \eta + \pslash |u| \mathcal{D} \eta \big) V_4 f \\
 \notag &\quad + \big( \pslash |u| \eta + \pslash |u| \mathcal{D} \eta \big)
\big[
(1 + |u| \eta \pslash) V_0 V_4 f 
+ |u|^3 (\chibar \pslash + \eta p^4) V_{4+C} V_4 f 
+ |u| \eta \pslash V_4 V_4 f
\big] \\  \notag
&\quad + (1 + 2 |u| \etabar \pslash) \mathcal{D}(|u| \eta \pslash) V_0 V_4 f \\  \notag
&\quad + |u| \eta \pslash (1 + |u| \eta \pslash) 
\big[
(1 + |u| \eta \pslash) V_0 V_0 V_4 f 
+ |u|^3 (\chibar \pslash + \eta p^4) V_{4+C} V_0 V_4 f 
+ |u| \eta \pslash V_4 V_0 V_4 f
\big] \\  \notag
&\quad + \mathcal{D}(|u|^3 (\chibar \pslash + \eta \pslash)) |u| \eta \pslash V_{4+A} V_4 f \\  \notag
&\quad + |u|^3 (\chibar \pslash + \eta \pslash) \mathcal{D}(|u| \eta \pslash) V_{4+A} V_4 f \\  \notag
&\quad + |u|^3 (\chibar \pslash + \eta \pslash) |u| \eta \pslash 
\big[
(1 + |u| \eta \pslash) V_0 V_{4+A} V_4 f 
+ |u|^3 (\chibar \pslash + \eta p^4) V_{4+C} V_{4+A} V_4 f 
+ |u| \eta \pslash V_4 V_{4+A} V_4 f
\big] \\  \notag
&\quad + |u| \eta \pslash \mathcal{D}(|u| \eta \pslash) \\  
&\quad + |u| \eta \pslash 
\big[
(1 + |u| \eta \pslash) V_0 V_4 V_4 f 
+ |u|^3 (\chibar \pslash + \eta p^4) V_{4+C} V_4 V_4 f 
+ |u| \eta \pslash V_4 V_4 V_4 f
\big]
\end{align}
where one uses the fact 
\begin{align}
 e_{3}(p^{4})=\frac{\Bigg(2\chibar _{AB}-\gslash _{BC}e_{A}(b^{C})-\gslash_{AC}e_{B}(b^{C})\Bigg)p^{A}p^{B}}{2p^{3}}.   
\end{align}
First of all, note that  $\Gamma^{C}_{3B}\gslash_{AC}$ appears again. Moreover, note an important point: wherever $V_{4+A}$ appears, we absorb a factor of $|u|$ with it. For example, we write $V_{K}V_{4+A}f=\frac{1}{|u|}V_{K}(|u|V_{4+A}f)+V_{K}(\frac{1}{|u|})|u|V_{4+A}f$. In light of the estimates on the Ricci coefficients including $\Gamma^{A}_{3B}=\chibar^{A}_{B}-e_{A}(b^{B})$, the coefficients of \[\tilde{V}_{(1)}\tilde{V}_{(2)}f, \hspace{1mm} \tilde{V}_{(1)}\tilde{V}_{(2)}\tilde{V}_{(3)}f,\] with $\tilde{V}_i \in \tilde{V}$ for $i=1,2,3$ (recall the definition of $\tilde{V}$ in equation \eqref{tildeV}), are uniformly bounded by $\lesssim 1$--this is vital, as we shall see soon. The coefficient of the lower derivatives verify better decay estimates, except for the zeroth order term. This will be important and a distinguishing feature of the large data problem. Now we estimate $||\mathcal{D}^{2}T_{\alpha\beta}||_{L^{2}_{sc}(S_{u,\ubar})}$ for all spacetime indices $\alpha,\beta$. First, recall that 

\begin{align}
s_{2}(T_{44}) &= s_{2}(\snabla_{4} \tr\chi) = 0, \\
s_{2}(T_{33}) &= s_{2}(\snabla_{3} \tr\chibar) = 1 + 1 = 2, \\
s_{2}(T_{A4}) &= s_{2}(\snabla_{4} \eta) = \frac{1}{2}, \\
s_{2}(T_{A3}) &= s_{2}(\snabla_{3} \etabar) = 1 + \frac{1}{2} = \frac{3}{2}, \\
s_{2}(T_{43}) &= s_{2}(\snabla_{4} \omegabar) = 1, \\
s_{2}(\slashed{T}_{AB}) &= s_{2}(T_{AB}) = s_{2}(\snabla_{4} \chibarhat) = 1,
\end{align}
by means of the signature conservation of the transport equations for the Ricci coefficients. we estimate $\mathcal{D}^{2}T_{44}$ first:
\begin{align}
|\mathcal{D}^{2}T_{44}|
&\;\lesssim\;
    \int_{p} \Big[
        V_{0}V_{0}f
        + |u|\,V_{4+A}V_{0}f
        + V_{4}V_{0}f
        + V_{0}(|u|\,V_{4+A}f) \notag \\
&\qquad\quad
        + |u|^{2} V_{4+A}V_{4+B}f
        + |u|^{-1} V_{4}(|u|\,V_{4+A}f)
        + V_{0}V_{4}f
        + |u|\,V_{4+A}V_{4}f \notag \\
&\qquad\quad
        + |u|^{-2} V_{4}V_{4}f
        + |u|^{-1} (V_{0},\,V_{4},\,|u|\,V_{4+A})f
    \Big]
    \sqrt{\det\gslash}\,\frac{dp^{1}dp^{2}dp^{3}}{p^{3}} \notag \\
&\quad
    + \int_{P} |f|\,\sqrt{\det\gslash}\,\frac{dp^{1}dp^{2}dp^{3}}{p^{3}} \notag \\[1em]
&\;\lesssim\;
    |u|^{-1} \Big[
        \|V_{0}V_{0}f\|_{L^{2}(\mathcal{P}_{x})}
        + \||u|\,V_{4+A}V_{0}f\|_{L^{2}(\mathcal{P}_{x})}
        + \|V_{4}V_{0}f\|_{L^{2}(\mathcal{P}_{x})} \notag \\
&\qquad\quad
        + \|V_{0}(|u|\,V_{4+A}f)\|_{L^{2}(\mathcal{P}_{x})}
        + \||u|^{2}V_{4+A}V_{4+B}f\|_{L^{2}(\mathcal{P}_{x})} \notag \\
&\qquad\quad
        + \|V_{4}(|u|\,V_{4+A}f)\|_{L^{2}(\mathcal{P}_{x})}
        + \|V_{0}V_{4}f\|_{L^{2}(\mathcal{P}_{x})}
        + \||u|\,V_{4+A}V_{4}f\|_{L^{2}(\mathcal{P}_{x})} \notag \\
&\qquad\quad
        + \|V_{4}V_{4}f\|_{L^{2}(\mathcal{P}_{x})}
        + |u|^{-1} \|(V_{0},\,V_{4},\,|u|\,V_{4+A})f\|_{L^{2}(\mathcal{P}_{x})}
    \Big]
    + \frac{a}{|u|^{2}},
\end{align}
where we have used H\"older's inequality on mass-shell with the measure $\sqrt{\det\gslash}\frac{dp^{1}dp^{2}dp^{3}}{p^{3}}$, the estimate on $f$ i.e., $\sup_{\mathcal{P}|_{u=u_{\infty}}}f=\sup_{\mathcal{P}}f\lesssim a$
and the decay estimate on the momentum support. Therefore the $L^{2}(S_{u,\ubar})$ norm of the spacetime entity $|\mathcal{D}^{2}T_{44}|$ is estimated as
\begin{align} \notag
\|\mathcal{D}^{2}T_{44}\|_{L^{2}(S_{u,\ubar})}
&\lesssim |u|^{-1} \Big[ 
\|V_{0}V_{0}f\|_{L^{2}(S_{u,\ubar})L^{2}(\mathcal{P}_{x})} \\ \notag 
&\quad + \||u|\,V_{4+A}V_{0}f\|_{L^{2}(S_{u,\ubar})L^{2}(\mathcal{P}_{x})} \\ \notag 
&\quad + \|V_{4}V_{0}f\|_{L^{2}(S_{u,\ubar})L^{2}(\mathcal{P}_{x})} \\ \notag 
&\quad + \|V_{0}(|u|\,V_{4+A}f)\|_{L^{2}(S_{u,\ubar})L^{2}(\mathcal{P}_{x})} \\ \notag 
&\quad + \||u|^{2}V_{4+A}V_{4+B}f\|_{L^{2}(S_{u,\ubar})L^{2}(\mathcal{P}_{x})} \\ \notag 
&\quad + \|V_{4}(|u|\,V_{4+A}f)\|_{L^{2}(S_{u,\ubar})L^{2}(\mathcal{P}_{x})} \\ \notag 
&\quad + \|V_{0}V_{4}f\|_{L^{2}(S_{u,\ubar})L^{2}(\mathcal{P}_{x})} \\ \notag 
&\quad + \||u|\,V_{4+A}V_{4}f\|_{L^{2}(S_{u,\ubar})L^{2}(\mathcal{P}_{x})} \\
&\quad + \|V_{4}V_{4}f\|_{L^{2}(S_{u,\ubar})L^{2}(\mathcal{P}_{x})}
\Big] + \frac{a}{|u|},
\end{align}
where we have used the uniform first-order estimate from Proposition \ref{firstvlasov}, the second derivative estimates from Proposition \ref{prop:second_derivative_estimate}, and the third derivative estimates from Proposition \ref{thirdvlasov}. We thus arrive at
\begin{align}
||\mathcal{D}^{2}T_{44}||_{L^{2}(S_{u,\ubar})}\lesssim \frac{a}{|u|}. 
\end{align}
In  the scale-invariant norm, noting $s_{2}(\mathcal{D}^{2}T_{44})=s_{2}(T_{44})=0$, we have
\begin{align}
||\mathcal{D}^{2}T_{44}||_{\mathcal{L}^{2}_{(sc)}(S_{u,\ubar})}\lesssim \frac{a}{|u|}~or~\frac{|u|}{a}||\mathcal{D}^{2}T_{44}||_{L^{2}_{sc}(S_{u,\ubar})}\lesssim 1.      
\end{align}
The most important point to note here is that this estimate is sharp when $\mathcal{D}=|u|\snabla_{3}$. This is due to the fact that when $\mathcal{D}=|u|\snabla_{3}$ hits $T_{44}=\int_{\mathcal{P}_{x}}f\sqrt{\det\gslash}\frac{dp^{1}dp^{2}dp^{3}}{p^{3}}$, one generates the following terms 
\begin{align}
 \mathcal{D}T_{44}\sim \int_{\mathcal{P}_{x}}f \hsp |u|\tr\chibar\sqrt{\det\gslash}\frac{dp^{1}dp^{2}dp^{3}}{p^{3}}+\dots \hsp,   
\end{align}
and subsequently 
\begin{align}
 \mathcal{D}^{2}T_{44}\sim \int_{\mathcal{P}_{x}}f|u|\snabla_{3}(|u|\tr\chibar)\sqrt{\det\gslash}\frac{dp^{1}dp^{2}dp^{3}}{p^{3}}+\dots \hsp ,\\
 \mathcal{D}^{3}T_{44}\sim \int_{\mathcal{P}_{x}}f|u|\snabla_{3}\Bigg(|u|\snabla_{3}(|u|\tr\chibar)\Bigg)\sqrt{\det\gslash}\frac{dp^{1}dp^{2}dp^{3}}{p^{3}}+ \dots \hsp ,  
\end{align}
and by the estimates on $|\tr\chibar+\frac{2}{|u|}|\lesssim \frac{\Gamma}{|u|^{2}}$, zeroth order estimate on $f$ $\sup_{\mathcal{P}|_{u=u_{\infty}}}f=\sup_{\mathcal{P}}f=a$, and the decay of momentum support \eqref{zeroorderbootstrap1}-\eqref{zeroorderbootstrap2}we get  
\begin{align}
 |\mathcal{D}^{2}T_{44}|\lesssim \frac{a}{|u|^{2}}+\text{better decay},~|\mathcal{D}^{3}T_{44}|\lesssim \frac{a}{|u|^{2}}+\text{better decay}.      
\end{align}
This, in the $\mathcal{L}^{2}_{(sc)}(S_{u,\ubar})$--norm, yields 
\begin{align}
 ||\mathcal{D}^{2}T_{44}||_{L^{2}_{sc}(S_{u,\ubar})}\lesssim \frac{a}{|u|},~||\mathcal{D}^{3}T_{44}||_{L^{2}_{sc}(H_{u})}\lesssim \frac{a}{|u|} .  
\end{align}
This is the worst possible case since, if one of $\mathcal{D}$ is $a^{\frac{1}{2}}\nablasl$ or $\nabla_{4}$, then that would produce additional decay. The largeness of $f$ is reflected in this estimate due to its sharpness. It can not be improved further while keeping the  norm of $f$ fixed. Moving on, we wish to control the following norms 
\begin{align}
 ||(a^{\frac{1}{2}}\nablasl)^{2}T_{44}||_{L^{2}(S_{u,\ubar})}, ||(a^{\frac{1}{2}}\nablasl|u|\nabla_{3})T_{44}||_{L^{2}(S_{u,\ubar})},||(\snabla_{4})^{2}T_{44}||_{L^{2}(S_{u,\ubar})} ||(a^{\frac{1}{2}}\nablasl)^{3}T_{44}||_{L^{2}(H_{u})}, \dots \hsp .
\end{align}

\noindent For this, we compute $e_{A}e_{B}$ in terms of $V_{\mu}V_{\nu}$. First recall 
\begin{align}
e_A =& V_{(A)}  + \modu^2 \left( \frac{\Gammaslash^C_{AB}\hsp p^B}{p^3} + {\chibarhat_A}^C +\frac{{\chi_A}^C \hsp p^4}{p^3}\right) V_{(4+C)}  \notag \\ &+\frac{\modu^2}{2} (\tr\chibar+\frac{2}{\modu}) V_{(4+A)} +\left( \frac{\frac{1}{2}\chi_{AB}\hsp p^B}{p^3} -\etabar_A\right)(V_{(4)}+V_{(0)})  
\end{align}
and so 
\begin{align}
& e_{A}e_{D}f= \Bigg[V_{(A)}  + \modu^2 \left( \frac{\Gammaslash^C_{AB}\hsp p^B}{p^3} + {\chibarhat_A}^C +\frac{{\chi_A}^C \hsp p^4}{p^3}\right) V_{(4+C)}  \notag \\ &+\frac{\modu^2}{2} (\tr\chibar+\frac{2}{\modu}) V_{(4+A)} +\left( \frac{\frac{1}{2}\chi_{AB}\hsp p^B}{p^3} -\etabar_A\right)(V_{(4)}+V_{(0)})\Bigg]\Bigg[V_{(D)}  + \modu^2 \left( \frac{\Gammaslash^E_{DF}\hsp p^F}{p^3} + {\chibarhat_D}^E +\frac{{\chi_D}^E \hsp p^4}{p^3}\right) V_{(4+E)}  \notag \\ &+\frac{\modu^2}{2} (\tr\chibar+\frac{2}{\modu}) V_{(4+D)} +\left( \frac{\frac{1}{2}\chi_{DF}\hsp p^F}{p^3} -\etabar_D\right)(V_{(4)}+V_{(0)})\Bigg]f.   
\end{align}
We thus infer that the leading order term reads (the remaining lower order terms appear with higher decay and are so irrelevant in the analysis):
\begin{align}
e_{A}e_{D}f\sim V_{A}V_{D}f+\frac{a^{\frac{1}{2}}\Gamma}{|u|}V_{A}(|u|V_{4+D}f)+\frac{a^{\frac{1}{2}}\Gamma}{|u|^{2}}(V_{A}V_{4}f+V_{A}V_{0}f)+l.o.t.
\end{align}
We will encounter terms of the type $\Gammaslash \nablasl T_{44}$ while applying $\nablasl$ twice. This yields
\begin{align}
\big| \nablasl^{2} T_{44} \big|
&\lesssim \Bigg| 
    \int_{\mathcal{P}_{x}} \Bigg[ 
        V_{A} V_{D} f
        + \frac{a^{\frac{1}{2}} \Gamma}{|u|} V_{A} \big( |u| V_{4+D} f \big)
        + \frac{a^{\frac{1}{2}} \Gamma}{|u|^{2}} \big( V_{A} V_{4} f + V_{A} V_{0} f \big) 
    \Bigg] p^{3} p^{3} 
    \frac{\sqrt{\gslash} \, dp^{1} dp^{2} dp^{3}}{p^{3}} 
\Bigg| 
+ |\Gammaslash| \, \big| \nablasl T_{44} \big| \notag \\
&\lesssim |u|^{-1} \Bigg[ 
    \big\| V_{A} V_{D} f \big\|_{L^{2}(\mathcal{P}_{x})}
    + \frac{a^{\frac{1}{2}} \Gamma}{|u|} \big\| V_{A} \, |u| V_{4+D} f \big\|_{L^{2}(\mathcal{P}_{x})}
    + \frac{a^{\frac{1}{2}} \Gamma}{|u|^{2}} \big\| V_{A} V_{4} f \big\|_{L^{2}(\mathcal{P}_{x})} \notag \\
&\hspace{3cm}
    + \frac{a^{\frac{1}{2}} \Gamma}{|u|^{2}} \big\| V_{A} V_{0} f \big\|_{L^{2}(\mathcal{P}_{x})} 
\Bigg] \notag \\
&\quad + |u|^{-1} \Bigg(
    \big\| V f \big\|_{L^{2}(\mathcal{P}_{x})}
    + \big\| |u| V_{4+A} f \big\|_{L^{2}(\mathcal{P}_{x})}
    + \big\| V_{4} f \big\|_{L^{2}(\mathcal{P}_{x})}
    + \big\| V_{(0)} f \big\|_{L^{2}(\mathcal{P}_{x})}
\Bigg).
\end{align}
Here, we have used the fact that $|\Gammaslash-\overline{\Gammaslash}|\lesssim \frac{a^{\frac{1}{2}}}{|u|}$ and $|\overline{\Gammaslash}|\lesssim 1$ (note that these are component-wise norms rather than tensor norm) and first order estimate of $Vf$ from Proposition \ref{firstvlasov}. 
This leads to the norm estimate
\begin{align}
\left\| \nablasl^{2} T_{44} \right\|_{L^{2}(S_{u,\ubar})}
&:= \left\| 
    \sqrt{
        \gslash^{AC} \gslash^{BD} 
        \nablasl_{A} \nablasl_{B} T_{44} \,
        \nablasl_{C} \nablasl_{D} T_{44}
    }
\right\|_{L^{2}(S_{u,\ubar})} \notag \\[1em]
&\lesssim |u|^{-3} \Bigg[
    \left\| V_{A} V_{D} f \right\|_{L^{2}(S_{u,\ubar}) L^{2}(\mathcal{P}_{x})}
    + \frac{a^{\frac12} \Gamma}{|u|}
      \left\| V_{A} \, |u| V_{4+D} f \right\|_{L^{2}(S_{u,\ubar}) L^{2}(\mathcal{P}_{x})} \notag \\
&\quad + \frac{a^{\frac12} \Gamma}{|u|^{2}}
      \left\| V_{A} V_{4} f \right\|_{L^{2}(S_{u,\ubar}) L^{2}(\mathcal{P}_{x})}
    + \frac{a^{\frac12} \Gamma}{|u|^{2}}
      \left\| V_{A} V_{0} f \right\|_{L^{2}(S_{u,\ubar}) L^{2}(\mathcal{P}_{x})}
\Bigg] \notag \\[1em]
&\quad + |u|^{-3} \Bigg(
    \left\| V f \right\|_{L^{2}(S_{u,\ubar}) L^{2}(\mathcal{P}_{x})}
    + \left\| \, |u| V_{4+A} f \right\|_{L^{2}(S_{u,\ubar}) L^{2}(\mathcal{P}_{x})} \notag \\
&\quad\quad + \left\| V_{4} f \right\|_{L^{2}(S_{u,\ubar}) L^{2}(\mathcal{P}_{x})}
    + \left\| V_{(0)} f \right\|_{L^{2}(S_{u,\ubar}) L^{2}(\mathcal{P}_{x})}
\Bigg) \notag \\[1em]
&\lesssim \frac{a}{|u|^{3}}.
\end{align}
\end{proof}
\noindent Recall that $s_{2}(\nablasl^{2}T_{44})=1$ and therefore 
\begin{align}
 ||\nablasl^{2}T_{44}||_{\mathcal{L}^{2}_{(sc)}(S_{u,\ubar})}=\frac{|u|^{2}}{a}||\nablasl^{2}T_{44}||_{L^{2}(S_{u,\ubar})}\lesssim \frac{1}{|u|},   
\end{align}
which implies 
\begin{align}
||\mathcal{D}^{2}T_{44}||_{L^{2}_{sc}(S_{u,\ubar})}=||(a^{\frac{1}{2}}\nablasl)^{2} T_{44}||_{L^{2}_{sc}(S_{u,\ubar})}\lesssim \frac{a}{|u|}~or~\frac{|u|}{a} ||\mathcal{D}^{2}T_{44}||_{L^{2}_{sc}(S_{u,\ubar})}\lesssim 1. 
\end{align}
Similarly, one obtains 
\begin{align}
||\nablasl^{3}T_{44}||_{L^{2}(H_{u})}\lesssim |u|^{-4}+\frac{a}{|u|^{4}},   
\end{align}
hence 
\begin{align}
||(a^{\frac{1}{2}}\nablasl)^{3}T_{44}||_{L^{2}(H_{u})}\lesssim \frac{a^{\frac{3}{2}}}{|u|^{4}}+\frac{a^{\frac{5}{2}}}{|u|^{4}} .  
\end{align}
Noting $s_{2}(\nablasl^{3}T_{44})=\frac{3}{2}$, we immediately obtain 
\begin{align}
||(a^{\frac{1}{2}}\nablasl)^{3}T_{44}||_{L^{2}_{sc}(H_{u})}=\frac{|u|^{3}}{a^{\frac{3}{2}}}\Bigg(\frac{a^{\frac{3}{2}}}{|u|^{4}}+\frac{a^{\frac{5}{2}}}{|u|^{4}}\Bigg)\lesssim \frac{1}{|u|}+\frac{a}{|u|}~or~\frac{|u|}{a}||\mathcal{D}^{3}T_{44}||_{L^{2}_{sc}(H_{u})}\lesssim 1.
\end{align}

\noindent Now for the $\snabla_{4}$ derivatives, compute 
\begin{align} \notag 
 e_{4}e_{4}f=\Bigg[ V_{(3)} + \modu^2\left( \frac{{\chi_B}^A p^B}{p^3} + 2\hsp \etabar^A \right)V_{(4+A)}\Bigg]\Bigg[ V_{(3)} + \modu^2\left( \frac{{\chi_C}^D p^C}{p^3} + 2\hsp \etabar^C \right)V_{(4+D)} \Bigg]f\\
 =V_{3}V_{3}f+\frac{a^{\frac{1}{2}}\Gamma}{|u|}V_{3}|u|V_{4+A}f+\frac{a\Gamma^{2}}{|u|^{2}}(|u|V_{4+A})(|u|V_{4+D})f+l.o.t
\end{align}
and note that the lower order terms $l.o.t$ involve terms with one derivative of $f$ and the coefficients enjoy higher decay since they contain one derivative of the Ricci coefficients. Therefore, we do not repeat these terms; rather, we pay attention to the top-order terms. Explicit computations yield 
\begin{align}
|\snabla^{2}_{4}\Te_{44}| 
&\lesssim \Bigg| \int_{\mathcal{P}_{x}} 
\Bigg[ 
    V_{3} V_{3} f 
    + \frac{a^{\frac{1}{2}} \Gamma}{|u|} V_{3} \big( |u| V_{4+A} f \big) 
    + \frac{a \Gamma^{2}}{|u|^{2}} \big( |u| V_{4+A} \big) \big( |u| V_{4+D} \big) f 
\Bigg] p^{3} p^{3} \frac{\sqrt{\det \gslash} \, dp^{1} dp^{2} dp^{3}}{p^{3}} \Bigg| \notag \\
&\quad + \int_{\mathcal{P}_{x}} \big( |f \nabla_{4} \tr\chi| + |f \tr\chi \tr\chi| \big) \frac{\sqrt{\det \gslash} \, dp^{1} dp^{2} dp^{3}}{p^{3}} + \text{l.o.t.} \notag \\
&\lesssim |u|^{-1} \Bigg[
    \| V_{3} V_{3} f \|_{L^{2}(\mathcal{P}_{x})} 
    + \frac{a^{\frac{1}{2}} \Gamma}{|u|} \| V_{3} (|u| V_{4+A} f) \|_{L^{2}(\mathcal{P}_{x})} 
    + \frac{a \Gamma^{2}}{|u|^{2}} \| (|u| V_{4+A})(|u| V_{4+D}) f \|_{L^{2}(\mathcal{P}_{x})} 
\Bigg],
\end{align}
and so the following holds
\begin{align}
\| \snabla^{2}_{4} T_{44} \|_{L^{2}(S_{u, \ubar})} 
&\lesssim |u|^{-1} \Bigg[
    \| V_{3} V_{3} f \|_{L^{2}(S_{u, \ubar}) L^{2}(\mathcal{P}_{x})} 
    + \frac{a^{\frac{1}{2}} \Gamma}{|u|} \| V_{3} (|u| V_{4+A} f) \|_{L^{2}(S_{u, \ubar}) L^{2}(\mathcal{P}_{x})} \notag \\
&\quad + \frac{a \Gamma^{2}}{|u|^{2}} \| (|u| V_{4+A})(|u| V_{4+D}) f \|_{L^{2}(S_{u, \ubar}) L^{2}(\mathcal{P}_{x})} 
\Bigg] \notag \\
&\lesssim \frac{a}{|u|}.
\end{align}
Consequently, 
\begin{align}
||\snabla^{2}_{4}T_{44}||_{L^{2}_{sc}(S_{u,\ubar})}\lesssim \frac{a}{|u|} ~or ~  \frac{|u|}{a}||\mathcal{D}^{2}T_{44}||_{L^{2}_{sc}(S_{u,\ubar})}\lesssim 1 
\end{align}
since $s_{2}(\snabla^{2}_{4}T_{44})=0$. Next, we move on to estimating the remaining stress-energy tensor components.\\
\vspace{3mm}

\noindent  \textbf{$T_{A3}$:} First consider $\mathcal{D}=|u|\nabla_{3}$. We compute
\begin{align}
|\mathcal{D}^{2}T_{A3}| \lesssim \; & 
\int_{\mathcal{P}_{x}} \Bigg[
    V_{0}V_{0}f
    + |u| V_{4+A} V_{0} f
    + V_{4} V_{0} f
    + V_{0}(|u| V_{4+A} f)
    + |u|^{2} V_{4+A} V_{4+B} f \notag \\
&\quad + |u|^{-1} V_{4}(|u| V_{4+A} f)
    + V_{0} V_{4} f
    + |u| V_{4+A} V_{4} f
    + |u|^{-2} V_{4} V_{4} f
\Bigg] |p^{4} p_{A}| \sqrt{\det \gslash} \frac{dp^{1} dp^{2} dp^{3}}{p^{3}} \notag \\
&+ \int_{\mathcal{P}_{x}} |f| \, |p^{4} p_{A}| \sqrt{\det \gslash} \frac{dp^{1} dp^{2} dp^{3}}{p^{3}}
+ \text{l.o.t.} \notag \\
\lesssim \; & |u|^{-3} \Bigg[
    \| V_{0} V_{0} f \|_{L^{2}(\mathcal{P}_{x})} 
    + \| |u| V_{4+A} V_{0} f \|_{L^{2}(\mathcal{P}_{x})} 
    + \| V_{4} V_{0} f \|_{L^{2}(\mathcal{P}_{x})} 
    + \| V_{0}(|u| V_{4+A} f) \|_{L^{2}(\mathcal{P}_{x})} \notag \\ \notag
&\quad + \| |u|^{2} V_{4+A} V_{4+B} f \|_{L^{2}(\mathcal{P}_{x})} 
    + \| V_{4}(|u| V_{4+A} f) \|_{L^{2}(\mathcal{P}_{x})} 
    + \| V_{0} V_{4} f \|_{L^{2}(\mathcal{P}_{x})} 
    + \| |u| V_{4+A} V_{4} f \|_{L^{2}(\mathcal{P}_{x})}\\
    &\quad 
    + \| V_{4} V_{4} f \|_{L^{2}(\mathcal{P}_{x})} 
\Bigg] \notag \\
&+ \frac{a}{|u|^{4}},
\end{align}
where the lower order terms $l.o.t$ involve first derivatives as well as terms with higher decay structure produced as a result of $e_{3}$ acting on $p^{4}$. Using the definition of the $L^{2}(S_{u,\ubar})$ norm, we obtain 
\begin{align}
\|\mathcal{D}^{2}T_{A3}\|_{L^{2}(S_{u,\ubar})} 
\lesssim &\; |u|^{-4} \Bigg[
    \| V_{0} V_{0} f \|_{L^{2}(S_{u,\ubar}) L^{2}(\mathcal{P}_{x})} 
    + \| |u| V_{4+A} V_{0} f \|_{L^{2}(S_{u,\ubar}) L^{2}(\mathcal{P}_{x})} \notag\\
&\quad + \| V_{4} V_{0} f \|_{L^{2}(S_{u,\ubar}) L^{2}(\mathcal{P}_{x})} 
    + \| V_{0} (|u| V_{4+A} f) \|_{L^{2}(S_{u,\ubar}) L^{2}(\mathcal{P}_{x})} \notag\\
&\quad + \| |u|^{2} V_{4+A} V_{4+B} f \|_{L^{2}(S_{u,\ubar}) L^{2}(\mathcal{P}_{x})} 
    + \| V_{4} |u| V_{4+A} f \|_{L^{2}(S_{u,\ubar}) L^{2}(\mathcal{P}_{x})} \notag\\
&\quad + \| V_{0} V_{4} f \|_{L^{2}(S_{u,\ubar}) L^{2}(\mathcal{P}_{x})} 
    + \| |u| V_{4+A} V_{4} f \|_{L^{2}(S_{u,\ubar}) L^{2}(\mathcal{P}_{x})} \notag\\
&\quad + \| V_{4} V_{4} f \|_{L^{2}(S_{u,\ubar}) L^{2}(\mathcal{P}_{x})}
\Bigg] + \frac{a}{|u|^{5}}
\end{align}
and by the estimates on Vlasov 
\begin{align}
||\mathcal{D}^{2}T_{A3}||_{L^{2}(S_{u,\ubar})}\lesssim \frac{a}{|u|^{4}}  
\end{align}
and similarly, for the third derivatives 
\begin{align}
||\mathcal{D}^{3}T_{A3}||_{L^{2}(H_{u})} \lesssim \frac{a}{|u|^{4}}    
\end{align}
or in scale invariant norm noting $s_{2}(\mathcal{D}^{2}T_{A3})=s_{2}(T_{A3})=\frac{3}{2}$
\begin{align}
||\mathcal{D}^{2}T_{A3}||_{L^{2}_{sc}(S_{u,\ubar})}\lesssim \frac{1}{a^{\frac{3}{2}}|u|}+\frac{|u|^{3}}{a^{\frac{3}{2}}}\frac{a}{|u|^{4}}\lesssim \frac{1}{a^{\frac{3}{2}}|u|}+\frac{1}{a^{\frac{1}{2}}|u|},\end{align}
\begin{align}
||\mathcal{D}^{3}T_{A3}||_{L^{2}_{sc}(H_{u})} \lesssim \frac{1}{a^{\frac{3}{2}}|u|}+\frac{1}{a^{\frac{1}{2}}|u|}, 
\end{align}
or 
\begin{align}
\frac{|u|}{a} ||\mathcal{D}^{2}T_{A3}||_{L^{2}_{sc}(S_{u,\ubar})}\lesssim \frac{1}{a^{\frac{3}{2}}} \lesssim 1,~\frac{|u|}{a}||\mathcal{D}^{3}T_{A3}||_{L^{2}_{sc}(H_{u})}\lesssim \frac{1}{a^{\frac{3}{2}}} \lesssim 1.   
\end{align}

\noindent Next, consider
\begin{align} \notag
|\nablasl^{2}T_{E3}|=&|\int_{\mathcal{P}_{x}}\Bigg[V_{A}V_{D}f+\frac{a^{\frac{1}{2}}\Gamma}{|u|}V_{A}(|u|V_{4+D}f)+\frac{a^{\frac{1}{2}}\Gamma}{|u|^{2}}(V_{A}V_{4}f+V_{A}V_{0}f \Bigg]p_{E}p^{4}\frac{\sqrt{\gslash}dp^{1}dp^{2}dp^{3}}{p^{3}}+\Gammaslash\Gammaslash T_{E3}|+l.o.t\\
\lesssim& |u|^{-3}\Bigg[||V_{A}V_{D}f||_{L^{2}(\mathcal{P}_{x})}+\frac{a^{\frac{1}{2}}\Gamma}{|u|}||V_{A}|u|V_{4+D}f||_{L^{2}(\mathcal{P}_{x})}+\frac{a^{\frac{1}{2}}\Gamma}{|u|^{2}}||V_{A}V_{4}f||_{L^{2}(\mathcal{P}_{x})}+\frac{a^{\frac{1}{2}}\Gamma}{|u|^{2}}||V_{A}V_{0}f||_{L^{2}(\mathcal{P}_{x})} \Bigg]\\
+&\frac{a}{|u|^{4}}.
\end{align}
Passing to the $L^2(S_{u,\ubar})-$norm, we obtain
\begin{align}
 \notag ||\nablasl^{2}T_{E3}||_{L^{2}(S_{u,\ubar})}:=&||\sqrt{\gslash^{AC}\gslash^{BD}\nablasl_{A}\nablasl_{B}T_{E3}\nablasl_{C}\nablasl_{D}T_{F3}\gslash^{EF}}||_{L^{2}(S_{u,\ubar})}\\ \notag 
 \lesssim& |u|^{-6}\Bigg[||V_{A}V_{D}f||_{L^{2}(S_{u,\ubar})L^{2}(\mathcal{P}_{x})}+\frac{a^{\frac{1}{2}}\Gamma}{|u|}||V_{A}|u|V_{4+D}f||_{L^{2}(S_{u,\ubar})L^{2}(\mathcal{P}_{x})}\\\notag +&\frac{a^{\frac{1}{2}}\Gamma}{|u|^{2}}||V_{A}V_{4}f||_{L^{2}(S_{u,\ubar})L^{2}(\mathcal{P}_{x})}+\frac{a^{\frac{1}{2}}\Gamma}{|u|^{2}}||V_{A}V_{0}f||_{L^{2}(S_{u,\ubar})L^{2}(\mathcal{P}_{x})} \Bigg]+\frac{a}{|u|^{6}}\\
 \lesssim& \frac{a}{|u|^{6}},
\end{align}    
which, after noting that $s_{2}(\nablasl^{2}T_{E3})=1+1+\frac{1}{2}=\frac{5}{2}$ yields
\begin{align}
 ||(a^{\frac{1}{2}}\nablasl)^{2}T_{E3}||_{L^{2}_{sc}(S_{u,\ubar})}\lesssim \frac{1}{a^{\frac{1}{2}}|u|}, \hspace{3mm} \text{or} \hspace{3mm}
 \frac{|u|}{a}||\mathcal{D}^{2}T_{E3}||_{L^{2}_{sc}(S_{u,\ubar})}\lesssim \frac{1}{a^{\frac{3}{2}}}\lesssim 1.
\end{align}
Similarly, for the third derivatives, the same calculations yield 
\begin{align}
\frac{|u|}{a}||\mathcal{D}^{3}T_{E3}||_{L^{2}_{sc}(H_{u})}\lesssim \frac{1}{a^{\frac{3}{2}}}\lesssim 1.
\end{align}
Moving on to $\snabla_4^2 T_{E3}$,
\begin{align} \notag 
|\snabla^{2}_{4}T_{E3}|=&|\int_{\mathcal{P}_{x}}\Bigg[V_{3}V_{3}f+\frac{a^{\frac{1}{2}}\Gamma}{|u|}V_{3}|u|V_{4+A}f+\frac{a\Gamma^{2}}{|u|^{2}}(|u|V_{4+A})(|u|V_{4+D})f\Bigg]p_{E}p^{4}\frac{\sqrt{\gslash}dp^{1}dp^{2}dp^{3}}{p^{3}}+l.o.t\\
\lesssim& |u|^{-3}\Bigg[||V_{3}V_{3}f||_{L^{2}(\mathcal{P}_{x})}+\frac{a^{\frac{1}{2}}||\Gamma}{|u|}V_{3}(|u|V_{4+A}f)||_{L^{2}(\mathcal{P}_{x})}+\frac{a\Gamma^{2}}{|u|^{2}}||(|u|V_{4+A})(|u|V_{4+D})f||_{L^{2}(\mathcal{P}_{x})} \Bigg], 
\end{align}
hence
\begin{align}
||\snabla^{2}_{4}T_{E3}||_{L^{2}(S_{u,\ubar})}=&||\sqrt{\snabla^{2}_{4}T_{E3}\snabla^{2}_{4}T_{F3}\gslash^{EF}}||_{L^{2}(S_{u,\ubar})}\\ \notag
\lesssim&  |u|^{-4}\Bigg[||V_{3}V_{3}f||_{L^{2}(S_{u,\ubar})L^{2}(\mathcal{P}_{x})}+\frac{a^{\frac{1}{2}}\Gamma}{|u|}V_{3}(|u|V_{4+A}f)||_{L^{2}(S_{u,\ubar})L^{2}(\mathcal{P}_{x})}\\ \notag &\hspace{2cm}+\frac{a\Gamma^{2}}{|u|^{2}}||(|u|V_{4+A})(|u|V_{4+D})f||_{L^{2}(S_{u,\ubar})L^{2}(\mathcal{P}_{x})} \Bigg]  
\lesssim \frac{a}{|u|^{4}}.
\end{align}
Noting that $s_{2}(\snabla^{2}_{4}T_{E3})=1+\frac{1}{2}=\frac{3}{2}$, the scale-invariant version becomes
\begin{align}
\frac{|u|}{a}||\snabla^{2}_{4}T_{E3}||_{L^{2}_{sc}(S_{u,\ubar})}\lesssim \frac{1}{a^{\frac{3}{2}}}\lesssim 1.   
\end{align}
In an exactly similar fashion, one obtains 
\begin{align}
\frac{|u|}{a}||\nabla^{3}_{4}T_{E3}||_{L^{2}_{sc}(H_{u})}\lesssim \frac{1}{a^{\frac{3}{2}}}\lesssim 1.
\end{align}
Similarly, one obtains $\frac{|u|}{a}||\mathcal{D}^{3}T_{E3}||_{L^{2}_{sc}(H_{u})}\lesssim 1$ and $\frac{|u|}{a}||\mathcal{D}^{2}T_{E3}||_{L^{2}_{sc}(S_{u,\ubar})}\lesssim 1$  for all $\mathcal{D}$.

\vspace{3mm}

\noindent For \textbf{$T_{4A}$:} 
\begin{align} \notag 
|\mathcal{D}^{2}T_{4A}|\lesssim& \int_{\mathcal{P}_{x}}\Bigg[V_{0}V_{0}f+|u|V_{4+A}V_{0}f+V_{4}V_{0}f+V_{0}(|u|V_{4+A}f)+|u|^{2}V_{4+A}V_{4+B}f\\ \notag +&|u|^{-1}V_{4}|u|V_{4+A}f+V_{0}V_{4}f+|u|V_{4+A}V_{4}f+ |u|^{-2}V_{4}V_{4}f\Bigg]\sqrt{\det\gslash}\frac{dp^{1}dp^{2}dp^{3}}{p^{3}}\\ \notag +&\int_{\mathcal{P}_{x}}|f|\sqrt{\det\gslash}\frac{dp^{1}dp^{2}dp^{3}}{p^{3}}
+l.o.t\\ \notag
\lesssim& |u|^{-1}\Bigg[||V_{0}V_{0}f||_{L^{2}(\mathcal{P}_{x})}+|||u|V_{4+A}V_{0}f||_{L^{2}(\mathcal{P}_{x})}+||V_{4}V_{0}f||_{L^{2}(\mathcal{P}_{x})}+||V_{0}(|u|V_{4+A}f)||_{L^{2}(\mathcal{P}_{x})}\\ &\hspace{2.5cm}+|||u|^{2}V_{4+A}V_{4+B}f||_{L^{2}(\mathcal{P}_{x})}+||V_{4}|u|V_{4+A}f||_{L^{2}(\mathcal{P}_{x})}+||V_{0}V_{4}f||_{L^{2}(\mathcal{P}_{x})}\\ \notag &\hspace{6cm}+|u|V_{4+A}V_{4}f||_{L^{2}(\mathcal{P}_{x})}+ ||V_{4}V_{4}f||_{L^{2}(\mathcal{P}_{x})}\Bigg]+\frac{a}{|u|^{2}}.
\end{align}Passing to the $L^2(S_{u,\ubar})$ norms yields
\begin{align}
\notag ||\mathcal{D}^{2}T_{4A}||_{L^{2}(S_{u,\ubar})}\lesssim& |u|^{-2}\Bigg[||V_{0}V_{0}f||_{L^{2}(S_{u,\ubar})L^{2}(\mathcal{P}_{x})}+|||u|V_{4+A}V_{0}f||_{L^{2}(S_{u,\ubar})L^{2}(\mathcal{P}_{x})}\\ \notag +&||V_{4}V_{0}f||_{L^{2}(S_{u,\ubar})L^{2}(\mathcal{P}_{x})}+||V_{0}(|u|V_{4+A}f)||_{L^{2}(S_{u,\ubar})L^{2}(\mathcal{P}_{x})}\\ \notag +&|||u|^{2}V_{4+A}V_{4+B}f||_{L^{2}(S_{u,\ubar})L^{2}(\mathcal{P}_{x})}\\  \notag +&||V_{4}|u|V_{4+A}f||_{L^{2}(S_{u,\ubar})L^{2}(\mathcal{P}_{x})}+||V_{0}V_{4}f||_{L^{2}(S_{u,\ubar})L^{2}(\mathcal{P}_{x})}\\    +&|u|V_{4+A}V_{4}f||_{L^{2}(S_{u,\ubar})L^{2}(\mathcal{P}_{x})}+ ||V_{4}V_{4}f||_{L^{2}(S_{u,\ubar})L^{2}(\mathcal{P}_{x})}\Bigg]+\frac{a}{|u|^{2}}.
\end{align} For the commutator $(\mathcal{D},\snabla_4)\De^2 \Te_{4A}$, we have
\begin{align}
 ||(\mathcal{D},\snabla_{4})\mathcal{D}^{2}T_{4A}||_{L^{2}(H_{u})}\lesssim  |u|^{-2}||\sum_{i_{1}+i_{2}+i_{3}+i_{4}+i_{5}=3}||(V_{0})^{i_{1}}(|u|V_{4+A})^{i_{2}}(V_{4})^{i_{3}}(V_{A})^{i_{4}}(V_{3})^{i_{5}}f||_{L^{2}(H_{u})L^{2}(\mathcal{P}_{x})}+\frac{a}{|u|^{2}},  
\end{align}
hence by the estimates on Vlasov from Proposition \ref{prop:second_derivative_estimate}, we have
\begin{align}
||\mathcal{D}^{2}T_{4A}||_{L^{2}(S_{u,\ubar})}\lesssim \frac{a}{|u|^{2}},~||(\mathcal{D},\snabla_{4})\mathcal{D}^{2}T_{4A}||_{L^{2}(H_{u})}\lesssim \frac{a}{|u|^{2}}    
\end{align}
or, in scale invariant norms, noting $s_{2}(\mathcal{D}^{2}T_{4A})=s_{2}(T_{4A})=\frac{1}{2}$,
\begin{align}
\frac{|u|}{a}||\mathcal{D}^{2}T_{4A}||_{L^{2}_{sc}(S_{u,\ubar})}\lesssim \frac{1}{a^{\frac{1}{2}}}\lesssim 1,~\frac{|u|}{a}||(\mathcal{D},\snabla_{4})\mathcal{D}^{2}T_{4A}||_{L^{2}_{sc}(H_{u})}\lesssim \frac{1}{a^{\frac{1}{2}}}\lesssim 1.
\end{align}
For $\snabla^2 T_{4E}$, we have 
\begin{align} \notag 
&|\nablasl^{2}T_{4E}|\\ \notag \lesssim& |\int_{\mathcal{P}_{x}}\Bigg[V_{A}V_{D}f+\frac{a^{\frac{1}{2}}\Gamma}{|u|}V_{A}(|u|V_{4+D}f)+\frac{a^{\frac{1}{2}}\Gamma}{|u|^{2}}(V_{A}V_{4}f+V_{A}V_{0}f \Bigg]p_{E}p^{3}\frac{\sqrt{\gslash}dp^{1}dp^{2}dp^{3}}{p^{3}}|+|\Gammaslash\Gammaslash T_{4E}|+l.o.t\\ \notag
\lesssim& |u|^{-1}\Bigg[||V_{A}V_{D}f||_{L^{2}(\mathcal{P}_{x})}+\frac{a^{\frac{1}{2}}\Gamma}{|u|}||V_{A}|u|V_{4+D}f||_{L^{2}(\mathcal{P}_{x})}+\frac{a^{\frac{1}{2}}\Gamma}{|u|^{2}}||V_{A}V_{4}f||_{L^{2}(\mathcal{P}_{x})}+\frac{a^{\frac{1}{2}}\Gamma}{|u|^{2}}||V_{A}V_{0}f||_{L^{2}(\mathcal{P}_{x})} \Bigg]+ \frac{a}{|u|^{2}}.
\end{align}
This implies 
\begin{align} \notag 
 ||\nablasl^{2}T_{4E}||_{L^{2}(S_{u,\ubar})}=&||\sqrt{\gslash^{AC}\gslash^{BD}\nablasl_{A}\nablasl_{B}T_{4E}\nablasl_{C}\nablasl_{D}T_{4F}\gslash^{EF}}||_{L^{2}(S_{u,\ubar})}\\
 \lesssim& |u|^{-4}\Bigg[||V_{A}V_{D}f||_{L^{2}(S_{u,\ubar})L^{2}(\mathcal{P}_{x})}+\frac{a^{\frac{1}{2}}\Gamma}{|u|}||V_{A}|u|V_{4+D}f||_{L^{2}(S_{u,\ubar})L^{2}(\mathcal{P}_{x})}\\+&\frac{a^{\frac{1}{2}}\Gamma}{|u|^{2}}||V_{A}V_{4}f||_{L^{2}(S_{u,\ubar})L^{2}(\mathcal{P}_{x})}+\frac{a^{\frac{1}{2}}\Gamma}{|u|^{2}}||V_{A}V_{0}f||_{L^{2}(S_{u,\ubar})L^{2}(\mathcal{P}_{x})} \Bigg]+\frac{a}{|u|^{4}}
 \lesssim \frac{a}{|u|^{4}}
\end{align}    
and similarly
\begin{align}
 ||\nablasl^{3}T_{4E}||_{L^{2}(H_{u})}\lesssim    |u|^{-5}||\sum_{i_{1}+i_{2}+i_{3}+i_{4}+i_{5}=3}||(V_{0})^{i_{1}}(|u|V_{4+A})^{i_{2}}(V_{4})^{i_{3}}(V_{A})^{i_{4}}(V_{3})^{i_{5}}f||_{L^{2}(H_{u})L^{2}(\mathcal{P}_{x})}+\frac{a}{|u|^{5}}\lesssim \frac{a}{|u|^{5}}.
\end{align}
Since $s_{2}(\nablasl^{j}T_{4E})=1+\frac{j}{2}=\frac{3}{2}$,
\begin{align}
\frac{|u|}{a}||(a^{\frac{1}{2}}\nablasl)^{2}T_{4E}||_{L^{2}_{sc}(S_{u,\ubar})}\lesssim \frac{1}{a^{\frac{1}{2}}}\lesssim 1,
\end{align}
and similarly 
\begin{align}
\frac{|u|}{a}||(a^{\frac{1}{2}}\nablasl)^{3}T_{4E}||_{L^{2}_{sc}(H_{u})}\lesssim \frac{1}{a^{\frac{1}{2}}}\lesssim 1.
\end{align}The rest of the derivatives $\mathcal{D}$ are handled similarly. 

\vspace{3mm}

\noindent For \textbf{$T_{AB}$:} First, consider $\mathcal{D}=|u|\snabla_{3}$ and observe
\begin{align} \notag  
|\mathcal{D}^{2}T_{AB}|\lesssim& \Bigg|\int_{\mathcal{P}_{x}}\Bigg[V_{0}V_{0}f+|u|V_{4+A}V_{0}f+V_{4}V_{0}f+V_{0}(|u|V_{4+A}f)+|u|^{2}V_{4+A}V_{4+B}f\\ \notag  +&|u|^{-1}V_{4}|u|V_{4+A}f+V_{0}V_{4}f+|u|V_{4+A}V_{4}f+ |u|^{-2}V_{4}V_{4}f\Bigg]p_{A}p_{B}\sqrt{\det\gslash}\frac{dp^{1}dp^{2}dp^{3}}{p^{3}}\Bigg|\\ \notag  +&\int_{\mathcal{P}_{x}}|fp_{A}p_{B}|\sqrt{\det\gslash}\frac{dp^{1}dp^{2}dp^{3}}{p^{3}}
+l.o.t.\\ \notag  
\lesssim& |u|^{-1}\Bigg[||V_{0}V_{0}f||_{L^{2}(\mathcal{P}_{x})}+|||u|V_{4+A}V_{0}f||_{L^{2}(\mathcal{P}_{x})}+||V_{4}V_{0}f||_{L^{2}(\mathcal{P}_{x})}+||V_{0}(|u|V_{4+A}f)||_{L^{2}(\mathcal{P}_{x})}\\ \notag  +&|||u|^{2}V_{4+A}V_{4+B}f||_{L^{2}(\mathcal{P}_{x})}\\ \notag  +&||V_{4}|u|V_{4+A}f||_{L^{2}(\mathcal{P}_{x})}+||V_{0}V_{4}f||_{L^{2}(\mathcal{P}_{x})}+|u|V_{4+A}V_{4}f||_{L^{2}(\mathcal{P}_{x})}+ ||V_{4}V_{4}f||_{L^{2}(\mathcal{P}_{x})}\Bigg]+\frac{a}{|u|^{2}}.
\end{align}Therefore,
\begin{align}  \notag  
||\mathcal{D}^{2}T_{AB}||_{L^{2}(S_{u,\ubar})}\lesssim &|u|^{-3}\Bigg[||V_{0}V_{0}f||_{L^{2}(S_{u,\ubar})L^{2}(\mathcal{P}_{x})}+|||u|V_{4+A}V_{0}f||_{L^{2}(S_{u,\ubar})L^{2}(\mathcal{P}_{x})}\\ \notag  +&||V_{4}V_{0}f||_{L^{2}(S_{u,\ubar})L^{2}(\mathcal{P}_{x})}+||V_{0}(|u|V_{4+A}f)||_{L^{2}(S_{u,\ubar})L^{2}(\mathcal{P}_{x})}\\ \notag  +&|||u|^{2}V_{4+A}V_{4+B}f||_{L^{2}(S_{u,\ubar})L^{2}(\mathcal{P}_{x})}\\ \notag  +&||V_{4}|u|V_{4+A}f||_{L^{2}(S_{u,\ubar})L^{2}(\mathcal{P}_{x})}+||V_{0}V_{4}f||_{L^{2}(S_{u,\ubar})L^{2}(\mathcal{P}_{x})}\\  +&|u|V_{4+A}V_{4}f||_{L^{2}(S_{u,\ubar})L^{2}(\mathcal{P}_{x})}+ ||V_{4}V_{4}f||_{L^{2}(S_{u,\ubar})L^{2}(\mathcal{P}_{x})}\Bigg]+\frac{a}{|u|^{3}}.    
\end{align}
Similarly, for the three derivatives 
\begin{align}
 ||\mathcal{D}^{3}T_{AB}||_{L^{2}(H_{u})}\lesssim |u|^{-3}||\sum_{i_{1}+i_{2}+i_{3}+i_{4}+i_{5}=3}||(V_{0})^{i_{1}}(|u|V_{4+A})^{i_{2}}(V_{4})^{i_{3}}(V_{A})^{i_{4}}(V_{3})^{i_{5}}f||_{L^{2}(H_{u})L^{2}(\mathcal{P}_{x})}+\frac{a}{|u|^{3}}   
\end{align}
and by the estimates on Vlasov from proposition \ref{prop:second_derivative_estimate}
\begin{align}
||\mathcal{D}^{2}T_{AB}||_{L^{2}(S_{u,\ubar})}\lesssim \frac{a}{|u|^{3}},~||\mathcal{D}^{3}T_{AB}||_{L^{2}(H_{u})}\lesssim \frac{a}{|u|^{3}}    
\end{align}
or in scale invariant norms, noting $s_{2}(\mathcal{D}^{2}T_{AB})=s_{2}(T_{AB})=1$,
\begin{align}
\frac{|u|}{a}||\mathcal{D}^{2}T_{AB}||_{L^{2}_{sc}(S_{u,\ubar})},\frac{|u|}{a}||\mathcal{D}^{3}T_{AB}||_{L^{2}_{sc}(S_{u,\ubar})}\lesssim \frac{1}{a}\lesssim 1. 
\end{align}
Next we consider $\mathcal{D}=a^{\frac{1}{2}}\nablasl$. Note 
\begin{align}
\notag &|\nablasl^{2}T_{EF}|\\ \lesssim& \Bigg|\int_{\mathcal{P}_{x}}\Bigg[V_{A}V_{D}f+\frac{a^{\frac{1}{2}}\Gamma}{|u|}V_{A}(|u|V_{4+D}f)+\frac{a^{\frac{1}{2}}\Gamma}{|u|^{2}}(V_{A}V_{4}f+V_{A}V_{0}f \Bigg]p_{E}p_{F}\frac{\sqrt{\gslash}dp^{1}dp^{2}dp^{3}}{p^{3}}\Bigg|\notag +|\Gammaslash\hsp \Gammaslash \hsp \Te_{EF}|+l.o.t.\\
\lesssim& |u|^{-1}\Bigg[||V_{A}V_{D}f||_{L^{2}(\mathcal{P}_{x})}+\frac{a^{\frac{1}{2}}\Gamma}{|u|}||V_{A}|u|V_{4+D}f||_{L^{2}(\mathcal{P}_{x})}+\frac{a^{\frac{1}{2}}\Gamma}{|u|^{2}}||V_{A}V_{4}f||_{L^{2}(\mathcal{P}_{x})}+\frac{a^{\frac{1}{2}}\Gamma}{|u|^{2}}||V_{A}V_{0}f||_{L^{2}(\mathcal{P}_{x})} \Bigg]
+\frac{a}{|u|^{2}},
\end{align}
hence
\begin{align} \notag 
 &||\nablasl^{2}T_{EF}||_{L^{2}(S_{u,\ubar})}:=||\sqrt{\gslash^{AC}\gslash^{BD}\nablasl_{A}\nablasl_{B}T_{HE}\nablasl_{C}\nablasl_{D}T_{IF}\gslash^{HI}\gslash^{EF}}||_{L^{2}(S_{u,\ubar})}\\ \notag
 \lesssim& |u|^{-5}\Bigg[||V_{A}V_{D}f||_{L^{2}(S_{u,\ubar})L^{2}(\mathcal{P}_{x})}+\frac{a^{\frac{1}{2}}\Gamma}{|u|}||V_{A}|u|V_{4+D}f||_{L^{2}(S_{u,\ubar})L^{2}(\mathcal{P}_{x})}\\   +&\frac{a^{\frac{1}{2}}\Gamma}{|u|^{2}}||V_{A}V_{4}f||_{L^{2}(S_{u,\ubar})L^{2}(\mathcal{P}_{x})}+\frac{a^{\frac{1}{2}}\Gamma}{|u|^{2}}||V_{A}V_{0}f||_{L^{2}(S_{u,\ubar})L^{2}(\mathcal{P}_{x})} \Bigg]+\frac{a}{|u|^{5}}
 \lesssim \frac{a}{|u|^{5}},
\end{align}    
and similarly 
\begin{align}
 &||\nablasl^{3}T_{4A}||_{L^{2}(H_{u})} \notag \\\lesssim&  |u|^{-6}||\sum_{i_{1}+i_{2}+i_{3}+i_{4}+i_{5}=3}||(V_{0})^{i_{1}}(|u|V_{4+A})^{i_{2}}(V_{4})^{i_{3}}(V_{A})^{i_{4}}(V_{3})^{i_{5}}f||_{L^{2}(H_{u})L^{2}(\mathcal{P}_{x})}+\frac{a}{|u|^{6}}
 \lesssim \frac{a}{|u|^{6}}.
\end{align}
and taking into account the signature of $T_{4E}$ we arrive at
\begin{align}
 \frac{|u|}{a}||\mathcal{D}^{2}T_{EF}||_{L^{2}_{sc}(S_{u,\ubar})}\lesssim \frac{1}{a}\lesssim 1
\end{align}
and 
\begin{align}
 \frac{|u|}{a}||\mathcal{D}^{3}T_{EF}||_{L^{2}_{sc}(H_{u})}\lesssim \frac{1}{a}\lesssim 1.
\end{align}
Next, we consider $\mathcal{D}=\snabla_{4}$. We compute
\begin{align} \notag
&|\snabla^{2}_{4}T_{EF}|\notag \\ \notag   \lesssim& \Bigg|\int_{\mathcal{P}_{x}}\Bigg[V_{3}V_{3}f+\frac{a^{\frac{1}{2}}\Gamma}{|u|}V_{3}|u|V_{4+A}f+\frac{a\Gamma^{2}}{|u|^{2}}(|u|V_{4+A})(|u|V_{4+D})f\Bigg]p_{E}p_{F}\frac{\sqrt{\det\gslash}dp^{1}dp^{2}dp^{3}}{p^{3}}\Bigg|\\ \notag  +&\int_{\mathcal{P}_{x}}|f|\frac{\sqrt{\det \gslash}dp^{1}dp^{2}dp^{3}}{p^{3}}+l.o.t.\\
\lesssim& |u|^{-1}\Bigg[||V_{3}V_{3}f||_{L^{2}(\mathcal{P}_{x})}+\frac{a^{\frac{1}{2}}||\Gamma}{|u|}V_{3}(|u|V_{4+A}f)||_{L^{2}(\mathcal{P}_{x})}+\frac{a\Gamma^{2}}{|u|^{2}}||(|u|V_{4+A})(|u|V_{4+D})f||_{L^{2}(\mathcal{P}_{x})} \Bigg]+\frac{a}{|u|^{2}},
\end{align}
since $|\tr\chi|\lesssim \frac{\Gamma}{|u|}\lesssim 1$, so
\begin{align} \notag 
&||\snabla^{2}_{4}T_{EF}||_{L^{2}(S_{u,\ubar})}=||\sqrt{\snabla^{2}_{4}T_{IE}\snabla^{2}_{4}T_{HF}\gslash^{IH}\gslash^{EF}}||_{L^{2}(S_{u,\ubar})}\\  \notag  
\lesssim&  |u|^{-3}\Bigg[||V_{3}V_{3}f||_{L^{2}(S_{u,\ubar})L^{2}(\mathcal{P}_{x})}+\frac{a^{\frac{1}{2}}\Gamma}{|u|}V_{3}(|u|V_{4+A}f)||_{L^{2}(S_{u,\ubar})L^{2}(\mathcal{P}_{x})}\\+&\frac{a\Gamma^{2}}{|u|^{2}}||(|u|V_{4+A})(|u|V_{4+D})f||_{L^{2}(S_{u,\ubar})L^{2}(\mathcal{P}_{x})} \Bigg]+\frac{a}{|u|^{4}}
\lesssim \frac{a}{|u|^{4}}.
\end{align}
Similarly 
\begin{align}
 ||\snabla^{3}_{4}T_{AB}||_{L^{2}(H_{u})}\lesssim  |u|^{-3}||\sum_{i_{1}+i_{2}+i_{3}+i_{4}+i_{5}=3}||(V_{0})^{i_{1}}(|u|V_{4+A})^{i_{2}}(V_{4})^{i_{3}}(V_{A})^{i_{4}}(V_{3})^{i_{5}}f||_{L^{2}(H_{u})L^{2}(\mathcal{P}_{x})}+\frac{a}{|u|^{4}}   
\end{align}
Consequently, noting $s_{2}(\snabla^{2}_{4}T_{EF})=1=s_{2}(\snabla^{3}_{4}T_{EF})$, we obtain the scale-invariant estimates
\begin{align}
\frac{|u|}{a}||\snabla^{2}_{4}T_{EF}||_{L^{2}_{sc}(S_{u,\ubar})}\lesssim \frac{1}{a}\lesssim 1, \hspace{3mm}
\frac{|u|}{a}||\snabla^{3}_{4}T_{EF}||_{L^{2}_{sc}(H_{u})}\lesssim \frac{1}{a}\lesssim 1.
\end{align}
\vspace{3mm}

\noindent For \textbf{$T_{43}$}, as with all previous components, 
\begin{align} \notag
|\mathcal{D}^{2}T_{43}|\lesssim& \Bigg|\int_{\mathcal{P}_x}\Bigg[V_{0}V_{0}f+|u|V_{4+A}V_{0}f+V_{4}V_{0}f+V_{0}(|u|V_{4+A}f)+|u|^{2}V_{4+A}V_{4+B}f\\ \notag  +&|u|^{-1}V_{4}|u|V_{4+A}f+V_{0}V_{4}f+|u|V_{4+A}V_{4}f+ |u|^{-2}V_{4}V_{4}f\Bigg]p^{4}p^{3}\sqrt{\det\gslash}\frac{dp^{1}dp^{2}dp^{3}}{p^{3}}|\\  \notag  +& \int_{\mathcal{P}_{x}}|f|p^{4}|p^{3}|\sqrt{\det\gslash}\frac{dp^{1}dp^{2}dp^{3}}{p^{3}}+l.o.t\\ \notag  
\lesssim& |u|^{-3}\Bigg[||V_{0}V_{0}f||_{L^{2}(\mathcal{P}_{x})}+|||u|V_{4+A}V_{0}f||_{L^{2}(\mathcal{P}_{x})}+||V_{4}V_{0}f||_{L^{2}(\mathcal{P}_{x})}+||V_{0}(|u|V_{4+A}f)||_{L^{2}(\mathcal{P}_{x})}\\ \notag  +&|||u|^{2}V_{4+A}V_{4+B}f||_{L^{2}(\mathcal{P}_{x})}\\   +&||V_{4}|u|V_{4+A}f||_{L^{2}(\mathcal{P}_{x})}+||V_{0}V_{4}f||_{L^{2}(\mathcal{P}_{x})}+|u|V_{4+A}V_{4}f||_{L^{2}(\mathcal{P}_{x})}+ ||V_{4}V_{4}f||_{L^{2}(\mathcal{P}_{x})}\Bigg]+\frac{a}{|u|^{4}},
\end{align}
and 
\begin{align} \notag  
||\mathcal{D}^{2}T_{43}||_{L^{2}(S_{u,\ubar})}\lesssim& |u|^{-3}\Bigg[||V_{0}V_{0}f||_{L^{2}(S_{u,\ubar})L^{2}(\mathcal{P}_{x})}+|||u|V_{4+A}V_{0}f||_{L^{2}(S_{u,\ubar})L^{2}(\mathcal{P}_{x})}\\ \notag  +&||V_{4}V_{0}f||_{L^{2}(S_{u,\ubar})L^{2}(\mathcal{P}_{x})}+||V_{0}(|u|V_{4+A}f)||_{L^{2}(S_{u,\ubar})L^{2}(\mathcal{P}_{x})}\\ \notag  +&|||u|^{2}V_{4+A}V_{4+B}f||_{L^{2}(S_{u,\ubar})L^{2}(\mathcal{P}_{x})}\\ \notag  +&||V_{4}|u|V_{4+A}f||_{L^{2}(S_{u,\ubar})L^{2}(\mathcal{P}_{x})}+||V_{0}V_{4}f||_{L^{2}(S_{u,\ubar})L^{2}(\mathcal{P}_{x})}\\  +&|u|V_{4+A}V_{4}f||_{L^{2}(S_{u,\ubar})L^{2}(\mathcal{P}_{x})}+ ||V_{4}V_{4}f||_{L^{2}(S_{u,\ubar})L^{2}(\mathcal{P}_{x})}\Bigg]+\frac{a}{|u|^{3}},    
\end{align}
and similarly,
\begin{align}
||\mathcal{D}^{3}T_{4A}||_{L^{2}(H_{u})}\lesssim  |u|^{-3}||\sum_{i_{1}+i_{2}+i_{3}+i_{4}+i_{5}=3}||(V_{0})^{i_{1}}(|u|V_{4+A})^{i_{2}}(V_{4})^{i_{3}}(V_{A})^{i_{4}}(V_{3})^{i_{5}}f||_{L^{2}(H_{u})L^{2}(\mathcal{P}_{x})}+\frac{a}{|u|^{4}}  
\end{align}
and by the estimates on Vlasov from the proposition \ref{prop:second_derivative_estimate}
\begin{align}
||\mathcal{D}^{2}T_{43}||_{L^{2}(S_{u,\ubar})}\lesssim \frac{a}{|u|^{4}},~||\mathcal{D}^{3}T_{4A}||_{L^{2}(H_{u})}\lesssim \frac{a}{|u|^{4}},     
\end{align}
or in scale invariant norm noting $s_{2}(\mathcal{D}^{2}T_{43})=s_{2}(T_{43})=1,
~s_{2}(\mathcal{D}^{3}T_{43})=s_{2}(T_{43})=1$
\begin{align}
||\mathcal{D}^{2}T_{43}||_{L^{2}_{sc}(S_{u,\ubar})}\lesssim \frac{1}{|u|},~||\mathcal{D}^{3}T_{4A}||_{L^{2}_{sc}(H_{u})}\lesssim \frac{1}{|u|}~or~\\
\frac{|u|}{a}||\mathcal{D}^{2}T_{43}||_{L^{2}_{sc}(S_{u,\ubar})}\lesssim \frac{1}{a}\lesssim 1,~\frac{|u|}{a}||\mathcal{D}^{3}T_{4A}||_{L^{2}_{sc}(H_{u})}\lesssim \frac{1}{a}\lesssim 1.
\end{align}
Next we consider $\mathcal{D}:=a^{\frac{1}{2}}\nablasl$ and compute 
\begin{align} \notag 
&|\nablasl^{2}T_{43}|\\ \notag  \lesssim& |\int_{\mathcal{P}_{x}}\Bigg[V_{A}V_{D}f+\frac{a^{\frac{1}{2}}\Gamma}{|u|}V_{A}(|u|V_{4+D}f)+\frac{a^{\frac{1}{2}}\Gamma}{|u|^{2}}(V_{A}V_{4}f+V_{A}V_{0}f \Bigg]p^{4}p^{3}\frac{\sqrt{\gslash}dp^{1}dp^{2}dp^{3}}{p^{3}}|+|\Gammaslash \nablasl T_{43}|+l.o.t.\\  
\lesssim& |u|^{-3}\Bigg[||V_{A}V_{D}f||_{L^{2}(\mathcal{P}_{x})}+\frac{a^{\frac{1}{2}}\Gamma}{|u|}||V_{A}|u|V_{4+D}f||_{L^{2}(\mathcal{P}_{x})}+\frac{a^{\frac{1}{2}}\Gamma}{|u|^{2}}||V_{A}V_{4}f||_{L^{2}(\mathcal{P}_{x})}+\frac{a^{\frac{1}{2}}\Gamma}{|u|^{2}}||V_{A}V_{0}f||_{L^{2}(\mathcal{P}_{x})} \Bigg]
+\frac{a}{|u|^{4}}
\end{align}
and recalling 
\begin{align} \notag 
 &||\nablasl^{2}T_{43}||_{L^{2}(S_{u,\ubar})}=||\sqrt{\gslash^{AC}\gslash^{BD}\nablasl_{A}\nablasl_{B}T_{43}\nablasl_{C}\nablasl_{D}T_{43}}||_{L^{2}(S_{u,\ubar})}\\ \notag 
 \lesssim& |u|^{-5}\Bigg[||V_{A}V_{D}f||_{L^{2}(S_{u,\ubar})L^{2}(\mathcal{P}_{x})}+\frac{a^{\frac{1}{2}}\Gamma}{|u|}||V_{A}|u|V_{4+D}f||_{L^{2}(S_{u,\ubar})L^{2}(\mathcal{P}_{x})}\\  +&\frac{a^{\frac{1}{2}}\Gamma}{|u|^{2}}||V_{A}V_{4}f||_{L^{2}(S_{u,\ubar})L^{2}(\mathcal{P}_{x})}+\frac{a^{\frac{1}{2}}\Gamma}{|u|^{2}}||V_{A}V_{0}f||_{L^{2}(S_{u,\ubar})L^{2}(\mathcal{P}_{x})} \Bigg]+\frac{a}{|u|^{5}}
 \lesssim \frac{a}{|u|^{5}}
\end{align}    
and similarly 
\begin{align} \notag 
||\mathcal{D}^{3}T_{43}||_{L^{2}(H_{u})}\lesssim&  |u|^{-5}||\sum_{i_{1}+i_{2}+i_{3}+i_{4}+i_{5}=3}||(V_{0})^{i_{1}}(|u|V_{4+A})^{i_{2}}(V_{4})^{i_{3}}(V_{A})^{i_{4}}(V_{3})^{i_{5}}f||_{L^{2}(H_{u})L^{2}(\mathcal{P}_{x})}+\frac{a}{|u|^{5}}\\
\lesssim& \frac{a}{|u|^{5}}.
\end{align}
Taking into account that $s_{2}(\nablasl^{2}T_{43})=1+1=2,~s_{2}(\nablasl^{3}T_{43})=\frac{3}{2}+1=\frac{5}{2}$,
\begin{align}
 ||\nablasl^{2}T_{43}||_{L^{2}_{sc}(S_{u,\ubar})}\lesssim \frac{1}{a|u|} \implies   ||(a^{\frac{1}{2}}\nablasl)^{2}T_{43}||_{L^{2}_{sc}(S_{u,\ubar})}\lesssim \frac{1}{|u|},\\
 ||\nablasl^{3}T_{43}||_{L^{2}_{sc}(H_{u})}\lesssim \frac{1}{a^{\frac{3}{2}}|u|} \implies   ||(a^{\frac{1}{2}}\nablasl)^{3}T_{43}||_{L^{2}_{sc}(S_{u,\ubar})}\lesssim \frac{1}{|u|},
\end{align}
or 
\begin{align}
\frac{|u|}{a}||\mathcal{D}^{2}T_{43}||_{L^{2}_{sc}(S_{u,\ubar})}\lesssim \frac{1}{a}\lesssim 1,~\frac{|u|}{a}||\mathcal{D}^{3}T_{43}||_{L^{2}_{sc}(H_{u})}\lesssim \frac{1}{a}\lesssim 1.    
\end{align}
Now we consider $\mathcal{D}=\snabla_{4}$ and evaluate 
\begin{align} \notag 
&|\snabla^{2}_{4}T_{43}|=\Bigg|\int_{\mathcal{P}_{x}}\Bigg[V_{3}V_{3}f+\frac{a^{\frac{1}{2}}\Gamma}{|u|}V_{3}|u|V_{4+A}f+\frac{a\Gamma^{2}}{|u|^{2}}(|u|V_{4+A})(|u|V_{4+D})f\Bigg]p^{4}p^{3}\frac{\sqrt{\det \gslash}dp^{1}dp^{2}dp^{3}}{p^{3}}\Bigg|+l.o.t.\\
\lesssim& |u|^{-3}\Bigg[||V_{3}V_{3}f||_{L^{2}(\mathcal{P}_{x})}+\frac{a^{\frac{1}{2}}\Gamma}{|u|}||V_{3}(|u|V_{4+A}f)||_{L^{2}(\mathcal{P}_{x})}+\frac{a\Gamma^{2}}{|u|^{2}}||(|u|V_{4+A})(|u|V_{4+D})f||_{L^{2}(\mathcal{P}_{x})} \Bigg]
\end{align}
and so 
\begin{align} \notag 
||\snabla^{2}_{4}T_{43}||_{L^{2}(S_{u,\ubar})}
\lesssim&  |u|^{-3}\Bigg[||V_{3}V_{3}f||_{L^{2}(S_{u,\ubar})L^{2}(\mathcal{P}_{x})}+\frac{a^{\frac{1}{2}}\Gamma}{|u|}V_{3}(|u|V_{4+A}f)||_{L^{2}(S_{u,\ubar})L^{2}(\mathcal{P}_{x})}\\+&\frac{a\Gamma^{2}}{|u|^{2}}||(|u|V_{4+A})(|u|V_{4+D})f||_{L^{2}(S_{u,\ubar})L^{2}(\mathcal{P}_{x})} \Bigg]
\lesssim \frac{a}{|u|^{3}}.
\end{align}
For the third derivatives, 
\begin{align}
||\mathcal{D}^{3}T_{43}||_{L^{2}(H_{u})}\lesssim  |u|^{-3}||\sum_{i_{1}+i_{2}+i_{3}+i_{4}+i_{5}=3}||(V_{0})^{i_{1}}(|u|V_{4+A})^{i_{2}}(V_{4})^{i_{3}}(V_{A})^{i_{4}}(V_{3})^{i_{5}}f||_{L^{2}(H_{u})L^{2}(\mathcal{P}_{x})}\\
\lesssim \frac{a}{|u|^{3}}
\end{align}
where we have used the Vlasov estimates from proposition \ref{prop:second_derivative_estimate}.
Consequently noting $s_{2}(\snabla^{2}_{4}T_{43})=1=s_{2}(\snabla^{3}_{4}T_{43})$
\begin{align}
||\snabla^{2}_{4}T_{EF}||_{L^{2}_{sc}(S_{u,\ubar})}\lesssim \frac{1}{|u|},~||\snabla^{3}_{4}T_{EF}||_{L^{2}_{sc}(H_{u})}\lesssim \frac{1}{|u|}~or~\\
\frac{|u|}{a}||\snabla^{2}_{4}T_{EF}||_{L^{2}_{sc}(S_{u,\ubar})}\lesssim \frac{1}{a}\lesssim 1,~\frac{|u|}{a}||\snabla^{3}_{4}T_{EF}||_{L^{2}_{sc}(H_{u})}\lesssim \frac{1}{a}\lesssim 1.
\end{align}
\textbf{$T_{33}$:} Recall that $T_{33}$ has the maximum decay rate among the stress-energy tensor components due to the appearance of $p^{4}p^{4}$. First consider $\mathcal{D}:=|u|\snabla_{3}$ and evaluate 
\begin{align} \notag 
|\mathcal{D}^{2}T_{33}|\lesssim& \Bigg|\int_{p}\Bigg[V_{0}V_{0}f+|u|V_{4+A}V_{0}f+V_{4}V_{0}f+V_{0}(|u|V_{4+A}f)+|u|^{2}V_{4+A}V_{4+B}f\\ \notag +&|u|^{-1}V_{4}|u|V_{4+A}f+V_{0}V_{4}f+|u|V_{4+A}V_{4}f+ |u|^{-2}V_{4}V_{4}f\Bigg]p^{4}p^{4}\sqrt{\det\gslash}\frac{dp^{1}dp^{2}dp^{3}}{p^{3}}\Bigg |\\ \notag 
+&\int_{\mathcal{P}_{x}}|fp^{4}p^{4}|\sqrt{\det\gslash}\frac{dp^{1}dp^{2}dp^{3}}{p^{3}}+l.o.t\\ \notag 
\lesssim& |u|^{-5}\Bigg[||V_{0}V_{0}f||_{L^{2}(\mathcal{P}_{x})}+|||u|V_{4+A}V_{0}f||_{L^{2}(\mathcal{P}_{x})}+||V_{4}V_{0}f||_{L^{2}(\mathcal{P}_{x})}+||V_{0}(|u|V_{4+A}f)||_{L^{2}(\mathcal{P}_{x})}\\ \notag +&|||u|^{2}V_{4+A}V_{4+B}f||_{L^{2}(\mathcal{P}_{x})}\\+&||V_{4}|u|V_{4+A}f||_{L^{2}(\mathcal{P}_{x})}+||V_{0}V_{4}f||_{L^{2}(\mathcal{P}_{x})}+|u|V_{4+A}V_{4}f||_{L^{2}(\mathcal{P}_{x})}+ ||V_{4}V_{4}f||_{L^{2}(\mathcal{P}_{x})}\Bigg]+\frac{a}{|u|^{6}},
\end{align}
hence, in the $L^2(S_{u,\ubar})-$norm,
\begin{align} \notag 
||\mathcal{D}^{2}T_{33}||_{L^{2}(S_{u,\ubar})}\lesssim& |u|^{-5}\Bigg[||V_{0}V_{0}f||_{L^{2}(S_{u,\ubar})L^{2}(\mathcal{P}_{x})}+|||u|V_{4+A}V_{0}f||_{L^{2}(S_{u,\ubar})L^{2}(\mathcal{P}_{x})}\\ \notag +&||V_{4}V_{0}f||_{L^{2}(S_{u,\ubar})L^{2}(\mathcal{P}_{x})}+||V_{0}(|u|V_{4+A}f)||_{L^{2}(S_{u,\ubar})L^{2}(\mathcal{P}_{x})}\\ \notag +&|||u|^{2}V_{4+A}V_{4+B}f||_{L^{2}(S_{u,\ubar})L^{2}(\mathcal{P}_{x})}\\ \notag +&||V_{4}|u|V_{4+A}f||_{L^{2}(S_{u,\ubar})L^{2}(\mathcal{P}_{x})}+||V_{0}V_{4}f||_{L^{2}(S_{u,\ubar})L^{2}(\mathcal{P}_{x})}\\  +&|u|V_{4+A}V_{4}f||_{L^{2}(S_{u,\ubar})L^{2}(\mathcal{P}_{x})}+ ||V_{4}V_{4}f||_{L^{2}(S_{u,\ubar})L^{2}(\mathcal{P}_{x})}\Bigg]+\frac{a}{|u|^{5}}.    
\end{align}
Similarly, for three derivatives
\begin{align}
||\mathcal{D}^{3}T_{33}||_{L^{2}(H_{u})}\lesssim  |u|^{-5}||\sum_{i_{1}+i_{2}+i_{3}+i_{4}+i_{5}=3}||(V_{0})^{i_{1}}(|u|V_{4+A})^{i_{2}}(V_{4})^{i_{3}}(V_{A})^{i_{4}}(V_{3})^{i_{5}}f||_{L^{2}(H_{u})L^{2}(\mathcal{P}_{x})}+\frac{a}{|u|^{5}}.    
\end{align}
By the estimates on Vlasov from Proposition \ref{prop:second_derivative_estimate}
\begin{align}
||\mathcal{D}^{2}T_{33}||_{L^{2}(S_{u,\ubar})}\lesssim \frac{a}{|u|^{5}},~ ||\mathcal{D}^{3}T_{33}||_{L^{2}(H_{u})}\lesssim \frac{a}{|u|^{5}},
\end{align}
or in scale invariant norms, noting $s_{2}(\mathcal{D}^{2}T_{33})=s_{2}(T_{33})=2=s_{2}(\mathcal{D}^{3}T_{33})$,
\begin{align} \notag 
||\mathcal{D}^{2}T_{33}||_{L^{2}_{sc}(S_{u,\ubar})}\lesssim \frac{1}{a|u|},~ ||\mathcal{D}^{3}T_{33}||_{L^{2}_{sc}(H_{u})}\lesssim \frac{1}{a|u|}~\text{or}~\\
\frac{|u|}{a}||\mathcal{D}^{2}T_{33}||_{L^{2}_{sc}(S_{u,\ubar})}\lesssim \frac{1}{a^{2}}\lesssim 1,~\frac{|u|}{a}||\mathcal{D}^{3}T_{33}||_{L^{2}_{sc}(H_{u})}\lesssim \frac{1}{a^{2}}\lesssim 1.
\end{align}
Now consider $\mathcal{D}=a^{\frac{1}{2}}\nablasl$ and evaluate
\begin{align} \notag 
&|\nablasl^{2}T_{33}|\\  \notag \lesssim& |\int_{\mathcal{P}_{x}}\Bigg[V_{A}V_{D}f+\frac{a^{\frac{1}{2}}\Gamma}{|u|}V_{A}(|u|V_{4+D}f)+\frac{a^{\frac{1}{2}}\Gamma}{|u|^{2}}(V_{A}V_{4}f+V_{A}V_{0}f \Bigg]p^{4}p^{4}\frac{\sqrt{\gslash}dp^{1}dp^{2}dp^{3}}{p^{3}}|+|\Gammaslash \nablasl T_{33}|+l.o.t.\\
\lesssim& |u|^{-5}\Bigg[||V_{A}V_{D}f||_{L^{2}(\mathcal{P}_{x})}+\frac{a^{\frac{1}{2}}\Gamma}{|u|}||V_{A}|u|V_{4+D}f||_{L^{2}(\mathcal{P}_{x})}+\frac{a^{\frac{1}{2}}\Gamma}{|u|^{2}}||V_{A}V_{4}f||_{L^{2}(\mathcal{P}_{x})}+\frac{a^{\frac{1}{2}}\Gamma}{|u|^{2}}||V_{A}V_{0}f||_{L^{2}(\mathcal{P}_{x})} \Bigg]
\end{align}
and recalling 
\begin{align} \notag 
 ||\nablasl^{2}T_{33}||_{L^{2}(S_{u,\ubar})}=&||\sqrt{\gslash^{AC}\gslash^{BD}\nablasl_{A}\nablasl_{B}T_{33}\nablasl_{C}\nablasl_{D}T_{33}}||_{L^{2}(S_{u,\ubar})}\\ \notag 
 \lesssim& |u|^{-7}\Bigg[||V_{A}V_{D}f||_{L^{2}(S_{u,\ubar})L^{2}(\mathcal{P}_{x})}+\frac{a^{\frac{1}{2}}\Gamma}{|u|}||V_{A}|u|V_{4+D}f||_{L^{2}(S_{u,\ubar})L^{2}(\mathcal{P}_{x})}\\+&\frac{a^{\frac{1}{2}}\Gamma}{|u|^{2}}||V_{A}V_{4}f||_{L^{2}(S_{u,\ubar})L^{2}(\mathcal{P}_{x})}+\frac{a^{\frac{1}{2}}\Gamma}{|u|^{2}}||V_{A}V_{0}f||_{L^{2}(S_{u,\ubar})L^{2}(\mathcal{P}_{x})} \Bigg]
 \lesssim \frac{a}{|u|^{7}}.
\end{align}    
Similarly, 
\begin{align}
 ||\nablasl^{3}T_{33}||_{L^{2}(H_{u})}\lesssim  |u|^{-8}||\sum_{i_{1}+i_{2}+i_{3}+i_{4}+i_{5}=3}||(V_{0})^{i_{1}}(|u|V_{4+A})^{i_{2}}(V_{4})^{i_{3}}(V_{A})^{i_{4}}(V_{3})^{i_{5}}f||_{L^{2}(H_{u})L^{2}(\mathcal{P}_{x})}
 \lesssim \frac{a}{|u|^{8}}
\end{align}
and so, noting $s_{2}(\nablasl^{2}T_{33})=1+2=3, s_{2}(\nablasl^{3}T_{33})=\frac{3}{2}+2=\frac{7}{2}$,
\begin{align}
  \notag ||\nablasl^{2}T_{33}||_{L^{2}_{sc}(S_{u,\ubar})}\lesssim \frac{1}{a^{2}|u|} \implies   ||(a^{\frac{1}{2}}\nablasl)^{2}T_{33}||_{L^{2}_{sc}(S_{u,\ubar})}\lesssim \frac{1}{a|u|},\\ \notag 
 ||\nablasl^{3}T_{33}||_{L^{2}_{sc}(H_{u})}\lesssim \frac{1}{a^{\frac{5}{2}}|u|}\implies ||(a^{\frac{1}{2}}\nablasl)^{3}T_{33}||_{L^{2}_{sc}(H_u)}\lesssim \frac{1}{a|u|},\end{align}
 or equivalently,
 \begin{align}
 \frac{|u|}{a}||\mathcal{D}^{2}T_{33}||_{L^{2}_{sc}(S_{u,\ubar})}\lesssim \frac{1}{a^{2}}\lesssim 1,~\frac{|u|}{a}||\mathcal{D}^{3}T_{33}||_{L^{2}_{sc}(H_{u})}\lesssim \frac{1}{a^{2}}\lesssim 1.
\end{align}
Finally, consider $\mathcal{D}=\snabla_{4}$ and compute 
\begin{align} \notag 
|\nabla^{2}_{4}T_{33}|=&
\Bigg|\int_{\mathcal{P}_{x}}\Bigg[V_{3}V_{3}f+\frac{a^{\frac{1}{2}}\Gamma}{|u|}V_{3}|u|V_{4+A}f+\frac{a\Gamma^{2}}{|u|^{2}}(|u|V_{4+A})(|u|V_{4+D})f\Bigg]p^{4}p^{4}\frac{\sqrt{\gslash}dp^{1}dp^{2}dp^{3}}{p^{3}}\Bigg|+l.o.t.\\
\lesssim &|u|^{-5}\Bigg[||V_{3}V_{3}f||_{L^{2}(\mathcal{P}_{x})}+\frac{a^{\frac{1}{2}}||\Gamma}{|u|}V_{3}(|u|V_{4+A}f)||_{L^{2}(\mathcal{P}_{x})}+\frac{a\Gamma^{2}}{|u|^{2}}||(|u|V_{4+A})(|u|V_{4+D})f||_{L^{2}(\mathcal{P}_{x})} \Bigg],
\end{align}
from where it follows that
\begin{align}\notag 
||\snabla^{2}_{4}T_{33}||_{L^{2}(S_{u,\ubar})}
\lesssim&  |u|^{-5}\Bigg[||V_{3}V_{3}f||_{L^{2}(S_{u,\ubar})L^{2}(\mathcal{P}_{x})}+\frac{a^{\frac{1}{2}}\Gamma}{|u|}V_{3}(|u|V_{4+A}f)||_{L^{2}(S_{u,\ubar})L^{2}(\mathcal{P}_{x})}\\+&\frac{a\Gamma^{2}}{|u|^{2}}||(|u|V_{4+A})(|u|V_{4+D})f||_{L^{2}(S_{u,\ubar})L^{2}(\mathcal{P}_{x})} \Bigg]
\lesssim \frac{a}{|u|^{5}}.
\end{align}
Similarly, for three derivatives, 
\begin{align}
||\snabla_{4}^{3}T_{33}||_{L^{2}(H_{u})}\lesssim&  |u|^{-5}||\sum_{i_{1}+i_{2}+i_{3}+i_{4}+i_{5}=3}||(|u|V_{0})^{i_{1}}(|u|V_{4+A})^{i_{2}}(V_{4})^{i_{3}}(V_{A})^{i_{4}}(V_{3})^{i_{5}}f||_{L^{2}(H_{u})L^{2}(\mathcal{P}_{x})}
\lesssim \frac{a}{|u|^{5}},
\end{align}
where the derivative estimates for the Vlasov fields from Proposition \ref{prop:second_derivative_estimate} are utilized.
Consequently, noting $s_{2}(\snabla^{2}_{4}T_{33})=2=s_{2}(\snabla^{3}_{4}T_{33})$
\begin{align}\notag 
||\snabla^{2}_{4}T_{33}||_{L^{2}_{sc}(S_{u,\ubar})}\lesssim \frac{1}{a|u|},~||\nabla_{4}^{3}T_{33}||_{L^{2}_{sc}(H_{u})}\lesssim \frac{1}{a|u|}~or~\\
\frac{|u|}{a}||\snabla^{2}_{4}T_{33}||_{L^{2}_{sc}(S_{u,\ubar})}\lesssim \frac{1}{a^{2}}\lesssim 1,~\frac{|u|}{a}||\snabla_{4}^{3}T_{33}||_{L^{2}_{sc}(H_{u})}\lesssim \frac{1}{a^{2}}\lesssim 1. 
\end{align}

\begin{remark}
Note that given the $L^{2}(S_{u,\ubar})$ estimates on the first and second derivatives of the stress-energy tensor components, $||\mathcal{D}T_{\mu\nu}||_{L^{4}_{sc}(S_{u,\ubar})}$ is controlled automatically by means of the scale-invariant Sobolev estimates. Notice that we do not have the same privilege for the Weyl curvature components since their second derivatives are controlled on the null hypersurfaces and therefore co-dimension one trace inequalities are indispensable. 
\end{remark}

\section{Energy Estimates for the Weyl curvature: Completion of the proof of the semi-global existence theorem}
In this final section, we complete the energy estimates for Weyl curvature components in terms of the initial data and therefore close the boot-strap. There are several ways to execute the energy estimates. Since Einstein's gravity is a variational theory, one expects the existence of a stress-energy tensor. However, due to the equivalence principle (or existence of geodesic normal coordinates), one can not simply find a stress-energy tensor for gravity. For the purpose of obtaining estimates, the Bel-Robinson tensor suffices. However, here we will adopt a direct integration by parts approach. 
First, recall the following generalized Gr\"onwall  type proposition 
\begin{proposition}\label{prop54}
Let $f(x,y), \hsp g(x,y)$ be positive functions defined on a rectangle \[U := \begin{Bmatrix} 0\leq x \leq x_0, \hsp 0\leq y \leq y_0 \end{Bmatrix}.       \] Suppose there exist constants $J ,c_1, c_2$ such that

\[ f(x,y)+g(x,y) \lesssim J + c_1 \int_0^{x} f(x^{\prime},y) \hsp \text{d}x^{\prime} +\int_0^{y} g(x,y^{\prime}) \hsp \text{d}y^{\prime},\]for all $(x,y) \in U$. Then there holds

\[ \forall (x,y) \in U: \hspace{5mm} f(x,y)+g(x,y) \lesssim J \mathrm{e}^{c_1 x + c_2y}. \]

\end{proposition}

\begin{proposition}\label{42}
For a Bianchi pair $(\Psi_1, \Psi_2)\in \left\{(\alpha,\widetilde{\beta}),(\widetilde{\beta},(\rho,\sigma)),((\rho,\sigma),\widetilde{\betabar}),(\widetilde{\betabar},\alphabar)\right\}$ satisfying
\[ \snabla_3 \snabla^i \Psi_1 + \left( \frac{i+1}{2} + s_2(\Psi_1) \right) \tr\chibar \snabla^i \Psi_1 - \widehat{\mathcal{D}}\snabla^i \Psi_2 = P_i,\]\[\snabla_4 \snabla^i \Psi_2 - \Hodge{\widehat{\mathcal{D}}} \snabla^i \Psi_1 = Q_i, \]the following holds true:
\end{proposition}
\begin{align} \notag 
&\int_{\Hu}\scaletwoSuubarprime{\snabla^i \Psi_1}^2 \dubarprime + \int_{\Hbu}\frac{a}{\upr^2} \scaletwoSuprime{\snabla^i \Psi_2}^2 \duprime \\  \notag \lesssim &\int_{H_{u_{\infty}}^{(0,\ubar)}}||\snabla^i \Psi_1||^{2}_{L^{2}_{sc}(S_{u,\ubar^{\prime}})} \dubarprime + \int_{\Hbar_0^{(u_{\infty},u)}} \frac{a}{\upr^2} ||\snabla^i \Psi_2||^{2}_{L^{2}_{sc}(S_{u,u^{\prime}})} \duprime \\&+   \notag \iint_{\mathcal{D}_{u,\ubar}}\frac{a}{\upr} \scaleoneSuprimeubarprime{\langle\snabla^i \Psi_1 , P_i\rangle}\duprime \dubarprime + \iint_{\mathcal{D}_{u,\ubar}}\frac{a}{\upr} \scaleoneSuprimeubarprime{\langle\snabla^i \Psi_2 ,Q_i\rangle}\duprime \dubarprime.
\end{align}
\begin{proof}
For the proof of this scale-invariant energy estimates, see \cite{AnAth}.
\end{proof}

\begin{proposition}
\label{alpha_energy_estimate}
Let the hypotheses of Theorem~\ref{mainone} be in force, and suppose in addition that the bootstrap assumptions \eqref{bootstrap} hold on the spacetime domain under consideration. Let 
\[
\mathcal{D} \in \big\{ \, |u|\, \snabla_{3}, \, a^{\frac12} \snabla, \, \snabla_{4} \, \big\}
\]
be one of the admissible commutation operators defined in Section~\ref{commute}, and let $i$ be an integer satisfying $0 \leq i \leq 2$. Then the following \emph{scale-invariant $L^{2}$ energy estimate} holds:
\begin{equation}
\label{eq:alpha_energy_estimate}
\frac{1}{a^{\frac{1}{2}}} \, \scaletwoHu{\mathcal{D}^{i} \alpha}
\;+\;
\frac{1}{a^{\frac{1}{2}}} \, \scaletwoHbaru{\mathcal{D}^{i} \widetilde{\beta}}
\;\; \lesssim \;\;
\frac{1}{a^{\frac{1}{2}}} \, \scaletwoHzero{\mathcal{D}^{i} \alpha}
\;+\;
\frac{1}{a^{\frac{1}{2}}} \, \scaletwoHbarzero{\mathcal{D}^{i} \widetilde{\beta}}
\;+\;
\frac{C}{a^{\frac13}}.
\end{equation}
In particular, for the case $\mathcal{D} = \mathrm{Id}$ (i.e.\ without commutation), \eqref{eq:alpha_energy_estimate} reads
\begin{equation}
\label{eq:alpha_energy_estimate_nocomm}
\frac{1}{a^{\frac{1}{2}}} \, \scaletwoHu{\alpha}
\;+\;
\frac{1}{a^{\frac{1}{2}}} \, \scaletwoHbaru{\widetilde{\beta}}
\;\; \lesssim \;\;
\frac{1}{a^{\frac{1}{2}}} \, \scaletwoHzero{\alpha}
\;+\;
\frac{1}{a^{\frac{1}{2}}} \, \scaletwoHbarzero{\widetilde{\beta}}
\;+\;
\frac{C}{a^{\frac13}}.
\end{equation}
Here:
\begin{itemize}
    \item $\scaletwoHu{\cdot}$ and $\scaletwoHbaru{\cdot}$ denote the scale-invariant $L^{2}$ flux norms along the outgoing and incoming null hypersurfaces $H_{u}$ and $\underline{H}_{\ubar}$, respectively.
    \item $\mathcal{D}^{i}$ denotes the $i$-fold composition of $\mathcal{D}$ with itself.
\end{itemize}
\end{proposition}

\begin{proof}
 The vital point to note here is that the bulk terms in estimating the energy involve a spacetime integral, and therefore the stress-energy terms involving the Vlasov source are always controllable on $H$. With this fundamental notion, we proceed to estimate the energy by means of Bianchi-pair integration. The energy estimates for $\nablasl$ are easy to prove. We consider the top derivatives.
 First, recall the equations
 \begin{align} \notag 
  &\snabla_3 \snabla^i \alpha + \frac{i+1}{2}\tr\chibar \snabla^i\alpha - \hat{\mathcal{D}}\snabla^{i}\tbeta\\\notag =&\sum_{i_{1}+i_{2}+i_{3}=i-1}\snabla^{i_{1}}\psi^{i_{2}+1}_{g}\snabla^{i_{3}+1}\beta^{R}+\sum_{i_{1}+i_{2}+i_{3}=i-1}\snabla^{i_{3}}\psi^{i_{2}+1}_{g}\snabla^{i_{3}}\alpha\\ \notag
  +&\sum_{i_{1}+i_{2}+i_{3}+i_{4}=i}\snabla^{i_{1}}\psi^{i_{2}}_{g}\snabla^{i_{3}}(\psi_{g},\chihat)\snabla^{i_{4}}(\rho,\sigma,\beta^{R})+\sum_{i_{1}+i_{2}+i_{3}+i_{4}=i}\snabla^{i_{1}}\psi^{i_{2}}_{g}\snabla^{i_{3}}(\chibarhat,\widetilde{\tr\chibar})\snabla^{i_{4}}\alpha\\
  \notag +&\sum_{i_{1}+i_{2}+i_{3}+i_{4}=i-1}\snabla^{i_{1}}\psi^{i_{2}+1}_{g}\snabla^{i_{3}}\tr\chibar\snabla^{i_{4}}\alpha+\sum_{i_{1}+i_{2}+i_{3}+i_{4}=i-2}\snabla^{i_{1}}\psi^{i_{2}+1}_{g}\snabla^{i_{3}}(\chibarhat,\tr\chi)\snabla^{i_{4}}\alpha\\
  +&\sum_{i_{1}+i_{2}+i_{3}=i}\snabla^{i_{1}}\psi^{i_{2}}_{g}\snabla^{i_{3}+1}T_{4A}+\sum_{i_{1}+i_{2}+i_{3}=i}\snabla^{i_{1}}\psi^{i_{2}}_{g}\snabla^{i_{3}}\nablasl_{4}T_{AB}+\sum_{i_{1}+i_{2}+i_{3}+i_{4}=i-1}\snabla^{i_{1}}\psi^{i_{2}}_{g}\snabla^{i_{3}}T_{A3}\snabla^{i_{4}}\alpha:=P_{i}.
 \end{align}
Similarly, we obtain for $\widetilde{\beta}$
\begin{align}\notag 
&\snabla_4 \snabla^i \tbeta - \Hodge{\hat{\mathcal{D}}}\snabla^i \alpha\\ \notag  =&\sum_{i_{1}+i_{2}+i_{3}+i_{4}=i}\snabla^{i_{1}}\psi^{i_{2}}_{g}\snabla^{i_{3}}(\psi_{g},\chihat)\snabla^{i_{4}}(\widetilde{\beta},\alpha)+\sum_{i_{1}+i_{2}+i_{3}+i_{4}=i}\snabla^{i_{1}}\psi^{i_{2}}_{g}\snabla^{i_{3}+1}T_{4A}\\ \notag 
+&\sum_{i_{1}+i_{2}+i_{3}+i_{4}=i}\snabla^{i_{1}}\psi^{i_{2}}_{g}\snabla^{i_{3}}(\psi_{g},\chihat)\snabla^{i_{4}}T_{4A}+\sum_{i_{1}+i_{2}+i_{3}+i_{4}=i}\snabla^{i_{1}}\psi^{i_{2}}_{g}\snabla^{i_{3}}\psi_{g}\snabla^{i_{4}}T_{4A}\\ \notag 
+&\sum_{i_{1}+i_{2}+i_{3}+i_{4}=i}\snabla^{i_{1}}\psi^{i_{2}}_{g}\snabla^{i_{3}}\nablasl_{4}T_{4A}+\sum_{i_{1}+i_{2}+i_{3}+i_{4}=i}\snabla^{i_{1}}\psi^{i_{2}}_{g}\snabla^{i_{3}}(\chihat,\tr\chi)\snabla^{i_{4}}\widetilde{\beta}\\ 
+&\sum_{i_{1}+i_{2}+i_{3}+i_{4}=i-2}\snabla^{i_{1}}\psi^{i_{2}+1}_{g}\snabla^{i_{3}}(\chihat,\tr\chi)\snabla^{i_{4}}\widetilde{\beta}:=Q_{i}
\end{align}
Now we apply the energy lemma. 
By H\"older's inequality, one has

\begin{equation}
    \begin{split}
         &\intubar \intu  \frac{a}{\upr} \scaleoneSuprimeubarprime{a^{\frac{i}{2}} P_i \cdot \aln \alpha }  \duprime \dubarprime \\\leq&\intu \frac{a}{\upr^2} \sum_{j=1}^{9}\left(  \intubar \scaletwoSuprimeubarprime{a^{\frac{i}{2}} P_i^j }^2\dubarprime \right)^{\frac{1}{2}}\duprime
        \cdot \sup_{u^{\prime}} \lVert \aln \alpha \rVert_{L^2_{(sc)}(H_{u^{\prime}}^{(0,\ubar)})}
    \end{split}
\end{equation}
Let us focus on the sum in the above line. For the first three terms, there holds

\[ \sum_{j=1}^3 \left( \intubar \scaletwoSuprimeubarprime{a^{\frac{i}{2}} P_i^j}^2 \dubarprime \right)^{\frac{1}{2}}\lesssim \frac{\al \Gamma \cdot R}{\upr}. \]

\noindent For the fourth term, there holds

\[ \left( \intubar \scaletwoSuprimeubarprime{a^{\frac{i}{2}} P_i^4}^2 \dubarprime \right)^{\frac{1}{2}}\lesssim \Gamma^2 + \Gamma R.\]For the fifth and sixth terms, there holds

\[ \left( \intubar \scaletwoSuprimeubarprime{a^{\frac{i}{2}} P_i^{5}}^2 \dubarprime \right)^{\frac{1}{2}} + \left( \intubar \scaletwoSuprimeubarprime{a^{\frac{i}{2}} P_i^{6}}^2 \dubarprime \right)^{\frac{1}{2}}\lesssim \Gamma^3 + \Gamma^2 R. \]
The seventh term is estimated as follows:
\begin{align}
a^{\frac{3}{2}}R \int_{u_{\infty}}^{u} \frac{1}{|u^{\prime}|^{2}}||a^{\frac{i}{2}}\nablasl^{i}\snabla_{4}T_{AB}||_{\mathcal{L}^{2}_{(sc)}(H_{u^{\prime}})}\text{d}u^{\prime} \lesssim a^{\frac{3}{2}}R\int_{u_{\infty}}^{u} \frac{1}{|u^{\prime}|^{2}}  \frac{1}{a|u^{\prime}|}\hsp a\hsp \text{d}u^{\prime} \lesssim \frac{a^{\frac{3}{2}}\hsp R\hsp \mathcal{V}}{|u|^{2}} \lesssim 1,
\end{align}
where we have used the top derivative Vlasov estimate from proposition \ref{prop:StressEnergyHighOrder}. Next, for the lower order terms in the following is estimated by $L^{4}(S)-L^{4}(S)$ type estimates and using the co-dimension-1 trace inequality or the Sobolev embedding
\begin{align}
\sum_{i_{1}+i_{2}+i_{3}=i}a^{\frac{3}{2}}R \int_{u_{\infty}}^{u} \frac{1}{|u^{\prime}|^{2}}||a^{\frac{i}{2}}\snabla^{i_{1}}\psi^{i_{2}}_{g}\nablasl^{i_{3}}\nabla_{4}T_{AB}||_{L^{2}_{sc}(H_{u^{\prime}})}\text{d}u^{\prime}   \lesssim \frac{a^{\frac{3}{2}}R\mathcal{V}}{|u|^{2}} \lesssim 1 
\end{align}
Next we estimate 
\begin{align}
\sum_{i_{1}+i_{2}+i_{3}=i}a^{\frac{3}{2}}R \int_{u_{\infty}}^{u} \frac{1}{|u^{\prime}|^{2}}||a^{\frac{i}{2}}\snabla^{i_{1}}\psi^{i_{2}}_{g}\nablasl^{i_{3}+1}T_{4A}||_{L^{2}_{sc}(H_{u^{\prime}})}\text{d}u^{\prime} \lesssim a^{\frac{3}{2}}R\mathcal{V}\int_{u_{\infty}}^{u}\frac{1}{a^{\frac{1}{2}}|u^{\prime}|}\frac{1}{|u^{\prime}|^{2}}\text{d}u^{\prime} \lesssim \frac{aR\mathcal{V}}{|u|^{2}}\lesssim 1
\end{align}
and for the ninth term 
\begin{align}
\sum_{i_{1}+i_{2}+i_{3}+i_{4}=i-1}a^{\frac{3}{2}}R\int_{u_{\infty}}^{u}\frac{1}{|u^{\prime}|^{2}}||a^{\frac{i}{2}}\snabla^{i_{1}}\psi^{i_{2}}_{g}\snabla^{i_{3}}T_{A3}\snabla^{i_{4}}\alpha||_{L^{2}_{sc}(H_{u^{\prime}})}\text{d}u^{\prime}     
\end{align}
we control in $\mathcal{L}^{4}_{sc}(S_{u,\ubar})-\mathcal{L}^{4}_{sc}(S_{u,\ubar})$ and use the scale-invariant codimension-1 trace inequality
 \begin{eqnarray}
  ||\varphi||_{L^{4}_{sc}(S_{u,\ubar})}\lesssim ||\varphi||_{L^{4}_{sc}(S_{u,0})}+||\snabla_{4}\varphi||^{\frac{1}{2}}_{L^{2}_{sc}(H)}\left(|||\varphi||^{\frac{1}{2}}_{L^{2}_{sc}(H)}+||a^{\frac{1}{2}}\nablasl\varphi||^{\frac{1}{2}}_{L^{2}_{sc}(H)}\right)  
\end{eqnarray}
to control $\nablasl\alpha$ in $L^{4}_{sc}(S_{u,\ubar})$
\begin{align}
||\nablasl\alpha||_{L^{4}_{sc}(S_{u,\ubar})}\lesssim ||\nablasl\alpha||_{L^{4}_{sc}(S_{u,0})}+\mathcal{R}
\end{align}
and so 
\begin{align}
\sum_{i_{1}+i_{2}+i_{3}+i_{4}=i-1}a^{\frac{3}{2}}R\int_{u_{\infty}}^{u}\frac{1}{|u^{\prime}|^{2}}||a^{\frac{i}{2}}\snabla^{i_{1}}\psi^{i_{2}}_{g}\snabla^{i_{3}}T_{A3}\snabla^{i_{4}}\alpha||_{L^{2}_{sc}(H_{u^{\prime}})}\text{d}u^{\prime}  \lesssim a^{\frac{3}{2}}\mathcal{R}\mathcal{V}\int_{u_{\infty}}^{u}\frac{1}{|u^{\prime}|^{2}}\frac{a}{|u^{\prime}|}\frac{1}{a^{\frac{5}{2}}}\text{d}u^{\prime} \lesssim \frac{\mathcal{R}\mathcal{V}}{|u|^{2}}\lesssim 1.
\end{align}
Collecting all the terms, we obtain \begin{equation} \label{alphabeta1} \intubar \intu \frac{a}{\upr} \scaleoneSuprimeubarprime{a^{\frac{i}{2}} P_i (\al\snabla)^i \alpha} \duprime \dubarprime\lesssim \Gamma^3 + \Gamma^2 R + \Gamma R+1.\end{equation}
Similarly, for the analogous term involving $\tbeta$, there holds

\begin{equation}
    \begin{split}
         &\intubar \intu  \frac{a}{\upr} \scaleoneSuprimeubarprime{a^{\frac{i}{2}} Q_i \cdot a^{\frac{i}{2}} \tbeta }  \duprime \dubarprime \\\leq&\sum_{j=1}^{7} \intu \frac{a}{\upr^2} \left(  \intubar \scaletwoSuprimeubarprime{a^{\frac{i}{2}} Q_i^j }^2\dubarprime \right)^{\frac{1}{2}}\duprime
        \cdot \sup_{u^{\prime}} \lVert (\al\snabla)^i  \tbeta \rVert_{L^2_{(sc)}(\Hb_{\ubar^{\prime}}^{(u_{\infty},u)})}\\
        &\lesssim a^{\frac{3}{2}}\mathcal{R}\intu \frac{a}{\upr^2} \sum_{j=1}^{7}\left(  \intubar \scaletwoSuprimeubarprime{a^{\frac{i}{2}} Q_i^j }^2\dubarprime \right)^{\frac{1}{2}}\duprime.
    \end{split}
    \end{equation}
    Note the important point that the Vlasov field can only be controlled on $H_{u}$, so naturally, we control the term $Q_{i}$ involving the Vlasov stress-energy terms along $H_{u}$. The first term is estimated as follows  
   \begin{equation}
    \intu \frac{a}{\upr^2}\left( \intubar \scaletwoSuprimeubarprime{a^{\frac{i}{2}} Q_i^1}^2 \dubarprime \right)^{\frac{1}{2}} \duprime \lesssim \frac{a\Gamma (R+\Gamma)}{\lvert u \rvert }.
\end{equation}
The sixth term is estimated as
\begin{equation}
    \intu \frac{a}{\upr^2}\left( \intubar \scaletwoSuprimeubarprime{a^{\frac{i}{2}} Q_i^6}^2 \dubarprime \right)^{\frac{1}{2}} \duprime \lesssim \frac{\al \Gamma (R+\Gamma)}{\lvert u \rvert}.
\end{equation}
For the seventh term, we have 
\begin{equation}
    \intu \frac{a}{\upr^2}\left( \intubar \scaletwoSuprimeubarprime{a^{\frac{i}{2}} Q_i^7}^2 \dubarprime \right)^{\frac{1}{2}} \duprime \lesssim \frac{a\Gamma^2 (R+\Gamma)}{\lvert u \rvert^2}
\end{equation}
Now the source terms are controlled. The second term is controlled as  
\begin{align}
&\intu \frac{a}{\upr^2}\Bigg(\int_{0}^{\ubar}||a^{\frac{i}{2}}\sum_{i_{1}+i_{2}+i_{3}=i}\snabla^{i_{1}}\psi^{i_{2}}_{g}\snabla^{i_{3}+1}T_{4A}||^{2}_{L^{2}_{sc}(S_{u^{\prime},\ubar^{\prime}})}\text{d}\ubar^{\prime} \Bigg)^{\frac{1}{2}}\text{d}u^{\prime} \notag \\
\lesssim& \int_{u_{\infty}}^{u}\frac{a}{|u^{\prime}|^{2}}\frac{\mathcal{V}(\mathcal{V}+\Gamma)}{a|u^{\prime}|}\text{d}u^{\prime} \lesssim \frac{\mathcal{V}(1+\Gamma)}{|u|^{2}}. 
\end{align}
The third term is estimated as 
\begin{align}
&\intu \frac{a}{\upr^2}\Bigg(\int_{0}^{\ubar}||a^{\frac{i}{2}}\sum_{i_{1}+i_{2}+i_{3}+i_{4}=i}\snabla^{i_{1}}\psi^{i_{2}}_{g}\snabla^{i_{3}}(\psi_{g},\chihat)\snabla^{i_{4}}T_{4A}||^{2}_{L^{2}_{sc}(S_{u^{\prime},\ubar^{\prime}})}\text{d}\ubar^{\prime} \Bigg)^{\frac{1}{2}}\text{d}u^{\prime}  \notag \\
\lesssim& \int_{u_{\infty}}^{u}\frac{a}{|u^{\prime}|^{2}}\frac{a}{|u^{\prime}|}\frac{\mathcal{V}\Gamma}{a|u^{\prime}|}\text{d}u^{\prime} \lesssim \frac{a\mathcal{V}\Gamma}{|u|^{3}}.
\end{align}
The fourth term is estimated as 
\begin{align}   \notag 
&\intu \frac{a}{\upr^2}\Bigg(\int_{0}^{\ubar}||a^{\frac{i}{2}}\sum_{i_{1}+i_{2}+i_{3}+i_{4}=i}\snabla^{i_{1}}\psi^{i_{2}}_{g}\snabla^{i_{3}}\psi_{g}\snabla^{i_{4}}T_{4A}||^{2}_{L^{2}_{sc}(S_{u^{\prime},\ubar^{\prime}})}\text{d}\ubar^{\prime} \Bigg)^{\frac{1}{2}}\text{d}u^{\prime}  \\\nonumber 
\lesssim& \int_{u_{\infty}}^{u}\frac{a}{|u^{\prime}|^{2}}\frac{a}{|u^{\prime}|}\frac{\mathcal{V}\Gamma}{a^{\frac{3}{2}}|u^{\prime}|}\text{d}u^{\prime} \lesssim \frac{a^{\frac{1}{2}}\mathcal{V}\Gamma}{|u|^{3}},
\end{align}
whereas finally the top order term (the fifth term)
\begin{align}\notag 
&\intu \frac{a}{\upr^2}\Bigg(\int_{0}^{\ubar}||a^{\frac{i}{2}}\sum_{i_{1}+i_{2}+i_{3}=i}\nabla^{i_{1}}\psi^{i_{2}}_{g}\nabla^{i_{3}}\nabla_{4}T_{AB}||^{2}_{L^{2}_{sc}(S_{u^{\prime},\ubar^{\prime}})}\text{d}\ubar^{\prime} \Bigg)^{\frac{1}{2}}\text{d}u^{\prime} \\
\lesssim& \int_{u_{\infty}}^{u}\frac{a}{|u^{\prime}|^{2}}a\frac{\mathcal{V}}{a|u^{\prime}|}\text{d}u^{\prime} +\int_{u_{\infty}}^{u}\frac{a}{|u^{\prime}|^{2}}a\frac{\mathcal{V}\Gamma}{a|u^{\prime}|^{2}}\text{d}u^{\prime} \lesssim \frac{a\mathcal{V}}{|u|^{2}}+\frac{a\mathcal{V}}{|u|^{3}} \lesssim \frac{a\mathcal{V}}{\modu^2}.   
\end{align}
Collecting all the terms together, we obtain 
\begin{equation}
    \begin{split}
        & \frac{1}{\al} \scaletwoHu{(a^{\frac{1}{2}} \snabla)^i \alpha} + \frac{1}{\al}\scaletwoHbaru{\aln \tbeta} \\ \lesssim & \frac{1}{\al} \scaletwoHzero{(a^{\frac{1}{2}} \snabla)^i\alpha} + \frac{1}{\al} \scaletwoHbarzero{\aln \tbeta} \\&+  \frac{1}{\al} \intubar \intu  \frac{a}{\upr} \scaleoneSuprimeubarprime{a^{\frac{i}{2}} P_i \cdot \aln \alpha }  \duprime \dubarprime
        \\ &+ \frac{1}{\al} \intubar \intu  \frac{a}{\upr} \scaleoneSuprimeubarprime{a^{\frac{i}{2}} Q_i \cdot \aln \tbeta }  \duprime \dubarprime \\ \lesssim &\frac{1}{\al} \scaletwoHzero{\aln \alpha} + \frac{1}{\al} \scaletwoHbarzero{\aln \tbeta} + \frac{1}{a^{\frac{1}{3}}}.
        \end{split}
\end{equation}
This proves the claim when $\mathcal{D}=a^{\frac{1}{2}}\snabla$. We actually need a commutation for $[\mathcal{D}^{k},\nablasl]$ appearing from the Hodge term in the Bianchi equation for $\mathcal{D}=|u|\nabla_{3},\nabla_{4}$. First, we obtain the relevant equations for the pair $(\mathcal{D}^{i}\alpha,\mathcal{D}^{i}\beta)$ for $\max{i}=2$. To this end, we recall the following lemma. 
\begin{lemma}

Assume a tensorfield $\phi$ satisfies an equation of the form \[ \snabla_3\phi+ a \tr\chibar \phi = G_0.    \] there holds:

\begin{align}
\snabla_3 \mathcal{D}^{k+\ell}\phi + a\tr\chibar \mathcal{D}^{k+\ell}\phi = G_i, \notag\end{align}where

\begin{align}
\nonumber G_i&= \sum_{i_1+i_2+i_3=i}\mathcal{D}^{i_1}\big(\eta,\etabar,1 \big)^{i_2} \mathcal{D}^{i_3}G_0 \\ \notag  +& \sum_{i_1+i_2+i_3+i_4+1=i}\dione\psi_g^{i_2}\dit \omegabar \mathcal{D}^{i_4+1}\phi \\ \notag +&\sumifm \dione \psi_g^{i_2}\dit(\eta,\etabar,\chibarhat)\dif \snabla \phi\\  \notag +&\sumifm \dione \psi_g^{i_2}\dit \sdiv \etabar \dif \phi\\ \notag +& \sumifim \dione \psi_g^{i_2}\dit \chihat \dif \chibarhat \difi \phi \\ \notag +& \sumifim \dione \psi_g^{i_2}\dit \tr\chi \dif \tr\chibar \difi \phi \\ \notag +& \sumifim \dione \psi_g^{i_2}\dit(\rho,T_{34})\dif \phi \\ \notag +& \sumifm \dione \psi_g^{i_2}\dit (\betabar,T_{3A})\dif \phi   \\ \notag &+\sum_{i_1+i_2+i_3+i_4+i_5=i-1}\mathcal{D}^{i_1}\psi_g^{i_2}\mathcal{D}^{i_3}(\eta,\etabar)\mathcal{D}^{i_4}(\eta,\etabar,\chibarhat, \tr\chibar)\mathcal{D}^{i_5}\phi \nonumber \\&+ \nonumber \sum_{i_1+i_2+i_3+i_4=i-1}\mathcal{D}^{i_1}\psi_g^{i_2}\mathcal{D}^{i_3}\sigma\mathcal{D}^{i_4}\phi \notag \\ &+\nonumber\sum_{i_1+i_2+i_3+i_4+i_5+i_6=i-1}\mathcal{D}^{i_1}\psi_g^{i_2}\mathcal{D}^{i_3}\tr\chibar\mathcal{D}^{i_4}\modu \mathcal{D}^{i_5}\tildetr\mathcal{D}^{i_6}\phi \\ &+\nonumber \sum_{i_1+i_2+i_3+i_4+i_5+i_6=i-1}\mathcal{D}^{i_1}\psi_g^{i_2}\mathcal{D}^{i_3} \modu\mathcal{D}^{i_4}\chibarhat\mathcal{D}^{i_5}\chibarhat\mathcal{D}^{i_6}\phi \notag \\ &+\nonumber \sum_{i_1+i_2+i_3+i_4+i_5+i_6=i-1}\mathcal{D}^{i_1}\psi_g^{i_2}\mathcal{D}^{i_3} \modu\mathcal{D}^{i_4}\omegabar\mathcal{D}^{i_5}\tr\chibar\mathcal{D}^{i_6}\phi  \\&+ \sum_{i_1+i_2+i_3+i_4+i_5+i_6=i-1}\mathcal{D}^{i_1}\psi_g^{i_2}\mathcal{D}^{i_3} \modu\mathcal{D}^{i_4}\slashed{T}_{33}\mathcal{D}^{i_5}\phi.
\end{align}
Here, $\psi_g \in \{\eta,\etabar,1\}.$
\end{lemma}
Now, the important point here is that for the Bianchi system, $G_{0}$ actually contains the Hodge term, which is principal.  
Let us recall the equations for $\alpha$ and $\beta$: 
\begin{align}
\snabla_{3}\alpha+\frac{1}{2}\tr\chibar \alpha=\snabla\widehat{\otimes}\beta+P,\\
\snabla_{4}\beta=\slashed{\div} \alpha+Q.
\end{align}
Here, $G_{0}$ is $\snabla\widehat{\otimes}\beta+P$. Now note the following commutation formula. Define, for an arbitrary $k$,
\be a^k := \notag [\De,\snabla]\De^{k-1}\Phi_2,
\ee
\be b^k := \mathcal{D}[\De^{k-1},\snabla]\Phi_2. \notag  \ee
We compute 
 \begin{align} \notag 
  [\mathcal{D}^{k},\nablasl]\Phi_{2}=&[\mathcal{D},\nablasl]\mathcal{D}^{k-1}\Phi_{2}+\mathcal{D}[\mathcal{D}^{k-1},\nablasl]\Phi_{2}=a^{k}+b^{k}=a^{k}+\mathcal{D}(a^{k-1}+b^{k-1})=\dots\\
  =&\sum_{l=0}^{k-1}\mathcal{D}^{l}a^{k-l}=\sum_{l=1}^{k}\mathcal{D}^{k-l}a^{l}.
 \end{align}
 Therefore, the goal is to evaluate $a^{j}$ for an arbitrary $j$. First, for $\mathcal{D}=|u|\snabla_{3}$,
 \begin{align} \notag 
  a^{j}=&[\mathcal{D},\nablasl]\mathcal{D}^{j-1}\Phi_{2}=|u|(\tr\chibar,\chibarhat)\nablasl\mathcal{D}^{j-1}\Phi_{2}+|u|\betabar \mathcal{D}^{j-1}\Phi_{2}+(\eta,\etabar)\mathcal{D}^{j}\Phi_{2}+|u|(\chihat,\tr\chi)\etabar \mathcal{D}^{j-1}\Phi_{2}\\
  +&|u|(\chibarhat,\tr\chibar)\eta\mathcal{D}^{j-1}\Phi_{2}
 \end{align}
 and for $\mathcal{D}=\snabla_{4}$,
 \begin{align}
  a^{j}=[\mathcal{D},\nablasl]\mathcal{D}^{j-1}\Phi_{2}=(\chihat,\tr\chi)\nablasl\mathcal{D}^{j-1}\Phi_{2}+\Bigg((\chihat,\tr\chi)\etabar+\beta+\Tslash_{4}\Bigg)\mathcal{D}^{j-1}\Phi_{2}.   
 \end{align}
 Hence, explicitly, the commutation reads for $\mathcal{D}:=|u|\snabla_{3}$,
 \begin{align} \notag 
[\mathcal{D}^{k},\nablasl]\Phi_{2}=&\sum_{j=1}^{k}\mathcal{D}^{k-j}\Bigg(|u|(\tr\chibar,\chibarhat)\nablasl\mathcal{D}^{j-1}\Phi_{2}+|u|\betabar \mathcal{D}^{j-1}\Phi_{2}+(\eta,\etabar)\mathcal{D}^{j}\Phi_{2}+|u|(\chihat,\tr\chi)\etabar \mathcal{D}^{j-1}\Phi_{2}\\ \notag 
  +&|u|(\chibarhat,\tr\chibar)\eta\mathcal{D}^{j-1}\Phi_{2}\Bigg)\\ \notag 
  =&\sum_{j=1}^{k}\sum_{i_{1}+i_{2}+i_{3}=k-j}\mathcal{D}^{i_{1}}|u|\mathcal{D}^{i_{2}}(\tr\chibar,\chibarhat)\mathcal{D}^{i_{3}}\nablasl\mathcal{D}^{j-1}\Phi_{2}+\sum_{j=1}^{k}\sum_{i_{1}+i_{2}+i_{3}=k-j}\mathcal{D}^{i_{1}}|u|\mathcal{D}^{i_{2}}\betabar\mathcal{D}^{i_{3}+j-1}\Phi_{2}\\ \notag 
  +&\sum_{j=1}^{k}\sum_{i_{1}+i_{2}=k-j}\mathcal{D}^{j_{1}}(\eta,\etabar)\mathcal{D}^{i_{2}+j}\Phi_{2}+\sum_{j=1}^{k}\sum_{i_{1}+i_{2}+i_{3}+i_{4}=k-j}\mathcal{D}^{i_{1}}|u|\mathcal{D}^{i_{2}}(\chihat,\tr\chi)\mathcal{D}^{i_{3}}\etabar\mathcal{D}^{i_{4}+j-1}\Phi_{2}\\
    +&\sum_{i_{1}+i_{2}+i_{3}+i_{4}=k-j}\mathcal{D}^{i_{1}}|u|\mathcal{D}^{i_{2}}(\chibarhat,\tr\chibar)\mathcal{D}^{i_{3}}\eta \mathcal{D}^{i_{4}+j-1}\Phi_{2},
 \end{align}
 while for $\mathcal{D}=\snabla_{4}$,
 \begin{align} \notag 
[\mathcal{D}^{k},\nablasl]\Phi_{2}=&\sum_{j=1}^{k}\sum_{i_{1}+i_{2}=k-j}\mathcal{D}^{i_{1}}(\chihat,\tr\chi)\mathcal{D}^{i_{2}}\nablasl\mathcal{D}^{j-1}\Phi_{2}+\sum_{j=1}^{k}\sum_{i_{1}+i_{2}+i_{3}=k-j} \mathcal{D}^{i_{1}}(\chihat,\tr\chi)\mathcal{D}^{i_{2}}\etabar\mathcal{D}^{i_{3}+j-1}\Phi_{2}\\ +&\sum_{j=1}^{k}\sum_{i_{1}+i_{2}=k-j}\mathcal{D}^{i_{1}}\beta\mathcal{D}^{i_{2}+j-1}\Phi_{2}+\sum_{j=1}^{k}\sum_{i_{1}+i_{2}=k-j}\mathcal{D}^{i_{1}}\Tslash_{4}\mathcal{D}^{i_{2}+j-1}\Phi_{2}.
 \end{align}
Therefore, recalling Proposition 
\ref{propositionnabla3}, the commuted equation for $\alpha$ reads 
\begin{align}\notag
\snabla_{3}\mathcal{D}^{i}\alpha+\frac{1}{2}\tr\chibar\mathcal{D}^{i}\alpha=\nabla\widehat{\otimes}\mathcal{D}^{i}\beta+P^{1}_{i},    
\end{align}
where $P^{1}_{i}$ reads 
\begin{align}
  \notag &P^{1}_{i}=\sum_{i_1+i_2+i_3=i-1}\mathcal{D}^{i_1}\big(\eta,\etabar,1 \big)^{i_2} \mathcal{D}^{i_3}\nabla \beta+\sum_{i_1+i_2+i_3=i}\mathcal{D}^{i_1}\big(\eta,\etabar,1 \big)^{i_2} \mathcal{D}^{i_3}P\\ \notag  +& \sum_{i_1+i_2+i_3+i_4+1=i}\dione\psi_g^{i_2}\dit \omegabar \mathcal{D}^{i_4+1}\alpha \\ \notag +&\sumifm \dione \psi_g^{i_2}\dit(\eta,\etabar,\chibarhat)\dif \snabla \alpha\\  \notag +&\sumifm \dione \psi_g^{i_2}\dit \sdiv \etabar \dif \alpha\\ \notag +& \sumifim \dione \psi_g^{i_2}\dit \chihat \dif \chibarhat \difi \alpha \\ \notag +& \sumifim \dione \psi_g^{i_2}\dit \tr\chi \dif \tr\chibar \difi \alpha \\ \notag +& \sumifim \dione \psi_g^{i_2}\dit(\rho,T_{34})\dif \alpha \\ \notag +& \sumifm \dione \psi_g^{i_2}\dit (\betabar,T_3A)\dif \alpha   \\ \notag &+\sum_{i_1+i_2+i_3+i_4+i_5=i-1}\mathcal{D}^{i_1}\psi_g^{i_2}\mathcal{D}^{i_3}(\eta,\etabar)\mathcal{D}^{i_4}(\eta,\etabar,\chibarhat, \tr\chibar)\mathcal{D}^{i_5}\alpha \nonumber \\&+ \nonumber \sum_{i_1+i_2+i_3+i_4=i-1}\mathcal{D}^{i_1}\psi_g^{i_2}\mathcal{D}^{i_3}\sigma\mathcal{D}^{i_4}\alpha \notag \\ &+\nonumber\sum_{i_1+i_2+i_3+i_4+i_5+i_6=i-1}\mathcal{D}^{i_1}\psi_g^{i_2}\mathcal{D}^{i_3}\tr\chibar\mathcal{D}^{i_4}\modu \mathcal{D}^{i_5}\tildetr\mathcal{D}^{i_6}\alpha \\ &+\nonumber \sum_{i_1+i_2+i_3+i_4+i_5+i_6=i-1}\mathcal{D}^{i_1}\psi_g^{i_2}\mathcal{D}^{i_3} \modu\mathcal{D}^{i_4}\chibarhat\mathcal{D}^{i_5}\chibarhat\mathcal{D}^{i_6}\alpha \notag \\ &+\nonumber \sum_{i_1+i_2+i_3+i_4+i_5+i_6=i-1}\mathcal{D}^{i_1}\psi_g^{i_2}\mathcal{D}^{i_3} \modu\mathcal{D}^{i_4}\omegabar\mathcal{D}^{i_5}\tr\chibar\mathcal{D}^{i_6}\alpha  \\&+ \sum_{i_1+i_2+i_3+i_4+i_5+i_6=i-1}\mathcal{D}^{i_1}\psi_g^{i_2}\mathcal{D}^{i_3} \modu\mathcal{D}^{i_4}\slashed{T}_{33}\mathcal{D}^{i_5}\alpha  \notag \\ \notag 
  &+\sum _{i_{1}+i_{2}+i_{3}+i_{4}=i-1}\mathcal{D}^{i_{1}}|u|\mathcal{D}^{i_{2}}(\tr\chibar,\chibarhat)\mathcal{D}^{i_{3}}\nablasl\mathcal{D}^{i_{4}}\beta \\
  \notag &+\sum_{i_{1}+i_{2}+i_{3}=i-1}\mathcal{D}^{i_{1}}|u|\mathcal{D}^{i_{2}}(\betabar,T_{A3})\mathcal{D}^{i_{3}}\beta \\
  &+\sum_{i_{1}+i_{2}=i}\mathcal{D}^{i_{1}}(\eta,\etabar)\mathcal{D}^{i_{2}}\beta \notag \\
  &+\sum_{i_{1}+i_{2}+i_{3}+i_{4}=i-1}\mathcal{D}^{i_{1}}|u|\mathcal{D}^{i_{2}}(\chihat,\tr\chi)\mathcal{D}^{i_{3}}\etabar\mathcal{D}^{i_{4}}\beta  \notag \\
  &+\sum_{i_{1}+i_{2}+i_{3}+i_{4}=i-1}\mathcal{D}^{i_{1}}|u|\mathcal{D}^{i_{2}}(\chibarhat,\tr\chibar)\mathcal{D}^{i_{3}}\eta\mathcal{D}^{i_{4}}\beta,
\end{align}
for $\mathcal{D}=|u|\snabla_{3}$ 
and 
\begin{align} \notag 
P^{1}_{i}= &\sum_{i_1+i_2+i_3=i-1}\mathcal{D}^{i_1}\big(\eta,\etabar,1 \big)^{i_2} \mathcal{D}^{i_3}\nabla\beta \\ \notag  +& \sum_{i_1+i_2+i_3+i_4+1=i}\dione\psi_g^{i_2}\dit \omegabar \mathcal{D}^{i_4+1}\alpha \\ \notag &+\frac{1}{\al}\sumifm \dione \psi_g^{i_2}\dit(\eta,\etabar,\chibarhat)\mathcal{D}^{i_4+1} \alpha\\  \notag +&\sumifm \dione \psi_g^{i_2}\dit \sdiv \etabar \dif \alpha\\ \notag +& \sumifim \dione \psi_g^{i_2}\dit \chihat \dif \chibarhat \difi \alpha \\ \notag +& \sumifim \dione \psi_g^{i_2}\dit \tr\chi \dif \tr\chibar \difi \alpha \\ \notag +& \sumifm \dione \psi_g^{i_2}\dit(\rho,\Te_{34})\dif \alpha \\ \notag +& \sumifm \dione \psi_g^{i_2}\dit (\betabar,\Te_3)\dif \alpha   \\ \notag +&\frac{1}{\al} \sumifm \dione \psi_g^{i_2}\mathcal{D}^{i_3+1}\tildetr \dif \alpha \\ \notag &+\sum_{i_1+i_2+i_3+i_4+i_5=i-1}\mathcal{D}^{i_1}\psi_g^{i_2}\mathcal{D}^{i_3}(\eta,\etabar)\mathcal{D}^{i_4}(\eta,\etabar,\chibarhat, \tr\chibar)\mathcal{D}^{i_5}\alpha \nonumber \\&+ \nonumber \sum_{i_1+i_2+i_3+i_4=i-1}\mathcal{D}^{i_1}\psi_g^{i_2}\mathcal{D}^{i_3}\sigma\mathcal{D}^{i_4}\alpha \notag \\ &+\nonumber\sum_{i_1+i_2+i_3+i_4+i_5+i_6=i-1}\mathcal{D}^{i_1}\psi_g^{i_2}\mathcal{D}^{i_3}\tr\chibar\mathcal{D}^{i_4}\modu \mathcal{D}^{i_5}\tildetr\mathcal{D}^{i_6}\alpha \\ &+\nonumber \sum_{i_1+i_2+i_3+i_4+i_5+i_6=i-1}\mathcal{D}^{i_1}\psi_g^{i_2}\mathcal{D}^{i_3} \modu\mathcal{D}^{i_4}\chibarhat\mathcal{D}^{i_5}\chibarhat\mathcal{D}^{i_6}\alpha \notag \\ \notag  &+\nonumber \sum_{i_1+i_2+i_3+i_4+i_5+i_6=i-1}\mathcal{D}^{i_1}\psi_g^{i_2}\mathcal{D}^{i_3} \modu\mathcal{D}^{i_4}\omegabar\mathcal{D}^{i_5}\tr\chibar\mathcal{D}^{i_6}\alpha  \\ \notag &+ \sum_{i_1+i_2+i_3+i_4+i_5+i_6=i-1}\mathcal{D}^{i_1}\psi_g^{i_2}\mathcal{D}^{i_3} \modu\mathcal{D}^{i_4}\slashed{T}_{33}\mathcal{D}^{i_5}\alpha\\ \notag 
&+\sum_{i_{1}+i_{2}=i-1}\mathcal{D}^{i_{1}}(\chihat,\tr\chi)\mathcal{D}^{i_{2}}\nablasl\beta\\ \notag &+\sum_{i_{1}+i_{2}+i_{3}=i-1} \mathcal{D}^{i_{1}}(\chihat,\tr\chi)\mathcal{D}^{i_{2}}\etabar\mathcal{D}^{i_{3}}\beta\\ &+\sum_{i_{1}+i_{2}=i-1}\mathcal{D}^{i_{1}}(\beta,\Tslash_{4})\mathcal{D}^{i_{2}}\beta
\end{align}
for $\mathcal{D}=\snabla_{4}$.
We now commute with $\nablasl^{j}$such that $i+j=2$
\begin{align}
 \snabla_{3}\nablasl^{j}\mathcal{D}^{i}\alpha+(\frac{j+1}{2})\tr\chibar \nablasl^{j}\mathcal{D}^{i}\alpha =\snabla\widehat{\otimes}\nablasl^{j}\mathcal{D}^{i}\beta+P_{i,j},   \notag 
\end{align}
where $P_{i,j}$ reads 
\begin{align} \notag 
 & P_{i,j}:=\sum_{j_{1}+j_{2}+j_{3}+j_{4}=j-1}\snabla^{j_{1}}(\eta+\etabar)^{j_{2}}\snabla^{j_{3}}\widetilde{\betabar}\snabla^{j_{4}}\mathcal{D}^{i}\alpha+\sum_{j_{1}+j_{2}+j_{3}=j}\snabla^{j_{1}}(\eta+\etabar)^{j_{2}}\snabla^{j_{3}}P^{1}_{i}\\ \notag 
 &+\sum_{j_{1}+j_{2}+j_{3}+j_{4}=j}\snabla^{j_{1}}(\eta+\etabar)^{j_{2}}\snabla^{j_{3}}\chibarhat\snabla^{j_{4}}\mathcal{D}^{i}\alpha+\sum_{j_{1}+j_{2}+j_{3}+j_{4}=j-1}\snabla^{j_{1}}(\eta+\etabar)^{j_{2}+1}\snabla^{j_{3}}\tr\chibar\snabla^{j_{4}}\mathcal{D}^{i}\alpha\\ \notag 
 &+\sum_{j_{1}+j_{2}+j_{3}=j-1}\snabla^{j_{1}}\psi^{j_{2}+1}_{g}\snabla^{j_{3}+1}\mathcal{D}^{i}\widetilde{\beta}+\sum_{j_{1}+j_{2}+j_{3}=j-1}\snabla^{j_{1}}\psi^{j_{2}+1}_{g}\snabla^{j_{3}}\mathcal{D}^{i}\alpha\\
 &+\sum_{j_{1}+j_{2}+j_{3}+j_{4}=j}\snabla^{j_{1}}\psi^{j_{2}}_{g}\snabla^{j_{3}}(\psi_{g},\chihat)\snabla^{j_{4}}(\rho,\sigma,\mathcal{D}^{i}\widetilde{\beta}),
\end{align}
with $P_{i,j}=P^{1}_{i}$ for $j=0$. Now we commute with the $\snabla_{4}$ equation, which is much simpler. First, recall the proposition.
Assume a tensor-field $\phi$ satisfies $\snabla_4 \phi = F_0$. Then, if $\mathcal{D} \in \{\modu \snabla_3, \snabla_4, \al \snabla \}$,  there holds $\snabla_4 \mathcal{D}^i \phi :=F_i,$ where

\begin{align}
\nonumber F_i =& \sum_{i_1+i_2+i_3=i}\mathcal{D}^{i_1}(\eta,\etabar,\modu\omegabar)^{i_2}\mathcal{D}^{i_3}F_0  \\  \nonumber +& \frac{1}{\al} \sum_{i_1+i_2+i_3+i_4+i_5=i-1}\mathcal{D}^{i_1}\psi_g^{i_2}\mathcal{D}^{i_3}\modu \mathcal{D}^{i_4}(\eta,\etabar) \mathcal{D}^{i_5+1}\phi \\ \nonumber +&\sum_{i_1+i_2+i_3+i_4+1=i}\mathcal{D}^{i_1}\psi_g^{i_2}\mathcal{D}^{i_3}(\tr\chi,\chihat)\mathcal{D}^{i_4+1}\phi \\ \nonumber &+ \al \sumifim \dione \psi_g^{i_2} \dit(\chihat,\tr\chi)\dif \etabar \difi \phi \\ \nonumber +& \al \sumifm \dione \psi_g^{i_2}\dit (\beta,\slashed{T}_{4}) \dif \phi  \\  \nonumber+&\sum_{i_1+i_2+i_3+i_4+i_5=i-1}\mathcal{D}^{i_1}\psi_g^{i_2}\mathcal{D}^{i_3}\modu \mathcal{D}^{i_4}\sigma\mathcal{D}^{i_5}\phi \\ +& \sum_{i_1+\dots+i_6=i-1}\mathcal{D}^{i_1}\psi_g^{i_2}\mathcal{D}^{i_3}\modu \mathcal{D}^{i_4}(\eta,\etabar)\mathcal{D}^{i_5}(\eta,\etabar)\mathcal{D}^{i_6}\phi. \label{commutationformulanabla42}
\end{align}We recall that here $\psi_g\in \{\eta,\etabar,\modu\hsp \omegabar\}$. We use it in the equation
\begin{align} \notag 
\snabla_{4}\widetilde{\beta}=\slashed{\div} \alpha+Q,    
\end{align}
to yield 
\begin{align}
 \snabla_{4}\mathcal{D}^{i}\widetilde{\beta}=\slashed{\div} \mathcal{D}^{i}\alpha+Q^{1}_{i},    
\end{align}
where $Q^{1}_{i}$ reads 
\begin{align}
 \notag Q^{1}_{i}:=&\sum_{i_1+i_2+i_3=i}\mathcal{D}^{i_1}(\eta,\etabar,\modu\omegabar)^{i_2}\mathcal{D}^{i_3}Q  \\  \nonumber +& \frac{1}{\al} \sum_{i_1+i_2+i_3+i_4+i_5=i-1}\mathcal{D}^{i_1}\psi_g^{i_2}\mathcal{D}^{i_3}\modu \mathcal{D}^{i_4}(\eta,\etabar) \mathcal{D}^{i_5+1}\widetilde{\beta} \\ \nonumber +&\sum_{i_1+i_2+i_3+i_4+1=i}\mathcal{D}^{i_1}\psi_g^{i_2}\mathcal{D}^{i_3}(\tr\chi,\chihat)\mathcal{D}^{i_4+1}\widetilde{\beta} \\ \nonumber &+ \al \sumifim \dione \psi_g^{i_2} \dit(\chihat,\tr\chi)\dif \etabar \difi \widetilde{\beta} \\ \nonumber +& \al \sumifm \dione \psi_g^{i_2}\dit (\beta,\slashed{T}_{4}) \dif \widetilde{\beta}  \\  \nonumber+&\sum_{i_1+i_2+i_3+i_4+i_5=i-1}\mathcal{D}^{i_1}\psi_g^{i_2}\mathcal{D}^{i_3}\modu \mathcal{D}^{i_4}\sigma\mathcal{D}^{i_5}\widetilde{\beta} \\ \notag  +& \sum_{i_1+\dots+i_6=i-1}\mathcal{D}^{i_1}\psi_g^{i_2}\mathcal{D}^{i_3}\modu \mathcal{D}^{i_4}(\eta,\etabar)\mathcal{D}^{i_5}(\eta,\etabar)\mathcal{D}^{i_6}\widetilde{\beta}\\ \notag 
 &+\sum _{i_{1}+i_{2}+i_{3}+i_{4}=i-1}\mathcal{D}^{i_{1}}|u|\mathcal{D}^{i_{2}}(\tr\chibar,\chibarhat)\mathcal{D}^{i_{3}}\nablasl\mathcal{D}^{i_{4}}\alpha \\ \notag 
  &+\sum_{i_{1}+i_{2}+i_{3}=i-1}\mathcal{D}^{i_{1}}|u|\mathcal{D}^{i_{2}}(\betabar,T_{A3})\mathcal{D}^{i_{3}}\beta\\
   \notag &+\sum_{i_{1}+i_{2}=i}\mathcal{D}^{i_{1}}(\eta,\etabar)\mathcal{D}^{i_{2}}\alpha\\
   \notag &+\sum_{i_{1}+i_{2}+i_{3}+i_{4}=i-1}\mathcal{D}^{i_{1}}|u|\mathcal{D}^{i_{2}}(\chihat,\tr\chi)\mathcal{D}^{i_{3}}\etabar\mathcal{D}^{i_{4}}\alpha\\ 
  &+\sum_{i_{1}+i_{2}+i_{3}+i_{4}=i-1}\mathcal{D}^{i_{1}}|u|\mathcal{D}^{i_{2}}(\chibarhat,\tr\chibar)\mathcal{D}^{i_{3}}\eta\mathcal{D}^{i_{4}}\alpha,
\end{align}
for $\mathcal{D}=|u|\snabla_{3}$ and 
\begin{align}  \notag 
Q^{1}_{i}:=&\sum_{i_{1}+i_{2}=i-1}\mathcal{D}^{i_{1}}(\chihat,\tr\chi)\mathcal{D}^{i_{2}}\nablasl\alpha\\ \notag &+\sum_{i_{1}+i_{2}+i_{3}=i-1} \mathcal{D}^{i_{1}}(\chihat,\tr\chi)\mathcal{D}^{i_{2}}\etabar\mathcal{D}^{i_{3}}\alpha\\
&+\sum_{i_{1}+i_{2}=i-1}\mathcal{D}^{i_{1}}(\beta,\Tslash_{4})\mathcal{D}^{i_{2}}\alpha,   
\end{align}
for $\mathcal{D}=\snabla_{4}$.
We now want to commute $\nablasl^{j}$ such that $i+j=2$, obtaining
\begin{align}
\snabla_{4}\nablasl^{j}\mathcal{D}^{i}\widetilde{\beta}=\slashed{\div} \nablasl^{j}\mathcal{D}^{i}\alpha+Q_{i,j},
\end{align}
where $Q_{i,j}$ reads 
\begin{align} \notag 
 & Q_{i,j}=\sum_{j_{1}+j_{2}+j_{3}=j}\snabla^{j_{1}}(\eta+\etabar)^{j_{2}}\snabla^{j_{3}}Q^{1}_{i}+\sum_{j_{1}+j_{2}+j_{3}+j_{4}=j}\snabla^{j_{1}}(\eta+\etabar)^{j_{2}}\snabla^{j_{3}}(\chihat,\tr\chi)\snabla^{j_{4}}\mathcal{D}^{i}\widetilde{\beta}\\ \notag 
 &+\sum_{j_{1}+j_{2}+j_{3}+j_{4}=j}\snabla^{j_{1}}\psi^{j_{2}}_{g}\snabla^{j_{3}}(\psi_{g},\chihat)\snabla^{j_{4}}(\mathcal{D}^{i}\widetilde{\beta},\mathcal{D}^{i}\alpha)+\sum_{j_{1}+j_{2}+j_{3}+j_{4}=j}\snabla^{j_{1}}\psi^{j_{2}}_{g}\snabla^{j_{3}}(\chihat,\tr\chi)\snabla^{j_{4}}\mathcal{D}^{i}\widetilde{\beta}\\
 &+\sum_{j_{1}+j_{2}+j_{3}+j_{4}=j-2}\snabla^{j_{1}}\psi^{j_{2}+1}_{g}\snabla^{j_{3}}(\chihat,\tr\chi)\snabla^{j_{4}}\mathcal{D}^{i}\widetilde{\beta}.
\end{align}
Here, the terms containing $j-2$ vanish when $j-2<0$.
Collecting the principal terms, we have the coupled Bianchi system 
\begin{gather}
\snabla_{3}\nablasl^{j}\mathcal{D}^{i}\alpha+(\frac{j+1}{2})\tr\chibar \nablasl^{j}\mathcal{D}^{i}\alpha =\snabla\widehat{\otimes}\nablasl^{j}\mathcal{D}^{i}\tbeta+P_{i,j},\\
\snabla_{4}\nablasl^{j}\mathcal{D}^{i}\widetilde{\beta}=\slashed{\div} \nablasl^{j}\mathcal{D}^{i}\alpha+Q_{i,j}.
\end{gather}
The vital point to note here is that $s_{2}(\Psi)=s_{2}(\mathcal{D}\Psi)$ and therefore we can apply the energy lemma for the fields $\mathcal{D}^{i}\alpha$ and $\mathcal{D}^{i}\tbeta$, namely
\begin{align}
&\int_{\Hu}\scaletwoSuubarprime{\snabla^j \mathcal{D}^{i}\alpha}^2 \dubarprime + \int_{\Hbu}\frac{a}{\upr^2} \scaletwoSuprime{\snabla^j \mathcal{D}^{i}\tbeta}^2 \duprime  \notag \\ \lesssim &\int_{H_{u_{\infty}}^{(0,\ubar)}}||\snabla^j \mathcal{D}^{i}\alpha||^{2}_{L^{2}_{sc}(S_{u,\ubar^{\prime}})} \dubarprime + \int_{\Hbar_0^{(u_{\infty},u)}} \frac{a}{\upr^2} ||\snabla^j \mathcal{D}^{i}\tbeta||^{2}_{L^{2}_{sc}(S_{u,u^{\prime}})} \duprime \notag \\+&  \iint_{\mathcal{D}_{u,\ubar}}\frac{a}{\upr} \scaleoneSuprimeubarprime{\langle\snabla^j \mathcal{D}^{i} \alpha, P_{i,j}\rangle}\duprime \dubarprime + \iint_{\mathcal{D}_{u,\ubar}}\frac{a}{\upr} \scaleoneSuprimeubarprime{\langle\snabla^j \mathcal{D}^{i}\widetilde{\beta} ,Q_{i,j}\rangle}\duprime \dubarprime.
\end{align}
By H\"older, one has 
\begin{equation}
    \begin{split}
         &\intubar \intu  \frac{a}{\upr} \scaleoneSuprimeubarprime{a^{\frac{j}{2}} P_{i,j} (a^{\frac{1}{2}}\snabla)^{j}\mathcal{D}^{i}\alpha }  \duprime \dubarprime \\\leq&\intu \frac{a}{\upr^2} \sum_{k}^{}\left(  \intubar \scaletwoSuprimeubarprime{a^{\frac{j}{2}} P_{i,j}^k }^2\dubarprime \right)^{\frac{1}{2}}\duprime
        \cdot \sup_{u^{\prime}} \lVert (a^{\frac{1}{2}}\snabla)^{j}\mathcal{D}^{i}\alpha \rVert_{L^2_{(sc)}(H_{u^{\prime}}^{(0,\ubar)})}
    \end{split}
\end{equation}
We only control explicitly the borderline terms here, while the rest are of higher decay. Consider the following terms 
\begin{align} \notag 
&\int_{u_{\infty}}^{u}\frac{a}{|u^{\prime}|^{2}}\Bigg(\int_{0}^{\ubar}||a^{\frac{j}{2}}\sum_{j_{1}+j_{2}+j_{3}+j_{4}=j-1}\snabla^{j_{1}}(\eta+\etabar)^{j_{2}}\snabla^{j_{3}}\widetilde{\betabar}\snabla^{j_{4}}\mathcal{D}^{i}\alpha||^{2}_{L^{2}_{sc}(S_{u^{\prime},\ubar^{\prime}})}\text{d}\ubar^{\prime} \Bigg)^{\frac{1}{2}}\text{d}u^{\prime} \\
\lesssim& \int_{u_{\infty}}^{u}\frac{a}{|u^{\prime}|^{2}}\frac{1}{|u^{\prime}|}a^{\frac{1}{2}}\mathcal{R}^{2}\text{d}u^{\prime} +\int_{u_{\infty}}^{u}\frac{a}{|u^{\prime}|^{2}} a^{\frac{1}{2}}\frac{1}{|u^{\prime}|^{2}}a^{\frac{1}{2}}\mathcal{R}^{2}\Gamma \text{d}u^{\prime} \lesssim \frac{a^{\frac{3}{2}}\mathcal{R}^{2}}{|u|^{2}}+\frac{a^{2}\Gamma\mathcal{R}^{2}}{|u|^{3}},
\end{align}
where we have used $L^{\infty}_{sc}(S_{u^{\prime},\ubar^{\prime}})-L^{2}_{sc}(S_{u^{\prime},\ubar^{\prime}})$ and $L^{\infty}_{sc}(S_{u^{\prime},\ubar^{\prime}})-L^{4}_{sc}(S_{u^{\prime},\ubar^{\prime}})-L^{4}_{sc}(S_{u^{\prime},\ubar^{\prime}})$ estimates. Consider the next term
\begin{align} \notag 
&\int_{u_{\infty}}^{u}\frac{a}{|u^{\prime}|^{2}}\Bigg(\int_{0}^{\ubar}||a^{\frac{j}{2}}\sum_{j_{1}+j_{2}+j_{3}+j_{4}=j}\snabla^{j_{1}}(\eta+\etabar)^{j_{2}}\snabla^{j_{3}}\chibarhat\snabla^{j_{4}}\mathcal{D}^{i}\alpha||^{2}_{L^{2}_{sc}(S_{u^{\prime},\ubar^{\prime}})} \text{d}\ubar^{\prime} \Bigg)^{\frac{1}{2}}\text{d}u^{\prime} \\
\lesssim& \int_{u_{\infty}}^{u}\frac{a}{|u^{\prime}|^{2}}\frac{1}{|u^{\prime}|}\frac{|u^{\prime}|}{a^{\frac{1}{2}}}\Gamma a^{\frac{1}{2}}\mathcal{R}\text{d}u^{\prime} +\int_{u_{\infty}}^{u}\frac{a}{|u^{\prime}|^{2}}\frac{a^{\frac{1}{2}}}{|u^{\prime}|}\frac{\Gamma|u^{\prime}|}{a^{\frac{1}{2}}}\Gamma\lesssim \frac{a\Gamma \mathcal{R}}{|u|}+\frac{a\Gamma^{2}}{|u|}.
\end{align}
The next borderline term is estimated as 
\begin{align} \notag 
&\int_{u_{\infty}}^{u}\frac{a}{|u^{\prime}|^{2}}\Bigg(\int_{0}^{\ubar}||a^{\frac{j}{2}}\sum_{j_{1}+j_{2}+j_{3}+j_{4}=j}\snabla^{j_{1}}(\eta+\etabar)^{j_{2}}\snabla^{j_{3}}\chibarhat\snabla^{j_{4}}\mathcal{D}^{i}\alpha||^{2}_{L^{2}_{sc}(S_{u^{\prime},\ubar^{\prime}})} \text{d}\ubar^{\prime} \Bigg)^{\frac{1}{2}}\text{d}u^{\prime} \\
\lesssim& \int_{u_{\infty}}^{u}\frac{a}{|u^{\prime}|^{2}}a^{\frac{1}{2}} \frac{1}{|u^{\prime}|}\Gamma\mathcal{R}\text{d}u^{\prime} +\int_{u_{\infty}}^{u}\frac{a}{|u^{\prime}|^{2}}\frac{1}{|u^{\prime}|^{2}}\Gamma \frac{\Gamma |u^{\prime}|}{a^{\frac{1}{2}}} a^{\frac{1}{2}}\mathcal{R}\text{d}u^{\prime} \lesssim \frac{a^{\frac{3}{2}}\Gamma \mathcal{R}}{|u|^{2}}+\frac{a\Gamma^{2}\mathcal{R}}{|u|^{2}}.
\end{align}
The next terms are estimated as (recall that $i+j=2$):
\begin{align}
     \notag &\int_{u_{\infty}}^{u}\frac{a}{|u^{\prime}|^{2}}\Bigg(\int_{0}^{\ubar}||a^{\frac{j}{2}}\sum_{\substack{
j_{1}+j_{2}+j_{3}+j_{4}=j \\
i_{1}+i_{2}+i_{3}+i_{4}=i-1 \\
k_{1}+k_{2}+k_{3}=j_{3}
}}
\snabla^{j_{1}}(\eta+\etabar)^{j_{2}}
\snabla^{k_{1}}\mathcal{D}^{i_{1}}\psi_{g}^{i_{2}}
\snabla^{k_{2}}\mathcal{D}^{i_{3}}\nablasl\etabar\,
\snabla^{k_{3}}\mathcal{D}^{i_{4}}\alpha||^{2}_{L^{2}_{sc}(S_{u^{\prime},\ubar^{\prime}})} \text{d}\ubar^{\prime} \Bigg)^{\frac{1}{2}}\text{d}u^{\prime} \\
\lesssim& \int_{u_{\infty}}^{u}\frac{a}{|u^{\prime}|^{2}}\frac{a^{\frac{1}{2}}}{|u^{\prime}|}\frac{\Gamma}{a}a^{\frac{1}{2}}\mathcal{R}\text{d}u^{\prime} +\int_{u_{\infty}}^{u}\frac{a}{|u^{\prime}|^{2}}\frac{a^{\frac{1}{2}}\Gamma}{|u^{\prime}|^{2}}\frac{\Gamma}{a^{\frac{1}{2}}}a^{\frac{1}{2}}\mathcal{R}\text{d}u^{\prime} \lesssim \frac{a\Gamma\mathcal{R}}{|u|^{2}}+\frac{a^{\frac{3}{2}}\Gamma^{2}\mathcal{R}}{|u|^{3}},
\end{align}
\begin{align} \notag 
&\int_{u_{\infty}}^{u}\frac{a}{|u^{\prime}|^{2}}\Bigg(\int_{0}^{\ubar}||a^{\frac{j}{2}}\sum_{\substack{
j_{1}+j_{2}+j_{3}+j_{4}=j \\
i_{1}+i_{2}+i_{3}+i_{4}=i-1 \\
k_{1}+k_{2}=j_{3}
}}
\snabla^{j_{1}}(\eta+\etabar)^{j_{2}}
\snabla^{k_{1}}\mathcal{D}^{i_{1}}(\eta,\etabar,1)^{i_{2}}
\snabla^{k_{2}}\mathcal{D}^{i_{3}}\nablasl\widetilde{\beta}||^{2}_{L^{2}_{sc}(S_{u^{\prime},\ubar^{\prime}})} \text{d}\ubar^{\prime} \Bigg)^{\frac{1}{2}}\text{d}u^{\prime} \\
\lesssim& \int_{u_{\infty}}^{u}\frac{a}{|u^{\prime}|^{2}}a^{\frac{1}{2}}\frac{\mathcal{R}}{a}\text{d}u^{\prime} +\int_{u_{\infty}}^{u}\frac{a}{|u^{\prime}|^{2}}a^{\frac{1}{2}}\frac{1}{|u^{\prime}|}\Gamma\frac{\mathcal{R}}{a^{\frac{1}{2}}}\text{d}u^{\prime} \lesssim \frac{a^{\frac{1}{2}}\mathcal{R}}{|u|}+\frac{a\Gamma\mathcal{R}}{|u|^{2}},
\end{align}

\begin{align} \notag
&\int_{u_{\infty}}^{u}\frac{a}{|u^{\prime}|^{2}}\Bigg(\int_{0}^{\ubar}||a^{\frac{j}{2}}\sum_{\substack{
j_{1}+j_{2}+j_{3}=j \\
i_{1}+i_{2}+i_{3}+i_{4}=i-1 \\
k_{1}+k_{2}+k_{3}=j_{3}
}}
\snabla^{j_{1}}(\eta+\etabar)^{j_{2}}
\snabla^{k_{1}}\mathcal{D}^{i_{1}}|u|
\snabla^{k_{2}}\mathcal{D}^{i_{2}}(\chibarhat,\tr\chibar)
\snabla^{k_{3}}\mathcal{D}^{i_{3}}\eta\snabla^{k_{4}}\mathcal{D}^{i_{4}}\widetilde{\beta}||^{2}_{L^{2}_{sc}(S_{u^{\prime},\ubar^{\prime}})} \text{d}\ubar^{\prime} \Bigg)^{\frac{1}{2}}\text{d}u^{\prime} \\
\lesssim& \int_{u_{\infty}}^{u}\frac{a}{|u^{\prime}|^{2}}\frac{a}{|u^{\prime}|}|u^{\prime}|\frac{1}{|u^{\prime}|^{3}}\frac{|u^{\prime}|^{2}}{a}\Gamma^{2}\mathcal{R}a^{-\frac{1}{2}}\text{d}u^{\prime} \lesssim \frac{a^{\frac{1}{2}}}{|u|^{2}}\Gamma^{2}\mathcal{R},
\end{align}

\begin{align} \notag 
&\int_{u_{\infty}}^{u}\frac{a}{|u^{\prime}|^{2}}\Bigg(\int_{0}^{\ubar}||a^{\frac{j}{2}}\sum_{\substack{
j_{1}+j_{2}+j_{3}=j \\
i_{1}+i_{2}+i_{3}+i_{4}=i-1 \\
k_{1}+k_{2}+k_{3}=j_{3}
}}
\snabla^{j_{1}}(\eta+\etabar)^{j_{2}}
\snabla^{k_{1}}\mathcal{D}^{i_{1}}|u|
\snabla^{k_{2}}\mathcal{D}^{i_{2}}(\chibarhat,\tr\chibar)
\snabla^{k_{3}}\mathcal{D}^{i_{3}}\nablasl\mathcal{D}^{i_{4}}\widetilde{\beta}||^{2}_{L^{2}_{sc}(S_{u^{\prime},\ubar^{\prime}})} \text{d}\ubar^{\prime} \Bigg)^{\frac{1}{2}}\text{d}u^{\prime} \\    \lesssim& \int_{u_{\infty}}^{u}\frac{a}{|u^{\prime}|^{2}}\frac{a}{|u^{\prime}|}|u^{\prime}|\frac{1}{|u^{\prime}|^{2}}\frac{|u^{\prime}|^{2}\Gamma}{a}\frac{\mathcal{R}}{a^{\frac{1}{2}}}\text{d}u^{\prime} 
\lesssim \frac{a^{\frac{1}{2}}\Gamma\mathcal{R}}{|u|},
\end{align}

\begin{align} \notag 
&\int_{u_{\infty}}^{u}\frac{a}{|u^{\prime}|^{2}}\Bigg(\int_{0}^{\ubar}||a^{\frac{j}{2}}\sum_{\substack{
j_{1}+j_{2}+j_{3}=j \\
i_{1}+i_{2}+i_{3}=i \\
k_{1}+k_{2}+k_{3}=j_{3}
}}
\snabla^{j_{1}}(\eta+\etabar)^{j_{2}}\snabla^{k_{1}}\mathcal{D}^{i_{1}}(\eta,\etabar,1)^{i_{2}}\snabla^{k_{2}}\mathcal{D}^{i_{3}}\snabla_{4}T_{AB} 
 ||^{2}_{L^{2}_{sc}(S_{u^{\prime},\ubar^{\prime}})} \text{d}\ubar^{\prime} \Bigg)^{\frac{1}{2}}\text{d}u^{\prime} \\
 \lesssim& \int_{u_{\infty}}^{u}\frac{a}{|u^{\prime}|^{2}}\frac{a^{\frac{1}{2}}}{|u^{\prime}|}\frac{\Gamma}{a^{\frac{1}{2}}}\frac{\mathcal{V}}{|u^{\prime}|}\text{d}u^{\prime} \lesssim \frac{a\Gamma\mathcal{V}}{|u|^{3}},
\end{align}
where we control the second derivative term $\snabla\mathcal{D}(\eta,\etabar)$ in $L^{2}_{sc}(S_{u,\ubar})$ and $\snabla_{4}T_{AB}$ in $L^{\infty}_{sc}(S_{u,\ubar})$ ,which in turn is controllable by means of $||\sum_{i\leq 2}(a^{\frac{1}{2}}\snabla)^{i}\snabla_{4}T_{AB}||_{L^{2}_{sc}(S_{u,\ubar})}$ by means of scale-invariant Sobolev embedding. The remaining terms enjoy similar or better estimates. Consequently, collecting all the terms, we have the following estimate for the $P_{i}$ terms
\begin{align}
\intubar \intu  \frac{a}{\upr} \scaleoneSuprimeubarprime{a^{\frac{j}{2}} P_{i,j} (a^{\frac{1}{2}}\snabla)^{j}\mathcal{D}^{i}\alpha }  \duprime \dubarprime\lesssim \Gamma^{2}+\Gamma \mathcal{R}+1.
\end{align}
Now we control the remaining error term 
\begin{equation}
    \begin{split}
         &\intubar \intu  \frac{a}{\upr} \scaleoneSuprimeubarprime{a^{\frac{j}{2}} Q_{i,j} \cdot (a^{\frac{1}{2}}\snabla)^{j}\tbeta }  \duprime \dubarprime \\\leq&\sum_{k} \intu \frac{a}{\upr^2} \left(  \intubar \scaletwoSuprimeubarprime{a^{\frac{j}{2}} Q^{k}_{i,j} }^2\dubarprime \right)^{\frac{1}{2}}\duprime
        \cdot \sup_{u^{\prime}} \lVert (a^{\frac{1}{2}}\snabla)^{j} \tbeta \rVert_{L^2_{(sc)}(\Hb_{\ubar^{\prime}}^{(u_{\infty},u)})}.
    \end{split}
\end{equation}
Now we control the worst terms, and the remaining terms verify better or similar estimates. 
\begin{align}
\intu \frac{a}{\upr^2} \left(  \intubar \scaletwoSuprimeubarprime{a^{\frac{j}{2}} Q^{k}_{i,j} }^2\dubarprime \right)^{\frac{1}{2}}\duprime.    
\end{align}
Again, we estimate term by term and provide the estimates for the terms with least decay
\begin{align}  \notag 
&\int_{u_{\infty}}^{u}\frac{a\hsp \text{d}u^{\prime}}{|u^{\prime}|^{2}}\Bigg(\int_{0}^{\ubar}||a^{\frac{j}{2}}\sum_{\substack{
j_{1}+j_{2}+j_{3}=j \\
i_{1}+i_{2}+i_{3}+i_{4}=i-1 \\
k_{1}+k_{2}+k_{3}=j_{3}
}}
\snabla^{j_{1}}(\eta+\etabar)^{j_{2}}
\snabla^{k_{1}}\mathcal{D}^{i_{1}}|u|
\snabla^{k_{2}}\mathcal{D}^{i_{2}}(\chibarhat,\tr\chibar)
\snabla^{k_{3}}\mathcal{D}^{i_{3}}\nablasl\mathcal{D}^{i_{4}}\alpha||^{2}_{L^{2}_{sc}(S_{u^{\prime},\ubar^{\prime}})} \text{d}\ubar^{\prime} \Bigg)^{\frac{1}{2}} \\
&\lesssim \int_{u_{\infty}}^{u}\frac{a}{|u^{\prime}|^{2}}\frac{1}{|u^{\prime}|^{2}}\frac{a}{|u^{\prime}|}|u^{\prime}|\frac{|u^{\prime}|^{2}\Gamma}{a}\mathcal{R}\text{d}u^{\prime} \lesssim \frac{a\Gamma\mathcal{R}}{|u|},
\end{align}

\begin{align} \notag 
 &\int_{u_{\infty}}^{u}\frac{a}{|u^{\prime}|^{2}}\Bigg(\int_{0}^{\ubar}||a^{\frac{j}{2}}\sum_{j_{1}+j_{2}+j_{3}+j_{4}=j}\snabla^{j_{1}}\psi^{j_{2}}_{g}\snabla^{j_{3}}(\chihat,\psi_{g})\snabla^{j_{4}}(\mathcal{D}^{i}\tbeta,\mathcal{D}^{i}\alpha)||^{2}_{L^{2}_{sc}(S_{u^{\prime},\ubar^{\prime}})} \text{d}\ubar^{\prime} \Bigg)^{\frac{1}{2}}\text{d}u^{\prime} \\ \lesssim&
 \int_{u_{\infty}}^{u}\frac{a}{|u^{\prime}|^{2}}\frac{1}{|u^{\prime}|}a^{\frac{1}{2}}\Gamma \mathcal{R}\text{d}u^{\prime} \lesssim \frac{a^{\frac{3}{2}}}{|u|^{2}}\Gamma\mathcal{R},
\end{align}

\begin{align} \notag 
&\int_{u_{\infty}}^{u}\frac{a}{|u^{\prime}|^{2}}\Bigg(\int_{0}^{\ubar}||a^{\frac{j}{2}}\sum_{j_{1}+j_{2}+j_{3}+j_{4}=j}\snabla^{j_{1}}\psi^{j_{2}}_{g}\snabla^{j_{3}}(\chihat,\tr\chi)\snabla^{j_{4}}\mathcal{D}^{i}\widetilde{\beta}||^{2}_{L^{2}_{sc}(S_{u^{\prime},\ubar^{\prime}})} \text{d}\ubar^{\prime} \Bigg)^{\frac{1}{2}}\text{d}u^{\prime} \\
\lesssim& \int_{u_{\infty}}^{u}\frac{a}{|u^{\prime}|^{2}}\frac{a^{\frac{1}{2}}}{|u^{\prime}|}a^{\frac{1}{2}}\Gamma \frac{\mathcal{R}}{a^{\frac{1}{2}}}\text{d}u^{\prime} \lesssim \frac{a^{\frac{3}{2}}\Gamma\mathcal{R}}{|u|^{2}},
\end{align}

\begin{align} \notag 
&\int_{u_{\infty}}^{u}\frac{a}{|u^{\prime}|^{2}}\Bigg(\int_{0}^{\ubar}||a^{\frac{j}{2}}\sum_{\substack{
j_{1}+j_{2}+j_{3}=j \\
i_{1}+i_{2}=i \\
k_{1}+k_{2}=j_{3}
}}
\snabla^{j_{1}}(\eta+\etabar)^{j_{2}}
\snabla^{k_{1}}\mathcal{D}^{i_{1}}(\eta,\etabar)
\snabla^{k_{2}}\mathcal{D}^{i_{2}}\alpha||^{2}_{L^{2}_{sc}(S_{u^{\prime},\ubar^{\prime}})} \text{d}\ubar^{\prime} \Bigg)^{\frac{1}{2}}\text{d}u^{\prime} \\
\lesssim& \int_{u_{\infty}}^{u}\frac{a}{|u^{\prime}|^{2}}\frac{a^{\frac{1}{2}}\Gamma}{|u^{\prime}|}\mathcal{R}\text{d}u^{\prime} \lesssim  \frac{a^{\frac{3}{2}}\Gamma\mathcal{R}}{|u|^{2}},
\end{align}

\begin{align}  \notag 
&\int_{u_{\infty}}^{u}\frac{a}{|u^{\prime}|^{2}}\Bigg(\int_{0}^{\ubar}||a^{\frac{j}{2}}\sum_{\substack{
j_{1}+j_{2}+j_{3}=j \\
i_{1}+i_{2}=i \\
k_{1}+k_{2}=j_{3}
}}
\snabla^{j_{1}}(\eta+\etabar)^{j_{2}}
\snabla^{k_{1}}\mathcal{D}^{i_{1}}(\eta,\etabar)
\snabla^{k_{2}}\mathcal{D}^{i_{2}}\alpha||^{2}_{L^{2}_{sc}(S_{u^{\prime},\ubar^{\prime}})} \text{d}\ubar^{\prime} \Bigg)^{\frac{1}{2}}\text{d}u^{\prime} \\
\lesssim& \int_{u_{\infty}}^{u}\frac{a}{|u^{\prime}|^{2}}\frac{a^{\frac{1}{2}}\Gamma}{|u^{\prime}|}\mathcal{R}\text{d}u^{\prime} \lesssim  \frac{a^{\frac{3}{2}}\Gamma\mathcal{R}}{|u|^{2}},  
\end{align}
\begin{align}  \notag 
&\int_{u_{\infty}}^{u}\frac{a}{|u^{\prime}|^{2}}\Bigg(\int_{0}^{\ubar}||a^{\frac{j}{2}}\sum_{\substack{
j_{1}+j_{2}+j_{3}=j \\
i_{1}+i_{2}+i_{3}+i_{4}=i-1 \\
k_{1}+k_{2}+k_{3}+k_{4}=j_{3}
}}
\snabla^{j_{1}}(\eta+\etabar)^{j_{2}}
\snabla^{k_{1}}\mathcal{D}^{i_{1}}|u^{\prime}|\snabla^{k_{2}}\mathcal{D}^{i_{2}}(\chibarhat,\tr\chibar)\snabla^{k_{3}}\mathcal{D}^{i_{3}}\eta\nabla^{{k_{4}}}\mathcal{D}^{i_{4}}\alpha||^{2}_{L^{2}_{sc}(S_{u^{\prime},\ubar^{\prime}})} \text{d}\ubar^{\prime} \Bigg)^{\frac{1}{2}}\text{d}u^{\prime} \\
\lesssim&\int_{u_{\infty}}^{u}\frac{a}{|u^{\prime}|^{2}}\frac{1}{|u^{\prime}|^{3}}\frac{a}{|u^{\prime}|}|u^{\prime}|\frac{|u^{\prime}|^{2}\Gamma}{a}\Gamma a^{\frac{1}{2}}\mathcal{R}\text{d}u^{\prime} \lesssim \frac{a^{\frac{3}{2}}\Gamma^{2}\mathcal{R}}{|u|^{2}}, 
\end{align}
while the worst matter term  (the one that contains $T_{44}$ that has the minimum decay among the stress-energy components) is controlled by
\begin{align}  \notag 
&\int_{u_{\infty}}^{u}\frac{a}{|u^{\prime}|^{2}}\Bigg(\int_{0}^{\ubar}||a^{\frac{j}{2}}\sum_{\substack{
j_{1}+j_{2}+j_{3}=j \\
i_{1}+i_{2}+i_{3}=i \\
k_{1}+k_{2}=j_{3}
}}
\snabla^{j_{1}}(\eta+\etabar)^{j_{2}}
\snabla^{k_{1}}\mathcal{D}^{i_{1}}(\eta,\etabar,|u^{\prime}|\omegabar)^{i_{2}}\snabla^{k_{2}}\mathcal{D}^{i_{3}}(\snabla_{4}T_{4A},\nablasl T_{44})||^{2}_{L^{2}_{sc}(S_{u^{\prime},\ubar^{\prime}})} \text{d}\ubar^{\prime} \Bigg)^{\frac{1}{2}}\text{d}u^{\prime}  \\
\lesssim& \int_{u_{\infty}}^{u}\frac{a}{|u^{\prime}|^{2}}\frac{\mathcal{V}}{a^{\frac{1}{2}}|u^{\prime}|}\text{d}u^{\prime} +\int_{u_{\infty}}^{u}\frac{a}{|u^{\prime}|^{2}}\frac{1}{|u^{\prime}|}\frac{1}{|u^{\prime}|}\frac{a}{|u^{\prime}|}|u^{\prime}|\Gamma\frac{\mathcal{V}}{a^{\frac{1}{2}}|u^{\prime}|}\text{d}u^{\prime} \lesssim \frac{a^{\frac{1}{2}}\mathcal{V}}{|u|^{2}}+\frac{a^{\frac{3}{2}}\Gamma\mathcal{V}}{|u|^{4}}.
\end{align}
Collecting, finally, all the terms, we estimate 
\begin{align}
\intubar \intu  \frac{a}{\upr} \scaleoneSuprimeubarprime{a^{\frac{j}{2}} Q_{i,j} \cdot (a^{\frac{1}{2}}\snabla)^{j}\tbeta }  \duprime \dubarprime\lesssim \Gamma\mathcal{R}+1.    
\end{align}
Therefore we obtain 
\begin{equation}
    \begin{split}
        &\frac{1}{\al} \scaletwoHu{\mathcal{D}^{i} \alpha} + \frac{1}{\al}\scaletwoHbaru{\mathcal{D}^{i} \tbeta}\\& \lesssim  \frac{1}{\al} \scaletwoHzero{\mathcal{D}^{i} \alpha} + \frac{1}{\al}\scaletwoHbarzero{\mathcal{D}^{i} \tbeta} + \frac{1}{a^{\frac{1}{3}}}, 
    \end{split}
\end{equation}
which completes the proof for $\mathcal{D}=|u|\snabla_{3}$ and $\mathcal{D}=a^{\frac{1}{2}}\nablasl$ (their joint appearance).

\end{proof}

\newpage 
\begin{proposition}
\label{prop:energy_estimates_curvature_pairs}
Let the hypotheses of Theorem~\ref{mainone} and the bootstrap assumptions \eqref{bootstrap} be in force on the spacetime region under consideration.  
Consider an ordered pair of null curvature components 
\[
(\Psi_{1}, \Psi_{2}) \in \big\{\, (\widetilde{\beta}, (\rho,\sigma)), \; ((\rho,\sigma), \tbeta), \; (\tbeta, \alphabar) \,\big\}.
\]  
Let $\mathcal{D} \in \big\{ \, |u|\, \snabla_{3}, \, a^{\frac12} \snabla, \, \snabla_{4} \, \big\}$ denote one of the admissible commutation operators introduced in Section~\ref{commute}, and let $i$ be an integer satisfying $0 \leq i \leq 2$.
\vspace{3mm}

\noindent Then the following \emph{scale-invariant $L^{2}$ energy inequality} holds:
\begin{equation}
\label{eq:energy_estimate_curvature_pairs}
\scaletwoHu{\mathcal{D}^{i} \Psi_{1}}
\;+\;
\scaletwoHbaru{\mathcal{D}^{i} \Psi_{2}}
\;\; \leq \;\;
\scaletwoHzero{\mathcal{D}^{i} \Psi_{1}}
\;+\;
\scaletwoHbarzero{\mathcal{D}^{i} \Psi_{2}}
\;+\;
\frac{C}{a^{\frac13}}.
\end{equation}

\noindent In particular, for the uncommuted case $i=0$, \eqref{eq:energy_estimate_curvature_pairs} reads
\begin{equation}
\label{eq:energy_estimate_curvature_pairs_nocomm}
\scaletwoHu{\Psi_{1}}
\;+\;
\scaletwoHbaru{\Psi_{2}}
\;\; \leq \;\;
\scaletwoHzero{\Psi_{1}}
\;+\;
\scaletwoHbarzero{\Psi_{2}}
\;+\;
\frac{C}{a^{\frac13}}.
\end{equation}

Here:
\begin{itemize}
    \item $\scaletwoHu{\cdot}$ and $\scaletwoHbaru{\cdot}$ denote the scale-invariant $L^{2}$-flux norms along the outgoing and incoming null hypersurfaces $H_{u}$ and $\underline{H}_{\ubar}$, respectively.
    \item $\mathcal{D}^{i}$ denotes the $i$-fold composition of $\mathcal{D}$.
\end{itemize}
\end{proposition}

\begin{proof}
The schematic equations for $\Psi_1, \Psi_2$ are:

\begin{align}
       \nonumber  \snabla_3 \Psi_1 + \left( \frac{1}{2} + s_2(\Psi_1)\right) \tr\chibar \Psi_1 - \mathcal{D} \Psi_2 =& (\psi,\chihat)\Psi + (\snabla T_{43},\snabla T_{3B},\snabla T_{33},\snabla_{4}T_{A3},\snabla_{3}T_{43},\snabla_{4}T_{33},\snabla_{3}T_{A3})\\ +&(\psi,\chihat,\chibarhat,\tr\chibar)(T_{43},T_{A3},T_{33}),
    \end{align}

\begin{equation}
    \begin{split}
        \snabla_4 \Psi_2- \Hodge{\mathcal{D}}\Psi_1 = (\psi,\chibarhat)(\Psi_u, \alpha) +(\nablasl T_{4A},\snabla_{3}T_{44},\snabla_{4}T_{43},\nablasl T_{43},\snabla_{3}T_{4A},\snabla_{3}T_{AB},\snabla_{A}T_{3B})\\ +(\psi,\chihat,\chibarhat,\tr\chibar)(T_{4A},T_{44},T_{43},T_{AB},T_{A3})
    \end{split}
\end{equation}Commuting with $i$ angular derivatives, for $\Psi_1$, we have:

\begin{equation}
    \begin{split}
        &\snabla_3 \snabla^i \Psi_1 + \left(\frac{i+1}{2}+s_2(\Psi_1)\right)\tr\chibar \snabla^i \Psi_1 - \mathcal{D}\snabla^i \Psi_2 \\= &\sum_{i_{1}+i_{2}+i_{3}=i-1}\snabla^{i_{1}}\psi^{i_{2}+1}_{g}\snabla^{i_{3}+1}\Psi_{2} +\sum_{i_{1}+i_{2}+i_{3}+i_{4}=i}\snabla^{i_{1}}\psi^{i_{2}}_{g}\snabla^{i_{3}}(\psi_{g},\chihat)\snabla^{i_{4}}\Psi_{2} \\ &+\sum_{i_{1}+i_{2}+i_{3}=i}\snabla^{i_{1}}\psi^{i_{2}}_{g}\snabla^{i_{3}+1}(T_{43},T_{3B},T_{33}) +\sum_{i_{1}+i_{2}+i_{3}=i}\snabla^{i_{1}}\psi^{i_{2}}_{g}\snabla^{i_{3}}\snabla_{4}(T_{A3},T_{33})\\
       & +\sum_{i_{1}+i_{2}+i_{3}=i}\snabla^{i_{1}}\psi^{i_{2}}_{g}\snabla^{i_{3}}\snabla_{3}(T_{43},T_{A3})+
       \sum_{i_{1}+i_{2}+i_{3}+i_{4}=i}\snabla^{i_{1}}\psi^{i_{2}}_{g}\snabla^{i_{3}}(\psi_{g},\chihat,\chibarhat,\tr\chibar)\snabla^{i_{4}}(T_{4A},T_{43})\\
       &+\sum_{i_{1}+i_{2}+i_{3}+i_{4}=i}\snabla^{i_{1}}\psi^{i_{2}}_{g}\snabla^{i_{3}}(\psi_{g},\chihat,\chibarhat,\tr\chibar)\snabla^{i_{4}}(T_{4A},T_{43})\\
       &+\sum_{i_{1}+i_{2}+i_{3}+i_{4}=i}\snabla^{i_{1}}\psi^{i_{2}}_{g}\snabla^{i_{3}}(\chibarhat,\widetilde{\tr\chibar})\snabla^{i_{4}}\Psi_{1}+\sum_{i_{1}+i_{2}+i_{3}+i_{4}=i}\snabla^{i_{1}}\psi^{i_{2}}_{g}\snabla^{i_{3}}(\chibarhat,\widetilde{\tr\chibar})\snabla^{i_{4}}\Psi_{1}\\
       &+\sum_{i_{1}+i_{2}+i_{3}+i_{4}=i-1}\snabla^{i_{1}}\psi^{i_{2}+1}_{g}\snabla^{i_{3}}\tr\chibar\snabla^{i_{4}}\Psi_{1}+\sum_{i_{1}+i_{2}+i_{3}+i_{4}=i-1}\snabla^{i_{1}}\psi^{i_{2}+1}_{g}\snabla^{i_{3}}(\chibarhat,\tr\chibar)\snabla^{i_{4}}\Psi_{1}\\
       &+\sum_{i_{1}+i_{2}+i_{3}+i_{4}=i-1}\snabla^{i_{1}}\psi^{i_{2}+1}_{g}\snabla^{i_{3}}T_{3A}\snabla^{i_{4}}\Psi_{1}:= P_i.
    \end{split}
\end{equation}Analogously, for $\Psi_2$, we have 

\begin{equation}
    \begin{split}
        &\snabla_4 \snabla^i \Psi_2 -\Hodge{ \mathcal{D}} \snabla^i \Psi_1 \\= &\sum_{i_{1}+i_{2}+i_{3}=i-1}\snabla^{i_{1}}\psi^{i_{2}+1}_{g}\snabla^{i_{3}+1}\Psi_{1}+\sum_{i_{1}+i_{2}+i_{3}+i_{4}=i}\snabla^{i_{1}}\psi^{i_{2}}_{g}\snabla^{i_{3}}(\psi_{g},\chibarhat)\snabla^{i_{4}}(\Psi_{1},\alpha)\\
        &+\sum_{i_{1}+i_{2}+i_{3}+i_{4}=i}\snabla^{i_{1}}\psi^{i_{2}}_{g}\snabla^{i_{3}+1}(T_{4A},T_{3A},T_{43})+\sum_{i_{1}+i_{2}+i_{3}+i_{4}=i}\snabla^{i_{1}}\psi^{i_{2}}_{g}\snabla^{i_{3}}(\psi_{g},\chihat,\chibarhat,\tr\chibar)\snabla^{i_{4}}(T_{4A},T_{3A},T_{43})\\
        &+\sum_{i_{1}+i_{2}+i_{3}+i_{4}=i}\snabla^{i_{1}}\psi^{i_{2}}_{g}\snabla^{i_{3}}\snabla_{4}(T_{43})+\sum_{i_{1}+i_{2}+i_{3}+i_{4}=i}\snabla^{i_{1}}\psi^{i_{2}}_{g}\snabla^{i_{3}}\snabla_{3}(T_{44},T_{4A},T_{AB})\\
        &+\sum_{i_{1}+i_{2}+i_{3}+i_{4}=i}\snabla^{i_{1}}\psi^{i_{2}}_{g}\snabla^{i_{3}}(\chihat,\tr\chi)\snabla^{i_{4}}\Psi_{2}+\sum_{i_{1}+i_{2}+i_{3}+i_{4}=i-2}\snabla^{i_{1}}\psi^{i_{2}+1}_{g}\snabla^{i_{3}}(\chihat,\tr\chi)\snabla^{i_{4}}\Psi_{2}\\
        &\sum_{i_{1}+i_{2}+i_{3}+i_{4}=i-1}\snabla^{i_{1}}\psi^{i_{2}+1}_{g}\snabla^{i_{3}}T_{4A}\snabla^{i_{4}}\Psi_{2}:= Q_i.
    \end{split}
\end{equation}
Making use of Proposition \ref{42} once again, we arrive at 

\begin{equation}\label{mainpsu}
    \begin{split}
        &\scaletwoHu{\aln \Psi_1}^2 + \scaletwoHbaru{\aln \Psi_2}^2 \\ \lesssim & \scaletwoHzero{\aln \Psi_1}^2 + \scaletwoHbarzero{\aln \Psi_2}^2 \\&+  \intubar \intu  \frac{a}{\upr} \scaleoneSuprimeubarprime{a^{\frac{i}{2}} P_i \cdot \aln \Psi_1 }  \duprime \dubarprime
        \\ &+\intubar \intu  \frac{a}{\upr} \scaleoneSuprimeubarprime{a^{\frac{i}{2}} Q_i \cdot \aln \Psi_2 }  \duprime \dubarprime .
        \end{split}
\end{equation}
For the first spacetime integral in the above, we estimate

\begin{equation}
    \begin{split}
       & \intubar \intu  \frac{a}{\upr} \scaleoneSuprimeubarprime{a^{\frac{i}{2}} P_i \cdot \aln \Psi_1 }  \duprime \dubarprime \\\lesssim & \intu \frac{a}{\upr^2} \left( \intubar \scaletwoSuprimeubarprime{a^{\frac{i}{2}} P_i}^2 \dubarprime \right)^{\frac{1}{2}} \duprime \cdot \left( \intubar \scaletwoSuprimeubarprime{\aln \Psi_1}^2 \dubarprime \right)^{\frac{1}{2}} \duprime.
    \end{split}
\end{equation}For the first term:

\begin{equation}
    \begin{split}
        \intu \frac{a}{\upr^2} \left( \intubar \scaletwoSuprimeubarprime{a^{\frac{i}{2}} \sumitm \snabla^{i_{1}}\psi^{i_{2}+1}_{g} \snabla^{i_3+1}\Psi_2 }^2 \dubarprime \right)^{\frac{1}{2}} \duprime\lesssim \frac{\mathcal{R}}{|u|^{\frac{1}{2}}}+\frac{a \Gamma^2}{|u|^2} \lesssim 1.
    \end{split}
\end{equation}
For the rest of the terms, we estimate using the same philosophy as appropriate. There holds  
\begin{equation}
    \begin{split}
        &\intu \frac{a}{\upr^2} \left( \intubar \scaletwoSuprimeubarprime{a^{\frac{i}{2}} \sumif \snabla^{i_{1}}\psi^{i_{2}}_{g} \snabla^{i_3}(\psi_g,\chihat) \snabla^{i_{4}}\Psi }^2 \dubarprime \right)^{\frac{1}{2}} \duprime \\ \lesssim & \frac{\al \Gamma (R+\Gamma)}{\lvert u \rvert}\lesssim 1.
    \end{split}
\end{equation}
Next we estimate
\begin{equation}
    \begin{split}
         &\intu \frac{a}{\upr^2} \left( \intubar \scaletwoSuprimeubarprime{a^{\frac{i}{2}} \sumifim \snabla^{i_{1}}\psi^{i_{2}}_{g} \snabla^{i_{3+1}} (\chibarhat,\tr\chibar) \snabla^{i_{4}} \Psi_1}^2 \dubarprime \right)^{\frac{1}{2}} \duprime \\ \lesssim& \lesssim \frac{a}{\lvert u \rvert}\Gamma[\eta,\etabar] \Gamma[\tr\chibar]\Gamma[\Psi_1] \lesssim \frac{a}{\lvert u \rvert}\scaletwoSu{\aln(\tbeta,\tbetabar)} \cdot 1 \cdot (\mathcal{R}[\alpha]+1)\\ \lesssim & (\mathcal{R}[\alpha]+1)^2 \lesssim 1.
         \end{split}
\end{equation}
Now we estimate the source terms (we only provide the worst estimates) 
\begin{align} \notag 
        &\intu \frac{a}{\upr^2} \left( \intubar \scaletwoSuprimeubarprime{a^{\frac{i}{2}} \sumitm \snabla^{i_{1}}\psi^{i_{2}}_{g} \snabla^{i_3+1}(T_{43},T_{33},T_{A3}}^2 \dubarprime \right)^{\frac{1}{2}} \duprime\\
        \lesssim& \int_{u_{\infty}}^{u}\frac{a}{|u^{\prime}|^{2}}\frac{a\mathcal{V}}{a^{\frac{5}{2}}|u^{\prime}|}\text{d}u^{\prime} +\int_{u_{\infty}}^{u}\frac{a}{|u^{\prime}|^{2}}\frac{1}{|u^{\prime}|}\frac{a\mathcal{V}\Gamma}{a^{2}|u^{\prime}|}\text{d}u^{\prime} \lesssim \frac{\mathcal{V}}{a^{\frac{1}{2}}|u|^{2}}+\frac{\Gamma\mathcal{V}}{|u|^{3}}\lesssim 1,
\end{align}
\begin{align}
\notag  &\intu \frac{a}{\upr^2} \left( \intubar \scaletwoSuprimeubarprime{a^{\frac{i}{2}} \sum_{i_{1}+i_{2}+i_{3}=i}\snabla^{i_{1}}\psi^{i_{2}}_{g}\snabla^{i_{3}}\snabla_{4}(T_{A3},T_{33})}^2 \dubarprime \right)^{\frac{1}{2}} \duprime \\
\lesssim& \int_{u_{\infty}}^{u}\frac{a}{|u^{\prime}|^{2}}\frac{a\mathcal{V}}{a^{\frac{5}{2}}|u^{\prime}|}\text{d}u^{\prime} +\int_{u_{\infty}}^{u}\frac{a}{|u^{\prime}|^{2}}\frac{1}{|u^{\prime}|}\frac{a\mathcal{V}\Gamma}{a^{2}|u^{\prime}|}\text{d}u^{\prime} \lesssim \frac{1}{a^{\frac{1}{2}}|u|^{2}}+\frac{\mathcal{V}\Gamma}{|u|^{3}}\lesssim 1,
\end{align}

\begin{align}\notag 
&\intu \frac{a}{\upr^2} \left( \intubar \scaletwoSuprimeubarprime{a^{\frac{i}{2}} \sum_{i_{1}+i_{2}+i_{3}=i}\snabla^{i_{1}}\psi^{i_{2}}_{g}\snabla^{i_{3}}\snabla_{3}(T_{43},T_{A3})}^2 \dubarprime \right)^{\frac{1}{2}} \duprime \\
\lesssim& \int_{u_{\infty}}^{u}\frac{a}{|u^{\prime}|^{2}}\frac{|u^{\prime}|^{2}}{a}a\frac{1}{|u^{\prime}|}\frac{\mathcal{V}}{a^{2}|u^{\prime}|}\text{d}u^{\prime} +\int_{u_{\infty}}^{u}\frac{a}{|u^{\prime}|^{2}}\frac{1}{|u^{\prime}|}\frac{1}{|u^{\prime}|}\frac{|u^{\prime}|^{2}}{a}\frac{a\mathcal{V}\Gamma}{a^{\frac{3}{2}}|u^{\prime}|}\text{d}u^{\prime} \lesssim \frac{\mathcal{V}}{a|u|}+\frac{\mathcal{V}\Gamma}{a^{\frac{1}{2}}|u|^{2}}\lesssim 1,
\end{align}

\begin{align} \notag 
&\intu \frac{a}{\upr^2} \left( \intubar \scaletwoSuprimeubarprime{a^{\frac{i}{2}}\sum_{i_{1}+i_{2}+i_{3}+i_{4}=i}\snabla^{i_{1}}\psi^{i_{2}}_{g}\snabla^{i_{3}}(\psi_{g},\chihat,\chibarhat,\tr\chibar)\snabla^{i_{4}}(T_{4A},T_{43})}^2 \dubarprime \right)^{\frac{1}{2}} \duprime \\
\lesssim& \int_{u_{\infty}}^{u}\frac{a}{|u^{\prime}|^{2}}\frac{|u^{\prime}|^{2}\Gamma}{a}\frac{1}{|u^{\prime}|}\frac{\mathcal{V}}{a^{\frac{3}{2}}|u^{\prime}|}\lesssim \frac{\Gamma\mathcal{V}}{a^{\frac{1}{2}}|u|}\lesssim 1, 
\end{align}

\begin{align} \notag 
 &\intu \frac{a}{\upr^2} \left( \intubar \scaletwoSuprimeubarprime{a^{\frac{i}{2}}\sum_{i_{1}+i_{2}+i_{3}+i_{4}=i-1}\snabla^{i_{1}}\psi^{i_{2}+1}_{g}\snabla^{i_{3}}T_{3A}\snabla^{i_{4}}\Psi_{1}}^2 \dubarprime \right)^{\frac{1}{2}} \duprime \\
 \lesssim& \int_{u_{\infty}}^{u}\frac{a}{|u^{\prime}|^{2}}\frac{a}{|u^{\prime}|^{2}}\frac{\Gamma\mathcal{V}\mathcal{R}}{a^{2}|u^{\prime}|}\text{d}u^{\prime} \lesssim \frac{\Gamma\mathcal{V}\mathcal{R}}{|u|^{4}}\lesssim 1.
\end{align}

For the second and last one, a double application of H\"older's inequality yields
\begin{equation}
    \begin{split}
        &\intubar \intu  \frac{a}{\upr} \scaleoneSuprimeubarprime{a^{\frac{i}{2}}Q_i \cdot \aln \Psi_2 }  \duprime \dubarprime \\ \lesssim & \intubar \left( \intu \frac{a}{\upr^2} \scaletwoSuprimeubarprime{a^{\frac{i}{2}}Q_i}^2 \duprime \right)^{\frac{1}{2}} \lVert \aln \Psi_2 \rVert_{L^2_{(sc)}(\Hbar_{\ubar^{\prime}}^{(u_{\infty}, u)})} \\ \lesssim& \left(\intubar \intu \frac{a}{\upr^2} \scaletwoSuprimeubarprime{a^{\frac{i}{2}}Q_i}^2 \duprime \dubarprime \right)^{\frac{1}{2}} \left( \intubar \lVert \aln \Psi_2 \rVert^2_{L^2_{(sc)}(\Hbar_{\ubar^{\prime}}^{(u_{\infty}, u)})} \dubarprime \right)^{\frac{1}{2}} \\ \lesssim &\intubar \intu \frac{a}{\upr^2} \scaletwoSuprimeubarprime{a^{\frac{i}{2}}Q_i}^2 \duprime \dubarprime + \frac{1}{4} \intubar \lVert \aln \Psi_2 \rVert^2_{L^2_{(sc)}(\Hbar_{\ubar^{\prime}}^{(u_{\infty}, u)})} \dubarprime
    \end{split}
\end{equation}Define $B:= \intubar \intu \frac{a}{\upr^2} \scaletwoSuprimeubarprime{a^{\frac{i}{2}}Q_i}^2 \duprime \dubarprime$. The most important point to note here is that the term $Q_{i}$ contains the Vlasov source terms and it appears as the spacetime integral and therefore, we can control the Vlasov on $H$.  We can then estimate $B$ term by term as follows. The first term in $B$ is estimated as

\begin{align} \notag 
    &\intubar \intu \frac{a}{\upr^2}\scaletwoSuprimeubarprime{a^{\frac{i}{2}}\sumitm \sum_{i_{1}+i_{2}+i_{3}=i-1}\snabla^{i_{1}}\psi^{i_{2}+1}_{g} \snabla^{i_3+1}\Psi_1}^2 \duprime \dubarprime \\ \notag  \lesssim &\intubar \intu \frac{a}{\upr^2}\scaletwoSuprimeubarprime{a^{\frac{i}{2}}\psi_g \snabla^i \Psi_1}^2 \duprime \dubarprime \\ \notag  +& \intubar \intu \frac{a}{\upr^2}\scaletwoSuprimeubarprime{a^{\frac{i}{2}}\snabla \psi_g \snabla^{i-1}\Psi_1 }^2 \duprime \dubarprime \\ \notag +& \intubar \intu \frac{a}{\upr^2}\scaletwoSuprimeubarprime{a^{\frac{i}{2}}\psi_g \hsp  \psi_g  \snabla^{i-1}\Psi_1}^2 \duprime \dubarprime \\ \notag +&\intubar \intu \frac{a}{\upr^2}\scaletwoSuprimeubarprime{a^{\frac{i}{2}}\sum_{\substack{i_1+i_2+i_3=i-1\\ i_3< i-2}} \snabla^{i_{1}}\psi^{i_{2}+1}_{g} \snabla^{i_3+1}\Psi_1}^2 \duprime \dubarprime \\ \lesssim& 1, 
\end{align}

\noindent where in the first integral we estimate $\psi_{g}$ in $L^{\infty}_{sc}(S_{u,\ubar})$ and $\snabla^{i}\Psi_{1}$ in the hypersurface norm $L^2_{(sc)}(H_u^{(0,\ubar)})$. In the second integral, we estimate $\snabla\psi_{g}$ in $L^{4}_{sc}(S_{u,\ubar})$ and $\snabla^{i-1}\Psi_{1}$ in $L^{4}_{sc}(S_{u,\ubar})$ and subsequent codimension-1 trace inequality on $H_u^{(0,\ubar)}$. The third and fourth integrals are straightforward to control.  For the second term in $B$ is estimated as

\begin{align}  \nonumber
    &\intubar \intu \frac{a}{\upr^2}\scaletwoSuprimeubarprime{a^{\frac{i}{2}} \sumif \snabla^{i_{1}}\psi^{i_{2}}_{g} \snabla^{i_{3}}(\psi_{g},\chibarhat) \snabla^{i_{4}}(\Psi, \alpha) }^2 \duprime \dubarprime \\ \lesssim & \left(\mathcal{R}[\alpha]+1 \right)^2 \lesssim 1,
\end{align} 
where we have used energy estimates on $\alpha$ from proposition \ref{alpha_energy_estimate}. Now we estimate the source terms. Here we notice that $i+1=3$ derivatives of Vlasov appear, which are controlled by means of the derivative estimates for Vlasov (proposition \ref{prop:StressEnergyHighOrder}) 
\begin{align} \notag 
&\intubar \intu \frac{a}{\upr^2}\scaletwoSuprimeubarprime{a^{\frac{i}{2}} \sum_{i_{1}+i_{2}+i_{3}+i_{4}=i}\snabla^{i_{1}}\psi^{i_{2}}_{g}\snabla^{i_{3}+1}(T_{4A},T_{3A},T_{43})  }^2 \duprime \dubarprime\\
\lesssim& \int_{u_{\infty}}^{u}\frac{a}{|u^{\prime}|^{2}}a^{2}\frac{\mathcal{V}}{a^{4}|u|^{2}}\text{d}u^{\prime} +\int_{u_{\infty}}^{u}\frac{a}{|u^{\prime}|^{2}}a^{2}\frac{1}{|u^{\prime}|^{2}}\frac{\mathcal{V}\Gamma}{a^{3}|u^{\prime}|^{2}}\text{d}u^{\prime} \lesssim \frac{\mathcal{V}}{a|u|^{3}}+\frac{\mathcal{V}\Gamma}{|u|^{5}}\lesssim 1.
\end{align}
We estimate the next term (lower order derivatives on the vlasov terms)
\begin{align} \notag 
&\intubar \intu \frac{a}{\upr^2}\scaletwoSuprimeubarprime{a^{\frac{i}{2}} \sum_{i_{1}+i_{2}+i_{3}+i_{4}=i}\snabla^{i_{1}}\psi^{i_{2}}_{g}\snabla^{i_{3}}(\psi_{g},\chihat,\chibarhat,\tr\chibar)\snabla^{i_{4}}(T_{4A},T_{3A},T_{43})  }^2 \duprime \dubarprime \\
\lesssim& \int_{u_{\infty}}^{u}\frac{a}{|u^{\prime}|^{2}}a^{2}\frac{1}{|u^{\prime}|^{2}}\frac{\Gamma |u^{\prime}|^{4}}{a^{2}}\frac{\mathcal{V}}{a^{3}|u^{\prime}|^{2}}\text{d}u^{\prime} \lesssim \frac{\mathcal{V}\Gamma}{a^{2}|u|}\lesssim 1.
\end{align}
The next 
\begin{align} \notag 
&\intubar \intu \frac{a}{\upr^2}\scaletwoSuprimeubarprime{a^{\frac{i}{2}} \sum_{i_{1}+i_{2}+i_{3}+i_{4}=i}\snabla^{i_{1}}\psi^{i_{2}}_{g}\snabla^{i_{3}}\snabla_{4}(T_{43})  }^2 \duprime \dubarprime \\
\lesssim& \int_{u_{\infty}}^{u}\frac{a}{|u^{\prime}|^{2}}a^{2}\frac{\mathcal{V}}{a^{4}|u^{\prime}|^{2}}\text{d}u^{\prime} +\int_{u_{\infty}}^{u}\frac{a}{|u^{\prime}|^{2}}a^{2}\frac{1}{|u^{\prime}|^{2}}\frac{\Gamma\mathcal{V}}{a^{3}|u^{\prime}|^{2}}\text{d}u^{\prime} \lesssim \frac{\mathcal{V}}{a|u|^{3}}+\frac{\Gamma\mathcal{V}}{|u|^{5}}\lesssim 1.
\end{align}
The next term is estimated as 
\begin{align} \notag 
 &\intubar \intu \frac{a}{\upr^2}\scaletwoSuprimeubarprime{a^{\frac{i}{2}} \sum_{i_{1}+i_{2}+i_{3}+i_{4}=i}\snabla^{i_{1}}\psi^{i_{2}}_{g}\snabla^{i_{3}}\snabla_{3}(T_{44},T_{4A},T_{AB})  }^2 \duprime \dubarprime \\ \notag 
 \lesssim& \int_{u_{\infty}}^{u}\frac{a}{|u^{\prime}|^{2}}a^{2}\frac{1}{|u^{\prime}|^{2}}\frac{|u^{\prime}|^{6}}{a^{2}}\frac{1}{|u^{\prime}|^{2}}\frac{\mathcal{V}}{a^{2}|u^{\prime}|^{2}}\text{d}u^{\prime} +\int_{u_{\infty}}^{u}\frac{a}{|u^{\prime}|^{2}}a^{2}\frac{1}{|u^{\prime}|^{2}}\frac{1}{|u^{\prime}|^{2}}\frac{|u^{\prime}|^{6}}{a^{2}}\frac{1}{|u^{\prime}|^{2}}\frac{\Gamma\mathcal{V}}{a|u^{\prime}|^{2}}\text{d}u^{\prime} \\
 \lesssim& \frac{\mathcal{V}}{a|u|}+\frac{\Gamma\mathcal{V}}{|u|^{3}}\lesssim 1,
\end{align}
and finally 
\begin{align}
\intubar \intu \frac{a}{\upr^2}\scaletwoSuprimeubarprime{a^{\frac{i}{2}} \sum_{i_{1}+i_{2}+i_{3}+i_{4}=i-1}\snabla^{i_{1}}\psi^{i_{2}+1}_{g}\snabla^{i_{3}}T_{4A}\snabla^{i_{4}}\Psi_{2}  }^2 \duprime \dubarprime .
\end{align}
Here,  we use $L^{\infty}_{sc}(S_{u,\ubar})-L^{4}_{sc}(S_{u,\ubar})-L^{4}_{sc}(S_{u,\ubar})$  type estimates together with the Sobolev embedding on $S_{u,\ubar}$ to yield 
\begin{align} \notag 
&\intubar \intu \frac{a}{\upr^2}\scaletwoSuprimeubarprime{a^{\frac{i}{2}} \sum_{i_{1}+i_{2}+i_{3}+i_{4}=i-1}\snabla^{i_{1}}\psi^{i_{2}+1}_{g}\snabla^{i_{3}}T_{4A}\snabla^{i_{4}}\Psi_{2}  }^2 \duprime \dubarprime\\
\lesssim& \int_{u_{\infty}}^{u}\frac{a}{|u^{\prime}|^{2}}a^{2}\frac{1}{|u^{\prime}|^{4}}\frac{\Gamma\mathcal{V}\mathcal{R}}{a}\text{d}u^{\prime} \lesssim \frac{a^{2}\Gamma\mathcal{V}\mathcal{R}}{|u|^{5}}\lesssim 1.
\end{align}
The rest of the terms can also be bounded above by $1$, using the same approach. We finally arrive at an estimate of the form 

\begin{equation}\label{comb2}
    \begin{split}
        &\intubar \intu  \frac{a}{\upr} \scaleoneSuprimeubarprime{a^{\frac{i}{2}} Q_i \cdot \aln \Psi_2 }  \duprime \dubarprime \\ \lesssim & 1 + \frac{1}{4} \intubar \lVert \aln \Psi_2 \rVert^2_{L^2_{(sc)}(\Hbar_{\ubar^{\prime}}^{(u_{\infty}, u)})} \dubarprime.
        \end{split}
\end{equation}
From here, collecting all the terms together and using Gronwall's inequality, we arrive at the desired result. The second energy estimate follows exactly in a similar fashion to the previous proposition. 

\end{proof}

\section{Dynamical Formation of a Trapped Surface}

We prove the trapped-surface formation statement. The argument is an ODE argument
along the incoming direction \(u\), followed by the Raychaudhuri equation along
the outgoing direction \(\ubar\). Throughout this section \(u<0\), all
implicit constants depend only on the constants appearing in the initial data bounds.

\noindent Assume that the initial flux satisfies, uniformly in \(\theta\in\mathbb S^2\),
\begin{equation}\label{initial-flux-lower}
\int_0^1
\left(
        |u_\infty|^2|\hat\chi|^2
        +
        |u_\infty|^2T_{44}
\right)
(u_\infty,\ubar',\theta)\,d\ubar'
\geq a .
\end{equation}

\noindent We first propagate the shear flux. Recall the null structure equation
\begin{equation}\label{chihat-e3-eq}
\snabla_3\hat\chi
+\frac12\tr\chibar\,\hat\chi
+\omegabar\hat\chi
=
\snabla\hat\otimes\eta
-\frac12\tr\chi\,\hat\chibar
+\eta\hat\otimes\eta
+\hat{\slashed T}.
\end{equation}
Contracting with \(\hat\chi\), and using
\(\omegabar=2\Omega^{-1}\snabla_3\Omega\), gives
\begin{align}
\snabla_3\left(\Omega^4|\hat\chi|^2\right)
+\Omega^4\tr\chibar|\hat\chi|^2
=
2\Omega^4\hat\chi\cdot
\left(
        \snabla\hat\otimes\eta
        -\frac12\tr\chi\,\hat\chibar
        +\eta\hat\otimes\eta
        +\hat{\slashed T}
\right).
\end{align}
Since \(\snabla_3=\partial_u+b^Ae_A\) on scalar functions, we have
\begin{align}
\partial_u\left(\Omega^4|\hat\chi|^2\right)
+\Omega^4\tr\chibar|\hat\chi|^2
=&\;
2\Omega^4\hat\chi\cdot
\left(
        \snabla\hat\otimes\eta
        -\frac12\tr\chi\,\hat\chibar
        +\eta\hat\otimes\eta
        +\hat{\slashed T}
\right)                                      \notag\\
&\;
-b^Ae_A\left(\Omega^4|\hat\chi|^2\right).
\end{align}
Using
\[
\Omega^4\tr\chibar
=
\Omega^4\left(\tr\chibar+\frac{2}{|u|}\right)
-\frac{2}{|u|}(\Omega^4-1)
-\frac{2}{|u|},
\]
we obtain
\begin{align}
\partial_u\left(|u|^2\Omega^4|\hat\chi|^2\right)
=&\;
2|u|^2\Omega^4\hat\chi\cdot
\left(
        \snabla\hat\otimes\eta
        -\frac12\tr\chi\,\hat\chibar
        +\eta\hat\otimes\eta
        +\hat{\slashed T}
\right)                                      \notag\\
&\;
-|u|^2b^Ae_A\left(\Omega^4|\hat\chi|^2\right) \notag\\
&\;
-|u|^2\Omega^4
\left(\tr\chibar+\frac{2}{|u|}\right)|\hat\chi|^2
+2|u|(\Omega^4-1)|\hat\chi|^2 .
\label{renormalized-shear-eq}
\end{align}
The estimates proved in the preceding sections give
\begin{align}
\int_{u_\infty}^{u}
|u'|^2|\hat\chi|
\left(
        |\snabla\eta|
        +|\tr\chi|\,|\hat\chibar|
        +|\eta|^2
        +|\hat{\slashed T}|
\right)\,du'
&\lesssim
\frac{a^{7/4}}{|u|}
-\frac{a^{7/4}}{|u_\infty|},                  \\
\int_{u_\infty}^{u}
|u'|^2|b|\,|\snabla(\Omega^4|\hat\chi|^2)|\,du'
&\lesssim
\frac{a^{7/4}}{|u|}
-\frac{a^{7/4}}{|u_\infty|},                  \\
\int_{u_\infty}^{u}
|u'|^2
\left|\tr\chibar+\frac{2}{|u'|}\right|
|\hat\chi|^2\,du'
&\lesssim
\frac{a^{7/4}}{|u|}
-\frac{a^{7/4}}{|u_\infty|},                  \\
\int_{u_\infty}^{u}
|u'|\,|\Omega^4-1|\,|\hat\chi|^2\,du'
&\lesssim
\frac{a^{7/4}}{|u|}
-\frac{a^{7/4}}{|u_\infty|}.
\end{align}
Therefore
\begin{equation}\label{shear-propagation}
\int_{u_\infty}^{u}
\left|
\partial_{u'}
\left(|u'|^2\Omega^4|\hat\chi|^2\right)
\right|\,du'
\lesssim
\frac{a^{7/4}}{|u|}
-\frac{a^{7/4}}{|u_\infty|}.
\end{equation}
Since \(\Omega=1+O(a^{1/2}|u|^{-1})\), the same estimate holds without the
factor \(\Omega^4\):
\begin{equation}\label{shear-propagation-no-omega}
\int_{u_\infty}^{u}
\left|
\partial_{u'}
\left(|u'|^2|\hat\chi|^2\right)
\right|\,du'
\lesssim
\frac{a^{7/4}}{|u|}
-\frac{a^{7/4}}{|u_\infty|}.
\end{equation}
Consequently,
\begin{equation}\label{shear-lower}
\int_0^1 |u|^2|\hat\chi|^2(u,\ubar',\theta)\,d\ubar'
\geq
\int_0^1 |u_\infty|^2|\hat\chi|^2(u_\infty,\ubar',\theta)\,d\ubar'
-
C\frac{a^{7/4}}{|u|}.
\end{equation}

\noindent We now propagate the Vlasov flux. We do the following renormalization. This is vital for our case. Without this renormalization, a vital cancellation structure is absent. Set
\[
        \mathcal T_{44}:=|u|^2T_{44}.
\]
Since \(\partial_u=\snabla_3-b^Ae_A\) and \(\partial_u|u|=-1\), we have
\begin{align}
\partial_u(|u|^2T_{44})
&=
|u|^2
\left(
        \snabla_3T_{44}
        -b^Ae_A(T_{44})
        -\frac{2}{|u|}T_{44}
\right).
\label{T44-renormalized-identity}
\end{align}
The cancellation in the first and third terms is exact. Indeed,
\(T_{44}=|u|^{-2}\mathcal T_{44}\), and hence
\[
\snabla_3T_{44}
=
\snabla_3(|u|^{-2}\mathcal T_{44})
=
\frac{2}{|u|}T_{44}
+
|u|^{-2}\snabla_3\mathcal T_{44}.
\]
Moreover \(e_A|u|=0\), so
\[
b^Ae_A(T_{44})=|u|^{-2}b^Ae_A\mathcal T_{44}.
\]
Thus
\begin{equation}\label{T44-renormalized-exact}
\partial_u(|u|^2T_{44})
=
\snabla_3\mathcal T_{44}
-
b^Ae_A\mathcal T_{44}.
\end{equation}

\noindent We estimate the two terms on the right. Parametrize the mass shell by
\((p^3,p^1,p^2)\). Since \(p_4=-2p^3\), the stress-energy component has the form
\[
T_{44}
=
C\int p^3 f\,\sqrt{\det\slashed g}\,dp^3dp^1dp^2 .
\]
Write
\[
\slashed g_{AB}=|u|^2\gamma_{AB},
\qquad
q^3=p^3,
\qquad
q^A=|u|^2p^A .
\]
Then
\[
\sqrt{\det\slashed g}\,dp^1dp^2
=
|u|^2\sqrt{\det\gamma}\cdot |u|^{-4}dq^1dq^2
=
|u|^{-2}\sqrt{\det\gamma}\,dq^1dq^2.
\]
Therefore
\begin{equation}\label{renorm-T44-moment}
\mathcal T_{44}
=
C\int_K q^3\sqrt{\det\gamma}\,F(u,\ubar,\theta,q)\,dq,
\end{equation}
where \(F(u,\ubar,\theta,q)=f(u,\ubar,\theta,p(u,q))\), and \(K\) is a fixed
compact momentum set.

\noindent We next compute the \(e_3\)-derivative of the weight
\(M=Cq^3\sqrt{\det\gamma}\). Since \(q\) is held fixed,
\[
\snabla_3M
=
Cq^3\snabla_3\sqrt{\det\gamma}.
\]
Now observe the elementary computations
\[
\snabla_3\gamma_{AB}
=
\snabla_3(|u|^{-2}\slashed g_{AB})
=
2|u|^{-3}\slashed g_{AB}
+
2|u|^{-2}\chibar_{AB}.
\]
Using
$
\chibar_{AB}
=
\frac12\tr\chibar\,\slashed g_{AB}
+
\hat\chibar_{AB},
$
we get
\[
\snabla_3\gamma_{AB}
=
\left(\tr\chibar+\frac{2}{|u|}\right)\gamma_{AB}
+
2|u|^{-2}\hat\chibar_{AB}.
\]
Taking the \(\gamma\)-trace, the trace-free part drops out, therefore yields
$
\gamma^{AB}\snabla_3\gamma_{AB}
=
2\left(\tr\chibar+\frac{2}{|u|}\right).
$
Hence
\[
\snabla_3\sqrt{\det\gamma}
=
\frac12\sqrt{\det\gamma}\,\gamma^{AB}\snabla_3\gamma_{AB}
=
\sqrt{\det\gamma}
\left(\tr\chibar+\frac{2}{|u|}\right).
\]
Thus
\[
|\snabla_3M|
\lesssim
|q^3|
\left|\tr\chibar+\frac{2}{|u|}\right|
\lesssim
\frac{1}{|u|^2}w(q).
\]
The commuted Vlasov equation in the renormalized momentum variables gives
\[
|\snabla_3F|
\lesssim
\frac{1}{|u|^2}\sum_{1\leq i\leq 7}|F_i f|,
\]
where we have noticed the crucial cancellation structure in terms of integrability as one of the important aspects of our work. Remarkably, this is also important to close this estimate at the level of the ODE argument of trapped surface formation.
Therefore, using \eqref{renorm-T44-moment},
\begin{align}
|\snabla_3\mathcal T_{44}|
&\leq
\int_K |\snabla_3M|\,|F|\,dq
+
\int_K |M|\,|\snabla_3F|\,dq                       \notag\\
&\lesssim
\frac{1}{|u|^2}
\int_K w(q)
\left(
        |f|+\sum_{1\leq i\leq 7}|F_i f|
\right)dq.
\end{align}
The Vlasov hierarchy and the codimension-one trace inequality imply at the worst order
\[
\sup_{\ubar,\theta}
\int_K w(q)
\left(
        |f|+\sum_{1\leq i\leq 7}|F_i f|
\right)dq
\lesssim a^{7/4}.
\]
Consequently,
\begin{equation}\label{nabla3-T44-renorm-bound}
|\snabla_3\mathcal T_{44}|
\lesssim
\frac{a^{7/4}}{|u|^2}.
\end{equation}

\noindent For the shift term, \(e_A|u|=0\), and hence
\[
|u|^2|b^Ae_AT_{44}|=|b^Ae_A\mathcal T_{44}|.
\]
Differentiating \eqref{renorm-T44-moment} angularly and using the first angular
commuted Vlasov estimate, together with the angular Ricci-coefficient estimates,
gives
\[
|e_A\mathcal T_{44}|
\lesssim
\frac{a^{5/4}}{|u|}.
\]
The shift estimate gives
\[
|b|\lesssim \frac{a^{1/2}}{|u|}.
\]
Therefore
\begin{equation}\label{shift-T44-renorm-bound}
|b^Ae_A\mathcal T_{44}|
\lesssim
\frac{a^{1/2}}{|u|}
\cdot
\frac{a^{5/4}}{|u|}
=
\frac{a^{7/4}}{|u|^2}.
\end{equation}
Combining \eqref{T44-renormalized-exact}, \eqref{nabla3-T44-renorm-bound}, and
\eqref{shift-T44-renorm-bound}, we obtain
\begin{equation}\label{T44-renormalized-derivative-bound}
\left|
\partial_u(|u|^2T_{44})
\right|
\lesssim
\frac{a^{7/4}}{|u|^2}.
\end{equation}
Hence
\begin{equation}\label{T44-propagation}
\int_{u_\infty}^{u}
\left|
\partial_{u'}(|u'|^2T_{44})
\right|du'
\lesssim
a^{7/4}\int_{u_\infty}^{u}\frac{du'}{|u'|^2}
=
\frac{a^{7/4}}{|u|}
-
\frac{a^{7/4}}{|u_\infty|}.
\end{equation}
Therefore
\begin{equation}\label{T44-lower}
\int_0^1 |u|^2T_{44}(u,\ubar',\theta)\,d\ubar'
\geq
\int_0^1 |u_\infty|^2T_{44}(u_\infty,\ubar',\theta)\,d\ubar'
-
C\frac{a^{7/4}}{|u|}.
\end{equation}

Combining \eqref{shear-lower}, \eqref{T44-lower}, and
\eqref{initial-flux-lower}, we get
\[
\int_0^1
\left(
        |u|^2|\hat\chi|^2
        +
        |u|^2T_{44}
\right)(u,\ubar',\theta)\,d\ubar'
\geq
a-C\frac{a^{7/4}}{|u|}.
\]
For \(u_\infty\leq u\leq -a/4\), we have \(|u|\geq a/4\), so
\[
C\frac{a^{7/4}}{|u|}
\leq
4Ca^{3/4}.
\]
Taking \(a\) sufficiently large, this is at most \(a/4\). Hence
\begin{equation}\label{renorm-flux-lower}
\int_0^1
\left(
        |u|^2|\hat\chi|^2
        +
        |u|^2T_{44}
\right)(u,\ubar',\theta)\,d\ubar'
\geq
\frac{3a}{4}.
\end{equation}
At \(u=-a/4\), this gives
\begin{equation}\label{final-unrenorm-flux}
\int_0^1
\left(
        |\hat\chi|^2+T_{44}
\right)\left(-\frac a4,\ubar',\theta\right)d\ubar'
\geq
\frac{12}{a}.
\end{equation}

\noindent We now use Raychaudhuri:
\[
\snabla_4\tr\chi+\frac12(\tr\chi)^2
=
-|\hat\chi|^2-T_{44}.
\]
Since \(e_4=\Omega^{-2}\partial_{\ubar}\),
\[
\partial_{\ubar}\tr\chi
=
-\Omega^2
\left(
        \frac12(\tr\chi)^2+|\hat\chi|^2+T_{44}
\right)
\leq
-\Omega^2\left(|\hat\chi|^2+T_{44}\right).
\]
Integrating from \(\ubar=0\) to \(\ubar=1\) on \(H_{-a/4}\), and using the
Minkowskian initial value
\[
\tr\chi\left(-\frac a4,0,\theta\right)=\frac{8}{a},
\]
we obtain
\[
\tr\chi\left(-\frac a4,1,\theta\right)
\leq
\frac{8}{a}
-
\int_0^1
\Omega^2
\left(|\hat\chi|^2+T_{44}\right)
\left(-\frac a4,\ubar',\theta\right)d\ubar'.
\]
By the lapse estimate, after increasing \(a\) if necessary, \(\Omega^2\geq5/6\).
Using \eqref{final-unrenorm-flux},
\[
\tr\chi\left(-\frac a4,1,\theta\right)
\leq
\frac{8}{a}
-
\frac56\cdot\frac{12}{a}
=
-\frac{2}{a}<0.
\]

\noindent Finally, the estimate
\[
\left|
\tr\chibar+\frac{2}{|u|}
\right|
\lesssim
\frac{1}{|u|^2}
\]
implies, at \(u=-a/4\),
\[
\tr\chibar\left(-\frac a4,1,\theta\right)
\leq
-\frac{8}{a}
+
C\frac{1}{a^2}
<0
\]
for \(a\) sufficiently large. Thus both null expansions are strictly negative on
\(S_{-a/4,1}\). Therefore \(S_{-a/4,1}\) is a trapped surface.


\begin{thebibliography}{}
\bibitem{A19}X. An, \textit{A scale-critical trapped surface formation criterion: A new proof via signature for decay rates}, Annals of PDE {Vol. 8} (1), 2022.

\bibitem{A17}  X. An. \textit{Emergence of apparent horizon in gravitational collapse}, Annals of PDE Vol. 6 (10), 2020.

\bibitem{AnThesis} X. An, \textit{Formation of trapped surfaces from past null infinity}, arXiv:1207.5271 

\bibitem{AnAth} X. An, N. Athanasiou, \textit{A scale critical trapped surface formation criterion for the Einstein-Maxwell}, Journal de Mathématiques Pures et Appliquées Vol. 167 (2022), 294-409.


\bibitem{An} X. An, T. He, D. Shen, \textit{Stability of Big Bang singularity for the Einstein-Maxwell-scalar field-Vlasov system in the full strong sub-critical regime}, arXiv:2507.18585, 2025.  


\bibitem{AL17} X. An, J. Luk. \textit{Trapped surfaces in vacuum arising dynamically from mild incoming radiation}. Advances in Theoretical and Mathematical Physics 21 (2017), 1-120.



\bibitem{fajman2} L. Andersson, D. Fajman, \textit{Nonlinear stability of the Milne model with matter}, Communications in Mathematical Physics vol. 378, 261-298, 2020. 

\bibitem{andr1} H. Andr{\'e}asson, M. Kunze, G. Rein, The formation of black holes in spherically symmetric gravitational collapse, \textit{Math. Ann.}, vol. 350, 683-705, 2011.   


\bibitem{andr2} H. Andr{\'e}asson, Black hole formation from a complete regular past for collisionless matter, \textit{Ann. Hen. Poincare}, vol. 13, 1511-1536. 

\bibitem{andr3}
H. Andréasson, G. Rein, Oppenheimer–Snyder Type Collapse for a Collisionless Gas, \textit{Comm. Math. Phys.}, 406 (2025). Paper No. 284.




\bibitem{A-M-Y} N. Athanasiou, P. Mondal, S-T. Yau, \textit{Formation of trapped surfaces in the Einstein-Yang-Mills system}, Journal de Math{\'e}matiques Pures et Appliqu{\'e}es, vol. 194, 2025.  




\bibitem{fajman1} L. Bigorgne, D. Fajman, J. Joudioux, J. Smulevici, \textit{Asymptotic stability of Minkowski space-time with non-compactly supported massless Vlasov matter}, vol. 242, 2021.   





\bibitem{C93} D. Christodoulou, \textit{Bounded variation solutions of the spherically symmetric Einstein-scalar field equations}, Communications in Pure and Applied Mathematics 46 (1993), 1131-1220. 
\bibitem{C94} D. Christodoulou, \textit{Examples of naked singularity formation in the gravitational collapse of a scalar field}, Annals of Mathematics 140, no. 3 (1994), 607-653. 
\bibitem{C91} D. Christodoulou, \textit{The formation of black holes and singularities in spherically symmetric gravitational collapse}. Communications in Pure and Applied Mathematics 44, 339-373 (1991).


\bibitem{C09} D. Christodoulou, \textit{The formation of black holes in general relativity}. Monographs in Mathematics, European Mathematical Society (2009).

\bibitem{C99} D. Christodoulou, \textit{The instability of naked singularities in the gravitational collapse of a scalar field}. Annals of Mathematics 149, 183-217 (1999).
\bibitem{ChrKl} D. Christodoulou, S. Klainerman, \textit{The global nonlinear stability of the Minkowski space}, Princeton University Press (1993).


\bibitem{dafermos} M. Dafermos, A.D. Rendall, \textit{An extension principle for the Einstein-Vlasov system in spherical symmetry}, Annales Henri. Poincare, vol. 6, 1137-1155, 2005 

\bibitem{Dafermosscattering} M. Dafermos, G. Holzegel, I. Rodnianski, \textit{A scattering theory construction of dynamical vacuum black holes}, arXiv:1306.5364, 2013.    




\bibitem{fajman} D. Fajman, J. Joudioux, J. Smulevici, \textit{The stability of the Minkowski space for the Einstein-Vlasov system}, Analysis$\&$ PDE, vol. 14, 425-531, 2021.  




\bibitem{Kl-Rod}  S. Klainerman, I. Rodnianski. \textit{On the formation of trapped surfaces.} Acta Mathematica Vol. 208 (2012), 
211-333.





\bibitem{KRS} S. Klainerman, I. Rodnianski, J. Szeftel, The bounded L 2 curvature conjecture, Vol. 202, 91-216, 2015. 



\bibitem{TL} H. Lindblad, M. Taylor, \textit{Global stability of Minkowski space for the Einstein-Vlasov system in the harmonic gauge}, Archive for rational mechanics and analysis. vol. 235, 517-633, 2020.   




\bibitem{luklocal} J. Luk, On the local existence for the characteristic initial value problem in general relativity, \textit{International Mathematics Research Notices}, vol. 2012, 4625-4678, 2012.
\bibitem{LukNull} J. Luk, \textit{Weak Null Singularities in General Relativity}, \textit{Journal of the American Mathematical Society} (2018).
\bibitem{M-Y} P. Mondal, S.T. Yau, \textit{Einstein-Yang-Mills equations in the double null framework}, \textit{Comm. Anal. Geom.}, Vol. 33, 2015-2152, 2025.

\bibitem{P73} R. Penrose. \textit{Naked singularities}. Annals of the  New York Academy of Sciences 224 (1973), 125–134.


\bibitem{rendall} G. Rein, A.D. Rendall, \textit{Global existence of solutions of the spherically symmetric Vlasov-Einstein system with small initial data}, Comm. Math. Phys., vol. 150, 561-583, 1992.  
\bibitem{rendallerratum}
G. Rein, A.D. Rendall, \textit{Erratum: 
Global existence of solutions of the spherically symmetric Vlasov-Einstein system with small initial data''}, Comm. Math. Phys, vol. 176, 475-478, 1996.

\bibitem{Sb} J. Sbierski, \textit{The $C^0$-inextendibility of the Schwarzschild spacetime and the spacelike diameter in Lorentzian geometry}, Journal of Differential Geometry 108 (2018), 319-378.


\bibitem{taylor} M. Taylor, The global nonlinear stability of Minkowski space for the massless Einstein--Vlasov system, Ann. PDE, vol. 3, 2017. 




\bibitem{SY83} R. Schoen, S.T. Yau. \textit{The existence of a black hole due to condensation of matter.} Comm. Math.
Phys, 90 no. 4 (1983),  575–579.

\bibitem{yau2001geometry} S.T. Yau, \textit{Geometry of three manifolds and existence of black hole due to boundary effect}, 
Adv. Theo. Math. Phys., 2001


 










\end{thebibliography}
\end{document}